\documentclass[a4paper,12pt]{article}

\newcommand{\spn}{\mathrm{span}}

\newcommand{\Cinf}{C^\infty}

\newcommand{\diff}{\mathrm{d}}
\newcommand{\pr}{\mathrm{pr}}
\newcommand{\id}{\mathrm{id}}

\newcommand{\n}[1]{\left\Vert #1\right\Vert} % NORMI
\newcommand{\la}{\left\langle}
\newcommand{\ra}{\right\rangle} % SISÄTULO

\newcommand{\R}{\mathbb{R}}

\newcommand{\N}{\mathbb{N}}
\newcommand{\Z}{\mathbb{Z}}

\newcommand{\be}{\begin{equation}}
\newcommand{\ee}{\end{equation}}
\newcommand{\bea}{\begin{eqnarray}}
\newcommand{\eea}{\end{eqnarray}}
\newcommand{\ben}{\begin{displaymath}}
\newcommand{\een}{\end{displaymath}}
\newcommand{\bean}{\begin{eqnarray*}}
\newcommand{\eean}{\end{eqnarray*}}
\def\ol#1{\overline{#1}}

\newcommand{\mc}[1]{\mathcal{#1}}

\newcommand{\di}[2]{\frac{\diff #1}{\diff #2}}
\newcommand{\dif}[1]{\frac{\diff}{\diff #1}}

\newcommand{\pa}[1]{{{\partial}\over{\partial #1}}}

\newcommand{\mscr}[1]{\textrm{\fontencoding{OMS}\fontfamily{ztmcm}\selectfont#1}
}

\newcommand{\sfrac}[2]{{\textstyle{\frac{#1}{#2}}}}

\usepackage{amsmath}
\usepackage{amssymb}
\usepackage{amsthm}
\usepackage[toc,page]{appendix}
\usepackage[english]{babel}
\usepackage{verbatim}
\usepackage[T1]{fontenc}
\usepackage{fancyhdr}
\usepackage[hmargin=3cm,vmargin=2.5cm]{geometry}
\usepackage{ifthen}
\usepackage{color}

\newcommand{\Euc}{\mathrm{SE}}

\newcommand{\IM}{\mathrm{im}}

\newcommand{\tr}{\mathrm{tr}}

\newcommand{\Rol}{\mathsf{Rol}}
\newcommand{\GL}{\mathrm{GL}}

\newcommand{\VF}{\mathrm{VF}}
\newcommand{\SO}{\mathrm{SO}}
\newcommand{\so}{\mathfrak{so}}
\newcommand{\gl}{\mathfrak{gl}}

\renewcommand{\di}[2]{\frac{\mathrm{d}{#1}}{\mathrm{d}{#2}}}
\newcommand{\Lie}{\mathrm{Lie}}
\newcommand{\rank}{\mathrm{rank\ }}

\newcommand{\sns}{(\Sigma)_{NS}}
\newcommand{\srol}{(\Sigma)_{R}}

\newcommand{\hM}{\hat{M}}
\newcommand{\hg}{\hat{g}}

\newcommand{\NSDist}{\mc{D}_{\mathrm{NS}}}
\newcommand{\RDist}{\mc{D}_{\mathrm{R}}}
\newcommand{\LNSD}{\mscr{L}_{\mathrm{NS}}}

\newcommand{\LRD}{\mscr{L}_{\mathrm{R}}}

\newcommand{\qmatrix}[1]{ \left( \begin{matrix} #1 \end{matrix} \right) }

\newcommand{\ENDP}{\mathrm{end}}

\def\[#1\]{\begin{align*}#1\end{align*}}

\newtheoremstyle{theorem}{0.5cm}{\topsep}%
   {\sffamily}%         Body font
   {}%         Indent amount (empty = no indent, \parindent = para indent)
   {\bfseries}% Thm head font
   {}%        Punctuation after thm head
   {2ex}%     Space after thm head (\newline = linebreak)
   {\thmname{#1}\thmnumber{ #2}\thmnote{ #3}}%         Thm head spec
\theoremstyle{theorem}
\newtheorem{theorem}{Theorem}[section]

\newtheorem{proposition}[theorem]{Proposition}
\newtheorem{example}[theorem]{Example}
\newtheorem{corollary}[theorem]{Corollary}
\newtheorem{remark}[theorem]{Remark}
\newtheorem{definition}[theorem]{Definition}
\newtheorem{lemma}[theorem]{Lemma}

\begin{document}

\title{Rolling Manifolds: Intrinsic Formulation and Controllability\thanks{
The work of the first author is supported by the ANR project GCM, program ``Blanche'',
(project number NT09$\_$504490)
and the DIGITEO-R\'egion Ile-de-France project CONGEO.
The work of the second author is supported by Finnish Academy of Science and Letters
and Saastamoinen Foundation.
}
}

\author{
Yacine Chitour\thanks{\texttt{yacine.chitour@lss.supelec.fr}, L2S, Universit\'e Paris-Sud XI, CNRS and Sup\'elec, Gif-sur-Yvette, 91192, France. 
}\and Petri Kokkonen\thanks{\texttt{petri.kokkonen@lss.supelec.fr}, L2S, Universit\'e Paris-Sud XI, CNRS and Sup\'elec, Gif-sur-Yvette, 91192, France and University of Eastern Finland, Department of Physics and Mathematics, 70211, Kuopio, Finland.}
}

\date{\today}

\maketitle

\begin{abstract}
  In this paper, we consider two cases of rolling of one smooth connected
  complete Riemannian manifold $(M,g)$ onto another one $(\hM,\hg)$
  of equal dimension $n\geq 2$.  The rolling problem $(NS)$ corresponds to the situation where  there is no relative spin (or twist) of one manifold with respect to the other one.  As for the rolling problem $(R)$, there is no relative spin and also no relative slip. Since the manifolds are not assumed to be embedded into an Euclidean space, we provide an intrinsic description of the two constraints ''without spinning'' and ''without slipping'' in terms of the Levi-Civita connections $\nabla^{g}$ and $\nabla^{\hg}$.  For that purpose, we recast the two rolling problems within the framework of geometric control and associate to each of them a distribution and a control system. We then investigate the relationships between the two control systems
 and we address for both of them the issue
  of complete controllability. For the rolling $(NS)$,
  the reachable set (from any point) can be described
  exactly in terms of the holonomy groups of $(M,g)$ and $(\hM,\hg)$
  respectively, and thus we achieve a complete understanding of the
  controllability properties of the corresponding control system. As
  for the rolling $(R)$, the problem turns out
  to be more delicate. We first provide basic global properties for the reachable set and investigate the associated Lie bracket structure. In particular, we point out the role played by a curvature tensor defined on the state space, that we call the \emph{rolling curvature}. In the case where one of the manifolds is a space form (let say $(\hM,\hg)$), we show that  it is enough to roll along loops of $(M,g)$ and the resulting orbits carry a structure of principal bundle which preserves the rolling $(R)$ distribution. In the zero curvature case, we deduce that the  rolling $(R)$ is completely controllable if and only if the holonomy group of $(M,g)$ is equal to $SO(n)$.  In the nonzero curvature case, we prove that the structure group of the principal bundle can be realized as the holonomy group of a connection on $TM\oplus \R$, that we call the rolling connection. We also show, in the case of positive (constant) curvature, that if the rolling connection is reducible, then $(M,g)$ admits, as Riemannian covering, the unit sphere with the metric induced from the Euclidean metric of $\R^{n+1}$.
When the two manifolds are three-dimensional, we provide a complete local characterization of the reachable sets when the two manifolds are three-dimensional and, in particular, we identify necessary and sufficient conditions for the existence of a non open orbit. Besides the trivial case where the manifolds  $(M,g)$ and $(\hM,\hg)$ are (locally) isometric, we show that (local) non controllability occurs if and only if  $(M,g)$ and $(\hM,\hg)$ are either warped products or contact manifolds with additional restrictions that we precisely describe. Finally, we extend the two types of rolling to the case where the manifolds have different dimensions.
  
  \end{abstract}
  
\tableofcontents

\newpage
\section{Introduction}

In this paper, we study the rolling of a manifold over another one. Unless otherwise precised, manifolds are smooth, connected, oriented, of finite dimension $n\geq 2$, endowed with a complete Riemannian metric. The rolling is assumed to be either without spinning 
$(NS)$ or without spinning nor slipping $(R)$. When both manifolds are isometrically embedded into an Euclidean space, the rolling problem is classical in differential geometry (see \cite{sharpe97}), through the notions of ''development of a manifold'' and ''rolling maps''.
To get an intuitive grasp of the problem, consider the rolling problem $(R)$ of a $2$D convex surface $S_1$ onto another one $S_2$ in the euclidean space $\R^3$, for instance the plate-ball problem, i.e., a sphere rolling onto a plane  in $\R^3$, (cf. \cite{jurd} and \cite{murray-sastry}).  The two surfaces are in contact i.e., they have a common tangent plane at the contact point and, equivalently, their exterior normal vectors are opposite at the contact point. 
If $\gamma:[0,T]\to S_1$ is a $C^1$ regular curve on $S_1$, one says that $S_1$ rolls onto $S_2$ along $\gamma$ without spinning nor slipping if the following holds. The curve traced on $S_1$ by the contact point is equal to $\gamma$ and let 
$\hat\gamma:[0,T]\to S_2$ be the curve traced on $S_2$ by the contact point. 
At time $t\in [0,T]$, the relative orientation of $S_2$ with respect to $S_1$ is measured by the angle 
$\theta(t)$ between  $\dot\gamma(t)$ and $\dot{\hat{\gamma}}(t)$ in the common tangent plane at the contact point. The state space $Q$ of the rolling problem is therefore five dimensional since a point in $Q$ is defined by fixing a point on $S_1$, a point on $S_2$ and an angle in $S^1$, the unit circle. The no-slipping condition says that $\dot{\hat{\gamma}}(t)$ is equal to $\dot\gamma(t)$ rotated by the angle $\theta(t)$ and the no-spinning condition
characterizes $\dot\theta(t)$ in term of the surface elements at $\gamma(t)$ and ${\hat{\gamma}}(t)$ respectively. Then, once a point on $S_2$ and an angle are chosen at time $t=0$, the curves $\hat\gamma$ and $\theta$ are uniquely determined. 
For the rolling $(NS)$, one must choose two  $C^1$ regular curves $\gamma$ and $\hat\gamma$ on $S_1$ and $S_2$ 
respectively, and an angle $\theta_0$ so that one says that $S_1$ rolls onto $S_2$ along $\gamma$ and $\hat\gamma$ without spinning if $(a)$ the curves traced on $S_1$ and $S_2$ by the contact point are equal to $\gamma$ and
$\hat\gamma$ respectively; $(b)$ the no-spin constraint and the initial condition $\theta_0$
determine a unique curve $\theta$ which measures the relative orientation of $S_2$ with respect to $S_1$ along the rolling. 
The most basic issue linked to the rolling problems is that of {\it controllability} i.e., to determine, for two given points $q_{\mathrm{init}}$ and $q_{\mathrm{final}}$ in the state space $Q$, if there exists a curve $\gamma$ so that the rolling of $S_1$ onto $S_2$ along $\gamma$ steers the system from $q_{\mathrm{init}}$ to $q_{\mathrm{final}}$. If this is the case for every points $q_{\mathrm{init}}$ and $q_{\mathrm{final}}$ in $Q$, then the rolling of $S_1$ onto $S_2$ is said to be {\it completely controllable}. 

If the manifolds rolling on each other are two-dimensional, then the controllability issue is well-understood thanks to the work of \cite{agrachev99}, \cite{bryant-hsu} and \cite{marigo-bicchi} especially. 
For instance, in the simply connected case, the rolling $(R)$ is completely controllable if and only if the manifolds are not isometric. In the case where the manifolds are isometric, \cite{agrachev99} also provides a description of the reachable 
sets in terms of isometries between the manifolds.
 
In particular, these reachable sets are immersed submanifolds of $Q$ of dimension either $2$ or $5$. In case the manifolds rolling on each other are isometric convex surfaces, \cite{marigo-bicchi} provides a beautiful description of a two dimensional reachable set: consider the initial configuration given by two (isometric) surfaces in contact so that one is the image of the other one by the symmetry with respect to the (common) tangent plane at the contact point. Then, this symmetry property (chirality) is preserved along the rolling $(R)$. Note that if the (isometric) convex surfaces are not spheres nor planes, the reachable set
starting at a contact point where the Gaussian curvatures are distinct, is open (and thus of dimension 5).
 
From a robotics point of view, once the controllability is well-understood, the next issue to address is that of {\it motion planning}, i.e., defining an effective procedure that produces, for every pair of points  ($q_{\mathrm{init}}$, $q_{\mathrm{final}}$) in the state space $Q$,  a curve $\gamma_{q_{\mathrm{init}},q_{\mathrm{final}}}$ so that the rolling of $S_1$ onto $S_2$ along $\gamma_{q_{\mathrm{init}},q_{\mathrm{final}}}$ steers the system from $q_{\mathrm{init}}$ to $q_{\mathrm{final}}$. In \cite{chelouah01}, an algorithm based on the continuation method 
was proposed to tackle the rolling problem $(R)$ of a strictly convex compact surface onto an Euclidean plane. That algorithm was also proved in \cite{chelouah01} to be convergent and it was numerically implemented in \cite{ACL} (see also \cite{mar-bic2} for another algorithm).

To the best of our knowledge, only the rolling $(R)$ was considered in the litterature, eventhough it is the more delicate, as explained below. The rolling problem $(R)$ is traditionally presented by isometrically embedding the rolling manifolds $M$ and $\hM$ in an Euclidean space  (cf. \cite{sharpe97}, \cite{huper07}) since it is the most intuitive way to provide a rigorous meaning to the notions of relative spin (or twist) and relative slip
of one manifold with respect to the other one. However, the rolling model will depend in general on the embedding. For instance, rolling two 2D spheres of different radii on each other can be isometrically embedded in (at least) two ways in $\R^3$: the smaller sphere can roll onto the bigger one either inside of it or outside. Then one should be able to define rolling without having to resort to any isometric embedding into an Euclidean space. To be satisfactory, that {\it intrinsic} formulation of the rolling should also allow one to address at least the controllability issue.

The first step towards an intrinsic formulation of the rolling starts 
with an intrinsic definition of  the state space $Q$. For $n\geq 3$, the relative orientation betwen two manifolds is deined (in coordinates) by an element of $SO(n)$. Therefore the state space $Q$ is of dimension $2n+n(n-1)/2$ since it is locally diffeomorphic to neighborhoods of 
$M\times\hM\times \SO(n)$.
There are two main approaches for an intrinsic formulation of the rolling problem $(R)$, first considered by \cite{agrachev99} and \cite{bryant-hsu} respectively. Note that the two references only deal with the two dimensional case but it is not hard to generalize them to higher dimensions. In \cite{agrachev99}, the state space $Q$ is given by
\[
Q=\{A:T|_x M\to T|_{\hat{x}} \hat{M}\ |\ A\ \textrm{o-isometry},\ x\in M,\ \hat{x}\in\hat{M}\},
\]
where ''o-isometry'' means positively oriented isometry, (see Definition \ref{defA} below) while in \cite{bryant-hsu}, one has equivalently
$$
Q=(F_{\mathrm{OON}}(M)\times F_{\mathrm{OON}}(\hat{M}))/\Delta,
$$
where $F_{\mathrm{OON}}(M)$, $F_{\mathrm{OON}}(\hat{M})$ be the oriented orthonormal frame bundles of $(M,g)$, $(\hat{M},\hat{g})$ respectively, and $\Delta$ is the diagonal right $\SO(n)$-action (see Proposition \ref{pr:char_of_Q} below) .

The next step towards an intrinsic formulation consists of using either the parallel transports with respect to $\nabla^g$ and $\nabla^{\hg}$ (Agrachev-Sachkov's approach) or alternatively,  
orthonormal moving frames and the structure equations (Bryant-Hsu's approach)
to translate the constraints of no-spinning and no-slipping and derive the admissible curves, i.e., the curves of $Q$ describing the rolling 
$(R)$, cf. Eq. (\ref{eq:cs_rolling}). Finally, one defines either a distribution or a codistribution depending which approach is chosen.  In the present paper, we adopt the Agrachev-Sachkov's approach and we construct an $n$-dimensional distribution
$\RDist$ on $Q$ so that the locally absolutely continuous curves tangent to $\RDist$ are exactly the admissible curves for the rolling problem, cf. Definition \ref{def:rdist}. The construction of $\RDist$ comes along with the construction of (local) basis of vector fields,
which allow one to compute the Lie algebraic structure associated to $\RDist$. 

One should mention the recent work \cite{norway} dealing with an intrinsic formulation of the rolling problem $(R)$ (see Definition 4 page 18 in the reference therein). However, that definition does not allow one to parameterize the admissible curves using a control system and a fortiori to construct a distribution (or a codistribution) associated to the rolling. Therefore, the computations in that paper related to controllability issues are all performed by embedding the rolling into an Euclidean space.

We now describe precisely the results of the present paper. In Section \ref{notations}, are gathered the notations used throughout the paper. After that, the control systems associated to the rolling problems $(NS)$ and $(R)$ are introduced in Section \ref{charac}. Besides the state space $Q$,  one must define the set of admissible controls. For $(NS)$, it is the set of locally absolutely continuous (l.a.c.) curves on $M\times \hM$ while, for $(R)$, it is the set of locally absolutely continuous (l.a.c.) curves on $M$ only. As control systems, we obtain two driftless control systems affine in the control $\sns$ and $\srol$ for $(NS)$ and $(R)$ respectively. We also provide, in Appendix \ref{app:local}, expressions in local coordinates for these control systems. 

The study of the rolling problem $(NS)$ is the objet of Section \ref{sec:2.0}. We first construct the distribution $\NSDist$ of rank $2n$ in $Q$ so that its tangent curves coincide with the admissible curves of $\sns$ and we provide (local) basis of vector fields for $\NSDist$. The controllability issue is completely addressed since we can describe exactly the reachable sets of $\sns$
in terms of  $H^{\nabla^g}$ and $\hat{H}^{\nabla^{\hg}}$, the holonomy groups of $\nabla^g$ and $\nabla^{\hg}$ respectively. We thus derive a necessary and sufficient condition for complete controllability of 
$(NS)$ in terms of the Lie algebras of $H^{\nabla^g}$ and $\hat{H}^{\nabla^{\hg}}$. For instance, if both manifolds $M$ and $\hM$ are simply connected and non symmetric, then the rolling problem $(NS)$ is completely controllable in dimension $n\neq 8$ if and only if  
$H^{\nabla^g}$ or $\hat{H}^{\nabla^{\hg}}$ is equal to 
$SO(n)$. We conclude that section by computing Lie brackets of vector fields tangent to $\NSDist$. 

In Section \ref{sec:2.55}, we start the study of the rolling problem $(R)$. As done for $(NS)$,
we construct the rolling distribution $\RDist$ as a sub-distribution of rank $n$ of $\NSDist$
so that its tangent curves coincide with the admissible curves of $\srol$ and we provide (local) basis of vector fields for $\RDist$. We show that  the rolling $(R)$ of $M$ over $\hM$ is  symmetric to that of $\hM$ over $M$ i.e., the reachable sets are diffeomorphic.
Already from these computations, one can see why we considered the rolling problem $(NS)$: from a technical point of view, it is much easier to
perform Lie brackets computations first with vector fields spanning $\NSDist$ and then specify these computations to vector fields spanning $\RDist$. Moreover, the complete controllability of $(NS)$ being a necessary condition for the  complete controllability of $(R)$, one can derive at once that, for simply connected and non symmetric rolling manifolds, if the rolling problem $(R)$ is completely controllable in dimension $n\neq 8$ then $H^{\nabla^g}$ or $\hat{H}^{\nabla^{\hg}}$ must be equal to 
$SO(n)$. 

The controllability issue for $(R)$ turns out to be much more delicate than that for $(NS)$. 
One reason is that, in general, there is no ''natural'' principal bundle structure on $\pi_{Q,M}:Q\to M$ which leaves invariant the rolling distribution $\RDist$.
Indeed, if it were the case, then all the reachable sets would be diffeomorphic and this is not true in general
(cf. the description of reachable sets of the rolling problem $(R)$ for two-dimensional isometric manifolds).
Despite this fact, we prove that the reachable sets are smooth bundles over $M$ (cf. Proposition \ref{pr:R_orbit_bundle}).

We also have an equivariance property of the reachable sets of  $\RDist$ with respect to the (global) isometries of the manifolds $M$ and $\hM$, as well as an interesting result linking the rolling problem $(R)$ for a pair of manifolds $M$ and $\hM$ and the rolling problem $(R)$ associated to Riemannian coverings of $M$ and $\hM$ respectively. As a consequence, we have that the complete controllability for the rolling problem $(R)$ associated to  a pair of manifolds $M$ and $\hM$ is equivalent to that of  the rolling problem $(R)$ associated to their universal Riemannian coverings. This implies that, as far as complete controllability is concerned, one can assume without loss of generality that $M$ and $\hM$ are simply connected. We then compute the first order Lie brackets of the vector fields generating $\RDist$ and find
that they are (essentially) equal to 
the vector fields given by the vertical lifts of 
\begin{equation}\label{rol00}
\Rol(X,Y)(A):=AR(X,Y)-\hat{R}(AX,AY)A,
\end{equation}
where $X,Y$ are smooth vector fields of $M$, $q=(x,\hat{x};A)\in Q$ and $R(\cdot,\cdot)$, $\hat{R}(\cdot,\cdot)$ are the curvature tensors of $g$ and $\hat{g}$ respectively. We call the vertical vector field
defined in Eq. (\ref{rol00}) the {\it Rolling Curvature}, cf Definition \ref{def-rol} below.
Higher order Lie brackets can now be expressed as linear combinations of covariant derivatives of the Rolling Curvature for the vertical part and evaluations on $\hM$
of the images of the Rolling Curvature and its covariant derivatives. 

In dimension two, the Rolling Curvature is (essentially) equal to $K^M(x)-K^{\hM}(\hat{x})$, where $K^M(\cdot)$, $K^{\hM}(\cdot)$ are the Gaussian curvatures of $M$ and $\hM$ respectively. At some point $q\in Q$ where $K^M(x)-K^{\hM}(\hat{x})\neq 0$, one immediately deduces that the dimension of the evaluation at $q$ of the  Lie algebra of the vector fields spanning  $\RDist$ is equal to five, (the dimension of $Q$) and thus the reachable set from $q$ is open in $Q$. From that fact, one has the following alternative: $(a)$ there exists $q_0\in Q$ so that $K^M-K^{\hM}\equiv 0$ over the reachable set from $q_0$, yielding easily that $M$ and $\hM$ have the same Riemannian covering space (cf. \cite{agrachev99} and \cite{bryant-hsu}); $(b)$ all the reachable sets are open and then the rolling problem $(R)$ is completely controllable. In dimension $n\geq 3$, the Rolling Curvature cannot be reduced to a scalar and it is seems difficult compute in general the rank of  the evaluations of the Lie algebra of the vector fields spanning  $\RDist$. 

We however propose several characterizations of isometry between two Riemannian manifolds based on the rolling perspective. The first one refers to a ''rolling against loops'' property which assumes that there is a $q_0=(x_0,\hat{x}_0;A_0)\in Q$
such that for every loop $\gamma$ on $M$ based at $x_0$, the corresponding rolling curve $\hat{\gamma}_{\RDist}(\gamma,q_0)$ on $\hM$ starting from $q_0$
is a loop based $\hat{x}_0$. Then we prove that, under the previous condition
$(M,g)$ and $(\hat{M},\hat{g})$ have the same universal Riemannian covering, cf. Theorem \ref{th:fixed_point}. 

The second characterization consists of revisiting the classical Ambrose theorem (see \cite{sakai91} Theorem III.5.1) and showing how the standard argument actually gets simplified when recast in the rolling context.
We also prove a version of the Cartan-Ambrose-Hicks theorem, Proposition \ref{pr:weak-C-A-H},
by using the rolling model. In this version, we also also include a condition for certain submersions to exist,
not only (local) geodesic embeddings.
Our proofs are in parallel to those presented in \cite{blumenthal89}, \cite{pawel02}.

In Section \ref{space-form}, we present controllability results when one of the manifolds, let say 
$(\hM,\hg)$, is a space form i.e., a simply connected complete Riemannian manifold of constant curvature.
Our results are actually preliminary and we hope to complete them in a future version of the present draft. 
Let us summarize them. The main feature of this particular case is that
there is a principal bundle structure on the bundle $\pi_{Q,M}:Q\to M$,
which is compatible with the rolling distribution $\RDist$.
In the case $\hat{M}$ has non-zero constant curvature,
this allows us to reduce the problem to a study of a vector bundle connection $\nabla^\Rol$
of the vector bundle $\pi_{TM\oplus\R}:TM\oplus\R\to M$ and its holonomy group,
which is a subgroup of $\SO(n+1)$ or $\SO(n,1)$ depending
whether the curvature of $\hat{M}$ is positive or negative, respectively. 
If $\hat{M}$ has zero curvature i.e., it is the Euclidean plane,
the problem reduces to the study of an affine connection and its holonomy group, a subgroup of $\Euc(n)$, in the sense of \cite{kobayashi63}.
In all the cases, the fibers over $M$ of the $\RDist$-orbits are all diffeomorphic
to the holonomy group of the connection in question.

In the zero curvature case, we prove that the rolling $(R)$ is completely controllable if and only if 
the (Riemannian) holonomy group of $\nabla^g$ is equal to $\SO(n)$. This result is actually similar to Theorem IV.7.1, p. 193 and Theorem IV.7.2, p. 194 in \cite{kobayashi63}. In the non-zero curvature case, we only study the rolling onto an $n$-dimensional sphere. We prove that if the holonomy group of the rolling connection 
$\nabla^{\Rol}$ is reducible, then the sphere endowed with the metric induced by the Euclidean metric of $\R^{n+1}$ must be a Riemannian covering space of $(M,g)$.

Section \ref{se:3D} collects our results for the rolling $(R)$ of three-dimensional Riemannian manifolds. We are able to provide a complete classification of the possible local structures of a non open orbit, and to each of them, to characterize precisely the manifolds $(M,g)$
and $(\hat{M},\hat{g})$ giving rise to such orbits. 

Roughly speaking, what we will prove is that the rolling problem $(R)$
\emph{is not} completely controllable i.e. $\mc{O}_{\RDist}(q_0)$ if and only if
the Riemannian manifolds $(M,g)$ and $(\hat{M},\hat{g})$
are \emph{locally} of the following types (i.e., in open dense sets):
\begin{itemize}
\item[(i)] isometric,
\item[(ii)] both are warped products with the same warping functions or
\item[(iii)] both are of class $\mc{M}_{\beta}$
with the same $\beta>0$.
\end{itemize}
Here, the manifolds of class $\mc{M}_{\beta}$ are defined as three-dimensional Riemannian manifolds carrying a contact structure of particular type, as described in \cite{agrachev10} and that we recall in Appendix \ref{app:3D}.
The possible values of the orbit dimension $d$ of a non open orbit $\mc{O}_{\RDist}(q_0)$ (i.e. $d=\dim\mc{O}_{\RDist}(q_0)$) are correspondingly in (i) $d=3$,
(ii) $d=6$ or $d=8$ where the latter corresponds to the case
where the initial orientation $A_0$ is "generic"
and finally (iii) we have $d=7$ or $d=8$ where again the latter case corresponds to a "generic" initial orientation $A_0$. 

Consequently, it follows that the possible orbit dimensions
for the rolling of 3D manifolds are
\[
\dim\mc{O}_{\RDist}(q_0)\in \{3,6,7,8,9\}
\]
where dimension $d=9$ corresponds to an open orbit (in $Q$).

We do not answer here to the question of global structure of $(M,g)$, $(\hat{M},\hat{g})$
when the rolling problem (R) is not completely controllable
and leave it to a future work.

In Section \ref{diff-dim}, we show how to extend the formalism developed previously to the case where the rolling manifolds have different dimensions. In that case, we show that the rolling of $M$ over $\hM$ is not anymore symmetric with that of $\hM$ over $M$, which is reasonable. We also provide basic controllability results. 

We finally gather in a series of appendices several results either used in the text or directly related to it. In particular, we show how the $\NSDist$ relates to the Sasaki-metric on the tensor space $T^*(M)\otimes T(\hat{M})$.
In the final appendix, we provide, for the sake of completeness, the classical formulation of the rolling problem $(R)$ as embedded in an Euclidean space.

{\bf Acknowledgements.} The authors want to thank P. Pansu and E. Falbel for helpful comments
as well as L. Rifford for having organized the conference "New Trends in Sub-Riemannian Geometry" in Nice
and where this work was first presented in April 2010.

\section{Notations}\label{notations}
For any sets $A,B,C$ and $U\subset A\times B$ and any map $F:U\to C$,
we write $U_a$ and $U^b$ for the sets defined by $\{b\in B\ |\ (a,b)\in U\}$
and $\{a\in A\ |\ (a,b)\in U\}$ respectively. Similarly, let $F_a:U_a\to C$ and $F^b:U^b\to C$
be defined by $F_a(b):=F(a,b)$ and $F^b(a):=F(a,b)$ respectively.
For any sets $V_1,\dots,V_n$ the map $\pr_i:V_1\times\dots\times V_n\to V_i$ denotes the projection onto 
the $i$-th factor. 

For a real matrix $A$, we use  $A^i_j$ to denote the real number on the $i$-th row and $j$-th column 
and the matrix $A$ can then be denoted by $[A^i_j]$.
If, for example, one has $A^i_j=a_{ij}$ for all $i,j$, then one uses the notation 
$A^i_j=(a_{ij})^i_j$ and thus $A=[(a_{ij})^i_j]$.
The matrix multiplication of $A=[A^i_j]$ and $B=[B^i_j]$ is therefore given by $AB=\big[\big(\sum_k A^i_k B^k_j\big)^i_j\big]$.

Suppose $V,W$ are finite dimensional $\R$-linear spaces, $L:V\to W$ is an $\R$-linear map and $F=(v_i)_{i=1}^{\dim V}$, $G=(w_i)_{i=1}^{\dim W}$
are bases of $V$, $W$ respectively. The $\dim W\times\dim V$-real matrix
corresponding to $L$ w.r.t. the bases $F$ and $G$ is denoted by $\mc{M}_{F,G}(L)$. In other words, $L(v_i)=\sum_j \mc{M}_{F,G}(L)_i^j w_j$
(corresponding to the right multiplication by a matrix of a row vector).
Notice that, if $K:W\to U$ is yet another $\R$-linear map to a finite dimensional linear space $U$ with basis $H=(u_i)_{i=1}^{\dim U}$, then 
\[
\mc{M}_{F,H}(K\circ L)=\mc{M}_{G,H}(K)\mc{M}_{F,G}(L).
\]
If $(V,g)$, $(W,h)$ are inner product spaces with inner products $g$ and $h$,
one defines $L^{T_{g,h}}:W\to V$ as the transpose (adjoint) of $A$ w.r.t $g$ and $h$ i.e., $g(L^{T_{g,h}}w,v)=h(w,Lv)$.
With bases $F$ and $G$ as above, one has 
$\mc{M}_{F,G}(L)^T=\mc{M}_{G,F}(L^{T_{g,h}})$,
where $T$ on the left is the usual transpose of a real matrix i.e., the transpose w.r.t standard Euclidean inner products in $\R^N$, $N\in\N$.

In this paper, by a smooth manifold, one means a smooth finite-dimensional, second countable, Hausdorff manifold
(see e.g. \cite{lee02}). 
A smooth manifold $N\subset M$ 
is an immersed submanifold of $M$ if the inclusion map $i:N\to M$
is a smooth immersion. We call $N$ embedded submanifold if the topology
on $N$ induced by the inclusion $i$ coincides with the manifold topology of $N$.
By a smooth submanifold of $M$, we always mean a smooth embedded submanifold.

A smooth bundle (over $M$) is a smooth map $\pi:E\to M$
between two smooth manifolds $E$ and $M$ together with a prescribed smooth manifold $F$
(unique up to diffeomorphism), called the typical fiber of $\pi$,
such that, for each $x\in M$, there is a neighbourhood $U$ of $x$ in $M$
and a smooth diffeomorphism $\tau:\pi^{-1}(U)\to U\times F$
with the property that $\pr_1\circ\tau=\pi|_{\pi^{-1}(U)}$. Such maps $\tau$ 
are called (smooth) local trivializations of $\pi$.

For any smooth map $\pi:E\to M$ between smooth manifolds $E$ and $M$,
 the set $\pi^{-1}(\{x\})=:\pi^{-1}(x)$
is called the $\pi$-fiber over $x$ and it is sometimes
denoted by $E|_x$, when $\pi$ is clear from the context.
A smooth section of a smooth map $\pi:E\to M$ is a smooth map $s:M\to E$
such that $\pi\circ s=\id_{M}$. The set of smooth sections of $\pi$
is denoted by $\Gamma(\pi)$. Local sections
of $\pi$ are sections defined only on open (possibly proper) subsets of $M$. The value $s(x)$ of a section $s$ at $x$ is
usually denoted by $s|_x$.

A smooth manifold $M$ is oriented if there exists a smooth (or continuous)
section, defined on all of $M$, of the bundle of $n$-forms $\pi_{\bigwedge^n(M)}:\bigwedge^n(M)\to M$
where $n=\dim M$. Otherwise mentioned, the smooth manifolds considered in this paper 
are connected and oriented.

A smooth vector bundle is a smooth bundle where the typical fiber $F$ is a finite dimensional $\R$-linear space together with a collection of local trivializations so that 
there is a well defined vector space structure on each $\pi$-fiber 
(see \cite{lee02} for the precise definition).
Some important vector bundles for us over a manifold $M$ are the tangent bundle $\pi_{T(M)}:T(M)\to M$
and different $(k,m)$-tensor bundles $\pi_{T^k_m(N)}:T^k_m(M)\to M$.
We will many times write $TM:=T(M)$ etc. to ease the notation.

If $G$ is a smooth Lie-group, a smooth bundle $\pi:E\to M$ is called a right principal $G$-bundle
if there exists a smooth right action $\mu:E\times G\to E$ of $G$ on $E$
(i.e., $\mu(\mu(y,g),h)=\mu(y,gh)$ where the product $gh$ is computed in $G$)
such that $\pi\circ\mu=\pi\circ\pr_1$ and $\mu$ is free (i.e., $\mu(y,g)=y$ for a $y\in E$ implies $g=e$ the identity of $G$)
and transitive on $\pi$-fibers (i.e., for every $y,z\in \pi^{-1}(x)$ there is a $g\in G$ such that $\mu(y,g)=z$).
It follows from the definition that this bundle has $G$ as the typical fiber.
Similarly, using a left action one defines a left principal $G$-bundle.
For short, by a principal bundle we mean a left or right principal bundle (the side of the action being clear
from context).
There is no difference between left and right principal bundles since
a right principal bundle $\pi_E$ with action $\mu:E\times G\to E$
can be identified with a left principal bundle $\pi_E$ with action $\lambda:G\times E\to E$;
$\lambda(g,y)=\mu(y,g^{-1})$ and vice versa.

For a smooth map $\pi:E\to M$ and $y\in E$, let $V|_y(\pi)$ be the set of all $Y\in T|_y E$ such that $\pi_*(Y)=0$.
If $\pi$ is a smooth bundle, the collection of spaces $V|_y(\pi)$, $y\in E$,
defines a smooth submanifold $V(\pi)$ of $T(E)$
and the restriction $\pi_{T(E)}:T(E)\to E$ to $V(\pi)$ is denoted by $\pi_{V(\pi)}$.
In this case $\pi_{V(\pi)}$ is a vector subbundle of $\pi_{T(E)}$ over $E$.

For a smooth manifold $M$, one uses $\VF(M)$ to denote the set of smooth vector fields on $M$
i.e., the set of smooth sections of the tangent bundle $\pi_{T(M)}:T(M)\to M$.
The flow of a vector field $Y\in \VF(M)$
is a smooth onto map $\Phi_Y:D\to M$ defined on an open subset $D$
of $\R\times M$ containing $\{0\}\times M$ such that $\frac{\partial}{\partial t}\Phi_Y(t,y)=Y|_{\Phi_Y(t,y)}$
for $(t,y)\in D$ and $\Phi_Y(0,y)=y$ for all $y\in M$.
As a default, we will take $D$ to be the maximal flow domain of $X$.

A subset $\mc{D}\subset T(M)$ of the tangent bundle of $M$
is called a smooth distribution on $M$ if
$\pi_{T(M)}|_{\mc{D}}$ is a smooth vector subbundle of $\pi_{T(M)}$ over $M$.
For $x\in M$, the fiber $\pi_{T(M)}|_{\mc{D}}^{-1}(\{x\})$ is denoted by $\mc{D}|_x$ and
the common dimension of the spaces $\mc{D}|_x$, $x\in M$, is called the rank of the distribution $\mc{D}$.

For any distribution $\mc{D}$ on a manifold $M$,  we use $\VF_{\mc{D}}$  to denote the set of vector fields $X\in\VF(M)$
tangent to $\mc{D}$ (i.e., $X|_x\in \mc{D}|_x$ for all $x\in M$)
and we define inductively for $k\geq 2$
\[
\VF_{\mc{D}}^k=\VF^{k-1}_{\mc{D}}+[\VF_{\mc{D}},\VF^{k-1}_{\mc{D}}],
\]
where $\VF^{1}_{\mc{D}}:=\VF_{\mc{D}}$.
The Lie algebra generated by $\VF_{\mc{D}}$
is denoted by $\Lie(\mc{D})$ and it equals $\bigcup_{k} \VF_{\mc{D}}^k$.

For any maps $\gamma:[a,b]\to X$, $\omega:[c,d]\to X$
into a set $X$ such that $\gamma(b)=\omega(c)$ we define
\[
\omega\sqcup\gamma:[a,b+d-c]\to X;
\quad
(\omega\sqcup\gamma)(t)
=\begin{cases}
\gamma(t), & t\in [a,b] \\
\omega(t-b+c), & t\in [b,b+d-c].
\end{cases}
\]
A map $\gamma:[a,b]\to X$ is a loop in $X$ based at $x_0\in X$ if $\gamma(a)=\gamma(b)=x_0$.
In the space of loops $[0,1]\to X$ based at some given point $x_0$,
one defines a group operation ''$.$'', concatenation,
by
\[
\omega.\gamma:=(t\mapsto \gamma(\sfrac{t}{2}))\sqcup (t\mapsto \omega(\sfrac{t}{2})).
\]
This operation gives a group structure on the set of loops of $X$ based at a given point $x_0$.
If $N$ is a smooth manifold and $y\in N$, we use $\Omega_y(N)$ to denote
the set of  all piecewise $C^1$-loops $[0,1]\to N$ of $N$ based at $y$.
In particular, $(\Omega_y(N),.)$ is a group.

A continuous map $c:I\to M$ from a real compact interval $I$ into a smooth manifold $M$
is called \emph{absolutely continuous}, or \emph{a.c.} for short
if, for every $t_0\in I$, there is a smooth coordinate chart $(\phi,U)$
of $M$ such that $c(t_0)\in U$ and $\phi\circ c|_{c^{-1}(U)}$ is absolutely continuous.

Given a smooth distribution $\mc{D}$
on a smooth manifold $M$, we call an absolutely continuous curve $c:I\to M$, $I\subset\R$,
$\mc{D}$-\emph{admissible} if $c$ it is tangent to $\mc{D}$ almost everywhere (a.e.)
i.e., if for almost all $t\in I$ it holds that $\dot{c}(t)\in \mc{D}|_{c(t)}$.
For $x_0\in M$, the endpoints of all the $\mc{D}$-admissible curves of $M$
starting at $x_0$ form the set called \emph{$\mc{D}$-orbit through $x_0$} and denoted $\mc{O}_{\mc{D}}(x_0)$.
More precisely,
\begin{align}\label{eq:orbit}
\mc{O}_{\mc{D}}(x_0)=\{c(1)\ |\ c:[0,1]\to M,\ \mc{D}\mathrm{-admissible},\ c(0)=x_0\}.
\end{align}
By the Orbit Theorem (see \cite{agrachev04}), it follows that $\mc{O}_{\mc{D}}(x_0)$ is an immersed
smooth submanifold of $M$ containing $x_0$.
It is also known that one may restrict to piecewise smooth curves in the
description of the orbit i.e.,
\[
\mc{O}_{\mc{D}}(x_0)=\{c(1)\ |\ c:[0,1]\to M\ \textrm{piecewise smooth and}\ \mc{D}\mathrm{-admissible},\ c(0)=x_0\}.
\] 
We call a smooth distribution $\mc{D}'$ on $M$ a subdistribution of $\mc{D}$
if $\mc{D}'\subset \mc{D}$. An immediate consequence
of the definition of the orbit shows that in this case
\[
\mc{O}_{\mc{D}'}(x_0)\subset \mc{O}_{\mc{D}}(x_0),\quad \forall x_0\in M.
\]

If $\pi:E\to M$, $\eta:F\to M$ are two smooth maps (e.g. bundles), let $C^\infty(\pi,\eta)$ be the set
of all bundle maps $\pi\to \eta$ i.e., smooth maps $g:E\to F$ such that $\eta\circ g=\pi$.
For a manifold $M$, let $\pi_{M_{\R}}:M\times \R\to M$  be the projection onto the first factor i.e., $(x,t)\mapsto x$ (i.e., $\pi_{M_{\R}}=\pr_1$).
Recall that there is a canonical bijection between the set $C^\infty(M)$ of smooth functions on $M$
and the set $C^\infty(\id_M,\pi_{M_{\R}})$ given by
$f\mapsto f_{\R}:=(x\mapsto (x,f(x)))$.

If $\pi:E\to M$, $\eta:F\to M$ are any smooth vector bundles over a
smooth manifold $M$, $f\in C^\infty(\pi,\eta)$ and $u,w\in \pi^{-1}(x)$,
one defines the vertical derivative $f$ at $u$ in the direction $w$ by
\begin{align}\label{eq:vert_diff_bundle_map}
\nu(w)|_u(f):=(D_\nu f)(u)(w):=\dif{t}\big|_0 f(u+tw).
\end{align}
Here $w\mapsto (D_\nu f)(u)(w)=\nu(w)|_u(f)$ is an $\R$-linear map between fibers $\pi^{-1}(x)\to \eta^{-1}(x)$.

In a similar way, in the case of $f\in\Cinf(E)$ and $u,w\in \pi^{-1}(x)$,
one defines the $\pi$-vertical derivative $\nu(w)|_u(f):=D_\nu f(u)(w):=\dif{t}|_0 f(u+tw)$ at $u$ in the direction $w$. This definition agrees with the above
one modulo the canonical bijection $\Cinf(E)\cong \Cinf(\id_E,\pi_{E_{\R}})$.  
This latter definition means that $\nu(w)|_u$ can be viewed as an element of $V|_u(\pi)$
and the mapping $w\mapsto \nu(w)|_u$ gives a (natural) $\R$-linear isomorphism between $\pi^{-1}(x)$
and $V|_u(\pi)$ where $\pi(u)=x$.
If $\tilde{u}\in\Gamma(\pi)$ is a smooth $\pi$-section,
let  $\nu(\tilde{w})$ be the $\pi$-vertical vector field on $E$ defined by
$\nu(\tilde{w})|_u(f)=\nu(\tilde{w}|_x)|_u(f)$, where $\pi(u)=x$ and $f\in\Cinf(E)$.
The same remark holds also locally.

In the case of smooth manifolds $M$ and $\hat{M}$, $x\in M$, $\hat{x}\in\hat{M}$,
we will use freely and without mention the natural inclusions ($\subset$) and isomorphisms ($\cong$): $T|_x M,T|_{\hat{x}}\hat{M}\subset T|_{(x,\hat{x})}(M\times \hat{M})\cong T|_x M\oplus T|_{\hat{x}} \hat{M}$, $T^*|_x M,T^*|_{\hat{x}} \hat{M}\subset T^*|_{(x,\hat{x})} (M\times \hat{M})\cong T^*|_x M\oplus T^*|_{\hat{x}} \hat{M}$.
An element of $T|_{(x,\hat{x})} (M\times \hat{M})\cong T|_{x}(M)\oplus T|_{\hat{x}}(\hat{M})$ with respect to the direct sum
splitting is denoted usually by $(X, \hat{X})$, where $X\in T|_x M$, $\hat{X}\in T|_{\hat{x}} \hat{M}$.
Sometimes it is even more convenient to write
$X+\hat{X}:=(X,\hat{X})$
when we make the identifications $(X, 0)=X$, $(0, \hat{X})=\hat{X}$.

Let $(M,g)$, $(\hat{M},\hat{g})$ be smooth Riemannian manifolds.
A map $f:M\to \hat{M}$ is a \emph{local isometry} if it is smooth, surjective
and for all $x\in M$, $f_*|_x:T|_x M\to T|_{f(x)}\hat{M}$ is an isometric linear map.
A bijective local isometry $f:M\to\hat{M}$ is called an \emph{isometry}
and then $(M,g)$, $(\hat{M},\hat{g})$ are said to be \emph{isometric}.

In this text we say that two Riemannian manifolds $(M,g)$, $(\hat{M},\hat{g})$ are \emph{locally isometric},
if there is a Riemannian manifold $(N,h)$ and local isometries $F:N\to M$ and $G:N\to \hat{M}$
which are also covering maps i.e. if they are \emph{Riemannian covering maps}.
One calls $(N,h)$ a common Riemannian covering space of $(M,g)$ and $(\hat{M},\hat{g})$.
Notice that being locally isometric is an equivalence relation in the class of smooth Riemannian manifolds
(the fact that we assume $F,G$ to be Riemannian covering maps, and not only local isometries,
implies the transitivity of this relation).

The space $\ol{M}=M\times \hat{M}$ is a Riemannian manifold, called the Riemannian product manifold of $(M,g)$, $(\hat{M},\hat{g})$,
when endowed with the product metric $\ol{g}:=g\oplus\hat{g}$.
One often writes this as $(M,g)\times (\hat{M},\hat{g})$.

Let $\nabla$, $\hat{\nabla},\ol{\nabla}$ (resp. $R,\hat{R},\ol{R}$) denote the Levi-Civita connections
(resp. the Riemannian curvature tensors)
of $(M,g)$, $(\hat{M},\hat{g}),(\ol{M}=M\times\hat{M},\ol{g}=g\oplus\hat{g})$ respectively.
From Koszul's formula (cf. \cite{lee02}), one has
\begin{align}
\ol{\nabla}_{(X, \hat{X})} (Y, \hat{Y})=(\nabla_{X} Y, \hat{\nabla}_{\hat{X}} \hat{Y}),
\end{align}
when $X,Y\in \VF(M)$, $\hat{X},\hat{Y}\in \mathrm{VF}(\hat{M})$
and hence from the definition of the Riemannian curvature tensor
\begin{align}
\ol{R}((X, \hat{X}),(Y, \hat{Y}))(Z, \hat{Z})=(R(X,Y)Z, \hat{R}(\hat{X},\hat{Y})\hat{Z}),
\end{align}
where $X,Y,Z\in T|_x M$, $\hat{X},\hat{Y},\hat{Z}\in T|_{\hat{x}}\hat{M}$.

For any $(k,m)$-tensor field $T$ on $M$
we define $\nabla T$ to be the $(k,m+1)$-tensor field
such that (see \cite{sakai91}, p. 30)
\begin{align}\label{eq:tensor_extension}
(\nabla T)(X_1,\dots,X_m,X)=(\nabla_X T)(X_1,\dots,X_m),
\end{align}
$X_1,\dots,X_m,X\in T|_x M$.

Let $x:I\to M$ and $X:I\to TM$ be a smooth curve and a smooth vector field 
along $x$ respectively i.e., a smooth map such that $X(t)\in T|_{x(t)} M$ for all $t\in I$.
A {\it local extension} of $X$ around $t_0$ is a vector field $\tilde{X}\in\VF(M)$ such that  there is an open interval $J$ with $t_0\in J\subset I$ and $\tilde{X}|_{x(t)}=X(t)$ for all $t\in J$.
Then one defines $\nabla_{\dot{x}(t_0)} X$ as $\nabla_{\dot{x}(t_0)} \tilde{X}$ and
it is easily seen that this vector does not depend on the choice of a local extension of $X$ around $t_0$.
The same construction holds true for tensor fields along the path $x(\cdot)$.

The parallel transport of a tensor $T_0\in T^k_m|_{x(0)}(M)$ from $x(0)$ to $x(t)$ along an absolutely continuous curve $x:I\to M$ (with $0\in I$) and with respect to the Levi-Civita connection of $(M,g)$
is denoted by $(P^{\nabla^g})^t_0(x) T_0$. In the notation of the Levi-Civita connection $\nabla^g$ (resp. parallel transport $P^{\nabla^g}$), the upper index $g$
(resp. $\nabla^g$) referring to the Riemannian metric $g$ (resp. the connection $\nabla^g$) is omitted if it is clear from the context. 
We also recall the following basic observation.

\begin{proposition}\label{pr:parallel_tensor}
Let $(M,g)$ be a smooth Riemannian manifold and $t\mapsto x(t)$ an absolutely continuous (a.c. for short) curve on $M$ defined on an open interval $I\ni 0$.
Then the parallel transport $T(t)=(P^{\nabla^g})^t_0(x) T_0$ along $t\mapsto x(t)$ w.r.t $g$ of any $(k,m)$-tensor $T_0\in T^k_m|_{x(0)}(M)$ uniquely exists and is absolutely continuous.
\end{proposition}

Let $(x,\hat{x}):I\to M\times\hat{M}$ be a smooth curve on $M\times\hat{M}$ defined on an open real interval $I$ containing $0$.
If $(X(t), \hat{X}(t)):I\to T(M\times\hat{M})$ is a smooth vector field on $M\times\hat{M}$ along $(x,\hat{x})$ i.e.,
$(X(t), \hat{X}(t))\in T|_{(x(t),\hat{x}(t))} (M\times\hat{M})$ then one has
\begin{align}\label{eq:product_connection}
\ol{\nabla}_{(\dot{x}(t),\dot{\hat{x}}(t))} (X, \hat{X})=(\nabla_{\dot{x}(t)} X, \hat{\nabla}_{\dot{\hat{x}}(t)} \hat{X})
\end{align}
only if the covariant derivatives on the right-hand side are well defined (see the next remark). 

\begin{remark}\label{re:product_connection}
Let $M=\R$, $\hat{M}=\R$ and $(c(t),\hat{c}(t))=(t,0)$, $(X(t), \hat{X}(t))=(1, t)$
and equip $M$ and $\hat{M}$ with the Euclidean metrics: $g(Y,Z)=YZ$,
$\hat{g}(\hat{Y},\hat{Z})=\hat{Y}\hat{Z}$.
Then the left hand side of (\ref{eq:product_connection}) is defined and equals $(0, 1)$
but on the right hand side the covariant derivative $\hat{\nabla}_{\dot{\hat{c}}(t)} \hat{X}=\hat{\nabla}_{0} t$
is not defined: if $\hat{Y}\in \VF(\hat{M})$ were a local extension of $\hat{X}$ around $t=0$
then $t=\hat{X}(t)=\hat{Y}|_{\hat{c}(t)}=\hat{Y}(0)$ for all $t$ in some open interval containing $0$.
This is a contradiction.
Note that an extension of $(X(t), \hat{X}(t))=(1, t)$ around $t=0$
is provided for example by $(x,\hat{x})\mapsto (1, x)$.
\end{remark}

\newcommand{\Iso}{\mathrm{Iso}}

If $(N,h)$ is a Riemannian manifold we
define $\Iso(N,h)$ to be the (smooth Lie) group of isometries of $(N,h)$
i.e., the set of diffeomorphisms $F:N\to N$ such that
$F_*|_y:T|_y N\to T|_{F(y)} N$ is an isometry for all $y\in N$,
cf. \cite{sakai91}, Lemma III.6.4, p. 118.

It is clear that the isometries respect parallel transport in the sense that
for any absolutely continuous $\gamma:[a,b]\to N$ and $F\in\Iso(N,g)$ one has
(cf. \cite{sakai91}, p. 41, Eq. (3.5))
\begin{align}\label{eq:iso_parallel}
F_*|_{\gamma(t)}\circ (P^{\nabla^h})^t_a(\gamma)=(P^{\nabla^h})^t_a(F\circ\gamma)\circ F_*|_{\gamma(a)}.
\end{align}

The following result is standard.

\begin{theorem}
Let $(N,h)$ be a Riemannian manifold
and for any absolutely continuous $\gamma:[0,1]\to M$, $\gamma(0)=y_0$,
define
\[
\Lambda^{\nabla^h}_{y_0}(\gamma)(t)=
\int_0^t (P^{\nabla^h})^0_s(\gamma)\dot{\gamma}(s)\diff s\in T|_{y_0} N,
\quad t\in [0,1].
\]
Then the map $\Lambda^{\nabla^h}_{y_0}:\gamma\mapsto \Lambda^{\nabla^h}_{y_0}(\gamma)(\cdot)$
is an injection from the set of absolutely continuous curves $[0,1]\to N$ starting at $y_0$
onto an open subset of the Banach space of absolutely continuous curves $[0,1]\to T|_{y_0} N$
starting at $0$.

Moreover, the map $\Lambda^{\nabla^h}_{y_0}$ is a bijection onto the latter Banach space if (and only if) $(N,h)$ is a complete Riemannian manifold.
\end{theorem}

\begin{remark}
\begin{itemize}
\item[(i)] For example, in the case where $\gamma$ is the geodesic
$t\mapsto \exp_{y_0} (tY)$ for $Y\in T|_{y_0} N$,
one has
\[
\Lambda^{\nabla^h}_{y_0}(\gamma)(t)=tY.
\]

\item[(ii)] It is directly seen from the definition of $\Lambda^{\nabla^h}_{y_0}$
that it maps injectively (piecewise) $C^k$-curves, $k=1,\dots,\infty$,
starting at $y_0$ to (piecewise) $C^k$-curves starting at $0$.
Moreover, these correspondences are bijective if $(N,h)$ is complete. 

\item[(iii)] The map $\Lambda_{y_0}^{\nabla^h}$
could be used to give the space of absolutely continuous curves $[0,1]\to N$
starting at $y_0$ a structure of a Banach space if $(N,h)$ is complete
or an open subset of a Banach space in the case $(N,h)$ is not complete. 

\end{itemize}
\end{remark}

\section{State Space, Distributions and Computational Tools}\label{charac}

\subsection{State Space}
\subsubsection{Definition of the state space}
After \cite{agrachev99}, \cite{agrachev04} we make the following definition.

\begin{definition}\label{defA}
The \emph{state space $Q=Q(M,\hat{M})$} for the rolling of two $n$-dimensional
\emph{connected, oriented} smooth Riemannian manifolds
$(M,g),(\hat{M},\hat{g})$ is defined as
\[
Q=\{A:T|_x M\to T|_{\hat{x}} \hat{M}\ |\ A\ \textrm{o-isometry},\ x\in M,\ \hat{x}\in\hat{M}\},
\]
where ``o-isometry'' stands for ``orientation preserving isometry''
i.e., if $(X_i)_{i=1}^n$ is a positively oriented $g$-orthonormal frame of $M$ at $x$ then $(AX_i)_{i=1}^n$
is a positively oriented $\hat{g}$-orthonormal frame of $\hat{M}$ at $\hat{x}$.
\end{definition}

The linear space of $\R$-linear map $A:T|_x M\to T|_{\hat{x}} \hat{M}$ 
is canonically isomorphic to the tensor product $T^*|_x M\otimes T|_{\hat{x}}\hat{M}$.
On the other hand, by using the canonical inclusions $T^*|_x M\subset T^*|_{(x,\hat{x})}(M\times\hat{M})$,
$T|_{\hat{x}}\hat{M}\subset T|_{(x,\hat{x})}(M\times\hat{M})$,
the space $T^*|_x M\otimes T|_{\hat{x}}\hat{M}$ is canonically included
in the space $T^1_1(M\times\hat{M})|_{(x,\hat{x})}$ of $(1,1)$-tensors of $M\times\hat{M}$
at $(x,\hat{x})$. These inclusions make
$T^*M\otimes T\hat{M}:=\bigcup_{(x,\hat{x})\in M\times\hat{M}} T^*|_x M\otimes T|_{\hat{x}}\hat{M}$
a subset of $T^1_1(M\times\hat{M})$
such that $\pi_{T^*M\otimes T\hat{M}}:=\pi_{T^1_1(M\times\hat{M})}|_{T^*M\otimes T\hat{M}}:T^*M\otimes T\hat{M}\to M\times\hat{M}$
is a smooth vector subbundle of the bundle of $(1,1)$-tensors $\pi_{T^1_1(M\times\hat{M})}$
on $M\times\hat{M}$.

The state space $Q=Q(M,\hat{M})$ can now be described as a subset of $T^*M\otimes T\hat{M}$
as
\[
Q=\{A\in T^*(M)\otimes T(\hat{M})|_{(x,\hat{x})}\ |\ & (x,\hat{x})\in M\times\hat{M},\\
& \n{AX}_{\hat{g}}=\n{X}_{g},\ \forall X\in T|_x M,\ \det(A)=1\}.
\]
In the next subsection, we will show that $\pi_Q:=\pi_{T^*M\otimes T\hat{M}}|_Q$ is moreover
a smooth subbundle of $\pi_{T^*M\otimes T\hat{M}}$ though it is not a vector subbundle.

It is also convenient to consider the manifold $T^*M\otimes T\hat{M}$ 
and we will refer to it as the \emph{extended state space} for the rolling.
This concept of extended state space naturally makes sense also in the case where $M$ and $\hat{M}$ are not assumed to be oriented (or connected).

A point $A\in T^*M\otimes T\hat{M}$ with $\pi_{T^*M\otimes T\hat{M}}(A)=(x,\hat{x})$
(or $A\in Q$ with $\pi_Q(A)=(x,\hat{x})$) will
be sometimes denoted by $(x,\hat{x};A)$ to emphasize the fact that $A:T|_x M\to T|_{\hat{x}} \hat{M}$.
Thus the notation $q=(x,\hat{x};A)$ simply means that $q=A$.

%%%%%%%%%%%%%%%%%%%%%%%%%%%%%%%%%%%%%%%%%%%%
\subsubsection{The Bundle Structure of $Q$}\label{sec:1.1}
%%%%%%%%%%%%%%%%%%%%%%%%%%%%%%%%%%%%%%%%%%%%

In this subsection, it is shown that $\pi_Q$ is a bundle with typical fiber $\SO(n)$.
We will also argue that, even though $\SO(n)$ is a Lie-group, the bundle $\pi_Q$
cannot in general be given a natural (or useful) $\SO(n)$-principal bundle if $n> 2$
(see also Theorem \ref{th:no_useful_principal_bundle_for_pi_Q}).
We will now present the local trivializations of $\pi_Q$.

\begin{definition}
Suppose the vector fields $X_i\in\VF(M)$ (resp. $\hat{X}_i\in\VF(\hat{M})$), $i=1,\dots,n$ form a $g$-orthonormal
(resp. $\hat{g}$-orthonormal)
frame of vector fields on an open subset $U$ of $M$ (resp. $\hat{U}$ of $\hat{M}$). We denote $F=(X_i)_{i=1}^n$, $\hat{F}=(\hat{X}_i)_{i=1}^n$ and for $x\in U$, $\hat{x}\in \hat{U}$ we let $F|_x=(X_i|_x)_{i=1}^n$,
$\hat{F}|_{\hat{x}}=(\hat{X}_i|_{\hat{x}})_{i=1}^n$
Then a local trivialization $\tau=\tau_{F,\hat{F}}$ of $Q$ over $U\times \hat{U}$ \emph{induced by $F,\hat{F}$} is given by
\begin{align}
\tau:\pi_Q^{-1}(U\times\hat{U})&\to (U\times\hat{U})\times \SO(n)\nonumber\\
(x,\hat{x};A)&\mapsto \big((x,\hat{x}), \mc{M}_{F|_x,\hat{F}|_{\hat{x}}}(A)\big),\nonumber
\end{align}
where $\mc{M}_{F|_x,\hat{F}|_{\hat{x}}}(A)^j_i=\hat{g}(AX_i,\hat{X_j})$
since $AX_i|_x=\sum_j \hat{g}(AX_i|_x,\hat{X}_j|_{\hat{x}})\hat{X}_j|_{\hat{x}}$.
\end{definition}

For the sake of clarity, we shall write $\mc{M}_{F|_x,\hat{F}|_{\hat{x}}}(A)$ as $\mc{M}_{F,\hat{F}}(A)$.
Obviously $\n{AX}_{\hat{g}}=\n{X}_g$ for all $X\in T|_x M$ is equivalent to $A^{T_{g,\hat{g}}} A=\id_{T|_x M}$ and thus
we get
\[
\mc{M}_{F,\hat{F}}(A)^T\mc{M}_{F,\hat{F}}(A)=\mc{M}_{\hat{F},F}(A^{T_{g,\hat{g}}})\mc{M}_{F,\hat{F}}(A)
=\mc{M}_{F,F}(\id_{T|_x M})=\id_{\R^n},
\]
where $T$ denotes the usual transpose in $\gl(n)$, the set of Lie algebra of $n\times n$-real matrices. 
Since $\det \mc{M}_{F,\hat{F}}(A)=\det(A)=+1$, one finally has $\mc{M}_{F,\hat{F}}(A)\in\SO(n)$.

\begin{remark}
Notice that the above local trivializations $\tau_{F,\hat{F}}$ of $\pi_Q$ are just the restrictions of the vector bundle local trivializations
$$(\pi_{T^*(M)\otimes T(\hat{M})})^{-1} (U\times\hat{U})\to (U\times\hat{U})\times \gl(n)$$
of the bundle $\pi_{T^*(M)\otimes T(\hat{M})}$ induced by $F,\hat{F}$ and defined by the same formula as $\tau_{F,\hat{F}}$.
In this setting, one does not even have to assume that the local frames $F$, $\hat{F}$ are $g$- or $\hat{g}$-orthonormal.
Hence $\pi_Q$ is a smooth subbundle of $\pi_{T^*M\otimes T\hat{M}}$
with $Q$ a smooth submanifold of $T^*M\otimes T\hat{M}$.
\end{remark}

We next  spell out the transition functions of the above defined local trivializations of $\pi_Q$
(and also of $\pi_{T^*M\otimes T\hat{M}}$ by the above remark).
If $F'=((X'_i),U')$,$\hat{F}'=((\hat{X}'_i),\hat{U}')$ are other  $g$-, $\hat{g}$-orthonormal frames (with $U\cap U'\neq\emptyset$, $\hat{U}\cap \hat{U}'\neq\emptyset$)
and $A=[A^j_i]\in \SO(n)$, then
\bean
&& (\tau_{F',\hat{F}'}\circ \tau_{F,\hat{F}}^{-1})((x,\hat{x}),A)
=\tau_{F',\hat{F}'}\big(x,\hat{x}; \sum_{i,j} A^j_i g(X_i,\cdot)\hat{X}_j \big) \\
&=&\Big((x,\hat{x}),\Big[\Big(\sum_{i,j} A^j_i g( X_i,X'_k)\hat{g}(\hat{X}_j,\hat{X}'_l))\Big)^l_k\Big]\Big)\\
&=&\big((x,\hat{x}), [(\hat{g}( \hat{X}_j,\hat{X}'_l))_j^l] A [(g( X_i,X'_k))^k_i]^T \big) \\
&=&\big((x,\hat{x}), \mc{M}_{\hat{F},\hat{F}'}(\id_{T|_{\hat{x}} \hat{M}}) A \mc{M}_{F,F'}(\id_{T|_x M})^T \big)
\eean
for $x\in U\cap U'$, $\hat{x}\in \hat{U}\cap \hat{U}'$.

Any local trivialization $\tau:\pi_Q^{-1}(\ol{U})\to \ol{U}\times\SO(n)$  of $\pi_Q$
defined on an open set $\ol{U}\subset M\times\hat{M}$
would define a principal $\SO(n)$-bundle structure on $\pi_Q^{-1}(\ol{U})$
(or rather for $\pi_Q|_{\pi_Q^{-1}(\ol{U})}$)
by the formula (see \cite{spivakII99}, p. 307)
\begin{align}\label{eq:principal_action}
\mu((x,\hat{x};A),B)=\tau^{-1}((x,\hat{x}),(\pr_2\circ \tau)(x,\hat{x};A)B),
\end{align}
with $\mu:\pi_Q^{-1}(\ol{U})\times \SO(n)\to \pi_Q^{-1}(\ol{U})$
the right $\SO(n)$-action of this principal bundle structure.
However, we will show that if we take for the local trivializations $\tau$
the ones induced by local orthonormal frames $\tau=\tau_{F,\hat{F}}$
as above, then the (local) actions $\mu_{F,\hat{F}}$ defined by the above formula
by these different local trivializations $\tau_{F,\hat{F}}$ do not
glue up to form a global principal bundle structure for $\pi_Q$
if the dimension $n$ of $M$ and $\hat{M}$ is greater than $2$.
We state this in the following proposition.

\begin{proposition}
The local actions (\ref{eq:principal_action})
do not render the bundle $\pi_Q$ to a principal $\SO(n)$-bundle except when $n\leq 2$.
\end{proposition}

\begin{proof}
If $\pi_Q$ were a principal $\SO(n)$-bundle w.r.t local trivializations
induced by the orthonormal frames of $M$ and $\hat{M}$, then the right action $\mu:Q\times \SO(n)\to Q$ of $\SO(n)$ on $Q$
of this principal bundle structure would be given (locally) by (see above)
\[
\mu((x,\hat{x};A),B)=\tau^{-1}((x,\hat{x}),(\pr_2\circ \tau)(x,\hat{x};A)B),
\]
for any of the local trivializations $\tau=\tau_{F,\hat{F}}$ induced by orthonormal local frames $F,\hat{F}$ of $M$, $\hat{M}$ and any $(x,\hat{x};A)$ with $x$, $\hat{x}$
in these domains and any $B\in \SO(n)$.
Equivalently, the above condition could be written as $$(\pr_2\circ \tau)(\mu(q,B))=(\pr_2\circ \tau)(q)B,$$
for any $q\in Q$ in the domain of definition of $\tau$ and $B\in\SO(n)$.

The formula for the transition maps of these local trivializations as expressed before this proposition
shows that the action $\mu$ is not well defined if $n\geq 3$.
In fact we would be led to an equation of the type
\[
(\pr_2\circ\tau_{F',\hat{F}'})(x,\hat{x};A)B=(\pr_2\circ \tau_{F',\hat{F}'}\circ \tau_{F,\hat{F}}^{-1})\big((x,\hat{x}),(\pr_2\circ\tau_{F,\hat{F}})(x,\hat{x};A)B\big),
\]
i.e.,
\[
\mc{M}_{F',\hat{F}'}(A)B&=\mc{M}_{\hat{F},\hat{F}'}(\id_{T|_{\hat{x}}\hat{M}})\big(\mc{M}_{F,\hat{F}}(A)B\big)\mc{M}_{F,F'}(\id_{T|_x M})^{-1} \\
&=\mc{M}_{F,\hat{F}'}(A)B\mc{M}_{F',F}(\id_{T|_x M}) \nonumber
\]
which, by multiplying by $\mc{M}_{F,\hat{F}'}(A)^{-1}$ from the left, is equivalent to
\begin{align}\label{eq:transition_of_action}
\mc{M}_{F',F}(\id_{T|_x M})B=B\mc{M}_{F',F}(\id_{T|_x M})
\end{align}
Since $\SO(n)$ is not commutative for $n\geq 3$,
the left and right hand sides are not equal in general:
they are equal for all $B,F,F'$ if and only if $\SO(n)$ is commutative i.e. if and only if $n\in \{1,2\}$.
Hence $\pi_Q$ is not a principal $\SO(n)$-bundle, at least w.r.t the trivializations that we used, if $n\geq 3$.
\end{proof}

\begin{remark}
If $M$ and $\hat{M}$ are parallelizable (e.g. if $M$ and $\hat{M}$ are Lie groups) i.e., if there are global frames
and hence global orthonormal frames $F$, $\hat{F}$,
then one can introduce a principal $\SO(n)$-bundle structure for $\pi_Q$
by Eq. (\ref{eq:principal_action}) even for $n>2$.
However, this principal bundle structure then depends on the choice
of the global frames $F$, $\hat{F}$ i.e., we might (and could if $n>2$) get a different
principal bundle structure by the choosing the orthonormal frames differently. 
We will define on $Q$ a distribution $\RDist$ (see Definition \ref{def:rdist})
that models the natural constraints for the rolling problem
and by simple computations one can check that in general for $n\geq 3$
the distribution $\RDist$ is not invariant with respect to this principal bundle action for $\pi_Q$.

Hence the principal bundle structure on parallelizable manifolds
(or, in the general case, the local principal bundle structures defined by (\ref{eq:principal_action}))
is (in general) not useful for the study of the rolling model.

We will also study briefly a less restrictive model of rolling (rolling with spinning allowed)
where one considers a distribution $\NSDist$ on $Q$.
In this case, it will be shown in Theorem \ref{th:no_useful_principal_bundle_for_pi_Q} below
that in general there cannot be a principal bundle structure for $\pi_Q$
which leaves $\NSDist$ invariant.

\end{remark}

\begin{remark}
Clearly the fact that we chose $\SO(n)$ to act on the right in (\ref{eq:principal_action}) 
does not affect the conclusion of the previous Proposition:
Left local actions (in an obvious manner) lead to the same conclusion i.e.,
they don't glue up correctly to give a "natural" global $\SO(n)$-action.

Indeed, if instead of (\ref{eq:principal_action}) we tried to define the \emph{left} $\SO(n)$-action on $Q$ by
demanding that locally
\[
\lambda(B,(x,\hat{x};A))=\tau^{-1}\big((x,\hat{x}),B(\pr_2\circ\tau)(x,\hat{x};A)\big),
\]
we still could not define the action globally.
Indeed, it is enough to notice that instead of (\ref{eq:transition_of_action}) we would get
\[
B\mc{M}_{F',\hat{F}'}(A)
=&\mc{M}_{\hat{F},\hat{F}'}(\id_{T|_{\hat{x}}\hat{M}})\big(B\mc{M}_{F,\hat{F}}(A)\big)\mc{M}_{F,F'}(\id_{T|_{x}M})^{-1} \\
=&\mc{M}_{\hat{F},\hat{F}'}(\id_{T|_{\hat{x}}\hat{M}})B\mc{M}_{F',\hat{F}}(A)
\]
i.e.
\[
B\mc{M}_{\hat{F},\hat{F}'}(\id_{T|_{\hat{x}}\hat{M}})=\mc{M}_{\hat{F},\hat{F}'}(\id_{T|_{\hat{x}}\hat{M}})B
\]
which, again, is only true for all $B,\hat{F},\hat{F}'$ if and only if $n\in \{1,2\}$.
\end{remark}

Despite the lack of a "natural" principal bundle structure for $\pi_Q$ when $n\geq 3$,
we may still make use of the vector bundle structure of the ambient bundle $\pi_{T^*(M)\otimes T(\hat{M})}$
(the extended state space).

Notice that any $\pi_Q$-vertical tangent vector (i.e., an element of $V|_{q}(\pi_Q)$)
is of the form $\nu(B)|_q$ for a unique $B\in T^*M\otimes T\hat{M}|_{(x,\hat{x})}$ where $q=(x,\hat{x};A)\in Q$.
The following simple proposition gives the 
condition when, for a $B\in T^*M\otimes T\hat{M}|_{(x,\hat{x})}$,
the vector $\nu(B)|_q\in V|_q(\pi_{T^*M\otimes T\hat{M}})$
is actually tangent to $Q$ i.e., an element of $V|_q(\pi_{Q})$.

\begin{proposition}\label{pr:vertical_of_Q}
Let $q=(x,\hat{x};A)\in Q$ and $B\in T^*(M)\otimes T(\hat{M})|_{(x,\hat{x})}$.
Then $\nu(B)|_q$ is tangent to $Q$ (i.e., is an element of $V|_q (\pi_Q)$) if and only if
\[
\hat{g}(AX,BY)+\hat{g}(BX,AY)=0
\]
for all $X,Y\in T|_x M$.
Denoting $\ol{T}=T_{g,\hat{g}}$, this latter condition can be stated equivalently as $A^{\ol{T}} B+B^{\ol{T}} A=0$ or more compactly as $B\in A(\so(T|_x M))$
\end{proposition}

We will be denoting the $(g,\hat{g})$-transpose operation $T_{g,\hat{g}}$ by $\ol{T}$ also in the sequel.
The proposition says that $V|_{(x,\hat{x};A)}(\pi_Q)$ is naturally $\R$-linearly isomorphic to $A(\so(T|_x M))$.

\begin{remark}
We may reformulate the fact given by the previous proposition as follows.
Define $\so(M)=\bigcup_{x\in M} \so(T|_x M)$
(with $M$ a Riemannian manifold)
i.e.,
\[
\so(M)=\{B\in T^1_1(M)\ |\ B^{T_g}+B=0\}.
\]
One sees that $\so(M)$ is a closed embedded submanifold of $T^1_1(M)=T^*M\otimes TM$.
Moreover, the map $\pi_{\so(M)}:=\pi_{T^1_1(M)}|_{\so(M)}$
clearly defines a smooth vector bundle with typical fiber $\so(n)$, where $n=\dim(M)$.

We may pull back $\pi_{\so(M)}$
with a map $\pi_{Q,M}:=\pr_1\circ\pi_Q:Q\to M$
to a smooth bundle $(\pi_{Q,M})^*(\pi_{\so(M)}):(\pi_{Q,M})^*(\so(M))\to Q$
over $Q$. Its elements are
all pairs $((x,\hat{x};A),B)\in Q\times \so(M)$
where $x=\pi_{\so(M)}(B)$
and the bundle map is defined by $(\pi_{Q,M})^*(\pi_{\so(M)})((x,\hat{x};A),B)=(x,\hat{x};A)$.

Proposition \ref{pr:vertical_of_Q} shows that the
bundle map $L:(\pi_{Q,M})^*(\pi_{\so(M)})\to V(\pi_Q)$
defined by
$L((x,\hat{x};A),B)=\nu(AB)|_{(x,\hat{x};A)}$
is a diffeomorphism.
\end{remark}

%%%%%%%%%%%%%%%%%%%%%%%
\subsubsection{The State Space as a Quotient}
%%%%%%%%%%%%%%%%%%%%%%%

In this subsection, we will show that (the $n$-dimensional version of) the construction of the state
space for rolling that has been used e.g. in \cite{bryant-hsu} in dimension two
is actually isomorphic to the state space $Q$.

\begin{proposition}\label{pr:char_of_Q}
Let $F_{\mathrm{OON}}(M)$, $F_{\mathrm{OON}}(\hat{M})$ be the oriented orthonormal frame bundles of $(M,g)$, $(\hat{M},\hat{g})$
(resp. let $F(M)$, $F(\hat{M})$ be the frame bundles of $M$ and $\hat{M}$).
Denote by $\mu$, $\hat{\mu}$ the right $\SO(n)$-actions
(resp. right $\GL(n)$-actions)
defining the usual principal bundle structures
on these spaces
i.e., $\mu((X_k)_{k=1}^n,[A^i_j])=(\sum_{k} A^k_iX_k)_{i=1}^n$ and similarly for $\hat{\mu}$.
Define a diagonal right $\SO(n)$-action
\[
\Delta:(F_{\mathrm{OON}}(M)\times F_{\mathrm{OON}}(\hat{M}))\times \SO(n)\to F_{\mathrm{OON}}(M)\times F_{\mathrm{OON}}(\hat{M}),
\]
by
(resp. right $\GL(n)$-action $\Delta:(F(M)\times F(\hat{M}))\times \GL(n)\to F(M)\times F(\hat{M}))$)
\[
\Delta(((X_i),(\hat{X}_j)),A)=(\mu((X_i),A),\hat{\mu}((\hat{X}_i),A)).
\]
The map $\xi:F_{\mathrm{OON}}(M)\times F_{\mathrm{OON}}(\hat{M})\to Q(M,\hat{M})$
(resp. $\xi:F(M)\times F(\hat{M})\to T^*M\otimes T\hat{M}$)
such that
\[
\xi((X_i),(\hat{X}_j)):=\big(\sum_i a_iX_i\mapsto \sum_i a_i \hat{X}_i)
\]
is a smooth surjective submersion. Moreover, for each
$q\in Q(M,\hat{M})$ (resp. $q\in T^*M\otimes T\hat{M}$)
the inverse image $\xi^{-1}(q)$ coincides with an orbit of $\Delta$.
Thus $\xi$ induces a diffeomorphism $\ol{\xi}:(F_{\mathrm{OON}}(M)\times F_{\mathrm{OON}}(\hat{M}))/\Delta\to Q(M,\hat{M})$.
(resp. $\ol{\xi}:(F(M)\times F(\hat{M}))/\Delta\to \{A\in T^*M\otimes T\hat{M}\ |\ A\ \textrm{is invertible}\}$).
\end{proposition}

\begin{proof}
The smoothness and surjectivity of $\xi$ are obvious and
it is also easy to see that $\xi$ is a submersion.
Thus it is enough to show that $\xi^{-1}(q)$ coincides with an orbit of $\Delta$.
First suppose that 
$$\mu((X_i),A)_i=\sum_{j} A^j_i X_j,\quad
\hat{\mu}((\hat{X}_i),A)_i=\sum_{j} A^j_i \hat{X}_j.$$
Then, for any real numbers $a_1,\cdots,a_n$, one has
\[
& \xi\big(\Delta\big(((X_i),(\hat{X}_i)),A\big)\big)(\sum_k a_k X_k)
=\xi(\mu((X_i),A),\hat{\mu}((\hat{X}_i),A))(\sum_k a_k X_k) \\
=&\xi(\mu((X_i),A),\hat{\mu}((\hat{X}_i),A))(\sum_{k,i} a_k (A^{-1})^i_k \mu((X_j),A)_i ) \\
=&\sum_{k,i} a_k (A^{-1})^i_k \hat{\mu}((\hat{X}_j),A)_i
=\sum_{k,i,j} a_k (A^{-1})^i_k A^j_i \hat{X}_j=\sum_k a_k \hat{X}_k
=\xi((X_i),(\hat{X}_i))(\sum_k a_k X_k).
\]
This shows that
\[
\Delta\big(\{((X_i),(\hat{X}_i))\}\times G\big)\subset \xi^{-1}(\xi((X_i),(\hat{X}_i))),
\]
with $G=\SO(n)$ (resp. $G=\GL(n)$).
The orbits of $\Delta$ all have the same dimension as $\SO(n)$, i.e., $\frac{n(n-1)}{2}$
(resp. $\dim \GL(n)=n^2$)
and since $$\dim\xi^{-1}(q)=\dim(F_{\mathrm{OON}}(M)\times F_{\mathrm{OON}}(\hat{M}))-\dim Q(M,\hat{M})=\dim\SO(n),$$ for any $q\in Q(M,\hat{M})$
(resp. 
\[
\dim\xi^{-1}(q)=\dim(F(M)\times F(\hat{M}))-\dim T^*M\otimes T\hat{M}=\dim \GL(n)),
\]
we have that this inclusion is actually an equality.
This proves the proposition.
\end{proof}

\begin{remark}
In the above proposition we implicitly assumed that $(F_{\mathrm{OON}}(M)\times F_{\mathrm{OON}}(\hat{M}))/\Delta$
(resp. $(F(M)\times F(\hat{M}))/\Delta$) already
has a natural structure of a smooth manifold namely that of a quotient manifold.
But it is easily seen that the action $\Delta$ is free and proper and hence
by a well known result (see \cite{lee02} Theorem 9.16) it follows that unique smooth quotient manifold structures
for the above quotient sets exist. Hence the facts established in the above proof guarantee that $\ol{\xi}$ is a diffeomorphism.
\end{remark}

\begin{remark}
Here is the product right action $$\mu\times\hat{\mu}:(F_{\mathrm{OON}}(M)\times F_{\mathrm{OON}}(\hat{M}))\times (\SO(n)\times\SO(n))\to F_{\mathrm{OON}}(M)\times F_{\mathrm{OON}}(\hat{M})$$ of $\SO(n)\times\SO(n)$ on $F_{\mathrm{OON}}(M)\times F_{\mathrm{OON}}(\hat{M})$
given by $$\mu\times\hat{\mu}\big(((X_i),(\hat{X}_i)),(A,\hat{A}))=(\mu((X_i),A),\hat{\mu}((\hat{X}_i),\hat{A})).$$
As it is easily seen, it is unfortunately \emph{not} true that the action $\mu\times\hat{\mu}$ maps a $\Delta$-orbit into a $\Delta$-orbit,
unless the dimension $n$ is equal to two (in which case $\SO(n)=\SO(2)$ is commutative)
and hence, in the case $n>2$, the map $\mu\times\hat{\mu}$ does not induce a map $Q\times\SO(n)\to Q$
(where $Q\cong (F_{\mathrm{OON}}(M)\times F_{\mathrm{OON}}(\hat{M}))/\Delta$ by the above proposition).
This is yet another way of seeing that $Q=Q(M,\hat{M})$ cannot be given a "natural" $\SO(n)$-principal bundle structure for $n\geq 3$
i.e., we cannot induce on $Q$ the principal bundle structures of the
frame bundles $F_{\mathrm{OON}}(M)$ and $F_{\mathrm{OON}}(\hat{M})$
if $n>2$.
\end{remark}

\begin{remark}

Notice that on $F(M)$ (resp. on $F_{\mathrm{OON}}(M)$)one may also consider the left $\GL(n)$ (resp. $\SO(n)$) action $\lambda$ given 
by
$\lambda(A,(X_i))_i=\sum_j A_j^i X_j$.
Since $A_j^i=(A^T)_i^j$ it is trivial
that this is related to the above right action by $\lambda(A,(X_i))=\mu((X_i),A^
T)$.
Notice that $\mu(\lambda(A,(X_i)),B)=\mu(\mu((X_i),A^T),B)=\mu((X_i),A^TB)$
which, if $n\geq 3$ and  $A^TB\neq BA^T$, is different
from $\lambda(A,\mu((X_i)),B))=\mu(\mu((X_i),B),A^T)=\mu((X_i),BA^T)$.
This means that the left and right actions $\lambda$ and $\mu$ do not "commute".

Another way to define naturally a left actions is to use instead of above $\lambda$ the inverse right-action
$\lambda_{I}(A,(X_i)):=\mu((X_i),A^{-1})$. Also in this case, $\mu(\lambda_I(A,(X_i)),B)=\mu(\mu((X_i),A^{-1}),B)=\mu((X_i),A^{-1}B)$
is not equal, if $n\geq 3$ and $AB\neq BA$, to $\lambda_I(A,\mu((X_i)),B))=\mu(\mu((X_i),B),A^{-1})=\mu((X_i),BA^{-1})$.
On $F_{\mathrm{OON}}(M)$ it is clear that the actions $\lambda$ and $\lambda_I$ coincide.

It was proposed in \cite{norway} that one could use the inverse left action on $F_{\mathrm{OON}}(M)$
and the left action on $F_{\mathrm{OON}}(\hat{M})$
to induce, respectively, left and right actions on $Q$.
However this is not possible for the following reason (which basically is a repetition of what
has been said above).
Suppose $q=(x,\hat{x};A)\in Q$ and let $F,F'\in F_{\mathrm{OON}}(M)$, $\hat{F},\hat{F}'\in F_{\mathrm{OON}}(\hat{M})$
are such that $\xi(F,\hat{F})=q$ and $\xi(F',\hat{F}')=q$.
Then there is a $B\in\SO(n)$ such that
$\mu(F,B)=F'$, $\hat{\mu}(\hat{F},B)=\hat{F}'$.
By using, for example, the left $\SO(n)$-action $\lambda$ on $F_{\mathrm{OON}}(M)$
we get $\lambda(C,F')=\lambda(C,\mu(F,B))=\mu(F,BC^T)$
and also $\mu(\lambda(C,F),B)=\mu(F,C^TB)$.
But $\xi(\lambda(C,F'),\hat{F}')=\xi(\lambda(C,F),\hat{F})$
if and only if $\mu(\lambda(C,F),B)=\lambda(C,F')$
which thus is not true unless $C^TB=BC^T$.
The case of the inverse left action (which is just the right action $\hat{\mu})$ on $F_{\mathrm{OON}}(\hat{M})$
leads to the same conclusion.
\end{remark}

%%%%%%%%%%%%%%%%%%%%%%%%%%%%%%%%%%%%%%
\subsection{Distribution and the Control Problems}
\subsubsection{From Rolling to Distributions}\label{sec:2}
%%%%%%%%%%%%%%%%%%%%%%%%%%%%%%%%%%%%%%

Each point $(x,\hat{x};A)$ of the state space $Q=Q(M,\hat{M})$
can be viewed as describing a contact point of the two manifolds
which is given by the points $x$ and $\hat{x}$ of $M$ and $\hat{M}$, respectively,
and an isometry $A$ of the tangent spaces $T|_x M$, $T|_{\hat{x}} \hat{M}$ at this contact point.
The isometry $A$ can be viewed as measuring the relative orientation of these tangent spaces
relative to each other
in the sense that rotation of, say, $T|_{\hat{x}} \hat{M}$
corresponds to a unique change of the isometry $A$ from $T|_x M$ to $T|_{\hat{x}}\hat{M}$.
A curve $t\mapsto (x(t),\hat{x}(t);A(t))$ in $Q$ can then be seen as a
motion of $M$ against $\hat{M}$ such that at an instant $t$, $x(t)$ and $\hat{x}(t)$
represent the common point of contact in $M$ and $\hat{M}$, respectively,
and $A(t)$ measures the relative orientation of coinciding tangent spaces $T|_{x(t)} M$, $T|_{\hat{x}(t)} \hat{M}$
at this point of contact.

In order to call this motion \emph{rolling}, there are two kinematic constraints
that will be demanded (see e.g. \cite{agrachev99}, \cite{agrachev04} Chapter 24, \cite{chelouah01})
namely
\begin{itemize}
\item[(i)] the \emph{no-spinning} condition; 
\item[(ii)] the \emph{no-slipping} condition.
\end{itemize}

In this section, these conditions will be defined explicitly
and it will turn out that they are modeled by certain smooth distributions
on the state space $Q$.
The subsequent sections are then devoted to the detailed definitions and
analysis of the distribution $\NSDist$ and $\RDist$ on the state space $Q$,
the former capturing the no-spinning condition (i) while the latter capturing both of
the conditions (i) and (ii).

The first restriction (i) for the motion is
that the relative orientation of the two manifolds should not change
along motion.
This \emph{no-spinning condition} (also known as the no-twisting condition)
can be formulated as follows.

\begin{definition}\label{def:2:2}
An absolutely continuous (\emph{a.c.}) curve 
\begin{align}
q:\ I&\to Q,\nonumber\\
t&\mapsto (x(t),\hat{x}(t);A(t)),\nonumber
\end{align} defined on some real interval $I=[a,b]$,
is said to describe a \emph{motion without spinning} of $M$ against $\hat{M}$ if,
for every a.c. curve $[a,b]\to TM$; $t\mapsto X(t)$ of vectors along $t\mapsto x(t)$, we have
\be\label{eq:nospin}
\nabla_{\dot{x}(t)} X(t)=0 \quad \Longrightarrow\quad \hat{\nabla}_{\dot{\hat{x}}(t)} (A(t)X(t))=0
\quad\textrm{for a.e.}\ t\in [a,b].
\ee
\end{definition}

(See also \cite{norway} for a similar definition.) Notice that Condition (\ref{eq:nospin})
is equivalent to the following: for almost every $t$ and all parallel vector fields $X(\cdot)$ along $x(\cdot)$, one has
\[
(\ol{\nabla}_{(\dot{x}(t),\dot{\hat{x}}(t))} A(t))X(t)=0.
\]
(This is well defined as mentioned in the paragraph immediately below Eq. (\ref{eq:tensor_extension}).)

Since the parallel translation
$P_0^t(x):T|_{x(0)} M\to T|_{x(t)} M$ along $x(\cdot)$ is an (isometric) isomorphism (here $X(t)=P_0^t(x) X(0)$),
this shows that (\ref{eq:nospin}) is equivalent to
\begin{align}\label{eq:nospin2}
\ol{\nabla}_{(\dot{x}(t),\dot{\hat{x}}(t))} A(t)=0\quad \textrm{for a.e.}\ t\in [a,b].
\end{align}

The second restriction (ii) is that the manifolds should not
slip along each other as they move i.e., the velocity of the contact point
should be the same w.r.t both manifolds.        
This \emph{no-slipping condition} can be formulated as follows.

\begin{definition}\label{def:2:1}
An a.c. curve $I\to Q$; $t\mapsto (x(t),\hat{x}(t);A(t))$, defined on some real interval 
$I=[a,b]$,
is said to describe a \emph{motion without slipping} of $M$ against $\hat{M}$ if
\be\label{eq:noslip}
A(t)\dot{x}(t)=\dot{\hat{x}}(t)\quad \textrm{for a.e.}\ t\in [a,b].
\ee
\end{definition}

\begin{definition}\label{def:rolling}
An a.c. curve $I\to Q$; $t\mapsto (x(t),\hat{x}(t);A(t))$, defined on some real interval $I=[a,b]$,
is said to describe a \emph{rolling motion} i.e., a \emph{motion without slipping or spinning} of $M$ against $\hat{M}$
if it satisfied both of the conditions (\ref{eq:nospin}),(\ref{eq:noslip})
(or equivalently (\ref{eq:nospin2}),(\ref{eq:noslip})).
The corresponding curve $t\mapsto (x(t),\hat{x}(t);A(t))$ that satisfies these conditions
is called a \emph{rolling curve}.
\end{definition}

It is easily seen that $t\mapsto q(t)=(x(t),\hat{x}(t);A(t))$, $t\in [a,b]$,
is a rolling curve
if and only if it satisfies the following driftless control affine system
\begin{align}\label{eq:cs_rolling}
\srol\quad \begin{cases}
& \dot{x}(t)=u(t), \cr
& \dot{\hat{x}}(t)=A(t)u(t), \cr
& \ol{\nabla}_{(u(t),A(t)u(t))} A(t)=0,
\end{cases}
\quad \textrm{for a.e.}\ t\in [a,b].
\end{align}
where the control $u$ belongs to
$\mc{U}(M)$, the set of measurable $TM$-valued functions $u$ defined on some interval $I=[a,b]$ such that there exists a.c. $ y:[a,b]\to M$ verifying $u=\dot{y}$ a.e. on $ [ a,b]$.
Conversely, given any control $u\in \mc{U}(M)$ and $q_0=(x_0,\hat{x}_0;A_0)\in Q$,
a solution $q(\cdot)$ to this control system exists on a subinterval $[a,b']$, $a<b'\leq b$
satisfying the initial condition $q(a)=q_0$.
The fact that System (\ref{eq:cs_rolling}) is driftless and control affine
can be seen from its representation in local coordinates (see (\ref{simon}) in Appendix \ref{app:local}).

We end up this subsection by the following simple remark. 
\begin{remark}
In many cases, it is more convenient to work in the
extended state space $T^*(M)\otimes T(\hat{M})$ rather than in (its submanifold) $Q$
because $\pi_{T^*(M)\otimes T(\hat{M})}$ is a vector bundle.
Since the above constraints of motion (\ref{eq:nospin}) and (\ref{eq:noslip})
can also be formulated in this space in verbatim,
we will sometimes take this more general approach
and then restrict to $Q$.
\end{remark}

%%%%%%%%%%%%%%%%%%%%%%%%%%%%%%%%%%%%%%%%%%%%%%%%%%%%
\subsubsection{The No-Spinning Distribution $\NSDist$}\label{sec:2.1}
%%%%%%%%%%%%%%%%%%%%%%%%%%%%%%%%%%%%%%%%%%%%%%%%%%%%

In this section, we build a smooth distribution $\NSDist$ on the spaces $Q$ and $T^*M\otimes T\hat{M}$
which plays the role of modelling the no-spinning condition for the rolling, see (\ref{eq:nospin}).
We will also study the geometry related to this distribution.
For more general constructions and some more general results than the ones in this section, see \cite{joyce07}, \cite{kobayashi63}.

We begin by recalling some basic observations on parallel transport.
As noted in Proposition \ref{pr:parallel_tensor}, if one starts with a $(1,1)$-tensor $A_0\in T^1_1|_{(x_0,\hat{x}_0)}(M\times\hat{M})$
and has an a.c. curve $t\mapsto (x(t),\hat{x}(t))$ on $M\times\hat{M}$ with $x(0)=x_0$, $\hat{x}(0)=\hat{x}_0$, defined on an open interval $I\ni 0$,
then the parallel transport $A(t)=P^t_0(x,\hat{x})A_0$ exists on $I$ and determines an a.c. curve.
But now, if $A_0$ rather belongs to the subspace $T^*M\otimes T\hat{M}$ or $Q$ of $T^1_1(M\times\hat{M})$,
it will actually happen that the parallel translate $A(t)$ belongs to this subspace as well for all $t\in I$.
This is the content of the next proposition.

\begin{proposition}\label{pr:2.1:1}
Let $t\mapsto (x(t),\hat{x}(t))$ be an absolutely continuous curve in $M\times \hat{M}$ defined on some real interval $I\ni 0$.
Then we have 
\[
A_0\in T^*M\otimes TM &\Longrightarrow A(t)=P^t_0(x,\hat{x}) A_0\in T^*M\otimes T\hat{M}\quad \forall t\in I, \\
A_0\in Q\quad &\Longrightarrow\quad A(t)=P^t_0(x,\hat{x}) A_0\in Q\quad \forall t\in I,
\]
and
\begin{align}\label{eq:parallel_trans_A}
P^t_0(x,\hat{x}) A_0=P^t_0(\hat{x})\circ A_0\circ P_t^0(x)\quad \forall t\in I.
\end{align}
\end{proposition}

\begin{proof}
Let $Y\in T|_{x(0)} M$, $\hat{Y}\in T|_{\hat{x}(0)}\hat{M}$ and let $Y(t)=P_0^t(x) Y$, $\hat{Y}(t)=P_0^t(\hat{x}) \hat{Y}$ be their
parallel translates along $t\mapsto x(t)$ and $t\mapsto \hat{x}(t)$ respectively.
Similarly, choose $\omega\in T^*|_{x(0)} M$ and denote $\omega(t)=P^t_0(x)\omega$ its parallel translate.
Then $Y(t)$, $\hat{Y}(t)$ and $\omega(t)$ can be viewed as a curves in $T(M\times\hat{M})$ and $T^*(M\times\hat{M})$
using the canonical inclusions $T|_{x(t)}M,T|_{\hat{x}(t)}\hat{M}\subset T|_{(x(t),\hat{x}(t))}(M\times\hat{M})$, $T^*|_{x(t)}M\subset T^*|_{(x(t),\hat{x}(t))}(M\times\hat{M})$.

With $A_0\in T^* M\otimes T\hat{M}|_{(x(0),\hat{x}(0))}\subset T^1_1(M\times\hat{M})|_{(x(0),\hat{x}(0))}$ and $A(t)=P^t_0(x,\hat{x}) A_0\in T^1_1(M\times\hat{M})|_{(x(t),\hat{x}(t))}$, we have, for a.e. $t$ (the contractions that use are obvious),
\[
\ol{\nabla}_{(\dot{x}(t),\dot{\hat{x}}(t))}(A(\cdot)\omega(\cdot))
=(\ol{\nabla}_{(\dot{x}(t),\dot{\hat{x}}(t))}A(\cdot))\omega(t)+A(t) (\nabla_{\dot{x}(t)}\omega(\cdot))=0,
\]
and similarly $\ol{\nabla}_{(\dot{x}(t),\dot{\hat{x}}(t))}(A(\cdot)\hat{Y}(\cdot))=0$. It implies
that $A(t)\omega(t)$, $A(t)\hat{Y}(t)$
(as elements of $T^1_1(M\times\hat{M})$) are parallel to $t\mapsto (\dot{x}(t),\dot{\hat{x}}(t))$
with initial conditions $A_0\omega=0$ and $A_0\hat{Y}=0$ since $A_0\in T^* M\otimes T\hat{M}$.
By the uniqueness of solutions of ODEs, this shows that $A(t)\omega(t)=0$ and $A(t)\hat{Y}(t)=0$ for all $t\in I$
i.e., since $\hat{Y},\omega$ were arbitrary, $A(t)\in T^*M\otimes T\hat{M}$ for all $t\in I$.

Suppose next that $A_0\in Q|_{(x(0),\hat{x}(0))}$ and denote $A(t)=P^t_0(x,\hat{x}) A_0$. Then $A_0\in T^*M\otimes T\hat{M}$
and, by what we just proved, $A(t)\in  T^*M\otimes T\hat{M}$ for all $t\in I$.
It follows that $A(t)Y(t)\in T|_{\hat{x}(t)} \hat{M}$ and thus taking its norm w.r.t $\hat{g}$
allows us to compute a.e.
\[
\dif{t} \n{A(t)Y(t)}_{\hat{g}}^2
=2\hat{g}\big((\ol{\nabla}_{(\dot{x}(t),\dot{\hat{x}}(t))} A(\cdot))Y(t)+A(t)\nabla_{\dot{x}(t)}Y(\cdot),A(t)Y(t)\big)=0.
\]
The initial condition for $\n{(A(t)Y(t)}_{\ol{g}}^2$
at $t=0$ is $\n{(A(0)Y(0)}_{\ol{g}}^2=\n{Y}_{g}^2$,
since $A_0=A(0)$ is an isometry (and $Y(0)=Y$).
Since $\n{Y(t)}_{\ol{g}}^2$ also satisfies 
$\dif{t}\n{Y(t)}_{\ol{g}}^2=0$ and the initial condition $\n{Y(0)}_{\ol{g}}^2=\n{Y}_{g}^2$,
we see that $\n{A(t)Y(t)}_{\ol{g}}^2=\n{Y(t)}_{\hat{g}}^2$
for all $t\in I$
(since the maps $t\mapsto\n{A(t)Y(t)}_{\ol{g}}^2$, $t\mapsto \n{Y(t)}_{\hat{g}}^2$ were a.c.).
Since the parallel translation $P_0^t(x):T|_{x(0)} M\to T|_{x(t)} M$ is
a linear (isometric) isomorphism for every $t$,
this proves that $A(t):T|_{x(t)} M\to T|_{\hat{x}(t)} \hat{M}$ is an isometry for every $t$.
Because $t\mapsto \det (A(t))$ is a continuous map $I\to \{-1,+1\}$
and $\det(A(0))=\det(A_0)=+1$, it follows that $\det(A(t))=+1$ for all $t$. Hence $A(t)\in Q$ for all $t$.

Finally Eq. (\ref{eq:parallel_trans_A}) is proved as follows. Consider $B(t):=P^t_0(\hat{x})\circ A_0\circ P_t^0(x)$, which is an a.c. curve in $T^*M\otimes T\hat{M}$(or even in $Q$ if $A_0\in Q$)
along $t\mapsto (x(t),\hat{x}(t))$.
Now $B(0)=A_0$ and, for $X_0\in T|_{x(0)} M$, $X(t):=P_0^t(x)X_0$, we have
\[
0=&\hat{\nabla}_{\dot{x}(t)} (P^t_0(\hat{x}) (A_0X_0))
=\hat{\nabla}_{\dot{\hat{x}}(t)} (B(t)X(t)) \\
=&\big(\ol{\nabla}_{(\dot{x},\dot{\hat{x}})(t)} B(t)\big) X(t)+B(t)\nabla_{\dot{x}(t)} X(t)
=\big(\ol{\nabla}_{(\dot{x},\dot{\hat{x}})(t)} B(t)\big) X(t),
\]
from which it follows, since $X_0$ was arbitrary, that $\ol{\nabla}_{(\dot{x},\dot{\hat{x}})(t)} B(t)=0$ for a.e. $t\in I$.
Thus $t\mapsto A(t)$ and $t\mapsto B(t)$ solve the same initial value problem and hence (being a.c.) are equal $A(t)=B(t)$ i.e.,
\[
P^t_0(x,\hat{x}) A_0=P^t_0(\hat{x})\circ A_0\circ P_t^0(x),\quad \forall t\in I,
\]
which is what we wished to prove.

\end{proof}

Let $T(M\times\hat{M})\times_{M\times\hat{M}} (T^*(M)\otimes T(\hat{M}))$
be the total space of the product vector bundle $\pi_{T(M\times\hat{M})}\times_{M\times\hat{M}}\pi_{T^*(M)\otimes T(\hat{M})}$ over $M\times\hat{M}$.
We will define certain \emph{lift} operations corresponding to parallel translation of elements of $T^*M \otimes T\hat{M}$.

\begin{definition}\label{def:LNSD}
The \emph{No-Spinning lift} is defined to be the map
\[
\LNSD:T(M\times\hat{M})\times_{M\times\hat{M}} \big(T^*(M)\otimes T(\hat{M})\big)\to T\big(T^*(M)\otimes T(\hat{M})\big),
\]
such that, if $q=(x,\hat{x};A)\in T^*(M)\otimes T(\hat{M})$, $X\in T|_xM$, $\hat{X}\in T|_{\hat{x}}\hat{M}$
and $t\mapsto (x(t),\hat{x}(t))$ is a smooth curve on in $M\times \hat{M}$
defined on an open interval $I\ni 0$ s.t. $\dot{x}(0)=X$, $\dot{\hat{x}}(0)=\hat{X}$, then one has
\begin{align}\label{eq:2.1:1}
\LNSD((X, \hat{X}),q)=\dif{t}\big|_0 P^t_0(x,\hat{x})A\in T|_{q} \big(T^*(M)\otimes T(\hat{M})\big).
\end{align}
\end{definition}

The smoothness of the map $\LNSD$ can be easily seen by using fiber or local coordinates (see Appendix \ref{app:local}).
We will usually use a notation $\LNSD(\ol{X})|_q$ for $\LNSD(\ol{X},q)$
when $\ol{X}\in T|_{(x,\hat{x})} (M\times\hat{M})$ and $q=(x,\hat{x};A)\in T^*(M)\otimes T(\hat{M})$. In particular, when $\ol{X}\in \VF(M\times\hat{M})$,
we get a \emph{lifted vector field} on $T^*(M)\otimes T(\hat{M})$ given by $q\mapsto \LNSD(\ol{X})|_q$.
The smoothness of $\LNSD(\ol{X})$ for $\ol{X}\in\VF(M\times\hat{M})$ follows immediately
from the smoothness of the map $\LNSD$.
Notice that, by Proposition \ref{pr:2.1:1}, the No-Spinning lift map $\LNSD$ restricts to
\[
\LNSD:T(M\times\hat{M})\times_{M\times\hat{M}} Q\to TQ,
\]
where  $T(M\times\hat{M})\times_{M\times\hat{M}} Q$
is the total space of the fiber product $\pi_{T(M\times\hat{M})}\times_{M\times\hat{M}} \pi_{Q}$.

We now define the distribution $\NSDist$ on $T^*(M)\otimes T(\hat{M})$ and $Q$
capturing the no-spinning condition (see Eq. (\ref{eq:nospin})).

\begin{definition}\label{def:2.1:1}
The \emph{No-Spinning (NS) distribution} $\NSDist$ on $T^*(M)\otimes T(\hat{M})$ is a $2n$-dimensional smooth distribution defined pointwise by 
\begin{align}\label{eq:2.1:2}
\NSDist|_{(x,\hat{x};A)}=\LNSD(T|_{(x,\hat{x})}(M\times \hat{M}))|_{(x,\hat{x};A)},
\end{align}
with $(x,\hat{x};A)\in T^*(M)\otimes T(\hat{M})$.
Since $\NSDist|_Q\subset T(Q)$ (by Proposition \ref{pr:2.1:1}) this distribution restricts to a $2n$-dimensional smooth distribution on $Q$ which we also denote by $\NSDist$ (instead of $\NSDist|_Q$).
\end{definition}

The No-Spinning lift $\LNSD$ will also be called $\NSDist$-lift
since it maps vectors of $M\times\hat{M}$ to vectors in $\NSDist$.

The distribution $\NSDist$ is smooth since  $\LNSD(\ol{X})$ is smooth for any smooth vector field $\ol{X}\in\VF(M\times\hat{M})$.
Also, the fact that the rank of $\NSDist$ exactly is $2n$ follows from
the next proposition,
which itself follows immediately from Eq. (\ref{eq:2.1:1}).

\begin{proposition}
For every $q=(x,\hat{x};A)\in T^*M\otimes T\hat{M}$
and $\ol{X}\in T|_{(x,\hat{x})} M\times\hat{M}$, one has
\[
(\pi_{T^*M\otimes T\hat{M}})_*(\LNSD(\ol{X})|_q)=\ol{X},
\]
and in particular $(\pi_Q)_*(\LNSD(\ol{X})|_q)=\ol{X}$
if $q\in Q$.
\end{proposition}

Thus $(\pi_{T^*M\otimes T\hat{M}})_*$ (resp. $\pi_Q$)
maps $\NSDist|_{(x,\hat{x};A)}$ isomorphically onto
$T|_{(x,\hat{x})} (M\times\hat{M})$
for every $(x,\hat{x};A)\in T^*M\otimes T\hat{M}$
(resp. $(x,\hat{x};A)\in Q$)
and the inverse map of $(\pi_{T^*M\otimes T\hat{M}})_*|_{\NSDist|_q}$
(resp. $(\pi_Q)_*|_{\NSDist|_q}$) is $\ol{X}\mapsto \LNSD(\ol{X})|_q$.

\begin{remark}
It should now be clear that an a.c. map $t\mapsto q(t)=(x(t),\hat{x}(t);A(t))$
in $T^*M\otimes T\hat{M}$ or $Q$
satisfies (\ref{eq:nospin}) if and only if $q$ is tangent a.e. to $\NSDist$
i.e., for a.e. $t$ it holds that $\dot{q}(t)\in \NSDist|_{q(t)}$.
\end{remark}

The following basic formula for the lift $\LNSD$ will be useful.

\begin{theorem}\label{th:2.1:1}\label{th:LNSD_formula}
For $\ol{X}\in T|_{(x,\hat{x})}(M\times\hat{M})$
and $A\in \Gamma(\pi_{T^*M\otimes T\hat{M}})$, we have
\begin{align}\label{eq:LNSD_formula}
\LNSD(\ol{X})|_{A|_{(x,\hat{x})}}=A_*(\ol{X})-\nu\big(\ol{\nabla}_{\ol{X}} A\big)|_{A|_{(x,\hat{x})}},
\end{align}
where $\nu$ denotes the vertical derivative in the
vector bundle $\pi_{T^* M\otimes T\hat{M}}$
and $A_*$ is the map $T(M\times\hat{M})\to T(T^*M\otimes T\hat{M})$.
\end{theorem}

\begin{proof}
Choose smooth paths $c:[-1,1]\to M$, $\hat{c}:[-1,1]\to \hat{M}$
such that $(\dot{c}(0), \dot{\hat{c}}(0))=\ol{X}$
and take an arbitrary $f\in\Cinf(T^*M\otimes T\hat{M})$.
Define $\tilde{A}(t)=P_0^t(c,\hat{c}) A|_{(x,\hat{x})}$.
Then
\[
\LNSD(\ol{X})|_{A|_{(x,\hat{x})}}=\dot{\tilde{A}}(0)=\tilde{A}_*(\pa{t}).
\]
Also, it is known that (see e.g. \cite{sakai91}, p.29)
\begin{align}\label{eq:2.1:3a}
P_t^0(c,\hat{c})(A|_{(c(t),\hat{c}(t)})=A|_{(x,\hat{x})}+t\ol{\nabla}_{\ol{X}} A+t^2F(t),
\end{align}
with $t\mapsto F(t)$ a $\Cinf$-function $]-1,1[\to T^*|_x M\otimes T|_{\hat{x}}\hat{M}$.
On the other hand, one has
\[
& \big(A_*(\ol{X})-\tilde{A}_*(\pa{t})\big)f
=\lim_{t\to 0} \frac{f(A|_{(c(t),\hat{c}(t))})-f(P_0^t(c,\hat{c}) A|_{(x,\hat{x})})}{t} \\
=&\lim_{t\to 0} \frac{f(P_0^t(c,\hat{c}) A|_{(x,\hat{x})}+tP_0^t(c,\hat{c})\ol{\nabla}_{\ol{X}} A+t^2P_0^t(c,\hat{c}) F(t))-f(P_0^t(c,\hat{c}) A|_{(x,\hat{x})})}{t} \\
=&\lim_{t\to 0} \frac{1}{t}\int_0^t \dif{s}f\big(P_0^t(c,\hat{c}) A|_{(x,\hat{x})}+sP_0^t(c,\hat{c})\ol{\nabla}_{\ol{X}} A+s^2P_0^t(c,\hat{c}) F(t)\big)\diff s \\
=&\dif{s}\Big|_{s=0} f\big(A|_{(x,\hat{x})}+s\ol{\nabla}_{\ol{X}} A+s^2 F(0)\big)
=\nu(\ol{\nabla}_{\ol{X}} A)|_A f.
\]

\end{proof}

We shall write Eq. (\ref{eq:LNSD_formula}) from now on with a compressed notation
\[
\LNSD(\ol{X})|_A=A_*(\ol{X})-\nu(\ol{\nabla}_{\ol{X}} A)|_A.
\]

\begin{remark}\label{re:tangent_splitting}
If $A\in \Gamma(\pi_{T^*(M)\otimes T(\hat{M})})$ and $q:=A|_{(x,\hat{x})}\in Q$
(e.g. if $A\in\Gamma(\pi_Q)$), then on the right hand side of (\ref{eq:LNSD_formula}), both terms are elements of $T|_{q}(T^*M\otimes T\hat{M})$
but their difference is actually an element of $T|_{q}Q$.

Also, it is clear that Eq. (\ref{eq:LNSD_formula}) only indicates the decomposition of the map $A_*$ w.r.t to the direct sum
decomposition
\begin{align}\label{eq:2.1:4}
T\big(T^*M\otimes T\hat{M}\big)=\NSDist\oplus_{T^*M\otimes T\hat{M}} V(\pi_{T^*M\otimes T\hat{M}}),
\end{align}
when $A\in \Gamma(\pi_{T^*(M)\otimes T(\hat{M})})$ and
\begin{align}\label{eq:2.1:5}
TQ=\NSDist\oplus_{Q} V(\pi_Q),
\end{align}
when $A\in \Gamma(\pi_{Q})$ respectively.
\end{remark}

As a trivial corollary of the theorem, one gets the following. 

\begin{corollary}\label{cor:LNSD_along_path}
Suppose $t\mapsto (x(t),\hat{x}(t);A(t))$ is an a.c. curve on $T^*M\otimes T\hat{M}$ or $Q$ defined on an open real interval $I$.
Then, for a.e. $t\in I$,
\[
\LNSD\big(\dot{x}(t), \dot{\hat{x}}(t)\big)\big|_{(x(t),\hat{x}(t);A(t))}=\dot{A}(t)-\nu(\ol{\nabla}_{(\dot{x}(t),\dot{\hat{x}}(t))} A)|_{(x(t),\hat{x}(t);A(t))}.
\]
Hence $t\mapsto (x(t),\hat{x}(t);A(t))$ is tangent to $\NSDist$ at $t_0\in I$ if and only if $\ol{\nabla}_{(\dot{x}(t_0),\dot{\hat{x}}(t_0))} A=0$.
\end{corollary}

%%%%%%%%%%%%%%%%%%%%%%%%%%%%%%%%%%%%%%%%%%%%%%%%%%%%
\subsubsection{The Rolling Distribution $\RDist$}\label{sec:2.5}
%%%%%%%%%%%%%%%%%%%%%%%%%%%%%%%%%%%%%%%%%%%%%%%%%%%%

We next define a subdistribution of $\NSDist$ which will correspond to the rolling with neither slipping nor spinning.
Recall that the no-spinning distribution $\NSDist$ defined on $Q$
models the fact that the admissible curves $t\mapsto q(t)=(x(t),\hat{x}(t);A(t))$ inscribed
on $Q$,
i.e., the curves describing the motion of $M$ against $\hat{M}$,
must verify the no-spinning condition (\ref{eq:nospin}).
The latter is equivalent to the condition that $t\mapsto q(t)$
is tangent (a.e.) to $\NSDist$, $\dot{q}(t)=\LNSD\big(\dot{x}(t),\dot{\hat{x}}(t)\big)\big|_{q(t)}$ for a.e. $t$.
As regards the rolling of one manifold onto another one, the admissible curve $q(\cdot)$ must also verify the no-slipping condition (\ref{eq:noslip})
that we recall next.
Since $q(\cdot)$ is tangent to $\NSDist$,
we have $A(t)=P_0^t(x,\hat{x}) A(0)$,
and hence the no-slipping condition (\ref{eq:noslip}) writes
$A(t)\dot{x}(t)=\dot{\hat{x}}(t)$.
It forces one to have, for a.e. $t$,
\[
\dot{q}(t)=\LNSD\big(\dot{x}(t), A(t)\dot{x}(t)\big)\big|_{q(t)}.
\]
Evaluating at $t=0$ and noticing that if $q_0:=q(0)$, with $q_0=(x_0,\hat{x}_0;A_0)\in Q$
and $\dot{x}(0)=:X\in T|_{x_0} M$ are arbitrary, we get
\[
\dot{q}(0)=\LNSD(X, A_0 X)|_{q_0}.
\]
This motivates the following definition.

\begin{definition}\label{def:LRD}
For $q=(x,\hat{x};A)\in Q$, we define the \emph{Rolling lift} or $\RDist$-lift as a bijective linear map
\[
\LRD:T|_x M\times Q|_{(x,\hat{x})}\to T|_q Q,
\]
given by
\begin{align}\label{eq:2.5:3}
\LRD(X,q)=\LNSD(X, AX)|_q.
\end{align}
\end{definition}

This map naturally induces $\LRD:\VF(M)\to \VF(Q)$ as follows.
For $X\in\VF(M)$ we define $\LRD(X)$, the \emph{Rolling lifted} vector field associated to $X$, by
\[
 \LRD(X):Q&\to T(Q), \\
 q&\mapsto \LRD(X)|_q,
\]
where $\LRD(X)|_q:=\LRD(X,q)$.

The Rolling lift map $\LRD$ allows one to construct a distribution on $Q$ (see \cite{bor06}) 
reflecting both of the rolling restrictions of motion
defined by the no-spinning condition, Eq. (\ref{eq:nospin}), and
the no-slipping condition, Eq. (\ref{eq:noslip}).

\begin{definition}\label{def:2.5:1}\label{def:rdist}
The \emph{rolling distribution} $\RDist$ on $Q$ is
the $n$-dimensional smooth distribution defined pointwise by
\begin{align}\label{eq:2.5:1}
\RDist|_{(x,\hat{x};A)}=\LRD(T|_x M)|_{(x,\hat{x};A)},
\end{align}
for $(x,\hat{x};A)\in Q$.
\end{definition}

The Rolling lift $\LRD$ will also be called $\RDist$-lift
since it maps vectors of $M$ to vectors in $\RDist$.
Thus an absolutely continuous curve $t\mapsto q(t)=(x(t),\hat{x}(t);A(t))$
in $Q$ is a rolling curve if and only if it is
a.e. tangent to $\RDist$
i.e., $\dot{q}(t)\in\RDist|_{q(t)}$ for a.e. $t$
or, equivalently,
if $\dot{q}(t)=\LRD(\dot{x}(t))|_{q(t)}$ for a.e. $t$.

Define $\pi_{Q,M}=\pr_1\circ \pi_Q:Q\to M$
and notice that its differential $(\pi_{Q,M})_*$
maps each $\RDist|_{(x,\hat{x};A)}$, $(x,\hat{x};A)\in Q$, isomorphically onto $T|_x M$.
This implies the following standard result.

\begin{proposition}\label{pr:rolling_curves}
For any $q_0=(x_0,\hat{x}_0;A_0)\in Q$ and absolutely continuous $\gamma:[0,a]\to M$, $a>0$,
such that $c(0)=x_0$,
there exists a unique absolutely continuous $q:[0,a']\to Q$, $q(t)=(\gamma(t),\hat{\gamma}(t);A(t))$, with $0<a'\leq a$ (and $a'$ maximal with the latter property),
which is tangent to $\RDist$ a.e. and $q(0)=q_0$.
We denote this unique curve $q$ by
\[
t\mapsto q_{\RDist}(\gamma,q_0)(t)=(\gamma(t),\hat{\gamma}_{\RDist}(\gamma,q_0)(t);A_{\RDist}(\gamma,q_0)(t)),
\]
and refer to it as the \emph{rolling curve} with initial conditions $(\gamma,q_0)$
or \emph{along $\gamma$ with initial position $q_0$}.
In the case that $\hat{M}$ is a complete manifold one has $a'=a$.

Conversely, any absolutely continuous curve $q:[0,a]\to Q$,
which is a.e. tangent to $\RDist$, is a rolling curve
along $\gamma=\pi_{Q,M}\circ q$
i.e., has the form $q_{\RDist}(\gamma,q(0))$.
\end{proposition}

\begin{proof}
We need to show only that completeness of $(\hat{M},\hat{g})$
implies that $a'=a$.
In fact, $\hat{X}(t):=A_0\int_0^t P^0_s(\gamma)\dot{\gamma}(s)\diff s$
defines an a.c. curve $t\mapsto \hat{X}(t)$ in $T|_{\hat{x}_0} \hat{M}$ defined on $[0,a]$
and the completeness of $\hat{M}$ implies that
there is a unique a.c. curve $\hat{\gamma}$ on $\hat{M}$
defined on $[0,a]$ such that $\hat{X}(t)=\int_0^t P^0_s(\hat{\gamma})\dot{\hat{\gamma}}(s)\diff s$
for all $t\in [0,a]$
(see also Remark \ref{re:Lambda} below).
Defining $A(t)=P_0^t(\hat{\gamma})\circ A_0\circ P_t^0(\gamma)$, $t\in [0,a]$
(parallel transports are always defined on the same interval as the a.c. curve along which
the parallel transport takes place)
we notice that $t\mapsto (\gamma(t),\hat{\gamma}(t);A(t))$
is the rolling curve along $\gamma$ starting at $q_0$ that is defined on the interval $[0,a]$.
Hence $a'=a$. 

\end{proof}

Of course, it is not important in the previous result that we start
the parametrization of the curve $\gamma$ at $t=0$.

\begin{remark}\label{re:group_property_of_rolling}
It follows immediately from the uniqueness statement
of the previous theorem that,
if $\gamma:[a,b]\to M$ and $\omega:[c,d]\to M$
are two a.c. curves
with $\gamma(b)=\omega(c)$
and $q_0\in Q$, then
\begin{align}
q_{\RDist}(\omega\sqcup \gamma,q_0)
=q_{\RDist}(\omega,q_{\RDist}(\gamma,q_0)(b))\sqcup q_{\RDist}(\gamma,q_0).
\end{align}

On the group $\Omega_{x_0}(M)$ of piecewise differentiable loops of $M$ based at $x_0$
one has
\[
q_{\RDist}(\omega. \gamma,q_0)
=q_{\RDist}(\omega,q_{\RDist}(\gamma,q_0)(1)). q_{\RDist}(\gamma,q_0),
\]
where $\gamma,\omega\in \Omega_{x_0}(M)$.
\end{remark}

Specializing to $(M,g)$ and $(\hat{M},\hat{g})$,
we will write in the sequel $\Lambda_{x_0}$ and $\hat{\Lambda}_{\hat{x}_0}$
for $\Lambda^{\nabla}_{x_0}$ and $\hat{\Lambda}^{\hat{\nabla}}_{\hat{x}_0}$ respectively,
where $x_0\in M$, $\hat{x}_0\in \hat{M}$.

\begin{remark}\label{re:Lambda}
It follows from Proposition \ref{pr:2.1:1}
that, for $q_0=(x_0,\hat{x}_0;A_0)$
and an a.c. curve $\gamma$ starting from $x_0$,
the corresponding rolling curve is given by
\begin{equation}\label{rol-curve}
q_{\RDist}(\gamma,q_0)(t)=(\gamma(t),\hat{\Lambda}_{\hat{x}_0}^{-1}(A_0\circ \Lambda_{x_0}(\gamma))(t);
P^t_0\big(\hat{\Lambda}_{\hat{x}_0}^{-1}(A_0\circ \Lambda_{x_0}(\gamma))\big)\circ A_0\circ P_t^0(\gamma)\big).
\end{equation}
\end{remark}

In the case where the curve $\gamma$ on $M$ is a geodesic,
we can give a more precise form of the rolling curve along $\gamma$
with a given initial position.

\begin{proposition}\label{pr:rol_geodesic}
Consider $q_0=(x_0,\hat{x}_0;A_0)\in Q$, $X\in T|_{x_0} M$
and $\gamma:[0,a]\to M$; $\gamma(t)=\exp_{x_0}(tX)$, a geodesic of $(M,g)$
with $\gamma(0)=x_0$, $\dot{\gamma}(0)=X$.
Then the rolling curve $q_{\RDist}{(\gamma,q_0)}=(\gamma,\hat{\gamma}_{\RDist}{(\gamma,q_0)};A_{\RDist}{(\gamma,q_0)}):[0,a']\to Q$, $0<a'\leq a$, along $\gamma$ with initial position $q_0$
is given by
\[
\hat{\gamma}_{\RDist}{(\gamma,q_0)}(t)=\widehat{\exp}_{\hat{x}_0}(tA_0 X),
\quad A_{\RDist}{(\gamma,q_0)}(t)=P_0^t(\hat{\gamma}_{\RDist}{(\gamma,q_0)})\circ A_0\circ P_t^0(\gamma).
\]
Of course, $a'=a$ if $\hM$ is complete.
\end{proposition}

\begin{proof}
Let $0<a'\leq a$ such that $\hat{\gamma}(t):=\widehat{\exp}_{\hat{x}_0}(tA_0 X)$
is defined on $[0,a']$.
Then, by proposition \ref{pr:2.1:1}, $q(t):=(\gamma(t),\hat{\gamma}(t);A(t))$ with
$A(t):=P_0^t(\hat{\gamma})\circ A_0\circ P_t^0(\gamma)$, $t\in [0,a']$,
is a curve on $Q$
and $A(t)$ is parallel to $(\gamma,\hat{\gamma})$ in $M\times \hat{M}$.
Therefore $t\mapsto q(t)$ is tangent to $\NSDist$ on $[0,a']$
and thus $\dot{q}(t)=\LNSD(\dot{\gamma}(t), \dot{\hat{\gamma}}(t))|_{q(t)}$.
Moreover, since $\gamma$ and $\hat{\gamma}$ are geodesics,
\[
A(t)\dot{\gamma}(t)
=(P_0^t(\hat{\gamma})\circ A_0)(P_t^0(\gamma)\dot{\gamma}(t))
=P_0^t(\hat{\gamma})(A_0X)=\dot{\hat{\gamma}}(t),
\]
which shows that for $t\in [0,a']$,
\begin{align}
\dot{q}(t)
&=\LNSD(\dot{\gamma}(t), A(t)\dot{\gamma}(t))|_{q(t)}\nonumber\\
&=\LRD(\dot{\gamma}(t))\big|_{q(t)}.\nonumber
\end{align}
Hence $t\mapsto q(t)$ is tangent to $\RDist$
i.e., it is a rolling curve
along $\gamma$ with initial position $q(0)=(\gamma(0),\hat{\gamma}(0);A(0))=(x_0,\hat{x}_0;A_0)=q_0$.

\end{proof}

\begin{remark}
If $\gamma(t)=\exp_{x_0}(tA_0 X)$ and $q_0=(x_0,\hat{x}_0;A_0)$, the statement of the proposition can be written in a compact form as
\[
A_{\RDist}{(\gamma,q_0)}(t)
=P_0^t\big(s\mapsto \ol{\exp}_{(x_0,\hat{x}_0)}(s(X, A_0X))\big)A_0,
\]
for all $t$ where defined.
\end{remark}

The next proposition describes the symmetry of the study of the rolling problem
of $(M,g)$ rolling against $(\hat{M},\hat{g})$
to the problem of $(\hat{M},\hat{g})$ rolling against $(M,g)$.

\begin{proposition}\label{pr:rol_inverse}
Let $\widehat{\RDist}$ be the rolling distribution in $\hat{Q}:=Q(\hat{M},M)$.
Then the map
\[
\iota:Q\to\hat{Q};\quad \iota(x,\hat{x};A)=(\hat{x},x;A^{-1})
\]
is a diffeomorphism of $Q$ onto $\hat{Q}$ and
\[
\iota_*\RDist=\widehat{\RDist}.
\]
In particular, $\iota(\mc{O}_{\RDist}(q))=\mc{O}_{\widehat{\RDist}}(\iota(q))$.
\end{proposition}

\begin{proof}
It is obvious that $\iota$ is a diffeomorphism (with the obvious inverse map)
and for an a.c. path $q(t)=(\gamma(t),\hat{\gamma}(t);A(t))$
in $Q$, $(\iota\circ q)(t)=(\hat{\gamma}(t),\gamma(t);A(t)^{-1})$
is a.c. in $\hat{Q}$ and for a.e. $t$,
\[
\begin{cases}
\dot{\hat{\gamma}}(t)=A(t)\dot{\gamma}(t) \cr
A(t)=P_0^{t}(\hat{\gamma})\circ A(0)\circ P_t^0(\gamma)
\end{cases}
\quad\iff\quad
\begin{cases}
\dot{\gamma}(t)=A(t)^{-1}\dot{\hat{\gamma}}(t) \cr
A(t)^{-1}=P_0^{t}(\gamma)\circ A(0)^{-1}\circ P_t^0(\hat{\gamma})
\end{cases}.
\]
These simple remarks prove the claims.

\end{proof}

\begin{remark}
Notice that Definitions \ref{def:LRD} and \ref{def:2.5:1}
make sense not only in $Q$ but also in the space $T^*M\otimes T\hat{M}$.
It is easily seen that $\RDist$ defined on $T^*M\otimes T\hat{M}$
by Eq. (\ref{eq:2.5:1}) is actually tangent to $Q$ so
its restriction to $Q$ gives exactly $\RDist$ on $Q$ as defined above.
Similarly, Propositions \ref{pr:rolling_curves}, \ref{pr:rol_geodesic} and \ref{pr:rol_inverse}
still hold if we replace $Q$ by $T^*M\otimes T\hat{M}$
and $\hat{Q}$ by $T^*\hat{M}\otimes TM$
everywhere in their statements.
\end{remark}

%%%%%%%%%%%%%%%%%%%%%%%%%%%%%%%%%%%%%%%%%%%%%%%%%%%%
\subsection{Lie brackets of vector fields on $Q$}\label{sec:2.3}
%%%%%%%%%%%%%%%%%%%%%%%%%%%%%%%%%%%%%%%%%%%%%%%%%%%%

In this section, we compute commutators of the vectors fields of $T^*M\otimes T\hat{M}$ and $Q$
with respect to the splitting of $T(T^*M\otimes T\hat{M})$ (resp. $TQ$)
as a direct sum $\NSDist\oplus V(\pi_{T^*M\otimes T\hat{M}})$ (resp. $\NSDist\oplus V(\pi_Q)$)
as given in Remark \ref{re:tangent_splitting} above.
The main results are Propositions \ref{pr:NS_comm_HH}, \ref{pr:NS_comm_HH} and \ref{pr:NS_comm_VV}. These 
computations will serve as preliminaries for the Lie bracket computations relative to the rolling distribution 
$\RDist$ studied in the next section.
It is convenient to make the computations in $T^*M\otimes T\hat{M}$
and then to simply restrict the results to $Q$.

%%%%%%%%%%%%%%%%%%%%%%%%%%%%%%%%%%%%%%%%
\subsubsection{Computational tools}
%%%%%%%%%%%%%%%%%%%%%%%%%%%%%%%%%%%%%%%%

The next lemmas will be useful in the subsequent calculations.

\begin{lemma}\label{le:nice_extension_A}
Let $(x,\hat{x};A)\in T^*M\otimes T\hat{M}$ (resp. $(x,\hat{x};A)\in Q$). Then there exists a local $\pi_{T^*M\otimes T\hat{M}}$-section (resp. $\pi_Q$-section)
$\tilde{A}$ around $(x,\hat{x})$ such that $\tilde{A}|_{(x,\hat{x})}=A$ and $\ol{\nabla}_{\ol{X}} \tilde{A}=0$
for all $\ol{X}\in T|_{(x,\hat{x})} (M\times\hat{M})$.
\end{lemma}

\begin{proof}
Let $U$ be an open neighborhood of the origin of $T|_{(x,\hat{x})} (M\times\hat{M})$, where
the $\ol{g}$-exponential map $\ol{\exp}:U\to M\times\hat{M}$ is a diffeomorphism onto its image.
Parallel translate $A$ along geodesics $t\mapsto \ol{\exp}(t\ol{X})$, $\ol{X}\in U$,
to get a local section $\tilde{A}$ of $T^*(M)\otimes T(\hat{M})$ in a neighborhood of $\ol{x}=(x,\hat{x})$.
More explicitly, one has
\[
\tilde{A}|_{\ol{y}}=P_0^1\big(t\mapsto \ol{\exp}\big(t(\ol{\exp}_{\ol{x}})^{-1}(\ol{y})\big)\big)A,
\]
for $\ol{y}\in U$.
If $(x,\hat{x};A)\in Q$, this actually provides a local $\pi_Q$-section.
Moreover, we clearly have $\ol{\nabla}_{\ol{X}} \tilde{A}=0$ for all $\ol{X}\in T|_{(x,\hat{x})} (M\times\hat{M})$.

\end{proof}

Notice that the choice of $\tilde{A}$ corresponding to $(x,\hat{x};A)$ is, of course, not unique.

\begin{lemma}\label{le:comm_nabla_A}
Let $\tilde{A}$ be a smooth local $\pi_{T^* M\otimes T\hat{M}}$-section
and $\tilde{A}|_{(x,\hat{x})}=A$. Then, for any vector fields $\ol{X},\ol{Y}\in \VF(M\times\hat{M})$
such that $\ol{X}|_{(x,\hat{x})}=(X, \hat{X})$, $\ol{Y}|_{(x,\hat{x})}=(Y, \hat{Y})$,
one has
\begin{align}\label{eq:2.3:1b}
([\ol{\nabla}_{\ol{X}},\ol{\nabla}_{\ol{Y}}] \tilde{A})|_{(x,\hat{x})}
=-AR(X,Y)+\hat{R}(\hat{X},\hat{Y})A+(\ol{\nabla}_{[\ol{X},\ol{Y}]}\tilde{A})|_{(x,\hat{x})}.
\end{align}
Here $[\ol{\nabla}_{\ol{X}},\ol{\nabla}_{\ol{Y}}]$ is given by $\ol{\nabla}_{\ol{X}}\circ \ol{\nabla}_{\ol{Y}}-\ol{\nabla}_{\ol{Y}}\circ \ol{\nabla}_{\ol{X}}$ and is 
an $\R$-linear map on the set of local sections of $\pi_{T^* M\otimes T\hat{M}}$ around $(x,\hat{x})$.
\end{lemma}

\begin{proof}
For an arbitrary $Z\in \VF(M)$, which we may interpret as a vector field on $M\times\hat{M}$ as usual, we calculate
\[
& ([\ol{\nabla}_{\ol{X}},\ol{\nabla}_{\ol{Y}}] \tilde{A})Z
=\ol{\nabla}_{\ol{X}}((\ol{\nabla}_{\ol{Y}} \tilde{A})Z)-(\ol{\nabla}_{\ol{Y}}\tilde{A})(\ol{\nabla}_{\ol{X}} Z)
-\ol{\nabla}_{\ol{Y}}((\ol{\nabla}_{\ol{X}} \tilde{A})Z)+(\ol{\nabla}_{\ol{X}}\tilde{A})(\ol{\nabla}_{\ol{Y}} Z) \\
=&\ol{\nabla}_{\ol{X}}(\ol{\nabla}_{\ol{Y}}(\tilde{A}Z)-\tilde{A}\ol{\nabla}_{\ol{Y}} Z)-(\ol{\nabla}_{\ol{Y}}\tilde{A})(\ol{\nabla}_{\ol{X}} Z)\\
&-\ol{\nabla}_{\ol{Y}}(\ol{\nabla}_{\ol{X}}(\tilde{A}Z)-\tilde{A}\ol{\nabla}_{\ol{X}} Z)+(\ol{\nabla}_{\ol{X}}\tilde{A})(\ol{\nabla}_{\ol{Y}} Z) \\
=&[\ol{\nabla}_{\ol{X}},\ol{\nabla}_{\ol{Y}}](\tilde{A}Z)-(\ol{\nabla}_{\ol{X}}\tilde{A})(\ol{\nabla}_{\ol{Y}}Z)-\tilde{A}\ol{\nabla}_{\ol{X}}(\ol{\nabla}_{\ol{Y}} Z)
-(\ol{\nabla}_{\ol{Y}}\tilde{A})(\ol{\nabla}_{\ol{X}} Z) \\
&+(\ol{\nabla}_{\ol{Y}}\tilde{A})(\ol{\nabla}_{\ol{X}}Z)+\tilde{A}\ol{\nabla}_{\ol{Y}}(\ol{\nabla}_{\ol{X}} Z)+(\ol{\nabla}_{\ol{X}}\tilde{A})(\ol{\nabla}_{\ol{Y}} Z) \\
=&[\ol{\nabla}_{\ol{X}},\ol{\nabla}_{\ol{Y}}](\tilde{A}Z)+\tilde{A}[\ol{\nabla}_{\ol{Y}},\ol{\nabla}_{\ol{X}}]Z \\
=&\ol{R}(\ol{X},\ol{Y})(\tilde{A}Z)+\ol{\nabla}_{[\ol{X},\ol{Y}]}(\tilde{A}Z)+\tilde{A}(\ol{R}(\ol{Y},\ol{X})Z)+\tilde{A}\ol{\nabla}_{[\ol{Y},\ol{X}]} Z \\
=&-\tilde{A}(\ol{R}(\ol{X},\ol{Y})Z)+\ol{R}(\ol{X},\ol{Y})(\tilde{A}Z)+(\ol{\nabla}_{[\ol{X},\ol{Y}]}\tilde{A})Z,
\]
and evaluating the above quantity at $(x,\hat{x})$, we get
\[
([\ol{\nabla}_{\ol{X}},\ol{\nabla}_{\ol{Y}}] \tilde{A})Z|_{(x,\hat{x})}
=-A(R(X,Y)Z)+\hat{R}(\hat{X},\hat{Y})(AZ)+(\ol{\nabla}_{[\ol{X},\ol{Y}]}\tilde{A})Z|_{(x,\hat{x})}.
\]
Since the value $Z|_{x}$ can be chosen arbitrarily in $T|_x M$, the claim follows.

\end{proof}

We next define the actions of vectors $\LNSD(\ol{X})|_q\in T|_q(T^*M\otimes T\hat{M})$, $\ol{X}\in T|_{(x,\hat{x})} (M\times\hat{M})$, and
$\nu(B)|_q\in V|_q(\pi_{T^*M\otimes T\hat{M}})$, $B\in T|_{x}^*M\otimes T|_{\hat{x}}\hat{M}$, on certain bundle maps
instead of just functions (e.g. from $\Cinf(T^*M\otimes T\hat{M})$).
Recall that if $\eta:E\to N$ is a vector bundle and $y\in N$, $u\in E|_y=\eta^{-1}(y)$,
we have defined the isomorphism
\[
\nu_\eta|_u:E|_{y}\to V|_u(\eta);\quad \nu_\eta|_u(v)(f)=\dif{t}\big|_0 f(u+tv),\quad \forall f\in\Cinf(E).
\]
We normally omit the index $\eta$ in $\nu_{\eta}$, when it is clear from the context, and simply write $\nu$ instead of $\nu_{\eta}$ and
it is sometimes more convenient to write $\nu(v)|_u$ for $\nu|_u(v)$. 
By using this we make the following definition.

\begin{definition}\label{def:general_vert}
Suppose $B$ is a smooth manifold, $\eta:E\to N$ a vector bundle,
$\tau:B\to N$ and $F:B\to E$ smooth maps such that $\eta\circ F=\tau$.
Then, for $b\in B$ and $\mc{V}\in V|_b(\tau)$, we define
the vertical derivative of $F$ as
\[
\mc{V}F:=\nu|_{F(b)}^{-1}(F_*\mc{V})\in E|_{\tau(b)}.
\]
\end{definition}

This is well defined since $F_*\mc{V}\in V|_{F(b)}(\eta)$.
In this matter, we will show the following simple lemma that will be used later on.

\begin{lemma}\label{le:bundle_extension}
Let $N$ be a smooth manifold, $\eta:E\to N$ a vector bundle, $\tau:B\to N$ a smooth map,
$\mc{O}\subset B$ an immersed submanifold
and $F:\mc{O}\to E$ a smooth map such that $\eta\circ F=\tau|_{\mc{O}}$.

\begin{itemize}
\item[(i)] Then for every $b_0\in \mc{O}$,
there exists an open neighbourhood $V$ of $b_0$ in $\mc{O}$,
an open neighbourhood $\tilde{V}$ of $b_0$ in $B$ such that $V\subset \tilde{V}$
and a smooth map $\tilde{F}:\tilde{V}\to E$ such that $\eta\circ \tilde{F}=\tau|_{\tilde{V}}$
and $\tilde{F}|_V=F|_V$.
We call $\tilde{F}$ a local extension of $F$ around $b_0$.

\item[(ii)] Suppose $\tau:B\to N$ is also a vector bundle
and $\tilde{F}$ is any local extension of $F$ around $b_0$ as in case (i).
Then if $v\in B|_{\tau(b_0)}$ is such that $\nu|_{b_0}(v)\in T|_{b_0}\mc{O}$,
one has
\[
\nu|_{b_0}(v)(F)=\dif{t}\big|_0 \tilde{F}(b_0+tv) \in E|_{\tau(b_0)},
\]
where on the right hand side one views $t\mapsto \tilde{F}(b_0+tv)$
as a map into a fixed (i.e. independent of $t$) vector space $E|_{F(b_0)}$
and the derivative $\dif{t}$ is just the classical derivative of a vector valued map
(and not a tangent vector).
\end{itemize}
\end{lemma}

\begin{proof}
(i) For a given $b_0\in\mc{O}$, take a neighbourhood $W$ of $y_0:=\tau(b_0)$
in $N$
such that there exists a local frame $v_1,\dots,v_k$ of $\eta$ defined on $W$
(here $k=\dim E-\dim N$).
Since $\eta\circ F=\tau|_{\mc{O}}$, it follows that
\[
F(b)=\sum_{i=1}^k f_i(b)v_i|_{\tau(b)},\quad \forall b\in \tau^{-1}(W)\cap \mc{O},
\]
for some smooth functions $f_i:\tau^{-1}(W)\cap \mc{O}\to\R$, $i=1,\dots,k$.
Now one can choose a small open neighbourhood $V$ of $b_0$ in $\mc{O}$ 
and an open neigbourhood $\tilde{V}$ of $b_0$ in $B$
such that $V\subset \tilde{V}\subset \tau^{-1}(W)$ and there exist smooth $\tilde{f}_1,\dots,\tilde{f}_k:\tilde{V}\to\R$
extending the functions $f_i|_V$ i.e. $\tilde{f}_i|_V=f_i|_V$, $i=1,\dots,k$.
To finish the proof of case (i), it suffices to define $\tilde{F}:\tilde{V}\to E$ by
\[
\tilde{F}(b)=\sum_{i=1}^k \tilde{f}_i(b)v_i|_{\tau(b)},\quad \forall b\in \tilde{V}.
\]

(ii)
The fact that $t\mapsto \tilde{F}(b_0+tv)$ is a map into a fixed vector space $E|_{F(b_0)}$
is clear since $\tilde{F}(b_0+tv)\in E|_{\eta(\tilde{F}(b_0+tv))}=E|_{\tau(b_0+tv)}=E|_{\tau(b_0)}$.
Since $F|_V=\tilde{F}|_V$ and $\nu|_{b_0}(v)\in T|_{b_0} V$, we have
$F_*\nu|_{b_0}(v)=\tilde{F}_*\nu|_{b_0}(v)$.
Also, $t\mapsto b_0+tv$ is a curve in $E|_{\tau(b_0)}$, and hence in $E$,
whose tangent vector at $t=0$ is exactly $\nu|_{b_0}(v)$.
Hence
\[
\nu|_{F(b_0)}(\nu|_{b_0}(v)F)=F_*\nu|_{b_0}(v)
=\tilde{F}_*\nu|_{b_0}(v)=\dif{t}\big|_0 \tilde{F}(b_0+tv).
\]
Here on the rightmost side, the derivative $=:T$ is still viewed as a tangent vector of $E$ at $\tilde{F}(b_0)$
i.e. $t\mapsto \tilde{F}(b_0+tv)$ is thought of as a map into $E$.
On the other hand, if one views $t\mapsto \tilde{F}(b_0+tv)$ as a map into a fixed linear space $E|_{\tau(b_0)}$,
its derivative $=:D$ at $t=0$, as the usual derivative of vector valued maps,
is just $D=\nu_{F(b_0)}^{-1}(T)$.
In the statement, it is exactly $D$ whose expression we wrote as $\dif{t}\big|_0 \tilde{F}(b_0+tv)$.
This completes the proof.
\end{proof}

\begin{remark}
The advantage of the formula in case (ii) of the above lemma is that
it simplifies in many cases the computations of $\tau$-vertical derivatives
because $t\mapsto \tilde{F}(b_0+tv)$ is a map from a real interval into a \emph{fixed}
vector space $E|_{F(b_0)}$
and hence we may use certain computational tools (e.g. Leibniz rule)
coming from the ordinary vector calculus.
\end{remark}

Let $\mc{O}$  be an immersed submanifold of $T^*M\otimes T\hat{M}$
and write $\pi_{\mc{O}}:=\pi_{T^*M\otimes T\hat{M}}|_{\mc{O}}$.
Then if $\ol{T}:\mc{O}\to T^k_m(M\times \hat{M})$ with $\pi_{T^k_m(M\times\hat{M})}\circ\ol{T}=\pi_{\mc{O}}$
(i.e. $\ol{T}\in \Cinf(\pi_{\mc{O}},\pi_{T^k_m(M\times\hat{M})})$)
and if $q=(x,\hat{x};A)\in \mc{O}$ and $\ol{X}\in T|_{(x,\hat{x})} (M\times \hat{M})$
are such that $\LNSD(\ol{X})|_q\in T|_q\mc{O}$, we next want to define
what it means to take the derivative $\LNSD(\ol{X})|_{q}\ol{T}$.
Our main interest will be the case where $k=1$, $m=0$ i.e. $T^k_m(M\times \hat{M})=T(M\times\hat{M})$,
but some arguments below require this slightly more general setting.

First, for a moment, we take $\mc{O}=T^*M\otimes T\hat{M}$.
Choose some local $\pi_{T^*M\otimes T\hat{M}}$-section $\tilde{A}$ defined on a neighbourhood of $(x,\hat{x})$
such that $\tilde{A}|_{(x,\hat{x})}=A$ and define
\begin{align}\label{eq:semi-general_LNSD}
\LNSD(\ol{X})|_q\ol{T}:=\ol{\nabla}_{\ol{X}}(\ol{T}(\tilde{A}))-\nu(\ol{\nabla}_{\ol{X}} \tilde{A})|_q\ol{T}\in T^k_m|_{(x,\hat{x})}(M\times\hat{M}),
\end{align}
which is inspired by Eq. (\ref{eq:LNSD_formula}).
Here as usual, $\tilde{T}(\tilde{A})=\tilde{T}\circ\tilde{A}$
is a locally defined $(k,m)$-tensor field on $M\times\hat{M}$.

Notice that this does not depend on the choice of $\tilde{A}$
since if $\ol{\omega} \in \Gamma(\pi_{T^m_k(M\times\hat{M})})$
and if we write $(\ol{T}\ol{\omega})(q):=\ol{T}(q)\ol{\omega}|_{(x,\hat{x})}$ as a full contraction
for $q=(x,\hat{x};A)\in T^*M\otimes T\hat{M}$, whence $\ol{T}\ol{\omega}\in\Cinf(T^*M\otimes T\hat{M})$,
we may compute (where all the contractions are full)
\[
(\LNSD(\ol{X})|_q\ol{T})\ol{\omega}=&\big(\ol{\nabla}_{\ol{X}}(\ol{T}(\tilde{A}))\big)\ol{\omega}
-\big(\dif{t}\big|_0 \ol{T}(A+t\ol{\nabla}_{\ol{X}} \tilde{A})\big)\ol{\omega} \nonumber \\
=&\ol{\nabla}_{\ol{X}}(\ol{T}(\tilde{A})\ol{\omega})
-\ol{T}(q)\ol{\nabla}_{\ol{X}}\ol{\omega}
-\dif{t}\big|_0 \big(\ol{T}(A+t\ol{\nabla}_{\ol{X}} \tilde{A})\ol{\omega}\big) \nonumber \\
=& \ol{\nabla}_{\ol{X}}((\ol{T}\ol{\omega}))(\tilde{A})\big)
-\dif{t}\big|_0 (\ol{T}\ol{\omega})(A+t\ol{\nabla}_{\ol{X}} \tilde{A})
-\ol{T}(q) \ol{\nabla}_{\ol{X}}\ol{\omega} \nonumber \\
\]
i.e.
\begin{align}\label{eq:general_LNSD}
(\LNSD(\ol{X})|_q\ol{T})\ol{\omega}=&\LNSD(\ol{X})|_q(\ol{T}\ol{\omega})-\ol{T}(q)\ol{\nabla}_{\ol{X}}\ol{\omega},
\end{align}
for all $\ol{\omega}\in\Gamma(\pi_{T^m_k(M\times\hat{M})})$
and where $\LNSD(\ol{X})|_q$ on the right hand side acts as a tangent vector to a function $\ol{T}\ol{\omega}\in\Cinf(T^*M\otimes T\hat{M})$
as defined in subsection \ref{sec:2.1}.

Now the right hand side is know to be independent of any choice of local extension $\tilde{A}$ of $A$
(i.e. $\tilde{A}|_{(x,\hat{x})}=A$),
it follows that the definition of $\LNSD(\ol{X})|_q\ol{T}$ is independent of this choice as well.
Alternatively, we could have taken Eq. (\ref{eq:general_LNSD}) as the definition
of $\LNSD(\ol{X})|_q\ol{T}$.

Now if $\mc{O}\subset T^*M\otimes T\hat{M}$ is just an immersed submanifold,
we take the formula \eqref{eq:general_LNSD} as the definition of $\LNSD(\ol{X})|_q\ol{T}$.

\begin{definition}\label{def:general_LNSD}
Let $\mc{O}\subset T^*M\otimes T\hat{M}$ be an immersed submanifold and $q=(x,\hat{x};A)\in \mc{O}$, $\ol{X}\in T|_{(x,\hat{x})} (M\times\hat{M})$
be such that $\LNSD(\ol{X})|_q\in T|_q \mc{O}$.
Then for $\ol{T}:\mc{O}\to T^k_m(M\times \hat{M})$ such that $\pi_{T^k_m(M\times\hat{M})}\circ\ol{T}=\pi_{\mc{O}}$,
we define $\LNSD(\ol{X})|_q\ol{T}$ to be the unique element in $T^k_m|_{(x,\hat{x})}(M\times\hat{M})$
such that Eq. (\ref{eq:general_LNSD}) holds for every $\ol{\omega}\in\Gamma(\pi_{T^m_k(M\times\hat{M})})$,
and call it the derivative of $\ol{T}$ with respect to $\LNSD(\ol{X})|_q$.
\end{definition}

We now to provide the (unique) decomposition of any vector field
of $T^*M\otimes T\hat{M}$
defined over $\mc{O}$ (not necessarily tangent to it) according to the direct sum \eqref{eq:2.1:4}
i.e. $T(T^*M\otimes T\hat{M})=\NSDist\oplus V(\pi_{T^*M\otimes T\hat{M}})$.

\begin{proposition}\label{pr:decomposition_of_VF}
Let $\mc{X}\in \Cinf(\pi_{\mc{O}}, \pi_{T(T^*M\otimes T\hat{M})})$ be a smooth bundle map
(i.e. a vector field of $T^*M\otimes T\hat{M}$ along $\mc{O}$)
where $\mc{O}\subset T^*M\otimes T\hat{M}$ is a smooth immersed submanifold.
Then there are unique smooth bundle maps
$\ol{T}\in C^\infty(\pi_{\mc{O}},\pi_{T(M\times\hat{M})})$,
$U\in C^\infty(\pi_{\mc{O}},\pi_{T^*M\otimes T\hat{M}})$
such that
\begin{align}
\mc{X}|_{q}=\LNSD(\ol{T}(q))|_{q}+\nu(U(q))|_{q},
\quad q\in \mc{O}.
\end{align}
\end{proposition}

\begin{proof}
First of all, there are unique smooth vector fields 
$$\mc{X}^h,\mc{X}^v\in 
\Cinf(\pi_{\mc{O}}, \pi_{T(T^*M\otimes T\hat{M}})),$$
of $T^*M\otimes T\hat{M}$ along $\mc{O}$
such that 
\[
\mc{X}^h|_{q}\in \NSDist|_{q},\ 
\mc{X}^v|_{q}\in V|_{q}(\pi_{T^*M\otimes T\hat{M}}),
\]
for all $q\in \mc{O}$
and $\mc{X}=\mc{X}^h+\mc{X}^v$.
Then, we define
\[
\ol{T}(q)=(\pi_{T^*M\otimes T\hat{M}})_*\mc{X}^h|_{q},\quad
U(q)=\nu|_{q}^{-1}(\mc{X}^v|_{q}),
\]
where $q=(x,\hat{x};A)\in \mc{O}$ and $\nu|_{q}$ is the isomorphism 
\begin{align}
T^*|_xM\otimes T|_{\hat{x}}\hat{M}\to V|_{q} (\pi_{T^*M\otimes T\hat{M}});\quad
B\mapsto \nu(B)|_{q}.\nonumber
\end{align}
This clearly proves the claims.
\end{proof}

\begin{remark}
The previous results shows that to know how to compute
the Lie brackets of two vector fields $\mc{X},\mc{Y}\in\VF(\mc{O})$
where $\mc{O}\subset T^*M\otimes T\hat{M}$ is an immersed submanifold (e.g. $\mc{O}=Q$),
one needs, in practice, just to know how to compute the Lie brackets
between vectors fields of the form $q\mapsto \LNSD(\ol{T}(q))|_q,\LNSD(\ol{S}(q))$
and $q\mapsto \nu(U(q))|_q,\nu(V(q))|_q$
where $\mc{X}|_q= \LNSD(\ol{T}(q))|_q+\nu(U(q))|_q$
and $\mc{Y}|_q= \LNSD(\ol{S}(q))|_q+\nu(V(q))|_q$
as above.
\end{remark}

\begin{remark}
Notice that if $\mc{O}\subset T^*M\otimes T\hat{M}$ is an immersed
submanifold, $q=(x,\hat{x};A)\in\mc{O}$, $\mc{X}\in T|_q\mc{O}$
and $\ol{T}\in\Cinf(\pi_{\mc{O}},\pi_{T^k_m(M\times\hat{M})})$,
then we may define the derivative $\mc{X}\ol{T}\in T^k_m(M\times\hat{M})$
by decomposing $\mc{X}=\LNSD(\ol{X})|_q+\nu(U)|_q$
for the unique $\ol{X}\in T|_{(x,\hat{x})} (M\times\hat{M})$ and $U\in (T^*M\otimes T\hat{M})|_{(x,\hat{x})}$.
\end{remark}

We finish this subsection with some obvious but useful rules of calculation,
that will be useful in the computations of Lie brackets on $\mc{O}\subset T^*M\otimes T\hat{M}$
and we will make use of them especially in section \ref{se:3D}.

\begin{lemma}\label{le:rules_of_computation}
Let $\mc{O}\subset T^*M\otimes T\hat{M}$ be an immersed submanifold, $q=(x,\hat{x};A)\in\mc{O}$,
$\ol{T}\in\Cinf(\pi_{\mc{O}},\pi_{T^k_m(M\times \hat{M})})$,
$F\in\Cinf(\mc{O})$, $h\in\Cinf(\R)$,
$\ol{X}\in T|_{(x,\hat{x})} (M\times\hat{M})$
such that $\LNSD(\ol{X})|_q\in T|_q\mc{O}$
and finally $U\in (T^*M\times T\hat{M})|_{(x,\hat{x})}$ such that $\nu(U)|_q\in T|_q \mc{O}$
Then
\begin{itemize}
\item[(i)] $\LNSD(\ol{X})|_q(F\ol{T})=(\LNSD(\ol{X})|_qF)\ol{T}(q)+F(q)\LNSD(\ol{X})|_q\ol{T}$
\item[(ii)] $\LNSD(\ol{X})|_q(h\circ F)=h'(F(q))\LNSD(\ol{X})|_q F$
\item[(iii)] $\nu(U)|_q(F\ol{T})=(\nu(U)|_q F)\ol{T}(q)+F(q)\nu(U)|_q \ol{T}$
\item[(iv)] $\nu(U)|_q(h\circ F)=h'(F(q))\nu(U)|_q F$
\end{itemize}

If $T:\mc{O}\to TM\subset T(M\times\hat{M})$ such that $T(q)\in T|_x M$ for all $q=(x,\hat{x};A)\in \mc{O}$
and one writes (see Remark \ref{re:rules_of_computations} below)
\[
(\cdot) T(\cdot):\mc{O}\to T\hat{M}\subset T(M\times\hat{M});\quad q=(x,\hat{x};A)\mapsto AT(q),
\]

then
\begin{itemize}
\item[(v)] $\LNSD(\ol{X})|_q\big((\cdot)T(\cdot)\big)=A\LNSD(\ol{X})|_q T\in T|_{\hat{x}}\hat{M}$
\item[(vi)] $\nu(U)|_q \big((\cdot)T(\cdot)\big)=UT(q)+A\nu(U)|_q T\in T|_{\hat{x}}\hat{M}$,
\end{itemize}
where $\LNSD(\ol{X})|_q T,\nu(U)|_q T\in T|_x M$.

Finally, if $Y\in\VF(M)$ is considered as a map $\mc{O}\to TM;\ (x',\hat{x}';A')\mapsto Y|_{x'}$
and if we write $\ol{X}=(X,\hat{X})\in T|_{x} M\oplus T|_{\hat{x}}\hat{M}$,
then
\begin{itemize}
\item[(vii)] $\LNSD(\ol{X})|_qY=\nabla_X Y$.
\end{itemize}
\end{lemma}

\begin{remark}\label{re:rules_of_computations}
\begin{itemize}
\item[(a)] In the cases (v) and (vii)
we think of $T:\mc{O}\to TM$, to adapt to our previous notations,
as a map $T:\mc{O}\to \pr_1^*(TM)$
where $\pr_1:M\times\hat{M}\to M$
is the projection onto the first factor.
Here $\pr_1^*(\pi_{TM})$ is a vector subbundle of $\pi_{T(M\times\hat{M})}$ 
which we wrote, slightly imprecisely, as $TM\subset T(M\times \hat{M})$
in the statement of the proposition.
Thus $T(q')\in T|_{x'} M$ for all $q'=(x',\hat{x}';A')\in\mc{O}$
just means that $\pr_1^*(\pi_{TM})\circ T=\pi_{\mc{O}}$.

\item[(b)]
We could write a more extensive list of rules of computation
by noticing that $\LNSD(\ol{X})|_q$ and $\nu(U)|_q$
act by Leibniz-rule to any contraction of $\ol{T}$ and $\ol{S}$
where $\ol{T}\in\Cinf(\pi_{\mc{O}},\pi_{T^k_m(M\times \hat{M})})$,
$\ol{S}\in\Cinf(\pi_{\mc{O}},\pi_{T^{k'}_{m'}(M\times \hat{M})})$.
The rules (i)-(vii) are, though, sufficient for our needs.
\end{itemize}
\end{remark}

\begin{proof}
In what follows, we choose 
a small open neighbourhood $V$ of $q$ in $\mc{O}$,
a small open neighbourhood $\tilde{V}$ of $q$ in $T^*M\otimes T\hat{M}$ such that $V\subset \tilde{V}$
a smooth $\tilde{\ol{T}}:\tilde{V}\to T(M\times \hat{M})$ such that $\tilde{\ol{T}}|_V=\ol{T}|_V$ and
$\pi_{T(M\times \hat{M})}\circ \tilde{\ol{T}}=\pi_{T^*M\otimes T\hat{M}}|_{\tilde{V}}$
and a smooth map $\tilde{F}:\tilde{V}\to\R$ such that $\tilde{F}|_V=F|_V$.
These are provided by Lemma \ref{le:bundle_extension}.

For the case (i) we take some $\ol{\omega}\in \Gamma(T^m_k(M\times\hat{M}))$ and
compute
\[
& (\LNSD(\ol{X})|_q(F\ol{T}))\ol{\omega}
=\LNSD(\ol{X})|_q(F\ol{T}\ol{\omega})-F(q)\ol{T}(q)\ol{\nabla}_{\ol{X}}\ol{\omega} \\
=&\LNSD(\ol{X})|_q(F)\ol{T}(q)\ol{\omega}|_{(x,\hat{x})} 
+F(q)\LNSD(\ol{X})|_q(\ol{T}\ol{\omega})
-F(q)\ol{T}(q)\ol{\nabla}_{\ol{X}}\ol{\omega} \\
=&\big(\LNSD(\ol{X})|_q(F)\ol{T}(q)+F(q)\LNSD(\ol{X})|_q\ol{T}\big)\ol{\omega}|_{(x,\hat{x})}.
\]

For (ii), take $t\mapsto \Gamma(t)=(\gamma(t),\hat{\gamma}(t);A(t))$
be any curve in $\mc{O}$ with $\Gamma(0)=q$, $\dot{\Gamma}(0)=\LNSD(\ol{X})|_q$.
Since $\LNSD(\ol{X})|_q\in T|_q\mc{O}=T|_q V$, we may compute that
\[
& \LNSD(\ol{X})|_q(h\circ F)=\LNSD(\ol{X})|_q(h\circ \tilde{F}) \\
=&\dif{t}\big|_0 (h\circ \tilde{F})(A(t))-\dif{t}\big|_0(h\circ \tilde{F})\big(A+t\nabla_{\ol{X}}A(\cdot)\big) \\
=&h'(\tilde{F}(q))\dif{t}\big|_0 \tilde{F}(A(t))-h'(\tilde{F}(q))\dif{t}\big|_0\tilde{F}\big(A+t\nabla_{\ol{X}}A(\cdot)\big) \\
=&h'(\tilde{F}(q))\LNSD(\ol{X})|_q(\tilde{F})=h'(F(q))\LNSD(\ol{X})|_q(F).
\]

To prove (iii), notice that $(\tilde{F}\tilde{\ol{T}})|_V=(F\ol{T})|_V$
and hence
\[
& \nu(U)|_q(F\ol{T})=\nu(U)|_q(\tilde{F}\tilde{\ol{T}})
=\dif{t}\big|_0\tilde{F}(A+tU)\tilde{\ol{T}}(A+tU) \\
=&\big(\dif{t}\big|_0\tilde{F}(A+tU)\big)\tilde{\ol{T}}(q)
+\tilde{F}(q)\dif{t}\big|_0\tilde{\ol{T}}(A+tU) \\
=&(\nu(U)|_q \tilde{F})\ol{T}(q)+F(q)\nu(U)|_q \tilde{\ol{T}}
=(\nu(U)|_q F)\ol{T}(q)+F(q)\nu(U)|_q \ol{T}.
\]

To prove (iv), take a curve $\Gamma$ in $\mc{O}$ with $\Gamma(0)=q$, $\dot{\Gamma}(0)=\nu(U)|_q$
and compute
\[
\nu(U)|_q(h\circ F)=\dif{t}\big|_0 h(F(\Gamma(t)))=h'(F(q))\dif{t}\big|_0 F(\Gamma(t))=h'(F(q))\nu(U)|_q(F).
\]

Let us prove (v) and (vi).
We take a small open neighbourhood $V$ of $q$ in $\mc{O}$,
a small open neighbourhood $\tilde{V}$ of $q$ in $T^*M\otimes T\hat{M}$ such that $V\subset \tilde{V}$
a smooth $\tilde{T}:\tilde{V}\to TM$ such that $\tilde{T}|_V=T|_V$
and $\tilde{T}(q')\in T|_{x'} M$ for all $q'=(x',\hat{x}';A')\in \tilde{V}$.
Such an extension $\tilde{T}$ of $T$ is
provided by Lemma \ref{le:bundle_extension} by taking $b_0=q$, $\tau=\pi_{T^*M\otimes T\hat{M}}$,
$\eta=\pr_1^*(\pi_{TM})$ with $\pr_1:M\times\hat{M}=M$
the projection onto the first factor (see also Remark \ref{re:rules_of_computations} above).
Then taking $t\mapsto \Gamma(t)=(\gamma(t),\hat{\gamma}(t);A(t))$
to be any curve in $\mc{O}$ with $\Gamma(0)=q$, $\dot{\Gamma}(0)=\LNSD(\ol{X})|_q$,
we have
\[
& \LNSD(\ol{X})|_q((\cdot)T(\cdot))
=\LNSD(\ol{X})|_q((\cdot)\tilde{T}(\cdot)) \\
=&\ol{\nabla}_{\ol{X}} (A(\cdot)\tilde{T}(A(\cdot)))
-\dif{t}\big|_0 (A+t\ol{\nabla}_{\ol{X}}A(\cdot))\tilde{T}(A+t\ol{\nabla}_{\ol{X}}A(\cdot)) \\
=&(\ol{\nabla}_{\ol{X}} A(\cdot))\tilde{T}(q)+A\ol{\nabla}_{\ol{X}} (\tilde{T}(A(\cdot)))
-(\ol{\nabla}_{\ol{X}}A(\cdot))\tilde{T}(q)-A\dif{t}\big|_0 \tilde{T}(A+t\ol{\nabla}_{\ol{X}}A(\cdot)) \\
=&A\LNSD(\ol{X})|_q\tilde{T}=A\LNSD(\ol{X})|_qT
\]
where the first and the last steps follow from the facts that $((\cdot)\tilde{T}(\cdot))|_V=((\cdot)T(\cdot))|_V$
and $\tilde{T}|_V=T|_V$.
This gives (v).

To prove (vi) we observe that $(\cdot)\tilde{T}(\cdot):\tilde{V}\to T\hat{M}$
satisfies $((\cdot)\tilde{T}(\cdot))|_{V}=((\cdot)T(\cdot))|_V$ which allows us to compute
\[
& \nu(U)|_q\big((\cdot)T(\cdot)\big)
=\nu(U)|_q\big((\cdot)\tilde{T}(\cdot)\big)
=\dif{t}\big|_0 (A+tU)\tilde{T}(A+tU) \\
=&\big(\dif{t}\big|_0 (A+tU)\big)\tilde{T}(q)+A\dif{t}\big|_0 \tilde{T}(A+tU)
=UT(q)+A\nu(U)|_q\tilde{T} \\
=&UT(q)+A\nu(U)|_q T.
\]

Finally, we prove (vii). Suppose that $Y\in\VF(M)$.
Then the map $\mc{O}\to TM;\ (x',\hat{x}';A')\mapsto Y|_{x'}$ is nothing more than $Y\circ \pr_1\circ \pi_{\mc{O}}$
where $\pr_1:M\times\hat{M}\to M$ is the projection onto the first factor.
Take a local $\pi_{T^*M\otimes T\hat{M}}$-section $\tilde{A}$ with $\tilde{A}|_{(x,\hat{x})}=A$.
Then since $Y\circ \pr_1\circ \pi_{\mc{O}}=Y\circ \pr_1\circ \pi_{T^*M\otimes T\hat{M}}|_{\mc{O}}$,
we have
\[
&\LNSD(\ol{X})|_q(Y\circ \pr_1\circ \pi_{\mc{O}})
=\LNSD(\ol{X})|_q(Y\circ \pr_1\circ \pi_{T^*M\otimes T\hat{M}}) \\
=&\ol{\nabla}_{(X,\hat{X})}(Y\circ \pr_1\circ \pi_{T^*M\otimes T\hat{M}}\circ\tilde{A})-\dif{t}\big|_0 (Y\circ \pr_1\circ \pi_{T^*M\otimes T\hat{M}})(A+t\ol{\nabla}_{\ol{X}}\tilde{A}).
\]
But $(Y\circ \pr_1\circ \pi_{T^*M\otimes T\hat{M}}\circ\tilde{A})|_{(x',\hat{x}')}=Y|_{x'}=(Y,0)|_{(x,\hat{x})}$ for all $(x',\hat{x}')$
and $(Y\circ \pr_1\circ \pi_{T^*M\otimes T\hat{M}})(A+t\ol{\nabla}_{\ol{X}}\tilde{A})=Y|_x$ for all $t$ and hence
\[
\LNSD(\ol{X})|_q(Y\circ \pr_1\circ \pi_{\mc{O}})=\ol{\nabla}_{(X,\hat{X})} (Y,0)-0=\nabla_X Y.
\]

\end{proof}

\subsubsection{Computation of Lie brackets}

We now embark into the computation of Lie brackets. 
\begin{proposition}\label{pr:NS_comm_HH}
Let $\mc{O}\subset T^*M\otimes T\hat{M}$ be an immersed submanifold,
$\ol{T}=(T,\hat{T}),\ol{S}=(S,\hat{S})\in\Cinf(\pi_{\mc{O}},\pi_{T(M\times\hat{M})})$
with $\LNSD(\ol{T}(q))|_q,\LNSD(\ol{S}(q))|_q\in T|_q\mc{O}$ for all $q=(x,\hat{x};A)\in \mc{O}$.
Then, for every $q\in \mc{O}$, one has
\begin{align}
[\LNSD(\ol{T}(\cdot)),\LNSD(\ol{S}(\cdot))]|_q=&\LNSD\big(\LNSD(\ol{T}(q))|_q\ol{S}-\LNSD(\ol{S}(q))|_q\ol{T}\big)\big|_q \nonumber \\
&+\nu\big(AR(T(q),S(q))-\hat{R}(\hat{T}(q),\hat{S}(q))A\big)\big|_q,
\end{align}
with both sides tangent to $\mc{O}$.
\end{proposition}

\begin{proof}
We will deal first with the case where $\mc{O}$ is an open subset of $T^*M\otimes T\hat{M}$.
Take a local $\pi_{T^*M\otimes T\hat{M}}]$ section $\tilde{A}$ around $(x,\hat{x})$
such that $\tilde{A}|_{(x,\hat{x})}=A$, $\ol{\nabla} \tilde{A}|_{(x,\hat{x})}=0$; see Lemma \ref{le:nice_extension_A}. 
In some expressions we will write $q=A$ for clarity.

Let $f\in C^\infty(T^* M\otimes T\hat{M})$.  By using the definition of $\LNSD$ and $\nu$, one obtains
\[
& \LNSD(\ol{T}(A))|_q(\LNSD(\ol{S}(\cdot)(f))) \\
=&\ \ol{T}(A)(\LNSD(\ol{S}(\tilde{A}))|_{\tilde{A}}(f))-\dif{t}\big|_0 \LNSD(\ol{S}(A+t\ol{\nabla}_{\ol{T}(A)} \tilde{A}))|_{A+t\ol{\nabla}_{\ol{T}(A)}\tilde{A}} (f) \\
=&\ \ol{T}(A)\big(\ol{S}(\tilde{A})(f(\tilde{A}))-\dif{t}\big|_0 f(\tilde{A}+t\ol{\nabla}_{\ol{S}(\tilde{A})}\tilde{A})\big)\\
&-\dif{t}\big|_0 \ol{S}(A+t\ol{\nabla}_{\ol{T}(A)} \tilde{A})(f(\tilde{A}+t\ol{\nabla}_{\ol{T}(\tilde{A})}\tilde{A})) \\
& +\frac{\partial^2}{\partial t\partial s}\big|_0 f\big(A+t\ol{\nabla}_{\ol{T}(A)}\tilde{A}+s\ol{\nabla}_{\ol{S}(A+t\ol{\nabla}_{\ol{T}(A)} \tilde{A})} (\tilde{A}+t\ol{\nabla}_{\ol{T}(\tilde{A})} \tilde{A}))\big).
\]
Here, we use the fact that $\ol{\nabla}_{\ol{X}}\tilde{A}=0$ for all $\ol{X}\in T|_{(x,\hat{x})} (M\times\hat{M})$
and the fact that $\frac{\partial}{\partial t}$ and $\ol{T}(\tilde{A})$ commute (as the obvious vector fields on $M\times\hat{M}\times \R$
with points $(x,\hat{x},t)$)
to write the last expression in the form
\[
&\ol{T}(A)(\ol{S}(\tilde{A})(f(\tilde{A}))\big)-\dif{t}\big|_0 \ol{T}(A)(f(\tilde{A}+t\ol{\nabla}_{\ol{S}(\tilde{A})}\tilde{A}))
-\dif{t}\big|_0 \ol{S}(A)(f(\tilde{A}+t\ol{\nabla}_{\ol{T}(\tilde{A})}\tilde{A})) \\
& +\frac{\partial^2}{\partial t\partial s}\big|_0 f\big(A+st\ol{\nabla}_{\ol{S}(A)} (\ol{\nabla}_{\ol{T}(\tilde{A})} \tilde{A}))\big).
\]
By interchanging the roles of $\ol{T}$ and $\ol{S}$ and using the definition of commutator of vector fields, we get from this
\[
& [\LNSD(\ol{T}(\cdot)),\LNSD(\ol{S}(\cdot))]|_q(f) \\
=&[\ol{T}(\tilde{A}),\ol{S}(\tilde{A})]|_q(f(\tilde{A})) 
+\frac{\partial^2}{\partial t\partial s}\big|_0 f\big(A+st\ol{\nabla}_{\ol{S}(A)} (\ol{\nabla}_{\ol{T}(\tilde{A})} \tilde{A}))\big) \\
&-\frac{\partial^2}{\partial t\partial s}\big|_0 f\big(A+st\ol{\nabla}_{\ol{T}(A)} (\ol{\nabla}_{\ol{S}(\tilde{A})} \tilde{A}))\big) \\
=&[\ol{T}(\tilde{A}),\ol{S}(\tilde{A})]|_q(f(\tilde{A}))+\dif{t}\big|_0 \nu(t\ol{\nabla}_{\ol{S}(A)} (\ol{\nabla}_{\ol{T}(\tilde{A})} \tilde{A}))|_q(f) \\
& -\dif{t}\big|_0 \nu(t\ol{\nabla}_{\ol{T}(A)} (\ol{\nabla}_{\ol{S}(\tilde{A})} \tilde{A}))|_q(f) \\
=&[\ol{T}(\tilde{A}),\ol{S}(\tilde{A})]|_q(f(\tilde{A}))+\nu(\ol{\nabla}_{\ol{S}(A)} (\ol{\nabla}_{\ol{T}(\tilde{A})} \tilde{A}))|_q(f)-\nu(\ol{\nabla}_{\ol{T}(A)} (\ol{\nabla}_{\ol{S}(\tilde{A})} \tilde{A}))|_q(f) \\
=&[\ol{T}(\tilde{A}),\ol{S}(\tilde{A})]|_q(f(\tilde{A}))
-\nu([\ol{\nabla}_{\ol{T}(\tilde{A})}, \ol{\nabla}_{\ol{S}(\tilde{A})}] \tilde{A}))|_q(f).
\] 
Using Lemma \ref{le:comm_nabla_A},  we get that the last line is equal to
\[
& [\ol{T}(\tilde{A}),\ol{S}(\tilde{A})]|_{(x,\hat{x})}(f(\tilde{A})) \\
& -\nu(\ol{\nabla}_{[\ol{T}(\tilde{A}),\ol{S}(\tilde{A})]|_{(x,\hat{x})}}\tilde{A}-AR(T(A),S(A))+\hat{R}(\hat{T}(A),\hat{S}(A))A)|_q(f),
\]
from which, by using the definition of $\LNSD$, linearity of $\nu(\cdot)|_q$ and arbitrariness
of $f\in\Cinf(T^*M\otimes T\hat{M})$, we get
\[
& [\LNSD(\ol{T}(\cdot)),\LNSD(\ol{S}(\cdot))]|_q=\LNSD([\ol{T}(\tilde{A}),\ol{S}(\tilde{A})])|_q \\
&-\nu(AR(T(A),S(A))-\hat{R}(\hat{T}(A),\hat{S}(A))A)|_q.
\]
Finally,
\[
\dif{t}\big|_0 \ol{S}(A+t\underbrace{\ol{\nabla}_{\ol{T}(q)} \tilde{A}}_{=0})=\dif{t}\big|_0 \ol{S}(A)=0, \\
\dif{t}\big|_0 \ol{T}(A+t\underbrace{\ol{\nabla}_{\ol{S}(q)} \tilde{A}}_{=0})=\dif{t}\big|_0 \ol{T}(A)=0,
\]
since $\ol{T}(q),\ol{S}(q)\in T|_{(x,\hat{x})} (M\times\hat{M})$
and hence by Eq. \eqref{semi-general_LNSD},
\[
[\ol{T}(\tilde{A}),\ol{S}(\tilde{A})]=\ol{\nabla}_{\ol{T}(q)}(\ol{S}(\tilde{A}))-\ol{\nabla}_{\ol{S}(q)}(\ol{T}(\tilde{A}))
=\LNSD(\ol{T}(q))|_q\ol{S}-\LNSD(\ol{S}(q))|_q\ol{T}.
\]
The claim thus holds in this case (i.e. when $\mc{O}$ is an open subset of $T^*M\otimes T\hat{M}$).

Now we let $\mc{O}\subset T^*M\otimes T\hat{M}$ to be an immersed submanifold
and $\ol{T},\ol{S}:\mc{O}\to T(M\times\hat{M})$
are such that, for all $q=(x,\hat{x};A)\in \mc{O}$,
$\ol{T}(x,\hat{x};A)$, $\ol{S}(x,\hat{x};A)$ belong to $T|_{(x,\hat{x})}(M\times\hat{M})$ and $\LNSD(\ol{T}(q))|_q$, $\LNSD(\ol{S}(q))|_q$ belong to $T|_q \mc{O}$.

For a fixed $q=(x,\hat{x};A)\in \mc{O}$, we may, thanks to Lemma \ref{le:bundle_extension}
(taking $\tau=\pi_{T^*M\otimes T\hat{M}}$, $\eta=\pi_{T(M\times\hat{M})}$, $b_0=q$
and $F=\ol{T}$ or $F=\ol{S}$ there)
take a small open neighbourhood $V$ of $q$ in $\mc{O}$,
a neighbourhood $\tilde{V}$ of $q$ in $Q$
such that $V\subset \tilde{V}$
and some extensions $\tilde{\ol{T}},\tilde{\ol{S}}:\tilde{V}\to T(M\times\hat{M})$
of $\ol{T}|_V,\ol{S}|_V$
with $\tilde{\ol{T}}(x',\hat{x}';A'),\tilde{\ol{S}}(x',\hat{x}';A')\in T|_{(x',\hat{x}')}(M\times\hat{M})$ for all $(x',\hat{x}';A')\in \tilde{V}$.
Then since $\LNSD(\ol{T}(\cdot))|_V=\LNSD(\tilde{\ol{T}}(\cdot))|_V$,
$\LNSD(\ol{S}(\cdot))|_V=\LNSD(\tilde{\ol{S}}(\cdot))|_V$,
we may compute, because of what has been shown already,
\[
& [\LNSD(\ol{T}),\LNSD(\ol{S})]|_q
=[\LNSD(\tilde{\ol{T}})|_V,\LNSD(\tilde{\ol{S}})|_V]|_q
=([\LNSD(\tilde{\ol{T}}),\LNSD(\tilde{\ol{S}})]|_V)|_q \\
=&\LNSD\big(\LNSD(\ol{T}(q))|_q\tilde{\ol{S}}-\LNSD(\ol{S}(q))|_q\tilde{\ol{T}}\big)\big|_q
+\nu\big(AR(T(q),S(q))-\hat{R}(\hat{T}(q),\hat{S}(q))A\big)\big|_q,
\]
where in the last line we used that $\tilde{\ol{T}}(q)=\ol{T}(q)=(T(q),\hat{T}(q))$,
$\tilde{\ol{S}}(q)=\ol{S}(q)=(S(q),\hat{S}(q))$.

Take any $\ol{\omega}\in\Gamma(\pi_{T^m_k(M\times\hat{M})})$.
Since $\LNSD(\ol{T}(q))|_q\in T|_q\mc{O}=T|_q V$ by assumption
and since $(\ol{S}\ol{\omega})|_V=(\tilde{\ol{S}}\ol{\omega})|_V$,
we have $\LNSD(\ol{T}(q))|_q(\ol{S}\ol{\omega})=\LNSD(\ol{T}(q))|_q(\tilde{\ol{S}}\ol{\omega})|_V$.
But then Eq. (\ref{eq:general_LNSD}) i.e. the definition of
$\LNSD(\ol{T}(q))|_q\ol{S}$ implies that
\[
&(\LNSD(\ol{T}(q))|_q\ol{S})\ol{\omega}
=\LNSD(\ol{T}(q))|_q(\ol{S}\ol{\omega})-\ol{S}(q)\ol{\nabla}_{\ol{T}(q)}\ol{\omega} \\
=&\LNSD(\ol{T}(q))|_q(\tilde{\ol{S}}\ol{\omega})-\tilde{\ol{S}}(q)\ol{\nabla}_{\ol{T}(q)}\ol{\omega}
=(\LNSD(\ol{T}(q))|_q\tilde{\ol{S}})\ol{\omega}
\]
i.e. $\LNSD(\ol{T}(q))|_q\ol{S}=\LNSD(\ol{T}(q))|_q\tilde{\ol{S}}$
and similarly $\LNSD(\ol{S}(q))|_q\ol{T}=\LNSD(\ol{S}(q))|_q\tilde{\ol{T}}$. This shows
that on $\mc{O}$ we have the formula
\[
[\LNSD(\ol{T}),\LNSD(\ol{S})]|_q=&\LNSD\big(\LNSD(\ol{T}(q))|_q \ol{S}-\LNSD(\ol{S}(q))|_q \ol{T}\big)\big|_q \\
&+\nu\big(AR(T(q),S(q))-\hat{R}(\hat{T}(q),\hat{S}(q))A\big)\big|_q,
\]
where both sides belong to $T|_q \mc{O}$ (since the left hand side obviously belongs to $T|_q\mc{O}$).

\end{proof}

\begin{proposition}\label{pr:NS_comm_HV}
Let $\mc{O}\subset T^*M\otimes T\hat{M}$ be an immersed submanifold,
$\ol{T}=(T,\hat{T})\in\Cinf(\pi_{\mc{O}},\pi_{T(M\times\hat{M})})$,
$U\in\Cinf(\pi_{\mc{O}},\pi_{T^*M\otimes T\hat{M}})$
be such that, for all $q=(x,\hat{x};A)\in\mc{O}$,
$$
\LNSD(\ol{T}(q))|_q\in T|_q\mc{O}, \quad \nu(U(q))|_q\in T|_q\mc{O}.
$$
Then
\[
[\LNSD(\ol{T}(\cdot)),\nu(U(\cdot))]|_q
=-\LNSD(\nu(U(q))|_q\ol{T})|_q+\nu(\LNSD(\ol{T}(q))|_q U)|_q,
\]
with both sides tangent to $\mc{O}$.
\end{proposition}

\begin{proof}
As in the proof of Proposition \ref{pr:NS_comm_HH},
we will deal first with the case where $\mc{O}$ is an open subset of $T^*M\otimes T\hat{M}$.
Take a local $\pi_{T^*M\otimes T\hat{M}}]$ section $\tilde{A}$ around $(x,\hat{x})$
such that $\tilde{A}|_{(x,\hat{x})}=A$, $\ol{\nabla} \tilde{A}|_{(x,\hat{x})}=0$; see Lemma \ref{le:nice_extension_A}. 
In some expressions we will write $q=A$ for clarity.

Let $f\in\Cinf(T^* M\otimes T\hat{M})$. Then
$\LNSD(\ol{T}(A))|_q\big(\nu(U(\cdot))(f)\big)$ is equal to 
$$
\ol{T}(A)\big(\nu(U(\tilde{A}))\big|_{\tilde{A}}(f)\big)-\dif{t}\big|_0 \nu(U(A+t\ol{\nabla}_{\ol{T}(A)} \tilde{A}))\big|_{A+t\ol{\nabla}_{\ol{T}(A)} \tilde{A}}(f),$$
which is equal to $\ol{T}(A)\big(\nu(U(\tilde{A}))\big|_{\tilde{A}}(f)\big)$
once we recall that $\ol{\nabla}_{\ol{T}(A)} \tilde{A}=0$.
In addition, one has
\[
&\nu(U(A))|_q\big(\LNSD(\ol{T}(\cdot))(f)\big) =\dif{t}\big|_0 \LNSD(\ol{T}(A+tU(A))\big|_{A+tU(A)}(f) \\
=&\dif{t}\big|_0 \ol{T}(A+tU(A))\big(f(\tilde{A}+tU(\tilde{A}))\big)\\&-\frac{\partial^2}{\partial s\partial t}\big|_0 f\big(A+tU(A)+s\ol{\nabla}_{\ol{T}(A+tU(A))} (\tilde{A}+tU(\tilde{A}))\big) \\
=&\dif{t}\big|_0 \ol{T}(A+tU(A))\big(f(\tilde{A}+tU(\tilde{A}))\big)-\frac{\partial^2}{\partial s\partial t}\big|_0 f\big(A+tU(A)+st\ol{\nabla}_{\ol{T}(A+tU(A))} (U(\tilde{A}))\big),
\]
since $\ol{\nabla}_{\ol{T}(A+tU(A))} \tilde{A}=0$.

We next simplify the first term on the last line to get
\[
&\dif{t}\big|_0 \ol{T}(A+tU(A))\big(f(\tilde{A}+tU(\tilde{A}))\big)\\
=&(\nu(U(q))|_q\ol{T})\big(f(\tilde{A})\big)+\ol{T}(A)\big(\nu(U(\tilde{A}))|_{\tilde{A}}(f)\big)
\]
and then, for the second term, one obtains
\[
&\frac{\partial^2}{\partial s\partial t}\big|_0 f\big(A+tU(A)+st\ol{\nabla}_{\ol{T}(A+tU(A))} (U(\tilde{A}))\big)\\
=&\dif{s}\big|_0 f_*|_q\nu\Big(\dif{t}\big|_0 \big(tU(A)+st\ol{\nabla}_{\ol{T}(A+tU(A))} (U(\tilde{A}))\big)\Big)\Big|_q \\
=&\dif{s}\big|_0 f_*|_q\nu\big(U(A)+s\ol{\nabla}_{\ol{T}(A)} (U(\tilde{A}))\big)\big)\big|_q \\
=&\dif{s}\big|_0 \Big(f_*|_q\nu(U(A))|_q+sf_*|_q\nu\big(\ol{\nabla}_{\ol{T}(A)} (U(\tilde{A}))\big)|_q\Big) \\
=&f_*\nu\big(\ol{\nabla}_{\ol{T}(A)} (U(\tilde{A}))\big)|_q
=\nu(\ol{\nabla}_{\ol{T}(A)} (U(\tilde{A})))|_q f.
\]

Therefore one deduces
\[
&[\LNSD(\ol{T}(\cdot)),\nu(U(\cdot))]|_q(f)
=-(\nu(U(q))|_q\ol{T})\big(f(\tilde{A})\big)+\nu(\ol{\nabla}_{\ol{T}(A)} (U(\tilde{A})))|_q f \\
=&-\tilde{A}_*(\nu(U(A))|_q\ol{T})(f)+\nu\big(\ol{\nabla}_{\ol{T}(A)} (U(\tilde{A}))\big)\big|_q(f) \\
=&-\LNSD(\nu(U(A))|_q\ol{T})|_q(f)+\nu\big(\ol{\nabla}_{\ol{T}(A)} (U(\tilde{A}))\big)\big|_q(f),
\]
where the last line follows from the definition of $\LNSD$ and the fact that 
\newline $\ol{\nabla}_{\nu(U(A))|_q\ol{T}}\tilde{A}=0$. 
Finally, Eq. (\ref{eq:semi-general_LNSD}) implies
\[
\ol{\nabla}_{T(q)} (U(\tilde{A}))=\ol{\nabla}_{T(q)} (U(\tilde{A}))-\nu(\underbrace{\ol{\nabla}_{T(q)}\tilde{A}}_{=0})|_q U
=\LNSD(\ol{T}(A))|_qU.
\]
Thus the claimed formula holds in the special case where $\mc{O}$ is an open
subset of $T^*M\otimes T\hat{M}$.

More generally, let $\mc{O}\subset T^*M\otimes T\hat{M}$ be an immersed submanifold,
and $\ol{T}=(T,\hat{T}):\mc{O}\to T(M\times\hat{M})=TM\times T\hat{M}$,
$U:\mc{O}\to T^*M\times T\hat{M}$
as in the statement of this proposition.

For a fixed $q=(x,\hat{x};A)\in \mc{O}$,
Lemma \ref{le:bundle_extension} implies the existence of a neighbourhood $V$ of $q$ in $\mc{O}$,
a neighbourhood $\tilde{V}$ of $q$ in $T^*M\otimes T\hat{M}$
and smooth $\tilde{\ol{T}}:\tilde{V}\to T(M\times \hat{M})$, $\tilde{U}:\tilde{V}\to T^*M\otimes T\hat{M}$
such that $\tilde{\ol{T}}(x,\hat{x};A)\in T|_{(x,\hat{x})} (M\times \hat{M})$,
$\tilde{U}(x,\hat{x};A)\in (T^*M\otimes T\hat{M})|_{(x,\hat{x})}$
and $\tilde{\ol{T}}|_V=\ol{T}|_V$, $\tilde{U}|_V=U|_V$
(for the case of an extension $\tilde{U}$ of $U$,
take in Lemma \ref{le:bundle_extension}, $\tau=\pi_{T^*M\otimes T\hat{M}}$, $\eta=\pi_{T_1^1(M\times\hat{M})}$, $F=U$, $b_0=q$).

In the same way as in the proof of Proposition \ref{le:NS_comm_HH},
we have $[\LNSD(\ol{T}),\nu(U)]|_q=[\LNSD(\tilde{\ol{T}}),\nu(\tilde{U})]_q$
and $\LNSD(\tilde{\ol{T}}(q))|_q\tilde{U}=\LNSD(\ol{T}(q))|_q U$.
Hence by what was already shown above, 
\[
[\LNSD(\ol{T}),\nu(U)]|_q=-\LNSD(\nu(U(q))|_q\tilde{\ol{T}})|_q+\nu(\LNSD(\ol{T}(q))|_q U)|_q.
\]

We are left to show that $\nu(U(q))|_q\tilde{\ol{T}}=\nu(U(q))|_q\ol{T}$
and for that, it suffices to show that $\nu(\nu(U(q))|_q\tilde{\ol{T}})|_{\ol{T}(q)}=\nu(\nu(U(q))|_q\ol{T})|_{\ol{T}(q)}$.
Indeed, if $f\in\Cinf(T(M\times \hat{M}))$, then
\[
\nu(\nu(U(q))|_q\tilde{\ol{T}})|_{\ol{T}(q)} f
=&\big(\tilde{\ol{T}}_*\nu(U(q))|_q\big)f
=\nu(U(q))|_q(f\circ \tilde{\ol{T}})
=\nu(U(q))|_q(f\circ \ol{T}) \\
=&\big(\ol{T}_*\nu(U(q))|_q\big)f
=\nu(\nu(U(q))|_q\ol{T})|_{\ol{T}(q)} f,
\]
where at the 3rd equality we used the fact that $(f\circ \tilde{\ol{T}})|_V=(f\circ \ol{T})|_V$
and $\nu(U(q))|_q\in T|_q\mc{O}=T|_q V$.
This completes the proof.

\end{proof}

Finally, we derive a formula for the commutators of two vertical vector fields.

\begin{proposition}\label{pr:NS_comm_VV}
Let $\mc{O}\subset T^*M\otimes T\hat{M}$ be an immersed submanifold
and $U,V\in C^\infty(\pi_{\mc{O}},\pi_{T^* M\otimes T\hat{M}})$
be such that $\nu(U(q))|_q,\nu(V(q))|_q\in T|_q\mc{O}$
for all $q\in \mc{O}$.
Then
\begin{align}\label{eq:NS_comm_VV}
[\nu(U(\cdot)),\nu(V(\cdot))]|_q
=&\nu\big(\nu(U(q))|_q V-\nu(V(q))|_q U\big)|_q.
\end{align}
\end{proposition}

\begin{proof}
Again we begin with the case where $\mc{O}$ is an open subset of $T^*M\otimes T\hat{M}$
and write $q=(x,\hat{x};A)\in\mc{O}$ simply as $A$.
Let $f\in C^\infty(T^*M\otimes T\hat{M})$. Then,
\[
& \nu(U(A))|_q\big(\nu(V(\cdot))(f)\big)
=\dif{t}\big|_0\nu(V(A+tU(A))|_{A+tU(A)} (f) \\
=&\frac{\partial^2}{\partial t\partial s}\big|_0 f(A+tU(A)+sV(A+tU(A))) \\
=&\dif{s}\big|_0 f_*|_q\nu\Big(\dif{t}\big|_0 \big(tU(A)+sV(A+tU(A))\big)\Big)\Big|_q \\
=&\dif{s}\big|_0 f_*\nu\big(U(A)+s\nu(U(A))|_q V\big)|_q \\
=&f_*\nu\big(\nu(U(A))|_q V\big)|_q=\nu(\nu(U(A))|_q V)|_q f.
\]
from which the result follows
in the case that $\mc{O}$ is an open subset of $T^*M\otimes T\hat{M}$.

The case where $\mc{O}$ is only an immersed submanifold of $T^*M\otimes T\hat{M}$
can be treated by using Lemma \ref{le:bundle_extension}
in the same way as in the proofs of Propositions \ref{pr:NS_comm_HH}, \ref{pr:NS_comm_HV}.

\end{proof}

As a corollary to the previous three propositions, we have the following.

\begin{corollary}
Let $\mc{O}\subset T^*M\otimes T\hat{M}$ be an immersed
submanifold and $\mc{X},\mc{Y}\in\VF(\mc{O})$.
Letting for $q\in\mc{O}$,
\[
\mc{X}|_q=\LNSD(\ol{T}(q))|_q+\nu(U(q))|_q,
\quad
\mc{Y}|_q=\LNSD(\ol{S}(q))|_q+\nu(V(q))|_q,
\]
to be the unique decompositions given by Proposition \ref{pr:decomposition_of_VF}.
Writing $\ol{T}=(T,\hat{T})$, $\ol{S}=(S,\hat{S})$
corresponding to $T(M\times\hat{M})=TM\times T\hat{M}$, we get
\[
[\mc{X},\mc{Y}]|_q=&
\big(\LNSD(\mc{X}|_q\ol{S})|_q+\nu(\mc{X}|_qV)|_q\big)
-\big(\LNSD(\mc{Y}|_q\ol{T})|_q+\nu(\mc{Y}|_qU)|_q\big) \\
&+\nu\big(AR(T(q),S(q))-\hat{R}(\hat{T}(q),\hat{S}(q))A\big)|_q
\]
(for the notation, see the second remark after Proposition \ref{pr:decomposition_of_VF}).
\end{corollary}

\begin{proof}
We will assume that $\ol{T},\ol{S},U,V$ are, temporarily in the course of computations below, extended
by Lemma \ref{le:bundle_extension} to an open neighbourhood $\tilde{O}$
of $q$ in $T^*M\otimes T\hat{M}$.
By abuse of notation, we don't give new names for them.
Then Propositions \ref{pr:NS_comm_HH}, \ref{pr:NS_comm_HV} and \ref{pr:NS_comm_VV}
we obtain
\[
[\mc{X},\mc{Y}]|_q
=&[\LNSD(\ol{T}(\cdot)),\LNSD(\ol{S}(\cdot))]|_q
+[\LNSD(\ol{T}(\cdot)),\nu(V(\cdot))]|_q \\
&+[\nu(U(\cdot)),\LNSD(\ol{S}(\cdot))]|_q
+[\nu(U(\cdot)),\nu(V(\cdot))]|_q \\
=&\LNSD(\LNSD(\ol{T}(q))\ol{S}-\LNSD(\ol{S}(q))\ol{T})|_q \\
&+\nu(AR(T(q),S(q))-\hat{R}(\hat{T}(q),\hat{S}(q))A)|_q \\
&-\LNSD(\nu(V(q))|_q\ol{T})|_q+\nu(\LNSD(\ol{T}(q))|_q V)|_q \\
&+\LNSD(\nu(U(q))|_q\ol{S})|_q-\nu(\LNSD(\ol{S}(q))|_q U)|_q \\
&+\nu(\nu(U(q))|_qV-\nu(V(q))|_q U)|_q \\
=&\LNSD(\mc{X}|_q\ol{S}-\mc{Y}|_q\ol{T})|_q
+\nu(\mc{X}|_qV-\mc{Y}|_qU)|_q \\
&+\nu(AR(T(q),S(q))-\hat{R}(\hat{T}(q),\hat{S}(q))A)|_q.
\]

\end{proof}

\section{Study of the Rolling problem $(NS)$}\label{sec:2.0}

We next parameterize the set of all absolutely continuous curves which are tangent to the distribution $\NSDist$ as a driftless control affine system.

An a.c. curve $t\mapsto q(t)=(x(t),\hat{x}(t);A(t))$ in $Q$
describes a rolling motion of $M$ against $\hat{M}$ without spinning
if and only if $\dot{q}(t)=\LNSD(\dot{x}(t), \dot{\hat{x}}(t))|_{q(t)}$ for a.e. $t$.
This can be expressed equivalently by saying that $q(\cdot)$ is a solution of
a control affine driftless system
\begin{align}\label{eq:cs_nsrolling}\sns\quad 
\begin{cases}
\dot{x}(t)=u(t), \\
\dot{\hat{x}}(t)=\hat{u}(t), \\
\ol{\nabla}_{(u(t),\hat{u}(t))} A(\cdot)=0
\end{cases},\quad\textrm{for a.e. $t\in [a,b]$},
\end{align}
where the control $(u,\hat{u})$
belongs to the set $\mc{U}([a,b],M)\times \mc{U}([a,b],\hat{M})$.
The fact that System (\ref{eq:cs_nsrolling}) is driftless and control affine
can also be seen from its representation in local coordinates; see (\ref{simon}) in Appendix \ref{app:local}.

In the rest of the section, we investigate the structure of the reachable sets associated to $\sns$ 
and relate them to the holonomy groups of the Riemannian manifolds $(M,g)$ and $(\hat{M},\hat{g})$.

\subsection{Description of the Orbits of $\sns$}

We begin this section by recalling some standard definitions
and introducing some notation concerning the subsequent subsections.
If $(N,h)$ is a Riemannian manifold, then the holonomy group $H^{\nabla^h}|_y$
of it at $y$ is defined by
\[
H^{\nabla^h}|_y=\{(P^{\nabla^h})_0^1(\gamma)\ |\ \gamma\in \Omega_y(N)\},
\]
and it is a subgroup $\mathrm{O}(T|_y N)$ of all $h$-orthogonal
transformations of $T|_y N$.
If $N$ is oriented, then one can easily prove that
$H^{\nabla^h}|_y$ is actually a subgroup of $\SO(T|_y N)$.
If $F=(Y_i)_{i=1}^n$, $n=\dim N$, is an orthonormal frame of $N$ at $y$
we write
\[
H^{\nabla^h}|_F=\{\mc{M}_{F,F}(A)\ |\ A\in H^{\nabla^h}|_y\}.
\]
This is a subgroup of $\SO(n)$, isomorphic (as Lie group) to $H^{\nabla^h}|_y$.
Lie algebra of the holonomy group $H^{\nabla^h}|_y$ (resp. $H^{\nabla^h}|_F$)
will be denoted by $\mathfrak{h}^{\nabla^h}|_y$ (resp. $\mathfrak{h}^{\nabla^h}|_F$).
The Lie algebra $\mathfrak{h}^{\nabla^h}|_y$
is a Lie subalgebra of the Lie algebra $\so(T|_y N)$
of $h$-antisymmetric linear maps $T|_y N\to T|_y N$
while $\mathfrak{h}^{\nabla^h}|_F$ is a Lie subalgebra of $\so(n)$.

In this setting, we will be using the notations $H|_x=H^{\nabla}|_x$ and $\hat{H}|_{\hat{x}}=H^{\hat{\nabla}}|_{\hat{x}}$ respectively 
for the holonomy groups of $(M,g)$ and $(\hat{M},\hat{g})$ at $x\in M$ and $\hat{x}\in \hat{M}$.
If $F$ and $\hat{F}$ are respectively orthonormal frames of $M$ and $\hat{M}$ we use 
$H|_F$ and $\hat{H}|_{\hat{F}}$ respectively to denote $H^{\nabla}|_{F}$
and $H^{\hat{\nabla}}|_{\hat{F}}$.
The corresponding Lie algebras will be written as
$\mathfrak{h}|_{x}$, $\hat{\mathfrak{h}}|_{\hat{x}}$, $\mathfrak{h}|_{F}$, $\hat{\mathfrak{h}}|_{\hat{F}}$.

We now describe the structure of the orbit $\mc{O}_{\NSDist}(A_0)$
of $\NSDist$ through a point $(x_0,\hat{x}_0;A_0)\in Q$ as follows.

\begin{theorem}\label{th:NS_orbit_V}
Let $q_0=(x_0,\hat{x}_0;A_0)\in Q$. Then the part of the orbit $\mc{O}_{\NSDist}(q_0)$ of $\NSDist$ through $q_0$
that lies in the $\pi_{Q}$-fiber over $(x_0,\hat{x}_0)$ is given by
\begin{align}\label{eq:NS_orbit_V}
\mc{O}_{\NSDist}(q_0)\cap Q|_{(x_0,\hat{x}_0)}
&=\{\hat{h}\circ A_0\circ h\ |\ \hat{h}\in \hat{H}|_{\hat{x}_0},\ h\in H|_{x_0}^{-1}\} \\
&=:\hat{H}|_{\hat{x}_0}\circ A_0\circ H|_{x_0}^{-1}, \nonumber 
\end{align}
and is an immersed submanifold of the fiber $Q|_{(x_0,\hat{x}_0)}=\pi_{Q}^{-1}(x_0,\hat{x}_0)$.

Moreover, if $F$, $\hat{F}$ are orthonormal frames at $x_0$, $\hat{x}_0$, as above,
then there is a diffeomorphism (depending on $F$ and $\hat{F}$)
\begin{align}\label{eq:NS_orbit_V_frame}
\mc{O}_{\NSDist}(q_0)\cap Q|_{(x_0,\hat{x}_0)}\cong \hat{H}|_{\hat{F}} \mc{M}_{F,\hat{F}}(A_0) H|_{F}^{-1},
\end{align}
where the groups on the right hand side are Lie subgroups of $\SO(n)$.

\end{theorem}

In the previous statement, we have used the following notation. If $G$ is a group and $S$ is a subset of $G$,
then $S^{-1}:=\{g^{-1}\ |\ g\in S\}$. Of course $G^{-1}=G$ but,
 in Eq. (\ref{eq:NS_orbit_V}), it is
somewhat more convenient to leave $H|_{x_0}^{-1}$ and not to replace it by $H|_{x_0}$.

\begin{proof}
Notice that $q_1=(x_0,\hat{x}_0;A_1)\in \mc{O}_{\NSDist}(q_0)\cap \pi_{Q}^{-1}(x_0,\hat{x}_0)$
if and only if there is a piecewise $C^1$ path $t\mapsto q(t)=(x(t),\hat{x}(t);A(t))$, $t\in [0,1]$,
with $q(0)=q_0$, $q(1)=q_1$ and tangent to $\NSDist$.
This is, on the other hand, equivalent, by the definiton of $\NSDist$,
to the fact that $A(t)=P_0^t(x,\hat{x})A_0$.
It is also clear that $t\mapsto (x(t),\hat{x}(t))$ is a piecewise $C^1$ loop of $M\times\hat{M}$ based at $(x_0,\hat{x}_0)$
i.e., it belongs to $\Omega_{(x_0,\hat{x}_0)}(M\times\hat{M})$
which can be identified, in a natural way, with $\Omega_{x_0}(M)\times \Omega_{\hat{x}_0}(\hat{M})$.
By these remarks, Eq. (\ref{eq:parallel_trans_A}) and the above definition of the holonomy groups, we get
\[
& \mc{O}_{\NSDist}(q_0)\cap \pi_{Q}^{-1}(x_0,\hat{x}_0)
=\{P_0^1(\ol{x}) A_0\ |\ \ol{x}\in \Omega_{(x_0,\hat{x}_0)}(M\times\hat{M})\} \\
=&\{P_0^1(x,\hat{x}) A_0\ |\ x\in \Omega_{x_0}(M), \hat{x}\in \Omega_{\hat{x}_0}(\hat{M})\} \\
=&\{P_0^1(\hat{x}) \circ A_0\circ P_1^0(x)\ |\ x\in \Omega_{x_0}(M), \hat{x}\in \Omega_{\hat{x}_0}(\hat{M})\}
=\hat{H}|_{\hat{x}_0}\circ A_0\circ H|_{x_0}^{-1}.
\]

We next prove that $\hat{H}|_{\hat{x}_0}\circ A_0\circ H|_{x_0}^{-1}$ is an immersed submanifold of $Q|_{(x_0,\hat{x}_0)}$.
Let $f:\hat{H}|_{\hat{x}_0}\times H|_{x_0}\to Q|_{(x_0,\hat{x}_0)}$
be a map given by
$f(\hat{h},h):=\hat{h}\circ A_0\circ h^{-1}$. The map $f$ is clearly smooth, when we consider
$H|_{x_0}$ (resp. $\hat{H}|_{\hat{x}_0}$) as a Lie subgroup of $\SO(T|_{x_0} M)$ (resp. $\SO(T|_{\hat{x}_0} \hat{M})$).
Moreover, denote $G=\hat{H}|_{\hat{x}_0}\times H|_{x_0}$ and consider the smooth (left) group actions
$\mu:G\times Q|_{(x_0,\hat{x}_0)}\to Q|_{(x_0,\hat{x}_0)}$
and $m:G\times G\to G$ of $G$ on $Q|_{(x_0,\hat{x}_0)}$ and itself given by
$\mu((\hat{h},h),A)=\hat{h}\circ A\circ h^{-1}$, $m((\hat{h},h),(\hat{k},k))=(\hat{h}\hat{k},hk)$.
Then we see that
\[
& \mu((\hat{h},h),f(\hat{k},k))=\hat{h}\circ \big(\hat{k}\circ A_0\circ k^{-1}\big)\circ h^{-1} \\
=&(\hat{h}\hat{k})\circ A_0\circ (hk)^{-1}=f(\hat{h}\hat{k},hk)=f\big(m((\hat{h},h),(\hat{k},k))\big),
\]
which shows that $f$ is $G$-equivariant map. Since $G$ acts transitively (by the action $m$) on itself,
it follows that $f$ has constant rank (see \cite{lee02} Theorem 9.7).

Unfortunately $f$ is not injective but there is an easy solution to this obstacle.
Notice that $K:=f^{-1}(A_0)$ is a closed subgroup of $G$, hence $G/K$ (the right coset space) is a smooth manifold
and $f$ induces a smooth map $\ol{f}:G/K\to Q|_{(x_0,\hat{x}_0)}$, which is still $G$-equivariant,
when one uses the (left) $G$-action $\ol{m}$ on $G/K$ induced by $m$. Now $\ol{f}$ is injective and constant rank,
hence an injective immersion (see \cite{lee02} Theorem 7.14) into $Q|_{(x_0,\hat{x}_0)}$.
But the image of $\ol{f}$ is exactly $\hat{H}|_{\hat{x}_0}\circ A_0\circ H|_{x_0}^{-1}$.

Moreover, given orthonormal frames $F$ and $\hat{F}$, we clearly see that
 $$\hat{h}\circ A_0\circ h\mapsto \mc{M}_{\hat{F},\hat{F}}(\hat{h})\mc{M}_{F,\hat{F}}(A_0)\mc{M}_{F,F}(h)$$
gives the desired diffeomorphism $$\hat{H}|_{\hat{x}_0}\circ A_0\circ H|_{x_0}^{-1}\to \hat{H}|_{\hat{F}} \mc{M}_{F,\hat{F}}(A_0) H|_{F}^{-1}.$$
\end{proof}

\begin{corollary}\label{cor:compact_fiber_NSD_orbit}
If $M$ and $\hat{M}$ are simply-connected,
then each $\pi_Q$-fiber $\mc{O}_{\NSDist}(q_0)\cap Q|_{(x,\hat{x})}$, with $(x,\hat{x})\in M\times\hat{M}$,
of any orbit $\mc{O}_{\NSDist}(q_0)$, $q_0=(x_0,\hat{x}_0;A_0)$, is a compact connected
embedded smooth submanifold $Q$.
In particular, if a $\NSDist$-orbit is open in $Q$ then it is equal to $Q$.
\end{corollary}

\begin{proof}
Without loss of generality we may assume that $(x,\hat{x})=(x_0,\hat{x}_0)$.
By Theorem 3.2.8 in \cite{joyce07}
(in this relation, see also Appendix 5 in \cite{kobayashi63}),
the simply connectedness assumption
implies that $H|_{x_0}$ and $\hat{H}|_{\hat{x}_0}$ 
are respectively (closed and hence) compact connected Lie-subgroups of $\SO(T|_{x_0} M)$ and $\SO(T|_{\hat{x}_0} \hat{M})$.

Now $\mc{O}_{\NSDist}(q_0)\cap Q|_{(x_0,\hat{x}_0)}$ 
is compact (as a subset of $Q$) and connected since it
is a continuous image (by the map $f$ in the proof of Theorem \ref{th:NS_orbit_V}) of the compact connected set $\hat{H}|_{\hat{x}_0}\times H|_{x_0}$,
Finally notice that a compact immersed submanifold is embedded.

The last claim follows from the fact that
an open orbit $\mc{O}_{\NSDist}(q_0)$
has a open fiber $\mc{O}_{\NSDist}(q_0)\cap Q|_{(x_0,\hat{x}_0)}$ in $Q|_{(x_0,\hat{x}_0)}$.
This fiber is also compact by what we just proved
and hence $\mc{O}_{\NSDist}(q_0)\cap Q|_{(x_0,\hat{x}_0)}=Q|_{(x_0,\hat{x}_0)}$
by connectedness of $Q|_{(x_0,\hat{x}_0)}$. This clearly implies that $Q=\mc{O}_{\NSDist}(q_0)$.
\end{proof}

The next corollary gives the infinitesimal version of Theorem \ref{th:NS_orbit_V}.

\begin{corollary}\label{cor:NS_orbit_V_inf}
Let $q_0=(x_0,\hat{x}_0;A_0)\in Q$.
Then
\begin{align}\label{eq:NS_orbit_V_inf}
T|_{q_0} \mc{O}_{\NSDist}(q_0)\cap V|_{q_0} (\pi_{Q})
&=\nu(\{\hat{k}\circ A_0-A_0\circ k\ |\ k\in \mathfrak{h}|_{x_0}, \hat{k}\in \hat{\mathfrak{h}}|_{\hat{x}_0}\})|_{q_0} \\
&=:\nu(\hat{\mathfrak{h}}|_{\hat{x}_0}\circ A_0-A_0\circ \mathfrak{h}|_{x_0})|_{q_0}, \nonumber
\end{align}
where $\mathfrak{h}|_{x_0},\hat{\mathfrak{h}}|_{\hat{x}_0}$
are the Lie algebras of the holonomy groups $H|_{x_0},\hat{H}|_{\hat{x}_0}$ of $M,\hat{M}$.
\end{corollary}

\begin{proof}
As in the previous proof, consider the map 
\begin{align}
f:\hat{H}|_{\hat{x}_0}\times H|_{x_0}&\to \mc{O}_{\NSDist}(q_0)\cap \pi_{Q}^{-1}(x_0,\hat{x}_0),\nonumber\\
(\hat{h},h)&\mapsto \hat{h}\circ A_0\circ h^{-1},\nonumber
\end{align}
which is known to be a submersion by the previous considerations. We deduce that$$f_*(\hat{\mathfrak{h}}\times \mathfrak{h})=T|_{q_0} (\mc{O}_{\NSDist}(A_0)\cap \pi_{Q}^{-1}(x_0,\hat{x}_0))=T|_{q_0} \mc{O}_{\NSDist}(q_0)\cap V|_{q_0}(\pi_{Q}).$$
But it is obvious that $f_*|_{\hat{\mathfrak{h}}\times \mathfrak{h}}(\hat{k},k)=\nu(\hat{k}\circ A_0-A_0\circ k)|_{q_0}$,
which then proves the claim. 

\end{proof}

\begin{remark}
By the previous corollary and the Ambrose-Singer holonomy theorem (see \cite{joyce07}, \cite{kobayashi63}),
we have for $q_0=(x_0,\hat{x}_0;A_0)\in Q$,
\[
T|_{q_0} \mc{O}_{\NSDist(q_0)}\cap V|_{q_0} (\pi_{Q})
=&\big\{ P_1^0(\hat{c}) \hat{R}|_{\hat{x}}(\hat{X},\hat{Y}) P_0^1(\hat{c}) A_0
   -A_0P_1^0(c) R|_{x}(X,Y) P_0^1(c)\ \big|\ \\
 &   x\in M,\ \hat{x}\in\hat{M},\ X,Y\in T|_x M,\ \hat{X},\hat{Y}\in T|_{\hat{x}}\hat{M}, \\
 &    \ c\in C^1_{\mathrm{pw}}([0,1],M),\ c(0)=x_0,\ c(1)=x,\\
  &  \hat{c}\in C^1_{\mathrm{pw}}([0,1],\hat{M}),\ \hat{c}(0)=\hat{x}_0,\ \hat{c}(1)=\hat{x}\big\},
\]
where $C^1_{\mathrm{pw}}([0,1],M)$ (resp. $C^1_{\mathrm{pw}}([0,1],\hat{M})$)
is the set of piecewise continuously differentiable maps $[0,1]\to M$ (resp. $[0,1]\to \hat{M}$).
\end{remark}

Theorem \ref{th:NS_orbit_V} shows that,
since $M,\hat{M}$ are connected, all the $\pi_{Q}$-fibers of the reachable set 
$\mc{O}_{\NSDist}(q_0)$ are diffeomorphic i.e.,
\[
\mc{O}_{\NSDist}(q_0)\cap \pi_{Q}^{-1}(x_0,\hat{x}_0)\cong \mc{O}_{\NSDist}(q_0)\cap \pi_{Q}^{-1}(x_1,\hat{x}_1),
\]
for every $(x_1,\hat{x}_1)\in M\times\hat{M}$.
This follows from the fact that if points $x,y\in M$, then (since $M$ is connected) $H|_x$ and $H|_y$
are isomorphic, the same observation holding in $\hat{M}$.
We will now prove that the reachable set $\mc{O}_{\NSDist}(q_0)$ has actually a bundle structure
over $M\times\hat{M}$.

\begin{proposition}\label{pr:NS_orbit_bundle}
For $q_0=(x_0,\hat{x}_0;A_0)\in Q$,
denote $\pi_{\mc{O}_{\NSDist}(q_0)}:=\pi_{Q}|_{\mc{O}_{\NSDist}(q_0)}$.
Then $\pi_{\mc{O}_{\NSDist}(q_0)}:\mc{O}_{\NSDist}(q_0)\to M\times\hat{M}$
is a smooth subbundle of $\pi_Q$
with typical fiber $\hat{H}|_{\hat{x}_0}\circ A_0\circ H|_{x_0}^{-1}$
and $\mc{O}_{\NSDist}(q_0)$ is a smooth immersed submanifold of $Q$.
\end{proposition}

\begin{proof}
The surjectivity of $\pi_{\mc{O}_{\NSDist}(q_0)}$ onto $M\times\hat{M}$ follows immediately from the connectivity of $M$, $\hat{M}$.

Choose local charts $(\phi,U)$ and $(\hat{\phi},\hat{U})$ of $M$ and $\hat{M}$
around $x_0$, $\hat{x}_0$ centered at $x_0$, $\hat{x}_0$ (i.e., $\phi(x_0)=0$, $\hat{\phi}(\hat{x}_0)=0$)
and so that $\phi(U)$ and $\hat{\phi}(\hat{U})$ are convex.
Then, define
\begin{align}
\tau_{(\phi,\hat{\phi})}:\pi_{\mc{O}_{\NSDist}(q_0)}^{-1}(U\times\hat{U})&\to (U\times\hat{U})\times \big(\hat{H}|_{\hat{x}_0}\circ A_0\circ H|_{x_0}^{-1}\big)\nonumber \\
(x,\hat{x};A)&\mapsto\big((x,\hat{x}),P_1^0\big(t\mapsto (\phi^{-1}(t\phi(x)),\hat{\phi}^{-1}(t\hat{\phi}(\hat{x}))) \big)A\big),\nonumber
\end{align}
where we notice that, since $A=P_0^1(c,\hat{c})A_0=P_0^1(\hat{c})\circ A_0\circ P_1^0(c)$
for some piecewise $C^1$ paths $c:[0,1]\to M$ and $\hat{c}:[0,1]\to \hat{M}$ with $c(0)=x_0$, $\hat{c}(0)=\hat{x}_0$,
one has
\[
& P_1^0\big(t\mapsto (\phi^{-1}(t\phi(x)),\hat{\phi}^{-1}(t\hat{\phi}(\hat{x}))) \big)A\\
&=P_0^1\big(t\mapsto \hat{\phi}^{-1}(t\hat{\phi}(\hat{x})) \big)\circ A \circ P_1^0\big(t\mapsto \phi^{-1}(t\phi(x)))\big) \\
&=P_0^1\big(t\mapsto \hat{\phi}^{-1}(t\hat{\phi}(\hat{x})) \big)\circ P_0^1(\hat{c})\circ A_0 \circ P_1^0(c)\circ P_1^0\big(t\mapsto \phi^{-1}(t\phi(x)))\big).
\]
The concatenation of the path $c$ and $t\mapsto \phi^{-1}(t\phi(x))$ is a piecewise $C^1$ loop of $M$ based at $x_0$
and the concatenation of $\hat{c}$ and $t\mapsto \hat{\phi}^{-1}(t\hat{\phi}(\hat{x}))$ is a piecewise $C^1$ loop based of $\hat{M}$ at $\hat{x}_0$.
Thus $P_1^0\big(t\mapsto (\phi^{-1}(t\phi(x)),\hat{\phi}^{-1}(t\hat{\phi}(\hat{x}))) \big)A$ is an element of $\hat{H}|_{\hat{x}_0}\circ A_0\circ H|_{x_0}^{-1}$.

It is clear that $\tau_{(\phi,\hat{\phi})}$ is a smooth bijection onto $(U\times\hat{U})\times \big(\hat{H}|_{\hat{x}_0}\circ A_0\circ H|_{x_0}^{-1}\big)$.
Its inverse map is given by $\psi_{(\phi,\hat{\phi})}$,
\[
\psi_{(\phi,\hat{\phi})}((x,\hat{x}),B)=\big(x,\hat{x};P_0^1\big(t\mapsto (\phi^{-1}(t\phi(x)),\hat{\phi}^{-1}(t\hat{\phi}(\hat{x}))) \big)B\big),
\]
which is clearly smooth into $Q$ with image contained in $\mc{O}_{\NSDist}(q_0)$
and hence it is smooth into $\mc{O}_{\NSDist}(q_0)$
by the basic properties of an orbit.
This shows that $\mc{O}_{\NSDist}(q_0)$ is a smooth bundle.

Since the maps $\ol{\tau}_{(\phi,\hat{\phi})}$
defined on $\pi_Q^{-1}(U\times \hat{U})$ by the same formula as $\tau_{(\phi,\hat{\phi})}$
are diffeomorphisms (by an identical argument as above) onto $(U\times\hat{U})\times \pi_Q^{-1}(x_0,\hat{x}_0)$,
we see that $\pi_{\mc{O}_{\NSDist}(q_0)}$ is a smooth (immersed) subbundle of $\pi_Q$.

\end{proof}

We may now also prove that $\pi_Q:Q\to M\times\hat{M}$
cannot be equipped with a principal bundle structure leaving the distribution $\NSDist$
invariant except in special cases.

\begin{theorem}\label{th:no_useful_principal_bundle_for_pi_Q}
Generically, in dimension $n\geq 3$, $\pi_Q$ cannot be equipped with a principal bundle structure
which leaves $\NSDist$ invariant.

More precisely, if $n\geq 3$ and $F,\hat{F}$ are oriented orthonormal frames of $M$ and $\hat{M}$ at $x_0$ and $\hat{x}_0$, respectively,
and if $H|_F\subset\SO(n)$, $\hat{H}|_{\hat{F}}\subset\SO(n)$ are the holonomy groups
with respect to these frames,
then $H|_F\cap \hat{H}|_{\hat{F}}\neq \{\id_{\R^n}\}$
implies that there is no principal bundle structure on $\pi_Q$ which leaves $\NSDist$ invariant.

Especially this holds if $M$ (resp. $\hat{M}$) has full holonomy $\SO(n)$
and $\hat{M}$ (resp. $M$) is not flat.
\end{theorem}

\begin{proof}
Suppose $\mu:G\times Q\to Q$ is a left principal bundle structure for $\pi_Q$
leaving $\NSDist$ invariant. Notice that $G$ is diffeomorphic to the $\pi_Q$-fibers i.e., to $\SO(n)$
(but, of course, does not need to be isomorphic to it as a Lie group).
The fact that for all $g\in G$ we have $(\mu_g)_*\NSDist\subset\NSDist$
is clearly equivalent to $(\mu_g)_*\LNSD(X, \hat{X})|_q=\LNSD(X, \hat{X})|_{\mu(g,q)}$
for all $q=(x,\hat{x};A)\in Q$ and $X\in T|_x M$, $\hat{X}\in T|_{\hat{x}}\hat{M}$.
But this means that for all $g\in G$, $(x,\hat{x};A)\in Q$ and a.c. paths $\gamma$, $\hat{\gamma}$
starting at $x$, $\hat{x}$ respectively, we have
$\mu(g,q_{\NSDist}(q)(t))=q_{\NSDist}(\mu(g,q))(t)$ where $q_{\NSDist}(q)$ is the
unique solution to $\dot{q}(t)=\LNSD(\dot{\gamma}(t), \hat{\dot{\gamma}}(t))|_{q(t)}$, $q(0)=q$.
Since we know that if $q=(x,\hat{x};A)\in Q$,
then $q_{\NSDist}(q)(t)=(\gamma(t),\hat{\gamma}(t);P_0^t(\hat{\gamma})\circ A\circ P_t^0(\gamma))$ for all $t$,
we get that
\[
\mu\big(g,P_0^t(\hat{\gamma})\circ A\circ P_t^0(\gamma)\big)=P_0^t(\hat{\gamma})\circ \mu(g,A)\circ P_t^0(\gamma).
\]

Let $F$, $\hat{F}$ be chosen as in the statement above.
Define $(x_0,\hat{x}_0;A_0)\in Q$ by $A_0=\sum_i g(X_i,\cdot)\hat{X}_i$
and choose $B\in H|_{F}\cap \hat{H}|_{\hat{F}}$, $B\neq\id_{\R^n}$.
Choose loops $\gamma$, $\hat{\gamma}$ based at $x_0$, $\hat{x}_0$
such that $\mc{M}_{F,F}(P_0^1(\gamma))=B$, $\mc{M}_{\hat{F},\hat{F}}(P_0^1(\hat{\gamma}))=B$.
Since $\mc{M}_{F,\hat{F}}(A_0)=\id_{\R^n}$ by the definition of $A_0$,
we have 
\[
\mc{M}_{F,\hat{F}}(A_0)=&\id_{\R^n}
=B\id_{\R^n}B^{-1}
=\mc{M}_{\hat{F},\hat{F}}(P_0^1(\hat{\gamma}))\mc{M}_{F,\hat{F}}(A_0)\underbrace{\mc{M}_{F,F}(P_0^1(\gamma))^{-1}}_{=\mc{M}_{F,F}(P_1^0(\gamma))} \\
=&\mc{M}_{F,\hat{F}}(P_0^1(\hat{\gamma})\circ A_0\circ P_1^0(\gamma))
\]
i.e.,
\[
A_0=P_0^1(\hat{\gamma})\circ A_0\circ P_1^0(\gamma).
\]
Applying to this what was done above, we get
\[
\mu(g,A_0)=\mu(g,P_0^1(\hat{\gamma})\circ A_0\circ P_1^0(\gamma))
=P_0^1(\hat{\gamma})\circ \mu(g,A_0)\circ P_1^0(\gamma),\quad \forall g\in G
\]
i.e.,
\[
\mc{M}_{F,\hat{F}}(\mu(g,A_0))=B\mc{M}_{F,\hat{F}}(\mu(g,A_0))B^{-1},\quad g\in G.
\]
But $\mu(G,A_0)=\pi_Q^{-1}(x_0,\hat{x}_0)$ whence $\mc{M}_{F,\hat{F}}(\mu(G,A_0))=\SO(n)$
and thus we have found a $B\in\SO(n)$ which is not the identity $\id_{\R^n}$ such that
$C=BCB^{-1}$ for all $C\in\SO(n)$
i.e., $B$ belongs to the center of $\SO(n)$. But in dimension $n\geq 3$ the center of $\SO(n)$
is $\{\id_{\R^n}\}$, contradicting the fact that $B\neq \id_{\R^n}$.
This contradiction shows that the existence of a principal bundle structure $\mu$
on $\pi_Q$ that preserves $\NSDist$ is impossible in this case.
\end{proof}

%%%%%%%%%%%%%%%%%%%%%%%%%%%%%%%%%%%%%%%%%%%%%%%
\subsection{Consequences for Controllability}
%%%%%%%%%%%%%%%%%%%%%%%%%%%%%%%%%%%%%%%%%%%%%%%

From the previous characterizations of the reachable set of  
$\sns$, we now derive consequences for 
the controllability of the control system $\sns$.

We start with the following remark.

\begin{remark}
All the results, except Theorem \ref{th:no_useful_principal_bundle_for_pi_Q}, of the previous section can obviously be formulated in verbatim
in the space $T^*M\otimes T\hat{M}$ instead of $Q$
(i.e., we may replace $Q$ by $T^*M\otimes T\hat{M}$ everywhere)
and the statements hold true in this setting.
However, Theorem \ref{th:NS_orbit_V} (formulated in $T^*M\otimes T\hat{M}$)
then implies each orbit $\mc{O}_{\NSDist}(q_0)$ of $\NSDist$ in $T^*M\otimes T\hat{M}$,
$q_0=(x_0,\hat{x};A_0)\in T^*M\otimes T\hat{M}$, can have dimension
of at most $2n+\dim H|_{x_0}+\dim \hat{H}|_{\hat{x}_0}\leq n^2+n$.
Since the dimension of $T^*M\otimes T\hat{M}$ is $n^2+2n$,
the orbit $\mc{O}_{\NSDist}(q_0)$ has a codimension of at least $n$.
This shows that $\NSDist$ (or the related control problem) is never
completely controllable in $T^*M\otimes T\hat{M}$.
\end{remark}

Theorem \ref{th:NS_orbit_V} states that the controllability of $\NSDist$ is completely determined by the holonomy groups of $M$ and $\hat{M}$.
The next theorem highlights that fact at the Lie algebraic level. 

\begin{theorem}\label{theo-00}
The control system $\sns$ is completely controllable if and only if,
for every $A\in \SO(n)$, the following holds:
\begin{align}\label{eq:lin-alg}
\mathfrak{h}+A^{-1}\hat{\mathfrak{h}}A=\so(n),
\end{align}
where $\mathfrak{h}$ and $\hat{\mathfrak{h}}$ are
respectively the Lie subalgebras of $\so(n)$
isomorphic (as Lie algebras) to the holonomy Lie algebras
of $\nabla$ and $\hat{\nabla}$.
\end{theorem}

\begin{proof}
Clearly, an orbit $\mc{O}_{\NSDist}(q_0)=Q$, where $(x_0,\hat{x}_0;A_0)=q_0\in Q$,
is an open subset of $Q$ if and only if $T|_{q} \mc{O}_{\NSDist}(q_0)=T|_q Q$
for some (and hence every) $q\in \mc{O}_{\NSDist}(q_0)$.
Thus the decomposition given by Eq. (\ref{eq:2.1:5})
implies that an orbit $\mc{O}_{\NSDist}(q_0)$
is open in $Q$ if and only if $V|_q(\pi_Q)\subset T|_{q} \mc{O}_{\NSDist}(q_0)$
for some $q\in \mc{O}_{\NSDist}(q_0)$.

By connectedness of $Q$, we get that $\NSDist$ is controllable
i.e., $\mc{O}_{\NSDist}(q_0)=Q$
for some (and hence every) $(x_0;\hat{x}_0;A_0)=q_0\in Q$
if and only if every orbit $\mc{O}_{\NSDist}(q)$,
$(x,\hat{x};A)=q\in Q$ is open in $Q$.

From now on, fix $(x_0,\hat{x}_0)\in M\times\hat{M}$.
Proposition \ref{pr:NS_orbit_bundle}
implies that every $\NSDist$ orbit intersects
every $\pi_Q$-fiber.
Hence $\NSDist$ is controllable
if and only if
$V|_q(\pi_Q)\subset T|_{q} \mc{O}_{\NSDist}(q)$
for every $q=(x_0,\hat{x}_0;A)\in Q|_{(x_0,\hat{x}_0)}$.
By Corollary \ref{cor:NS_orbit_V_inf},
this condition is equivalent to the condition that, for every 
$q=(x_0,\hat{x}_0;A)\in Q|_{(x_0,\hat{x}_0)}$,
\[
\nu(\hat{\mathfrak{h}}|_{\hat{x}_0}\circ A-A\circ\mathfrak{h}|_{x_0})=V|_q(\pi_Q).
\]

Next, by Proposition \ref{pr:vertical_of_Q}, one deduces that, 
for every $q\in Q$, 
$$V|_q(\pi_Q)=\nu(A(\so(T|_x M)))|_q$$ and thus
we conclude that $\NSDist$ is controllable if and only if, for all $q=(x_0,\hat{x}_0;A)\in Q|_{(x_0,\hat{x}_0)}$,
\[
A^{-1}\circ \hat{\mathfrak{h}}|_{\hat{x}_0}\circ A-\mathfrak{h}|_{x_0}=\so(T|_x M).
\]

Choosing arbitrary orthonormal local frames $F$ and $\hat{F}$
of $M$ and $\hat{M}$ at $x_0$ and $\hat{x}_0$, respectively,
we see that the above condition is equivalent to
\[
\mc{M}_{F,\hat{F}}(A)^{-1}\hat{\mathfrak{h}}|_{\hat{F}}\mc{M}_{F,\hat{F}}(A)-\mathfrak{h}|_{F}=\so(n),
\quad \forall A\in Q|_{(x_0,\hat{x}_0)},
\]
where
$$\mathfrak{h}|_{F}=\{\mc{M}_{F}(k)\ |\ k\in \mathfrak{h}\},\quad 
\hat{\mathfrak{h}}|_{\hat{F}}=\{\mc{M}_{\hat{F}}(\hat{k})\ |\ \hat{k}\in \hat{\mathfrak{h}}\},$$
are the holonomy Lie algebras as subalgebras
of $\so(n)$ w.r.t. the frames $F$ and $\hat{F}$ respectively.

The proof is finished by noticing that $\{\mc{M}_{F,\hat{F}}(A)\ |\ A\in Q|_{(x_0,\hat{x}_0)}\}=\SO(n)$
and that the orthonormal frames $F$, $\hat{F}$ were arbitrary chosen.

\end{proof}

\begin{theorem}\label{theo-01}
Suppose $M,\hat{M}$ are simply connected. Then $\sns$ is completely controllable 
if and only if
\begin{align}\label{eq:lin-alg2}
\mathfrak{h}+\hat{\mathfrak{h}}=\so(n)
\end{align}
where $\mathfrak{h}$,$\hat{\mathfrak{h}}$ are the Lie subalgebra of $\so(n)$
isomorphic (as Lie algebras) to the holonomy Lie algebras
of $\nabla$ and $\hat{\nabla}$ respectively.
\end{theorem}

\begin{proof}
By Theorem \ref{theo-00}, necessity of the condition is obvious.

Conversely suppose that the condition in Eq. (\ref{eq:lin-alg2}) holds.
This condition implies
that for $(x_0,\hat{x}_0)\in M\times\hat{M}$
there is an $q_0=(x_0,\hat{x}_0;A_0)\in Q|_{(x_0,\hat{x}_0)}$
such that
\[
A_0^{-1}\circ \hat{\mathfrak{h}}|_{\hat{x}_0}\circ A_0-\mathfrak{h}|_{x_0}=\so(T|_{x_0} M).
\]
By Proposition \ref{pr:vertical_of_Q} and Corollary \ref{cor:NS_orbit_V_inf}
this means that $T|_{q_0} \mc{O}_{\NSDist}(q_0)\cap V|_{q_0}(\pi_Q)=V|_{q_0}(\pi_Q)$
and hence $T|_{q_0} \mc{O}_{\NSDist}(q_0)=T|_{q_0}Q$ by Eq. (\ref{eq:2.1:5})
which implies that $\mc{O}_{\NSDist}(q_0)$ is open in $Q$.
Corollary \ref{cor:compact_fiber_NSD_orbit} then implies
that $\mc{O}_{\NSDist}(q_0)=Q$
i.e., $\sns$ is completely controllable.
\end{proof}

There is a complete classification of holonomy groups of Riemannian manifolds by Cartan (for symmetric spaces, see \cite{helgason78})
and Berger (for non-symmetric spaces, see \cite{joyce07}).
Hence the above theorems reduce the question of complete controllability of $\sns$
to an essentially  linear algebraic problem. 

For instance, in the case where both manifolds are non-symmetric, simply connected and irreducible, we get the following proposition.
\begin{theorem}\label{th:non-sym}
Assume that the manifolds $M$ and $\hat M$ are complete non-symmetric, simply connected, irreducible and $n\neq 8$.
Then, the control system $\sns$ is completely controllable if and only if
either $H$ or $\hat{H}$ is equal to $\SO(n)$ (w.r.t some orthonormal frames).
\end{theorem}

\newcommand{\Sp}{\mathrm{Sp}}
\newcommand{\Ug}{\mathrm{U}}

\begin{proof} Suppose first that $H|_F=\SO(n)$.
Choose any $q_0=(x_0,\hat{x}_0;A_0)\in Q$ and define $\hat{F}=A_0F$ (which is an orthonormal frame of $\hat{M}$ at $\hat{x}_0$ since $A_0\in Q$)
and compute, noticing that $\mc{M}_{F,\hat{F}}(A_0)=\id_{\R^n}$,
\[
\pi_Q^{-1}(x_0,\hat{x}_0)\cap \mc{O}_{\NSDist}(q_0)
\cong \hat{H}|_{\hat{F}} H|_F=H|_{\hat{F}}\SO(n)=\SO(n),
\]
where the first diffeomorphism comes from Theorem \ref{th:NS_orbit_V}.
But the $\pi_Q$-fibers of $Q$ are diffeomorphic to $\SO(n)$ and hence
$\pi_Q^{-1}(x_0,\hat{x}_0)\cap \mc{O}_{\NSDist}(q_0)=\pi_Q^{-1}(x_0,\hat{x}_0)$.
By connectedness of $M,\hat{M}$ it follows that $Q=\mc{O}_{\NSDist}(q_0)$.

Assume now that both holonomy groups are different from $\SO(n)$. 
We also remark that if one holonomy group is included in the other one, then
complete controllability cannot hold according to Eq. (\ref{eq:lin-alg}).
Using Berger's list, see \cite{joyce07},
and taking into account that
\[
\Sp(m)\subset \mathrm{SU}(2m)\subset \mathrm{U}(2m)\subset\SO(4m) 
\]
where $n=4m$, it only remains  to study the following case:
$n=4m$ with $m\geq 2$, one group is equal to $U(2m)$ and the other one to
$\Sp(m)\cdot \Sp(1)$. Recall that
\[
\dim \big(\underbrace{U(2m)\big(\Sp(m)\cdot\Sp(1)\big)}_{U(2m)\cdot\Sp(1)}\big)
\leq \dim U(2m)+\dim \Sp(1)=4m^2+3.
\]
On the other hand $\dim\SO(4m)=8m^2-2m$ which is always strictly larger than
$4m^2+3$ for all $m\geq 2$.

\end{proof}

\begin{remark}
If $n=8$, one is left with the study of the case where one of the holonomy groups is  equal to $\mathrm{Spin}(7)$ and the other one is either equal to $\mathrm{U}(4)$ or to $\Sp(2)\cdot \Sp(1)$.
\end{remark}

As a corollary to Theorem \ref{th:NS_orbit_V} and Theorem \ref{theo-00}, we get the following result of non controllability in the case where both manifolds are reducible.

\begin{proposition}\label{pr:NS_orbit_of_reducible}
Suppose that both $(M,g)$ and $(\hat{M},\hat{g})$ are reducible Riemannian  manifolds.
Then $\sns$ is not completely controllable.
\end{proposition}

\begin{proof}
We need to show that, under the assumptions,
there exists $q_0=(x_0;\hat{x}_0;A_0)\in Q$ so that the orbit $\mc{O}_{\NSDist}(q_0)$
is a proper subset of $Q$.

Fix $x_0\in M$ and $\hat{x}_0\in\hat{M}$.
Since $(M,g)$ and $(\hat{M},\hat{g})$ are reducible,
there exist subspaces $V_1,V_2\subset T|_{x_0} M$ and
$\hat{V}_1,\hat{V}_2\subset T|_{\hat{x}_0} \hat{M}$,
with $n_i=\dim(V_i)\geq 1$,
$\hat{n}_i=\dim(\hat{V}_i)\geq 1$
and such that $H|_{x_0}(V_i)\subset V_i$,
$\hat{H}|_{\hat{x}_0}(\hat{V}_i)\subset \hat{V}_i$, for $i=1,2$.

Let $X^1_1,\dots,X^1_{n_1}$ and $X^2_{1},\dots,X^2_{n_2}$ be an orthonormal basis of $V_1$
and an orthonormal basis of $V_2$ respectively
and, similarly, let 
$\hat{X}^1_1,\dots,\hat{X}^1_{\hat{n}_1}$ and $\hat{X}^2_{1},\dots,\hat{X}^2_{\hat{n}_2}$
be an orthonormal basis of $\hat{V}_1$
and an orthonormal basis of $\hat{V}_2$ respectively .
Here, $V_i$ and $\hat{V}_i$, $i=1,2$, are equipped with the metrics induced
by $g|_{x_0}$ and  $\hat{g}|_{\hat{x}_0}$ respectively.
These vectors form orthonormal frames $F$ and  $\hat{F}$ of $M$ and 
$\hat{M}$ at $x_0$ and $\hat{x}_0$ respectively.

It follows from the Ambrose-Singer Holonomy Theorem
(cf. \cite{joyce07} Theorem 2.4.3, \cite{kobayashi63} Theorem 8.1)
that the Lie algebras $\mathfrak{h}|_F$ and $\hat{\mathfrak{h}}|_{\hat{F}}$
of $H|_F$ and $\hat{H}|_{\hat{F}}$ respectively 
split into direct sums of Lie-subalgebras,
\[
\mathfrak{h}|_F & =\mathfrak{h}_1\oplus\mathfrak{h}_2\subset\so(n_1)\oplus\so(n_2), \\
\mathfrak{\hat{h}}|_{\hat{F}} & =\hat{\mathfrak{h}}_1\oplus\hat{\mathfrak{h}}_2\subset\so(\hat{n}_1)\oplus\so(\hat{n}_2).
\]
Without loss of generality, we assume that $\hat{n}_1\geq n_1$.

Finally, we define the linear map $A_0:T|_{x_0} M\to T|_{\hat{x}_0} \hat{M}$ by
$$A_0(X^1_j)=\hat{X}^1_j,\quad j=1,\dots,n_1,\quad
A_0(X^2_j)=\hat{X}^1_{n_1+j}, \quad j=1,\dots,\hat{n}_1-n_1,$$
and 
$$A_0(X^2_j)=\hat{X}^2_{j-(\hat{n}_1-n_1)},\quad j=\hat{n}_1-n_1+1,\dots,n_2.$$
Thus, we have $\mc{M}_{F,\hat{F}}(A_0)=\id_{\R^n}$ and hence 
\[
\mathfrak{h}|_F+\mc{M}_{F,\hat{F}}(A_0)^{-1}\hat{\mathfrak{h}}|_{\hat{F}}\mc{M}_{F,\hat{F}}(A_0)
=\mathfrak{h}_1\oplus\mathfrak{h}_2+\hat{\mathfrak{h}}_1\oplus\hat{\mathfrak{h}}_2.
\]
The latter linear vector space is necessarily a proper subset of $\so(n)$.
In fact, if $E_{ij}$ is the $n\times n$-matrix with $1$ at the $i$-th row, $j$-th column and zero otherwise,
then the above linear space does not contain $E_{n1}-E_{1n}\in\so(n)$.
Therefore, the claim follows from Theorem \ref{theo-00}.

\end{proof}

\begin{corollary}\label{cor:NS_orbit_of_product}
Suppose that $(M,g)$ and  $(\hat{M},\hat{g})$ are equal to the Riemannian products
$(M_1\times M_2,g_1\oplus g_2)$ and $(\hat{M}_1\times\hat{M}_2,\hat{g}_1\oplus\hat{g}_2)$, with $\dim M_i\geq 1$, $\dim \hat{M}_i\geq 1$, $i=1,2$ respectively.
Then, $\sns$ is not controllable on $Q$. 
\end{corollary}

\begin{proof}
From the basic result on holonomy groups, we get the following decomposition
$H|_x=H^{\nabla^{g_1}}|_{x_1}\times H^{\nabla^{g_2}}|_{x_2}$,
where $x=(x_1,x_2)\in M$, and 
$\hat{H}|_{x}=H^{\nabla^{\hat{g}_1}}|_{\hat{x}_1}\times H^{\nabla^{\hat{g}_2}}|_{\hat{x}_2}$,
where $\hat{x}=(\hat{x}_1,\hat{x}_2)\in\hat{M}$.
This shows that the actions of $H$ and $\hat{H}$
on $T|_x M$, $T|_{\hat{x}} \hat{M}$, respectively, are both
reducible. Thus, the claim follows from the previous proposition.
\end{proof}

%%%%%%%%%%%%%%%%%%%%%%%%%%%%%%%%%%%%%%%%%%%%%%%%%%%%
\section{Study of the Rolling problem $(R)$}\label{sec:2.55}
%%%%%%%%%%%%%%%%%%%%%%%%%%%%%%%%%%%%%%%%%%%%%%%%%%%%

In this section, we investigate the rolling problem
as a control system $\srol$ associated to a subdistribution $\RDist$
of $\NSDist$ defined as follows.

%%%%%%%%%%%%%%%%%%%%%%%%%%%%%%%%%%%%%%%%%%%%%%%%%%%%%%
\subsection{Global properties of a $\RDist$-orbit}
%%%%%%%%%%%%%%%%%%%%%%%%%%%%%%%%%%%%%%%%%%%%%%%%%%%%%%

We begin with the following remark.

\begin{remark}\label{re:Q_bundle_over_M}
Notice that the map $\pi_{Q,M}:Q\to M$ is in fact a bundle.
Indeed, let $F=(X_i)_{i=1}^n$ be a local oriented orthonormal frame
of $M$ defined on an open set $U$. Then the local trivialization of  $\pi_{Q,M}$ induced by $F$ is
\[
\tau_F:\pi_{Q,M}^{-1}(U)\to U\times F_{\mathrm{OON}}(\hat{M});\quad
\tau_F(x,\hat{x};A)=(x,(AX_i|_x)_{i=1}^n),
\]
is a diffeomorphism. 

We also notice that since $\pi_{Q,M}$-fibers are diffeomorphic to $F_{\mathrm{OON}}(\hat{M})$,
in order that there would be a principal $G$-bundle structure for $\pi_{Q,M}$,
it is necessary (but not sufficient) that $F_{\mathrm{OON}}(\hat{M})$ is diffeomorphic to the Lie-group $G$.
In section \ref{space-form} we consider special cases where there is indeed a
principal bundle structure on $\pi_{Q,M}$ which moreover leaves $\RDist$ invariant.
\end{remark}

From Proposition \ref{pr:rol_geodesic}, we deduce
that each $\RDist$-orbit is a smooth bundle over $M$.
This is given in the next proposition
(the proof being similar to that of Proposition \ref{pr:NS_orbit_bundle}).

\begin{proposition}\label{pr:R_orbit_bundle}
Let $q_0=(x_0,\hat{x}_0;A_0)\in Q$ and suppose that $\hat{M}$ is complete.
Then $$
\pi_{\mc{O}_{\RDist}(q_0),M}:=\pi_{Q,M}|_{\mc{O}_{\RDist}(q_0)}:\mc{O}_{\RDist}(q_0)\to M,$$
is a smooth subbundle of $\pi_{Q,M}$.
\end{proposition}

\begin{proof}
We first show that $\pi_{\mc{O}_{\RDist}(q_0),M}$ is surjective.
If $x\in M$, there is a piecewise smooth path $\gamma:[a,b]\to M$ from $x_0$ to $x$
such that each smooth piece is a $g$-geodesic.
By Proposition \ref{pr:rol_geodesic} and completeness of $\hat{M}$ it follows
that there is a rolling path $q_{\RDist}{(\gamma,q_0)}:[a,b]\to Q$
along $\gamma$ with initial position $q_0$
defined on the whole interval $[a,b]$.
But then $\pi_{\mc{O}_{\RDist}(q_0),M}(q_{\RDist}{(\gamma,q_0)}(b))=x$
which proves the claimed surjectivity.

Since $\RDist|_q\subset T|_q \mc{O}_{\RDist}(A_0)$
for every $q\in \mc{O}_{\RDist}(q_0)$ and
$(\pi_{Q,M})_*$ maps $\RDist|_q$ isomorphically onto $T|_{\pi_{Q,M}(q)} M$, one immediately deduces that $\pi_{\mc{O}_{\RDist}(q_0),M}$ is also a submersion.
This implies that each fiber
$(\pi_{\mc{O}_{\RDist}(q_0),M})^{-1}(x)=\mc{O}_{\RDist}(q_0)\cap \pi_{Q,M}^{-1}(x)$, $x\in M$,
is a smooth closed submanifold of $\mc{O}_{\RDist}(q_0)$.

Choose next, for each $x\in M$, an open convex $U_x\subset T|_x M$ such that $\exp_x|_{U_x}$ is a diffeomorphism onto its image and $0\in U$.
Define 
\begin{align}
\tau_x:\pi_{Q,M}^{-1}(U_x)&\to U_x\times \pi_{Q,M}^{-1}(x),\nonumber\\
q=(y,\hat{y};A)&\mapsto \big(y,\big(x,\hat{\gamma}_{\RDist}{(\gamma_{y,x},q)}(1);A_{\RDist}{(\gamma_{y,x},q)}(1)\big)\big),\nonumber
\end{align}
where $\gamma_{y,x}:[0,1]\to M$; $\gamma_{y,x}(t)=\exp_x((1-t)\exp_x^{-1}(y))$
is a geodesic from $y$ to $x$.
It is obvious that $\tau_x$ is a smooth bijection.
Moreover, restricting $\tau_x$ to $\mc{O}_{\RDist}(q_0)$
clearly gives a smooth bijection 
\[
\mc{O}_{\RDist}(q_0)\cap \pi_{Q,M}^{-1}(U_x)\to U_x\times (\mc{O}_{\RDist}(q_0)\cap \pi_{Q,M}^{-1}(x)).
\]
The inverse of $\tau_x$,
$\tau_x^{-1}:U_x\times \pi_{Q,M}^{-1}(x)\to \pi_{Q,M}^{-1}(U_x)$
is constructed with a formula similar to that of $\tau_x$
and  is seen, in the same way, to be smooth.
This inverse restricted to $U_x\times (\mc{O}_{\RDist}(q_0)\cap \pi_{Q,M}^{-1}(x))$
maps bijectively onto $\mc{O}_{\RDist}(q_0)\cap \pi_{Q,M}^{-1}(U_x)$
and thus $\tau_x$ is a smooth local trivialization
of $\mc{O}_{\RDist}(q_0)$.
This completes the proof.

\end{proof}

\begin{remark}\label{re:R_orbit_bundle}
In the case where $\hat{M}$ is not complete,
the result of Proposition \ref{pr:R_orbit_bundle}
remains valid if we just claim
that $\pi_{\mc{O}_{\RDist}(q_0),M}$ is a bundle
over its image $M^\circ:=\pi_{Q,M}(\mc{O}_{\RDist}(q_0))$, which is an open connected subset of $M$.

Write $\hat{M}^\circ:=\pi_{Q,\hat{M}}(\mc{O}_{\RDist}(q_0))$.
Then using the diffeomorphism $\iota:Q:=Q(M,\hat{M})\to \hat{Q}:=Q(\hat{M},M)$; $(x,\hat{x};A)\mapsto (\hat{x},x;A^{-1})$
(Proposition \ref{pr:rol_inverse})
one gets
\[
\pi_{\mc{O}_{\RDist}(q_0),\hat{M}}=&\pi_{Q,\hat{M}}|_{\mc{O}_{\RDist}(q_0)}
=\pi_{Q,\hat{M}}\circ\iota^{-1}|_{\mc{O}_{\widehat{\RDist}}(\iota(q_0))}\circ\iota|_{\mc{O}_{\RDist}(q_0)} \\
=&\pi_{\hat{Q},\hat{M}}|_{\mc{O}_{\widehat{\RDist}}(\iota(q_0))}\circ\iota|_{\mc{O}_{\RDist}(q_0)}
=\pi_{\mc{O}_{\widehat{\RDist}}(\iota(q_0)),\hat{M}}\circ \iota|_{\mc{O}_{\RDist}(q_0)},
\]
from which we see that $\pi_{\mc{O}_{\RDist}(q_0),\hat{M}}$
is also a bundle over its image $\hat{M}^\circ$
since $\iota|_{\mc{O}_{\RDist}(q_0)}:\mc{O}_{\RDist}(q_0)\to \mc{O}_{\widehat{\RDist}}(\iota(q_0))$
is a diffeomorphism and since by the previous proposition and the above remark
$\pi_{\mc{O}_{\widehat{\RDist}}(\iota(q_0)),\hat{M}}$ is a bundle over its image,
which necessarily is $\hat{M}^\circ$.
Notice also that if $M$ is complete, then $\hat{M}^\circ=\hat{M}$.
\end{remark}

The next remark illustrates this point.

\begin{remark}
In the previous proposition, the assumption of completeness of $\hat{M}$
cannot be removed.
In fact, choose $M=\R^2$, $\hat{M}=\{\hat{x}\in\R^2\ |\ \n{\hat{x}}<1\}$
(with $\n{\cdot}$ the Euclidean norm).
Then
\[
Q\cong M\times \hat{M}\times\SO(2),
\quad T(Q)\cong Q\times\R^2\times\R^2\times\so(2)
\]
and $\RDist$ is given by
\[
\RDist|_{(x,\hat{x};A)}=\{(v,Av,0)\ |\ v\in\R^2\},
\]
as a subspace of $T|_{(x,\hat{x};A)} Q\cong\R^2\times\R^2\times\so(2)$.
If $x_0=0$, $\hat{x}_0=0$ and $A_0=\id_{\R^2}$ is the identity map 
$T|_0 M\cong\R^2\to T|_0\hat{M}\cong\R^2$,
we have that the orbit is equal to the 2-dimensional submanifold of $Q$
given by 
$\{(x,A_0x,A_0)\ |\ \n{x}<1\}$
and its image under the projection on the first factor, $\pi_{Q,M}$
is a proper open subset $\{x\in\R^2\ |\ \n{x}<1\}$
of $M$.
Thus $\pi_{Q,M}|_{\mc{O}_{\RDist}(x_0,\hat{x}_0;A_0)}$ is not a bundle over $M$,
since this map is not surjective.
\end{remark}

\begin{proposition}\label{pr:iso_equivariance}
For any Riemannian isometries $F\in\Iso(M,g)$ and $\hat{F}\in\Iso(\hat{M},\hat{g})$
of $(M,g)$, $(\hat{M},\hat{g})$ respectively,
one defines smooth free right and left actions of $\Iso(M,g)$, $\Iso(\hat{M},\hat{g})$
on $Q$ by
\[
q_0\cdot F:=(F^{-1}(x_0),\hat{x}_0;A_0\circ F_*|_{F^{-1}(x_0)}),
\quad
\hat{F}\cdot q_0:=(x_0,\hat{F}(\hat{x}_0);\hat{F}_*|_{\hat{x}_0}\circ A_0),
\]
where $q_0=(x_0,\hat{x}_0;A_0)\in Q$.
We also set 
\[
\hat{F}\cdot q_0\cdot F:=(\hat{F}\cdot q_0)\cdot F=\hat{F}\cdot (q_0\cdot F).
\]

Then for any $q_0=(x_0,\hat{x}_0;A_0)\in Q$, a.c. $\gamma:[0,1]\to M$, $\gamma(0)=x_0$,
and $F\in\Iso(M,g)$, $\hat{F}\in\Iso(\hat{M},\hat{g})$, one has
\begin{align}\label{eq:equivariance_of_rolling}
\hat{F}\cdot q_{\RDist}(\gamma,q_0)(t)\cdot F
=q_{\RDist}(F^{-1}\circ\gamma,\hat{F}\cdot q_0\cdot F)(t),
\end{align}
for all $t\in [0,1]$.
In particular,
\[
\hat{F}\cdot\mc{O}_{\RDist}(q_0)\cdot F=\mc{O}_{\RDist}(\hat{F}\cdot q_0\cdot F).
\]
\end{proposition}

\begin{proof}
The fact that the group actions are well defined is clear
and the smoothness of these actions
can be proven by writing out the Lie-group structures 
of the isometry groups (using e.g. Lemma III.6.4 in \cite{sakai91}). 
If $q_0\cdot F=q_0\cdot F'$ for some $F,F'\in\Iso(M,g)$ and $q_0\in Q$,
then $F^{-1}(x_0)={F'}^{-1}(x_0)$, $F_*|_{x_0}=F'_*|_{x_0}$
and hence $F=F'$ since $M$ is connected (see \cite{sakai91}, p. 43).
This proves the freeness of the right $\Iso(M,g)$-action. The same argument
proves the freeness of the left $\Iso(\hat{M},\hat{g})$-action.

Finally, Eq. (\ref{eq:equivariance_of_rolling}) follows from
a simple application of Eq. (\ref{eq:iso_parallel}).
In fact, by Remark~\ref{re:Lambda} the rolling curve $q_{\RDist}(\gamma,q_0)=(\gamma,\hat{\gamma}_{\RDist}(\gamma,q_0);A_{\RDist}(\gamma,q_0))$
is defined by
\[
& P^0_t(\hat{\gamma}_{\RDist}(\gamma,q_0))\dot{\hat{\gamma}}_{\RDist}(\gamma,q_0)(t)
=A_0P^0_t(\gamma)\dot{\gamma}(t), \\
& A_{\RDist}(\gamma,q_0)(t)=P_0^t(\hat{\gamma}_{\RDist}(\gamma,q_0))\circ A_0\circ P^0_t(\gamma).
\]

First, by using (\ref{eq:iso_parallel}), we get
\[
& P^0_t(\hat{F}\circ \hat{\gamma}_{\RDist}(\gamma,q_0))\dif{t}(\hat{F}\circ \hat{\gamma}_{\RDist}(\gamma,q_0))(t)
=\hat{F}_*P^0_t(\hat{\gamma}_{\RDist}(\gamma,q_0))\hat{F}_*^{-1}(\hat{F}_*\dot{\hat{\gamma}}_{\RDist}(\gamma,q_0)(t)) \\
=&\hat{F}_*A_0P^0_t(\gamma)\dot{\gamma}(t)
=(\hat{F}_*A_0F_*) (F^{-1}_*P^0_t(\gamma)F_*) F^{-1}_*\dot{\gamma}(t) \\
=&(\hat{F}_*A_0F_*) P^0_t(F^{-1}\circ \gamma)\dif{t}(F^{-1}\circ \gamma)(t),
\]
and since by definition one has
\[
& P_t^0(\hat{\gamma}_{\RDist}(F^{-1}\circ \gamma,\hat{F}\cdot q_0\cdot F))\dot{\hat{\gamma}}_{\RDist}(F^{-1}\circ\gamma,\hat{F}\cdot q_0\cdot F) \\
=&(\hat{F}_*A_0F_*)P_t^0(F^{-1}\circ \gamma)\dif{t}(F^{-1}\circ \gamma)(t),
\]
the uniqueness of solutions of a system of ODEs 
gives that
\[
& \hat{F}\circ \hat{\gamma}_{\RDist}(\gamma,q_0)=\hat{\gamma}_{\RDist}(F^{-1}\circ\gamma,\hat{F}\cdot q_0\cdot F).
\]
Hence
\[
& \hat{F}_*A_{\RDist}(\gamma,q_0)F_*
=\hat{F}_*(P_0^t(\hat{\gamma}_{\RDist}(\gamma,q_0))\circ A_0\circ P^0_t(\gamma))F_* \\
=&P_0^t(\hat{F}\circ \hat{\gamma}_{\RDist}(\gamma,q_0))\circ (\hat{F}_*A_0F_*)\circ P^0_t(F^{-1}\circ \gamma) \\
=&P_0^t(\hat{\gamma}_{\RDist}(F^{-1}\circ\gamma,\hat{F}\cdot q_0\cdot F))\circ (\hat{F}_*A_0F_*)\circ P^0_t(F^{-1}\circ \gamma)
=A_{\RDist}(F^{-1}\circ \gamma,\hat{F}\cdot q_0\cdot F)
\]
which proves (\ref{eq:equivariance_of_rolling}).

\end{proof}

\begin{corollary}
Let $q_0=(x_0,\hat{x}_0;A_0)\in Q$ and $\gamma,\omega:[0,1]\to M$ be absolutely continuous
such that $\gamma(0)=\omega(0)=x_0$, $\gamma(1)=\omega(1)$.
Then assuming that $q_{\RDist}(\gamma,q_0)$,
$q_{\RDist}(\omega,q_0)$, $q_{\RDist}(\omega^{-1}.\gamma,q_0)$ exist
and if there exists $\hat{F}\in\Iso(\hat{M},\hat{g})$
such that
\[
\hat{F}\cdot q_0=q_{\RDist}(\omega^{-1}.\gamma,q_0)(1),
\]
then
\[
\hat{F}\cdot q_{\RDist}(\omega,q_0)(1)=q_{\RDist}(\gamma,q_0)(1).
\]
\end{corollary}

\begin{proof}
\[
& q_{\RDist}(\gamma,q_0)(1)
=q_{\RDist}(\omega.\omega^{-1}.\gamma,q_0)(1)
=\big(q_{\RDist}(\omega,q_{\RDist}(\omega^{-1}.\gamma,q_0)(1)).q_{\RDist}(\omega^{-1}.\gamma,q_0)\big)(1) \\
=&\big(q_{\RDist}(\omega,\hat{F}\cdot q_0).q_{\RDist}(\omega^{-1}.\gamma,q_0)\big)(1)
=q_{\RDist}(\omega,\hat{F}\cdot q_0)(1)=\hat{F}\cdot q_{\RDist}(\omega,q_0)(1).
\]
\end{proof}

\begin{proposition}\label{cov}
Let $\pi_1:(M_1,g_1)\to (M,g)$ and $\hat{\pi}:(\hat{M}_1,\hat{g}_1)\to (\hat{M},\hat{g})$
be Riemannian coverings.
Write $Q_1=Q(M_1,\hat{M}_1)$ and $(\RDist)_1$
for the rolling distribution in $Q_1$. Then the map
\[
\Pi:Q_1\to Q;
\quad \Pi(x_1,\hat{x}_1;A_1)=\big(\pi(x_1),\hat{\pi}(\hat{x}_1);\hat{\pi}_*|_{\hat{x}_1}\circ A_1\circ (\pi_*|_{x_1})^{-1}\big)
\]
is a covering map of $Q_1$ over $Q$ and
\[
\Pi_*(\RDist)_1=\RDist.
\]
Moreover, for every $q_1\in Q_1$ the restriction onto $\mc{O}_{(\RDist)_1}(q_1)$
of $\Pi$ is a covering map $\mc{O}_{(\RDist)_1}(q_1)\to \mc{O}_{\RDist}(\Pi(q_1))$.
Then, for every $q_1\in Q_1$, $\Pi(\mc{O}_{(\RDist)_1}(q_1))=\mc{O}_{\RDist}(\Pi(q_1))$ 
and  one has $\mc{O}_{(\RDist)_1}(q_1)=Q_1$ 
if and only if $\mc{O}_{\RDist}(\Pi(q_1))=Q$. 
\end{proposition}
As an immediate corollary  of the above proposition, we obtain the following result regarding the complete controllability of $(\RDist)$.
\begin{corollary}\label{CC-simply-conn}
Let $\pi_1:(M_1,g_1)\to (M,g)$ and $\hat{\pi}:(\hat{M}_1,\hat{g}_1)\to (\hat{M},\hat{g})$
be Riemannian coverings. Write $Q=Q(M,\hat{M})$, $\RDist$ and $Q_1=Q(M_1,\hat{M}_1)$, $(\RDist)_1$ respectively for the state space and 
for the rolling distribution in the respective state space. Then the control system associated to $\RDist$ is completely controllable if and only if the control system associated to $(\RDist)_1$ is completely controllable. As a consequence, when one addresses the complete
controllability issue for the rolling distribution $\RDist$, one can assume with no loss of generality that both manifolds $M$ and $\hat{M}$ are simply connected. 
\end{corollary}

We now proceed with the proof of Proposition \ref{cov}.
\begin{proof}
It is clear that $\Pi$ is a local diffeomorphism onto $Q$.
To show that it is a covering map,
let $q_1=(x_1,\hat{x}_1;A_1)$
and choose evenly covered w.r.t $\pi$, $\hat{\pi}$ open sets $U$ and $\hat{U}$
of $M$, $\hat{M}$ containing $\pi(x_1)$, $\hat{\pi}(\hat{x}_1)$,
respectively. Thus $\pi^{-1}(U)=\bigcup_{i\in I} U_i$
and $\hat{\pi}^{-1}(\hat{U})=\bigcup_{i\in \hat{I}} \hat{U}_i$
where $U_i$, $i\in I$ (resp. $\hat{U}_i$, $i\in\hat{I}$)
are mutually disjoint connected open subsets of $M_1$ (resp. $\hat{M}_1$)
such that $\pi$ (resp. $\hat{\pi}$) maps each $U_i$ (resp. $\hat{U}_i$) diffeomorphically onto $U$ (resp. $\hat{U}$).
Then
\[
\Pi^{-1}(\pi_Q^{-1}(U\times\hat{U}))
=\pi_{Q_1}^{-1}((\pi\times\hat{\pi})^{-1}(U\times\hat{U}))
=\bigcup_{i\in I,j\in\hat{I}}\pi_{Q_1}^{-1}(U_i\times\hat{U}_j),
\]
where $\pi_{Q_1}^{-1}(U_i\times\hat{U}_j)$ for $(i,j)\in I\times\hat{I}$
are clearly mutually disjoint and connected.
Now if for a given $(i,j)\in I\times\hat{I}$
we have $(y_1,\hat{y}_1,B_1),(z_1,\hat{z}_1;C_1)\in \pi_{Q_1}^{-1}(U_i\times\hat{U}_j)$
such that $\Pi(y_1,\hat{y}_1;B_1)=\Pi(z_1,\hat{z}_1,C_1)$,
then $y_1=z_1$, $\hat{y}_1=\hat{z}_1$ 
and hence $B_1=C_1$, which shows that $\Pi$ restricted to $\pi_{Q_1}^{-1}(U_i\times\hat{U}_j)$ is injective.
It is also a local diffeomorphism, as mentioned above, and clearly surjective
onto $\pi_Q^{-1}(U\times\hat{U})$,
which proves that $\pi_Q^{-1}(U\times\hat{U})$ is evenly covered with respect to $\Pi$.
This finishes the proof that $\Pi$ is a covering map.

Suppose next that $q_1(t)=(\gamma_1(t),\hat{\gamma}_1(t);A_1(t))$ is a smooth path on $Q_1$ tangent to 
$(\RDist)_1$ and defined on an interval containing $0\in\R$. Define $q(t)=(\gamma(t),\hat{\gamma}(t);A(t)):=(\Pi\circ q_1)(t)$. Then
\[
\dot{\hat{\gamma}}(t)=&\hat{\pi}_*\dot{\hat{\gamma}}_1(t)=\hat{\pi}_*A_1(t)\dot{\gamma}_1(t)
=A(t)\pi_*\dot{\gamma}_1(t)=A(t)\dot{\gamma}(t) \\
A(t)=&\hat{\pi}_*|_{\hat{\gamma}_1(t)}\circ P_0^t(\hat{\gamma}_1(t))\circ A_1(0)\circ P_t^0(\gamma_1)\circ (\pi_*|_{\gamma_1(t)})^{-1} \\
=&P_0^t(\hat{\gamma}(t))\circ \hat{\pi}_*|_{\hat{\gamma}_1(t)}\circ A_1(0)\circ (\pi_*|_{\gamma_1(t)})^{-1}\circ P_t^0(\gamma) \\
=&P_0^t(\hat{\gamma}(t))\circ A(0)\circ P_t^0(\gamma),
\]
which shows that $q(t)$ is tangent to $\RDist$. This shows that $\Pi_*(\RDist)_1\subset\RDist$
and the equality follows from the fact that $\Pi$ is a local diffeomorphism and
the ranks of $(\RDist)_1$ and $\RDist$ are the same i.e., $=n$.

Let $q_1=(x_1,\hat{x}_1;A_1)$. We proceed to show that the restriction of $\Pi$ gives a covering
$\mc{O}_{(\RDist)_1}(q_1)\to \mc{O}_{\RDist}(\Pi(q_1))$.
First, since $\Pi_*(\RDist)_1=\RDist$ and $\Pi:Q_1\to Q$ is a covering map, it follows that $\Pi(\mc{O}_{(\RDist)_1}(q_1))=\mc{O}_{\RDist}(\Pi(q_1))$.

Let $q:=\Pi(q_1)$ and let $U\subset Q$ be an evenly covered neighbourhood of $q$ w.r.t. $\Pi$.
By the Orbit Theorem, there exists vector fields $Y_1,\dots,Y_d\in\VF(Q)$
tangent to $\RDist$ and $(u_1,\dots,u_d)\in (L^1([0,1]))^d$
and a connected open neighbourhood $W$ of $(u_1,\dots,u_d)$ in $(L^1([0,1]))^d$
such that the image of the end point map $\ENDP_{(Y_1,\dots,Y_d)}(q,W)$
is an open subset of the orbit $\mc{O}_{\RDist}(q)$ containing $q$
and included in the $\Pi$-evenly covered set $U$.
Let $(Y_i)_1$, $i=1,\dots,d$, be the unique vector fields on $Q_1$
defined by $\Pi_*(Y_i)_1=Y_i$, $i=1,\dots,d$.
Since $\Pi_*(\RDist)_1=\RDist$, it follows that $(Y_i)_1$ are tangent to $(\RDist)_1$
and also, $\Pi\circ \ENDP_{((Y_1)_1,\dots,(Y_d)_1)}=\ENDP_{(Y_1,\dots,Y_d)}\circ (\Pi\times \id)$.
It follows that $\ENDP_{((Y_1)_1,\dots,(Y_d)_1)}(q_1',W)$ is an open subset of 
$\mc{O}_{(\RDist)_1}(q_1)$ contained in $\Pi^{-1}(U)$ for every $q_1'\in (\Pi|_{\mc{O}_{(\RDist)_1}(q_1)})^{-1}(q)$.

Since $\ENDP_{((Y_1)_1,\dots,(Y_d)_1)}$ is continuous and $W$ is connected,
it thus follows that for each $q_1'\in (\Pi|_{\mc{O}_{(\RDist)_1}(q_1)})^{-1}(q)$,
the connected set $\ENDP_{((Y_1)_1,\dots,(Y_d)_1)}(q_1',W)$
is contained in a single component of $\Pi^{-1}(U)$
which, since $U$ was evenly covered, is mapped diffeomorphically by $\Pi$ onto $U$.
But then $\Pi$ maps $\ENDP_{((Y_1)_1,\dots,(Y_d)_1)}(q_1',W)$
diffeomorphically onto $\ENDP_{(Y_1,\dots,Y_d)}(q,W)$.
Since it is also obvious that
\[
(\Pi|_{\mc{O}_{(\RDist)_1}(q_1)})^{-1}\big(\ENDP_{(Y_1,\dots,Y_d)}(q,W)\big)
=\bigcup_{q_1'\in (\Pi|_{\mc{O}_{(\RDist)_1}(q_1)})^{-1}(q)} \ENDP_{((Y_1)_1,\dots,(Y_d)_1)}(q_1',W),
\]
we have proved that $\ENDP_{(Y_1,\dots,Y_d)}(q,W)$ is an evenly covered
neighbourhood of $q$ in $\mc{O}_{\RDist}(q)$ w.r.t $\Pi|_{\mc{O}_{(\RDist)_1}(q_1)}$.

Finally, let us prove that for every $q_1\in Q_1$, the following implication
holds true, 
\[
\mc{O}_{\RDist}(\Pi(q_1))=Q\quad \Longrightarrow\quad
\mc{O}_{(\RDist)_1}(q_1)=Q_1,
\]
 (the converse statement being trivial).
Indeed, if $\mc{O}_{\RDist}(\Pi(q_1))=Q$,
then, for every $q\in Q$, $\mc{O}_{\RDist}(q)=Q$ 
and, on the other hand, the fact that $\Pi$ restricts to a covering map $\mc{O}_{(\RDist)_1}(q_1')\to \mc{O}_{\RDist}(\Pi(q_1'))=Q$
for any $q'_1\in Q_1$
implies that all the orbits $\mc{O}_{(\RDist)_1}(q_1')$, $q_1'\in Q_1$,
are open on $Q_1$. But $Q_1$ is connected (and orbits are non-empty) and hence there cannot be but one orbit.
In particular, $\mc{O}_{(\RDist)_1}(q_1)=Q_1$.
\end{proof}

%%%%%%%%%%%%%%%%%%%%%%%%%%%%%%%%%%%%%%%%%%%%%%%%%%%%%%%%%%%%
\subsection{Rolling Curvature and Lie Algebraic Structure of $\RDist$}
%%%%%%%%%%%%%%%%%%%%%%%%%%%%%%%%%%%%%%%%%%%%%%%%%%%%%%%%%%%%
\subsubsection{Rolling Curvature}
We compute some commutators of the
vector fields of the form $\LRD(X)$ with $X\in\VF(M)$.
The formulas obtained hold both in $Q$ and $T^*M\otimes T\hat{M}$
and thus we do them in the latter space.

The first commutators of the $\RDist$-lifted fields are given in the following theorem.

\begin{proposition}\label{pr:2.5:1}\label{pr:R_comm_L2}
If $X,Y\in \VF(M)$, $q=(x_0,\hat{x}_0;A)\in T^*(M)\otimes T(\hat{M})$, then
the commutator of the lifts $\LRD(X)$ and $\LRD(Y)$ at $q$ is given by
\begin{align}\label{eq:2.5:4}
[\LRD(X),\LRD(Y)]|_q
&=\LRD([X,Y])|_q+\nu(AR(X,Y)-\hat{R}(AX,AY)A)|_q.
\end{align}
\end{proposition}

\begin{proof}
Choosing $\ol{T}(B)=(X, BX)$, $\ol{S}(B)=(Y, BY)$ for $B\in T^*(M)\otimes T(\hat{M})$ in proposition \ref{pr:NS_comm_HH}
we have
\[
[\LRD(X),\LRD(Y)]|_q=\LNSD([X+\tilde{A}X,Y+\tilde{A}Y])|_q+\nu(AR(X,Y)-\hat{R}(AX,AY)A)|_q,
\]
where
\[
&\big[(X, \tilde{A}X),(Y, \tilde{A}Y)\big]\big|_{(x_0,\hat{x_0})}
=\ol{\nabla}_{(X, \tilde{A}X)} (Y, \tilde{A}Y)-\ol{\nabla}_{(Y, \tilde{A}Y)} (X, \tilde{A}X) \\
=&\big(\nabla_X Y-\nabla_Y X, \hat{\nabla}_{AX} (\tilde{A}Y)-\hat{\nabla}_{AX} (\tilde{A}Y)\big)\big|_{(x_0,\hat{x}_0)}+\ol{\nabla}_{(0,AX)} Y-\ol{\nabla}_{(0,AY)} X\\
&+\ol{\nabla}_{(X,0)} (\tilde{A}Y)|_{(x_0,\hat{x}_0)}-\ol{\nabla}_{(Y,0)} (\tilde{A}X)|_{(x_0,\hat{x}_0)},
\]
in which e.g.
\[
& \hat{\nabla}_{AX} (\tilde{A}Y)|_{\hat{x}_0}
=(\ol{\nabla}_{(0,AX)}\tilde{A})|_{(x_0,\hat{x}_0)}\tilde{Y}|_{\hat{x}_0}+\tilde{A}|_{(x_0,\hat{x}_0)}(\nabla_{0} Y)|_{x_0}=0, \\
& \ol{\nabla}_{(0,AX)} Y=0, \\
& \ol{\nabla}_{(X,0)} (\tilde{A}Y)|_{(x_0,\hat{x}_0)}
=(\ol{\nabla}_{(X,0)} \tilde{A})|_{(x_0,\hat{x}_0)}Y|_{x_0}+A\nabla_{X} Y|_{x_0}
=A\nabla_{X} Y|_{x_0}.
\]
Therefore
\[
[\LRD(X),\LRD(Y)]|_q
=&\LNSD\Big(\big(\nabla_X Y-\nabla_Y X, 0\big)+0+\big(0, A\nabla_X Y-A\nabla_Y X\big)\Big)\Big|_q \\
&+\nu(AR(X,Y)-\hat{R}(AX,AY)A)|_q,
\]
which proves the claim after noticing
that, by torsion freeness of $\nabla$,
one has $\nabla_X Y-\nabla_Y X=[X,Y]$.
\end{proof}

Proposition \ref{pr:R_comm_L2} justifies the next definition.

\begin{definition}\label{def-rol}
Given vector fields $X,Y,Z_1,\dots,Z_k\in \VF(M)$, we define the \emph{Rolling Curvature} of the rolling of $M$ against $\hat{M}$ as 
the smooth mapping
$$\Rol(X,Y):\pi_{T^*M\otimes T\hat{M}}\to \pi_{T^*M\otimes T\hat{M}},$$ by
\begin{align}
 \Rol(X,Y)(A):=AR(X,Y)-\hat{R}(AX,AY)A, \label{rol0}.
\end{align}
Moreover, we use $\Rol_q$ to denote the linear map $\wedge^2 T|_x M\to T^*|_x M\wedge T|_{\hat{x}}\hat{M}$ defined on pure elements of $\wedge^2 T|_x M$ by 
\begin{align}\label{def:rol-q}
\Rol_q(X\wedge Y)= \Rol(X,Y)(A).
\end{align}
\end{definition}
Similarly, for $k\geq 0$, the smooth mapping
$$\ol{\nabla}^k\Rol(X,Y,Z_1,\dots,Z_k):\pi_{T^*M\otimes T\hat{M}}\to \pi_{T^*M\otimes T\hat{M}},$$
by
\begin{align}
\ol{\nabla}^k\Rol(X,Y,Z_1,\dots,Z_k)(A)&:=
A\nabla^k R(X,Y,(\cdot),Z_1,\dots,Z_k)\nonumber\\
&-\hat{\nabla}^k \hat{R}(AX,AY,A(\cdot),AZ_1,\dots,AZ_k).\label{rol1}
\end{align}
Restricting to $Q$, we have
\[
\Rol(X,Y),\ol{\nabla}^k \Rol(X,Y,Z_1,\dots,Z_k)(A)\in C^\infty(\pi_{Q},\pi_{T^*M\otimes T\hat{M}}),
\]
such that, for all $(x,\hat{x};A)\in Q$,
\[
\Rol(X,Y)(A),\ol{\nabla}^k \Rol(X,Y,Z_1,\dots,Z_k)(A)\in A(\so(T|_x M)).
\]

\begin{remark}
With this notation, Eq. (\ref{eq:2.5:4}) of Proposition \ref{pr:2.5:1} can be written as
\[
[\LRD(X),\LRD(Y)]|_q&=\LRD([X,Y])|_q+\nu(\Rol(X,Y)(A))|_q.
\]
\end{remark}

\begin{remark}
Recall that both $R|_x$, $\hat{R}|_{\hat{x}}$ are (real) symmetric endomorphisms on $\wedge^2 T|_x M$ and $\wedge^2 T|_{\hat{x}} \hat{M}$ respectively. 
Since $A:T|_x M\to T|_{\hat{x}}\hat{M}$ is an isometry,
it follows that $\Rol_q$, defined in \eqref{de:rol-q}, is a (real) symmetric map $\wedge^2 T|_x M\to \wedge^2 T|_{\hat{x}}\hat{M}$. Here of course, we understand that $\wedge^2 T|_x M$ and $\wedge^2 T|_{\hat{x}} \hat{M}$ are endowed with the metrics induced in a natural way from $g|_x$ and $\hat{g}|_{\hat{x}}$.
\end{remark}
In order to take advantage of the spectral properties of a (real) symmetric endomorphism, we introduce the following operator associated to the rolling curvature.

Recall that using the metric $g$, one may identify $T^*|_xM\wedge T|_xM=\so(T|_x M)$ with $\wedge^2 T|_x M$
as we usually do without mention.
Given this, we make the following definition.

\begin{definition}\label{def:rol-tilde}
If $q=(x,\hat{x};A)\in Q$, let $\widetilde{\Rol}_q:\wedge^2 T|_x M\to \wedge^2 T|_x M$ be the  (real) symmetric endomorphism defined by
\begin{align}\label{eq:rol-tilde}
\widetilde{\Rol}_q:=A^{\ol{T}}\Rol_q.
\end{align}
In particular, eigenvalues of $R|_x$, $\hat{R}|_{\hat{x}}$ and $\widetilde{\Rol}_q$  are real and the eigenspaces corresponding to distinct eigenvalues are orthogonal one to the other.
\end{definition}

Recall that, on a Riemannian manifold $(N,h)$,
a smooth vector field $t\mapsto Y(t)$
along a smooth curve $t\mapsto \gamma(t)$
is a Jacobi field if $Y$ satisfies the following second order ODE:
\[
\nabla^h_{\dot{\gamma}(t)}\nabla^h_{\dot{\gamma}(\cdot)} Y(\cdot)=R^h(\dot{\gamma}(t),Y(t))\dot{\gamma}(t).
\]
The next lemma relates the rolling curvature $\Rol$
to the Jacobi fields of $M$ and $\hat{M}$.

\begin{lemma}\label{le:rol_Jacobi}
Suppose that $q_0=(x_0,\hat{x}_0;A_0)\in T^*M\otimes T\hat{M}$,
$\gamma:[a,b]\to M$ is a smooth curve with $\gamma(a)=x_0$
and that the rolling problem along $\gamma$
has a solution $q_{\RDist}{(\gamma,q_0)}=(\gamma,\hat{\gamma}_{\RDist}{(\gamma,q_0)};A_{\RDist}{(\gamma,q_0)})$
on the interval $[a,b]$.
If $t\mapsto Y(t)$ is a Jacobi field of $(M,g)$ 
along $\gamma$,
then $\hat{Y}(t)=A_{\RDist}{(\gamma,q_0)}(t)Y(t)$ is a vector field along
$\hat{\gamma}_{\RDist}{(\gamma,q_0)}$ and, for all $t\in [a,b]$,
\[
\hat{\nabla}_{\dot{\hat{\gamma}}_{\RDist}{(\gamma,q_0)}(t)} \hat{\nabla}_{\dot{\hat{\gamma}}_{\RDist}{(\gamma,q_0)}(\cdot)} \hat{Y}(\cdot)
=&\hat{R}\big(\dot{\hat{\gamma}}_{\RDist}{(\gamma,q_0)}(t),\hat{Y}(t)\big)\dot{\hat{\gamma}}_{\RDist}{(\gamma,q_0)}(t) \\
&+\Rol(\dot{\gamma}(t),Y(t))(A_{\RDist}{(\gamma,q_0)}(t))\dot{\gamma}(t).
\]
\end{lemma}

\begin{proof}
Since $\ol{\nabla}_{(\dot{\gamma}(t),\dot{\hat{\gamma}}_{\RDist}{(\gamma,q_0)}(t))} A_{\RDist}{(\gamma,q_0)}(\cdot)=0$
and $Y$ is a Jacobi field, one has
\[
& \hat{\nabla}_{\dot{\hat{\gamma}}_{\RDist}{(\gamma,q_0)}(t)} \hat{\nabla}_{\dot{\hat{\gamma}}_{\RDist}{(\gamma,q_0)}(\cdot)} \hat{Y}(\cdot) \\
=&A_{\RDist}{(\gamma,q_0)}(t)\nabla_{\dot{\gamma}(t)} \nabla_{\dot{\gamma}(\cdot)} Y(\cdot)
=A_{\RDist}{(\gamma,q_0)}(t)R(\dot{\gamma}(t),Y(t))\dot{\gamma}(t) \\
=&\Rol(\dot{\gamma}(t),Y(t))(A_{\RDist}{(\gamma,q_0)}(t))\dot{\gamma}(t) \\
&+\hat{R}(A_{\RDist}{(\gamma,q_0)}(t)\dot{\gamma}(t),A_{\RDist}{(\gamma,q_0)}(t)Y(t))A_{\RDist}{(\gamma,q_0)}(t)\dot{\gamma}(t)
\]
from which the claim follows by using the facts that
$$A_{\RDist}{(\gamma,q_0)}(t)\dot{\gamma}(t)=
\dot{\hat{\gamma}}_{\RDist}{(\gamma,q_0)}(t)\hbox{ and }
A_{\RDist}{(\gamma,q_0)}(t)Y(t)=\hat{Y}(t).$$
\end{proof}

We will use Lemma \ref{le:rol_Jacobi} to prove Theorem \ref{th:cartan}.

\begin{remark}
Notice that if, in Lemma \ref{le:rol_Jacobi},
it held that 
$$\Rol(Y(t),\dot{\gamma}(t))(A_{\RDist}{(\gamma,q_0)}(t))\dot{\gamma}(t)=0,$$
for all $t\in [a,b]$,
then $\hat{Y}$ defined there
would be a Jacobi field along $\hat{\gamma}_{\RDist}{(\gamma,q_0)}$.
Hence, $\Rol$ measures the obstruction for $\hat{Y}=A_{\RDist}{(\gamma,q_0)}(t)Y(t)$ to be a Jacobi field
of $\hat{M}$, if $Y(t)$ is a Jacobi field on $M$ along $\gamma$.
\end{remark}

Before proceeding with the computations of higher order brackets of the vector fields $\LRD(X)$,
we prove the following lemma.

\begin{lemma}\label{le:Rbar_cov_diff}
Let $\tilde{A}\in\Gamma(\pi_{T^*M\otimes T\hat{M}})$ and $(x,\hat{x};A)\in T^*M\otimes T\hat{M}$
such that $\tilde{A}|_{(x,\hat{x})}=A$ and $\ol{\nabla}_{\ol{X}} \tilde{A}=0$ for all $\ol{X}\in T|_{(x,\hat{x})} (M\times \hat{M})$.
Then, for $X_1,\dots,X_{k+2},Y\in \VF(M)$,
\begin{align}\label{eq:Rbar_cov_diff}
& \ol{\nabla}_{(Y, AY)} \big(\ol{\nabla}^k \Rol(X_1,X_2,X_3,\dots,X_{k+2})(\tilde{A})\big) \\
=&\ol{\nabla}^{k+1} \Rol(X_1,\dots,X_{k+2},Y)(A)+\sum_{i=1}^{k+2} \ol{\nabla}^k \Rol(X_1,\dots,\nabla_Y X_i \dots,X_{k+2})(A) 
\end{align}
\end{lemma}

\begin{proof}
If $k=0$, we have $\ol{\nabla}^0 \Rol(X_1,X_2)(\tilde{A})=\Rol(X_1,X_2)(\tilde{A})$ and
since $\ol{\nabla}_{(Y, AY)} \tilde{A}=0$, one gets
\[
& \ol{\nabla}_{(Y, AY)} \Rol(X_1,X_2)(\tilde{A})
=\ol{\nabla}_{(Y, AY)} \big(\tilde{A}R(X_1,X_2)-\hat{R}(\tilde{A}X_1,\tilde{A}X_2)\tilde{A}\big) \\
=& A \nabla_{Y} (R(X_1,X_2))-(\ol{\nabla}_{(Y, AY)} \hat{R}(\tilde{A}X_1,\tilde{A}X_2)) \tilde{A} \\
=&A\nabla R(X_1,X_2,(\cdot),Y)+AR(\nabla_Y X_1,X_2)+AR(X_1,\nabla_Y X_2) \\
&-\hat{\nabla} \hat{R}(AX_1,AX_2,A(\cdot),AY)-\hat{R}(A\nabla_Y X_1,AX_2)A-\hat{R}(AX_1,A\nabla_Y X_2),
\]
where on the last line we have computed $\ol{\nabla}_{(Y, AY)} (\tilde{A}X_i)=A\nabla_Y X_i$.
The case $k> 0$ is proved by induction and similar computations.  

\end{proof}

\subsubsection{Computation of more Lie brackets}

\begin{proposition}\label{pr:R_comm_L3}
Let $X,Y,Z\in\VF(M)$. Then, for $q=(x,\hat{x};A)\in T^*M\otimes T\hat{M}$, one has
\[
[\LRD(Z),\nu(\Rol(X,Y)(\cdot))]|_{q}
=&-\LNSD(\Rol(X,Y)(A)Z)|_q+\nu\big(\ol{\nabla}^1\Rol(X,Y,Z)(A)\big)\big|_q \\
& +\nu\big(\Rol(\nabla_Z X,Y)(A)\big)\big|_q
+\nu\big(\Rol(X,\nabla_Z Y)(A)\big)\big|_q.
\]
\end{proposition}

\begin{proof}
Taking $\ol{T}(B)=(Z, BZ)$ and $U=\Rol(X,Y)$ for $B\in T^*M\otimes T\hat{M}$
in Proposition \ref{pr:NS_comm_HV}, we get
\[
& [\LRD(Z),\nu(\Rol(X,Y)(\cdot))]|_{q} \\
=&-\LNSD(\nu(\Rol(X,Y)(A))|_q (Z+(\cdot)Z))|_q+\nu(\ol{\nabla}_{Z+AZ} (\Rol(X,Y)(\tilde{A})))|_q.
\]
From here, one easily computes that
\[
\nu(\Rol(X,Y)(A))|_q (Z+(\cdot)Z)
=\dif{t}|_0 \big(Z+(A+t\Rol(X,Y)(A))Z\big)=\Rol(X,Y)(A)Z,
\]
and by Lemma \ref{le:Rbar_cov_diff}, one gets
\[
\ol{\nabla}_{Z+AZ} (\Rol(X,Y)(\tilde{A}))
=\ol{\nabla}^1\Rol(X,Y,Z)(A)+\Rol(\nabla_Z X,Y)(A)+\Rol(X,\nabla_Z Y)(A).
\]

\end{proof}

By Proposition \ref{pr:R_comm_L2}, the last two terms (when considered as vector fields on $T^*M\otimes T\hat{M}$) on the right hand side belong to $\VF_{\RDist}^2$.

Since for $X,Y\in\VF(M)$ and $q=(x,\hat{x};A)\in Q$
we have $\nu(\Rol(X,Y)(A))|_q\in \mc{O}_{\RDist}(q)$ by Proposition \ref{pr:R_comm_L2},
it is reasonable to compute the Lie-bracket of two elements of this type.
This is given in the following proposition.

\begin{proposition}\label{vert-vert0}
For any $q=(x,\hat{x};A)\in Q$ and $X,Y,Z,W\in\VF(M)$ we have
\[
& \big[\nu(\Rol(X,Y)(\cdot)),\nu(\Rol(Z,W)(\cdot))\big]\big|_q \\
=&
\nu\big(\Rol(X,Y)(A)R(Z,W)-\hat{R}(\Rol(X,Y)(A)Z,AW)A-\hat{R}(AZ,\Rol(X,Y)(A)W)A \\
&-\hat{R}(AZ,AW)\Rol(X,Y)(A)
-\Rol(Z,W)(A)R(X,Y)+\hat{R}(\Rol(Z,W)(A)X,AY)A \\
&+\hat{R}(AX,\Rol(Z,W)(A)Y)A+\hat{R}(AX,AY)\Rol(Z,W)(A)\big)\big|_q.
\]
\end{proposition}

\begin{proof}
We use Proposition \ref{pr:NS_comm_VV} where for $U,V$ we take
$U(A)=\Rol(X,Y)(A)$ and $V(A)=\Rol(Z,W)(A)$.
First compute for $B$ such that $\nu(B)|_q\in V|_q(Q)$ that
\[\nu(B)|_q U
=&\nu(B)|_q \big(\tilde{A}\mapsto \tilde{A}R(X,Y)-\hat{R}(\tilde{A}X,\tilde{A}Y)
\tilde{A}\big) \\
=&\dif{t}\big|_0 \big((A+tB)R(X,Y)-\hat{R}((A+tB)X,(A+tB)Y)(A+tB)\big) \\
=&BR(X,Y)-\hat{R}(BX,AY)A-\hat{R}(AX,BY)A-\hat{R}(AX,AY)B
\]
So by taking $B=V(A)$
we get
\[\nu(V(A))|_q U
=&\Rol(Z,W)(A)R(X,Y)-\hat{R}(\Rol(Z,W)(A)X,AY)A \\
&-\hat{R}(AX,\Rol(Z,W)(A)Y)A-\hat{R}(AX,AY)\Rol(Z,W)(A)
\]
and similarly for $\nu(U(A))|_q V$.
\end{proof}

For later use, we find it convenient to provide another expression for Proposition 
\ref{vert-vert0} and, for that purpose, we recall the following notation. For $A,B\in \so(T|_x M)$, we define
\[
[A,B]_{\so}:=A\circ B-B\circ A\in \so(T|_x M).
\]
Then, one has the following corollary. 

\begin{corollary}\label{cor:vcomm}
For any $q=(x,\hat{x};A)\in Q$ and $X,Y,Z,W\in\VF(M)$ we have
\begin{align}\label{eq:vcomm}
& \nu|_q^{-1}\big[\nu(\Rol(X,Y)(\cdot)),\nu(\Rol(Z,W)(\cdot))\big]\big|_q \nonumber \\
=&A\big[R(X,Y),R(Z,W)\big]_{\so}-\big[\hat{R}(AX,AY),\hat{R}(AZ,AW)\big]_{\so}A \nonumber \\
&-\hat{R}(\Rol(X,Y)(A)Z,AW)A-\hat{R}(AZ,\Rol(X,Y)(A)W)A \nonumber \\
&+\hat{R}(AX,\Rol(Z,W)(A)Y)A+\hat{R}(\Rol(Z,W)(A)X,AY)A.
\end{align}
\end{corollary}

\begin{proof} This is immediate by standard computations and the definition of $\Rol$.
\end{proof}

From Proposition \ref{pr:NS_comm_HV} we get the following proposition.

\begin{proposition}\label{pr:H_directions_from_V}
Let $q_0=(x_0,\hat{x}_0;A_0)\in Q$. Suppose that, for some $X\in\VF(M)$ and a real sequence $(t_n)_{n=1}^\infty$ s.t. $t_n\neq 0$ for all $n$, $\lim_{n\to\infty} t_n=0$, we have
\begin{align}\label{eq:V_in_orbit}
V|_{\Phi_{\LRD(X)}(t_n,q_0)}(\pi_Q)\subset T(\mc{O}_{\RDist}(q_0)),\quad \forall n.
\end{align}
Then $\LNSD(Y, \hat{Y})\big|_{q_0}\in T|_{q_0} \mc{O}_{\RDist}(q_0)$
for every $Y\in T|_{x_0} M$ that is $g$-orthogonal to $X|_{x_0}$
and every $\hat{Y}\in T|_{\hat{x}_0} \hat{M}$ that is $\hat{g}$-orthogonal to $A_0X|_{x_0}$.
Hence the orbit $\mc{O}_{\RDist}(q_0)$
has codimension at most $1$ inside $Q$.

\end{proposition}

\begin{proof}
Letting $n$ tend to infinity, it follows from (\ref{eq:V_in_orbit}) that $V|_{q_0} (\pi_Q)\subset T|_{q_0} \mc{O}_{\RDist}(q_0)$.
Recall, from Proposition \ref{pr:vertical_of_Q}, that every element of $V|_{q_0}(\pi_Q)$
is of the form $\nu(B)|_{q_0}$, with a unique $B\in Q|_{(x_0,\hat{x}_0)}$ satisfying
$A_0^{\ol{T}} B\in \so(T|_{x_0} M)$.
Fix such a $B$ and define a smooth local section $\tilde{S}$ 
of $\so(TM)\to M$
defined on an open set $W\ni x_0$
by
\[
\tilde{S}|_{x}=P_0^1\big(t\mapsto \exp_{x_0}(t\exp_{x_0}^{-1}(x))\big) (A_0^{\ol{T}} B).
\]
Then clearly, $\tilde{S}|_{x_0}=A_0^{\ol{T}} B$ and $\nabla_Y \tilde{S}=0$ for all $Y\in T|_{x_0} M$ and it is easy to verify that $\tilde{S}|_{x}\in \so(T|_x M)$ for all $x\in W$. 

We next define a smooth map $U:\pi_Q^{-1}(W\times \hat{M})\to T^*M\otimes T\hat{M}$ by $U(x,\hat{x};A)=A\tilde{S}|_x$.
Obviously $\nu(U(x,\hat{x};A))\in V|_{(x,\hat{x};A)} (\pi_Q)$ for all $(x,\hat{x};A)$.
Then, choosing in Proposition \ref{pr:NS_comm_HV}, $\ol{T}=X+(\cdot)X$ (and the above $U$) and noticing that
\[
\nu(U(A_0))|_{q_0} \ol{T}=U(A_0)X=BX,
\]
one gets
\begin{align}\label{eq:special_NS_comm_VH}
[\LRD(X),\nu(U(\cdot))]|_{q_0}=-\LNSD(BX)|_{q_0}
+\nu(\ol{\nabla}_{(X, A_0X)}(U(\tilde{A})))|_{q_0}
\end{align}
where $\tilde{A}|_{(x_0,\hat{x}_0)}=A_0$.
By the choice of $\tilde{S}$ and $\tilde{A}$, we have, for all $\ol{Y}=(Y, \hat{Y})\in T|_{(x_0,\hat{x}_0)} M\times\hat{M}$,
\[
\nabla_{\ol{Y}} (U(\tilde{A}))=\nabla_{\ol{Y}} (\tilde{A}\tilde{S})=(\nabla_{\ol{Y}}\tilde{A})\tilde{S}|_{(x_0,\hat{x}_0)}+\tilde{A}|_{(x_0,\hat{x}_0)}\nabla_{Y}\tilde{S}=0,
\]
and hence the last term on the right hand side of (\ref{eq:special_NS_comm_VH}) actually vanishes.

By definition, the vector field $q\mapsto \LRD(X)|_q$
is tangent to the orbit $\mc{O}_{\RDist}(q_0)$
and, by the assumption of Equation (\ref{eq:V_in_orbit}), the values of the map $q=(x,\hat{x};A)\mapsto \nu(U(A))|_q$ are also tangent to $\mc{O}_{\RDist}(q_0)$
at the points $\Phi_{\LRD(X)}(t_n,q_0)$, $n\in\N$.
Hence $\big((\Phi_{\LRD(X)})_{-t_n}\big)_* \nu(U(\cdot))|_{\Phi_{\LRD(X)}(t_n,q_0)}\in T|_{q_0} \mc{O}_{\RDist}(q_0)$ and therefore,
\[
& [\LRD(X),\nu(U(\cdot))]|_{q_0} \\
=&\lim_{n\to\infty} \frac{\big((\Phi_{\LRD(X)})_{-t_n}\big)_* \nu(U(\cdot))|_{\Phi_{\LRD(X)}(t_n,q_0)}-\nu(B)|_{q_0}}{t_n}\in T|_{q_0} \mc{O}_{\RDist}(q_0),
\]
i.e., the left hand side of (\ref{eq:special_NS_comm_VH}) must belong to $T|_{q_0} \mc{O}_{\RDist}(q_0)$.
But this implies that
\[
\LNSD(BX)|_{q_0}\in T|_{q_0} \mc{O}_{\RDist}(q_0),\quad \forall B\textrm{\ s.t.\ }\nu(B)\in V|_{q_0}(\pi_Q)
\]
i.e.,
\[
\LNSD(A_0\so(T|_{x_0} M)X)|_{q_0}\subset T|_{q_0} \mc{O}_{\RDist}(q_0).
\]
Notice next that $\so(T|_{x_0} M)X$ is exactly the set  $X|_{x_0}^{\perp}$ of vectors of $T|_{x_0} M$ that are $g$-perpendicular to $X|_{x_0}$. Since $A_0\in Q$, it follows that the set $A_0 \so(T|_{x_0} M)X$ is equal to 
$A_0 X|_{x_0}^\perp$ which is the set of vectors of $T|_{\hat{x}_0}\hat{M}$
that are $\hat{g}$-perpendicular to $A_0X|_{x_0}$.
We conclude that $\LNSD(Y)|_{q_0}=\LRD(Y)|_{q_0}-\LNSD(A_0Y)|_{q_0}\in T|_{q_0} \mc{O}_{\RDist}(q_0)$
for all $Y\in X|_{x_0}^\perp$.

Finally notice that
since the subspaces $X^\perp\times\{0\}$, $\R(X,A_0X)$ and $\{0\}\times (A_0X)^\perp$
of $T|_{(x_0,\hat{x}_0)}(M\times\hat{M})$ are linearly independent, their $\LNSD$-lifts at $q_0$ are that also
and hence these lifts span a $(n-1)+1+(n-1)=2n-1$ dimensional
subspace of $T|_{q_0}\mc{O}_{\RDist}(q_0)$. 
This combined with the fact that $V|_{q_0}(\pi_Q)\subset T|_{q_0}\mc{O}_{\RDist}(q_0)$
shows $\dim \mc{O}_{\RDist}(q_0)\geq 2n-1+\dim V|_{q_0}(\pi_Q)=\dim(Q)-1$
i.e., the orbit $\mc{O}_{\RDist}(q_0)$ has codimension at most $1$ in $Q$.
This finishes the proof.

\end{proof}

\begin{corollary}\label{cor:vert}
Suppose there is a point $q_0=(x_0,\hat{x}_0;A_0)\in Q$
and $\epsilon>0$ such that
for every $X\in \VF(M)$ with $\n{X}_g<\epsilon$ on $M$ one has
\[
V|_{\Phi_{\LRD(X)}(t,q_0)}(\pi_Q)\subset T\big(\mc{O}_{\RDist}(q_0)\big),
\quad |t|<\epsilon.
\]
Then the orbit $\mc{O}_{\RDist}(q_0)$ is open in $Q$.

As a consequence, we have the following characterization of complete controllability: the control system $(\Sigma)_R$ is completely controllable if and only if
\begin{equation}\label{spinning0}
\forall q\in Q, \quad V|_{q} (\pi_Q)\subset T|_{q} \mc{O}_{\RDist}(q).
\end{equation}

\end{corollary}

\begin{proof}
For the first part of the corollary, the assumptions and the previous proposition imply that for every $X\in T|_{x_0} M$
we have $\LNSD(Y, \hat{Y})|_{q_0}\in T|_{q_0} \mc{O}_{\RDist}(q_0)$
for every $Y\in X^\perp$, $\hat{Y}\in A_0 X^\perp$.
But since $X$ is an arbitrary element of $T|_{x_0} M$,
this means that $\NSDist|_{q_0}\subset T|_{q_0} \mc{O}_{\RDist}(q_0)$
and because $T|_{q_0} Q=\NSDist|_{q_0}\oplus V|_{q_0}(\pi_Q)$,
we get $T|_{q_0} Q=T|_{q_0} \big(\mc{O}_{\RDist}(q_0)\big)$.
This implies that $\mc{O}_{\RDist}(q_0)$ is open in $Q$.
The last part of the corollary is an immediate consequence of this
and the fact that $Q$ is connected.
\end{proof}

\begin{remark}\label{crabe}
The above corollary is  intuitively obvious. Assumption given by Eq. (\ref{spinning0}) simply means that there is complete freedom for infinitesimal spinning, i.e., for reorienting one manifold with respect to the other one without moving in $M\times \hM$. In that case, proving complete controllability is easy, by using a crab-like motion.
\end{remark}

%%%%%%%%%%%%%%%%%%%%%%%%%%%%%%%%%%%%%%%%%%%%%%%%%
\subsection{Controllability Properties of $\RDist$}
%%%%%%%%%%%%%%%%%%%%%%%%%%%%%%%%%%%%%%%%%%%%%%%%%

%%%%%%%%%%%%%%%%%%%%%%%%%%%
\subsubsection{First Results}
%%%%%%%%%%%%%%%%%%%%%%%%%%%

Proposition \ref{pr:R_comm_L2} has the following simple consequence.

\begin{corollary}\label{cor:2.5:1}
The following cases are equivalent:
\begin{itemize}
\item[(i)] The rolling distribution $\RDist$ on $Q$ is involutive.

\item[(ii)] For all $X,Y,Z\in T|_x M$ and $(x,\hat{x};A)\in T^*(M)\otimes T(\hat{M})$
\[
\Rol(X,Y)(A)=0.
\]

\item[(iii)] $(M,g)$ and $(\hat{M},\hat{g})$ both have constant and equal curvature.
\end{itemize}
The same result holds when one replaces $Q$ by $T^*M\otimes T\hat{M}$.
\end{corollary}

\begin{proof}
(i)$\iff$(ii) follows from Proposition \ref{pr:2.5:1}.

For the rest of the proof, we use 
\[
\sigma_{(X,Y)}=g(R(X,Y)Y,X),\hbox{ and }
\sigma_{(\hat{X},\hat{Y})}=\hat{g}(\hat{R}(\hat{X},\hat{Y})\hat{Y},\hat{X}),
\]
to denote the sectional curvature of $M$ w.r.t orthonormal vectors $X,Y\in T|_x M$ and the sectional curvature of $\hat{M}$ w.r.t. orthonormal vectors $\hat{X},\hat{Y}\in T|_{\hat{x}} \hat{M}$ respectively.
We have seen that the involutivity of $\RDist$ is equivalent
to the condition in (ii) which is again equivalent
(since sectional curvatures completely determine Riemannian curvatures)
to the equation 
\begin{align}\label{eq:cor:2.5.1:1}
\sigma_{(X,Y)}=\hat{\sigma}_{(AX,AY)},\quad \forall (x,\hat{x};A)\in Q,\ X,Y\in T|_x M.
\end{align}

(ii)$\Rightarrow$(iii)
If we fix $x\in M$ and $g$-orthonormal vectors $X,Y\in T|_x M$,
then, for any $\hat{x}\in \hat{M}$ and any $\hat{g}$-orthonormal vectors $\hat{X},\hat{Y}\in T|_{\hat{x}}\hat{M}$,
we may choose $A\in Q|_{(x,\hat{x})}$ such that $AX=\hat{X}$, $AY=\hat{Y}$
(in the case $n=2$ we may have to replace, say, $\hat{X}$ by $-\hat{X}$ but this does not change anything
in the argument below).
Hence the above equation (\ref{eq:cor:2.5.1:1}) shows that the sectional curvatures at every point $\hat{x}\in \hat{M}$
and w.r.t every orthonormal pair $\hat{X},\hat{Y}$ are all the same i.e., $\sigma_{(X,Y)}$. Thus $(\hat{M},\hat{g})$
has constant sectional curvatures i.e., it has a constant curvature. Changing the roles of $M$ and $\hat{M}$
we see that $(M,g)$ also has constant curvature and  the constants of curvatures are the same.

(iii)$\Rightarrow$(ii) Suppose that $M,\hat{M}$ have constant and equal curvatures. 
By a standard result (see \cite{sakai91} Lemma II.3.3), this is equivalent to the fact that
there exists $k\in\R$ such that
\[
& R(X,Y)Z=k\big(g(Y,Z)X-g(X,Z)Y\big),\quad X,Y,Z\in T|_x M,\ x\in M, \\
& \hat{R}(\hat{X},\hat{Y})\hat{Z}=k\big(\hat{g}(\hat{Y},\hat{Z})\hat{X}-\hat{g}(\hat{X},\hat{Z})\hat{Y}\big),\quad
\hat{X},\hat{Y},\hat{Z}\in T|_{\hat{x}}\hat{M},\ \hat{x}\in \hat{M}.
\]
On the other hand, if $A\in Q$, $X,Y,Z\in T|_x M$, we would then have
\[
& \hat{R}(AX,AY)(AZ)=k(\hat{g}(AY,AZ)AX-\hat{g}(AX,AZ)(AY)) \\
=& A(k(g(Y,Z)X-g(X,Z)Y)=A(R(X,Y)Z).
\]
This implies that $\Rol(X,Y)(A)=0$ since $Z$ was arbitrary.

\end{proof}

In the situation of the previous corollary,
the control system $\srol$ is as far away from
being controllable as possible: all the
orbits $\mc{O}_{\RDist}(q)$, $q\in Q$, are integral manifolds of $\RDist$.

The next consequence of Proposition \ref{pr:R_comm_L2}
can be seen as a (partial) generalization of the previous corollary
and a special case of the Ambrose's theorem \ref{th:ambrose}.
The corollary gives a necessary and sufficient condition
describing the case in which at least one $\RDist$-orbit is an integral manifold of $\RDist$.
It will be used in the proof of Theorem \ref{th:fixed_point} below.

\begin{corollary}\label{cor:weak_ambrose}
Suppose that $(M,g)$ and $(\hat{M},\hat{g})$ are complete.
The following cases are equivalent:
\begin{itemize}
\item[(i)] There exists a $q_0=(x_0,\hat{x}_0;A_0)\in Q$
such that the orbit $\mc{O}_{\RDist}(q_0)$ is an integral manifold of $\RDist$.

\item[(ii)] There exists a $q_0=(x_0,\hat{x}_0;A_0)\in Q$
such that
\[
\Rol(X,Y)(A)=0,\quad \forall (x,\hat{x};A)\in \mc{O}_{\RDist}(q_0),\ X,Y\in T|_x M.
\]

\item[(iii)] There is a complete Riemannian manifold $(N,h)$
and Riemannian covering maps $F:N\to M$, $G:N\to\hat{M}$.
In particular, $(M,g)$ and $(\hat{M},\hat{g})$ are locally isometric.
\end{itemize}
\end{corollary}

\begin{proof}
(i) $\Rightarrow$ (ii):
Notice that the restrictions of vector fields $\LRD(X)$, $X\in \VF(M)$,
to the orbit $\mc{O}_{\RDist}(q_0)$
are smooth vector fields of that orbit.
Thus $[\LRD(X),\LRD(Y)]$ is also tangent to this orbit
for any $X,Y\in\VF(M)$
and hence Proposition \ref{pr:R_comm_L2} implies the claim.

(ii) $\Rightarrow$ (i):
It follows, from Proposition \ref{pr:R_comm_L2}, that $\RDist|_{\mc{O}_{\RDist}(q_0)}$, the restriction of $\RDist$ 
to the manifold $\mc{O}_{\RDist}(q_0)$,
is involutive.
Since maximal connected integral manifolds of an involutive
distribution are exactly its orbits, it follows
that $\mc{O}_{\RDist}(q_0)$ is an integral manifold of $\RDist$.

(i) $\Rightarrow$ (iii):
Let $N:=\mc{O}_{\RDist}(q_0)$ and $h:=(\pi_{Q,M}|_N)^*(g)$ i.e.,
for $q=(x,\hat{x};A)\in N$ and $X,Y\in T|_x M$,
define 
\[
h(\LRD(X)|_q,\LRD(Y)|_q)=g(X,Y).
\]
If $F:=\pi_{Q,M}|_N$ and $G:=\pi_{Q,\hat{M}}|_N$,
we  immediately see that $F$ is a local isometry (note that $\dim(N)=n$)
and the fact that $G$ is a local isometry follows from the
following computation: for $q=(x,\hat{x};A)\in N$, $X,Y\in T|_x M$, one has
\[
\hat{g}(G_*(\LRD(X)|_q),G_*(\LRD(Y)|_q))
=\hat{g}(AX,AY)=g(X,Y)=h(\LRD(X)|_q,\LRD(Y)|_q).
\]
The completeness of $(N,h)$ 
can be easily deduced from the completeness of $M$ and $\hat{M}$
together with Proposition \ref{pr:rol_geodesic}.
Proposition II.1.1 in \cite{sakai91} proves that the maps
$F,G$ are in fact (surjective and) Riemannian coverings.

(iii) $\Rightarrow$ (ii):
Let $x_0\in M$ and choose $z_0\in N$ such that $F(z_0)=x_0$.
Define $\hat{x}_0=G(z_0)\in\hat{M}$
and $A_0:=G_*|_{z_0}\circ (F_*|_{z_0})^{-1}$
which is an element of $Q|_{(x_0,\hat{x}_0)}$
since $F,G$ were local isometries. Write $q_0=(x_0,\hat{x}_0;A_0)\in Q$.

Let $\gamma:[0,1]\to M$ be an a.c. curve with $\gamma(0)=x_0$.
Since $F$ is a smooth covering map, there is a unique a.c. curve $\Gamma:[0,1]\to N$
with $\gamma=F\circ\Gamma$ and $\Gamma(0)=z_0$.
Define $\hat{\gamma}=G\circ\Gamma$
and $A(t)=G_*|_{\Gamma(t)}\circ (F_*|_{\Gamma(t)})^{-1}\in Q$, $t\in [0,1]$.
It follows that,  for a.e. $t\in[0,1]$,
$$\dot{\hat{\gamma}}(t)=G_*|_{\Gamma(t)}\dot{\Gamma}(t)=A(t)\dot{\gamma}(t).$$
Since $F,G$ are local isometries,
$\ol{\nabla}_{(\dot{\gamma}(t),\dot{\hat{\gamma}}(t))} A(\cdot)=0$ for a.e. $t\in [0,1]$.
Thus $t\mapsto (\gamma(t),\hat{\gamma}(t);A(t))$ is the unique rolling
curve along $\gamma$ starting at $q_0$ and defined on $[0,1]$
and therefore curves of $Q$ formed in this fashion fill
up the orbit $\mc{O}_{\RDist}(q_0)$.
Moreover, since $F,G$ are local isometries,
it follows that for every $z\in N$ and $X,Y\in T|_{F(z)} M$,
$\Rol(X,Y)(G_*|_z\circ (F_*|_z)^{-1})=0$.
These facts prove that the condition in (ii) holds
and the proof is therefore finished.
\end{proof}

\begin{remark}\label{re:weak_ambrose}
If one does not assume that $(M,g)$ and $(\hat{M},\hat{g})$ are complete in Corollary \ref{cor:weak_ambrose},
then (iii) in the above corollary must be replaced by the following:
\begin{itemize}
\item[(iii)'] There is a connected Riemannian manifold $(N,h)$ (not necessarily complete)
and Riemannian covering maps $F:N\to M^\circ$, $G:N\to \hat{M}^\circ$
where $M^\circ$, $\hat{M}^\circ$ are open sets of $M$ and $\hat{M}$
and there is a $z_0\in N$
such that if $q_0=(x_0,\hat{x}_0;A_0)\in Q$ is defined
by $A_0:=G_*|_{z_0}\circ (F_*|_{z_0})^{-1}$,
then $M^\circ=\pi_{Q,M}(\mc{O}_{\RDist}(q_0))$ and $\hat{M}^\circ=\pi_{Q,\hat{M}}(\mc{O}_{\RDist}(q_0))$.
\end{itemize}
In particular, $M^\circ$, $\hat{M}^\circ$ are connected and $(M^\circ,g)$, $(\hat{M}^\circ,\hat{g})$
are locally isometric.

Indeed, the argument in the implication (i) $\Rightarrow$ (iii)
goes through except for the completeness of $(N,h)$,
where $N=\mc{O}_{\RDist}(q_0)$ (connected).
Proposition \ref{pr:R_orbit_bundle} and Remark \ref{pr:R_orbit_bundle}
show that $F=\pi_{Q,M}|_N:N\to M^\circ$, $G=\pi_{Q,\hat{M}}|_N:N\to \hat{M}^\circ$
are bundles with discrete fibers. Now it is a standard (easy) fact that a bundle
$\pi:X\to Y$ with connected total space $X$ and discrete fibers is a covering map (this could have been
used in the above proof instead of referring to \cite{sakai91}).

On the other hand, in the argument of the implication (iii) $\Rightarrow$ (ii)
we did not even use completeness of $(N,h)$
but only the fact that $F:N\to M$
is a covering map to lift a curve $\gamma$ in $M$ to the curve $\Gamma$ in $Q$.
In this non-complete setting,
we just have to consider using curves $\gamma$ in $M^\circ$
and lift them to $N$ by using $F:N\to M^\circ$.
Indeed, if $q=(x,\hat{x};A)\in \mc{O}_{\RDist}(q_0)$,
there is a curve $\gamma:[0,1]\to M$ such that $q_{\RDist}(\gamma,q_0)(1)=q$.
But for all $t$ one has
\[
\gamma(t)=\pi_{Q,M}(q_{\RDist}(\gamma,q_0)(t))\in \pi_{Q,M}(\mc{O}_{\RDist}(q_0))=M^\circ,
\]
so $\gamma$ is actually a curve in $M^\circ$.

Finally, notice that the assumption in (iii)' that $\hat{M}^\circ=\pi_{Q,\hat{M}}(\mc{O}_{\RDist}(q_0))$
follows from the others. Indeed, making only the other assumptions,
it is first of all clear that if $q$ and $\gamma$ are as above, then
\[
\pi_{Q,\hat{M}}(q)=\pi_{Q,\hat{M}}(q_{\RDist}(\gamma,q_0)(1))=G(\Gamma(1))\in \hat{M}^\circ,
\]
so $\pi_{Q,\hat{M}}(\mc{O}_{\RDist}(q_0))\subset \hat{M}^\circ$.
Then if $\hat{x}\in \hat{M}^\circ$, one may take a path $\hat{\gamma}:[0,1]\to \hat{M}^\circ$
such that $\hat{\gamma}(0)=\hat{x}_0$, $\hat{\gamma}(1)=\hat{x}$
and lift it by the covering map $G$ to a curve $\hat{\Gamma}(t)$ in $N$
starting from $z_0$. Then if $\gamma(t):=F(\hat{\Gamma}(t))$, $t\in [0,1]$,
we easily see that $\hat{\gamma}=\hat{\gamma}_{\RDist}(\gamma,q_0)$,
whence $\hat{x}=\hat{\gamma}(1)\in \pi_{Q,\hat{M}}(\mc{O}_{\RDist}(q_0))$.

\end{remark}

We conclude this subsection with a necessary condition for complete controllability, which is an immediate consequence of Theorem \ref{th:non-sym}.
\begin{proposition}
Assume that the manifolds $M$ and $\hat M$ are complete non-symmetric, simply connected, irreducible and $n\neq 8$.
If the control system $\srol$ is completely controllable, then one of the holonomy groups of $M$ or $\hat{M}$ is equal to $\SO(n)$ (w.r.t some orthonormal frames).
\end{proposition}

\subsubsection{A ''Rolling Along Loops'' Characterization of Isometry}

In this paragraph, we provide a general non-controllability result
that will be used later on.
It will be the converse of the following simple proposition.

\begin{proposition}
Suppose $(M,g)$ and $(\hat{M},\hat{g})$
have a common Riemannian covering space $(N,h)$
with projections (local isometries) $F:N\to M$ and $G:N\to\hat{M}$.
Then if there exist $x_0\in M$, $\hat{x}_0\in\hat{M}$ such that
\[
F^{-1}(x_0)\subset G^{-1}(\hat{x}_0),
\]
then for $q_0=(x_0,\hat{x}_0;A_0)\in Q$ with $A_0=G_*\circ (F_*|_{q_0})^{-1}$
one has that for every loop $\gamma\in\Omega_{x_0}(M)$ based at $x_0$
the corresponding curve $\hat{\gamma}_{\RDist}(\gamma,q_0)$ on $\hat{M}$
determined by the rolling curve starting from $q_0$
(exists and) is a loop based $\hat{x}_0$
i.e.,
\[
\gamma\in\Omega_{x_0}(M)\quad\Longrightarrow\quad \hat{\gamma}_{\RDist}(\gamma,q_0)\in \Omega_{\hat{x}_0}(\hat{M}).
\]
\end{proposition}

\begin{proof}
If $\gamma\in\Omega_{x_0}(M)$,
let $\Gamma$ be the unique lift of $\gamma$ to $N$
such that $\Gamma(0)=q_0$
and define $\hat{\gamma}=F\circ \Gamma$, $A(t)=G_*\circ (F_*|_{\Gamma(t)})^{-1}$.
Then $q(t)=(\gamma(t),\hat{\gamma}(t);A(t))$ is an element of $Q|_{(\gamma(t),\hat{\gamma}(t))}$,
since $F,G$ are local isometries and moreover, $q_0=q(0)$,
\[
\dot{\hat{\gamma}}(t)=&\dif{t}(G\circ \Gamma)(t)
=G_*\dot{\Gamma}(t)=(G_*\circ (F_*|_{\Gamma(t)})^{-1})(F_*\dot{\Gamma}(t))
=A(t)\dot{\gamma}(t) \\
\hat{\nabla}_{\dot{\hat{\gamma}}(t)} (A(t)P_0^t(\gamma)X)
=&\hat{\nabla}_{G_*\dot{\Gamma}(t)} (G_*(F_*|_{\Gamma(t)})^{-1}P_0^t(\gamma)X) \\
=&G_*\nabla^h_{(F_*|_{\Gamma(t)})^{-1}\dot{\gamma}(t)}((F_*|_{\Gamma(t)})^{-1}P_0^t(\gamma)X) \\
=&(G_*\circ (F_*|_{\Gamma(t)})^{-1})\nabla_{\dot{\gamma}(t)} (P_0^t(\gamma)X)=0,
\]
for every $t\in [0,1]$
and every $X\in T|_{x_0} M$.
This proves that $q(t)=q_{\RDist}(\gamma,q_0)(t)$
and since $\gamma$ is a loop based at $x_0$,
$F(\Gamma(1))=\gamma(1)=x_0$,
which means that $\Gamma(1)\in F^{-1}(x_0)\subset G^{-1}(\hat{x}_0)$
and thus $\hat{\gamma}_{\RDist}(\gamma,q_0)(1)=\hat{\gamma}(1)=G(\Gamma(1))=\hat{x}_0$.
By definition, $\hat{\gamma}_{\RDist}(\gamma,q_0)(0)=\hat{x}_0$
and hence $\hat{\gamma}_{\RDist}(\gamma,q_0)\in\Omega_{\hat{x}_0}(\hat{M})$.
This completes the proof.
\end{proof}

Conversely, we have the following theorem which is the main result of this subsection.

\begin{theorem}\label{th:fixed_point}
Let $(M,g)$, $(\hat{M},\hat{g})$ be complete Riemannian manifolds
and suppose that there is a $q_0=(x_0,\hat{x}_0;A_0)\in Q$
such that for every loop $\gamma\in\Omega_{x_0}(M)$ based at $x_0$
the corresponding curve $\hat{\gamma}_{\RDist}(\gamma,q_0)$ on $\hat{M}$
determined by the rolling curve starting from $q_0$
is a loop based $\hat{x}_0$
i.e.,
\begin{align}\label{eq:loops_give_loops}
\gamma\in\Omega_{x_0}(M)\quad\Longrightarrow\quad \hat{\gamma}_{\RDist}(\gamma,q_0)\in \Omega_{\hat{x}_0}(\hat{M}).
\end{align}
Then $(M,g)$ and $(\hat{M},\hat{g})$ have a common Riemannian covering space
$(N,h)$ such that if $F:N\to M$, $G:N\to\hat{M}$
are the corresponding covering maps,
then
\[
F^{-1}(x_0)\subset G^{-1}(\hat{x}_0).
\]
\end{theorem}

\begin{proof}
For $u,v\in T|_{x_0} M$, a Jacobi field
along the geodesic $t\mapsto \exp_{x_0}(tu)=:\gamma_u(t)$ is given by
\[
Y_{u,v}(t)=\pa{s}\big|_{s=0} \exp_{x_0}(t(u+sv))=t(\exp_{x_0})_*|_{tu}(v),
\]
together with the initial conditions: $Y_{u,v}(0)=0$, $\nabla_{\dot{\gamma}_u(0)} Y_{u,v}=v$.
Define a function $\hat{\omega}_{u,v}:[0,1]\times [-1,1]\to\hat{M}$ by
\[
&\hat{\omega}_{u,v}(t,s):=\\
&\hat{\gamma}_{\RDist}\Big(\tau\mapsto \exp_{x_0}\big( (1-\tau)(u+sv)\big),q_{\RDist}\big(\sigma\mapsto \exp_{x_0}(u+\sigma v),q_{\RDist}(\gamma_u,q_0)(1)\big)(s)\Big)(t).
\]
It is clear from Proposition \ref{pr:rol_geodesic} that for every $s\in [-1,1]$
the map $t\mapsto \hat{\omega}_{u,v}(t,s)$ is a geodesic
and moreover it is clear that $\hat{\omega}(t,0)=\hat{\gamma}_{\RDist}(\gamma_u,q_0)(1-t)$.
This implies that 
\[
\hat{Y}_{u,v}(t):=\pa{s}\Big|_{s=0} \hat{\omega}_{u,v}(1-t,s),\quad t\in [0,1],
\]
defines a Jacobi field of $(\hat{M},\hat{g})$ along the geodesic $\hat{\gamma}_{\RDist}(\gamma_u,q_0)$.
We now derive some properties of this Jacobi field. 

We first observe that 
\begin{align}\label{fp:Jacobi_p1}
\hat{Y}_{u,v}(1)=&\pa{s}\big|_{s=0}  \hat{\gamma}_{\RDist}\big(\sigma\mapsto \exp_{x_0}(u+\sigma v),q_{\RDist}(\gamma_u,q_0)(1)\big)(s) \nonumber \\
=&A_{\RDist}(\gamma_u,q_0)(1)\pa{s}\big|_0 \exp_{x_0}(u+s v) \nonumber \\
=&A_{\RDist}(\gamma_u,q_0)(1)Y_{u,v}(1).
\end{align}

We now claim that $\hat{\omega}_{u,v}(1,s)=\hat{x}_0$ for all $s$.
Indeed, we may write $\hat{\omega}_{u,v}(1,s)$ as
\[
\hat{\omega}_{u,v}(1,s)
=\hat{\gamma}_{\RDist}\big(\underbrace{\big(\tau\mapsto \exp_{x_0}((1-\tau)(u+sv))\big).\big(\sigma\mapsto \exp_{x_0}(u+\sigma sv)\big).\gamma_u}_{=:(\star)\in\Omega_{x_0}(M)},q_0\big)(1)
=\hat{x}_0
\]
and since the expression $(\star)$ is a loop on $M$ based at $x_0$,
it follows from the assumption that the path defined on right of the first equality sign
is a loop on $\hat{M}$ based at $\hat{x}_0$, hence its value at $t=1$ is $\hat{x}_0$.
From this follows the second property of $\hat{Y}_{u,v}$, namely
\begin{align}\label{fp:Jacobi_p2}
\hat{Y}_{u,v}(0)=0,
\end{align}
since $\hat{Y}_{u,v}(0)=\pa{s}\big|_{s=0} \hat{\omega}_{u,v}(1,s)=\pa{s}\big|_0 (s\mapsto x_0)=0$.

This is a key property since it implies that $\hat{Y}_{u,v}$ has the form
\begin{align}\label{fp:Jacobi_explicit}
\hat{Y}_{u,v}(t)=\pa{s}\big|_0\widehat{\exp}_{\hat{x}_0}\big(t(A_0u+s\hat{v}(u,v))\big)
=t(\widehat{\exp}_{\hat{x}_0})_*|_{tA_0u}(\hat{v}(u,v)),
\end{align}
where $\hat{v}(u,v):=\hat{\nabla}_{\dif{t}\widehat{\exp}_{\hat{x}_0}(tA_0 u)} \hat{Y}_{u,v}\big|_{t=0}$.
It is clear that $(u,v)\mapsto \hat{v}(u,v)$ is a smooth map $(T|_{x_0} M)^2\to T|_{\hat{x}_0} \hat{M}$.
We also observed that $\hat{\gamma}_u(t):=\hat{\gamma}_{\RDist}(\gamma_u,q_0)=\widehat{\exp}_{\hat{x}_0}(tA_0 u)$.

We next show the following relation.
\begin{lemma}\label{le-loop-1}
With the above notations,
\begin{align}\label{fp:Jacobi_p3}
\hat{\nabla}_{\dot{\hat{\gamma}}_u(t)} \hat{Y}_{u,v}|_{t=1}=
A_{\RDist}(\gamma_u,q_0)(1)\nabla_{\dot{\gamma}_u(t)} Y_{u,v}|_{t=1}.
\end{align}
\end{lemma}
\begin{proof}
Writing
$\partial_t:=\pa{t}\exp_{x_0}(t(u+sv))$,
$\partial_s:=\pa{s}\exp_{x_0}(t(u+sv))$
$\hat{\partial}_t:=\pa{t}\hat{\omega}(1-t,s)$,
$\hat{\partial}_s:=\pa{s}\hat{\omega}(1-t,s)$
and $\ol{\partial}_s=(\partial_s,\hat{\partial}_s)$, we have
\[
& \hat{\nabla}_{\dot{\hat{\gamma}}_u(t)} \hat{Y}_{u,v}|_{t=1}
=\hat{\nabla}_{\hat{\partial}_t} \pa{s}\Big|_0 \hat{\omega}(1-t,s)\Big|_{t=1}
=\hat{\nabla}_{\hat{\partial}_s} \pa{t}\Big|_1 \hat{\omega}(1-t,s)\Big|_{s=0} \\
=&\hat{\nabla}_{\hat{\partial}_s} \pa{t}\Big|_1 \hat{\gamma}_{\RDist}
\Big(\tau\mapsto \exp_{x_0}((1-\tau)(u+sv)), \\
& \quad q_{\RDist}\big(\big(\sigma\mapsto \exp_{x_0}(u+\sigma sv)\big).\gamma_u,q_0\big)(1)\Big)(1-t)\Big|_{s=0} \\
=&\hat{\nabla}_{\hat{\partial}_s}\Big(A_{\RDist}\big(\big(\sigma\mapsto \exp_{x_0}(u+\sigma sv)\big).\gamma_u,q_0\big)(1)\pa{t}\big|_1 \exp_{x_0}(t(u+sv))\Big)\Big|_{s=0} \\
=&\hat{\nabla}_{\hat{\partial}_s}\Big(A_{\RDist}\big(\sigma\mapsto \exp_{x_0}(u+\sigma v),q_{\RDist}(\gamma_u,q_0)(1)\big)(s)\pa{t}\big|_1 \exp_{x_0}(t(u+sv))\Big)\Big|_{s=0} \\
=&\underbrace{\Big(\ol{\nabla}_{\ol{\partial}_s}A_{\RDist}\big(\sigma\mapsto \exp_{x_0}(u+\sigma v),q_{\RDist}(\gamma_u,q_0)(1)\big)(s)\Big)\Big|_{s=0}}_{=0}\pa{t}\big|_1 \exp_{x_0}(t(u+sv)) \\
&+A_{\RDist}\big(\sigma\mapsto \exp_{x_0}(u+\sigma v),q_{\RDist}(\gamma_u,q_0)(1)\big)(0)\Big(\nabla_{\partial_s}\pa{t}\big|_1 \exp_{x_0}(t(u+sv))\Big)\Big|_{s=0} \\
=&A_{\RDist}(\gamma_u,q_0)(1)\nabla_{\partial_t} \underbrace{\pa{s}\big|_0 \exp_{x_0}(t(u+sv))}_{=Y_{u,v}(t)}\big|_{t=1} \\
=&A_{\RDist}(\gamma_u,q_0)(1)\nabla_{\dot{\gamma}_u(t)} Y_{u,v}|_{t=1},
\]
which gives (\ref{fp:Jacobi_p3}).

\end{proof}

The next technical result goes as follows.
\begin{lemma}\label{le-loop-2}
Consider $\hat{v}(u,v)$ defined by (\ref{fp:Jacobi_explicit}). Then, 
\begin{align}\label{fp:hat_v}
\hat{v}(u,v)=A_0v,\quad \forall u,v\in T|_{x_0} M.
\end{align}
\end{lemma}

\begin{proof}
Notice first that for any $\tau\in\R$,
\[
Y_{\tau u,v}(t)=\pa{s}\big|_0 \exp_{x_0}(t(\tau u+sv))
=\frac{1}{\tau}\pa{\sigma}\big|_0 \exp_{x_0}(t\tau (u+\sigma v))
=\frac{1}{\tau} Y_{u,v}(t\tau),
\]
where, in the first equality, we substituted $\sigma:=\frac{s}{\tau}$.
Therefore (\ref{fp:Jacobi_p1}) implies that
\[
\hat{Y}_{\tau u,v}(1)=A_{\RDist}(\gamma_{\tau u},q_0)(1)Y_{\tau u,v}(1)
=\frac{1}{\tau}A_{\RDist}(\gamma_{u},q_0)(\tau)Y_{u,v}(\tau),
\]
i.e.,
\begin{align}\label{fp:Jacobi_scaling}
A_{\RDist}(\gamma_{u},q_0)(\tau)Y_{u,v}(\tau)=\tau \hat{Y}_{\tau u,v}(1).
\end{align}
On one hand, from (\ref{fp:Jacobi_p3}), (\ref{fp:Jacobi_scaling}) and (\ref{fp:Jacobi_explicit}) one has
(recall that $\hat{\gamma}_u(t)=\gamma_{\RDist}(\gamma_u,q_0)(t)=\widehat{\exp}_{\hat{x}_0}(tA_0u)$)
\[
\hat{\nabla}_{\dot{\hat{\gamma}}_u(t)} \hat{Y}_{u,v}\big|_{t=1}
=&A_{\RDist}(\gamma_u,q_0)(1)\nabla_{\dot{\gamma}_u(t)} Y_{u,v}|_{t=1}
=\hat{\nabla}_{\dot{\hat{\gamma}}_u(t)} \big(A_{\RDist}(\gamma_u,q_0)(t)Y_{u,v}(t)\big)|_{t=1} \\
=&\hat{\nabla}_{\dot{\hat{\gamma}}_u(t)} (t \hat{Y}_{t u,v}(1))\big|_{t=1}
=\hat{Y}_{u,v}(1)+\hat{\nabla}_{\dot{\hat{\gamma}}_u(t)} \hat{Y}_{t u,v}(1)\big|_{t=1} \\
=&\hat{Y}_{u,v}(1)+\hat{\nabla}_{\dot{\hat{\gamma}}_u(t)} \pa{s}\big|_0 \widehat{\exp}_{\hat{x}_0}\big(tA_0 u+s\hat{v}(tu,v))\big)\big|_{t=1} \\
=&\hat{Y}_{u,v}(1)+\hat{\nabla}_{\dot{\hat{\gamma}}_u(t)} \big((\widehat{\exp}_{\hat{x}_0})_*|_{tA_0 u}(\hat{v}(tu,v)\big)\big|_{t=1}.
\]
On the other hand, using only (\ref{fp:Jacobi_explicit}) one has
\[
\hat{\nabla}_{\dot{\hat{\gamma}}_u(t)} \hat{Y}_{u,v}\big|_{t=1}
=&\hat{\nabla}_{\dot{\hat{\gamma}}_u(t)} \pa{s}\big|_0 \widehat{\exp}_{\hat{x}_0}\big(t(A_0 u+s\hat{v}(u,v))\big)\big|_{t=1} \\
=&\hat{\nabla}_{\dot{\hat{\gamma}}_u(t)} \big(t(\widehat{\exp}_{\hat{x}_0})_*|_{tA_0 u}(\hat{v}(u,v))\big)\big|_{t=1} \\
=&(\widehat{\exp}_{\hat{x}_0})_*|_{A_0 u}(\hat{v}(u,v))+\hat{\nabla}_{\dot{\hat{\gamma}}_u(t)} \big((\widehat{\exp}_{\hat{x}_0})_*|_{tA_0 u}(\hat{v}(u,v))\big)\big|_{t=1} \\
=&\pa{s}\big|_0 \widehat{\exp}_{\hat{x}_0}(A_0u+s\hat{v}(u,v))+\hat{\nabla}_{\dot{\hat{\gamma}}_u(t)} \big((\widehat{\exp}_{\hat{x}_0})_*|_{tA_0 u}(\hat{v}(u,v))\big)\big|_{t=1}  \\
=&\hat{Y}_{u,v}(1)+\hat{\nabla}_{\dot{\hat{\gamma}}_u(t)} \big((\widehat{\exp}_{\hat{x}_0})_*|_{tA_0 u}(\hat{v}(u,v))\big)\big|_{t=1}.
\]
Combining these two formulas, whose left hand sides are equal, and canceling the common
terms $\hat{Y}_{u,v}(1)$ we end up with
\[
\hat{\nabla}_{\dot{\hat{\gamma}}_u(t)} \big((\widehat{\exp}_{\hat{x}_0})_*|_{tA_0 u}(\hat{v}(tu,v))\big)\big|_{t=1}
=\hat{\nabla}_{\dot{\hat{\gamma}}_u(t)} \big((\widehat{\exp}_{\hat{x}_0})_*|_{tA_0 u}(\hat{v}(u,v))\big)\big|_{t=1}.
\]
Here we can simplify the left hand side by the following computation:
With the notation 
\[
D_s:=\pa{s} \widehat{\exp}_{\hat{x}_0}\big(tA_0 u+s\hat{v}(tu,v))\big),
\ \ \ D_t:=\pa{t} \widehat{\exp}_{\hat{x}_0}\big(tA_0 u+s\hat{v}(tu,v))\big),
\]
(notice also that $D_t|_{s=0}=\dot{\hat{\gamma}}_u(t)$) we get
\[
& \hat{\nabla}_{\dot{\hat{\gamma}}_u(t)} \big((\widehat{\exp}_{\hat{x}_0})_*|_{tA_0 u}(\hat{v}(tu,v))\big)\big|_{t=1}
=\hat{\nabla}_{D_t} \pa{s}\big|_0 \widehat{\exp}_{\hat{x}_0}\big(tA_0 u+s\hat{v}(tu,v))\big)\big|_{t=1} \\
=&\hat{\nabla}_{D_s} \pa{t}\big|_1 \widehat{\exp}_{\hat{x}_0}\big(tA_0 u+s\hat{v}(tu,v))\big)\big|_{s=0}\\
=&\hat{\nabla}_{D_s} (\widehat{\exp}_{\hat{x}_0})_*|_{A_0u+s\hat{v}(u,v)}\big(A_0 u+s\partial_1\hat{v}(u,v)(u))\big)\big|_{s=0} \\
=&\hat{\nabla}_{D_s} (\widehat{\exp}_{\hat{x}_0})_*|_{A_0u+s\hat{v}(u,v)}\big(A_0 u\big)|_{s=0}+\hat{\nabla}_{D_s} (\widehat{\exp}_{\hat{x}_0})_*|_{A_0u+s\hat{v}(u,v)}\big(s\partial_1\hat{v}(u,v)(u)\big)\big|_{s=0} \\
=&\hat{\nabla}_{D_s}\pa{t}\big|_1 \widehat{\exp}_{\hat{x}_0}(tA_0u+s\hat{v}(u,v))\big|_{s=0}\\
&-\hat{\nabla}_{D_s} \pa{t}\big|_1 \widehat{\exp}_{\hat{x}_0}\big(A_0u+s\hat{v}(u,v)+(1-t)s\partial_1\hat{v}(u,v)(u)\big)\big|_{s=0} \\
=&\hat{\nabla}_{D_t} \pa{s}\big|_0 \widehat{\exp}_{\hat{x}_0}\big(tA_0 u+s\hat{v}(u,v))\big)\big|_{t=1}\\
&-\hat{\nabla}_{D_t} \pa{s}\big|_0 \widehat{\exp}_{\hat{x}_0}\big(A_0u+s\hat{v}(u,v)+(1-t)s\partial_1\hat{v}(u,v)(u)\big)\big|_{t=1} \\
=&\hat{\nabla}_{\dot{\hat{\gamma}}_u(t)} \big((\widehat{\exp}_{\hat{x}_0})_*|_{tA_0 u}(\hat{v}(u,v))\big)\big|_{t=1}\\
&-\hat{\nabla}_{D_t} (\widehat{\exp}_{\hat{x}_0})_*|_{A_0 u}\big(\hat{v}(u,v)+(1-t)\partial_1\hat{v}(u,v)(u)\big)\big|_{t=1},
\]
where $\partial_1 \hat{v}(u,v)(w)$ for $u,v,w\in T|_{x_0} M$ denotes the directional derivative 
of $\hat{v}$ at $(u,v)$ in the direction $w$.
The last term on the right of the previous formula simplifies to
\[
& \hat{\nabla}_{D_t} (\widehat{\exp}_{\hat{x}_0})_*|_{A_0 u}\big(\hat{v}(u,v)+(1-t)\partial_1\hat{v}(u,v)(u)\big)\big|_{t=1} \\
=&\hat{\nabla}_{D_t} \Big((\widehat{\exp}_{\hat{x}_0})_*|_{A_0 u}\big(\hat{v}(u,v)\big)+(1-t)(\widehat{\exp}_{\hat{x}_0})_*|_{A_0 u}\big(\partial_1\hat{v}(u,v)(u)\big)\Big)\Big|_{t=1} \\
=&-(\widehat{\exp}_{\hat{x}_0})_*|_{A_0 u}\big(\partial_1\hat{v}(u,v)(u)\big).
\]
Combining the last three formulas, one obtains
\[
(\widehat{\exp}_{\hat{x}_0})_*|_{A_0 u}\big(\partial_1\hat{v}(u,v)(u)\big)=0.
\]
Thus for all $u$ such that $A_0 u$ is not in the tangent conjugate locus $\hat{Q}_{\hat{x}_0}$ of $\widehat{\exp}_{\hat{x}_0}$
one has $\partial_1\hat{v}(u,v)(u)=0$. Moreover, since the complement of $\hat{Q}_{\hat{x}_0}$  
is a dense subset of $T|_{\hat{x}_0} \hat{M}$,
the continuity of $(u,v)\mapsto \partial_1\hat{v}(u,v)(u)$ implies that
\[
\partial_1\hat{v}(u,v)(u)=0,\quad \forall u,v\in T|_{\hat{x}_0} M.
\] 
But this implies that
\[
\hat{v}(u,v)-\hat{v}(0,v)=\int_0^1 \dif{t}\hat{v}(tu,v)\diff t
=\int_0^1 \frac{1}{t}\underbrace{\partial_1\hat{v}(tu,v)(tu)}_{=0}\diff t=0,
\]
and hence we need only to know the values of $\hat{v}(0,v)$ to know all values of $\hat{v}(u,v)$.
By the definition of $\hat{\omega}_{u,v}(t)$ one sees that
\[
\hat{\omega}_{0,v}(t,s)=&\hat{\gamma}_{\RDist}\Big(\tau\mapsto \exp_{x_0}\big( (1-\tau)sv\big),q_{\RDist}\big(\sigma\mapsto \exp_{x_0}(\sigma v),\underbrace{q_{\RDist}(\gamma_0,q_0)(1)}_{=q_0}\big)(s)\Big)(t) \\
=&\hat{\gamma}_{\RDist}\Big(\tau\mapsto \exp_{x_0}\big( (1-\tau)sv\big),q_{\RDist}\big(\sigma\mapsto \exp_{x_0}(\sigma sv),q_0\big)(1)\Big)(t) \\
=&\hat{\gamma}_{\RDist}\big(\tau\mapsto \exp_{x_0}\big( \tau sv\big),q_0\big)(1-t)
=\widehat{\exp}_{\hat{x}_0}((1-t)sA_0 v),
\]
which implies that
\[
\hat{Y}_{0,v}(t)=\pa{s}\big|_0 \hat{\omega}_{0,v}(1-t,s)=\pa{s}\big|_0 \widehat{\exp}_{\hat{x}_0}\big(t(0+sA_0 v)\big),
\]
and therefore, comparing to (\ref{fp:Jacobi_explicit}), one obtains
\[
\hat{v}(0,v)=A_0 v.
\]
This finally proves (\ref{fp:hat_v}) since by the above considerations,
$\hat{v}(u,v)=\hat{v}(0,v)=A_0 v$.

\end{proof}

Equations (\ref{fp:Jacobi_p1}), (\ref{fp:Jacobi_explicit}) and (\ref{fp:hat_v})
show that (take $t=1$)
\begin{align}\label{eq:loop_isometry}
(\widehat{\exp}_{\hat{x}_0})_*|_{A_0u}(A_0v)=A_{\RDist}(\gamma_u,q_0)(1)\big((\exp_{x_0})_*|_{u}(v)\big),
\quad \forall u,v\in T|_{x_0} M.
\end{align}

We now show that (\ref{eq:loops_give_loops}) holds
with $q_0$ replaced by any element of the fiber $\mc{O}_{\RDist}(q_0)\cap \pi_{Q,M}^{-1}(x_0)$ of the orbit above $x_0$.

\begin{lemma}
Write $F_{x_0}:=\mc{O}_{\RDist}(q_0)\cap \pi_{Q,M}^{-1}(x_0)$.
Then
\begin{align}\label{eq:loops_give_loops_2}
q\in F_{x_0},\ \gamma\in\Omega_{x_0}(M)\quad\Longrightarrow\quad \hat{\gamma}_{\RDist}(\gamma,q)\in \Omega_{\hat{x}_0}(\hat{M}).
\end{align}

\end{lemma}

\noindent Remark: $\pi_{Q}(F_{x_0})=(x_0,\hat{x}_0)$.

\begin{proof}
Let $q\in F_{x_0}$. Then there is a $\omega\in\Omega_{x_0}(M)$ such that $q=q_{\RDist}(\omega,q_0)(1)$.
Then if $\gamma\in\Omega_{x_0}(M)$,
\[
\hat{\gamma}_{\RDist}(\gamma,q)(1)=\hat{\gamma}_{\RDist}(\gamma,q_{\RDist}(\omega,q_0)(1))(1)
=\hat{\gamma}_{\RDist}(\gamma.\omega,q_0)(1)=\hat{x}_0,
\]
where the last equality follows from (\ref{eq:loops_give_loops})
since $\gamma.\omega\in\Omega_{x_0}(M)$.
Since $\hat{\gamma}_{\RDist}(\gamma,q)(0)=\pi_{Q,\hat{M}}(q)=\hat{x}_0$
as remarked just before the proof,
we have $\hat{\gamma}_{\RDist}(\gamma,q)\in \Omega_{\hat{x}_0}(\hat{M})$.

\end{proof}

Define $U$ to be the subset of $T|_{x_0} M$ of points before the cut time i.e.,
if for $X\in T|_{x_0} M$, $\n{X}_g=1$ we let $\tau(X)\in ]0,\infty]$ denote the time such that the geodesic $\gamma_X$
is optimal on $[0,\tau(X)]$ but not after,
then
\[
U=\{sX\ |\ X\in T|_{x_0} M,\ \n{X}_g=1,\ 0\leq s<\tau(X)\}.
\]
Since $(M,g)$ is complete, $\tilde{U}:=\exp_{x_0}(U)$ is dense in $M$
and $\exp_{x_0}:U\to\tilde{U}$ is a diffeomorphism.

We now have the following result.

\begin{lemma}
For each $q\in F_{x_0}$ let
\[
\phi_q:\tilde{U}\to \hat{M};\quad \phi_q=\widehat{\exp}_{\hat{x}_0}\circ A\circ (\exp_{x_0}|_U)^{-1}
\]
where $q=(x_0,\hat{x}_0;A)$.
Then each mapping $\phi_q$ is a local isometry $(\tilde{U},g|_{\tilde{U}})\to (\hat{M},\hat{g})$
and $(\phi_q)_*|_{T|_{x_0} M}=A$.
\end{lemma}

\begin{proof}
Since $q=(x_0,\hat{x}_0;A)\in F_{x_0}$, the previous lemma implies that (\ref{eq:loop_isometry})
holds with $q_0=(x_0,\hat{x}_0;A_0)$ replaced by $q$.
Therefore, if $x\in \tilde{U}$ and $X\in T|_x M$
write $u=(\exp_{x_0}|_U)^{-1}(x)\in T|_{x_0} M$ and $v=((\exp_{x_0}|_U)^{-1})_*(X)\in T|_u (T|_{x_0} M)=T|_{x_0} M$
and (\ref{eq:loop_isometry}) with $q_0$ replaced by $q$ implies
\[
\n{(\phi_q)_*(X)}_{\hat{g}}=&\n{\big((\widehat{\exp}_{\hat{x}_0})_*\circ A\circ ((\exp_{x_0}|_U)^{-1})_*\big)(X)}_{\hat{g}} \\
=&\n{(\widehat{\exp}_{\hat{x}_0})_*|_{Au}(Av)}_{\hat{g}}
=\n{A_{\RDist}(\gamma_u,q)(1)(\exp_{x_0})_*|_{u}(v)}_{\hat{g}} \\
=&\n{(\exp_{x_0})_*|_{u}(v)}_{g}=\n{(\exp_{x_0})_*|_{u}\big(((\exp_{x_0}|_U)^{-1})_*(X)\big)}_{g} \\
=&\n{X}_g,
\]
where the 4. equality follows from the fact that $A\in T|_{x_0} M\to T|_{\hat{x}_0}\hat{M}$ is an isometry.
The claim $(\phi_q)_*|_{T|_{x_0} M}=A$ is obviously true.
\end{proof}

We will now start proving that $\Rol(\cdot,\cdot)(A_{\RDist}(\gamma,q_0)(t))=0$
for every piecewise $C^1$-path (not necessarily a loop) $\gamma$ on $M$ such that $\gamma(0)=x_0$ and for all $t$.
First we prove a special case of this (but with $q_0$ replaced by any $q\in F_{x_0}$).

\begin{lemma}
Let $q\in F_{x_0}$, $u\in T|_{x_0} M$ be a unit vector and let $\gamma_u$ be the geodesic $t\mapsto \exp_{x_0}(tu)$,
Then
\[
\Rol(\cdot,\cdot)(A_{\RDist}(\gamma_u,q)(t))=0,\quad\forall t\in [0,\tau(u)].
\]
\end{lemma}

\begin{proof}
Write $q=(x_0,\hat{x}_0;A)$ and notice that by definition of $U$
we have $tu\in U$ for all $t\in [0,\tau(u)[$.

Since $\phi_q$ of the previous lemma is a local isometry, it follows that
\[
P_0^t(\phi_q(\gamma_u))\circ A=(\phi_q)_*\circ P_0^t(\gamma_u),\quad \forall 0\leq t<\tau(u).
\]
Also, $\dif{t}\phi_q(\gamma_u)(t)=(\phi_q)_*\dot{\gamma}_u(t)$ for all $t$
so we may conclude that
\[
q_{\RDist}(\gamma_u,q)(t)=\big(\gamma_u(t),(\phi_q\circ\gamma_u)(t);(\phi_q)_*|_{\gamma_u(t)}\big),
\quad\forall 0\leq t<\tau(u).
\]
Again, since $\phi_q$ is a local isometry, for all $X,Y,Z\in T|_x M$, $x\in\tilde{U}$
we have $(\phi_q)_*(R(X,Y)Z)=\hat{R}((\phi_q)_*X,(\phi_q)_*Y)((\phi_q)_*Z)$ i.e., $\Rol(\cdot,\cdot)((\phi_q)_*|_x)=0$
for all $x\in\tilde{U}$. But then
\[
\Rol(\cdot,\cdot)(A_{\RDist}(\gamma_u,q)(t))=\Rol(\cdot,\cdot)\big((\phi_q)_*|_{\gamma_u(t)}\big)=0,
\quad 0\leq t<\tau(u).
\]
Continuity of $\Rol$ and $q_{\RDist}(\gamma_u,q)$ now allows us to conclude
that the above equation holds for all $0\leq t\leq\tau(u)$.
\end{proof}

Now we may prove the claim that was asserted before the previous lemma.

\begin{lemma}
Let $\gamma:[0,1]\to M$ a piecewise $C^1$-path on $M$ such that
$\gamma(0)=x_0$. Then
\[
\Rol(\cdot,\cdot)(A_{\RDist}(\gamma,q_0)(t))=0,\quad\forall t\in [0,1].
\]
\end{lemma}

\begin{proof}
It is clearly enough to prove the claim in the case $t=1$.
Choose any vector $u\in T|_{x_0} M$
such that $\gamma_u:[0,1]\to M$ is the minimal geodesic from $x_0$ to $\gamma(1)$.
Define $q:=q_{\RDist}(\gamma_u^{-1}.\gamma,q_0)(1)$
and notice that since $\gamma_u^{-1}.\gamma\in\Omega_{x_0}(M)$,
we have $q\in F_{x_0}$.
Thus by the previous lemma, 
\[
\Rol(\cdot,\cdot)(A_{\RDist}(\gamma_u,q)(1))
=\Rol(\cdot,\cdot)(A_{\RDist}(\gamma_{\frac{u}{\n{u}_g}},q)(\n{u}_g))=0,
\]
since $\tau(\frac{u}{\n{u}_g})=\n{u}_g$.
But
\[
q_{\RDist}(\gamma_u,q)(1)=q_{\RDist}(\gamma_u,q_{\RDist}(\gamma_u^{-1}.\gamma,q_0)(1))(1)
=q_{\RDist}(\gamma_u.\gamma_u^{-1}.\gamma,q_0)(1)=q_{\RDist}(\gamma,q_0)(1)
\]
and hence
\[
0=\Rol(\cdot,\cdot)(A_{\RDist}(\gamma_u,q)(1))=\Rol(\cdot,\cdot)(A_{\RDist}(\gamma,q_0)(1))
\]
which concludes the proof.
\end{proof}

Finally we may proceed to the proof of the theorem itself.
Indeed, since $\mc{O}_{\RDist}(q_0)$ is the set of all $q_{\RDist}(\gamma,q_0)(1)$
with all the possible piecewise $C^1$-curves $\gamma:[0,1]\to M$ such that $\gamma(0)=x_0$,
the previous lemma implies that the condition (ii)
of Corollary \ref{cor:weak_ambrose} is satisfied.
Thus there is a Riemannian manifold $(N,h)$ and Riemannian covering maps $F:N\to M$, $G:N\to\hat{M}$
i.e., $(M,g)$ and $(\hat{M},\hat{g})$ have a common Riemannian covering space.

Actually, by Corollary \ref{cor:weak_ambrose}, we may take $N=\mc{O}_{\RDist}(q_0)$,
$F=\pi_{Q,M}|_{N}$, $G=\pi_{Q,\hat{M}}|_N$
and hence if $q\in F^{-1}(x_0)$, then there exists
a $\gamma\in\Omega_{x_0}(M)$ such that $q=q_{\RDist}(\gamma,q_0)(1)$
and hence $G(q)=\hat{\gamma}_{\RDist}(\gamma,q_0)(1)=\hat{x}_0$
since $\hat{\gamma}_{\RDist}(\gamma,q_0)\in\Omega_{\hat{x}_0}(\hat{M})$ by the assumption.
This shows that $F^{-1}(x_0)\subset G^{-1}(\hat{x}_0)$
and concludes the proof.

\end{proof}

\begin{remark}
The difficulty in the proof of the previous theorem is due to the fact
that the contact points $x_0$, $\hat{x}_0$
are fixed i.e., we only assume that loops
that are based at $x_0$ generate, by rolling, loops that are based at $\hat{x}_0$.

If we were allowed to have an open neighbourhood of points on $M$
with the property that loops based at these points generate loops on $\hat{M}$,
one could prove that $(M,g)$ and $(\hat{M},\hat{g})$
have the same universal Riemannian covering
by an easier argument than above.

More precisely, suppose there is a $q_0=(x_0,\hat{x}_0;A_0)\in Q$ and a neighbourhood $U$ of $x_0$
which consists of points $x$ such that
whenever one rolls along a geodesic from $x_0$ to $x$
followed by any loop at $x$, then
the corresponding curve on $\hat{M}$, generated by rolling, is a geodesic followed by a loop based
at the end point of this geodesic.

This means that there is a (possibly smaller) normal neighbourhood $U$ of $x_0$ such that defining
a local $\pi_{Q,M}$-section $\tilde{q}$ on $U$ by
\[
\tilde{q}(x)=(x,\hat{f}(x);\tilde{A}|_x):=q_{\RDist}\big((t\mapsto \exp_{x_0}(t\exp_{x_0}^{-1}(x))),q_0\big)(1),
\]
then it holds that
\[
\forall x\in U,\ \gamma\in\Omega_x(M)\quad\Longrightarrow\quad \hat{\gamma}_{\RDist}(\gamma,\tilde{q}(x))\in\Omega_{\hat{f}(x)}(\hat{M}).
\]
Notice that $q_0=\tilde{q}(x_0)$.
(In the case of the previous theorem, we had $U=\{x_0\}$, which is not open.)

We will now sketch an easy argument to reach the conclusion of the theorem
under this stronger assumption.

Write $\pi_{\mc{O}_{\RDist}(q_0)}=\pi_{Q,M}|_{\mc{O}_{\RDist}(q_0)}$ as usual.
We show that the vertical bundle $V(\pi_{\mc{O}_{\RDist}(q_0)})$ is actually trivial
in the sense that all its fibers consist of one point only (the origin).
From this one concludes
that $\mc{O}_{\RDist}(q_0)$ is an integral manifold of $\RDist$ an hence
$\pi_{\mc{O}_{\RDist}(q_0)}$ is (complete and) a Riemannian covering map once
the manifold $\mc{O}_{\RDist}(q_0)$ is equipped with the Riemannian metric pulled
back from that of $M$ (or $\hat{M}$). 

Take $x\in U$ and $v\in V|_{\tilde{q}(x)}(\pi_{\mc{O}_{\RDist}(q_0)})$.
This means that there is a smooth curve $s\mapsto \Gamma(s)$, $s\in [0,1]$,
in $\mc{O}_{\RDist}(q_0)$ such that $\pi_{Q,M}(\Gamma(s))=x$ for all $s$
and $\dot{\Gamma}(0)=v$.

One may then choose for each $s$ a smooth path $\gamma_s$ in $M$ starting at $x$ and defined on $[0,1]$
such that $q_{\RDist}(\gamma_s,\tilde{q}(x))(1)=\Gamma(s)$.
We have $\gamma_s\in\Omega_x(M)$ since
\[
\gamma_s(1)=\pi_{Q,M}\big(q_{\RDist}(\gamma_s,\tilde{q}(x))(1)\big)=\pi_{Q,M}(\Gamma(s))=x.
\]
Thus by assumption,
\[
\hat{\gamma}_{\RDist}(\gamma_s,\tilde{q}(x))\in \Omega_{\tilde{f}(x)}(\hat{M})
\]
from which
\[
(\pi_{Q,\hat{M}})_*v=\dif{s}\big|_0 \pi_{Q,\hat{M}}(\Gamma(s))
=\dif{s}\big|_0 \hat{\gamma}_{\RDist}(\gamma_s,\tilde{q}(x))(1)
=\dif{s}\big|_0 (s\mapsto\tilde{f}(x))=0.
\]

This proves that every element of $V|_{\tilde{q}(x)}(\pi_{\mc{O}_{\RDist}(q_0)})$, $x\in U$,
is of the form $\nu(B)|_{\tilde{q}(x)}$ where $B\in T^*|_xM\otimes T|_{\tilde{f}(x)}\hat{M}$
and $\hat{g}(BX,\tilde{A}|_xY)+\hat{g}(\tilde{A}|_x X,BY)=0$, $\forall X,Y\in T|_x M$.

Take a vector field of the form $q\mapsto \nu(B|_q)|_q$ on $\mc{O}_{\RDist}(q_0)$ defined along the image of $\tilde{q}$.
Arguing as in Proposition \ref{pr:H_directions_from_V} and using Eq. (\ref{eq:special_NS_comm_VH}), we conclude that for every $X\in T|_{x_0} M$
we have
\[
-\LNSD(B|_{q_0}X)|_{q_0}+\nu\big(\ol{\nabla}_{(X, A_0 X)} B|_{\tilde{q}(\cdot)}\big)\in T|_{q_0}\mc{O}_{\RDist}(q_0)
\]
and hence, by what we just proved above,
the image of this vector under $(\pi_{Q,\hat{M}})_*$ must be zero
i.e., $B|_{q_0}X=0$.
Since this holds for all $X\in T|_{x_0} M$, it means that $B|_{q_0}=0$
and hence we have that $V|_{q_0}(\pi_{\mc{O}_{\RDist}(q_0)})=\{0\}$.
Thus the vertical bundle $V(\pi_{\mc{O}_{\RDist}(q_0)})$
has rank $=0$ since its fiber is $=\{0\}$ at one point.

\end{remark}

\begin{remark}
The assumption given by Formula (\ref{eq:loops_give_loops})
is a special case of a more general one: 
There is $q_0=(x_0,\hat{x}_0;A_0)\in Q$
and points $x_1\in M$, $\hat{x}_1\in\hat{M}$
such that
\begin{align}\label{eq:two_fixed_points}
\gamma\in\Omega_{x_0,x_1}(M)\quad\Longrightarrow\quad
\hat{\gamma}_{\RDist}(\gamma,q_0)\in \Omega_{\hat{x}_0,\hat{x}_1}(\hat{M}),
\end{align}
where $\Omega_{x_0,x_1}(M)$ is used to denote the set of piecewise $C^1$-curves
from $x_0$ to $x_1$ in $M$ with $\Omega_{\hat{x}_0,\hat{x}_1}(\hat{M})$
defined similarly for $\hat{M}$.

We actually reduce this setting to the one given in Theorem \ref{th:fixed_point} as follows.
Fix once and for all a curve $\omega:[0,1]\to M$
s.t. $\omega(0)=x_0$, $\omega(1)=x_1$
and write $q_1=q_{\RDist}(\omega,q_0)(1)$.
Then $q_1=(x_1,\hat{x}_1;A_1)$ by assumption given by Eq. (\ref{eq:two_fixed_points}),
with $A_1:T|_{x_1} M\to T|_{\hat{x}_1} \hat{M}$.
Then if $\gamma\in \Omega_{x_1}(M)$ is any loop in $M$  based at $x_1$,
one gets that  $\gamma.\omega\in \Omega_{x_0,x_1}(M)$ is a path from $x_0$ and $x_1$.
By assumption in Eq. (\ref{eq:two_fixed_points}) again, one has 
\[
\hat{\gamma}_{\RDist}(\gamma,q_1)(1)
=\hat{\gamma}_{\RDist}(\gamma.\omega,q_0)(1)
=\hat{x}_1,
\]
and since $\hat{\gamma}_{\RDist}(\gamma,q_1)(0)=\hat{x}_1$,
we have obtained
\[
\gamma\in \Omega_{x_1}(M)\quad\Longrightarrow\quad
\hat{\gamma}_{\RDist}(\gamma,q_1)\in \Omega_{\hat{x}_1}(\hat{M}).
\]
Therefore any result obtained under  Assumption (\ref{eq:loops_give_loops}) will also hold true under
the assumption given by Formula (\ref{eq:two_fixed_points}).
\end{remark}

%%%%%%%%%%%%%%%%%%%%%%%%%%%%%%%%%%%%%%%%%%%%
\subsubsection{The Ambrose's Theorem Revisited}\label{sec:3}\label{app:ambrose}
%%%%%%%%%%%%%%%%%%%%%%%%%%%%%%%%%%%%%%%%%%%%

The results developed so far allow us to somewhat simplify the proof of the Ambrose's theorem (see \cite{sakai91} Theorem III.5.1).
In fact, the elaborate construction of the covering space $X$ (of the manifold $M$)
is no longer needed since we build this space by simple integrating the distribution $\RDist$.
Actually, as in \cite{sakai91}, we will first prove (a version of) the Cartan's theorem (\cite{sakai91} Theorem II.3.2) 
by using the rolling framework and then use that result and some ``patching'' to obtain the Ambrose's theorem.
The considerations are in parallel to those found in \cite{blumenthal89}, \cite{pawel02}.

\begin{definition}\label{def:once_broken_geodesic}
A continuous curve $\gamma:[0,a]\to M$ on a Riemannian manifold $(M,g)$ is called \emph{once broken geodesic}, broken at $t_0$,
if there is a $t_0\in [0,a]$ such that $\gamma|_{[0,t_0]}$, $\gamma|_{[t_0,a]}$ are geodesics of $(M,g)$.
\end{definition}

Notice that if $q=(x,\hat{x};A)\in Q$ and $\gamma$ is a once broken geodesic on $M$
starting at $x$ broken at $t_0$, then $\hat{\gamma}_{\RDist}(\gamma,q)$ is a once broken
geodesic on $\hat{M}$ broken at $t_0$.

Ambrose's theorem can now be stated as follows.

\begin{theorem}[(Ambrose)]\label{th:ambrose}
Let $(M,g)$, $(\hat{M},\hat{g})$ be \emph{complete} $n$-dimensional Riemannian manifolds
and let $q_0=(x_0,\hat{x}_0;A_0)\in Q$.
Suppose that $M$ is simply connected and that, for any once broken geodesic $\gamma:[0,a]\to M$ starting from $x_0$,
we have
\begin{align}\label{eq:ambrose:1}
A_{\RDist}{(\gamma,q_0)}(t)(R(X,Y)Z)=\hat{R}(A_{\RDist}{(\gamma,q_0)}(t) X,A_{\RDist}{(\gamma,q_0)}(t) Y)(A_{\RDist}{(\gamma,q_0)}(t) Z),\quad 
\end{align}
for all $X,Y,Z\in T|_{\gamma(t)} M$ and $t\in [0,a]$.
Then, if for any minimal geodesic $\gamma:[0,a]\to M$ starting from $x_0$,
one defines $\Phi(\gamma(t))=\widehat{\exp}_{\hat{x}_0}(tA_0\dot{\gamma}(0))$, $t\in [0,a]$,
it follows that the map $\Phi:M\to\hat{M}$ is a well-defined Riemannian covering.
\end{theorem}

\begin{remark}
The assumption of Ambrose's theorem is equivalent to the following:
For any once broken geodesic $\gamma:[0,a]\to M$ starting from $x_0$ and
for all $X,Y\in \VF(M)$, $t\in [0,a]$,
\[
\Rol(X,Y)(A_{\RDist}{(\gamma,q_0)}(t))=0,
\]
which by Proposition \ref{pr:R_comm_L2} is equivalent to
\[
[\LRD(X),\LRD(Y)]|_{q_{\RDist}{(\gamma,q_0)}(t)}=\LRD([X,Y])|_{q_{\RDist}{(\gamma,q_0)}(t)}.
\]
i.e., that the distribution $\RDist$ is involutive at each point of $Q$ of the form $q_{\RDist}{(\gamma,q_0)}(t)$.
This should suggest that it is worthwhile to study the integrability of $\RDist$ near the point $q_0\in Q$,
although we are not allowed to use Frobenius theorem.
\end{remark}

On a Riemannian manifold $(N,h)$, we use $d_h$ to denote the distance function (metric) on $N$ induced by $h$ and, for $y\in N$, $X\in T|_y N$, $r>0$, we use  $B_{d_h}(y,r)\subset N$ (resp. $B_h(X,r)\subset T|_y N$)
 to denote the open ball of radius $r$ on $N$ (resp. $T|_y N$) centered at $y$ (resp. $X$)
w.r.t $d_h$ (resp. $h$).

The next result provides a local integral manifold of $\RDist$
under milder assumptions than those given in the statement of Ambrose's theorem.

\begin{theorem}[(Cartan)] \label{th:cartan}
Let $(M,g)$ and $(\hat{M},\hat{g})$ be (not necessarily complete) Riemannian manifolds.
Consider $q=(x,\hat{x};A)\in Q$ and $\epsilon>0$ such that
the exponential maps $\exp_x:B_g(0_x,\epsilon)\subset T|_x M\to B_{d_g}(x,\epsilon)$
and $\widehat{\exp}_{\hat{x}}:B(0_{\hat{x}},\epsilon)\subset T|_{\hat{x}} \hat{M}\to B_{d_{\hat{g}}}(\hat{x},\epsilon)$ are (defined and) diffeomorphisms.
Then the following are equivalent:
\begin{itemize}
\item[(i)] For every (non-broken) \emph{geodesic} $\gamma:[0,1]\to B_{d_g}(x,\epsilon)$ starting from $x$, we have
\begin{align}\label{eq:cartan:1}
&A_{\RDist}{(\gamma,q)}(t)(R(X,Y)Z)\nonumber\\
=&\hat{R}(A_{\RDist}{(\gamma,q)}(t) X,A_{\RDist}{(\gamma,q)}(t) Y)(A_{\RDist}{(\gamma,q)}(t) Z)
\end{align}
i.e., $\Rol(X,Y)(A_{\RDist}{(\gamma,q)}(t))Z=0$
for every $X,Y,Z\in T|_{\gamma(t)} M$ and $t\in [0,1]$.

\item[(ii)] For every (non-broken) \emph{geodesic} $\gamma:[0,1]\to B_{d_g}(x,\epsilon)$ starting from $x$, we have
\begin{align}\label{eq:cartan:2}
&A_{\RDist}{(\gamma,q)}(t)(R(X,\dot{\gamma}(t))\dot{\gamma}(t))\nonumber\\
=&\hat{R}(A_{\RDist}{(\gamma,q)}(t) X,\dot{\hat{\gamma}}_{\RDist}(\gamma,q)(t))\dot{\hat{\gamma}}_{\RDist}(\gamma,q)(t)
\end{align}
i.e., $\Rol(X,\dot{\gamma}(t))(A_{\RDist}{(\gamma,q)}(t))\dot{\gamma}(t)=0$
for every $X\in T|_{\gamma(t)} M$ and $t\in [0,1]$
(except the break point of $\gamma$).

\item[(iii)] There is a connected integral manifold $N$ of $\RDist$
passing through $q$ such that $\pi_{Q,M}|_N\to B_{d_g}(x,\epsilon)$
(or $\pi_{Q,\hat{M}}|_N\to B_{d_{\hat{g}}}(\hat{x},\epsilon)$)
is a bijection.

\item[(iv)] The map $\Phi:=\widehat{\exp}_{\hat{x}}\circ A\circ \exp_x^{-1}|_{B_{d_g}(x,\epsilon)}$ is an isometric diffeomorphism (onto $B_{d_{\hat{g}}}(\hat{x},\epsilon)$).

\end{itemize}

Moreover, if any of the above cases holds, then, for every $X\in B_g(0_x,\epsilon)$, it holds that
\begin{align}\label{eq:cartan:3}
\Phi_*|_{\exp_x(X)}&=P_0^1(s\mapsto \widehat{\exp}_{\hat{x}} (sAX))\circ A\circ P_1^0(s\mapsto \exp_{x} (sX))\nonumber  \\
&=P_0^1\Big(s\mapsto \ol{\exp}_{(x,\hat{x})} \big(s(X, AX)\big)\Big).
\end{align}
\end{theorem}

\begin{proof}
(i) $\Rightarrow$ (iii): 
By taking $Y=Z=\dot{\gamma}$, one has
\[
A_{\RDist}{(\gamma,q)}(t) Y=A_{\RDist}{(\gamma,q)}(t) Z=\dot{\hat{\gamma}}_{\RDist}(\gamma,q)(t),
\]
for all $t\in [0,1]$.

(ii) $\Rightarrow$ (iv): 
Let $u,v\in T|_x M$, $\n{u}_g<\epsilon$
and define for $t\in [0,1]$
\[
Y_{u,v}(t):=\dif{s}\big|_0 \exp_x(t(u+sv))=t(\exp_x)_*|_{tu} v.
\]
It is the Jacobi field on $M$
along the geodesic $\gamma_u(t):=\exp_x(tu)$, $t\in [0,1]$,
with $Y_{u,v}(0)=0$, $\nabla_{u} Y_{u,v}=v$.

Proposition \ref{pr:rol_geodesic}
implies that the rolling
curve $q_{\RDist}{(\gamma_u,q)}$ along $\gamma_u$
is given as 
\[
\hat{\gamma}_{\RDist}{(\gamma_u,q)}(t)=\widehat{\exp}_{\hat{x}}(tAu),\quad
A_{\RDist}{(\gamma_u,q)}(t)=P_0^t(\hat{\gamma}_{\RDist}{(\gamma_u,q)})\circ A\circ P_t^0(\gamma_u).
\]
On the other hand, the assumption
implies that
\[
\Rol(Y_{u,v}(t),\dot{\gamma}_u(t))(A_{\RDist}(\gamma_u,q)(t))\dot{\gamma}_u(t)=0,\quad t\in [0,1]
\]
and Proposition \ref{le:rol_Jacobi}
imply that $\hat{Y}_{u,v}:=A_{\RDist}{(\gamma_u,q)}Y_{u,v}$
is a Jacobi field on $\hat{M}$ along
the geodesic $\hat{\gamma}_{\RDist}{(\gamma_u,q)}$.

Clearly, $\hat{Y}_{u,v}(0)=0$ and 
$\hat{\nabla}_{A u} \hat{Y}_{u,v}=A\nabla_u Y_{u,v}=Av$,
from which it follows (by the uniqueness of solutions of second order ODEs) that $\hat{Y}_{u,v}$ must be the Jacobi field given by
\[
\hat{Y}_{u,v}(t)=\dif{s}\big|_0 \widehat{\exp}_{\hat{x}}(tA(u+sv))
=t(\widehat{\exp}_{\hat{x}})_*|_{tAu} (Av).
\]
Setting $t=1$, we see that
\[
A_{\RDist}{(\gamma_u,q)}(1)(\exp_x)_*|_{u} v=(\widehat{\exp}_{\hat{x}})_*|_{Au} (Av),
\]
for all $u,v\in T|_x M$ with $\n{u}_g<\epsilon$.
In other words, this means that
\[
\Phi_*|_{y}=(\widehat{\exp}_{\hat{x}})_*|_{Au}\circ A\circ (\exp_x^{-1})_*|_y
=A_{\RDist}{(\gamma_{\exp_x^{-1}(y)},q)}(1),
\]
for every $y\in B_{d_g}(x,\epsilon)$, where $B_{d_g}(x,\epsilon)$ is also equal to $\{\exp_x(u)\in T|_x M\ |\ \n{u}_g<\epsilon\}$.
Since $A_{\RDist}{(\gamma_{\exp_x^{-1}(y)},q)}(1)\in Q$,
this means that $\Phi_*|_y$ is an isometry $T|_y M\to T|_{\Phi(y)} \hat{M}$
i.e., $\Phi$ is an isometry.

(iv) $\Rightarrow$ (iii):
This follows from Lemma \ref{le:isometries_as_sections_of_Q} below.

(iii) $\Rightarrow$ (i):
Proposition \ref{pr:R_comm_L2} implies that  $\Rol(X,Y)(A')=0$ for all $(x',\hat{x}';A')\in N$ and $X,Y\in T|_{x'} M$.

On the other hand, the assumption implies that
$f:=(\pi_{Q,M})|_N^{-1}$ is a smooth local section
of $\pi_{Q,M}$ defined on $B_{d_g}(x,\epsilon)$
whose image is the integral manifold $N$ of $\RDist$.

Let $\gamma(t)=\exp_x(tu)$, $t\in [0,1]$, be a geodesic of $M$ with $\n{u}_g<\epsilon$.
Then, since $f\circ\gamma$ is an integral curve of $\RDist$
and $f(\gamma(0))=q$, the rolling curve $q_{\RDist}{(\gamma,q)}$ is defined on $[0,1]$ 
and is given by $q_{\RDist}{(\gamma,q)}(t)=f(\gamma(t))$ for all $t\in [0,1]$.
Hence $q_{\RDist}{(\gamma,q)}(t)\in N$ for all $t\in [0,1]$,
which implies that $\Rol(X,Y)(A_{\RDist}{(\gamma,q)}(t))Z=0$
for all $X,Y,Z\in T|_{\gamma(t)} M$. This completes the proof.
\end{proof}

\begin{lemma}\label{le:isometries_as_sections_of_Q}
Let $x_0\in M$ and $\hat{x}_0\in\hat{M}$
with corresponding open neighborhoods $U$ and $\hat{U}$. Then
 there is a isometry onto $\Phi:U\to \hat{U}$
if and only if there is a smooth local $\pi_{Q,M}$-section
$f:U\to Q$, whose image is an integral manifold of $\RDist$
and projects bijectively by $\pi_{Q,\hat{M}}$ onto $\hat{U}$.

Moreover, the correspondence $\Phi\leftrightarrow f$
is given by
\[
& f_\Phi(x)=(x,\Phi(x);\Phi_*|_x), \\
& \Phi_f(x)=\pi_{Q,\hat{M}}\circ f.
\]
\end{lemma}

\begin{proof}
Let $\Phi$ be an isometry onto $U\to \hat{U}$
and define $f_\Phi$ as above.
For every $x\in U$ and $u\in T|_x M$,
let $\gamma_u(t):=\exp_x(tX)$.
Since $\Phi$ is an isometry,
$\Phi\circ\gamma_u$ is a $\hat{g}$-geodesic
starting at $\Phi(x)$.
Moreover, defining $A(t)=\Phi_*|_{\gamma_u(t)}\in Q$
and taking any $X_0\in T|_x M$, $X(t)=P^t_0(\gamma_u)X_0$,
we have $A(t)X(t)=P^t_0(\Phi\circ\gamma_u)(\Phi_*(X_0))$
and hence
\[
\big(\ol{\nabla}_{(\dot{\gamma}_u(t),\dif{t}(\Phi\circ\gamma_u)(t))} A(t)\big)X(t)
=\hat{\nabla}_{\Phi_*\dot{\gamma}_u(t)} (A(t)X(t))-A(t)\nabla_{\dot{\gamma}_u(t)} X(t)=0,
\]
which proves that $t\mapsto (\gamma_u(t),(\Phi\circ\gamma_u)(t);A(t))=:q(t)$ is an integral curve of $\RDist$
through $q(0)=(x,\Phi(x);\Phi_*|_x)$. 
On the other hand, $q(t)=f_{\Phi}(\gamma_u(t))$
and thus it follows that
\[
(f_{\Phi})_*(u)=\dot{q}(0)\in\RDist|_{q(0)}.
\]
Hence the image of $f_{\Phi}$ is an integral manifold of $\RDist$
and it clearly projects bijectively onto $\hat{U}$ by $\pi_{Q,\hat{M}}$.

Conversely, suppose that $f:U\to Q$
is a local $\pi_Q$-section whose image is an integral manifold of $\RDist$
and which projects onto $\hat{U}$.
Define $\Phi_f$ as above. Then $\Phi_f:U\to \hat{U}$
and, for every $x\in U$ and $X\in T|_x M$,
we have $f_*(X)=\LRD(X)|_{f(x)}$ and thus
\[
\n{(\Phi_f)_*X}_{\hat{g}}
=\n{(\pi_{Q,\hat{M}})_*(f_*(X))}_{\hat{g}}
=\n{(\pi_{Q,\hat{M}})_*(\LRD(X)|_{f(x)})}_{\hat{g}}
=\n{f(x)X}_{\hat{g}}=\n{X}_g,
\]
where the final equality follows from the fact that $f(x)\in Q$.
The fact that $\Phi$ is a bijection $U\to\hat{U}$ is clear.
Hence the conclusion.
\end{proof}

We can now provide an argument for Theorem \ref{th:ambrose}.
According to the assumptions done in the statement, Theorem \ref{th:cartan} implies that there is an integral manifold of $\RDist$ passing through $q_0=(x_0,\hat{x}_0;A_0)$.
Hence, we may choose the maximal connected integral manifold $N$ of $\RDist$ passing throught $q_0$
(where $N$ is the union of all connected integral manifold of $\RDist$ passing through $q_0$, see e.g. Lemma 3.19 in \cite{kolar93}).

Endow $N$ with a Riemannian metric $h$
given by: $h(\LRD(X)|_q,\LRD(Y)|_q)=g(X,Y)$ for $q=(x,\hat{x};A)\in N$ and $X,Y\in T|_x M$.
It is then clear that 
$$F:=\pr_1\circ\pi_Q|_N:N\to M,\hbox{ and }
G:=\pr_2\circ\pi_Q|_N:N\to\hat{M},$$
are local isometries onto open subsets of $M$ and $\hat{M}$
(see also the proof of Corollary \ref{cor:weak_ambrose}).

We next intend to prove that $(N,h)$ is a complete Riemannian manifold.
Here, we have to be more careful than in the proof of "(i) $\Rightarrow$ (iii)" in Corollary \ref{cor:weak_ambrose}
since we cannot assume that $N$ is the whole orbit $\mc{O}_{\RDist}(q_0)$.

First of all, the facts that $N$ is an integral manifold of $\RDist$
and $F$ is a local isometry
imply that, for any $q=(x,\hat{x};A)\in N$
and any $g$-geodesic $t\mapsto \gamma(t)=\exp_x(tX)$ on $M$ starting at $x$, the rolling curve $t\mapsto q_{\RDist}(\gamma,q)(t)$ 
stays in $N$ is a $h$-geodesic on $N$ for $t$ in a small interval containing $0$.

Let us assume that $N$ is not complete.
Then there exists a $h$-geodesic
$\Gamma:[0,T[\to N$ starting from $q_0$
where $[0,T[$ is the maximal non-negative interval of definition
and $T<\infty$.
Since $F$ is a local isometry, $F\circ\Gamma$ is a $g$-geodesic on $M$
and since $\Gamma$ is an integral curve of $\RDist$,
it follows that there is a unique $X\in T|_{x_0} M$ such that,
for $t\in [0,T[$, one has 
\[
\Gamma(t)&=q_{\RDist}\big((s\mapsto \exp_{x_0}(sX)),q_0\big)(t) \\
&=\big(\exp_{x_0}(tX),\widehat{\exp}_{x_0} (tA_0X);
P_0^t(s\mapsto \widehat{\exp}_{\hat{x}_0}(sA_0X))\circ A_0\circ P^0_t(s\mapsto \exp_{x_0}(sX))\big).
\]
We write $(\gamma(t),\hat{\gamma}(t);A(t)):=\Gamma(t)$.
Since $M$ and $\hat{M}$ are complete,
the right hand side of the above equation makes sense for all $t\geq 0$
and we define $\Gamma$ on $[T,\infty[$ by this formula.
We emphasize that we assume $\Gamma$ to be a geodesic on $N$ only for $[0,T[$.

Write $q_T=(x_T,\hat{x}_T;A_T):=\Gamma(T)$.
Choose $\epsilon>0$ such that $\exp_{x_T}$ and $\widehat{\exp}_{\hat{x}_T}$
are diffeomorphisms $B(0,\epsilon)\to B_{d_g}(x_T,\epsilon)$, $B(0,\epsilon)\to B_{d_{\hat{g}}}(\hat{x}_T,\epsilon)$ respectively.

If $\omega$ is any geodesic $[0,1]\to B_{d_g}(x_T,\epsilon)$
starting from $x_T$,
then the concatenation $\omega\sqcup\gamma$ of $\omega$ and $\gamma$
is a once broken geodesic
starting from $x_0$ and therefore, Eq. (\ref{eq:ambrose:1}) implies that
the assumptions of Theorem \ref{th:cartan}, Case (i),
are satisfied (with $(\omega,q_T)$ in place of $(\gamma,q)$).
Indeed, for every $X,Y\in T|_{\omega(t)} M$ and $t$,
\[
& \Rol(X,Y)(A_{\RDist}(\omega,q_T)(t))
=\Rol(X,Y)(A_{\RDist}(\omega,q_{\RDist}(\gamma,q_0)(T))(t)) \\
=&\Rol(X,Y)(A_{\RDist}(\omega\sqcup\gamma,q_0)(t+T))=0.
\]
Therefore, Case (iii) there implies
the existence of a connected integral manifold $\tilde{N}$
of $\RDist$ passing through $q_T=\Gamma(T)$.

Since $\tilde{N}$ is an integral manifold of $\RDist$
and $\Gamma$ is an integral curve of $\RDist$
and since $\Gamma(T)\in \tilde{N}$,
it follows that $\Gamma(t)\in \tilde{N}$
for all $t$ in an open interval $]T-\eta,T+\eta[$ containing $T$.
Since $\Gamma(t)\in N$ for $t\in [0,T[$,
it follows that, for some $t_0\in ]T-\eta,T[$,
we have $\Gamma(t_0)\in N\cap\tilde{N}$.

Thus $N\cap\tilde{N}\neq\emptyset$
and hence $N\cup \tilde{N}$ is a connected integral manifold of $\RDist$
passing through $q_0$ which, because of the maximality of $N$,
implies that $\tilde{N}\subset N$.
This implies that $\Gamma$ is a geodesic of $N$
(since $F\circ\Gamma=\gamma$ is a geodesic of $M$
and $F$ is a local isometry)
on the interval $[0,T+\eta[$,
contradicting the choice of the finite time $T$.
Thus $(N,h)$ is complete.

Since $F=\pr_1\circ \pi_Q|_N$
and $G=\pr_2\circ \pi_Q|_N$ are local Riemannian isometries,
it follows from Proposition II.1.1 in \cite{sakai91} 
that they are covering maps. Taking finally into account that $M$ is simply connected, one gets that $F$ is an isometric diffeomorphism $N\to M$
and hence $G\circ F^{-1}:M\to\hat{M}$ is a Riemannian covering map.

Finally notice that if $\gamma:[0,a]\to M$ is a minimal geodesic starting from $x_0$,
then $(G\circ F^{-1})(\gamma(t))=\widehat{\exp}_{\hat{x}_0}(tA_0\dot{\gamma}(0))$
and hence $\Phi=G\circ F^{-1}$.

%%%%%%%%%%%%%%%%%%%%%%%%%%%%%%%%%%%%%%%%%%
\section{Rolling Against a Space Form}\label{space-form}
%%%%%%%%%%%%%%%%%%%%%%%%%%%%%%%%%%%%%%%%%%

This section is devoted to the special case of
the rolling problem $(R)$ with one of the Riemannian manifolds, usually 
$(\hat{M},\hat{g})$,
being equal to a space form i.e., a simply connected complete Riemannian manifold of constant curvature.
The possible cases are: (i) Euclidean space with Euclidean metric (zero curvature), (ii) Sphere (positive curvature) and (iii) Hyperbolic space (negative curvature), cf. e.g. \cite{sakai91}.

As mentioned in the introduction,
the rolling problem against a space form actually presents a fundamental feature:
on the bundle $\pi_{Q,M}:Q\to M$ one can define a principal bundle structure
that preserves the rolling distribution $\RDist$,
and this renders the study of controllability of $\srol$ easier to handle.

We will first provide a detailed study for the rolling against an Euclidean space
and then proceed to the case of space forms with non-zero curvature.

%%%%%%%%%%%%%%%%%%%%%%%%%%%%%%%%%%%%%%%%%%
\subsection{Rolling Against an Euclidean Space}\label{flat}
%%%%%%%%%%%%%%%%%%%%%%%%%%%%%%%%%%%%%%%%%%

In this section, we give a necessary and sufficient
condition for the controllability
of $\srol$ in the case that $\hat{M}=\R^n$
equipped with the Euclidean metric $\hat{g}=s_n$.

Recall that if $V$ is a finite dimensional inner product space with $h$
the inner product, the special Euclidean group of $(V,h)$
also denoted $\Euc(V)$ is equal to $V\times \SO(V)$,
and is equipped with the group operation $\star$ given by
\[
(v,L)\star (u,K):=(Lu+v,L\circ K).
\]
Here $\SO(V)$ is defined with respect to the inner product $h$ of $V$.
In particular, we write $\Euc(n)$ for $\Euc(\R^n)$ with $\R^n$ equipped with the standard
inner product.

Now fix a point $q_0$ of $Q=Q(M,\R^n)$ of the form $q_0=(x_0,0;A_0)$
i.e., the initial contact point on $M$ is equal to $x_0$ and, on $\R^n$, it is the origin.
Since $(\R^n,s_n)$ is flat, for any a.c. curve $t\mapsto \hat{x}(t)$ in $\R^n$ and $\hat{X}\in \R^n$
we have $P^t_0(\hat{x}(t))\hat{X}=\hat{X}$,
where we understand the canonical isomorphisms
$T|_{\hat{x}(0)} \R^n\cong \R^n\cong T|_{\hat{x}(t)} \R^n$.
It follows that we parameterize the rolling curves
explicitly in the form:
\begin{align}\label{eq:roll_plane_explicit}
q_{\RDist}(\gamma,(x_0,\hat{x};A))(t)
=\Big(\gamma(t),\hat{x}+A\int_0^t P^0_s(\gamma)\dot{\gamma}(s)\diff s;
AP^0_t(\gamma)\Big),
\end{align}
where $\gamma\in\Omega_{x_0}(M)$.

From this it follows that for any $(x_0,0;A_0),(x_0,\hat{x};A)\in Q$
and $\gamma\in\Omega_{x_0}(M)$, the point $q_{\RDist}(\gamma,(x_0,\hat{x};A))(1)$
is equal to 
\[
\big(x_0,\hat{x}+AA_0^{-1}\hat{\gamma}_{\RDist}(\gamma,(x_0,0;A_0))(1);
AA_0^{-1}A_{\RDist}(\gamma,(x_0,0;A_0))(1)\big).
\]
Let $\gamma\in\Omega_{x_0}(M)$ be a piecewise $C^1$-loop of $M$ based at $x_0$.
We define a map 
\[
& \rho=\rho_{q_0}:\Omega_{x_0}(M)\to \Euc(n); \\
& \rho(\gamma)=\big(\hat{\gamma}_{\RDist}(\gamma,q_0)(1),A_{\RDist}(\gamma,q_0)(1)A_0^{-1}\big),
\]
where $q_0=(x_0,0;A_0)\in Q$.
Hence by Remark \ref{re:group_property_of_rolling} and the above formulas we have
\[
\rho(\omega.\gamma)
=&\big(
\hat{\gamma}_{\RDist}\big(\omega,q_{\RDist}(\gamma,q_0)(1)\big)(1),
A_{\RDist}\big(\omega,q_{\RDist}(\gamma,q_0)(1)\big)(1)A_0^{-1}\big) \\
=&\big(
\hat{\gamma}_{\RDist}(\gamma,q_0)(1)+A_{\RDist}(\gamma,q_0)(1)A_0^{-1}\hat{\gamma}_{\RDist}(\omega,q_0)(1), \\
& A_{\RDist}(\gamma,q_0)(1)A_0^{-1}A_{\RDist}(\omega,q_0)(1)A_0^{-1}
\big) \\
=&\big(\hat{\gamma}_{\RDist}(\gamma,q_0)(1),A_{\RDist}(\gamma,q_0)(1)A_0^{-1}\big)
\star \big(\hat{\gamma}_{\RDist}(\omega,q_0)(1),A_{\RDist}(\omega,q_0)(1)A_0^{-1}\big) \\
=&\rho(\gamma)\star\rho(\omega).
\]
Thus $\rho$ is a group anti-homomorphism
$(\Omega_{x_0}(M),.)\to (\Euc(n),\star)$.
This proves that the elements of the form $\rho(\omega)$,
$\omega\in\Omega_{x_0} (M)$, form a subgroup of $\Euc(n)$.
We also see that
\[
(\hat{\gamma}_{\RDist}(\gamma,q)(1),A_{\RDist}(\gamma,q)(1))
=(\hat{x},A)\star (0,A_0)^{-1}\star \rho_{q_0}(\gamma)\star (0,A_0),
\]
where $q=(x_0,\hat{x};A),q_0=(x_0,0;A_0)\in Q$ and $\gamma\in\Omega_{x_0}(M)$.

We also make the simple observation from Eq. (\ref{eq:roll_plane_explicit})
that the image of $\pr_2\circ \rho:\Omega_{x_0}(M)\to\SO(n)$
is exactly $A_0H|_{x_0}A_0^{-1}$, where $H|_{x_0}$ is the holonomy group of $(M,g)$ at $x_0$.
Here $A_0H|_{x_0}A_0^{-1}=H|_{F}$ with respect to the orthonormal frame $F=(A_0^{-1}e_1,\dots,A_0^{-1}e_n)$
where $e_1,\dots,e_n$ is the standard basis of $\R^n$.

From these remarks the next proposition follows easily.

\begin{proposition}\label{pr:principal_bundle-against_plane}
Let $Q=Q(M,\R^n)$ and $q_0=(x_0,\hat{x}_0;A_0)\in Q$.
Then the map
\[
& K_{q_0}:\pi_{Q,M}^{-1}(x_0)\to \Euc(n); \\
& (x_0,\hat{x};A)\mapsto ( \hat{x}-\hat{x}_0 ,AA_0^{-1})
\]
is a diffeomorphism which carries the fiber $\pi_{\mc{O}_{\RDist}(q_0),M}^{-1}(x_0)$
of the orbit $\mc{O}_{\RDist}(q_0)$
to a
submanifold of $\Euc(n)$.
In particular, if $\hat{x}_0=0$ we have that
\[
K_{q_0}(\pi_{\mc{O}_{\RDist}(q_0),M}^{-1}(x_0))=\rho_{q_0}(\Omega_{x_0}(M))
\]
which is a Lie subgroup of $\Euc(n)$.
\end{proposition}

We will make some standard observations
of subgroups $G$ of an Euclidean group $\Euc(V)$, where $(V,h)$ is a finite dimensional inner product space.
Call an element of $G$ of the form $(v,\id_V)$ a \emph{pure translation} of $G$
and write $T=T(G)$ for the set that they form. Clearly $T$ is a subgroup of $G$.
As before, $\pr_1$, $\pr_2$ denote the projections $\Euc(V)\to V$ and $\Euc(V)\to \SO(V)$.
The natural action, also written by $\star$, of $\SO(V)$ on $V$
is defined as
\[
(u,K)\star v:=Kv+u,\quad (u,K)\in\SO(V),\ v\in V.
\]

\begin{proposition}\label{pr:subgroup_of_SE}
Let $G$ be a Lie subgroup of $\Euc(V)$ with $\pr_2(G)=\SO(V)$.
Then either of the following cases hold:
\begin{itemize}
\item[(i)] $G=\Euc(V)$ or
\item[(ii)] there exists $v^*\in V$ which is a fixed point of $G$.
\end{itemize}
\end{proposition}

\begin{proof}
Suppose first that $T=T(G)$ is non-trivial
i.e., there exists a pure translation $(v,\id_V)\in T$, $v\neq 0$.
Then for any $(w,A)\in G$ it holds that
\[
G\ni &(w,A)^{-1}\star (v,\id_V)\star (w,A)
=(-A^{-1}w,A^{-1})\star (v+w,A) \\
&=(A^{-1}(v+w)-A^{-1}w,\id_V)=(A^{-1}v,\id_V)
\]
which implies that
\[
T\supset\{(A^{-1}v,\id_V)\ |\ (w,A)\in G\}
=&\{(A^{-1}v,\id_V)\ |\ A\in \pr_2(G)=\SO(V)\} \\
=&S^{n-1}(0,\n{v})\times \{\id_V\}
\]
where $S^{n-1}(w,r)$, $w\in\R^n,r>0$ is the sphere of radius $r$ centered at $w\in V$
and $\n{\cdot}=h(\cdot,\cdot)^{1/2}$.
If $w\in V$ such that $\n{w}\leq \n{v}$ then it is clear
that there are $u,u'\in S^{n-1}(0,\n{v})$ such that $u+u'=w$
(choose $u\in S^{n-1}(0,\n{v})\cap S^{n-1}(w,\n{v})$
and $u'=w-u$).
Therefore
\[
(w,\id_V)=(u,\id_V)\star (u',\id_V)\in T
\]
i.e., $\ol{B}(0,\n{v})\subset T$ where $\ol{B}(w,r)$ is the closed ball of radius $r$ centered at $w$.
Thus for all $k\in\N$,
\[
& \{\underbrace{\ol{B}(0,\n{v})+\dots+\ol{B}(0,\n{v})}_{k\ \textrm{times}}\}\times\{\id_{V}\} \\
=&\underbrace{(\ol{B}(0,\n{v})\times \{\id_V\})\star \dots \star (\ol{B}(0,\n{v})\times \{\id_V\})}_{k\ \textrm{times}}\subset T.
\]
From this we conclude that
$V\times\{\id_V\}=T$.

Therefore we get the case (i) since
\[
G=&T\star G=\{(u,\id_V)\star (w,A)\ |\ u\in V,\ (w,A)\in G\} \\
=&\{(u+w,A)\ |\ u\in V,\ (w,A)\in G\} \\
=&\{(u,A)\ |\ u\in V,\ A\in\pr_2(G)=\SO(V)\} \\
=&V\times\SO(V)=\Euc(V).
\]

The case that is left to investigate is the one where $T$ is trivial i.e., $T=\{(0,\id_V)\}$.
In this case the smooth surjective Lie group homomorphism $\pr_2|_G:G\to\SO(V)$ is also injective.
In fact, if $A=\pr_2(v,A)=\pr_2(w,A)$ for $(v,A),(w,A)\in G$ and $v\neq w$,
then
\[
G\ni (w,A)\star (v,A)^{-1}=(w,A)\star (-A^{-1}v,A^{-1})=(w-v,\id_V)\in T
\]
and since $(w-v,\id_V)\neq (0,\id_V)$, this contradicts the triviality of $T$.
It follows that $\pr_2|_G$ is a Lie group isomorphism onto $\SO(V)$ and hence a diffeomorphism. 
In particular, $G$ is compact since $\SO(V)$ is compact.

We next show that there exists $v^*\in V$ which is a fixed point of $G$.
Indeed, taking arbitrary $v\in V$ and writing $\mu_{\mathrm{H}}$ for the (right- and) left-invariant normalized (to $1$) Haar measure of the compact group $G$, then we may define
\[
v^*:=\int_{G} (B\star v)\diff\mu_{\mathrm{H}}(B).
\]
Thus for $(w,A)\in G$,
\[
(w,A)\star v^*=&w+Av^*
=\int_{G} \big(w+A(B\star v)\big)\diff\mu_{\mathrm{H}}(B)
=\int_{G} \big(((w,A)\star B)\star v\big)\diff\mu_{\mathrm{H}}(B) \\
=&\int_G (B\star v)\diff\mu_{\mathrm{H}}(B)=v^*,
\]
where, in the second equality, we have used the linearity of the integral and normality of the Haar measure
and in the last phase the left invariance of the Haar measure.
This proves that $v^*$ is a fixed point of $G$
and completes the proof.

\end{proof}

\begin{remark}\label{re:rem1}
With a slight modification, the previous proof actually gives the following generalisation of the last proposition:
If $G$ is a connected subgroup of $\Euc(V)$ such that the subgroup $\pr_2(G)$ of $\SO(V)$
acts transitively on the unit sphere of $V$
then either (i) $G=V\times\pr_2(G)$ or (ii) there is a fixed point $v^*$ of $G$.

\end{remark}

The previous proposition allows us prove the main theorem of this section.

\begin{theorem}
Suppose $(M,g)$ is a complete Riemannian $n$-manifold and $(\hat{M},\hat{g})=(\R^n,s_n)$ is the Euclidean $n$-space.
Then the rolling problem $\srol$ is completely controllable
if and only if the holonomy group of $(M,g)$ is $\SO(n)$
(w.r.t. an orthonormal frame).
\end{theorem}

\begin{proof}
Suppose first that $\srol$ is completely controllable.
Then for any given $q_0=(x_0,\hat{x}_0;A_0)\in Q$
we have that $\pi_{Q,M}^{-1}(x_0)=\pi_{\mc{O}_{\RDist}(q_0),M}^{-1}(x_0)$.
In particular, taking any $q_0\in Q$ of the form $q_0=(x_0,0;A_0)$ (i.e., $\hat{x}_0=0$),
we have by Proposition \ref{pr:principal_bundle-against_plane} that
\[
\Euc(n)=K_{q_0}(\pi_{Q,M}^{-1}(x_0))=K_{q_0}(\pi_{\mc{O}_{\RDist}(q_0),M}^{-1}(x_0))
=\rho_{q_0}(\Omega_{x_0}(M)).
\]
Hence the image of $\pr_2\circ\rho_{q_0}$ is $\SO(n)$
and, on the other hand, this image is also $A_0H|_{x_0} A_0^{-1}$
as noted previously. This proves the necessity of the condition.

Assume now that the holonomy group of $M$ is $\SO(n)$
or, more precisely, that for any $x\in M$
we have $H|_{x}=\SO(T|_{x} M)$.
Let $q=(x,0;A)\in Q$ and let $G_q:=K_q(\pi_{\mc{O}_{\RDist}(q),M}^{-1}(x))$
(see Proposition \ref{pr:principal_bundle-against_plane}) which is a subgroup of $\Euc(n)$.
Since
\[
\SO(n)=AH|_x A^{-1}=(\pr_2\circ\rho)(\Omega_x(M))=\pr_2(K_q(\pi_{\mc{O}_{\RDist}(q),M}^{-1}(x)))=\pr_2(G_q),
\]
by Proposition \ref{pr:subgroup_of_SE}, either (i) $G_q=\Euc(n)$ or (ii) there exists
a fixed point $w_q^*\in \R^n$ of $G_q$.

If (i) is the case for some $q_0=(x_0,0;A_0)\in Q$,
then, since $K_{q_0}$ maps $\pi_{Q,M}^{-1}(q_0)\cap \mc{O}_{\RDist}(q_0)$
diffeomorphically onto $G_q=\Euc(n)$,
it follows that $\pi_{Q,M}^{-1}(q_0)\cap \mc{O}_{\RDist}(q_0)=\pi_{Q,M}^{-1}(q_0)$
and hence $\mc{O}_{\RDist}(q_0)=Q$
(since $\pi_{\mc{O}_{\RDist}(q_0),M}$ is a subbundle of $\pi_{Q,M}$)
i.e., $\srol$ is completely controllable.

Therefore suppose that (ii) holds i.e.,
for every $q\in Q$ of the form $q=(x,0;A)$ there is a fixed point $w_q^*\in\R^n$
of $G_q$. We will prove that this implies that $(M,g)$ is flat
which is a contradiction since $(M,g)$ does not have a trivial holonomy group.

Thus for any point of $Q$ of the form $q=(x,0;A)$
we have for all loops $\gamma\in\Omega_{x}(M)$ that
\[
AP^0_1(\gamma)A^{-1}w_q^*+A\int_0^1 P^0_s(\gamma)\dot{\gamma}(s)\diff s=w_q^*,
\]
since $\rho_q$ is a bijection onto $G_q$ and $w_q^*$ is a fixed point of $G_q$.
In other words we have
\[
(P^0_1(\gamma)-\id)A^{-1}w_q^*+\int_0^1 P^0_s(\gamma)\dot{\gamma}(s)\diff s=0.
\]
Thus if $q=(x,0;A)$ and $q'=(x,0;A')$ are on the same $\pi_Q$ fiber over $(x,0)$,
then
\[
(P^0_1(\gamma)-\id)(A^{-1}w_q^*-A'{}^{-1}w_{q'}^*)=0
\]
for every $\gamma\in\Omega_x(M)$.
On the other hand, since $M$ has full holonomy i.e., $H|_x=\SO(T|_{x} M)$,
and $H|_x=\{P^0_1(\gamma)\ |\ \gamma\in\Omega_x(M)\}$,
it follows from the above equation that
\[
A^{-1}w_q^*=A'{}^{-1}w_{q'}^*.
\]

This means that for every $x\in M$ there is a unique vector $V|_x\in T|_x M$
such that
\[
V|_x=A^{-1}w_q^*,\quad \forall q\in \pi_{Q}^{-1}(x,0).
\]
Moreover, the map $V:M\to TM$; $x\mapsto V|_x$ is a vector field
on $M$ satisfying
\begin{align}\label{eq:parallel_transport_of_V}
P^0_1(\gamma)V|_x-V|_x=-\int_0^1 P^0_s(\gamma)\dot{\gamma}(s)\diff s,
\quad \forall \gamma\in\Omega_{x}(M).
\end{align}

It follows from this that, for any piecewise $C^1$ path $\gamma\in C^1_{\mathrm{pw}}([0,1],M)$,
we have
\begin{align}\label{eq:parallel_transport_of_V_2}
V|_{\gamma(1)}=P_0^1(\gamma)\Big(V|_{\gamma(0)}-\int_0^1 P_s^0(\gamma)\dot{\gamma}(s)\diff s\Big).
\end{align}
Indeed, if $\omega\in\Omega_{\gamma(1)}(M)$,
then $\gamma^{-1}.\omega.\gamma\in\Omega_{\gamma(0)}(M)$ and therefore
\[
& P_1^0(\gamma)P_1^0(\omega)P_0^1(\gamma)V|_{\gamma(0)}-V|_{\gamma(0)}
=P^0_1(\gamma^{-1}.\omega.\gamma)V|_{\gamma(0)}-V|_{\gamma(0)} \\
=&-\int_0^1 P^0_s(\gamma^{-1}.\omega.\gamma)\dif{s}(\gamma^{-1}.\omega.\gamma)(s)\diff s \\
=&-\int_0^1 P^0_s(\gamma)\dot{\gamma}(s)\diff s
-P^0_1(\gamma)\int_0^1 P^0_s(\omega)\dot{\omega}(s)\diff s
-P^0_1(\gamma)P^0_1(\omega)\int_0^1 P^0_s(\gamma^{-1})\dif{s}\gamma^{-1}(s)\diff s \\
=&-\int_0^1 P^0_s(\gamma)\dot{\gamma}(s)\diff s
+P^0_1(\gamma)(P^0_1(\omega)V|_{\gamma(1)}-V|_{\gamma(1)})\\
&+P^0_1(\gamma)P^0_1(\omega)P_0^1(\gamma)\int_0^1 P^0_s(\gamma)\dot{\gamma}(s)\diff s,
\]
that is
\[
\big(P_1^0(\omega)-\id\big)P_0^1(\gamma)\Big(V|_{\gamma(0)}-\int_0^1 P^0_s(\gamma)\dot{\gamma}(s)\diff s\Big)
=\big(P^0_1(\omega)-\id\big)V|_{\gamma(1)}.
\]
Equation (\ref{eq:parallel_transport_of_V_2}) then follows from this 
since $\{P_1^0(\omega)\ |\ \omega\in\Omega_{\gamma(1)}(M)\}=H|_{\gamma(1)}=\SO(T|_{\gamma(1)} M)$.

Since $(M,g)$ is complete, the geodesic $\gamma_X(t)=\exp_x (tX)$
is defined for all $t\in [0,1]$.
Inserting this to Eq. (\ref{eq:parallel_transport_of_V_2})
and noticing that $P^0_s(\gamma_X)\dot{\gamma}_X(s)=X$ in this case for all $s\in [0,1]$,
we get
\[
V|_{\gamma_X(1)}=P^1_0(\gamma_X)(V|_x-X),
\]
Therefore, if $X=V|_x$ and $z:=\gamma_X(1)=\exp_x(V|_x)$, we get
\begin{align}\label{plane:V_zero}
V|_{z}=0.
\end{align}

Inserting this fact into Eq. (\ref{eq:parallel_transport_of_V}), one gets
\[
\int_0^1 P^0_s(\gamma)\dot{\gamma}(s)\diff s=0,
\quad \forall \gamma\in\Omega_{z}(M).
\]
Fix $q^*=(z,0;A_0)\in Q$ (for any isometry $A_0:T|_z M\to T|_0 \R^n$).
Eq. (\ref{eq:roll_plane_explicit}) implies that
\[
\hat{\gamma}_{\RDist}(\gamma,(x_0,0;A_0))(1)=0,\quad \forall \gamma\in\Omega_z(M).
\]

We now apply Theorem \ref{th:fixed_point} to conclude that $(M,g)$ has $(\R^n,s_n)$
as a Riemannian covering (i.e., $(M,g)$ is flat)
and hence reach the desired contradiction mentioned above.
Even though this allows to conclude the proof, we will also give below a direct argument
showing this.

Equation (\ref{eq:parallel_transport_of_V_2}) is trivially equivalent to
\[
P_t^0(\gamma)V|_{\gamma(t)}=V|_{\gamma(0)}-\int_0^t P_s^0(\gamma)\dot{\gamma}(s)\diff s
\]
where $\gamma\in C^1_{\mathrm{pw}}([a,b],M)$, $a<b$, is arbitrary.
Taking $\gamma$ to be smooth and differentiating the above equation w.r.t to $t$
(notice that both sides of the equation are in $T|_{\gamma(0)} M$ for all $t$),
we get
\[
P_t^0(\gamma)\nabla_{\dot{\gamma}(t)} V|_{\gamma(\cdot)}
=-P_t^0(\gamma)\dot{\gamma}(t),
\]
that is
\[
\nabla_{\dot{\gamma}(t)} V|_{\gamma(\cdot)}=-\dot{\gamma}(t).
\]
Since $\gamma$ was an arbitrary smooth curve,
this implies that $V$ is a smooth vector field on $M$ and
\begin{align}\label{plane:V_local}
\nabla_X V=-X,\quad \forall X\in\VF(M).
\end{align}

For any $X\in\VF(M)$, the special curvature $R(X,V)V$
can be seen to vanish everywhere since
\[
R(X,V)V=&\nabla_{X} \nabla_V V-\nabla_V \nabla_X V-\nabla_{[X,V]} V
=-\nabla_X V+\nabla_V X+[X,V] \\
=&[V,X]+[X,V]=0,
\]
where, in the second equality, we used (\ref{plane:V_local}).

For any $X\in T|_z M$, we write $\gamma_X(t)=\exp_z(tX)$ for the geodesic through $z$
in the direction of $X$.
It follows that
\begin{align}\label{plane:V_explicit}
V|_{\gamma_X(t)}=&P_0^t(\gamma_X)(V|_z-\int_0^t P^0_s(\gamma_X)\dot{\gamma}_X(s)\diff s) \nonumber \\
=&P_0^t(\gamma_X)(-\int_0^t X\diff s)=P_0^t(\gamma_X)(-tX)=-t\dot{\gamma}_X(t).
\end{align}
Now for given $X,v\in T|_z M$
let $Y(t)=\pa{s}\big|_0 \exp_z(t(X+sv))$
be the Jacobi field along $\gamma_X$ such that $Y(0)=0$, $\nabla_{\dot{\gamma}_X(t)}Y|_{t=0}=v$.
Then one has
\[
\nabla_{\dot{\gamma}_X(t)} \nabla_{\dot{\gamma}_X} Y
=R(\dot{\gamma}_X(t),Y(t))\dot{\gamma}_X(t)=\frac{1}{t^2}R(V|_{\gamma_X(t)},Y(t))V|_{\gamma_X(t)}=0,
\]
for $t\neq 0$
which means that $t\mapsto \nabla_{\dot{\gamma}_X(t)} Y$ is parallel along $\gamma_X$
i.e.,
\[
\nabla_{\dot{\gamma}_X(t)} Y=P_0^t(\gamma_X)\nabla_{\dot{\gamma}_X(0)} Y=P_0^t(\gamma_X)v.
\]
This allows us to compute
\[
\frac{\diff^2}{\diff t^2} \n{Y(t)}_g^2
=&2\dif{t} g(\nabla_{\dot{\gamma}_X(t)} Y,Y(t)) \\
=&2g(\underbrace{\nabla_{\dot{\gamma}_X(t)} \nabla_{\dot{\gamma}_X} Y}_{=0},Y(t))
+2g(\nabla_{\dot{\gamma}_X(t)} Y,\nabla_{\dot{\gamma}_X(t)} Y) \\
=&2g(P_0^t(\gamma_X)v,P_0^t(\gamma_X)v)=2\n{v}_g^2
\]
and hence for any $t$
\[
\dif{t}\n{Y(t)}_g^2=2\n{v}_g^2 t+\dif{t}\big|_0\n{Y(t)}_g^2=2\n{v}_g^2 t,
\]
because $\dif{t}\big|_0\n{Y(t)}_g^2=2g(\nabla_{\dot{\gamma}_X(0)} Y,Y(0))=0$
since $Y(0)=0$. Again, since $Y(0)=0$,
\[
\n{Y(t)}_g^2=\n{v}_g^2 t^2+\n{Y(0)}_g^2=\n{v}_g^2 t^2
\]
which, when spelled out, means that $\n{t(\exp_z)_*|_{tX}(v)}_g=\n{tv}_g$ and hence,
when $t=1$,
\begin{align}\label{plane:exp_is_isometric}
\n{(\exp_z)_*|_{X}(v)}_g=\n{v}_g,\quad \forall X,v\in T|_z M.
\end{align}

This proves that $\exp_z$ is a local isometry $(T|_z M,g|_z)\to (M,g)$
and hence a Riemannian covering.
Thus $(M,g)$ is flat and the proof if finished.

\end{proof}

\begin{remark}
For results and proofs in similar lines
to those of the above Proposition and Theorem,
see Theorem IV.7.1, p. 193 and Theorem IV.7.2, p. 194 in \cite{kobayashi63}.
\end{remark}

%%%%%%%%%%%%%%%%%%%%%%%%%%%%%%%%%%%
\subsection{Rolling Against a Non-Flat Space Form}
%%%%%%%%%%%%%%%%%%%%%%%%%%%%%%%%%%%

In this subsection, we study the controllability problem of $\srol$
in the case where $\hat{M}$ is a simply connected $n$-dimensional manifold with non zero constant curvature equal to
$\frac{1}{k}$,  with $k\neq 0$.

%%%%%%%%%%%%%%%%%%%%%%%%%%%%%%%%%%%
\subsubsection{Standard Results on Space Forms}
%%%%%%%%%%%%%%%%%%%%%%%%%%%%%%%%%%%

Following section V.3 of \cite{kobayashi63}, we define the space form $\hat{M}_k$ of curvature $\frac{1}{k}$ as a subset of $\R^{n+1}$, $n\in \N$,  given by
\[
\hat{M}_k:=\big\{(x_1,\dots,x_{n+1})\in\R^{n+1}\ |\ x_1^2+\dots+x_n^2+kx_{n+1}^2=k,\ 
x_{n+1}+\frac{k}{|k|}\geq 0\big\}.
\]
Equip $\hat{M}_k$ with a Riemannian metric $\hat{g}_k$ defined as 
the restriction to $\hat{M}_k$ of the non-degenerate symmetric $(0,2)$-tensor 
\[
s_{n,k}:=(\diff x_1)^2+\dots+(\diff x_n)^2+k(\diff x_{n+1})^2.
\]
The condition $x_{n+1}+\frac{k}{|k|}\geq 0$ in the definition $\hat{M}_k$
guarantees that $\hat{M}_k$ is connected also when $k<0$.
If the dimension $n$ is not clear from context, we write $(\hat{M}_{n,k},\hat{g}_{n,k})$
for the above Riemannian manifolds.

\begin{remark}
\begin{itemize}
\item[(i)] If $k=1$ then $\hat{M}_{1}=S^n$ (the usual Euclidean unit sphere in $\R^{n+1}$)
and $s_{n,1}$ is the usual Euclidean metric $s_{n+1}$ on $\R^{n+1}$.
For a fixed $n\in\N$, the spaces $\hat{M}_{k}$ for $k>0$ are all diffeomorphic:
the map $\phi_k:(x_1,\dots,x_n,x_{n+1})\mapsto (\frac{x_1}{\sqrt{k}},\dots,\frac{x_{n}}{\sqrt{k}},x_{n+1})$
gives a diffeomorphism from $\hat{M}_k$ onto $\hat{M}_1$.
Moreover, $\phi_k$ is a homothety since
$\phi_k^*\hat{g}_1=\frac{1}{k}\hat{g}_k$.

\item[(ii)] If $k=-1$ then $s_{n,-1}$ is the usual Minkowski "metric" on $\R^{n+1}$.
For a fixed $n\in\N$, the spaces $\hat{M}_{k}$ for $k<0$ are all diffeomorphic:
the map $\phi_k:(x_1,\dots,x_n,x_{n+1})\mapsto (\frac{x_1}{\sqrt{-k}},\dots,\frac{x_{n}}{\sqrt{-k}},x_{n+1})$
gives a diffeomorphism from $\hat{M}_k$ onto $\hat{M}_{-1}$.
Moreover, $\phi_k$ is a homothety since
$\phi_k^*\hat{g}_1=-\frac{1}{k}\hat{g}_k$.
\end{itemize}
\end{remark}

Let $G(n,k)$ be the Lie group of linear maps $\R^{n+1}\to \R^{n+1}$
that leave invariant the bilinear form
\[
\la x,y\ra_{n,k}:=\sum_{i=1}^n x_i y_i+kx_{n+1}y_{n+1},
\]
for $x=(x_1,\dots,x_{n+1})$, $y=(y_1,\dots,y_{n+1})$ and having determinant $+1$.
In other words, a linear map $B:\R^{n+1}\to\R^{n+1}$ belongs to $G(n,k)$ if and only if
$\det(B)=+1$ and
\[
\la Bx,By\ra_{n,k}=\la x,y\ra_{n,k},\quad \forall x,y\in\R^{n+1},
\]
or, equivalently,
\[
B^T I_{n,k} B=I_{n,k},\quad \det(B)=+1,
\]
where $I_{n,k}=\mathrm{diag}(1,1,\dots,1,k)$.
In particular, $G(n,1)=\SO(n+1)$ and $G(n,-1)=\SO(n,1)$.

The Lie algebra of the Lie group $G(n,k)$ will be denoted by $\mathfrak{g}(n,k)$.
Notice that an $(n+1)\times (n+1)$ real matrix $B$ belongs to $\mathfrak{g}(n,k)$
if and only if
\[
B^T I_{n,k} + I_{n,k}B=0,
\]
where $I_{n,k}$ was introduced above.

Sometimes we identify the form $s_{n,k}$ on $\R^{n+1}$
with $\la\cdot,\cdot\ra_{n,k}$ using the canonical identification of the tangent spaces $T|_{v}\R^{n+1}$
with $\R^{n+1}$.
Notice that if $\hat{x}\in\hat{M}_k$ and $V\in T|_{\hat{x}} \R^{n+1}$,
then
\[
V\in T|_{\hat{x}} \hat{M}_k
\quad\iff\quad
s_{n,k}(V,\hat{x})=0.
\]
In fact, if we identify $V$ as a vector $(V_1,\dots,V_{n+1})$ in $\R^{n+1}$,
then the condition for $V$ to be tangent to the hypersurface $\hat{M}_k$ is
\[
0=&\la V,\hbox{ grad }(x_1^2+\dots+x_n^2+kx_{n+1})\ra_{n+1} \\
=&\la V,(x_1,\dots,x_n,kx_{n+1})\ra_{n+1} \\
=&\sum_{i=1}^n x_i V_i+kx_{n+1}V_{n+1} \\
=&s_{n,k}(V,\hat{x}),
\]
with $\la\cdot,\cdot\ra_{n+1}$ the standard Euclidean inner product of $\R^{n+1}$.

\begin{remark}
By using the bilinear form $\la\cdot,\cdot\ra_{n,k}$ one may restate the definition of $\hat{M}_k$ by
\[
\hat{M}_k=\big\{\hat{x}\in\R^{n+1}\ |\ \la \hat{x},\hat{x}\ra_{n,k}=k,\ 
x_{n+1}+\frac{k}{|k|}\geq 0\big\}.
\]

\end{remark}

\begin{remark}For convenience we recall a standard result (\cite{kobayashi63}, Theorem V.3.1):
The Riemannian manifold $(\hat{M}_k,\hat{g}_k)$ has constant sectional curvature $\frac{1}{k}$
and the isometry group $\Iso(\hat{M}_k,\hat{g}_k)$ is equal to $G(n,k)$.
\end{remark}

We understand without mention that when considering the action of $G(n,k)$
on $\hat{M}_k$ we consider the restriction of the maps of $G(n,k)$ onto the set $\hat{M}_k$.

%%%%%%%%%%%%%%%%%%%%%%%%%%%%%%%%%%%
\subsubsection{Orbit Structure}\label{seq:orbit_structure_constant_curvature}
%%%%%%%%%%%%%%%%%%%%%%%%%%%%%%%%%%%

\begin{proposition}
The bundle $\pi_{Q,M}:Q\to M$ is a principal $G(n,k)$-bundle
with a left action $\mu:G(n,k)\times Q\to Q$ defined by
\[
\mu(B,q)=(x,B\hat{x};B\circ A),
\]
where $q=(x,\hat{x};A)$
and in $B\circ A$ we understand the range $T|_{\hat{x}} \hat{M}_k$
of $A$ to be identified with a linear subspace of $\R^{n+1}$
in the canonical way.

Moreover, the action $\mu$ preserves the distribution $\RDist$
i.e., for any $q\in Q$ and $B\in G(n,k)$,
\[
(\mu_B)_*\RDist|_q=\RDist|_{\mu(B,q)}
\]
where $\mu_B:Q\to Q$; $q\mapsto \mu(B,q)$.
\end{proposition}

\begin{proof}
Let us first check that for $(x,\hat{x};A)\in Q$,
$B\circ A:T|_x M\to \R^{n+1}$ can be viewed as an orientation preserving map $T|_x M\to T|_{B\hat{x}} \hat{M}_k$
and that really $(x,B\hat{x};B\circ A)$ is an element of $Q$.
First of all, $B\hat{x}\in \hat{M}_k$ when $\hat{x}\in\hat{M}_k$
as remarked above.
Moreover, for $X\in T|_x M$,
\[
s_{n,k}((B\circ A)(X),B\hat{x})=s_{n,k}(AX,\hat{x})=0,
\]
since $AX\in T|_{\hat{x}}\hat{M}_k$. Hence $B\circ A:T|_x M\to T|_{B\hat{x}} \hat{M}$.
Similarly, for $X,Y\in T|_x M$,
\[
& \hat{g}_k((B\circ A)(X),(B\circ A)(Y))=s_{n,k}((B\circ A)(X),(B\circ A)(Y)) \\
=&s_{n,k}(AX,AY)=\hat{g}_k(AX,AY)=g(X,Y),
\]
and clearly $B\circ A$ preserves orientation (since $G(n,k)$ is connected).
Thus $(x,B\hat{x};B\circ A)\in Q$.

It is clear that $\mu$ is a well defined left $G(n,k)$-action
on $Q$, that it is free, maps each $\pi_{Q,M}$-fiber to itself
($\pi_{Q,M}\circ\mu(B,q)=\pi_{Q,M}(q)$) and that it
is transitive fiberwise (for each $q,q'\in \pi_{Q,M}^{-1}(x)$,
$\mu(B,q)=q'$ for some $B\in G(n,k)$).
It remains to check the claim that this action
preserves $\RDist$ in the sense stated above.

Let $B\in G(n,k)$ and $q_0=(x_0,\hat{x}_0;A_0)$. The fact that $\Iso(\hat{M}_k,\hat{g})=G(n,k)$
means that defining $F:\hat{M}_k\to\hat{M}_k$; $F=B|_{\hat{M}_k}$
then $F\in \Iso(\hat{M}_k,\hat{g})$.
Clearly $F(\hat{x}_0)=B\hat{x}_0$ and $\hat{F}_*|_{\hat{x}_0}=B|_{T|_{\hat{x}_0} \hat{M}_k}$
and hence by Proposition \ref{pr:iso_equivariance}
\[
\mu(B,q_{\RDist}(\gamma,q_0)(t))=\hat{F}\cdot q_{\RDist}(\gamma,q_0)(t)
=q_{\RDist}(\gamma,\hat{F}\cdot q_0)(t)
=q_{\RDist}(\gamma,\mu(B,q_0))(t),
\]
for any smooth curve $\gamma:[0,1]\to M$, $\gamma(0)=x_0$ and $t\in [0,1]$.
Taking derivative with respect to $t$ at $t=0$ and using the fact that, by definition,
$q_{\RDist}(\gamma,q_0)$ is tangent to $\RDist$, we find that 
\[
(\mu_B)_*\LRD(\dot{\gamma}(0))|_{q_0}
=&(\mu_B)_*\dif{t}\big|_0 q_{\RDist}(\gamma,q_0)(t)
=\dif{t}\big|_0 \mu(B,q_{\RDist}(\gamma,q_0)(t)) \\
=&\dif{t}\big|_0 q_{\RDist}(\gamma,\mu(B,q_0))(t)
=\LRD(\dot{\gamma}(0))|_{\mu(B,q_0)}.
\]
This allows us to conclude.
\end{proof}

We will denote the left action of $B\in G(n,k)$ on $q\in Q$
usually by $B\cdot q=\mu(B,q)$.

\begin{proposition}
For any given $q=(x,\hat{x};A)\in Q$
there is a unique subgroup $G_q$ of $G(n,k)$,
called the holonomy group of $\RDist$,
such that
\[
G_q\cdot q=\mc{O}_{\RDist}(q)\cap \pi_{Q,M}^{-1}(x).
\]
Also, if $q'=(x,\hat{x}';A')\in Q$ is in the same $\pi_{Q,M}$-fiber as $q$,
then $G_q$ and $G_{q'}$ are conjugate in $G(n,k)$
and all conjugacy classes of $G_q$ in $G(n,k)$ are of the form $G_{q'}$.
This conjugacy class will be denoted by $G$.

Moreover, $\pi_{\mc{O}_{\RDist}(q),M}:\mc{O}_{\RDist}(q)\to M$ is a principal $G$-bundle over $M$.
\end{proposition}

\begin{proof}
These results follow from the general theory of principal bundle connections
(cf. \cite{joyce07}, \cite{kobayashi63}) but the argument is 
reproduced here for convenience.

Let $q'\in \mc{O}_{\RDist}(q)\cap \pi_{Q,M}^{-1}(x)$ and choose a $\gamma\in\Omega_x(M)$ such that
$q'=q_{\RDist}(\gamma,q)(1)$.
Since the $G(n,k)$ action is free and transitive on $\pi_{Q,M}^{-1}(x)$,
it follows that there is a unique $B_q(\gamma)\in G(n,k)$ such that
$B_q(\gamma)\cdot q=q'$.
We define $G_q=\{B_q(\gamma)\ |\ \gamma\in\Omega_x(M)\}$
and note that for $\gamma,\omega\in\Omega_x(M)$
one has
\[
& (B_q(\gamma)B_q(\omega))\cdot q
=B_q(\gamma)\cdot (B_q(\omega)\cdot q)
=B_q(\gamma)\cdot q_{\RDist}(\omega,q)(1)
=q_{\RDist}(\omega,B_q(\gamma)\cdot q)(1) \\
=&q_{\RDist}(\omega,q_{\RDist}(\gamma,q)(1))(1)
=q_{\RDist}(\omega.\gamma,q)(1)\in \mc{O}_{\RDist}(q)\cap \pi_{Q,M}^{-1}(x),
\]
which proves that $B_q(\gamma)B_q(\omega)=B_q(\omega.\gamma)\in G_q$.
Next if $\gamma^{-1}:[0,1]\to M$ denotes the inverse path of $\gamma$ i.e., $\gamma^{-1}(t)=\gamma(1-t)$ for $t\in [0,1]$,
it follow that
\[
(B_q(\gamma)B_q(\gamma^{-1}))\cdot q=q_{\RDist}(\gamma^{-1}.\gamma,q)(1)=q,
\]
i.e., $B_q(\gamma)^{-1}=B_q(\gamma^{-1})\in G_q$.
This shows that $G_q$ is indeed a subgroup of $G(n,k)$.
Moreover, it is clear that
\[
G_q\cdot q=\mc{O}_{\RDist}(q)\cap \pi_{Q,M}^{-1}(x),
\]
where the left hand side is $\{B\cdot q\ |\ B\in G_q\}$.

Let us prove the statement about the conjugacy class of $G_q$.
Take $q'=(x,\hat{x}';A')\in Q$.
Because $G(n,k)$ acts transitively on the fibers,
there exists a $B\in G(n,k)$ such that $q'=B\cdot q$.
Therefore for any $\gamma\in\Omega_x(M)$,
\[
(B^{-1}B_{q'}(\gamma)B)\cdot q
=&(B^{-1} B_{q'}(\gamma))\cdot q'
=B^{-1}\cdot q_{\RDist}(\gamma,q') \\
=&q_{\RDist}(\gamma,B^{-1}\cdot q')
=q_{\RDist}(\gamma,q)=B_{q}(\gamma)\cdot q,
\]
i.e., $B^{-1}B_{q'}(\gamma)B=B_q(\gamma)$ since the $G(n,k)$ action is free.
This proves that $B^{-1}G_{q'}B=G_q$.
Moreover, if there is a $B\in G(n,k)$ and a subgroup $G'$ of $G(n,k)$
such that $B^{-1}G'B=G_q$, then defining $q':=B\cdot q$
one gets that $G'=G_{q'}$. 

By Proposition \ref{pr:R_orbit_bundle}, $\pi_{\mc{O}_{\RDist}(q),M}$ is a smooth bundle
and, by what has been said already,
it is clear that $G_q$ preserves the fibers
$\pi_{\mc{O}_{\RDist}(q),M}^{-1}(x)=\pi_{Q,M}^{-1}(x)\cap \mc{O}_{\RDist}(q)$
and the action is free.
Recall that, if a map from some manifold to the ambient manifold
is smooth and its image is contained in the orbit (as a set),
then this map is also smooth as a map into the orbit (as a manifold) (cf. \cite{kolar93}, Theorem 3.22 and Lemma 2.17). As a consequence, the action of $G_q$
is also smooth.
From this, one concludes that $\pi_{\mc{O}_{\RDist}(q),M}$
is a $G_q$-bundle and hence a $G$-bundle
since the Lie groups in the conjugacy class are all isomorphic.
\end{proof}

%%%%%%%%%%%%%%%%%%%%%%%%%%%%%%%%%%%%%%%%%%
\subsubsection{The Rolling Connection}
%%%%%%%%%%%%%%%%%%%%%%%%%%%%%%%%%%%%%%%%%%

Let $\pi_{TM\oplus\R}:TM\oplus\R\to M$
be the vector bundle over $M$ where  $\pi_{TM\oplus\R}(X,r)=\pi_{TM}(X)$.
In this section we will prove the following result.

\begin{theorem}\label{th:rolling_connection}
There exists a vector bundle connection $\nabla^{\Rol}$ of the vector bundle $\pi_{TM\oplus\R}$ that we call the \emph{rolling connection},
and which we define as follows: for every $x\in M$, $Y\in T|_x M$, $X\in\VF(M)$, $r\in\Cinf(M)$,
\begin{align}\label{eq:nabla_rol_explicit}
\nabla^{\Rol}_Y (X,r)=\Big(\nabla_Y X+r(x)Y,Y(r)-\frac{1}{k}g\big(X|_x,Y)\Big),
\end{align}
such that in the case of $M$ rolling
against the space form $\hat{M}_k$, $k\neq 0$, the holonomy group $G$ of $\RDist$
is isomorphic to the holonomy group $H^{\nabla^{\Rol}}$ of $\nabla^{\Rol}$.

Moreover, if one defines a fiber inner product $h_k$ on $TM\oplus\R$ by
\[
h_k((X,r),(Y,s))=g(X,Y)+krs,
\]
where $X,Y\in T|_x M$, $r,s\in\R$,
then $\nabla^\Rol$ is a metric connection in the sense that
for every $X,Y,Z\in\VF(M)$, $r,s\in\Cinf(M)$,
\[
Z\big( h_k((X,r),(Y,s)) \big)=h_k(\nabla^{\Rol}_Z (X,r),(Y,s))+h_k((X,r),\nabla^{\Rol}_Z (Y,s)).
\]
\end{theorem}

Before providing the proof of the theorem, we present the equations of parallel transport
w.r.t $\nabla^\Rol$ along a general curve and along a geodesic of $M$
and also the curvature of $\nabla^\Rol$.
Let $\gamma:[0,1]\to M$ be an a.c. curve on $M$, $\gamma(0)=x$ and let $(X_0,r_0)\in T|_x M\oplus\R$.
Then the parallel transport $(X(t),r(t))=(P^{\nabla^{\Rol}})_0^t(\gamma)(X_0,r_0)$
of $(X_0,r_0)$ is determined from the equations
\begin{align}\label{eq:rol_parallel}
\begin{cases}
\displaystyle \nabla_{\dot{\gamma}(t)} X+r(t)\dot{\gamma}(t)=0, \cr
\displaystyle \dot{r}(t)-\frac{1}{k}g(\dot{\gamma}(t),X(t))=0,
\end{cases}
\end{align}
for a.e. $t\in [0,1]$.
In particular, if $\gamma$ is a geodesic on $(M,g)$, one may derive the following uncoupled 
second order differential equations for $X$ and $r$,
\begin{align}\label{eq:rol_parallel_geodesic}
\begin{cases}
\displaystyle \nabla_{\dot{\gamma}(t)}\nabla_{\dot{\gamma}(t)} X+\frac{1}{k}g(X(t),
\dot{\gamma}(t))\dot{\gamma}(t)=0, \cr
\displaystyle \ddot{r}(t)+\frac{ \n{\dot{\gamma}(t)}^2_g }{k}r(t)=0,
\end{cases}
\end{align}
for all $t$.

One easily checks by direct computation that the connection $\nabla^\Rol$ on $\pi_{TM\oplus\R}$
has the curvature, 
\begin{align}\label{eq:curvature_of_nabla_rol}
R^{\nabla^{\Rol}}(X,Y)(Z,r)
=\big(R(X,Y)Z-\frac{1}{k}(g(Y,Z)X-g(X,Z)Y),0\big),
\end{align}
where $X,Y,Z\in\VF(M)$, $r\in\Cinf(M)$.

We will devote the rest of the subsection to prove Theorem \ref{th:rolling_connection}.

\begin{proof}
The rolling distribution $\RDist$ is a principal bundle
connection for the principal $G(n,k)$-bundle $\pi_{Q,M}:Q\to M$
and hence there is a vector bundle $\xi:E\to M$
with fibers isomorphic to $\R^{n+1}$
and a unique linear vector bundle connection $\nabla^{\Rol}:\Gamma(\xi)\times \VF(M)\to \Gamma(\xi)$
which induces the distribution $\RDist$ on $Q$.
This clearly implies that the holonomy group $G$ of $\RDist$
and $H^{\nabla^{\Rol}}$ of $\nabla^{\Rol}$ are isomorphic.
We will eventually show that $\xi$ is further isomorphic to $\pi_{TM\oplus\R}$
and give the explicit expression (\ref{eq:nabla_rol_explicit}) for the connection 
of $\pi_{TM\oplus\R}$ induced by this isomorphims
from $\nabla^{\Rol}$ on $\xi$.

There is a canonical non-degenerate metric $h_k:E\odot E\to M$ on the vector bundle $\xi$
(positive definite when $k>0$ and indefinite of Minkowskian type if $k<0$)
and the connection $\nabla^{\Rol}$ is a metric connection w.r.t. to $h_k$
i.e., for any $Y\in\VF(M)$ and $s,\sigma\in \Gamma(\nu)$,
\begin{align}\label{eq:rol_connection_metric}
Y\big( h_k(s,\sigma) \big)=h_k(\nabla^{\Rol}_Y s,\sigma)+h_k(s,\nabla^{\Rol}_Y \sigma).
\end{align}

The construction of $\xi$ goes as follows (see \cite{joyce07}, section 2.1.3).
Define a left $G(n,k)$-group action $\beta$ on $Q\times \R^{n+1}$
by
\[
\beta(B,(q,v))=(B\cdot q,Bv),
\]
where $q\in Q$, $v\in\R^{n+1}$, $B\in G(n,k)$.
The action $\beta$ is clearly smooth, free and proper.
Hence $E:=(Q\times\R^{n+1})/\beta$ is a smooth manifold
of dimension $n+(n+1)=2n+1$. The $\beta$-equivalence classe (i.e., $\beta$-orbit)
of $(q,v)\in Q\times\R^{n+1}$
is denoted by $[(q,v)]$. Then one defines $\xi\big([(q,v)]\big)=\pi_{Q,M}(q)$
which is well defined since the $\beta$-action preserves the fibers of 
$Q\times\R^{n+1}\to M$;
$(q,v)\mapsto \pi_{Q,M}(q)$. 
We prove now that $\xi$ is isomorphic, as a vector bundle over $M$,
to 
\begin{align}
\pi_{TM\oplus\R}:TM\oplus\R&\to M,\nonumber\\
(X,t)&\mapsto \pi_{TM}(X).\nonumber
\end{align}
Indeed, let $f\in\Gamma(\xi)$
and notice that for any $q\in Q$ there exists a unique $\ol{f}(q)\in\R^{n+1}$
such that $[(q,\ol{f}(q))]=f(\pi_{Q,M}(q)))$ by the definition of the action $\beta$.
Then $\ol{f}:Q\to\R^{n+1}$ is well defined and,
for each $q=(x,\hat{x};A)$, there are unique $X|_{q}\in T|_xM$, $r(q)\in\R$ such that
\[
\ol{f}(q)=AX|_{q}+r(q)\hat{x}.
\]
The maps $q\mapsto X|_q$ and $q\mapsto r(q)$
are smooth.
We show that the vector $X|_q$ and the real number $r(q)$ depend only on $x$ and hence define a vector field
and a function on $M$.
One has $[((x,\hat{x};A),v)]=[((x,\hat{y};B),w)]$ if and only if
there is $C\in G(n,k)$ such that
$C\hat{x}=\hat{y}$, $CA=B$ and $Cv=w$.
This means that $C|_{\IM A}=BA^{-1}|_{\IM A}:T|_{\hat{x}} \hat{M}_k\to T|_{\hat{y}} \hat{M}_k$
(with $\IM A$ denoting the image of $A$)
and this defines $C$ uniquely as an element of $G(n,k)$
and also, by the definition of $\ol{f}$,
$C\ol{f}(x,\hat{x},A)=\ol{f}(x,\hat{y},B)$.
Therefore,
\[
BX|_{(x,\hat{y};B)}+r(x,\hat{y};B)\hat{y}
=C(AX|_{(x,\hat{x};A)}+r(x,\hat{x};A)\hat{x})
=BX|_{(x,\hat{x};A)}+r(x,\hat{x};A)\hat{y},
\]
which shows that $X|_{(x,\hat{y};B)}=X|_{(x,\hat{x};A)}$,
$r(x,\hat{y};B)=r(x,\hat{x};A)$ and proves the claim.

Hence for each $f\in\Gamma(\xi)$ there are unique $X_f\in\VF(M)$ and $r_f\in\Cinf(M)$
such that
\[
f(x)=\big[\big((x,\hat{x};A),AX_f|_x+r_f(x)\hat{x}\big)\big],
\]
(here the right hand side does not depend on the choice of $(x,\hat{x};A)\in\pi_{Q,M}^{-1}(x)$).

Conversely, given $X\in\VF(M)$, $r\in\Cinf(M)$
we may define $f_{(X,r)}\in\Gamma(\xi)$ by
\[
f_{(X,r)}(x)=\big[\big((x,\hat{x};A),AX|_x+r(x)\hat{x}\big)\big],
\]
where the right hand side does not depend on the choice of $(x,\hat{x};A)\in\pi_{Q,M}^{-1}(x)$.

Clearly, for $f\in\Gamma(\xi)$, one has $f_{(X_f,r_f)}=f$
and, for $(X,r)\in\VF(M)\times\Cinf(M)$,
one has $(X_{f_{(X,r)}},r_{f_{(X,r)}})=(X,r)$.
This proves that the map defined by 
\begin{align}
\Gamma(\xi)&\to \VF(M)\times \Cinf(M)\nonumber\\
f&\mapsto (X_f,r_f)\nonumber
\end{align}
is a bijection. It is easy to see that it is actually
a $\Cinf(M)$-module homomorphism.
Since $\Cinf(M)$-modules $\Gamma(\xi)$ and $\VF(M)\times\Cinf(M)$ are
isomorphic and since $\VF(M)\times\Cinf(M)$ is 
obviously isomorphic, as a $\Cinf(M)$-module, to $\Gamma(\pi_{TM\oplus\R})$,
it follows that $\xi$ and $\pi_{TM\oplus\R}$
are isomorphic vector bundles over $M$.

We now describe the connection $\nabla^{\Rol}$ and the inner product structure
$h_k$ on $\xi$ and we determine to which objects they correspond to in the isomorphic bundle $\pi_{TM\oplus\R}$.

By Section 2.1.3 in \cite{joyce07} and the above notation, one defines
for $f\in\Gamma(\xi)$, $Y\in T|_x M$, $x\in M$
\[
\nabla^{\Rol}_Y f|_x:=\big[\big((x,\hat{x};A),\LRD(Y)|_{(x,\hat{x};A)}\ol{f}\big)\big],
\]
where $\ol{f}:Q\to\R^{n+1}$ is defined above and
$\LRD(Y)|_{(x,\hat{x};A)}\ol{f}$ is defined componentwise
(i.e., we let $\LRD(Y)|_{(x,\hat{x};A)}$ to operate separately to each of the $n+1$ component
functions of $\ol{f}$).
The definition does not depend on $(x,\hat{x};A)\in\pi_{Q,M}^{-1}(x)$
as should be evident from the above discussions.
The inner product on $\xi$, on the other hand, is defined by
\[
h_k([((x,\hat{x};A),v)],[((x,\hat{y};B),w)])
=g(X,Y)+krt,
\]
where $v=AX+r\hat{x}$, $w=BY+t\hat{y}$. It is clear that $h_k$ is well defined.

We slightly work out the expression for $\nabla^{\Rol}$.
Let $f\in\Gamma(\xi)$, $Y\in T|_x M$, $x\in M$.
Then $\ol{f}(y,\hat{y},B)=BX_{f}|_y+r_f(y)\hat{y}$ where $X_f\in\VF(M)$, $r_f\in\Cinf(M)$,
\[
\LRD(Y)|_{(x,\hat{x};A)}\ol{f}
=\LRD(Y)|_{(x,\hat{x};A)}\big((y,\hat{y};B)\mapsto BX_f|_y\big)+Y(r_f)\hat{x}+r_f(x)AY
\]
and choosing some path $\gamma$ on $M$ such that $\dot{\gamma}(0)=Y$,
then $\dot{q}_{\RDist}(\gamma,q)(0)=\LRD(Y)|_q$, where $q=(x,\hat{x};A)$ and therefore
\[
& s_{n,k}\big(\LRD(Y)|_{(x,\hat{x};A)}\big((y,\hat{y};B)\mapsto BX_f|_y\big),\hat{x}\big)
=s_{n,k}\big(\dif{t}\big|_0(A_{\RDist}(\gamma,q)(t)X_f|_{\gamma(t)}),\hat{x}\big) \\
=&\dif{t}\big|_0 s_{n,k}\big(A_{\RDist}(\gamma,q)(t)X_f|_{\gamma(t)},\hat{\gamma}_{\RDist}(\gamma,q)(t)\big)
-s_{n,k}\big(AX_f|_x,AY) \\
=&-\hat{g}_k(AX_f|_x,AY)=-g\big(X_f|_x,Y)
=s_{n,k}(-\sfrac{1}{k}g\big(X_f|_x,Y)\hat{x},\hat{x}).
\]
Therefore,
\[
\LRD(Y)|_{(x,\hat{x};A)}\big((y,\hat{y};B)\mapsto BX_f|_y\big)+\frac{1}{k}g\big(X_f|_x,Y)\hat{x}\in T|_{\hat{x}} \hat{M}_k,
\]
and we write
\[
\LRD(Y)|_{(x,\hat{x};A)}\ol{f}
=&\Big(\LRD(Y)|_{(x,\hat{x};A)}\big((y,\hat{y};B)\mapsto BX_f|_y\big)+\frac{1}{k}g\big(X_f|_x,Y)\hat{x}+r_f(x)AY\Big) \\
&+(Y(r_f)-\frac{1}{k}g(X_f|_x,Y))\hat{x}.
\]

Correspondingly,
using the isomorphism of $\xi$ and $\pi_{TM\oplus\R}$,
to the connection $\nabla^{\Rol}$ and the non-degenerate metric $h_k$ on $\xi$,
there is a connection $\nabla^{\Rol}$ and an indefinite metric $h_k$ (with the same names as the ones on $\xi$)
on $\pi_{TM\oplus\R}$ such that for $X\in\VF(M)$, $r\in\Cinf(M)$ and $Y\in T|_x M$,
\begin{align}\label{eq:rolling_connection}
\nabla^{\Rol}_Y (X,r)=&\Big(A^{-1}\big(\LRD(Y)|_{(x,\hat{x};A)}\big((y,\hat{y};B)\mapsto BX|_y\big)+\frac{1}{k}g\big(X|_x,Y)\hat{x}\big)+r(x)Y, \nonumber \\
&Y(r)-\frac{1}{k}g\big(X|_x,Y)\Big),
\end{align}
where $(x,\hat{x};A)\in Q$ is arbitrary point of $Q$ over $x$
and
\[
h_k((X,r),(Y,s))=g(X,Y)+krs,
\]
for $X,Y\in T|_x M$, $r,s\in\R$.

We will now prove the metric property (\ref{eq:rol_connection_metric})
of the connection $\nabla^\Rol$. This
will be done in the case of the bundle $\pi_{TM\oplus\R}$
but it gives the equivalent result on $\xi$.

If $(X,r),(Y,s)\in\Gamma(\pi_{TM\oplus\R})$ and $Z\in T|_x M$ then
\[
Z(h_k((X,r),(Y,s)))=&Z(g(X,Y)+krs) \\
=&g(\nabla_Z X,Y|_x)+g(X|_x,\nabla_Z Y)+kZ(r)s(x)+kr(x)Z(s).
\]
On the other hand,
\[
h_k(\nabla^{\Rol}_Z (X,r),(Y,s))
=&s_{n,k}\Big(\LRD(Z)|_{(x,\hat{x};A)}\big((y,\hat{y};B)\mapsto BX|_y\big)\\+&\frac{1}{k}g(X|_x,Z)\hat{x}+r(x)AZ,AY\Big) 
+k\big(Z(r)-\frac{1}{k}g(X|_x,Z)\big)s(x),
\]
for any $q=(x,\hat{x};A)\in\pi_{Q,M}^{-1}(x)$
and choosing a path $\gamma$ s.t. $\dot{\gamma}(0)=Z$ we get
\[
h(\nabla^{\Rol}_Z (X,r),(Y,s))
=&s_{n,k}(\dif{t}\big|_0 \big(A_{\RDist}(\gamma,q)(t)X|_{\gamma(t)}\big),AY) \\
&+r(x)g(Z,Y|_x)+\big(kZ(r)-g(X|_x,Z)\big)s(x),
\]
from which we finally get
\[
& h_k(\nabla^{\Rol}_Z (X,r),(Y,s))+h_k((X,r),\nabla^{\Rol}_Z (Y,s)) \\
=&s_{n+1}(\dif{t}\big|_0 \big(A_{\RDist}(\gamma,q)(t)X|_{\gamma(t)}\big),AY)
+s_{n+1}(AX,\dif{t}\big|_0 \big(A_{\RDist}(\gamma,q)(t)Y|_{\gamma(t)}\big)) \\
&+r(x)g(Z,Y|_x)+\big(kZ(r)-g(X|_x,Z)\big)s(x)
+s(x)g(Z,X|_x)\\
&+\big(kZ(s)-g(Y|_x,Z)\big)r(x) \\
=&\dif{t}\big|_0 s_{n+1}(A_{\RDist}(\gamma,q)(t)X|_{\gamma(t)},A_{\RDist}(\gamma,q)(t)Y|_{\gamma(t)})+kZ(r)s(x)+kr(x)Z(s) \\
=&\dif{t}\big|_0 g(X|_{\gamma(t)},Y|_{\gamma(t)})+kZ(r)s(x)+kr(x)Z(s) \\
=&g(\nabla_Z X,Y|_x)+g(X|_x,\nabla_Z Y)+kZ(r)s(x)+kr(x)Z(s),
\]
which is exactly $Z\big(h_k((X,r),(Y,s))\big)$.

Motivated by Eq. (\ref{eq:rolling_connection}), we make the following definition.
If $Y\in T|_x M$ and $X\in\VF(M)$ then define
\[
\tilde{\nabla}^{\Rol}_Y X:=A^{-1}\big(\LRD(Y)|_{(x,\hat{x};A)}\big((y,\hat{y};B)\mapsto BX|_y\big)+\frac{1}{k}g\big(X|_x,Y)\hat{x}\big),
\]
where $(x,\hat{x};A)$ is an arbitrary point on the fiber $\pi_{Q,M}^{-1}(x)$ over $x$.
It is easily seen that it is $\R$-linear in $X$ and $Y$
and, for $f\in\Cinf(M)$,
\[
\tilde{\nabla}^{\Rol}_Y (fX)=Y(f)X|_x+f(x)\tilde{\nabla}^{\Rol}_Y X,
\]
so $\tilde{\nabla}^{\Rol}$ is a connection on $M$.
Moreover, from the above computations, we see that $\tilde{\nabla}^{\Rol}$
is a metric connection with respect to $g$ i.e.,
for $X,Y\in\VF(M)$ and $Z\in T|_x M$,
\[
Z(g(X,Y))=g(\tilde{\nabla}^{\Rol}_Z X,Y)+g(X,\tilde{\nabla}^{\Rol}_Z Y).
\]

We will prove that $\tilde{\nabla}^{\Rol}=\nabla$ i.e., that $\tilde{\nabla}^{\Rol}$
is the Levi-Civita connection of $g$.
To do this, we show that the connection
$\tilde{\nabla}^{\Rol}$ is torsion-free. 

Let $X,Y\in \VF(M)$, $x\in M$.
Then taking any $q=(x,\hat{x};A)\in \pi_{Q,M}^{-1}(x)$
and any local smooth $\pi_Q$-section $\tilde{A}$ such that $\tilde{A}|_x=A$
and $\ol{\nabla} \tilde{A}|_x=0$, we compute
\[
(\tilde{\nabla}^{\Rol}_X Y-\tilde{\nabla}^{\Rol}_Y X)|_x
=&A^{-1}\big(\LRD(X)|_{(x,\hat{x};A)}((y,\hat{y};B)\mapsto BY|_y)+g(X|_x,Y|_x)\hat{x}\big) \\
&-A^{-1}\big(\LRD(Y)|_{(x,\hat{x};A)}((y,\hat{y};B)\mapsto BX|_y)+g(X|_x,Y|_x)\hat{x}\big) \\
=&A^{-1}\big(\ol{\nabla}_{(X, AX)} (\tilde{A}Y)-\ol{\nabla}_{(Y, AY)} (\tilde{A}X)\big) \\
=&(\nabla_X Y-\nabla_Y X)|_x=[X,Y]|_x.
\]

Since $\tilde{\nabla}^{\Rol}$ is a torsion-free metric connection w.r.t. $g$
on $M$, it follows by uniqueness of Levi-Civita connection that
\[
\tilde{\nabla}^{\Rol}=\nabla.
\]

Thus if $X\in\VF(M)$, $r\in\Cinf(M)$ and $Y\in T|_x M$,
\[
\nabla^{\Rol}_Y (X,r)=\Big(\nabla_Y X+r(x)Y,Y(r)-\frac{1}{k}g\big(X|_x,Y)\Big).
\]
This concludes the proof of Theorem \ref{th:rolling_connection}.

\end{proof}

\begin{remark}
Define a number $\delta_{ij}^k$ for $i,j=1,\dots,n+1$ as follows,
\[
\delta_{ij}^k:=\begin{cases}
0, &i\neq j , \\
1, & i=j=1,\dots,n, \\
k, &i=j=n+1.
\end{cases} 
\]
We say that a frame $(X_i,t_i)_{i=1}^{n+1}$ of $T|_x M\oplus\R$
is $h_k$-orthonormal if $h_k((X_i,t_i),(X_j,t_j))=\delta_{ij}^k$.
We may build the manifold $F_{\mathrm{OON}}^{h_k}(\pi_{TM\oplus\R})$
of $h_k$-orthonormal frames in the standard way.

Now we will prove that the bundle $F_{\mathrm{OON}}^{h_k}(\pi_{TM\oplus\R})$
of $h_k$-orthonormal frames of $\pi_{TM\oplus\R}$
is isomorphic to $\pi_{Q,M}$
as a bundle over $M$. The isomorphism $\Phi_k:\pi_{Q,M} \to F_{\mathrm{OON}}^{h_k}(\pi_{TM\oplus\R})$
can be described as follows.
Let $(x,\hat{x};A)\in Q$. Then there are unique $(X_i,t_i)\in T|_x M\oplus\R$, $i=1,\dots,n+1$
such that $e_i=AX_i+t_i\hat{x}$ where $e_i$, $i=1,\dots,n+1$, is the standard basis of $\R^{n+1}$.
One easily computes that
\[
&h_k((X_i,t_i),(X_j,t_j))
=g(X_i,X_j)+kt_it_j
=s_{n,k}(AX_i,AX_j)+s_{n,k}(t_i\hat{x},t_j\hat{x}) \\
=&s_{n,k}(e_i,e_j)=\delta_{ij}^k,
\]
since $s_{n,k}(AX_i,t_j\hat{x})=0$, $s_{n,k}(t_i\hat{x},AX_j)=0$.
Thus define $\Phi(x,\hat{x};A):=(X_i,t_i)_{i=1}^{n+1}$.

We will give a description the inverse map $\Phi^{-1}$.
Let $(X_i,t_i)_{i=1}^{n+1}\in F_{\mathrm{OON}}^{h_k}(\pi_{TM\oplus\R})$.
Then there are unique $a_i\in\R$ such that $\sum_{i=1}^{n+1} a_i(X_i,t_i)=(0,1)$.
We notice that $a_i=kt_i$ for all $i=1,\dots,n$ and $a_{n+1}=t_{n+1}$,
since 
$$
0=g(\sum_{i=1}^{n+1} a_i X_i,X_j)=\sum_{i=1}^{n+1} a_i(\delta_{ij}^k-kt_it_j),
$$
and because $\sum_{i=1}^{n+1} a_i t_i=1$.
Hence $k\sum_{i=1}^n t_i^2+t_{n+1}^2=1$.
Define $\hat{x}:=\sum_{i=1}^{n} (kt_i)e_i+t_{n+1}e_{n+1}$
for which $s_{n,k}(\hat{x},\hat{x})=k(k\sum_{i=1}^n t_i^2+t_{n+1}^2)=k$
i.e., $\hat{x}\in \hat{M}_k$. Moreover, it is easy to see that each $e_i-t_i\hat{x}$
is $s_{n,k}$-orthogonal to $\hat{x}$
and hence we may define $A:T|_x M\to T|_{\hat{x}} \hat{M}_k$
by requiring that $AX_i=e_i-t_i\hat{x}$, $i=1,\dots,n+1$.
It can be shown that $A$ is well defined by this formula and an orthogonal linear map
i.e., $(x,\hat{x};A)\in Q$. Also, evidently $\Phi(x,\hat{x};A)=(X_i,t_i)_{i=1}^{n+1}$.
\end{remark}

%%%%%%%%%%%%%%%%%%%%%%%%%%%%%%%%%%%%%%%%%%
\subsection{Controllability Results for Rolling Against a Non-Flat Space Form}
%%%%%%%%%%%%%%%%%%%%%%%%%%%%%%%%%%%%%%%%%%

It is now clear, thanks to Theorem \ref{th:rolling_connection}, that the controllability of the rolling problem
of a manifold $M$ against a space form $\hat{M}_k$
amounts to checking whether the connection $\nabla^{\Rol}$
of $\pi_{TM\oplus\R}$ has full holonomy or not 
i.e., whether $H^{\nabla^{\Rol}}=G(n,k)$ or not.

For the rest of the section, we assume that $k$ only takes the values $1$ and $-1$, and for notational purposes, we use the letter ''$c$'' instead of ''$k$'' and thus $c\in \{+1,-1\}$.

In Riemannian geometry, 
the reducibility of the Riemannian holonomy group
is characterized (in the complete simply connected case) by the
de Rham Theorem (see \cite{sakai91}).
We aim at giving an analog of this result w.r.t. $\nabla^{\Rol}$ 
in Theorem \ref{th:reducible_rol} below.
Before doing so, we first prove a simpler result
showing that the conclusion of Theorem \ref{th:reducible_rol} below is not trivial.

\begin{proposition}\label{pr:rol_holonomy_space_form}
Suppose that $(M,g)$ is a space form of constant curvature
equal to $c\in \{+1,-1\}$. Then the rolling connection $\nabla^\Rol$ defined by the rolling problem$(R)$ of $(M,g)$
against $(\hat{M}_c,s_{n,c})$ (i.e., we roll $(M,g)$ against itself) is reducible and, for each $x\in M$,
the irreducible subspaces of the action of the holonomy group $H^{\nabla^\Rol}|_x$ on
$T|_x M\oplus\R$ are all $1$-dimensional.
\end{proposition}

\begin{proof}
Let $(p^1,\dots,p^{n+1})$ be the canonical chart of $\R^{n+1}$
where $p^j$ is the projection onto the $j$-th factor
and write $h=h_c$ for the inner product in $TM\oplus\R$.
We will assume that the space form $M$ is the subset $\hat{M}_c$ of $\R^{n+1}$
as defined previously.
Define a vector field
$Z:=\sum_{i=1}^{n+1} p^i\pa{p^i}$
i.e., $Z$ is equal to the half of the gradient  in $(\R^{n+1},s_{n,c})$ of the function $(p^1)^2+\dots+(p^{n})^2+c(p^{n+1})^2$.
Notice that $Z$ is $s_{n,c}$-orthogonal to the submanifold $M=\hat{M}_c$ of $\R^{n+1}$
and hence $T|_x M$ is the $s_{n,c}$-orthogonal complement of $Z|_x$ for $x\in M$.
Moreover, $s_{n,c}(Z,Z)=c$.

Next we define, for $j=1,\dots,n+1$, the vector fields
\[
Y_j:=\pa{p^j}-cs_{n,c}\big(\pa{p^j},Z)Z
\]
and functions
\[
r^j(x)=cs_{n,c}(\pa{p^j},Z|_x).
\]
The restrictions of $Z,Y_j,r^j$ onto $M$ will be denoted by the same letters.
Notice that $(Y_j,r^j)$, $j=1,\dots,n+1$ are $h$-orthogonal at
each point of $M$ and hence they form a global orthogonal frame of $\pi_{TM\oplus\R}$.

Denote, as usual, by $\nabla$ the Levi-Civita connection of $(M,g)$.
Take any vector fields $X=\sum_{i=1}^{n+1} X^i\pa{p^i}\in \VF(M)$
and $U=\sum_{i=1}^{n+1} U^i\pa{p^i}\in \VF(M)$
and let $\tilde{U}$ be some extension of $U$
onto a neighbourhood of $M$ in $\R^{n+1}$
with corresponding components $\tilde{U}^i$.
Then we have for $x\in M$,
\[
\nabla_X U|_x=\tilde{U}_*(X|_x)-cs_{n,c}(\tilde{U}_*(X|_x),Z|_x)Z|_x.
\]
where we understand $\tilde{U}_*$ as a map $T\R^{n+1}\to T\R^{n+1}$
using the obvious isomorphisms $T|_X (T\R^{n+1})\to T|_x\R^{n+1}$ for each $X\in T|_x \R^{n+1}$.

Then we compute for any $x\in M$ and $X=\sum_{i=1}^{n+1} X^i\pa{p^i}\in T|_x M$
\[
Z_*(X)=\sum_{i=1}^{n+1} X^i Z_*\big(\pa{p^i}\big)=\sum_{i=1}^{n+1} X^i \pa{p^i}=X,
\]
and (notice that $(\pa{p^j}\big)_*=0$)
\[
(Y_j)_*(X)=&\sum_{i=1}^{n+1} X^i (Y_j)_*\big(\pa{p^i}\big)
=-cs_{n,c}(\pa{p^j},Z_*(X))Z|_x-cs_{n,c}(\pa{p^j},Z|_x)Z_*(X) \\
=& -cs_{n,c}(\pa{p^j},X)Z|_x-cs_{n,c}(\pa{p^j},Z|_x)X,
\]
from which we get
\[
& \nabla_X Y_j=(Y_j)_*(X|_x)-cs_{n,c}((Y_j)_*(X|_x),Z|_x)Z|_x \\
=& -cs_{n,c}(\pa{p^j},X)Z|_x-cs_{n,c}(\pa{p^j},Z|_x)X
+cs_{n,c}(\pa{p^j},X)Z|_x\\
+&s_{n,c}(\pa{p^j},Z|_x)\underbrace{s_{n,c}(X,Z|_x)}_{=0}Z|_x 
=-cs_{n,c}(\pa{p^j},Z|_x)X.
\]
Moreover,
\[
X(r^j)=&cs_{n,c}(\pa{p^j},Z_*(X))=cs_{n,c}(\pa{p^j},X) \\
s_{n,c}(X,Y_j)=&s_{n,c}\big(X,\pa{p^j}\big)-c\underbrace{s_{n,c}\big(X,Z|_x\big)}_{=0}s_{n,c}\big(Z,\pa{p^j}\big)
=s_{n,c}\big(X,\pa{p^j}\big),
\]
and thus
\[
\nabla^{\Rol}_X (Y_j,r^j)
=&\big(\nabla_X Y_j+r^j(x)X,X(r^j)-\frac{1}{c}s_{n,c}(X,Y_j|_x)\big) \\
=&\Big(-cs_{n,c}(\pa{p^j},Z|_x)X+cs_{n,c}(\pa{p^j},Z|_x)X,
cs_{n,c}(\pa{p^j},X)-cs_{n,c}\big(X,\pa{p^j}\big)\Big) \\
=&(0, 0).
\]

This means that all the $\pi_{TM\oplus\R}$-sections $(Y_j,r^j)$, $j=1,\dots,n+1$
are $\nabla^\Rol$-parallel globally.
In particular, for any $x\in M$ and loop $\gamma\in \Omega_x(M)$,
\[
\dif{t}(P^{\nabla^\Rol})_t^0(\gamma)\big((Y_j,p^j)|_{\gamma(t)}\big)
=(P^{\nabla^\Rol})_t^0(\gamma)\nabla_{\dot{\gamma}(t)} (Y_j,p^j)=0,
\]
which means that ($x=\gamma(0)$)
\[
(Y_j,p^j)|_{\gamma(t)}=(P^{\nabla^\Rol})^t_0(\gamma)(Y_j,p^j)|_x,\forall t,
\]
and hence
\[
(P^{\nabla^\Rol})^t_0(\gamma)(Y_j,p^j)|_x=(Y_j,p^j)|_{\gamma(1)}=(Y_j,p^j)|_x,
\]
i.e., that the $1$-dimensional subbundles
spanned by each $(Y_j,r^j)$ are invariant under the
holonomy group of $\nabla^{\Rol}$. Thus we have proved what we claimed.

\end{proof}

Below we will only consider the case of positive curvature $c=+1$ i.e., rolling against the unit sphere.

\begin{theorem}\label{th:reducible_rol}
Let $(M,g)$ be a complete Riemannian manifold and $(\hat{M}_1,s_{n+1})$ be the unit sphere
with the metric induced from the Euclidean metric of $\R^{n+1}$.
If the rolling connection $\nabla^{\Rol}$ (see (\ref{eq:nabla_rol_explicit}))
corresponding to rolling of $(M,g)$ against $(\hat{M}_1,s_{n+1})$
is reducible, then $(\hat{M}_1,s_{n+1})$ is a Riemannian covering of $(M,g)$ .

\end{theorem}

Recall that the reducibility of the connection $\nabla^{\Rol}$
means that its holonomy group, which is a subgroup of $G(n,c)$,
is reducible i.e., there exists two nontrivial invariant subspaces $V_1,V_2\notin\{ \{0\},\R^{n+1}\}$
of $\R^{n+1}$ which are invariant by the action of this group.

\begin{proof}
In this case we have $c=+1$ (corresponding to the sphere space form) and we will write
$h=h_1$ for the inner product on $TM\oplus\R$. 

Fix once and for all a point $x_0\in M$.
The assumption that $\nabla^{\Rol}$ is reducible means that
there are two subspaces $V_1,V_2\subset T|_{x_0}M\oplus\R$
which are nontrivial (i.e., $V_1,V_2\notin \{\{0\},T|_{x_0}M\oplus\R\}$)
and invariant by the action of the holonomy group of $\nabla^{\Rol}$
at $x_0$.
Since the holonomy group of $\nabla^\Rol$ acts
$h$-orthogonally on $T|_{x_0} M$,
it follows that $V_1\perp V_2$.

Define subbundles $\pi_{\mc{D}_j}:\mc{D}_j\to M$, $j=1,2$ of $\pi_{TM\oplus\R}$
such that for any $x\in M$ one chooses a piecewise $C^1$ curve $\gamma:[0,1]\to M$
from $x_0$ to $x$ and defines
\[
\mc{D}_j|_{x}=(P^{\nabla^{\Rol}})_0^1(\gamma)V_j,\quad j=1,2.
\]
These definitions are independent of the chosen path $\gamma$
since if $\omega$ is another such curve, then
$\omega^{-1}.\gamma\in\Omega_{x_0}(M)$ is a loop based at $x_0$
and hence by the invariance of $V_j$, $j=1,2$ under the holonomy of $\nabla^{\Rol}$,
\[
(P^{\nabla^{\Rol}})_0^1(\gamma)V_j
=(P^{\nabla^{\Rol}})_0^1(\omega)\underbrace{(P^{\nabla^{\Rol}})_0^1(\omega^{-1}.\gamma)V_j}_{=V_j}
=(P^{\nabla^{\Rol}})_0^1(\omega)V_j.
\]
Moreover, since parallel transport $(P^{\nabla^{\Rol}})_0^1(\gamma)$
is an $h$-orthogonal map,
it follows that $\mc{D}_1\perp \mc{D}_2$ w.r.t the vector bundle metric $h$.

It is a standard fact that $\mc{D}_j$, $j=1,2$,
are smooth embedded submanifolds of $TM\oplus\R$
and that the restriction of $\pi_{TM\oplus\R}$
to $\mc{D}_j$ defines a smooth subbundle $\pi_{\mc{D}_j}$ as claimed.
Moreover, it is clear that
\[
\pi_{\mc{D}_1}\oplus\pi_{\mc{D}_2}=\pi_{TM\oplus\R},
\]
and this sum is $h$-orthogonal.

We will now assume that both $\mc{D}_j$, $j=1,2$,
have dimension at least $2$. The case where one of them has dimension $=1$
can be treated in a similar fashion and will be omitted.
So we let $m+1=\dim\mc{D}_1$ where $m\geq 1$ and then $n-m=(n+1)-(m+1)=\dim\mc{D}_2\geq 2$
i.e., $1\leq m\leq n-2$.
Define for $j=1,2$
\[
\mc{D}_j^M=\pr_1(\mc{D}_j)=\big\{X\ |\ (X,r)\in \mc{D}_j\}\subset TM,
\]
and
\[
N_j=\{x\in M\ |\ (0,1)\in \mc{D}_j|_x\}\subset M.
\]
Trivially, $N_1\cap N_2=\emptyset$.
Also, $N_j$, $j=1,2$, are closed subsets of $M$
since they can be written as $N_j=\{x\in M\ |\ p_j^{\perp}(T|_x)=T|_x\}$
where $p_j^\perp:TM\oplus\R\to \mc{D}_j$ is the $h$-orthogonal projection
onto $\mc{D}_j$
and $T$ is the (smooth) constant section $x\mapsto (0,1)$ of $\pi_{TM\oplus\R}$.

We next briefly sketch the rest of proof. We will show that  $N_j$ are nonempty totally geodesic submanifolds of $M$
and, for any given $x_j\in N_j$, $j=1,2$, that $(M,g)$ is locally isometric to
the sphere
\[
S=\{(X_1,X_2)\in T|_{x_1}^{\perp} N_1\oplus T|_{x_2}^{\perp} N_2\ |\ \n{X_1}^2_g+\n{X_2}^2_g=1\},
\]
with the metric $G:=(g|_{T|_{x_1}^{\perp} N_1}\oplus g|_{T|_{x_2}^{\perp} N_2})|_S$.
Here $\perp$ denotes the orthogonal complement inside $T|_x M$ w.r.t. $g$.
Since $(S,G)$ is isometric to the Euclidean sphere $(\hat{M}_1,s_{n,1})$
this would finish the argument.
The latter is rather long and we decompose it in a sequence of ten lemmas and we start with the first one.

\begin{lemma}\label{le-redu-1}
The sets $N_j$, $j=1,2$, are non-empty.
\end{lemma}
\begin{proof} 
Note first that $N_1\cup N_2\neq M$ since otherwise
$N_1=M\backslash N_2$ would be open and closed
and similarly for $N_2$. But then if, say, $N_1\neq\emptyset$ 
we have $N_1=M$ by connectedness of $M$ i.e., the point $(0,1)\in \mc{D}_1|_x$ for all $x\in M$.
Then for all $x\in M$, $X\in\VF(M)$ one has, by the invariance of $\mc{D}_1$
by the holonomy of $\nabla^{\Rol}$ and by (\ref{eq:nabla_rol_explicit}),
\[
\mc{D}_1|_x\ni \nabla^{\Rol}_{X|_x}(0,1)=(X|_x,0),
\]
which implies that $\mc{D}_1=TM\oplus\R$, a contradiction.

Let $x'\in M\backslash (N_1\cup N_2)$ be arbitrary.
Choose a basis $(X_0,r_0),\dots (X_m,r_m)$ of $\mc{D}_1|_{x'}$.
Then at least one of the numbers $r_0,\dots,r_m$ is non-zero,
since otherwise one would have $(X_i,r_i)=(X_i,0)\perp (0,1)$ for all $i$
and thus $\mc{D}_1|_{x'}\perp (0,1)$ i.e., $(0,1)\in \mc{D}_2|_{x'}$ i.e., $x'\in N_2$ which is absurd.
We assume that it is $r_0$ which is non-zero.
By taking appropriate linear combinations of $(X_i,r_i)$, $i=0,\dots,m$
(and by Gram-Schmidt's process), one may change the basis $(X_i,r_i)$, $i=0,\dots,m$,
of $\mc{D}_1|_{x}$
so that $r_1,\dots,r_m=0$, $r_0\neq 0$
and that $(X_0,r_0),(X_1,0)\dots,(X_m,0)$ are $h$-orthonormal.
Also, $X_0,\dots,X_m$ are non-zero: for $X_1,\dots,X_m$ this is evident,
and for $X_0$ it follows from the fact that if $X_0=0$,
then $r_0=1$ and hence $x'\in N_1$, which contradicts our choice of $x'$. 

Now let $\gamma:\R\to M$ be the unit speed geodesic with $\gamma(0)=x'$, $\dot{\gamma}(0)=\frac{X_0}{\n{X_0}_g}$.
Parallel translate $(X_i,r_i)$ along $\gamma$ by $\nabla^{\Rol}$
to get $\pi_{\mc{D}_1}$-sections $(X_i(t),r_i(t))$ along $\gamma$.
In particular, from (\ref{eq:rol_parallel_geodesic}) one gets
\[
\ddot{r_i}(t)+r_i(t)=0,
\]
with $r_0(0)\neq 0$, $r_1(0)=\dots=r_m(0)=0$.
From the second equation in (\ref{eq:rol_parallel})
one obtains $\dot{r}_i(0)=g(\dot{\gamma}(0),X_i(0))=\n{X_0}_g^{-1}g(X_0,X_i)$
and thus $\dot{r}_i(0)=0$ for $i=1,\dots,m$ since $(X_i,0)$ is $h$-orthogonal to
$(X_0,r_0)$. Moreover, $\dot{r}_0(0)=\n{X_0}_g$.
Hence $r_i(t)=0$ for all $t$ and $i=1,\dots,m$ and
$r_0(t)=\n{X_0}_g\sin(t)+r_0\cos(t)$.
In particular, at $t=t_0:=\arctan(-\frac{r_0}{\n{X_0}_g})$
one has $r_i(t_0)=0$ for all $i=0,\dots,m$
which implies that $\mc{D}_1|_{\gamma(t_0)}\perp (0,1)$ i.e., $\gamma(t_0)\in N_2$.
This proves that $N_2$ is non-empty. The same argument with $\mc{D}_1$ and $\mc{D}_2$
interchanged shows that $N_1$ is non-empty.

\end{proof}
\begin{lemma}\label{le-redu-2}
For any $x\in M$ and any unit vector $u\in T|_x M$,
\begin{align}\label{eq:parallel_translate_of_0_1}
(P^{\nabla^\Rol})_0^t(\gamma_u)(0,1)=(-\sin(t)\dot{\gamma}_u(t),\cos(t)).
\end{align}
\end{lemma}
\begin{proof} Here and in what follows, $\gamma_u(t):=\exp_x(tu)$.
Write 
$$(X_0(t),r_0(t)):=(P^{\nabla^\Rol})_0^t(\gamma_u)(0,1).$$
The second equation in (\ref{eq:rol_parallel})
implies that $\dot{r}_0(0)=g(\dot{\gamma}_u(0),X_0(0))=g(u,0)=0$
and, since $r_0(0)=1$, the second equation in (\ref{eq:rol_parallel_geodesic})
gives
\[
r_0(t)=\cos(t).
\]
Notice that, for all $t\in\R$,
\[
&\nabla_{\dot{\gamma}_u(t)} (-\sin(t)\dot{\gamma}_u(t))+r_0(t)\dot{\gamma}_u(t) \\
=&\nabla_{\dot{\gamma}_u(t)} (-\sin(t))\dot{\gamma}_u(t)-\sin(t)\nabla_{\dot{\gamma}_u(t)} \dot{\gamma}_u(t)
+\cos(t)\dot{\gamma}_u(t) \\
=&-\cos(t)\dot{\gamma}_u(t)-0+\cos(t)\dot{\gamma}_u(t)=0,
\]
i.e., $-\sin(t)\dot{\gamma}_u(t)$ solves the same first order ODE as $X_0(t)$,
namely $\nabla_{\dot{\gamma}_u(t)} X_0+r_0(t)\dot{\gamma}_u(t)=0$ for all $t$
by the first equation in (\ref{eq:rol_parallel}).
Moreover, since 
$$\big(-\sin(t)\dot{\gamma}_u(t)\big)|_{t=0}=0=X_0(0),$$
it follows that $X_0(t)=-\sin(t)\dot{\gamma}_u(t)$,
which, combined with the fact that $r_0(t)=\cos(t)$ proven above, gives (\ref{eq:parallel_translate_of_0_1}).

\end{proof}

\begin{lemma}\label{le-redu-3}
The sets $N_j$, $j=1,2$, are complete, totally geodesic submanifolds of $(M,g)$
and $\mc{D}^M_j|_x=T|_x N_j$, $\forall x\in N_j$, $j=1,2$.
\end{lemma}
\begin{proof}
We show this for $N_1$. The same argument then proves the claim for $N_2$.
Let $x\in N_1$ and $u\in \mc{D}_1^M|_x$ a unit vector. Since $(0,1)\in \mc{D}_1|_x$,
Eq. (\ref{eq:parallel_translate_of_0_1}) implies that
\[
\mc{D}_1|_{\gamma_u(t)}\ni (P^{\nabla^\Rol})_0^t(\gamma_u)(0,1)=(-\sin(t)\dot{\gamma}_u(t),\cos(t))
\]
Next notice that
\[
\nabla^{\Rol}_{\dot{\gamma}_u(t)} \big(\cos(t)\dot\gamma_u(t),\sin(t)\big)
=&\big(-\sin(t)\dot{\gamma}_u(t)+\sin(t)\dot{\gamma}_u(t),\cos(t)-g(\dot{\gamma}_u(t),\cos(t)\dot\gamma_u(t))\big)\\=&(0,0),
\]
and hence, since $\big(\cos(t)\dot\gamma_u(t),\sin(t)\big)|_{t=0}=(u,0)\in \mc{D}_1|_x$
(this is so because $u\in \mc{D}_1^M|_x$,
hence there is some $r\in\R$ such that $(u,r)\in \mc{D}_1|_x$
and since $(0,1)\in \mc{D}_1|_x$ because $x\in N_1$, then $\mc{D}_1|_x\ni (u,r)-r(0,1)=(u,0)$),
we have, for all $t\in\R$,
\[
\big(\cos(t)\dot\gamma_u(t),\sin(t)\big)
=(P^{\nabla^\Rol})^t_0(u,0)\in \mc{D}_1|_{\gamma_u(t)}.
\]
Hence for all $t\in\R$,
\[
\mc{D}_1|_{\gamma_u(t)} \ni \sin(t)\big(\cos(t)\dot\gamma_u(t),\sin(t)\big)+\cos(t)\big(-\sin(t)\dot{\gamma}_u(t),\cos(t)\big)
=(0,1).
\]

This proves that any geodesic starting from a point of $N_1$ with
the initial direction from $\mc{D}^M_1$ stays in $N_1$ forever.
Hence, once it has been shown that $N_1$ is a submanifold of $M$
with tangent space $T|_{x} N_1=\mc{D}^M_1|_{x}$ for all $x\in N_1$,
then automatically $N_1$ is totally geodesic and complete.

Let $x\in N_1$.  If one takes an open neighbourhood $U$ of $x$
and local $\pi_{\mc{D}_2}$-sections $(X_{m+1},r_{m+1}),\dots,(X_n,r_n)$
which form a basis of $\mc{D}_2$ over $U$,
then it is clear that $N_1\cap U=\{x\in U\ |\ r_{m+1}(x)=\dots=r_n(x)=0\}$

Thus let $(X_{m+1},r_{m+1}),\dots,(X_n,r_n)\in \mc{D}_2|_x$ be a basis of $\mc{D}_2|_x$.
Choose $\epsilon>0$ such that $\exp_x$ is a diffeomorphism from $B_g(0,\epsilon)$
onto its image $U_{\epsilon}$ and define for $y\in U_{\epsilon}$, $j=m+1,\dots,n$,
\[
(X_j,r_j)|_y=(P^{\nabla^\Rol})^1_0\big(\tau\mapsto\exp_x(\tau\exp_x^{-1}(y))\big)(X_j,r_j).
\]
Then $(X_j,r_j)$ are local $\pi_{\mc{D}_2}$-sections
and it is clear that 
$$
N_1\cap U_{\epsilon}=\{y\in U_{\epsilon}\ |\ r_{m+1}(y)=\dots=r_n(y)=0\}.
$$
Moreover, from (\ref{eq:rol_parallel}),
\[
\nabla r_j|_x=X_j|_x,\quad j=m+1,\dots,n,
\]
which are linearly independent. Hence, by taking $\epsilon>0$ possibly smaller,
we may assume that the local vector fields $\nabla r_j$, $j=m+1,\dots,n$,
are linearly independent on $U_\epsilon$.
But this means that $N_1\cap U_{\epsilon}=\{y\in U_{\epsilon}\ |\ r_{m+1}(y)=\dots=r_n(y)=0\}$
is a smooth embedded submanifold of $U_{\epsilon}$
with tangent space
\[
T|_x N_1=&\{X\in T|_x M\ |\ g(\nabla r_j,X)=0,\ j=m+1,\dots,n\} \\
=&\{X\in T|_x M\ |\ g(X_j,X)=0,\ j=m+1,\dots,n\} \\
=&\mc{D}_1^M|_x.
\]

Since $x\in N_1$ was arbitrary, this proves that $N_1$ is indeed an embedded submanifold of $M$
and $T|_x N_1 = \mc{D}_1^M|_x$ for all $x\in N_1$.

\end{proof}

\begin{lemma}\label{le-redu-4}
Let $d_i(x):=d_g(N_i,x)$, $x\in M$. Then
in the set where $d_i$ is smooth,
\begin{align}\label{eq:distance_function_stuff}
(\nabla \cos(d_i(\cdot)),\cos(d_i(\cdot)))\in \mc{D}_i^M,
\end{align}
where $\nabla$ is the gradient w.r.t $g$.
\end{lemma}
\begin{proof}
Let $x\in M\backslash N_1$. Choose $y\in N_1$, $u\in (T|_y N_1)^\perp$
such that $\gamma_u:[0,d_i(x)]\to M$ is the minimal normal unit speed geodesic from $N_1$ to $x$.
Since $(0,1)\in \mc{D}_1|_y$ (because $y\in N_1$), it follows that
the parallel translate of $(0,1)$ along $\gamma_u$ stays in $\mc{D}_1$
which, in view of (\ref{eq:parallel_translate_of_0_1}), gives
\[
\mc{D}_1|_x\ni (P^{\nabla^\Rol})_0^{d_1(x)}(\gamma_u)(0,1)
=&\big(-\sin(d_1(x))\dot{\gamma}_u(d_1(x)),\cos(d_1(x))\big) \\
=&\big(-\sin(d_1(x))\nabla(d_1(\cdot))|_x,\cos(d_1(x))\big) \\
=&\big(\nabla\cos(d_1(\cdot))|_x,\cos(d_1(x))\big),
\]
where the last two equalities hold true if $x$ is not in the cut nor the conjugate
locus of $N_1$ (nor is $x$ in $N_1$, by assumption). Working in the complement of these points, which is a dense subset of $M$
and using a continuity argument, we may assure that the result
holds true everywhere where $d_i$ is smooth.
The same argument proves the formula (\ref{eq:distance_function_stuff}) for $d_2$ as well.

\end{proof}

\begin{lemma}\label{le-redu-5}
For every $Y\in\VF(M)$, one has
\begin{align}\label{eq:the_curvature_formula}
g\big(R(Y,\nabla d_1(\cdot))\nabla d_1(\cdot),Y\big)=g(Y,Y)-\big(\nabla_Y (d_1(\cdot))\big)^2,
\end{align}
wherever $d_1(\cdot)$ is smooth.
\end{lemma}
\begin{proof}

\newcommand{\Hess}{\mathrm{Hess}}

It is known  (see \cite{petersen06}) 
that for any $Y,Z\in \VF(M)$, $d_1(\cdot)$ satisfies a PDE
\[
-g(R(Y|_y,\nabla d_1(y))\nabla d_1(y),Z|_y)=&\Hess^2 (d_1(\cdot))(Y|_y,Z|_y)\\
+&\big(\nabla_{\nabla d_1(y)} \Hess (d_1(\cdot))\big)(Y|_y,Z|_y),
\]
for every $y\in M$ such that $d_1$ is smooth at $y$ (and this is true in a dense subset of $M$). In particular, $y\notin N_1$.
Also, since the set of points $y\in M$ where $\cos(d_1(y))=0$ or $\sin(d_1(y))=0$ is clearly
Lebesgue zero-measurable, we may assume that $\cos(d_1(y))\neq 0$ and $\sin(d_1(y))\neq 0$.

Notice that $(X_0,r_0):=(\nabla\cos(d_1(\cdot)),\cos(d_1(\cdot)))$ belongs to $\mc{D}_1$ and has $h$-norm equal to $1$.
We may choose in a neighbourhood $U$ of $y$ vector fields $X_1,\dots,X_m\in \VF(U)$
such that $(X_0,r_0),(X_1,0),\dots,(X_m,0)$ is an $h$-orthonormal basis of $\mc{D}_1$ over $U$.
Assume also that $(X_0,r_0)$ is smooth on $U$.
This implies that there are smooth one-forms $\omega^i_j$, $i,j=0,\dots,m$
defined by (set here $r_1=\dots=r_m=0$)
\[
\nabla^{\Rol}_{Y} (X_i,r_i)=\sum_{i=0}^m \omega^j_i(Y)(X_j,r_j),\quad Y\in\VF(M),
\]
or, more explicitly,
\[
\begin{cases}
\displaystyle \nabla_Y X_j+r_jY=\sum_{i=0}^m \omega^i_j(Y)X_i \\
\displaystyle Y(r_j)-g(Y,X_j)=\sum_{i=0}^m \omega^i_j(Y)r_i,
\end{cases}
\]
Since $(X_0,r_0),\dots,(X_m,r_m)$ are $h$-orthonormal, it follows that $\omega^i_j=-\omega^j_i$.
The fact that $r_1=\dots=r_m=0$ implies that
\[
-g(Y,X_j)=\omega_j^0(Y)r_0,\quad j=1,\dots,m
\]
i.e.,
\[
\omega_0^j(Y)=\frac{g(Y,X_j)}{\cos(d_1(\cdot))}.
\]
But then one has that (notice that $\omega^0_0=0$)
\[
\nabla_Y X_0+r_0 Y=\sum_{j=1}^m \omega_0^j(Y)X_j,
\]
which simplifies to
\[
& -\sin(d_1(\cdot))\nabla_Y \nabla d_1(\cdot)
-\cos(d_1(\cdot))\nabla_Y (d_1(\cdot))\nabla d_1(\cdot) \\
=&-\cos(d_1(\cdot))Y+\frac{1}{\cos(d_1(\cdot))}\sum_{j=1}^m g(X_j,Y)X_j,
\]
or
\[
\nabla_Y \nabla d_1(\cdot)=&-\cot(d_1(\cdot))\nabla_Y (d_1(\cdot))\nabla d_1(\cdot)+\cot(d_1(\cdot))Y\\
&-\frac{1}{\sin(d_1(\cdot))\cos(d_1(\cdot))}\sum_{j=1}^m g(X_j,Y)X_j.
\]

Writing $S(Y):=\nabla_Y \nabla d_1(\cdot)=\Hess (d_1(\cdot))(Y,\cdot)$, one obtains
\[
& (\nabla_{\nabla d_1(\cdot)} S)(Y)
=\nabla_{\nabla d_1(\cdot)} (S(Y))-S(\nabla_{\nabla d_1(\cdot)} Y) \\
=&\frac{1}{\sin^2 (d_1(\cdot))}\nabla_Y (d_1(\cdot))\nabla d_1(\cdot)
-\cot(d_1(\cdot))g(\nabla_{\nabla d_1(\cdot)} Y,\nabla d_1(\cdot))\nabla d_1(\cdot) \\
&-\frac{1}{\sin^2(d_1(\cdot))}Y-\Big(\frac{1}{\cos^2(d_1(\cdot))}-\frac{1}{\sin^2(d_1(\cdot))}\Big)\sum_{j=1}^m g(Y,X_j)X_j \\
&-\frac{1}{\sin(d_1(\cdot))\cos(d_1(\cdot))}\sum_{j=1}^m \big(g(Y,\nabla_{\nabla d_1(\cdot)} X_j)X_j+g(Y,X_j)\nabla_{\nabla d_1(\cdot)} X_j\big) \\
&+\cot(d_1(\cdot))\underbrace{\nabla_{\nabla_{\nabla d_1(\cdot)} Y} (d_1(\cdot))}_{=g(\nabla d_1(\cdot),\nabla_{\nabla d_1(\cdot)} Y)}\nabla d_1(\cdot),
\]
where we used that $\nabla_{\nabla d_1(\cdot)} (d_1(\cdot))=g(\nabla d_1(\cdot),\nabla d_1(\cdot))=1$.
On the other hand,
\[
&\Hess^2 (d_1(\cdot))(Y,\cdot)=S^2(Y)=S(S(Y)) \\
=&S\Big(-\cot(d_1(\cdot))\nabla_Y (d_1(\cdot))\nabla d_1(\cdot)+\cot(d_1(\cdot))Y\\
&-\frac{1}{\sin(d_1(\cdot))\cos(d_1(\cdot))}\sum_{j=1}^m g(X_j,Y)X_j\Big) \\
=&-\cot^2(d_1(\cdot))\nabla_Y (d_1(\cdot))\nabla d_1(\cdot) 
+\cot^2(d_1(\cdot))Y
-\frac{2}{\sin^2(d_1(\cdot))}\sum_{j=1}^m g(X_j,Y)X_j \\
&+\frac{1}{\sin^2(d_1(\cdot))\cos^2(d_1(\cdot))}\sum_{j=1}^m g(X_j,Y)X_j,
\]
where we used that $\nabla d_1(\cdot),X_1,\dots,X_m$ are $g$-orthonormal
(recall that  
$$
X_0=-\sin(d_1(\cdot))\nabla d_1(\cdot).)
$$
Thus, for any $Y,Z\in\VF(M)$, one has on $U$ that
\[
&-g(R(Y,\nabla d_1(\cdot))\nabla d_1(\cdot),Z) \\
=&-g(Y,Z)+\Big(\frac{1}{\sin^2 (d_1(\cdot))}-\cot^2(d_1(\cdot))\Big)\nabla_Y (d_1(\cdot))\nabla_Z (d_1(\cdot)) \\
&-\frac{1}{\sin(d_1(\cdot))\cos(d_1(\cdot))}\sum_{j=1}^m \big(g(Y,\nabla_{\nabla d_1(\cdot)} X_j)g(X_j,Z)+g(Y,X_j)g(\nabla_{\nabla d_1(\cdot)} X_j,Z)\big).
\]
We also set $Z=Y$ and hence get
\[
-g(R(Y,\nabla d_1(\cdot))\nabla d_1(\cdot),Y)
=&-g(Y,Y)+\nabla_Y (d_1(\cdot))\nabla_Y (d_1(\cdot)) \\
-&\frac{2}{\sin(d_1(\cdot))\cos(d_1(\cdot))}\sum_{j=1}^m g(Y,\nabla_{\nabla d_1(\cdot)} X_j)g(X_j,Y). \\
\]
Here
\[
&\sum_{j=1}^m g(Y,\nabla_{\nabla d_1(\cdot)} X_j)g(X_j,Y)
=-\frac{1}{\sin (d_1(\cdot))}\sum_{j=1}^m g(Y,\nabla_{X_0} X_j)g(X_j,Y) \\
=&-\frac{1}{\sin (d_1(\cdot))}\sum_{j=1}^m g(Y,\sum_{i=1}^m \omega_j^i(X_0)X_i g(X_j,Y) \\
=&-\frac{1}{\sin (d_1(\cdot))}\sum_{i,j=1}^m \underbrace{\omega_j^i(X_0)}_{(\star)_1}\underbrace{g(Y,X_i) g(X_j,Y)}_{(\star)_2}=0,
\]
where expression $(\star)_1$ is skew-symmetric in $(i,j)$ while
$(\star)_2$ is symmetric on $(i,j)$. Hence the sum is zero.
We finally obtain
\[
g(R(Y,\nabla d_1(\cdot))\nabla d_1(\cdot),Y)
=g(Y,Y)-\big(\nabla_Y (d_1(\cdot))\big)^2,
\]
as claimed.
It is clear that this formula now holds at every point of $M$ where $d_1(\cdot)$
is smooth and for any $Y\in \VF(M)$.
In particular, if $Y$ is a unit vector $g$-perpendicular to $\nabla d_1(\cdot)$
at a point $y$ of $M$, then $\nabla_Y d_1(\cdot)|_y=g(\nabla d_1(\cdot)|_y,Y|_y)=0$
and hence
\[
\mathrm{sec}(Y,d_1(\cdot))|_y=+1.
\]
\end{proof}

\begin{lemma}\label{le-redu-6}
For every $x\in N_1$, a unit vector $u\in (T|_x N_1)^\perp$
and $v\in T|_x M$ with $v\perp u$,
\begin{align}\label{eq:diff_of_exp}
\n{(\exp_x)_*|_{tu} (v)}_g=\big|\frac{\sin(t)}{t}\big|\n{v}_g,\quad t\in\R.
\end{align}
In particular, for all unit vectors $u_1,u_2\in (T|_x N_1)^\perp$ one has
\[
\exp_x(\pi u_1)=\exp_x(\pi u_2).
\]
\end{lemma}
\begin{proof}
Let $Y_{u,v}(t)=\pa{s}|_0 \exp_x(t(u+sv))$ be the Jacobi field along $\gamma_u(t)=\exp_x(tu)$
such that $Y_{u,v}(0)=0$, $\nabla_{\dot{\gamma}_u(0)} Y_{u,v}=v$.
Since $v\perp u$, it follows from the Gauss lemma (see \cite{sakai91})
that $Y_{u,v}(t)\perp \dot{\gamma}_u(t)$ for all $t$.
Moreover, the assumption $u\in (T|_{x} N_1)^\perp$
implies that $\nabla d_1(\cdot)|_{\gamma_u(t)}=\dot{\gamma}_u(t)$
and thus $\nabla_{Y_{u,v}(t)} (d_1(\cdot))=g(\dot{\gamma}_u(t),Y_{u,v}(t))=0$.

By polarization, one may write (\ref{eq:the_curvature_formula}) into the form
\[
R(Z(t),\dot{\gamma}_u(t))\dot{\gamma}_u(t)=Z(t)-g(Z(t),\dot{\gamma}_u(t))\dot{\gamma}_u(t),
\]
for any vector field $Z$ along $\gamma_u$.
In particular,
\[
\nabla_{\dot{\gamma}_u}\nabla_{\dot{\gamma}_u} Y_{u,v}
=-R(Y_{u,v},\dot{\gamma}_u)\dot{\gamma}_u
=-Y_{u,v},
\]
since $g(Y_{u,v}(t),\dot{\gamma}_u(t))=0$ for all $t$.
On the other hand, the vector field 
$Z(t)=\sin(t)P_0^t(\gamma_u)v$ satisfies along $\gamma_u$ 
\[
&\nabla_{\dot{\gamma}_u(t)}\nabla_{\dot{\gamma}_u} Z=-Z(t),\quad \forall t \\
&Z(0)=0,\ \nabla_{\dot{\gamma}_u} Z|_{t=0}=v,
\]
i.e., the same initial value problem as $Y_{u,v}$.
This implies that
\begin{align}\label{eq:Jacobi_explicit}
Y_{u,v}(t)=\sin(t)P_0^t(\gamma_u)v,
\end{align}
from which we obtain (\ref{eq:diff_of_exp}) because $Y_{u,v}(t)=t(\exp_x)_*|_{tu}(v)$.

The last claim follows from the fact that
the map $\exp_x|_S:S\to M$ where $S=\{u\in (T|_x N_1)^\perp\ |\ \n{u}=\pi\}$
is a constant map. Indeed,
if $u\in S$, $v\in T|_u S$ and we identify $v$ as an element of $T|_x M$ as usual,
then by what we have just proved (note that $u=\pi\frac{u}{\n{u}}_g$),
\[
\n{(\exp_x)_*|_{u}(v)}_g=\frac{\sin(\pi)}{\pi}\n{v}_g=0.
\]
Hence $\exp_x|_S$ has zero differential on all over $S$ which is connected,
since its dimension is $n-m-1\geq 1$ by assumption. Hence $\exp_x|_S$ is a 
constant map.

\end{proof}
\begin{lemma}\label{le-redu-7}
For every $x\in N_1$ and unit normal vector $u\in (T|_x N_1)^\perp$, the geodesic $t\mapsto \gamma_u(t)$ meets $N_2$ exactly at $t\in (\Z+\frac{1}{2})\pi$.
The same holds with the roles of $N_1$ and $N_2$ interchanged.
\end{lemma}
\begin{proof}
Let $x\in N_1$ and $u\in (T|_x N_1)^\perp$ be a unit vector normal vector to $N_1$.
For $(X,r)\in \mc{D}_1|_x$ define $(X(t),r(t))=(P^{\nabla^{\Rol}})^t_0(\gamma_u)(X,r)$.
Then by (\ref{eq:rol_parallel}), (\ref{eq:rol_parallel_geodesic}) 
we have (notice that $g(u,X)=0$ since $u\in (T|_x N_1)^\perp=(\mc{D}_1^M|_x)^\perp$
and $X\in \mc{D}^M_1|_x$)
\[
r(t)=r(0)\cos(t).
\]
Hence, $(X(t),r(t))$ is $h$-orthogonal to $(0,1)$ if and only of $r(t)=0$ i.e., $r(0)\cos(t)=0$.
This proves that $(0,1)\perp \mc{D}_1|_{\gamma_u(t)}$ i.e., $(0,1)\in \mc{D}_2|_{\gamma_u(t)}$
i.e., $\gamma_u(t)\in N_2$ if and only if $t\in (\sfrac{1}{2}+\Z)\pi$
(obviously, there is a vector $(X,r)\in \mc{D}_1|_x$ with $r\neq 0$).

\end{proof}

\begin{lemma}\label{le-redu-8}
The submanifolds $N_1$, $N_2$ are isometrically covered by Euclidean spheres
of dimensions $m$ and $n-m$, respectively,
and the fundamental groups of $N_1$ and $N_2$ are finite and have the same number of elements.
More precisely, for any $x\in N_1$ define
\[
S_{x}=\{u\in (T|_{x} N_1)^\perp\ |\ \n{u}_g=1\},
\]
equipped with the restriction of the metric $g|_x$ of $T|_x M$.
Then
\[
S_x\to N_2;\quad u\mapsto \exp_x(\frac{\pi}{2}u),
\]
is a Riemannian covering.
The same claim holds with $N_1$ and $N_2$ interchanged.
\end{lemma}

\begin{proof}
Denote by $C_1$ the component of $N_1$ containing $x$.
We will show first that $C_1=N_1$ i.e., $N_1$ is connected.

Let $y_1\in N_1$.
Since $C_1$ is a closed subset of $M$, there is a minimal geodesic
$\gamma_v$ in $M$ from $C_1$ to $y_1$
with $\dot{\gamma}_v(0)=v$ a unit vector, $x_1:=\gamma_v(0)\in C_1$
and $\gamma_v(d)=y_1$, with $d:=d_g(y_1,C_1)$.
By minimality, $v\in (T|_{x_1} C_1)^{\perp}=(T|_{x_1} N_1)^{\perp}$.
Hence by Lemma \ref{le-redu-7} the point $x_2:=\exp_{x_1}(\frac{\pi}{2}v)=\gamma_v(\frac{\pi}{2})$
belongs to $N_2$.
Since the set
\[
S_{x_2}=\{u\in (T|_{x_2} N_2)^\perp\ |\ \n{u}_g=1\}.
\]
is connected (its dimension is $m\geq 1$, by assumption that we made before),
Lemma \ref{le-redu-7} implies that $\exp_{x_2}\big(\frac{\pi}{2}S_{x_2}\big)$ is contained in a single component $C'_1$ of $N_1$.
Writing $u:=\dot{\gamma}_v(\frac{\pi}{2})$,
we have $\pm u\in S_{x_2}$
so 
\[
C_1'\ni \exp_{x_2}(-\sfrac{\pi}{2}u)=\exp_{x_2}\big(-\frac{\pi}{2}\dif{t}\big|_{\frac{\pi}{2}}\exp_{x_1}(tv)\big)
=\exp_{x_1}((\sfrac{\pi}{2}-t)v)|_{t=\frac{\pi}{2}}=x_1,
\]
and since also $x_1\in C_1$, it follow that $C_1'=C_1$.
But this implies that
\[
\gamma_v(\pi)=\exp_{x_1}(\pi v)=\exp_{x_2}\big(\frac{\pi}{2}\dif{t}\big|_{\frac{\pi}{2}}\exp_{x_1}(tv)\big)=\exp_{x_2}(\sfrac{\pi}{2}u)\in C_1.
\]
It also follows from $u\in (T|_{x_2} N_2)^\perp$ that $\dot{\gamma}_v(\pi)=\dif{t}\big|_{\frac{\pi}{2}} \exp_{x_2}(tu)\in (T|_{\gamma_v(\pi)} N_1)^\perp$.
Since $\exp_{x_2}((d-\frac{\pi}{2})u)=y_1\in N_1$,
Lemma \ref{le-redu-7} implies that $d-\frac{\pi}{2}\in (\frac{1}{2}+\Z)\pi$,
from which, since $d\geq 0$, we get $d\in \N_0\pi$, where $\N_0=\{0,1,2,\dots\}$.

By taking $x_2'=\gamma_v(\frac{3}{2}\pi)\in N_2$ we may show similarly
that $\gamma_v(2\pi)\in C_1$
and by induction we get $\gamma_v(k\pi)\in C_1$ for every $k\in\N_0$.
In particular, since $d\in \N_0 \pi$,
we get $y_1=\gamma_v(d)\in C_1$.
Since $y_1\in N_1$ was arbitrary, we get $N_1\subset C_1$ which proves the claim.

By repeating the argument with $N_1$ and $N_2$ interchanged, we see that $N_2$
is connected.

Eq. (\ref{eq:diff_of_exp}) shows that, taking $u\in S_x$ and $v\in T|_{u} S_x$,
i.e., $v\perp u$, $v\perp T|_x N_1$, 
\[
\n{\dif{t}\big|_0 \exp_x\big(\frac{\pi}{2}(u+tv)\big)}_g
=\n{(\exp_x)_*|_{\frac{\pi}{2} u} (\frac{\pi}{2} v)}_g=\n{v}_g.
\]
This shows that $u\mapsto \exp_x(\frac{\pi}{2}u)$ is a local isometry $S_x\to N_2$.
In particular, the image is open and closed in $N_2$, which is connected, hence $u\mapsto \exp_x(\frac{\pi}{2}u)$ is onto $N_2$.
According to Proposition II.1.1 in \cite{sakai91}, $u\mapsto \exp_x(\frac{\pi}{2}u)$ is a covering $S_x\to N_2$.

Similarly, for any $y\in N_2$ the map $S_y\to N_1$; $u\mapsto \exp_y(\frac{\pi}{2}u)$
is a Riemannian covering.

Finally, let us prove the statement about fundamental groups.
Fix a point $x_i\in N_i$ and write $\phi_i(u)=\exp_{x_i} (\sfrac{\pi}{2}u)$,
$i=1,2$,
for maps $\phi_1:S_{x_1}\to N_2$, $\phi_2:S_{x_2}\to N_1$.
The fundamental groups $\pi_1(N_1)$, $\pi_1(N_2)$
of $N_1$, $N_2$ are finite since their universal coverings are the (normal) spheres $S_{x_2}$, $S_{x_1}$ which are compact. Also, $\phi_1^{-1}(x_2)$ and $\phi_2^{-1}(x_1)$
are in one-to-one correspondence with $\pi_1(N_2)$ and $\pi_1(N_1)$ respectively.

Define $\Phi_1:\phi_1^{-1}(x_2)\to S_{x_2}$; $\Phi_1(u)=-\dif{t}\big|_{\frac{\pi}{2}} \exp_{x_1}(tu)\in S_{x_2}$
and similarly $\Phi_2:\phi_2^{-1}(x_1)\to S_{x_1}$; $\Phi_2(u)=-\dif{t}\big|_{\frac{\pi}{2}} \exp_{x_2}(tu)\in S_{x_1}$.
Clearly, for $u\in \phi_1^{-1}(x_2)$,
\[
\phi_2(\Phi_1(u))=\exp_{x_2}\big(-\frac{\pi}{2}\dif{t}\big|_{\frac{\pi}{2}} \exp_{x_1}(tu)\big)
=\exp_{x_1}((\sfrac{\pi}{2}-t)u)|_{t=\frac{\pi}{2}}=x_1,
\]
i.e., $\Phi_1$ maps $\phi_1^{-1}(x_2)\to \phi_2^{-1}(x_1)$.
Similarly $\Phi_2$ maps $\phi_2^{-1}(x_1)\to \phi_1^{-1}(x_2)$.
Finally, $\Phi_1$ and $\Phi_2$ are inverse maps to each other
since for $u\in \phi_1^{-1}(x_2)$,
\[
\Phi_2(\Phi_1(u))=-\dif{t}\big|_{\frac{\pi}{2}} \exp_{x_2}\big(-t\dif{s}\big|_{\frac{\pi}{2}} \exp_{x_1}(su) \big)
=-\dif{t}\big|_{\frac{\pi}{2}}  \exp_{x_1}((\sfrac{\pi}{2}-t)u)=u,
\]
and similarly $\Phi_1(\Phi_2(u))=u$ for $u\in \phi_2^{-1}(x_1)$.

\end{proof}

For the sake of simplicity, we will finish the proof of Theorem \ref{th:reducible_rol} under the assumption that $N_2$ is simply connected and indicate in Remark \ref{rem00} following the proof how to handle the general case. 

The fact that  $N_2$ is simply connected is clearly
equivalent to saying that 
\begin{align}
S_x\to N_2;\quad u\mapsto \exp_x(\frac{\pi}{2}u),\nonumber
\end{align}
defined in Lemma \ref{le-redu-8}
is an isometry for some (and hence every) $x\in N_1$.
It then follows from Lemma \ref{le-redu-8} that $N_1$ is (simply connected and) isometric to a sphere as well. 

We next get the following.
\begin{lemma}\label{le-redu-9}
Fix $x_i\in N_j$, $j=1,2$ and let
\[
S_{x_1}=\{u\in (T|_{x_1} N_1)^\perp\ |\ \n{u}_g=1\}, \qquad
S_{x_2}=\{u\in (T|_{x_2} N_2)^\perp\ |\ \n{u}_g=1\},
\]
the unit normal spheres to $N_1,N_2$ at $x_1,x_2$ respectively. Consider first the maps
\begin{align}
 f_1:S_{x_1}&\to N_2\qquad \qquad \qquad  f_2:S_{x_2}\to N_1\\
 f_1(u)&=\exp_{x_1}(\frac{\pi}{2}u)\qquad \qquad  f_2(v)=\exp_{x_2}(\frac{\pi}{2}v),\nonumber
 \end{align}
and the map $w$ which associates to $(u,v)\in S_{x_1}\times S_{x_2}$ the unique element of $S_{f_2(v)}$ such that $\exp_{f_2(v)}(\sfrac{\pi}{2}w(u,v))=f_1(u)$.
Finally let 
\begin{align}
\Psi:]0,\frac{\pi}{2}[&\times S_{x_1}\times S_{x_2}\to M\\
\Psi(t,u,v)&=\exp_{f_2(v)} (tw(u,v)).\nonumber
\end{align}
Suppose that $\tilde{S}:=]0,\frac{\pi}{2}[\times S_{x_1}\times S_{x_2}$ is endowed
with the metric $\tilde{g}$ such that
\[
\tilde{g}|_{(t,u,v)}=\diff{t}^2+\sin^2(t)g|_{T|_u S_{x_1}}+\cos^2(t)g|_{T|_v S_{x_2}}.
\] 
Then $\Psi$ is a local isometry.
\end{lemma}
\begin{proof}
We use $G$ to denote the \emph{geodesic vector field} on $TM$
i.e., for $u\in TM$ we have
\[
G|_u:=\ddot{\gamma}_u(0)=\frac{\diff^2}{\diff t^2}\big|_0 \exp_{\pi_{TM}(u)}(tu).
\]
Then the projections on $M$ by $\pi_{TM}$ of its integral curves
are geodesics.
Indeed, first we notice that
\[
G|_{\dot{\gamma}u(t)}
=\frac{\diff^2}{\diff s^2}\big|_0 \exp_{\gamma_u(t)}(s\dot{\gamma}_u(t))
=\frac{\diff^2}{\diff s^2}\big|_0 \gamma_u(t+s)=\ddot{\gamma}_u(t),
\]
and hence, if $\Gamma$ be a curve on $TM$ defined by $\Gamma(t)=\dot{\gamma}_u(t)$,
then
\[
\dot{\Gamma}(t)=\ddot{\gamma}_u(t)=G|_{\dot{\gamma}_u(t)}=G|_{\Gamma(t)},
\]
and $\Gamma(0)=u$.
Hence $\Gamma$ satisfies the same initial value problem as $t\mapsto \Phi_G(t,u)$,
which implies that
\[
\Phi_G(t,u)=\dot{\gamma}_u(t),\quad \forall t\in\R,\ u\in TM,
\]
and in particular,
\[
(\pi_{TM}\circ \Phi_G)(t,u)=\gamma_u(t),\quad \forall t\in\R,\ u\in TM.
\]

For every $u\in TM$ there is a direct sum decomposition $H_u\oplus V_u$
of $T|_u TM$ where $V_u=V|_{u}(\pi_{TM})$ is the $\pi_{TM}$-vertical fiber over $u$
and $H_u$ is defined as
\[
H_u=\{\dif{t}\big|_0 P_0^t(\gamma_X)u\ |\ X\in T|_{\pi_{TM}(u)} M\}.
\]
We write the elements of $T|_u TM$ w.r.t. this direct sum decomposition
as $(A,B)$ where $A\in H_u$, $B\in V_u$.
It can now be shown that (see \cite{sakai91} Lemma 4.3, Chapter II)
\[
((\Phi_G)_t)_*|_u(A,B)=(Z_{(A,B)}(t),\nabla_{\dot{\gamma}_u(t)} Z_{(A,B)}),
\quad (A,B)\in T|_u TM,\quad u\in TM,
\]
where $Z_{(A,B)}$ is the unique Jacobi field along geodesic $\gamma_u$ such that
$Z_{(A,B)}(0)=A$, $\nabla_{\dot{\gamma}_u(0)} Z_{(A,B)}=B$.

We are now ready to prove the claim.
First observe that
\[
\Psi(t,u,v)=(\pi_{TM}\circ \Phi_G)(t,w(u,v))
\]
and hence, for $(\pa{t},X_1,X_2)\in T\tilde{S}$,
\[
\Psi_*(\pa{t},X_1,X_2)=&(\pi_{TM})_*\big(\pa{t}\Phi_G(t,w(u,v))+((\Phi_G)_t)_*|_{w(u,v)}w_*(X_1,X_2)\big) \\
=&(\pi_{TM})_*\Big(G|_{\Phi_G(t,w(u,v))}+\big(Z_{w_*(X_1,X_2)}(t),\nabla_{\pa{t}(\pi_{TM}\circ \Phi_G)(t,w(u,v))} Z_{w_*(X_1,X_2)}\big)\Big) \\
=&\dot{\gamma}_{w(u,v)}(t)+Z_{w_*(X_1,X_2)}(t).
\]
On the other hand,
\[
(\pi_{TM}\circ\Phi_G)\big(\frac{\pi}{2},w(u,v)\big)=f_1(u),
\]
from where
\[
(f_1)_*|_u(X_1)=Z_{w_*(X_1,X_2)}\big(\frac{\pi}{2}\big).
\]
Similarly, since
\[
(\pi_{TM}\circ\Phi_G)\big(0,w(u,v)\big)=\pi_{TM}(w(u,v))=f_2(v),
\]
we get
\[
(f_2)_*|_v(X_2)=Z_{w_*(X_1,X_2)}(0).
\]

As in the proof of Lemma \ref{le-redu-6},
we see that the Jacobi equation that $Z_{w_*(X_1,X_2)}$ satisfies is
$\nabla_{\dot{\gamma}_{w(u,v)}(t)} \nabla_{\dot{\gamma}_{w(u,v)}} Z_{w_*(X_1,X_2)}=-Z_{w_*(X_1,X_2)}(t)$.
It is clear that this implies that $Z_{w_*(X_1,X_2)}$ has the form
\[
Z_{w_*(X_1,X_2)}(t)=\sin(t)P_0^t(\gamma_{w(u,v)})V_1+\cos(t)P_0^t(\gamma_{w(u,v)})V_2,
\]
for some $V_1,V_2\in T|_{f_2(u)} M$.
Now, taking into account the boundary values of $Z_{w_*(X_1,X_2)}(t)$
at $t=0$ and $t=\frac{\pi}{2}$ as derived above,
we get
\[
V_1=&P^0_{\frac{\pi}{2}}(\gamma_{w(u,v)})((f_1)_*|_u(X_1)), \\
V_2=&(f_2)_*|_v(X_2).
\]
Define
\[
& Y_1(t)=\sin(t)P_0^t(\gamma_{w(u,v)})V_1
=\sin(t)P_{\frac{\pi}{2}}^t(\gamma_{w(u,v)})((f_1)_*|_u(X_1)), \\
& Y_2(t)=\cos(t)P_0^t(\gamma_{w(u,v)})V_2
=\cos(t)P_{0}^t(\gamma_{w(u,v)})((f_2)_*|_v(X_2)),
\]
which means that
\[
Z=Y_1+Y_2.
\]
Notice that $Y_1$ and $Y_2$ are Jacobi fields along $\gamma_{w(u,v)}$.

Since $w(u,v)\in (T|_{f_2(v)}N_1)^\perp$ and $\dot{\gamma}_{w(u,v)}(\frac{\pi}{2})\in (T|_{f_1(u)} N_2)^\perp$
and
\[
Y_1(\sfrac{\pi}{2})=(f_1)_*|_u(X_1)\in T|_{f_1(u)}N_2,\quad Y_2(0)=(f_2)_*|_v(X_2)\in T|_{f_2(v)}N_1,
\]
it follows that
\[
Y_1,Y_2\perp \gamma_{w(u,v)}.
\]

We claim that moreover
\[
Y_1\perp Y_2.
\]
Indeed, since $(f_2)_*|_v(X_2)\in T|_{f_2(v)}N_1$ and $(0,1)\in \mc{D}_1|_{f_2(v)}$ (by definition of $N_1$),
we have $((f_2)_*|_v(X_2),0)\in \mc{D}_1|_{f_2(v)}$ and hence, for all $t$,
\[
(Z_1(t),r_1(t)):=(P^{\nabla^\Rol})_0^t(\gamma_{w(u,v)})((f_2)_*|_v(X_2),0)\in \mc{D}_1.
\]
On the other hand, $r_1$ satisfies $\ddot{r}_1+r_1=0$
with initial conditions $r_1(0)=0$ and $\dot{r}_1(0)=g(\dot{\gamma}_{w(u,v)}(0),Z_1(0))=g(w(u,v),(f_2)_*|_v(X_2))=0$
so $r_1(t)=0$ for all $t$. Thus $Z_1(t)$ satisfies $\nabla_{\dot{\gamma}_{w(u,v)}(t)}Z_1=0$
i.e., $Z_1(t)=P_0^t(\gamma_{w(u,v)})((f_2)_*|_v(X_2))$.
Similarly,
if $w'(u,v):=-\dif{t}\big|_{\frac{\pi}{2}} \exp_{f_2(v)}(tw(u,v))=-\dot{\gamma}_{w(u,v)}(\frac{\pi}{2})$,
\[
(Z_2(\sfrac{\pi}{2}-t),r_2(\sfrac{\pi}{2}-t)):=(P^{\nabla^\Rol})_0^t(\gamma_{w'(u,v)})((f_1)_*|_u(X_1),0)\in \mc{D}_2,
\]
and we have $r_2(\frac{\pi}{2}-t)=0$ and $Z_2(\frac{\pi}{2}-t)=P_0^t(\gamma_{w'(u,v)})((f_1)_*|_v(X_1))$
i.e.,
$Z_2(t)=P_{\frac{\pi}{2}}^t(\gamma_{w(u,v)})((f_1)_*|_v(X_1))$.
But since $\mc{D}_1\perp\mc{D}_2$ w.r.t. $h$,
we have that $(Z_1,r_1)\perp (Z_2,r_2)$ w.r.t. $h$
i.e., $g(Z_1(t),Z_2(t))=0$ for all $t$ (since $r_1(t)=r_2(t)=0$).
Thus,
\[
g(Y_1(t),Y_2(t))
=&\sin(t)\cos(t)g\big(P_{\frac{\pi}{2}}^t(\gamma_{w(u,v)})((f_1)_*|_u(X_1)),
P_{0}^{t}(\gamma_{w(u,v)})((f_2)_*|_v(X_1))\big) \\
=&\sin(t)\cos(t)g(Z_2(t),Z_1(t))=0
\]
This proves the claim, i.e., $Y_1\perp Y_2$.

Since $\n{w(u,v)}_g=1$, one has
\[
\n{\Psi_*(\pa{t},X_1,X_2)}_g^2
=&\n{\dot{\gamma}_{w(u,v)}(t)+Y_1(t)+Y_2(t)}_g^2\\
=&\n{\dot{\gamma}_{w(u,v)}(t)}_g^2+\n{Y_1(t)}_g^2+\n{Y_2(t)}_g^2 \\
=&1+\sin^2(t)^2\n{(f_1)_*|_u(X_1)}_g^2+\cos^2(t)\n{(f_2)_*|_v(X_2)}_g.
\]
Finally, since
\[
(f_1)_*|_u(X_1)=(\exp_{x_1})_*|_{\frac{\pi}{2}u}(\frac{\pi}{2}X_1)
\hbox{ and }
(f_2)_*|_v(X_2)=(\exp_{x_2})_*|_{\frac{\pi}{2}v}(\frac{\pi}{2}X_2),
\]
Eq. (\ref{eq:diff_of_exp}) implies that
\[
& \n{(f_1)_*|_u(X_1)}_g=|\sin(\frac{\pi}{2})|\n{X_1}_g=\n{X_1}_g, \\
& \n{(f_2)_*|_v(X_2)}_g=|\sin(\frac{\pi}{2})|\n{X_2}_g=\n{X_2}_g,
\]
and therefore
\[
\n{\Psi_*\big(\pa{t},X_1,X_2\big)}_g^2=&1+\sin^2(t)\n{X_2}_g^2+\cos^2(t)\n{X_1}_g^2 \\
=&\tilde{g}|_{(t,u,v)}\big(\pa{t},X_1,X_2\big),
\]
i.e., $\Psi$ is a local isometry $\tilde{S}\to M$.
\end{proof}
We next need one extra lemma.
\begin{lemma}\label{le-redu-10}
The manifold $M$ has constant constant curvature equal to $1$.
\end{lemma}
\begin{proof}
By Lemma \ref{le-redu-9}, we know that $\Psi:\tilde{S}\to M$ is a local isometry.
Now $(\tilde{S},\tilde{g})$ has constant curvature $=1$
since it is isometric to an open subset of the unit sphere (cf. \cite{petersen06} Chapter 1, Section 4.2).
The image $\Psi(\tilde{S})$ of $\Psi$ is clearly a dense subset of $M$
(indeed, $\Psi(\tilde{S})=M\backslash (N_1\cup N_2)$),
which implies that $M$ has constant curvature $=1$. 

\end{proof}

This completes the proof the theorem in the case $1\leq m\leq n-2$,
since a complete Riemannian manifold $(M,g)$ with constant curvature $=1$
is covered, in a Riemannian sense, by the unit sphere i.e., $\hat{M}_1$.
The cases $m=0$ and $m=n-1$ i.e., $\dim\mc{D}_1=1$ and $\dim\mc{D}_2=1$, respectively,
are treated exactly in the same way as above, but in this case $N_1$ is a discrete set
which might not be connected.

\end{proof}

\begin{remark}\label{rem00} 
The argument can easily be modified to deal with the case where $N_2$ (nor $N_1$)
is not simply-connected. The simplifying assumption of simply connectedness of $N_1$ and $N_2$
made previously just serves to render the map $w(\cdot,\cdot)$ globally defined on $S_{x_1}\times S_{x_2}$.
Otherwise we must define $w$ only locally and, in its definition, make a choice corresponding
to different sheets (of which there is a finite number).
\end{remark}

\begin{remark}
As mentioned in the introduction, the following issue to address is that of an irreducible holonomy group of the rolling connection $\nabla^\Rol$ i.e., for a given $x\in M$, the only non-trivial subspace of $T|_x M\oplus\R$
left invariant by parallel transport w.r.t $\nabla^\Rol$ along loops based at $x\in M$
is $T|_x M\oplus\R$.

\end{remark}

\section{Rolling Problem $(R)$ in 3D}\label{se:3D}

As mentioned in introduction, the goal of this chapter is to provide
a local structure theorem of the orbits $\mc{O}_{\RDist}(q_0)$
when $M$ and $\hat{M}$ are 3-dimensional Riemannian manifolds.
Recall that complete controllability of $\srol$ is equivalent to openess of {\it all} the orbits of $\srol$, thanks to the fact that $Q$ is connected and $\srol$ is driftless. In case there is no complete controllability, then there exists a non open orbit which is an immersed manifold in $Q$ of dimension at most eigth. Moreover, as a fiber bundle over $M$, the fiber has dimension at most five.

%%%%%%%%%%%%%%%%%%%%%%%%%%%%%%%%%%%%%%%%%%%
\subsection{Statement of the Results and Proof Strategy}
%%%%%%%%%%%%%%%%%%%%%%%%%%%%%%%%%%%%%%%%%%%

Our first theorem provides all the possibilities for the local structure of a non open orbit for the rolling $(R)$ of two 3D Riemannian manifolds. 

\begin{theorem}\label{th:3D-1}
Let $(M,g)$, $(\hat{M},\hat{g})$ be 3-dimensional Riemannian manifolds.
Assume that $\srol$ is not completely controllable and let $\mc{O}_{\RDist}(q_0)$, for some 
$q_0\in Q$, be a non open orbit. Then, there exists an open and dense subset $O$ of $\mc{O}_{\RDist}(q_0)$ so that, for every $q_1=(x_1,\hat{x}_1;A_1)\in O$, 
there are neighbourhoods $U$ of $x_1$ and $\hat{U}$ of $\hat{x}_1$
such that one of the following holds:
\begin{itemize}
\item[(a)] $(U,g|_U)$ and $(\hat{U},\hat{g}|_{\hat{U}})$ are (locally) isometric;
\item[(b)] $(U,g|_U)$ and $(\hat{U},\hat{g}|_{\hat{U}})$ are both of class $\mc{M}_{\beta}$
for some $\beta>0$;
\item[(c)] $(U,g|_U)$ and $(\hat{U},\hat{g}|_{\hat{U}})$ are both isometric
to warped products $(I\times N,h_f)$, $(I\times \hat{N},\hat{h}_{\hat{f}})$
for some open interval $I\subset\R$ and warping functions $f,\hat{f}$
which moreover satisfy either

\begin{itemize}
\item[(A)] $\displaystyle \frac{f'(t)}{f(t)}=\frac{\hat{f}'(t)}{\hat{f}(t)}$ for all $t\in I$ or
\item[(B)] there is a constant $K\in\R$
such that $\displaystyle \frac{f''(t)}{f(t)}=-K=\frac{\hat{f}''(t)}{\hat{f}(t)}$ for all $t\in I$.
\end{itemize}
\end{itemize}
\end{theorem}
For the definition and results on warped products and class $\mc{M}_{\beta}$,
we refer to Appendix \ref{app:m_beta}.

Note that we do not address here to the issue of the global structure of a non open orbit for the rolling $(R)$ of two 3D Riemmanian manifolds. For that, one would have to ''glue'' together the local information provided by Theorem \ref{th:3D-1}. Instead, our second theorem below shows, in some sense, that the list of possibilities established in Theorem \ref{th:3D-1} is complete. We will exclude the case where $\mc{O}_{\RDist}(q_0)$
is an integral manifold since in this case this orbit has dimension 3
and $(M,g)$, $(\hat{M},\hat{g})$ are locally isometric, see Corollary \ref{cor:weak_ambrose} and Remark \ref{re:weak_ambrose}.

\begin{theorem}\label{th:3D-2}
Let $(M,g)$, $(\hat{M},\hat{g})$ be 3D Riemannian manifolds,
$q_0=(x_0,\hat{x}_0;A_0)\in Q$ and suppose $\mc{O}_{\RDist}(q_0)$ is
not an integral manifold of $\RDist$.
If one writes $M^\circ:=\pi_{Q,M}(\mc{O}_{\RDist}(q_0))$, $\hat{M}^\circ:=\pi_{Q,\hat{M}}(\mc{O}_{\RDist}(q_0))$,
then the following holds true.
\begin{itemize}
\item[(a)] If $(M,g)$, $(\hat{M},\hat{g})$ are both of class $\mc{M}_{\beta}$
and if
$E_1,E_2,E_3$ and $\hat{E}_1,\hat{E}_2,\hat{E}_3$ are adapted frames
of $(M,g)$ and $(\hat{M},\hat{g})$, respectively,
then one has:
\begin{itemize}
\item[(A)] If $A_0E_2|_{x_0}=\pm \hat{E}_2|_{\hat{x}_0}$, then $\dim\mc{O}_{\RDist}(q_0)=7$;
\item[(B)] If $A_0E_2|_{x_0}\neq \pm\hat{E}_2|_{\hat{x}_0}$
and if (only) one of $(M^\circ,g)$ or $(\hat{M}^\circ,\hat{g})$ has constant curvature,
then $\dim\mc{O}_{\RDist}(q_0)=7$;
\item[(C)] Otherwise, $\dim\mc{O}_{\RDist}(q_0)=8$.
\end{itemize}
\item[(b)] If $(M,g)=(I\times N,h_f)$, $(\hat{M},\hat{g})=(\hat{I}\times\hat{N},\hat{h}_{\hat{f}})$
are warped products, where $I,\hat{I}\subset\R$ are open intervals,
and if $x_0=(r_0,y_0)$, $\hat{x}_0=(\hat{r}_0,\hat{y}_0)$, then one has
\begin{itemize}
\item[(A)] If $\displaystyle A_0\pa{r}|_{(r_0,y_0)}=\pa{r}|_{(\hat{r}_0,\hat{y}_0)}$
and if for every $t$ s.t. $(t+r_0,t+\hat{r}_0)\in I\times \hat{I}$ it holds
\[
\frac{f'(t+r_0)}{f(t+r_0)}=\frac{\hat{f}'(t+\hat{r}_0)}{\hat{f}(t+\hat{r}_0)},
\]
then $\dim\mc{O}_{\RDist}(q_0)=6$;
\item[(B)] Suppose there is a constant $K\in\R$ such that $\displaystyle \frac{f''(r)}{f(r)}=-K=\frac{\hat{f}''(\hat{r})}{\hat{f}(\hat{r})}$
for all $(r,\hat{r})\in I\times \hat{I}$.
\begin{itemize}

\item[(B1)] If $\displaystyle A_0\pa{r}|_{(r_0,y_0)}=\pm \pa{r}|_{(\hat{r}_0,\hat{y}_0)}$
and $\displaystyle \frac{f'(r_0)}{f(r_0)}=\pm \frac{\hat{f}'(\hat{r}_0)}{\hat{f}(\hat{r}_0)}$,
with $\pm$-cases correspondingly on both cases, then $\dim\mc{O}_{\RDist}(q_0)=6$;

\item[(B2)] If (only) one of $(M^\circ,g)$, $(\hat{M}^\circ,\hat{g})$ has constant curvature,
then one has $\dim\mc{O}_{\RDist}(q_0)=6$;

\item[(B3)] Otherwise $\dim\mc{O}_{\RDist}(q_0)=8$.
\end{itemize}

\end{itemize}
Here $(r,y)\mapsto \pa{r}|_{(r,y)}$, $(\hat{r},\hat{y})\mapsto \pa{r}|_{(\hat{r},\hat{y})}$,
are the vector fields in $I\times N$ and $\hat{I}\times\hat{N}$
induced by the canonical, positively oriented vector field $r\mapsto \pa{r}\big|_r$ on $I,\hat{I}\subset\R$.
\end{itemize}
\end{theorem}

From now on$(M,g)$, $(\hat{M},\hat{g})$
will be connected, oriented 3-dimensional Riemannian manifolds.
The Hodge-duals of $(M,g)$, $(\hat{M},\hat{g})$
are denoted by $\star:=\star_M$ and $\hat{\star}:=\star_{\hat{M}}$.

As a reminder, for $q_0=(x_0,\hat{x}_0;A_0)\in Q$,
we will write
\[
\pi_{\mc{O}_{\RDist}(q_0)}:=&\pi_Q|_{\mc{O}_{\RDist}(q_0)}:\mc{O}_{\RDist}(q_0)\to M\times\hat{M} \\
\pi_{\mc{O}_{\RDist}(q_0),M}:=&\pr_1\circ \pi_{\mc{O}_{\RDist}(q_0)}:\mc{O}_{\RDist}(q_0)\to M \\
\pi_{\mc{O}_{\RDist}(q_0),\hat{M}}:=&\pr_2\circ \pi_{\mc{O}_{\RDist}(q_0)}:\mc{O}_{\RDist}(q_0)\to \hat{M} \\
\]
where $\pr_1:M\times\hat{M}\to M$, $\pr_2:M\times\hat{M}\to \hat{M}$
are projections onto the first and second factor, respectively.

Before we start the arguments for Theorems \ref{th:3D-1} and \ref{th:3D-2}, we give next two propositions which are both instrumental in these arguments and also of independant interest.

\begin{proposition}\label{pr:3D-1}
Let $(M,g)$, $(\hat{M},\hat{g})$ be two Riemannian manifolds of dimension $3$,
$q_0=(x_0,\hat{x}_0;A_0)\in Q$
and suppose there
is an open subset $O$ of $\mc{O}_{\RDist}(q_0)$
and a smooth unit vector field $E_2\in \VF(\pi_{Q,M}(O))$ such that $\nu(A\star E_2)|_q$ is tangent
to $\mc{O}_{\RDist}(q_0)$ for all $q\in O$.
If the orbit  $\mc{O}_{\RDist}(q_0)$ is not open in $Q$, then
for any $x\in \pi_{Q,M}(O)$ and any unit vector fields $E_1,E_3$
such that $E_1,E_2,E_3$ is an orthonormal frame
in some neighbourhood $U$ of $x$ in $M$, then the connection table associated to 
$E_1,E_2,E_3$ is given by 
\[
\Gamma=\qmatrix{
\Gamma^1_{(2,3)} &                                0                                & -\Gamma^1_{(1,2)} \cr
\Gamma^1_{(3,1)} & \Gamma^2_{(3,1)} 				& \Gamma^3_{(3,1)} \cr
\Gamma^1_{(1,2)} &                                0                               & \Gamma^1_{(2,3)} \cr
},
\]
and
\[
& V(\Gamma^1_{(2,3)})=0,\quad V(\Gamma^1_{(1,2)})=0,\quad \forall V\in E_2|_y^\perp,\quad y\in U, \\
\]
where $\Gamma=[(\Gamma_{\star i}^j)_j^i],\quad \Gamma^j_{(i,k)}=g(\nabla_{E_j} E_i,E_k)$
and $\star 1=(2,3)$, $\star 2=(3,1)$ and $\star 3=(1,2)$
\end{proposition}

\begin{remark}
In particular, this means that the assumptions of the previous propositions
imply that the assumptions of Proposition \ref{pr:special_3D} are fulfilled.
\end{remark}

\begin{proof}
Notice that $\pi_{Q,M}(O)$ is open in $M$ since $\pi_{\mc{O}_{\RDist}(q_0),M}=\pi_{Q,M}|_{\mc{O}_{\RDist}(q_0)}$ is a submersion.
Without loss of generality, we may assume that there exist $E_1,E_3\in\VF(\pi_{Q,M}(O))$
such that $E_1,E_2,E_3$ form an orthonormal basis.

We begin by computing in $O$,
\[
[\LRD(E_2),\nu((\cdot)\star E_2)]|_q
=&-\LNSD(A(\star E_2)E_2)|_q+\nu(A\star \nabla_{E_2} E_2)|_q  \\
=&\nu(A\star (-\Gamma^2_{(1,2)} E_1+\Gamma^2_{(2,3)} E_3)|_q=:V_2|_q
\]
whence $V_2$ is a vector field in $O$ and furthermore
\[
[V_2,\nu((\cdot)\star E_2)]|_q=&\nu(A[\star (-\Gamma^2_{(1,2)} E_1+\Gamma^2_{(2,3)} E_3),\star E_2]_{\so})|_q \\
=&\nu(A\star (-\Gamma^2_{(1,2)} E_3-\Gamma^2_{(2,3)} E_1))|_q=:M_2|_q
\]
where $M_2$ is a vector field in $O$ as well.
Now if there were an open subset $O'$ of $O$
the $\pi_{\mc{O}_{\RDist}(q_0)}$-vertical vector fields
where $\nu(A\star E_2)|_q,V_2|_q,M_2|_q$ were linearly independent
for all $q\in O'$, it would follow that they form a basis of $V|_q(\pi_Q)$
for $q\in O'$ and hence $V|_q(\pi_Q)\subset T|_q(\mc{O}_{\RDist}(q_0))$ for $q\in O'$.
Then Corollary \ref{cor:vert} would imply that $\mc{O}_{\RDist}(q_0)$ is open,
which is a contradiction.
Hence in a dense subset $O_d$ of $O$ one has
that $\nu(A\star E_2)|_q,V_2|_q,M_2|_q$ are linearly dependent
which implies
\[
0=\det\qmatrix{0 & 1 & 0\cr -\Gamma^2_{(1,2)} & 0 & \Gamma^2_{(2,3)} \cr -\Gamma^2_{(2,3)} & 0 & -\Gamma^2_{(1,2)}}
=-((\Gamma^2_{(1,2)})^2+(\Gamma^2_{(2,3)})^2)
\]
i.e.
\[
\Gamma^2_{(1,2)}=0,\quad \Gamma^2_{(2,3)}=0
\]
on $\pi_{\mc{O}_{\RDist}(q_0),M}(O_d)$.
It is clear that  $\pi_{\mc{O}_{\RDist}(q_0),M}(O_d)$ is dense in $\pi_{\mc{O}_{\RDist}(q_0),M}(O)$
so the above relation holds on the open subset $\pi_{\mc{O}_{\RDist}(q_0),M}(O)$ of $M$.

Next compute
\[
[\LRD(E_1),\nu((\cdot)\star E_2)]|_q
=&\LNSD(AE_3)|_q+\nu(A\star (-\Gamma^1_{(1,2)}E_1+\Gamma^1_{(2,3)}E_3))|_q
=\LRD(E_3)|_q-L_3|_q, \\
[\LRD(E_3),\nu((\cdot)\star E_2)]|_q
=&-\LNSD(AE_1)|_q-\nu(A\star (-\Gamma^3_{(1,2)}E_1+\Gamma^3_{(2,3)}E_3))|_q\\
=&-\LRD(E_1)|_q+L_1|_q,
\]
where $L_1,L_3\in\VF(O')$ such that
\[
L_1|_q:=&\LNSD(E_1)|_q+\nu(A\star (-\Gamma^3_{(1,2)}E_1+\Gamma^3_{(2,3)}E_3))|_q, \\
L_3|_q:=&\LNSD(E_3)|_q-\nu(A\star (-\Gamma^1_{(1,2)}E_1+\Gamma^1_{(2,3)}E_3))|_q.
\]
Continuing by taking brackets of these against $\nu(A\star E_2)|_q$ gives
\[
[L_1,\nu((\cdot)\star E_2)]|_q
=&\nu(A\star (-\Gamma^1_{(1,2)}E_1+\Gamma^1_{(2,3)}E_3))|_q
+\nu(A[\star (-\Gamma^3_{(1,2)}E_1+\Gamma^3_{(2,3)}E_3),\star E_2]_{\so})|_q \\
=&\nu(A\star (-(\Gamma^1_{(1,2)}+\Gamma^3_{(2,3)})E_1+(\Gamma^1_{(2,3)}-\Gamma^3_{(1,2)})E_3)|_q=:M_{3} \\
[L_3,\nu((\cdot)\star E_2)]|_q
=&\nu(A\star (-\Gamma^3_{(1,2)}E_1+\Gamma^3_{(2,3)}E_3))|_q
-\nu(A[\star (-\Gamma^1_{(1,2)}E_1+\Gamma^1_{(2,3)}E_3),\star E_2]_{\so})|_q \\
=&\nu(A\star ((-\Gamma^3_{(1,2)}+\Gamma^1_{(2,3)})E_1+(\Gamma^3_{(2,3)}+\Gamma^1_{(1,2)})E_3)|_q=:M_{1}.
\]
Since $\nu(A\star E_2)|_q,M_{1}|_q,M_{3}|_q$ are smooth $\pi_{\mc{O}_{\RDist}(q_0)}$-vertical vector fields
defined on $O'$,
we may again resort to Corollary \ref{cor:vert} to
deduce that
\[
0=\det\qmatrix{0 & 1 & 0 \cr -(\Gamma^1_{(1,2)}+\Gamma^3_{(2,3)}) & 0 & \Gamma^1_{(2,3)}-\Gamma^3_{(1,2)} \cr
-\Gamma^3_{(1,2)}+\Gamma^1_{(2,3)} & 0 & \Gamma^3_{(2,3)}+\Gamma^1_{(1,2)}}
=-((\Gamma^1_{(1,2)}+\Gamma^3_{(2,3)})^2+(\Gamma^1_{(2,3)}-\Gamma^3_{(1,2)})^2)
\]
i.e.
\[
\Gamma^3_{(2,3)}=-\Gamma^1_{(1,2)},
\quad \Gamma^3_{(1,2)}=\Gamma^1_{(2,3)}
\]
on $\pi_{\mc{O}_{\RDist}(q_0),M}(O)$.

We will now prove that
derivatives of $\Gamma^1_{(2,3)}$ and $\Gamma^1_{(1,2)}$
in the $E_2^\perp$-directions
vanish on $\pi_{\mc{O}_{\RDist}(q_0),M}(O)$.
To reach this we first notice that
\[
L_1|_q=\LNSD(E_1)|_q-\nu(A\star (\Gamma^1_{(2,3)}E_1+\Gamma^1_{(1,2)}E_3))|_q
\]
and then compute
\[
[\LRD(E_1),L_1]|_q
=&\LNSD(\Gamma^1_{(1,2)}E_2-\Gamma^1_{(3,1)}E_3)|_q-\LRD(\nabla_{E_1} E_1)|_q \\
&+\nu(AR(E_1\wedge E_1)-\hat{R}(AE_1\wedge 0)A)|_q \\
&+\Gamma^1_{(1,2)}\LNSD(AE_2)|_q
-\nu\big(A\star (E_1(\Gamma^1_{(2,3)})E_1+E_1(\Gamma^1_{(1,2)})E_3)\big)|_q \\
&-\nu\big(A\star (\Gamma^1_{(2,3)}(\Gamma^1_{(1,2)}E_2-\Gamma^1_{(3,1)}E_3)
+\Gamma^1_{(1,2)}(\Gamma^1_{(3,1)}E_1-\Gamma^1_{(2,3)}E_2))\big)|_q \\
=&\Gamma^1_{(1,2)}\LRD(E_2)|_q-\Gamma^1_{(3,1)}L_3|_q-\LRD(\nabla_{E_1} E_1)|_q \\
&-\nu\big(A\star (E_1(\Gamma^1_{(2,3)})E_1+E_1(\Gamma^1_{(1,2)})E_3)\big)|_q.
\]
So if one define $J_1|_q:=\nu\big(A\star (E_1(\Gamma^1_{(2,3)})E_1+E_1(\Gamma^1_{(1,2)})E_3)\big)|_q$,
then $J_1$ is a smooth vector field in $O$ (tangent to $\mc{O}_{\RDist}(q_0)$)
and
\[
[J_1,\nu((\cdot)\star E_2]|_q
=\nu\big(A\star (E_1(\Gamma^1_{(2,3)})E_3-E_1(\Gamma^1_{(1,2)})E_1)\big)|_q.
\]
Since $\nu(A\star E_1)|_q,J_1|_q$ and $[J_1,\nu((\cdot)\star E_2]|_q$
are $\pi_{\mc{O}_{\RDist}(q_0)}$ vertical vector fields in $O$
and $\mc{O}_{\RDist}(q_0)$ is not open, we again deduce that
\[
E_1(\Gamma^1_{(2,3)})=0,\quad E_1(\Gamma^1_{(1,2)})=0.
\]
In a similar way,
\[
[\LRD(E_3),L_3]|_q
=&\LNSD(\Gamma^3_{(3,1)}E_1+\Gamma^1_{(1,2)}E_2)|_q-\LRD(\nabla_{E_3} E_3)|_q \\
&+\nu(AR(E_3\wedge E_3)-\hat{R}(AE_3\wedge 0)A)|_q \\
&+\Gamma^1_{(1,2)}\LNSD(AE_2)|_q
-\nu\big(A\star (-E_3(\Gamma^1_{(1,2)})E_1+E_3(\Gamma^1_{(2,3)})E_3)\big)|_q \\
&-\nu\big(A\star (-\Gamma^1_{(1,2)}(\Gamma^1_{(2,3)}E_2-\Gamma^3_{(3,1)}E_3)
+\Gamma^1_{(2,3)}(\Gamma^3_{(3,1)}E_1+\Gamma^1_{(1,2)}E_2)\big)|_q \\
=&\Gamma^3_{(3,1)}L_1|_q+\Gamma^1_{(1,2)}\LRD(E_2)|_q-\LRD(\nabla_{E_3} E_3)|_q  \\
&-\nu\big(A\star (-E_3(\Gamma^1_{(1,2)})E_1+E_3(\Gamma^1_{(2,3)})E_3)\big)|_q
\]
so $J_3|_q:=\nu\big(A\star (-E_3(\Gamma^1_{(1,2)})E_1+E_3(\Gamma^1_{(2,3)})E_3)\big)|_q$
defines a smooth vector field on $O$
and
\[
[J_3,\nu((\cdot)\star E_2)]|_q
=\nu\big(A\star (-E_3(\Gamma^1_{(1,2)})E_3-E_3(\Gamma^1_{(2,3)})E_1)\big)|_q.
\]
The same argument as before implies that
\[
E_3(\Gamma^1_{(1,2)})=0,\quad E_3(\Gamma^1_{(2,3)})=0.
\]
Since $E_2^\perp$ is spanned by $E_1,E_3$, the claim follows.
This completes the proof.
\end{proof}

We next provide a completary result to Proposition \ref{pr:3D-1} which will be fundamental for the proof of Theorem \ref{th:3D-2}.

\begin{proposition}\label{pr:3D-2}
Let $(M,g)$, $(\hat{M},\hat{g})$ be two Riemannian manifolds of dimension $3$,
$q_0=(x_0,\hat{x}_0;A_0)\in Q$.
Assume that there is an open subset $O$ of $\mc{O}_{\RDist}(q_0)$
and a smooth orthonormal local frame $E_1,E_2,E_3\in \VF(U)$
defined on the open subset $U:=\pi_{Q,M}(O)$
 of $M$ with respect to which the connection table has the form
\[
\Gamma=\qmatrix{
\Gamma^1_{(2,3)} &                                0                                & -\Gamma^1_{(1,2)} \cr
\Gamma^1_{(3,1)} & \Gamma^2_{(3,1)} 				& \Gamma^3_{(3,1)} \cr
\Gamma^1_{(1,2)} &                                0                               & \Gamma^1_{(2,3)} \cr
},
\]
and that moreover
\[
& V(\Gamma^1_{(2,3)})=0,\quad V(\Gamma^1_{(1,2)})=0,\quad \forall V\in E_2|_y^\perp,\quad y\in U.
\]
Define smooth vector fields $L_1,L_2,L_3$ on the open subset $\tilde{O}:=\pi_{Q,M}^{-1}(U)$ of $Q$ by
\[
L_1|_q=&\LNSD(E_1)|_q-\nu(A\star (\Gamma^1_{(2,3)}E_1+\Gamma^1_{(1,2)}E_3))|_q \\
L_2|_q=&\Gamma^1_{(2,3)}(x)\LNSD(E_2)|_q \\
L_3|_q=&\LNSD(E_3)|_q-\nu(A\star (-\Gamma^1_{(1,2)}E_1+\Gamma^1_{(2,3)}E_3))|_q.
\]
Then we have the following:

\begin{itemize}
\item[(i)]  If $\nu(A\star E_2)|_q$ is tangent to the orbit $\mc{O}_{\RDist}(q_0)$ at every point $q\in O$,
then the vectors
\[
\LRD(E_1)|_q,\ \LRD(E_2)|_q,\ \LRD(E_3)|_q,\ \nu(A\star E_2)|_q,\ L_1|_q,\ L_2|_q,\ L_3|_q
\]
are all tangent to $\mc{O}_{\RDist}(q_0)$ for every $q\in O$.

\item[(ii)] On $\tilde{O}$ we have the following Lie-bracket formulas
\[
[\LRD(E_1),\nu((\cdot)\star E_2)]|_q&=\LRD(E_3)|_q-L_3|_q \\
[\LRD(E_2),\nu((\cdot)\star E_2)]|_q&=0 \\
[\LRD(E_3),\nu((\cdot)\star E_2)]|_q&=-\LRD(E_1)|_q+L_1|_q \\
[L_1,\nu((\cdot)\star E_2)]|_q&=0 \\
[L_3,\nu((\cdot)\star E_2)]|_q&=0 \\
[\LRD(E_1),L_1]|_q&=-\Gamma^1_{(3,1)}L_3|_q+\Gamma^1_{(3,1)}\LRD(E_3)|_q \\
[\LRD(E_3),L_3]|_q&=\Gamma^3_{(3,1)}L_1|_q-\Gamma^3_{(3,1)}\LRD(E_1)|_q \\
[\LRD(E_2),L_1]|_q=&\Gamma^1_{(1,2)}L_1|_q-(\Gamma^1_{(2,3)}+\Gamma^2_{(3,1)})L_3|_q \\
[\LRD(E_2),L_3]|_q=&(\Gamma^1_{(2,3)}+\Gamma^2_{(3,1)})L_1|_q+\Gamma^1_{(1,2)}L_3|_q \\
[\LRD(E_3),L_1]|_q =&2L_2|_q-\Gamma^3_{(3,1)}L_3|_q-\LRD(\nabla_{E_1} E_3)|_q-\Gamma^1_{(2,3)}\LRD(E_2)|_q \\
&-(K_2+(\Gamma^1_{(2,3)})^2+(\Gamma^1_{(1,2)})^2)\nu(A\star E_2)|_q \\
[\LRD(E_1),L_3]|_q=&-2L_2|_q+\Gamma^1_{(3,1)}L_1|_q-\LRD(\nabla_{E_3} E_1)|_q+\Gamma^1_{(3,1)}\LRD(E_2)|_q \\
&+(K_2+(\Gamma^1_{(1,2)})^2+(\Gamma^1_{(2,3)})^2)\nu(A\star E_2)|_q \\
[L_3,L_1]|_q=&2L_2|_q-\Gamma^1_{(3,1)}L_1|_q-\Gamma^3_{(3,1)} L_3|_q \\
&-(K_2+(\Gamma^1_{(2,3)})^2+(\Gamma^1_{(1,2)})^2)\nu(A\star E_2)|_q.
\]
\end{itemize}
\end{proposition}

\begin{proof}
It has been already shown in the course of the proof of Proposition \ref{pr:3D-1} that the vectors $\LRD(E_1)|_q,\LRD(E_2)|_q,\LRD(E_3)|_q,\nu(A\star E_2)|_q,L_1|_q,L_3|_q$
are tangent to $\mc{O}_{\RDist}(q_0)$ for $q\in O$.
Moreover,
the first 7 brackets appearing in the statement of this corollary
are immediately established from the computations done explicitly in the proof of Proposition \ref{pr:3D-1}.

We compute,
\[
&[\LRD(E_2),L_1]|_q \\
=&-\LRD(\nabla_{E_1} E_2)|_q+\LNSD(-\Gamma^2_{(3,1)} E_3)|_q
+\nu(AR(E_2\wedge E_1)-\hat{R}(AE_2\wedge 0)A) \\
&+\LNSD(A(\star (\Gamma^1_{(2,3)}E_1+\Gamma^1_{(1,2)}E_3))E_2)|_q \\
&-\nu\big(A\star (\Gamma^1_{(2,3)}(-\Gamma^2_{(3,1)}E_3)+\Gamma^1_{(1,2)}(\Gamma^2_{(3,1)}E_1))\big)\big|_q \\
&-\nu\big(A\star (E_2(\Gamma^1_{(2,3)}) E_1+E_2(\Gamma^1_{(1,2)})E_3))|_q \\
=&-\LRD(\nabla_{E_1} E_2)|_q-\Gamma^2_{(3,1)}L_3|_q+K\nu(A\star E_3)|_q+\LNSD(A(\Gamma^1_{(2,3)}E_3-\Gamma^1_{(1,2)}E_1))|_q \\
&-\nu\big(A\star (E_2(\Gamma^1_{(2,3)}) E_1+E_2(\Gamma^1_{(1,2)})E_3))|_q \\
=&-\LRD(\nabla_{E_1} E_2)|_q-\Gamma^2_{(3,1)}L_3|_q+\LRD(\Gamma^1_{(2,3)}E_3-\Gamma^1_{(1,2)}E_1)|_q-\Gamma^1_{(2,3)}L_3+\Gamma^1_{(1,2)}L_1 \\
&+(2\Gamma^1_{(2,3)}\Gamma^1_{(1,2)}-E_2(\Gamma^1_{(2,3)}))\nu(A\star E_1)|_q \\
&+(-E_2(\Gamma^1_{(1,2)})+K-(\Gamma^1_{(2,3)})^2+(\Gamma^1_{(1,2)})^2)\nu(A\star E_3)|_q
\]
But one knows from Eq. (\ref{eq:RsE3sE3}) that $-K=-E_2(\Gamma^1_{(1,2)})+(\Gamma^1_{(1,2)})^2-(\Gamma^1_{(2,3)})^2$
and $-E_2(\Gamma^1_{(2,3)})+2\Gamma^1_{(1,2)}\Gamma^1_{(2,3)}=0$ and since also
$\nabla_{E_1} E_2=-\Gamma^1_{(1,2)} E_1+\Gamma^1_{(2,3)}E_3$, this simplifies to
\[
[\LRD(E_2),L_1]_q=-\Gamma^2_{(3,1)}L_3|_q-\Gamma^1_{(2,3)}L_3+\Gamma^1_{(1,2)}L_1
\]
Bracket $[\LRD(E_2),L_3]_q$ can be found by similar computations.

We compute $[\LRD(E_3),L_1]|_q$.
We have, recalling that $E_i(\Gamma^1_{(2,3)})=0$, $E_i(\Gamma^1_{(2,3)})=0$ for $i=1,3$,
\[
& [\LRD(E_3),L_1]|_q \\
=&-\LRD(\nabla_{E_1} E_3)|_q+\LNSD(\Gamma^1_{(2,3)}E_2-\Gamma^3_{(3,1)}E_3)|_q \\
&+\nu(AR(E_3\wedge E_1)|_q-\hat{R}(AE_3\wedge 0)|_q \\
&+\LNSD(A(\star(\Gamma^1_{(2,3)}E_1+\Gamma^1_{(1,2)}E_3))E_3)|_q \\
&-\nu\big(A\star (\Gamma^1_{(2,3)}(\Gamma^1_{(2,3)}E_2-\Gamma^3_{(3,1)}E_3)
+\Gamma^1_{(1,2)}(\Gamma^3_{(3,1)}E_1+\Gamma^1_{(1,2)}E_2)\big)\big|_q \\
=&-\LRD(\nabla_{E_1} E_3)|_q+(-K_2-(\Gamma^1_{(2,3)})^2-(\Gamma^1_{(1,2)})^2)\nu(A\star E_2)|_q \\
&-\Gamma^3_{(3,1)}L_3|_q-\Gamma^1_{(2,3)}\LRD(E_2)|_q+2L_2|_q.
\]
The computation of $[\LRD(E_1),L_3]|_q$ is similar.

We compute $[L_3,L_1]$ with the following 4 steps:
\[
[\LNSD(E_3),\LNSD(E_1)]|_q=&\LNSD(-\Gamma^1_{(3,1)}E_1+2\Gamma^1_{(2,3)}E_2-\Gamma^3_{(3,1)}E_3)|_q \\
&+\nu(AR(E_3\wedge E_1)-\hat{R}(0\wedge 0)A)|_q
\]
\[
& \big[\LNSD(E_3),\nu\big((\cdot)\star (\Gamma^1_{(2,3)}E_1+\Gamma^1_{(1,2)}E_3)\big)\big]|_q \\
=&\nu\big(A\star (\Gamma^1_{(2,3)}(\Gamma^1_{(2,3)}E_2-\Gamma^3_{(3,1)}E_3)+\Gamma^1_{(1,2)}(\Gamma^3_{(3,1)}E_1+\Gamma^1_{(1,2)}E_2))\big)\big|_q
\]
\[
& \big[\nu\big((\cdot)\star (-\Gamma^1_{(1,2)}E_1+\Gamma^1_{(2,3)}E_3)\big),\LNSD(E_1)\big]|_q \\
=&-\nu\big(A\star (-\Gamma^1_{(1,2)}(\Gamma^1_{(1,2)}E_2-\Gamma^1_{(3,1)}E_3)+\Gamma^1_{(2,3)}(\Gamma^1_{(3,1)}E_1-\Gamma^1_{(2,3)}E_2))\big)\big|_q
\]
\[
&\big[\nu\big((\cdot)\star (-\Gamma^1_{(1,2)}E_1+\Gamma^1_{(2,3)}E_3)\big),
\nu\big((\cdot)\star (\Gamma^1_{(2,3)}E_1+\Gamma^1_{(1,2)}E_3)\big)\big]|_q \\
=&\nu\big(A\big[\star (-\Gamma^1_{(1,2)}E_1+\Gamma^1_{(2,3)}E_3),
\star (\Gamma^1_{(2,3)}E_1+\Gamma^1_{(1,2)}E_3)\big]_{\so}\big)\big|_q \\
=&((\Gamma^1_{(1,2)})^2+(\Gamma^1_{(2,3)})^2)\nu(A\star E_2)|_q.
\]
Collecting these gives,
\[
[L_3,L_1]|_q
=&-\Gamma^1_{(3,1)}L_1|_q-\Gamma^3_{(3,1)} L_3|_q
+2\Gamma^1_{(2,3)}\LNSD(E_2)|_q \\
&-(K_2+(\Gamma^1_{(2,3)})^2+(\Gamma^1_{(1,2)})^2)\nu(A\star E_2)|_q.
\]

\end{proof}

%%%%%%%%%%%%%%%%%%%%%%%%%%%%%%%%%%%%%%%%%
\subsection{Proof of Theorem \ref{th:3D-1}}
%%%%%%%%%%%%%%%%%%%%%%%%%%%%%%%%%%%%%%%%%

In this subsection we will prove Theorem \ref{th:3D-1}. We therefore fix for the rest of the paragraph a non open orbit $\mc{O}_{\RDist}(q_0)$, for some $q_0\in Q$. By Proposition \ref{pr:R_orbit_bundle}, one has that  $\dim\mc{O}_{\RDist}(q_0)<9=\dim Q$ and, by Corollary \ref{cor:vcomm}, one knows that the rank of $\Rol_q$ is less than or equal to two, for every $q\in \mc{O}_{\RDist}(q_0)$. 

For $j=0,1,2$, we define $O_j$ as the set of points of $\mc{O}_{\RDist}(q_0)$
where $\rank\Rol_q$ is locally equal to $j$, i.e.,
\[
O_j=\{q=(x,\hat{x};A)\in \mc{O}_{\RDist}(q_0)\ |\ & \textrm{there exists an open neighbourhood $O$} \\
& \textrm{of $q$ in $\mc{O}_{\RDist}(q_0)$
such that}\ \rank\Rol_{q'}=j,\ \forall q'\in O\}.
\]
Notice that the union of the $O_j$'s, when $j=0,1,2$, is an open and dense in $\mc{O}_{\RDist}(q_0)$ since each $O_j$ is open in $\mc{O}_{\RDist}(q_0)$ (but might be empty). 

Clearly, Item $(a)$ in Theorem \ref{th:3D-1} describes the local structures of $(M,g)$
and $(\hat{M},\hat{g})$ at a point $q\in O_0$. The rest of the argument consists in addressing the same issue, first for $q\in O_2$ and then $q\in O_1$.

%%%%%%%%%%%%%%%%%%%%%%%%%%%%%%%%%%%%%%%%%%%%%
\subsubsection{Local Structures for the Manifolds Around $q\in O_2$}\label{sss:local-O2}
%%%%%%%%%%%%%%%%%%%%%%%%%%%%%%%%%%%%%%%%%%%%%

Throughout the subsection, we assume, if not otherwise stated, that the orbit 
$\mc{O}_{\RDist}(q_0)$ is not open in $Q$
(i.e., $\dim\mc{O}_{\RDist}(q_0)<9=\dim Q$)
and, in the statements involving $O_2$, the latter is non empty. Note that $O_2$ could have been defined simply as the set of points of $\mc{O}_{\RDist}(q_0)$ where $\rank\Rol_q$ is equal to $2$.

\begin{proposition}\label{pr:Rol2:good_basis}
Let $q_0=(x_0,\hat{x}_0;A_0)\in Q$
and assume that the orbit $\mc{O}_{\RDist}(q_0)$ is not open in $Q$.
Then for every $q=(x,\hat{x};A)\in O_2$
there exist an orthonormal pair $X_A,Y_A\in T|_x M$
such that if $Z_A:=\star (X_A\wedge Y_A)$
then $X_A,Y_A,Z_A$ is a positively oriented orthonormal pair
with respect to which $R$ and $\widetilde{\Rol}$ may be written as
\[
R(X_A\wedge Y_A)&=\qmatrix{
0 & K(x) & 0 \cr
-K(x) & 0 & 0 \cr
0 & 0 & 0
},\quad \star R(X_A\wedge Y_A)=\qmatrix{0 \cr 0 \cr -K(x)} \\
R(Y_A\wedge Z_A)&=\qmatrix{
0 & 0 & 0 \cr
0 & 0 & K_1(x) \cr
0 & -K_1(x) & 0
},\quad \star R(Y_A\wedge Z_A)=\qmatrix{-K_1(x) \cr 0 \cr 0} \\
R(Z_A\wedge X_A)&=\qmatrix{
0 & 0 & -K_2(x) \cr
0 & 0 & 0 \cr
K_2(x) & 0 & 0
},\quad \star R(Z_A\wedge X_A)=\qmatrix{0 \cr -K_2(x) \cr 0} \\
\widetilde{\Rol}_q(X_A\wedge Y_A)&=0, \\
\widetilde{\Rol}_q(Y_A\wedge Z_A)&=\qmatrix{
0 & 0 & -\alpha(q) \cr
0 & 0 & K^{\Rol}_1(q) \cr
\alpha(q) & -K^{\Rol}_1(q) & 0
},\quad \star \widetilde{\Rol}_q(Y_A\wedge Z_A)=\qmatrix{-K^{\Rol}_1(q) \cr -\alpha(q) \cr 0}, \\
\widetilde{\Rol}_q(Z_A\wedge X_A)&=\qmatrix{
0 & 0 & -K^{\Rol}_2(q) \cr
0 & 0 & \alpha(q) \cr
K^{\Rol}_2(q) & -\alpha (q)& 0
},\quad \star \widetilde{\Rol}_q(Z_A\wedge X_A)=\qmatrix{-\alpha(q) \cr -K^{\Rol}_2(q) \cr 0},
\]
where $K,K_1,K_2:M\to\R$.
\end{proposition}

Consequently, we see that
with respect to the orthonormal oriented basis $X_A,Y_A,Z_A$ of $T|_{\hat{x}}\hat{M}$
given by the proposition, we have
\begin{align}\label{eq:Rol2:good_basis-hatR}
\star A^{\ol{T}}\hat{R}(AX_A\wedge AY_A)A&=\qmatrix{0 \cr 0 \cr -K(x)} \nonumber \\
\star A^{\ol{T}}\hat{R}(AY_A\wedge AZ_A)A&=\qmatrix{-K_1(x)+K_1^\Rol(q) \cr \alpha(q) \cr 0} \nonumber \\
\star A^{\ol{T}}\hat{R}(AZ_A\wedge AX_A)A&=\qmatrix{\alpha(q) \cr -K_2(x)+K_2^\Rol(q) \cr 0}.
\end{align}

Before pursuing to the proof,
we want to fix some additional notations and so we make the following remark.

\begin{remark}
By the last proposition $-K_1(x),-K_2(x),-K(x)$
are the eigenvalues of $R|_{x}$ corresponding to eigenvectors
$\star X_A,\star Y_A,\star Z_A$ given by Proposition \ref{pr:Rol2:good_basis},
for $q=(x,\hat{x};A)\in O_2$.

Recall that $Q(M,\hat{M})\to Q(\hat{M},M)$
such that $q=(x,\hat{x};A)\mapsto \hat{q}=(\hat{x},x;A^{\ol{T}})$
is an diffeomorphism which maps $\RDist$ to $\widehat{\RDist}$,
where the latter is the rolling distribution on $Q(\hat{M},M)$.
Hence this map maps $\RDist$-orbits $\mc{O}_{\RDist}(q)$
to $\widehat{\RDist}$-orbits $\mc{O}_{\widehat{\RDist}}(\hat{q})$, for all $q\in Q$.
So the rolling problem (R) is completely symmetric w.r.t.
the changing of the roles of $(M,g)$ and $(\hat{M},\hat{g})$.

Hence Proposition \ref{pr:Rol2:good_basis} gives, when the roles of $(M,g)$, $(\hat{M},\hat{g})$
are changed, for every $q=(x,\hat{x};A)\in O_2$
vectors $\hat{X}_A,\hat{Y}_A,\hat{Z}_A\in T|_{\hat{x}} \hat{M}$
such that $\widetilde{\Rol}_q((A^{\ol{T}}X_A)\wedge (A^{\ol{T}}\hat{Y}_A))=0$
and that $\hat{\star} \hat{X}_A,\hat{\star} \hat{Z}_A,\hat{\star} \hat{Z}_A$
are eigenbasis of $\hat{R}|_{\hat{x}}$ with eigenvalues which we call
$-\hat{K}_1(\hat{x}),-\hat{K}_2(\hat{x}),-\hat{K}(\hat{x})$, respectively.

The condition $\widetilde{\Rol}_q(X_A\wedge Y_A)=0$ implies
that $K(x)=\hat{K}(\hat{x})$
for every $q=(x,\hat{x};A)\in O_2$
and also that $AZ_A=\hat{Z}_A$, since
$\star (X_A\wedge Y_A)=Z_A$, $\hat{\star} (\hat{X}_A\wedge \hat{Y}_A)=\hat{Z}_A$.
\end{remark}

We divide the proof of Proposition \ref{pr:Rol2:good_basis} into several lemmas.

\begin{lemma}\label{le:Rol2:good_basis}
For every $q=(x,\hat{x};A)\in O_2$
and any orthonormal pair (which exists) $X_A,Y_A\in T|_x M$ such that
$\Rol(X_A\wedge Y_A)=0$
and $X_A,Y_A,Z_A:=\star (X_A\wedge Y_A)$
is an oriented orthonormal basis of $T|_x M$,
one has with respect to the basis $X_A,Y_A,Z_A$,
\[
R(X_A\wedge Y_A)&=\qmatrix{
0 & K_A & \eta_A \cr
-K_A & 0 & -\beta_A \cr
-\eta_A & \beta_A & 0
},\quad \star R(X_A\wedge Y_A)=\qmatrix{\beta_A \cr \eta_A \cr -K_A} \\
R(Y_A\wedge Z_A)&=\qmatrix{
0 & -\beta_A & \xi_A \cr
\beta_A & 0 & K^1_A \cr
-\xi_A & -K^1_A & 0
},\quad \star R(Y_A\wedge Z_A)=\qmatrix{-K^1_A \cr \xi_A \cr \beta_A} \\
R(Z_A\wedge X_A)&=\qmatrix{
0 & -\eta_A & -K^2_A \cr
\eta_A & 0 & -\xi_A \cr
K^2_A & \xi_A & 0
},\quad \star R(Z_A\wedge X_A)=\qmatrix{\xi_A \cr -K^2_A \cr \eta_A} \\
\widetilde{\Rol}_q(X_A\wedge Y_A)&=0, \\
\widetilde{\Rol}_q(Y_A\wedge Z_A)&=\qmatrix{
0 & 0 & -\alpha \cr
0 & 0 & K^{\Rol}_1 \cr
\alpha & -K^{\Rol}_1 & 0
},\quad \star \widetilde{\Rol}_q(Y_A\wedge Z_A)=\qmatrix{-K^{\Rol}_1 \cr -\alpha \cr 0}, \\
\widetilde{\Rol}_q(Z_A\wedge X_A)&=\qmatrix{
0 & 0 & -K^{\Rol}_2 \cr
0 & 0 & \alpha \cr
K^{\Rol}_2 & -\alpha & 0
},\quad \star \widetilde{\Rol}_q(Z_A\wedge X_A)=\qmatrix{-\alpha \cr -K^{\Rol}_2 \cr 0} .
\]
Here $\eta_A,\beta_A,\xi_A,\alpha,K_1^\Rol,K_2^\Rol$
depend \emph{a priori} on the basis $X_A,Y_A,Z_A$ and on the point $q$.

Moreover, the choice of the above quantities can be made \emph{locally smoothly} on $O_2$
i.e. every $q\in O_2$ admits an open neighbourhood $O_2'$ in $O_2$
such that the selection of these quantities can be performed smoothly on $O_2'$.
\end{lemma}

\begin{proof}
Since $\rank \Rol_q=2<3$ for $q\in O_2$,
it follows that there is a unit vector $\omega_A\in \wedge^2 T|_x M$
such that $\Rol_q(\omega_A)=0$.
But in dimension 3, as mentioned in Appendix,
one then has an orthonormal pair $X_A,Y_A\in T|_x M$ such that $\omega_A=X_A\wedge Y_A$.
Moreover, the assignments $q\mapsto \omega_A,X_A,Y_A$ can be made locally smoothly.
Then defining $Z_A=\star (X_A\wedge Y_A)$,
the fact that $\widetilde{\Rol}_q$ is a symmetric map implies that
\[
g(\widetilde{\Rol}_q(Y_A\wedge Z_A),X_A\wedge Y_A)
=&g(\widetilde{\Rol}_q(X_A\wedge Y_A),Y_A\wedge Z_A)=0 \\
g(\widetilde{\Rol}_q(Z_A\wedge X_A),X_A\wedge Y_A)
=&g(\widetilde{\Rol}_q(X_A\wedge Y_A),Z_A\wedge X_A)=0.
\]
This finishes the proof.
\end{proof}

As a consequence of the previous result and
because, for $X,Y\in T|_x M$, one gets
$$
A^{\ol{T}}\hat{R}(AX\wedge AY)A=R(X\wedge Y)-\widetilde{\Rol}_q(X\wedge Y)
$$ then 
we have that, w.r.t. the oriented orthonormal basis $AX_A,AY_Z,AZ_A$ of $T|_{\hat{x}}\hat{M}$,
\begin{align}
\hat{\star} A^{\ol{T}}\hat{R}(AX_A\wedge AY_A)A&=\qmatrix{\beta_A \cr \eta_A \cr -K_A} \nonumber \\
\hat{\star} A^{\ol{T}}\hat{R}(AY_A\wedge AZ_A)A&=\qmatrix{-K^1_A+K_1^\Rol \cr \xi_A+\alpha \cr \beta_A} \nonumber \\
\hat{\star} A^{\ol{T}}\hat{R}(AZ_A\wedge AX_A)A&=\qmatrix{\xi_A+\alpha \cr -K^2_A+K_2^\Rol \cr \eta_A}.
\end{align}

Notice that the assumption that $\rank \Rol_q=2$ on $O_2$
is equivalent to
the fact that for any choice of $X_A,Y_A,Z_A$ as above, $\widetilde{\Rol}_q(Y_A\wedge Z_A)$
and $\widetilde{\Rol}_q(Z_A\wedge X_A)$ are linearly independent
for every $q=(x,\hat{x};A)\in O_2$ i.e.
\begin{align}\label{eq:rank2}
K_1^{\Rol}(q)K_2^{\Rol}(q)-\alpha(q)^2\neq 0.
\end{align}

We will now show that,
with any (non-unique) choice of $X_A,Y_A$
as in Lemma \ref{le:Rol2:good_basis},
one has that $\eta_A=\beta_A=0$.

\begin{lemma}\label{le:Rol2:bundle}
Choose any $X_A,Y_A,Z_A=\star (X_A\wedge Y_A)$ as in Lemma \ref{le:Rol2:good_basis}.
Then for every $q=(x,\hat{x};A)\in O_2$ and any vector fields $X,Y,Z,W\in\VF(M)$
one has
\begin{align}\label{eq:Rol2:span_Vorbit}
\big[\nu(\Rol(X\wedge Y)(\cdot)),\nu(\Rol(Z\wedge W)(\cdot))\big]\big|_q\in \nu(\spn\{\star X_A,\star Y_A\})|_q\subset T|_q\mc{O}_{\RDist}(q_0).
\end{align}
Moreover, $\pi_{Q}|_{O_2}$ is an submersion (onto an open subset of $M\times\hat{M}$),
$\dim V|_q(\mc{O}_{\RDist}(q_0))=2$ for all $q\in O_2$ and $\dim\mc{O}_{\RDist}(q_0)=8$.
\end{lemma}

\begin{proof}
First notice that by Lemma \ref{le:Rol2:good_basis}
\[
\qmatrix{
\Rol_q(\star X_A) \cr
\Rol_q(\star Y_A)}
=\qmatrix{-K_1^\Rol & -\alpha \cr -\alpha & -K_2^\Rol}\qmatrix{\star X_A\cr \star Y_A}
\]
for $q=(x,\hat{x};A)\in O_2$
and since the determinant of the matrix on the right hand side is, at $q\in O_2$,
$K_1^{\Rol}(q)K_2^{\Rol}(q)-\alpha(q)^2\neq 0$,
as noticed in (\ref{eq:rank2}) above, it follows that
\[
\star X_A,\star Y_A\in \spn\{\Rol_q(\star X_A),\Rol_q(\star Y_A)\}.
\]
Next, from Proposition \ref{pr:R_comm_L2} we know that, for every $q=(x,\hat{x};A)\in \mc{O}_{\RDist}(q_0)$ and every $Z,W\in T|_x M$
$$
\nu(\Rol_q(Z\wedge W))|_q\in V|_q(\pi_{\mc{O}_{\RDist}(q_0)})\subset T|_q\mc{O}_{\RDist}(q_0).
$$
Hence, $\nu(\Rol_q(\star X_A)),\nu(\Rol_q(\star Y_A))\in V|_q(\pi_{\mc{O}_{\RDist}(q_0)})$ for every $q\in O_2$
and it follows from the above that
\begin{align}\label{eq:Rol2:X_A-Y_A-tangent}
\nu(A\star X_A),\nu(A\star Y_A)\in V|_q(\pi_{\mc{O}_{\RDist}(q_0)}),
\end{align}
for all $q=(x,\hat{x};A)\in O_2$.

We claim that $\pi_{\mc{O}_{\RDist}(q_0)}|_{O_2}$ is a submersion (onto an open subset of $M\times\hat{M}$).
Indeed, for any vector field $W\in\VF(M)$ one has $\LRD(W)|_q\in T|_q\mc{O}_{\RDist}(q_0)$
for $q=(x,\hat{x};A)\in O_2$
and since the assignments $q\mapsto X_A,Y_A$ can be made locally smoothly,
then also $[\LRD(W),\nu(A\star X_A)]|_q \in T|_q\mc{O}_{\RDist}(q_0)$.
But then Proposition \ref{pr:NS_comm_HV} implies that
\[
& (\pi_{\mc{O}_{\RDist}(q_0)})_*([\LRD(W),\nu(A\star X_A)]|_q) \\
=&(\pi_{\mc{O}_{\RDist}(q_0)})_*\big(-\LNSD(A(\star X_A)W)|_q+\nu(A\star \LRD(W)|_qX_{(\cdot)})|_q\big) \\
=&(0,-A(\star X_A)W)
\]
where we wrote $X_{(\cdot)}$ as for the map $q\mapsto X_A$.
Similarly,
\[
(\pi_{\mc{O}_{\RDist}(q_0)})_*([\LRD(W),\nu(A\star Y_A)]|_q)=(0,-A(\star Y_A)W).
\]
This shows that for all $q=(x,\hat{x};A)\in O_2$ and $Z,W\in T|_x M$, we have
\[
(0,-A(\star X_A)W),(0,-A(\star Y_A)W)\in (\pi_{\mc{O}_{\RDist}(q_0)})_* T|_q\mc{O}_{\RDist}(q_0)\subset T|_x M\times T|_{\hat{x}}\hat{M}.
\]
Because $\star X_A,\star Y_A$ are linearly independent,
this implies that
\[
\{0\}\times T|_{\hat{x}}\hat{M}\subset (\pi_{\mc{O}_{\RDist}(q_0)})_*T|_q\mc{O}_{\RDist}(q_0).
\]
Finally, because $\LRD(W)|_q\in T|_q\mc{O}_{\RDist}(q_0)$ for any $q=(x,\hat{x};A)\in\mc{O}_{\RDist}(q_0)$ and any $W\in T|_x M$,
and $(\pi_{\mc{O}_{\RDist}(q_0)})_*\LRD(W)|_q=(W,AW)$,
one also has
\[
(W,0)=(W,AW)-(0,AW)\in (\pi_{\mc{O}_{\RDist}(q_0)})_*T|_q\mc{O}_{\RDist}(q_0),
\]
which implies
\[
T|_{x}M\times\{0\}\subset (\pi_{\mc{O}_{\RDist}(q_0)})_*T|_q\mc{O}_{\RDist}(q_0).
\]
This proves that $\pi_{\mc{O}_{\RDist}(q_0)}|_{O_2}=\pi_Q|_{O_2}$ is indeed a submersion.

Because $O_2$ is not open in $Q$ (otherwise $\mc{O}_{\RDist}(q_0)$ would be an open subset of $Q$),
it follows that $\dim O_2\leq 8$
and since $\pi_{\mc{O}_{\RDist}(q_0)}|_{O_2}$ has rank $6$, being a submersion,
we deduce that for all $q\in O_2$,
\[
\dim V|_q(\pi_{\mc{O}_{\RDist}(q_0)})=\dim O_2-6\leq 2.
\]
But because of (\ref{eq:Rol2:X_A-Y_A-tangent}) we see that 
$\dim V|_q(\pi_{\mc{O}_{\RDist}(q_0)})\geq 2$ i.e.
\[
\dim V|_q(\pi_{\mc{O}_{\RDist}(q_0)})=2,
\]
which shows that $\dim O_2=8$, hence $\dim\mc{O}_{\RDist}(q_0)=8$
and
\[
\spn\{\nu(A\star X_A)|_q,\nu(A\star Y_A)|_q\}=V|_q(\mc{O}_{\RDist}(q_0)),
\quad \forall q=(x,\hat{x};A)\in O_2.
\]

To conclude the proof, it is enough to notice that since
for any $X,Y,Z,W\in\VF(M)$,
$\nu(\Rol(X\wedge Y)(A))|_q,\nu(\Rol(Z\wedge W)(A))|_q\in V|_q(\mc{O}_{\RDist}(q_0))$,
then
\[
[\nu(\Rol(X\wedge Y)(\cdot)),\nu(\Rol(Z\wedge W)(\cdot))]|_q\in V|_q(\mc{O}_{\RDist}(q_0)).
\]
\end{proof}

\begin{lemma}\label{le:Rol2:ab_zero}
If one chooses any $X_A,Y_A,Z_A=\star (X_A\wedge Y_A)$ as in Lemma \ref{le:Rol2:good_basis},
then
\[
\eta_A=\beta_A=0,
\quad \forall q=(x,\hat{x};A)\in O_2.
\]
\end{lemma}

\begin{proof}
Fix $q=(x,\hat{x};A)\in O_2$.
Choosing in Corollary \ref{cor:vcomm} $X,Y\in\VF(M)$
such that $X|_x=X_A$, $Y|_x=Y_A$,
we get, since $\Rol_q(X_A\wedge Y_A)=0$,
\[
& \nu|_q^{-1}\big[\nu(\Rol(X\wedge Y)(\cdot)),\nu(\Rol(Z\wedge W)(\cdot))\big]\big|_q \nonumber \\
=&A\big[R(X_A\wedge Y_A),R(Z|_x\wedge W|_x)\big]_{\so}-\big[\hat{R}(AX_A\wedge AY_A),\hat{R}(AZ|_x\wedge AW|_x)\big]_{\so}A \nonumber \\
&+\hat{R}(AX_A, A\widetilde{\Rol}_q(Z|_x\wedge W|_x)Y_A)A+\hat{R}(A\widetilde{\Rol}_q(Z|_x\wedge W|_x)X_A,AY_A)A \nonumber.
\]

We compute the right hand side of this formula in
in two special cases (a)-(b) below.

\begin{itemize}
\item[(a)] Take $Z,W\in\VF(M)$ such that $Z|_x=Y_A$, $W|_x=Z_A$.
\end{itemize}

In this case,
computing the matrices in the basis $\star X_A,\star Y_A,\star Z_A$,
\[
& A^{\ol{T}}\nu|_q^{-1}\big[\nu(\Rol(X\wedge Y)(\cdot)),\nu(\Rol(Z\wedge W)(\cdot))\big]\big|_q \\
=&\big[R(X_A\wedge Y_A),R(Y_A\wedge Z_A)\big]_{\so}-A^{\ol{T}}\big[\hat{R}(AX_A\wedge AY_A),\hat{R}(AY_A\wedge AZ_A)\big]_{\so}A \\
&+A^{\ol{T}}\hat{R}(AX_A,A\widetilde{\Rol}_q(Y_A\wedge Z_A)Y_A)A
+A^{\ol{T}}\hat{R}(A\widetilde{\Rol}_q(Y_A\wedge Z_A)X_A,AY_A)A \\
=&\qmatrix{\beta_A \cr \eta_A \cr -K_A }\wedge\qmatrix{-K_A^1 \cr \xi_A \cr \beta_A}
-\qmatrix{\beta_A \cr \eta_A \cr -K_A }\wedge \qmatrix{-K_A^1+K_1^\Rol \cr \xi_A+\alpha \cr \beta_A} \\
&+A^{\ol{T}}\hat{R}(AX_A,-K_1^\Rol AZ_A)A
+A^{\ol{T}}\hat{R}(\alpha AZ_A,AY_A)A \\
=&-\qmatrix{\beta_A \cr \eta_A \cr -K_A }\wedge \qmatrix{K_1^\Rol \cr \alpha \cr 0}
+K_1^\Rol\qmatrix{\xi_A+\alpha \cr -K^2_A+K_2^\Rol \cr \eta_A}
-\alpha\qmatrix{-K^1_A+K_1^\Rol \cr \xi_A+\alpha \cr \beta_A} \\
=&\qmatrix{
-\alpha K_A+K_1^\Rol(\xi_A+\alpha)-\alpha (-K_A^1+K_1^\Rol) \cr
K_A K_1^\Rol+K_1^\Rol(-K_A^2+K_2^\Rol)-\alpha (\xi_A+\alpha) \cr
-\alpha \beta_A+K_1^\Rol \alpha_A+K_1^\Rol \alpha_A-\alpha \beta_A
}=\qmatrix{\star \cr \star \cr 2(K_1^\Rol \eta_A-\alpha \beta_A)}.
\]
By Lemma \ref{le:Rol2:bundle} the right hand side should
belong to the span of $\star X_A,\star Y_A$ which implies
\begin{align}\label{pr:ab_zero:eq1}
K_1^\Rol \eta_A-\alpha \beta_A=0.
\end{align}

\begin{itemize}
\item[(b)] Take $Z,W\in\VF(M)$ such that $Z|_x=Z_A$, $W|_x=X_A$.
\end{itemize}
Again, computing w.r.t. the basis $\star X_A,\star Y_A,\star Z_A$, yields
\[
& A^{\ol{T}}\nu|_q^{-1}\big[\nu(\Rol(X,Y)(\cdot)),\nu(\Rol(Z,W)(\cdot))\big]\big|_q \nonumber \\
=&\big[R(X_A,Y_A),R(Z_A,X_A)\big]_{\so}-A^{\ol{T}}\big[\hat{R}(AX_A,AY_A),\hat{R}(AZ_A,AX_A)\big]_{\so}A \\
&+A^{\ol{T}}\hat{R}(AX_A,A\widetilde{\Rol}_q(Z_A,X_A)Y_A)A+A^{\ol{T}}\hat{R}(A\widetilde{\Rol}_q(Z_A,X_A)X_A,AY_A)A \\
=&\Big(\qmatrix{\beta_A\cr \alpha_A\cr -K_A}\wedge \qmatrix{\xi_A \cr -K_A^2 \cr \eta_A}
-\qmatrix{\beta_A\cr \eta_A\cr -K_A}\wedge \qmatrix{\xi_A+\alpha \cr -K^2_A+K_2^\Rol \cr \eta_A} \\
&+A^{\ol{T}}\hat{R}(AX_A,-\alpha AZ_A)A+A^{\ol{T}}\hat{R}(K_2^\Rol AZ_A,AY_A)A \\
=&
-\qmatrix{\beta_A\cr \eta_A\cr -K_A}\wedge \qmatrix{\alpha \cr K_2^\Rol \cr 0}
+\alpha \qmatrix{\xi_A+\alpha \cr -K^2_A+K_2^\Rol \cr \eta_A}-K_2^\Rol \qmatrix{-K^1_A+K_1^\Rol \cr \xi_A+\alpha \cr \beta_A} \\
=&\qmatrix{-K_AK_2^\Rol+\alpha (\xi_A+\alpha)-K_2^\Rol (-K_A^1+K_1^\Rol) \cr
\alpha K_A+\alpha (-K_A^2+K_2^\Rol)-K_2^\Rol (\xi_A+\alpha) \cr
-\beta_A K_2^\Rol + \alpha \eta_A + \alpha \eta_A - K_2^\Rol \beta_A
}=\qmatrix{\star \cr \star \cr 2(\alpha \eta_A-\beta_A K_2^\Rol)}
\]
Since the right hand side
belongs to the span of $\star X_A,\star Y_A$, by Lemma \ref{le:Rol2:bundle}, we obtain
\begin{align}\label{pr:ab_zero:eq2}
\alpha \eta_A-K_2^\Rol\beta_A=0.
\end{align}

Combining Equations (\ref{pr:ab_zero:eq1}) and (\ref{pr:ab_zero:eq2})
we get
\[
\qmatrix{K_1^\Rol & \alpha \cr \alpha & K_2^\Rol}\qmatrix{\eta_A\cr \beta_A}=\qmatrix{0 \cr 0}.
\]
But by Eq. (\ref{eq:rank2}) the determinant of the $2\times 2$-matrix on the left hand side
does not vanish, which implies that $\eta_A=\beta_A=0$.
The proof is finished.
\end{proof}

\begin{lemma}\label{le:Rol2:xi_zero}
For every $q=(x,\hat{x};A)\in O_2$
there are orthonormal $X_A,Y_A\in T|_x M$
such that $X_A,Y_A,Z_A=\star (X_A\wedge Y_A)$ is an oriented orthonormal basis of $T|_x M$
with respect to which in Lemma \ref{le:Rol2:good_basis}
one has
\[
\eta_A=\beta_A=\xi_A=0
\]
i.e. $\star X_A,\star Y_A,\star Z_A$ are eigenvectors of $R|_x$.
\end{lemma}

\begin{proof}
Fix $q=(x,\hat{x};A)\in O_2$, choose any $X_A,Y_A,Z_A=\star (X_A\wedge Y_A)$ as in Lemma \ref{le:Rol2:good_basis}
and suppose $\xi_A\neq 0$ (otherwise we are done).
Notice that by Lemma \ref{le:Rol2:ab_zero}, we have $\eta_A=\beta_A=0$
which means that $\star Z_A$ is an eigenvector of $R|_x$.

We let $t\in\R$,
\[
\qmatrix{X_A(t)\cr Y_A(t)}:=\qmatrix{\cos(t) & \sin(t) \cr -\sin(t) & \cos(t)}\qmatrix{X_A \cr Y_A}.
\]
Then clearly $Z_A(t):=\star (X_A(t)\wedge Y_A(t))=\star (X_A\wedge Y_A)=Z_A$,
and $X_A(t),Y_A(t),Z_A(t)$ is an orthonormal positively oriented basis of $T|_x M$.
Since 
$$
\Rol_q(\star Z_A(t))=\Rol_q(\star Z_A)=0,
$$
Lemma \ref{le:Rol2:ab_zero} implies that $\eta_A(t),\beta_A(t)=0$
if one writes $\eta_A(t),\beta_A(t),\xi_A(t)$ for the
coefficients of matrices in Lemma \ref{le:Rol2:good_basis}
w.r.t $X_A(t),Y_A(t),Z_A(t)$.
Our goal is to show that $\xi_A(t)=0$ for some $t\in\R$.

First of all $\star Z_A(t)=\star Z_A$ is a unit eigenvector of $R|_x$ which does not depend on $t$.
On the other hand, $R|_x$ is a symmetric map $\wedge^2 T|_x M\to\wedge^2 T|_x M$,
so it has two orthogonal unit eigenvectors, say, $u_1,u_2$
in $(\star Z_A)^\perp=\star (Z_A^\perp)$. Thus $u_1,u_2,\star Z_A$ forms an orthonormal basis of $\wedge^2 T|_x M$,
which we may assume to be oriented (otherwise swap $u_1,u_2$).
But then $\spn\{u_1,u_2\}=\star Z_A^\perp=\spn\{\star X_A,\star Y_A\}$
so there is definitely $t_0\in\R$
such that $\star X_A(t_0)=u_1$, $\star Y_A(t_0)=u_2$ (in this order, by the assumption on orientation of $u_1,u_2,\star Z_A$ and $X_A,Y_A,Z_A$).
Since $R|_x(\star X_A(t_0))=-K_1\star X_A(t_0)$, $R|_x(\star Y_A(t_0))=-K_2\star Y_A(t_0)$,
we have $\xi_A(t_0)=0$ as well as $\eta_A(t_0)=\beta_A(t_0)=0$
This allows us to conclude.
\end{proof}

\begin{remark}
Notice that the choice of $Z_A$ can be made locally smoothly on $O_2$
but, at this stage of the argument, it is not clear that one can choose $X_A,Y_A$,
with $\xi_A=0$, locally smoothly on $O_2$.
However, it will be the case cf. Corollary \ref{cor:Rol2:loc_smooth}.
\end{remark}

We now aim to prove, roughly speaking, that the eigenvalue $-K$
has to be double for both spaces $(M,g)$, $(\hat{M},\hat{g})$
if neither one of them has constant curvature.

\begin{lemma}\label{le:double_eig_K}
If the eigenspace at $x_1\in \pi_{\mc{O}_{\RDist}(q_0),M}(O_2)$ corresponding to the eigenvalue $-K(x_1)$ of the
curvature operator $R$
has multiplicity 1,
then $(\hat{M},\hat{g})$ has constant curvature $K(x_1)$
on the open set $\pi_{\mc{O}_{\RDist}(q_0),\hat{M}}(\pi_{\mc{O}_{\RDist}(q_0),M}^{-1}(x_1))$
of $\hat{M}$.

The claim also holds with the roles of $(M,g)$ and $(\hat{M},\hat{g})$
interchanged. 
\end{lemma}

\begin{proof}
So suppose that at $x_1\in \pi_{\mc{O}_{\RDist}(q_0),M}(O_2)$
the eigenspace of $R|_{x_1}$ corresponding to the eigenvalue $-K(x_1)$
has multiplicity $1$.
By continuity then,
the $-K(\cdot)$-eigenspace of $R$ is of multiplicity 1
on an open neighbourhood $U$ of $x_1$.
Since this eigenspace depends smoothly on a point of $M$,
we may choose, taking $U$ smaller around $x_1$ if needed, positively oriented orthonormal
smooth vector fields $\tilde{X},\tilde{Y},\tilde{Z}$ on $U$
such that $\star\tilde{Z}=\tilde{X}\wedge \tilde{Y}$ spans the $-K(\cdot)$-eigenspace
of $R$ at each point of $U$.

Taking arbitrary $q'=(x',\hat{x}';A')\in (\pi_{\mc{O}_{\RDist}(q_0),M})^{-1}(U)\cap O_2$
and letting $X_{A'},Y_{A'},Z_{A'}$
be the vectors provided by Theorem \ref{le:Rol2:good_basis} at $q$,
we have that the $-K(x')$-eigenspace of $R|_{x'}$ is also spanned by $X_{A'}\wedge Y_{A'}$.
By the orthonormality and orientability, $X_{A'}\wedge Y_{A'}=\tilde{X}|_{x'}\wedge \tilde{Y}|_{x'}$
from which $\tilde{Z}|_{x'}=Z_{A'}$
and $\Rol(\tilde{X}|_{x'}\wedge \tilde{Y}|_{x'})(A')=\Rol(X_{A'}\wedge Y_{A'})(A')=0$.

Now fix, for a moment, $q=(x,\hat{x};A)\in (\pi_{\mc{O}_{\RDist}(q_0),M})^{-1}(U)\cap O_2$.
By replacing $\tilde{X}$ by $\cos(t)\tilde{X}+\sin(t)\tilde{Y}$
and $\tilde{Y}$ by $-\sin(t)\tilde{X}+\cos(t)\tilde{Y}$
on $U$ for a certain constant $t=t_x\in\R$,
we may assume that $\tilde{X}|_x=X_A$, $\tilde{Y}|_x=Y_A$.

Since, as we just proved, for all $(x',\hat{x}';A')\in (\pi_{\mc{O}_{\RDist}(q_0),M})^{-1}(U)\cap O_2$, one has 
$$
\Rol(\tilde{X}|_{x'}\wedge \tilde{Y}|_{x'})(A')=0,
$$ then
the vector field $\nu(\Rol(\tilde{X}\wedge \tilde{Y})(\cdot))\in \VF(\pi_{\mc{O}_{\RDist}(q_0),M})$
vanishes identically
i.e. $\nu(\Rol(\tilde{X}\wedge \tilde{Y})(\cdot))=0$ on $(\pi_{\mc{O}_{\RDist}(q_0),M})^{-1}(U)\cap O_2$.

Therefore, the computation in part (a) of the proof of Lemma \ref{le:Rol2:ab_zero}
(replace $X\to \tilde{X}$, $Y\to\tilde{Y}$, $Z\to\tilde{Y}$, $W\to\tilde{Z}$ there;
recall also that $\xi_A=0$ by the choice of $X_A,Y_A,Z_A$)
gives,
by noticing also that here $K_A=K(x)$, $K_A^1=K_1(x)$ and $K_A^2=K_2(x)$,
\[
0=&A^{\ol{T}}\nu|_q^{-1}\big[\nu(\Rol(\tilde{X},\tilde{Y})(\cdot)),\nu(\Rol(\tilde{Y},\tilde{Z})(\cdot))\big]\big|_q \\
=&\qmatrix{
-\alpha K_A+\alpha K_1^\Rol-\alpha (-K_A^1+K_1^\Rol) \cr
K_A K_1^\Rol+K_1^\Rol(-K_A^2+K_2^\Rol)-\alpha^2 \cr
0}
=&\qmatrix{
\alpha (-K+K_1) \cr
K_1^\Rol (K-K_2+K_2^\Rol)-\alpha^2 \cr
0}.
\]
Similarly, the computation in part (b) of the proof of Lemma \ref{le:Rol2:ab_zero}
(now replace $X\to \tilde{X}$, $Y\to\tilde{Y}$, $Z\to\tilde{Z}$, $W\to\tilde{X}$ there)
gives,
\[
0=&A^{\ol{T}}\nu|_q^{-1}\big[\nu(\Rol(\tilde{X},\tilde{Y})(\cdot)),\nu(\Rol(\tilde{Z},\tilde{X})(\cdot))\big]\big|_q \nonumber \\
=&\qmatrix{
-K_AK_2^\Rol+\alpha^2-K_2^\Rol (-K_A^1+K_1^\Rol) \cr
\alpha K_A+\alpha (-K_A^2+K_2^\Rol)-K_2^\Rol \alpha \cr
0
}
=\qmatrix{
K_2^\Rol(-K+K_1-K_1^\Rol)+\alpha^2 \cr
\alpha(K-K_2) \cr
0
}.
\]

By assumption, $-K(\cdot)$ is an eigenvalue of $R$ distinct from the other eigenvalues
$-K_1(\cdot)$, $-K_2(\cdot)$ on $U$,
and hence we must have $\alpha(q)=0$.
Since $0\neq K_1^\Rol(q) K_2^\Rol(q)-\alpha(q)^2=K_1^\Rol(q) K_2^\Rol(q)$,
we have $K_1^\Rol(q)\neq 0$ and $K_2^\Rol(q)\neq 0$
and hence $K(x)-K_1(x)+K_1^\Rol(q)=0$ and $K(x)-K_2(x)+K_2^\Rol(q)=0$
for $q=(x,\hat{x};A)\in (\pi_{\mc{O}_{\RDist}(q_0),M})^{-1}(U)\cap O_2$.

Since $q=(x,\hat{x};A)\in (\pi_{\mc{O}_{\RDist}(q_0),M})^{-1}(U)\cap O_2$
was arbitrary, we have proven that
\[
\alpha(q)&=0, \\
-K_1(x)+K_1^\Rol(q)&=-K(x), \\
-K_2(x)+K_2^\Rol(q)&=-K(x),
\]
for all $q=(x,\hat{x};A)\in (\pi_{\mc{O}_{\RDist}(q_0),M})^{-1}(U)\cap O_2$.

Looking at (\ref{eq:Rol2:good_basis-hatR}) now reveals that
for every $q=(x,\hat{x};A)\in (\pi_{\mc{O}_{\RDist}(q_0),M})^{-1}(U)\cap O_2$,
the three 2-vectors
$AX_A\wedge AY_A$,
$AY_A\wedge AZ_A$ and
$AZ_A\wedge AX_A$
are mutually orthonormal eigenvectors
of $\hat{R}|_{\hat{x}}$
corresponding all to the eigenvalue $-K(x)$
which means that $(\hat{M},\hat{g})$ has constant curvature $-K(x)$ at $\hat{x}$.

In particular, since $x_1\in U$, the Riemannian space $(\hat{M},\hat{g})$ has
constant curvature $-K(x_1)$ at all points
$\hat{x}_1\in \pi_{\mc{O}_{\RDist}(q_0),\hat{M}}\big((\pi_{\mc{O}_{\RDist}(q_0),M})^{-1}(x_1)\cap O_2\big)$.

Finally, we argue that $\hat{S}:= \pi_{\mc{O}_{\RDist}(q_0),\hat{M}}\big((\pi_{\mc{O}_{\RDist}(q_0),M})^{-1}(x_1)\cap O_2\big)$
is an open subset of $\hat{M}$.
It is enough to show that $\pi_{Q,\hat{M}}|_{\hat{O}_{x_1}}:\hat{O}_{x_1}\to \hat{M}$
is a submersion
where $\hat{O}_{x_1}:=(\pi_{\mc{O}_{\RDist}(q_0),M})^{-1}(x_1)\cap O_2$
is a submanifold of $O_2$.

To begin with, recall that $\pi_{Q}|_{O_2}$ is an submersion from onto an open subset of $M\times\hat{M}$
by Lemma \ref{le:Rol2:bundle}.
Let $q\in \hat{O}_{x_1}$ and write $q=(x_1,\hat{x};A)$.
Choose any frame $\hat{X}_1,\hat{X}_2,\hat{X}_3$ of $T|_{\hat{x}} \hat{M}$.
Then there are $\hat{W}_i\in T|_q(\mc{O}_{\RDist}(q_0))$, $i=1,2,3$,
such that $(\pi_Q)_*(\hat{W}_i)=(0,\hat{X}_i)$.
In particular, $(\pi_{Q,M})_*(\hat{W}_i)=0$, so $\hat{W}_i\in V|_q(\pi_{\mc{O}_{\RDist}(q_0),M})$.
But since $T|_q \hat{O}_{x_1}=V|_q(\pi_{\mc{O}_{\RDist}(q_0),M})$,
we have $\hat{W}_i\in T|_q \hat{O}_{x_1}$
and thus $\hat{X}_i=(\pi_{Q,\hat{M}})_*\hat{W}_i\in \IM(\pi_{Q,\hat{M}}|_{\hat{O}_{x_1}})_*$,
which proves the claim and finishes the proof.
\end{proof}

\begin{remark}
It is actually obvious that the eigenvalue $-K(\cdot)$ of $R$ of $(M,g)$
is constant, equal to $K(x_1)$ say,
in a some neighbourhood of $x_1$ in $M$,
if $-K(x_1)$ were a single eigenvalue of $R|_{x_1}$.

Even more is true: One could show, even without questioning
whether $-K(\cdot)$ is a single eigenvalue for $R$ and/or $\hat{R}$ or not,
that on $\pi_{Q,M}(O_2)$ and $\pi_{Q,\hat{M}}(O_2)$
this eigenvalue is actually locally constant (i.e. the function $K(\cdot)$ is locally constant).
This is fact will be observed e.g. in Lemma \ref{le:Rol2:special_3D} below.
\end{remark}

\begin{lemma}\label{le:Rol2:no_simple_eig}
The following hold:
\begin{itemize}
\item[(1)] For any $q_1=(x_1,\hat{x}_1;A_1)\in O_2$,
the space $(\hat{M},\hat{g})$ cannot have constant curvature at $\hat{x}_1$.

\item[(2)] There does not exist a $q_1=(x_1,\hat{x}_1;A_1)\in O_2$
such that $-K(x_1)$ is a single eigenvalue of $R|_{x_1}$.
\end{itemize}
This also holds with the roles of $(M,g)$ and $(\hat{M},\hat{g})$ interchanged.
\end{lemma}

\begin{proof}
(1) Suppose $(\hat{M},\hat{g})$ has a constant curvature $\hat{K}$ at $\hat{x}_1$.
Let $E_1,E_2,E_3$ be an oriented orthonormal frame on a neighbourhood $U$ of $x_1$
such that $\star E_1|_{x_1},\star E_2|_{x_1},\star E_3|_{x_1}$ are
eigenvectors of $R$ at $x_1$
with eigenvalues $-K_1(x_1)$, $-K_2(x_1)$, $-K(x_1)$, respectively,
where these eigenvalues are as in Proposition \ref{pr:Rol2:good_basis}.
As we have noticed, $\hat{K}=K(x_1)$.

Because $\hat{R}|_{\hat{x}_1}=-\hat{K}\id_{\wedge^2 T|_{\hat{x}_1}\hat{M}}$,
one has
\[
\widetilde{\Rol}_{q_1}(\star E_1)=&(-K_1(x_1)+\hat{K}) \star E_1|_{x_1} \\
\widetilde{\Rol}_{q_1}(\star E_2)=&(-K_2(x_1)+\hat{K}) \star E_2|_{x_1} \\
\widetilde{\Rol}_{q_1}(\star E_3)=&(-K(x_1)+\hat{K})\star E_3|_{x_1}=0.
\]
Since $\rank\widetilde{\Rol}_{q_1}=2$, we have $-K_1(x_1)+\hat{K}\neq 0$,
$-K_2(x_1)+\hat{K}\neq 0$.

Because the vector fields $\nu(\Rol(\star E_1)(\cdot))$, $\nu(\Rol(\star E_2)(\cdot))$
are tangent to the orbit $\mc{O}_{\RDist}(q_0)$
on $O_2':=O_2\cap \pi_{Q,M}^{-1}(U)$, so is their Lie bracket. According to Proposition \ref{pr:NS_comm_VV}, the value of this bracket at $q_1$ is equal to
\[
[\nu(\Rol(\star E_1)(\cdot)),\nu(\Rol(\star E_2)(\cdot))]|_{q_1}
=(-K_1(x_1)+\hat{K})(-K_2(x_1)+\hat{K})\nu(A\star E_3)|_{q_1}.
\]
Hence $\nu(\Rol(\star E_1)(\cdot)$, $\nu(\Rol(\star E_2)(\cdot)$, $[\nu(\Rol(\star E_1)(\cdot)),\nu(\Rol(\star E_2)(\cdot))]$
are tangent to $\mc{O}_{\RDist}(q_0)$ and since they are linearly independent at $q_1$,
hence they are linearly independent on an open neighbourhood of $q_1$ in $\mc{O}_{\RDist}(q_0)$.
Therefore, from Corollary \ref{cor:vert} it follows that the orbit $\mc{O}_{\RDist}(q_0)$
is open in $Q$, which is a contradiction.

(2) Suppose $-K(x_1)$ is a single eigenvector of $R|_{x_1}$,
where $q_1=(x_1,\hat{x}_1;A_1)\in O_2$.
Then by Lemma \ref{le:double_eig_K}, the space $(\hat{M},\hat{g})$
would have a constant curvature in an open set which is a neighbourhood of $\hat{x}_1$.
By the case (1), this leads to a contradiction.
\end{proof}

By the last two lemmas, we may thus assume that
for every $q=(x,\hat{x};A)\in O_2$ the common eigenvalue $-K(x)=-\hat{K}(\hat{x})$
of $R|_x$, $\hat{R}|_{\hat{x}}$ has multiplicity two. 
It has the following consequence.

\begin{corollary}\label{cor:Rol2:loc_smooth}
The assignments $q\mapsto X_A,Y_A,Z_A$ and $q\mapsto K_1^\Rol(q),K_2^\Rol(q),\alpha(q)$
as in Proposition \ref{pr:Rol2:good_basis} can be made locally smoothly on $O_2$.
\end{corollary}

\begin{proof}
Let $q_1=(x_1,\hat{x}_1;A_1)\in O_2$. 
By Lemma \ref{le:no_simple_eig} there are open neighbourhoods $U\ni x_1$ and $\hat{U}\ni \hat{x}_1$
such that the eigenvalue $-K_2(x)$ of $R|_x$ and $-\hat{K}_2(\hat{x})$ of $\hat{R}|_{\hat{x}}$ are both
simple. Therefore the assignment $q\mapsto Y_A$ can be made locally smoothly on $O_2$.
Moreover, recall that the assignment $q\mapsto Z_A$ can be made locally smoothly
since it corresponds to the 1-dimensional kernel of $\widetilde{\Rol}_q$
and $X_A=\star (Y_A\wedge Z_A)$.
\end{proof}

\begin{lemma}\label{le:Rol2:special_3D}
For every $q_1=(x_1,\hat{x}_1;A_1)\in O_2$,
there is are open neighbourhoods $U,\hat{U}$ of $x_1,\hat{x}_1$
and oriented orthonormal frames $E_1,E_2,E_3$ on $M$, $\hat{E}_1,\hat{E}_2,\hat{E}_3$ on $\hat{M}$
with respect to which the connections tables are of the form
\[
\Gamma=\qmatrix{
\Gamma^1_{(2,3)} &                                0                                & -\Gamma^1_{(1,2)} \cr
\Gamma^1_{(3,1)} & \Gamma^2_{(3,1)} 				& \Gamma^3_{(3,1)} \cr
\Gamma^1_{(1,2)} &                                0                               & \Gamma^1_{(2,3)} \cr
},\quad
\hat{\Gamma}=\qmatrix{
\hat{\Gamma}^1_{(2,3)} &                                0                                & -\hat{\Gamma}^1_{(1,2)} \cr
\hat{\Gamma}^1_{(3,1)} & \hat{\Gamma}^2_{(3,1)} 				& \hat{\Gamma}^3_{(3,1)} \cr
\hat{\Gamma}^1_{(1,2)} &                                0                               & \hat{\Gamma}^1_{(2,3)} \cr
},
\]
and
\[
& V(\Gamma^1_{(2,3)})=0,\quad V(\Gamma^1_{(1,2)})=0,\quad \forall V\in E_2|_x^\perp,\quad x\in U, \\
& \hat{V}(\hat{\Gamma}^1_{(2,3)})=0,\quad \hat{V}(\hat{\Gamma}^1_{(1,2)})=0,\quad \forall \hat{V}\in \hat{E}_2|_{\hat{x}}^\perp,\quad \hat{x}\in \hat{U}.
\]
Moreover, $\star E_1,\star E_2,\star E_3$ are eigenvectors of $R$ with eigenvalues $-K,-K_2(\cdot),-K$
on $U$
and similarly $\hat{\star} \hat{E}_1,\hat{\star} \hat{E}_2,\hat{\star} \hat{E}_3$ are eigenvectors of $\hat{R}$ with eigenvalues $-K,-\hat{K}_2(\cdot),-K$
on $\hat{U}$,
where $K\in\R$ is constant.
\end{lemma}

\begin{proof}
As we just noticed, for every $q=(x,\hat{x};A)\in O_2$,
the common eigenvalue $-K(x)=-\hat{K}(\hat{x})$
of $R|_x$ and $\hat{R}|_{\hat{x}}$ has multiplicity equal to two.

Fix $q_1=(x_1,\hat{x}_1;A_1)\in O_2$
and let $E_1,E_2,E_3$ (resp. $\hat{E}_1,\hat{E}_2,\hat{E}_3$) be an orthonormal oriented frame of $(M,g)$
defined on an open set $U\ni x_1$ (resp. $\hat{U}\ni \hat{x}_1$) such that $U\times\hat{U}\subset \pi_{Q}(O_2)$
and that $\star E_1,\star E_2,\star E_3$ (resp. $\hat{\star} \hat{E}_1,\hat{\star} \hat{E}_2,\hat{\star} \hat{E}_3$) are eigenvectors with eigenvalues $-K_1(\cdot),-K_2(\cdot),-K(\cdot)$
(resp. $-\hat{K}_1(\cdot),-\hat{K}_2(\cdot),-\hat{K}_3(\cdot)$)
on $U$ (resp. $\hat{U}$) as given by Proposition \ref{pr:Rol2:good_basis}.
Since $-K$ is double on $U$
(resp. $-\hat{K}$ is double on $\hat{U}$), we assume that $K_1(\cdot)=K(\cdot)\neq K_2(\cdot)$ everywhere on $U$,
(resp.  $\hat{K}_1(\cdot)=\hat{K}(\cdot)\neq \hat{K}_2(\cdot)$ everywhere on $\hat{U}$)
without loss of generality.
Recall that $K(x)=\hat{K}(\hat{x})$ for all $q=(x,\hat{x};A)\in O_2$ by Proposition \ref{pr:Rol2:good_basis}
(and the remark that follows it)
and hence for all $x\in U$, $\hat{x}\in \hat{U}$, $K(x)=\hat{K}(\hat{x})$.
Taking $U,\hat{U}$ to be connected, this immediately imples that both
$K$ and $\hat{K}$ are constant functions on $U$ and $\hat{U}$.
We denote the common constant value simply by $K$.

Let $X_A,Y_A,Z_A$ be chosen as in Proposition \ref{pr:Rol2:good_basis}
for every $q=(x,\hat{x};A)\in O_2$.
Then as $\star Y_A$ is a unit eigenvector of $R|_x$ corresponding to the single eigenvalue $-K_2(x)$,
we must have $E_2|_x=\pm Y_A$
and since $\nu(A\star Y_A)|_q$ is tangent to the orbit $\mc{O}_{\RDist}(q_0)$,
by Lemma \ref{le:Rol2:bundle},
it follows that for every $q=(x,\hat{x};A)\in O_2$,
the vector $\nu(A\star E_2|_x)|_q$ is tangent to $\mc{O}_{\RDist}(q_0)$.
This with Proposition \ref{pr:3D-1} proves the claim for $(M,g)$. Symmetrically (working in $Q(\hat{M},M)$)
the claim also holds for $(\hat{M},\hat{g})$. The proof is complete.
\end{proof}

We will now aim at proving that, using the notations of the previous lemma, $\Gamma^1_{(2,3)}(x)=\hat{\Gamma}^1_{(2,3)}(\hat{x})$
for all $(x,\hat{x})\in \pi_Q(O_2')$,
where $O_2'=\pi_{Q}^{-1}(U\times\hat{U})\cap O_2$
and $U,\hat{U}$ are the domains of definition of
orthonormal frames $E_1,E_2,E_3$ and $\hat{E}_1,\hat{E}_2,\hat{E}_3$
as given by Lemma \ref{le:Rol2:special_3D} above.
That will then allow us, as will be seen, to conclude
the study of the case where orbit is not open and rank of $\Rol$ is equal to $2$.

To this end, we define $\theta:O_2'\to\R$ (restricting to smaller sets $U$, $\hat{U}$ if necessary) to be a smooth function
such that for all $q=(x,\hat{x};A)\in O_2'$,
\[
X_A=&\cos(\theta(q))E_1+\sin(\theta(q))E_3 \\
Z_A=&-\sin(\theta(q))E_1+\cos(\theta(q))E_3
\]
where $X_A,Z_A$ (and also $Y_A$) are chosen using Proposition \ref{pr:Rol2:good_basis}.
Indeed, this is well defined since $X_A,Z_A$ lie in the plane $Y_A^\perp=E_2|_x^\perp$
as do also $E_1|_x,E_3|_x$, for all $q=(x,\hat{x};A)\in O_2'$.

To simplify the notation, we write $c_\theta:=\cos(\theta(q))$ and $s_\theta:=\sin(\theta(q))$
as well as $\Gamma^i_{(j,k)}=\Gamma^i_{(j,k)}(x)$,
when there is no room for confusion.
We will be always working on $O_2'$ if not mentioned otherwise.
Moreover, it is convenient to denote the vector field $E_2$ of $M$ by $Y$
in the computations that follow (since $E_2|_x$ is parallel to $Y_A$ for all $q\in O_2'$,
this notation is justified).
We will do computations on the "side of $M$" but the results are, by symmetry,
always valid for $\hat{M}$ as well.

We will make use of the following formulas which are easily verified (see Lemma \ref{le:rules_of_computation}),
\begin{align}\label{eq:Rol2:LR_X_A}
\LRD(X_A)|_q X_{(\cdot)}&=(\LRD(X_A)|_q\theta-c_{\theta}\Gamma^1_{(3,1)}-s_\theta \Gamma^3_{(3,1)})Z_A+\Gamma^1_{(1,2)}Y \nonumber \\
\LRD(Y)|_q X_{(\cdot)}&=(\LRD(Y)|_q\theta-\Gamma^2_{(3,1)})Z_A \nonumber \\
\LRD(Z_A)|_q X_{(\cdot)}&=(\LRD(Z_A)|_q\theta+s_{\theta}\Gamma^1_{(3,1)}-c_\theta \Gamma^3_{(3,1)})Z_A+\Gamma^1_{(2,3)}Y \nonumber \\
\LRD(X_A)|_qY&=-\Gamma^1_{(1,2)}X_A+\Gamma^1_{(2,3)}Z_A \nonumber \\
\LRD(Y)|_qY&=0 \nonumber \\
\LRD(Z_A)|_qY&= -\Gamma^1_{(2,3)}X_A-\Gamma^1_{(1,2)}Z_A \nonumber \\
\LRD(X_A)|_q Z_{(\cdot)}&=(-\LRD(X_A)|_q\theta+c_\theta\Gamma^1_{(3,1)}+s_\theta\Gamma^3_{(3,1)})X_A-\Gamma^1_{(2,3)} Y \nonumber \\
\LRD(Y)|_q Z_{(\cdot)}&=(-\LRD(Y)|_q\theta+\Gamma^2_{(3,1)})X_A \nonumber \\
\LRD(Z_A)|_q Z_{(\cdot)}&=(-\LRD(Z_A)|_q\theta-s_\theta\Gamma^1_{(3,1)}+c_\theta\Gamma^3_{(3,1)})X_A+\Gamma^1_{(1,2)}Y.
\end{align}

\begin{remark}
Notice that $\nu(A\star Z_A)|_q$ is not tangent to the orbit $\mc{O}_{\RDist}(q_0)$
for any $q=(x,\hat{x};A)\in O_2'$. Indeed, otherwise there 
would be an open neighbourhood $O\subset O_2'$ of $q$
such that for all $q'=(x',\hat{x}';A')$
the vectors $\nu(A'\star X_{A'})|_{q'}, \nu(A'\star Y)|_{q'}, \nu(A'\star Z_{A'})|_{q'}$
would span $V|_{q'}(\pi_Q)$ while being tangent to $T|_{q'} \mc{O}_{\RDist}(q_0)$,
which implies $V|_{q'}(\pi_Q)\subset T|_{q'}\mc{O}_{\RDist}(q_0)$.
Then Corollary \ref{cor:vert} would imply that $\mc{O}_{\RDist}(q_0)$
is open, which is not the case.
We will use this fact frequently in what follows.
\end{remark}

Taking $U$, $\hat{U}$ smaller if necessary,
we may also assume that $\theta$ is actually defined on only on $O_2'$
but on an open neighbourhood $\tilde{O}_2'$ of $O_2$ in $Q$.
We will make this technical assumption
to be able to write e.g. $\nu(A\star Z_A)|_q\theta$ whenever needed.

\begin{lemma}\label{le:Rol2:deriv_theta}
For every $q=(x,\hat{x};A)\in O_2'$ we have
\[
\nu(A\star Y)|_q\theta&=1, \\ 
\LRD(X_A)|_q\theta&=c_\theta \Gamma^1_{(3,1)}+s_\theta \Gamma^3_{(3,1)}, \\
\LRD(Y)|_q\theta&=\Gamma^2_{(3,1)}-\Gamma^1_{(2,3)}.
\]
Moreover, if one defines for $q=(x,\hat{x};A)\in O_2'$,
\[
F_X|_q:=&\LNSD(X_A)|_q-\Gamma^1_{(1,2)}\nu(A\star Z_A)|_q \\
F_Z|_q:=&\LNSD(Z_A)|_q-\Gamma^1_{(2,3)}\nu(A\star Z_A)|_q,
\]
then $F_X,F_Z$ are smooth vector fields on $O_2'$ tangent to
the orbit $\mc{O}_{\RDist}(q_0)$.
\end{lemma}

\begin{proof}
We begin by showing that $\nu(A\star Y)|_q\theta=1$.
Indeed, we have for every $q=(x,\hat{x};A)\in O_2'$
that $\hat{g}(AZ_A,\hat{E}_2)=0$.
Differentiating this w.r.t. $\nu(A\star Y)|_q$ yields
\[
0=\hat{g}(A(\star Y)Z_A,\hat{E}_2)
-\nu(A\star Y)|_q\theta \hat{g}(AX_A,\hat{E}_2)
=\hat{g}(AX_A,\hat{E}_2)(1-\nu(A\star Y)|_q\theta).
\]
We show that $\hat{g}(AX_A,\hat{E}_2)\neq 0$, whence $\nu(A\star Y)|_q\theta=1$.
Indeed, if it were that $\hat{g}(AX_A,\hat{E}_2)=0$,
then $AX_A\in E_2^\perp$ and hence $\hat{\star} (AX_A)$ 
would be an eigenvector of $\hat{R}|_{\hat{x}}$ with eigenvalue $-K$.
But this would then imply that
\[
\widetilde{\Rol}_q(\star X_A)=R(\star X_A)-A^{\ol{T}}\hat{R}(\hat{\star} (AX_A))A
=-K\star X_A+KA^{\ol{T}}(\hat{\star} (AX_A)A=0.
\]
Because, $\widetilde{\Rol}_q(X_A\wedge Y)=0$ as well,
we see that $\widetilde{\Rol}_q$ has rank $\leq 1$ as a map $\wedge^2 T|_x M\to \wedge^2 T|_x M$,
which is a contradiction since $q\in O_2'\subset O_2$
and $O_2$ is, by definition, the set of points of the orbit
where $\widetilde{\Rol}_q$ has rank 2.
This contradiction establishes the above claim.

Next we may compute the Lie brackets
\[
[\LRD(Y),\nu((\cdot)\star X_{(\cdot)}]|_q
=&-\LNSD(A(\star X_A)Y)|_q+\nu(A\star \LRD(Y)|_qX_{(\cdot)})|_q \\
=&-\LNSD(AZ_A)|_q+(\LRD(Y)|_q\theta-\Gamma^2_{(3,1)})\nu(A\star Z_A)|_q \\
[\LRD(X_{(\cdot)}),\nu((\cdot)\star Y)]|_q
=&-\LRD(\nu(A\star Y)|_qX_{(\cdot)})|_q-\nu(A\star Y)|_q\theta \LNSD(A(\star Y)X_A)|_q \\
&+\nu(A\star (c_\theta (-\Gamma^1_{(1,2)}E_1+\Gamma^1_{(2,3)}E_3)+s_\theta (-\Gamma^1_{(2,3)}E_1-\Gamma^1_{(1,2)}E_3)))|_q \\
=&-\LRD(\nu(A\star Y)|_qX_{(\cdot)})|_q+\LNSD(AZ_A)|_q \\
&-\Gamma^1_{(1,2)}\nu(A\star X_A)|_q+\Gamma^1_{(2,3)}\nu(A\star  Z_A)|_q,
\]
from where by adding we get
\[
&[\LRD(Y),\nu((\cdot)\star X_{(\cdot)}]|_q+[\LRD(X_{(\cdot)}),\nu((\cdot)\star Y)]|_q \\
=&(\LRD(Y)|_q\theta-\Gamma^2_{(3,1)}+\Gamma^1_{(2,3)})\nu(A\star Z_A)|_q
-\LRD(\nu(A\star Y)|_qX_{(\cdot)})|_q \\
&-\Gamma^1_{(1,2)}\nu(A\star X_A)|_q.
\]
Since this has to be tangent to $\mc{O}_{\RDist}(q_0)$, we get that the $\nu(A\star Z_A)|_q$-component vanished i.e.,
\[
\LRD(Y)|_q\theta=\Gamma^2_{(3,1)}-\Gamma^1_{(2,3)}.
\]
Next compute
\[
[\LRD(X_{(\cdot)}),\nu((\cdot)\star X_{(\cdot)}]|_q
=&-\nu(A\star X_A)|_q\theta \LRD(Z_A)|_q-\LNSD(A\underbrace{(\star X_A)X_A}_{=0})|_q \\
&+\nu\big(A\star \big((\LRD(X_A)|_q\theta-c_\theta\Gamma^1_{(3,1)}-s_\theta\Gamma^3_{(3,1)})Z_A)+\Gamma^1_{(1,2)}Y\big)\big|_q
\]
and so we must have again that the $\nu(A\star Z_A)|_q$-component is zero i.e.,
\[
\LRD(X_A)|_q\theta=c_\theta\Gamma^1_{(3,1)}+s_\theta\Gamma^3_{(3,1)}.
\]

Notice that $[\LRD(X_{(\cdot)}),\nu((\cdot)\star Y)]|_q$
can be written, since $\LNSD(AZ_A)|_q=\LRD(Z_A)|_q-\LNSD(Z_A)|_q$, as
\[
[\LRD(X_{(\cdot)}),\nu((\cdot)\star Y)]|_q
=&-F_Z|_q+\LRD(Z_A)|_q-\LRD(\nu(A\star Y)|_qX_{(\cdot)})|_q-\Gamma^1_{(1,2)}\nu(A\star X_A)|_q \\
=&-F_Z|_q-\Gamma^1_{(1,2)}\nu(A\star X_A)|_q,
\]
which proves that $F_Z$, as defined in the statement, is indeed tangent to the orbit on $O_2'$.
To show that $F_X$ is also tangent to the orbit we compute
\[
[\LRD(Z_{(\cdot)}),\nu((\cdot)\star Y)]|_q
=&-\LRD(\nu(A\star Y)|_qZ_{(\cdot)})|_q-\nu(A\star Y)|_q\theta \LNSD(A(\star Y)Z_A)|_q \\
&+\nu(A\star (-s_\theta (-\Gamma^1_{(1,2)}E_1+\Gamma^1_{(2,3)}E_3)+c_\theta (-\Gamma^1_{(2,3)}E_1-\Gamma^1_{(1,2)}E_3)))|_q \\
=&-\LRD(\nu(A\star Y)|_q Z_{(\cdot)})|_q-\LNSD(AX_A)|_q \\
&-\Gamma^1_{(1,2)}\nu(A\star Z_A)|_q-\Gamma^1_{(2,3)}\nu(A\star  X_A)|_q \\
=&F_X|_q-\LRD(X_A)|_q-\LRD(\nu(A\star Y)|_q Z_{(\cdot)})|_q-\Gamma^1_{(2,3)}\nu(A\star  X_A)|_q \\
=&F_X|_q-\Gamma^1_{(2,3)}\nu(A\star  X_A)|_q,
\]
which finishes the proof.
\end{proof}

\begin{lemma}\label{le:Rol2:gamma_123}
For all $(x,\hat{x})\in \pi_Q(O_2')$ one has
\[
\Gamma^1_{(2,3)}(x)=\hat{\Gamma}^1_{(2,3)}(\hat{x}).
\]
\end{lemma}

\begin{proof}
We begin by observing that for all $q=(x,\hat{x};A)\in O_2'$ one has
\[
\hat{g}(AZ_A,\hat{E}_2)=0.
\]
Indeed, $AZ_A$  and $\hat{E}_2|_{\hat{x}}$ are eigenvectors of $\hat{R}|_{\hat{x}}$
corresponding to non-equal eigenvalues $-K$ and $-\hat{K}_2(\hat{x})$,
hence they must be orthogonal.

Since $AZ_A\in \hat{E}_2|_{\hat{x}}^\perp$, there is a $\hat{\theta}=\hat{\theta}(q)$, for all $q=(x,\hat{x};A)\in O_2'$,
such that
\[
AZ_A=-s_{\hat{\theta}}\hat{E}_1+c_{\hat{\theta}}\hat{E}_3.
\]
Because $AX_A,AY\in (AZ_A)^\perp$, there exits also a $\hat{\phi}=\hat{\phi}(q)$
such that
\[
AX_A=&c_{\hat{\phi}}(c_{\hat{\theta}}\hat{E}_1+s_{\hat{\theta}}\hat{E}_3)+s_{\hat{\phi}}\hat{E}_2 \\
AY=&-s_{\hat{\phi}}(c_{\hat{\theta}}\hat{E}_1+s_{\hat{\theta}}\hat{E}_3)+c_{\hat{\phi}}\hat{E}_2.
\]
Moreover, Lemma \ref{le:Rol2:deriv_theta} along with Eq. (\ref{eq:Rol2:LR_X_A}) implies that $\LRD(Y)|_qZ_{(\cdot)}$ simplifies to
\[
\LRD(Y)|_qZ_{(\cdot)}=\Gamma^1_{(2,3)}(x) X_A.
\]
Therefore, differentiating $\hat{g}(AZ_A,\hat{E}_2)=0$ with respect to $\LRD(X_A)|_q$, one obtains
\[
0=&\LRD(Y)|_q\hat{g}((\cdot)Z_{(\cdot)},\hat{E}_2)
=\hat{g}(A\LRD(Y)|_qZ_{(\cdot)},\hat{E}_2)+\hat{g}(AZ_A,\hat{\nabla}_{AY}\hat{E}_2) \\
=&\Gamma^1_{(2,3)}\hat{g}(AX_A,\hat{E}_2)+\hat{g}\big(AZ_A,-s_{\hat{\phi}} c_{\hat{\theta}}(-\hat{\Gamma}^1_{(1,2)}\hat{E}_1+\hat{\Gamma}^1_{(2,3)}\hat{E}_3)-s_{\hat{\phi}} s_{\hat{\theta}}(-\hat{\Gamma}^1_{(2,3)}\hat{E}_1-\hat{\Gamma}^1_{(1,2)}\hat{E}_3)\big) \\
=&s_{\hat{\phi}}\Gamma^1_{(2,3)}-s_{\hat{\phi}}\hat{g}\big(AZ_A,\hat{\Gamma}^1_{(2,3)}AZ_A
-\hat{\Gamma}^1_{(1,2)}(c_{\hat{\theta}}\hat{E}_1+s_{\hat{\theta}}\hat{E}_3)\big) \\
=&s_{\hat{\phi}}(\Gamma^1_{(2,3)}(x)-\hat{\Gamma}^1_{(2,3)}(\hat{x})).
\]
We claim that $\sin(\hat{\phi}(q))\neq 0$ for $q=(x,\hat{x};A)\in O_2'$,
which would then imply that $\Gamma^1_{(2,3)}(x)-\hat{\Gamma}^1_{(2,3)}(\hat{x})=0$
and finish the proof.

Indeed, $\sin(\hat{\phi}(q))=0$ would mean that
$AX_A=\pm (c_{\hat{\theta}}\hat{E}_1+s_{\hat{\theta}}\hat{E}_3)$,
thus $AX_A\in \hat{E}_2^\perp$.
By the argument at the beginning of the proof of Lemma \ref{le:Rol2:deriv_theta},
this would be a contradiction.
\end{proof}

\begin{corollary}
\begin{itemize}
\item[(i)] If for some $(x_1,\hat{x}_1)\in \pi_{Q}(O_2')$,
one has $\Gamma^1_{(2,3)}(x_1)\neq 0$ (or $\hat{\Gamma}^1_{(2,3)}(\hat{x}_1)\neq 0$),
there are open neighbourhoods $U'\ni x_1$, $\hat{U}'\ni \hat{x}_1$
such that $(U',g)$, $(\hat{U}',\hat{g})$
are both of class $\mc{M}_\beta$ for $\beta=\Gamma^1_{(2,3)}(x_1)$
(or $\beta=\hat{\Gamma}^1_{(2,3)}(\hat{x}_1)$).

\item[(ii)] If for some $(x_1,\hat{x}_1)\in \pi_{Q}(O_2')$,
one has $\Gamma^1_{(2,3)}(x_1)=0$ (or $\hat{\Gamma}^1_{(2,3)}(\hat{x}_1)=0$),
there are open neighbourhoods $U'\ni x_1$, $\hat{U}'\ni \hat{x}_1$
such that $U'\times \hat{U}'\subset\pi_{Q}(O_2')$
and isometries $F:(I\times N,h_f)\to (U,g)$, $\hat{F}:(I\times \hat{N},\hat{h}_{\hat{f}})\to (\hat{U},\hat{g})$,
where $I\subset\R$ is an open interval, such that
\[
\frac{f''(t)}{f(t)}=-K=\frac{\hat{f}''(t)}{\hat{f}(t)},\quad \forall t\in I.
\]
\end{itemize}
\end{corollary}

\begin{proof}
Let $U',\hat{U}'$ be connected neighbourhoods of $x_1,\hat{x}_1$
such that $U'\times \hat{U}'\subset\pi_{Q}(O_2')$
(recall that by Lemma \ref{le:Rol2:bundle}, $\pi_{Q}(O_2')$ is open in $M\times\hat{M}$).

(i) Set $\beta=\Gamma^1_{(2,3)}(x_1)\neq 0$.
By Lemma \ref{le:Rol2:gamma_123}, one has for every $x\in U'$, $\hat{x}\in\hat{U}'$
that
\[
& \hat{\Gamma}^1_{(2,3)}(\hat{x})=\Gamma^1_{(2,3)}(x_1)=\beta, \\
& \Gamma^1_{(2,3)}(x)=\hat{\Gamma}^1_{(2,3)}(\hat{x}_1)=\beta.
\]
By Proposition \ref{pr:special_3D} case (ii), it follows that (after shrinking $U',\hat{U}'$)
$(U,g)$ and $(\hat{U},\hat{g})$ are both of class $\mc{M}_{\beta}$
This gives (i).

(ii)
By Lemma \ref{le:Rol2:gamma_123}, one has for every $x\in U'$, $\hat{x}\in\hat{U}'$
that
\[
& \hat{\Gamma}^1_{(2,3)}(\hat{x})=\Gamma^1_{(2,3)}(x_1)=0 \\
& \Gamma^1_{(2,3)}(x)=\hat{\Gamma}^1_{(2,3)}(\hat{x}_1)=0,
\]
i.e. $\Gamma^1_{(2,3)}$ and $\hat{\Gamma}^1_{(2,3)}$ vanish on $U',\hat{U}'$, respectively.

Then Proposition \ref{pr:special_3D} case (iii) gives (after shrinking $U',\hat{U}'$)
the desired isometries $F,\hat{F}$.
Moreover, Eq. (\ref{eq:RsE3sE3}) in that proposition gives, since $E_2=\pa{r}$, $\hat{E}_2=\pa{r}$,
\[
-K=\dif{r}\frac{f'(r)}{f(r)}+\big(-\frac{f'(r)}{f(r)}\big)-0^2=\frac{f''(r)}{f(r)} \\
-K=\dif{r}\frac{\hat{f}'(r)}{\hat{f}(r)}+\big(-\frac{\hat{f}'(r)}{\hat{f}(r)}\big)-0^2=\frac{\hat{f}''(r)}{\hat{f}(r)},
\]
where $r\in I$. This proves (ii).
\end{proof}

%%%%%%%%%%%%%%%%%%%%%%%%%%%%%%%%%%%%%%%%
\subsubsection{Local Structures for the Manifolds Around $q\in O_1$}
%%%%%%%%%%%%%%%%%%%%%%%%%%%%%%%%%%%%%%%%

In analogy to Proposition (\ref{pr:Rol2:good_basis})
we will first prove the following result.
In the results below that concern $O_1$, we always assume that $O_1\neq\emptyset$.

For the next proposition, contrary to an analogous Proposition \ref{pr:Rol2:good_basis} of Subsubsection \ref{sss:local-O2},
we do not need to assume that $\mc{O}_{\RDist}(q_0)$ is not open.
The subsequent result only relies on the fact that $O_1$ is not empty.

\begin{proposition}\label{pr:Rol1:good_basis}
Let $q_0=(x_0,\hat{x}_0;A_0)\in Q$.
Then for every $q=(x,\hat{x};A)\in O_1$
there exist an orthonormal pair $X_A,Y_A\in T|_x M$
such that if $Z_A:=\star (X_A\wedge Y_A)$
then $X_A,Y_A,Z_A$ is a positively oriented orthonormal pair
with respect to which $R$ and $\widetilde{\Rol}$ may be written as
\begin{align}\label{eq2:good_basis_for_R_Rol}
R(X_A\wedge Y_A)&=\qmatrix{
0 & K(x) & 0 \cr
-K(x) & 0 & 0 \cr
0 & 0 & 0
},\quad \star R(X_A\wedge Y_A)=\qmatrix{0 \cr 0 \cr -K(x)} \nonumber \\
R(Y_A\wedge Z_A)&=\qmatrix{
0 & 0 & 0 \cr
0 & 0 & K(x) \cr
0 & -K(x) & 0
},\quad \star R(Y_A\wedge Z_A)=\qmatrix{-K(x) \cr 0 \cr 0} \nonumber \\
R(Z_A\wedge X_A)&=\qmatrix{
0 & 0 & -K_2(x) \cr
0 & 0 & 0 \cr
K_2(x) & 0 & 0
},\quad \star R(Z_A\wedge X_A)=\qmatrix{0 \cr -K_2(x) \cr 0} \nonumber \\
\widetilde{\Rol}_q(X_A\wedge Y_A)&=0, \nonumber \\
\widetilde{\Rol}_q(Y_A\wedge Z_A)&=0, \nonumber \\
\widetilde{\Rol}_q(Z_A\wedge X_A)&=\qmatrix{
0 & 0 & -K^{\Rol}_2(q) \cr
0 & 0 & 0 \cr
K^{\Rol}_2(q) & 0 & 0
},\quad \star \widetilde{\Rol}_q(Z_A\wedge X_A)=\qmatrix{0 \cr -K^{\Rol}_2(q) \cr 0},
\end{align}
where $K,K_2:M\to\R$.
\end{proposition}

With respect to $X_A,Y_A,Z_A$ given by the theorem, we also have
\begin{align}\label{eq:good_basis_for_hatR}
\star A^{\ol{T}}\hat{R}(AX_A\wedge AY_A)A&=\qmatrix{0 \cr 0 \cr -K(x)} \nonumber \\
\star A^{\ol{T}}\hat{R}(AY_A\wedge AZ_A)A&=\qmatrix{-K(x) \cr 0 \cr 0} \nonumber \\
\star A^{\ol{T}}\hat{R}(AZ_A\wedge AX_A)A&=\qmatrix{0 \cr -K_2(x)+K_2^\Rol(q) \cr 0}.
\end{align}

We collect some important observations concerning the previous proposition
into the following remark.

\begin{remark}\label{re:Rol1:good_basis}
\begin{itemize}
\item[(a)] The last proposition says that $\star X_A,\star Y_A,\star Z_A$
are eigenvectors of $R|_x$, for every $q=(x,\hat{x};A)\in O_1$,
with the corresponding eigenvalues $-K(x)$, $-K_2(x)$ and $-K(x)$.

Changing the roles of $(M,g)$ and $(\hat{M},\hat{g})$,
the proposition gives, for every $q=(x,\hat{x};A)\in O_1$, eigenvectors
$\hat{\star} \hat{X}_A,\hat{\star} \hat{Y}_A,\hat{\star} \hat{Z}_A$
are eigenvectors of $\hat{R}|_x$, 
with the corresponding eigenvalues $-\hat{K}(\hat{x}),-\hat{K}_2(\hat{x}),-\hat{K}(x)$.

\item[(b)] Moreover, the eigenvalues $K$ and $\hat{K}$ coincide on the set of points
that can be reached, locally, by the rolling.
More precisely, Proposition \ref{pr:Rol1:good_basis} tells us that
\[
-\hat{K}(\hat{x})=-K(x),\quad \forall (x,\hat{x})\in \pi_Q(O_1)
\]
and that this eigenvalue is at least a double eigenvalue for  both $R|_x$ and $\hat{R}|_{\hat{x}}$.

\item[(c)]
It is also seen that the above at-least-double eigenvalue cannot be a triple
eigenvalue for both $R|_x$ and $\hat{R}|_{\hat{x}}$ at the same time,
for $(x,\hat{x})\in\pi_Q(O_1)$.
Indeed, if $K_2(x)=K(x)$ and $\hat{K}_2(\hat{x})=\hat{K}(\hat{x})$,
then clearly this would imply that $\Rol_q=0$, which
contradicts the fact that $q\in O_1$ implies $\rank\Rol_q=1$.

\item[(d)] Finally, notice that
it is not clear that the assignments $q\mapsto X_A,Z_A$ can be made locally smoothly on $O_1$.
However, it is the case for the assignment $q\mapsto Y_A$.
In addition, for every $q=(x,\hat{x};A)\in O_1$,
the choice of $Y_A$ and $\hat{Y}_A$ are uniquely determined
up to multiplication by $-1$.
Indeed, $\star Y_A=Z_A\wedge X_A$ is a unit eigenvector of $\widetilde{\Rol}_q$
corresponding to the simple non-zero eigenvalue $-K_2^\Rol(q)$
(it is non-zero since $\rank\Rol_q=1$, $q\in O_1$).
By symmetry, the same claim holds of $\hat{Y}_A$ as well.
Moreover, this implies that
\[
AY_A=\pm\hat{Y}_A,\quad \forall q=(x,\hat{x};A)\in O_1.
\]
\end{itemize}
\end{remark}

We begin by the following simple lemma.

\begin{lemma}\label{le:Rol1:good_basis}
For every $q=(x,\hat{x};A)\in O_1$
and any orthonormal pair (which exists) $X_A,Y_A\in T|_x M$ such that
$X_A,Y_A,Z_A:=\star (X_A\wedge Y_A)$
is an oriented orthonormal basis of $T|_x M$
and $\Rol_q(X_A\wedge Y_A)=0$, $\Rol_q(Y_A\wedge Z_A)=0$,
one has with respect to the basis $X_A,Y_A,Z_A$,
\[
R(X_A\wedge Y_A)&=\qmatrix{
0 & K_A & \alpha_A \cr
-K_A & 0 & -\beta_A \cr
-\alpha_A & \beta_A & 0
},\quad \star R(X_A\wedge Y_A)=\qmatrix{\beta_A \cr \alpha_A \cr -K_A} \\
R(Y_A\wedge Z_A)&=\qmatrix{
0 & -\beta_A & \xi_A \cr
\beta_A & 0 & K^1_A \cr
-\xi_A & -K^1_A & 0
},\quad \star R(Y_A\wedge Z_A)=\qmatrix{-K^1_A \cr \xi_A \cr \beta_A} \\
R(Z_A\wedge X_A)&=\qmatrix{
0 & -\alpha_A & -K^2_A \cr
\alpha_A & 0 & -\xi_A \cr
K^2_A & \xi_A & 0
},\quad \star R(Z_A\wedge X_A)=\qmatrix{\xi_A \cr -K^2_A \cr \alpha_A} \\
\widetilde{\Rol}_q(X_A\wedge Y_A)&=0, \\
\widetilde{\Rol}_q(Y_A\wedge Z_A)&=0, \\
\widetilde{\Rol}_q(Z_A\wedge X_A)&=\qmatrix{
0 & 0 & -K^{\Rol}_2 \cr
0 & 0 & 0 \cr
K^{\Rol}_2 & 0& 0
},\quad \star \widetilde{\Rol}_q(Z_A\wedge X_A)=\qmatrix{0 \cr -K^{\Rol}_2 \cr 0} .
\]
Moreover, the choice of the above quantities can be made locally smoothly on $O_1$.
\end{lemma}

\begin{proof}
We only need to prove the existence of
an oriented orthonormal basis $X_A$, $Y_A$ and $Z_A$
such that $\Rol_q(X_A\wedge Y_A)=0$, $\Rol_q(Y_A\wedge Z_A)=0$.
Indeed, when this has been established,
one may use Lemma \ref{le:Rol2:good_basis},
where we now have $K_1^\Rol(q)=0$, $\alpha(q)=0$ because $\Rol_q(Y_A\wedge Z_A)=0$,
to conclude.

Since for a given $q=(x,\hat{x};A)\in O_1$,
$\widetilde{\Rol}_q:\wedge^2 T|_x M\to \wedge^2 T|_x M$
is symmetric linear map that has rank 1,
it follows that its eigenspaces are orthogonal and its kernel 
has dimension exactly 2.
Thus there is an orthonormal basis $\omega_1,\omega_2,\lambda$ of $\wedge^2 T|_x M$
such that $\widetilde{\Rol}_q(\omega_i)=0$, $i=1,2$.
Taking $X_A=\star \omega_1$, $Z_A=\star \omega_2$ and $Y_A=\star \lambda$
we get, up to replacing $X_A$ with $-X_A$ if necessary,
an oriented orthonormal basis of $T|_x M$ such that
$\Rol(X_A\wedge Y_A)=0$, $\Rol(Y_A\wedge Z_A)=0$.
\end{proof}

As a consequence of the lemma and
because $A^{\ol{T}}\hat{R}(AX,AY)A=R(X,Y)-\widetilde{\Rol}_q(X,Y)$ for $X,Y\in T|_x M$,
we have that w.r.t. the oriented orthonormal basis $X_A,Y_A,Z_A$,
\begin{align}
\star A^{\ol{T}}\hat{R}(AX_A,AY_A)A&=\qmatrix{\beta_A \cr \alpha_A \cr -K_A} \nonumber \\
\star A^{\ol{T}}\hat{R}(AY_A,AZ_A)A&=\qmatrix{-K^1_A \cr \xi_A \cr \beta_A} \nonumber \\
\star A^{\ol{T}}\hat{R}(AZ_A,AX_A)A&=\qmatrix{\xi_A \cr -K^2_A+K_2^\Rol \cr \alpha_A}.
\end{align}

Notice that the assumption that $\rank \Rol_q=1$
is equivalent to
the fact that 
for every $q=(x,\hat{x};A)\in O_1$,
\begin{align}\label{eq2:rank1}
K_2^{\Rol}(q)\neq 0.
\end{align}
This implies that $Y_A$ is uniquely determined up to multiplication by $-1$
(see also Remark \ref{re:Rol1:good_basis} above).
Hence, in particular, for every $q=(x,\hat{x};A)\in O_1$, 
\[
\Rol_q(\wedge^2 TM)(A)=\spn\{\nu(A(Z_A\wedge X_A))|_q\}=\spn\{\nu(A\star Y_A)|_q\}.
\]
We will now show that,
with any (non-unique) choice of a pair $X_A,Y_A$
as in Lemma \ref{le:Rol1:good_basis},
one has that $\alpha_A=0$ and $K_A=K_A^1$.

\begin{lemma}\label{le2:beta_zero}
If one chooses any $X_A,Y_A,Z_A=\star (X_A\wedge Y_A)$ as in Lemma \ref{le:Rol1:good_basis},
then
\[
\beta_A=0,\quad K_A=K_A^1,
\quad \forall q=(x,\hat{x};A)\in \mc{O}_{\RDist}(q_0).
\]
\end{lemma}

\begin{proof}
Fix $q=(x,\hat{x};A)\in O_1$.
Choosing in Corollary \ref{cor:vcomm} $X,Y\in\VF(M)$
such that $X|_x=X_A$, $Y|_x=Y_A$,
we get, since $\Rol_q(X_A\wedge Y_A)=0$,
\[
& \nu|_q^{-1}\big[\nu(\Rol(X,Y)(\cdot)),\nu(\Rol(Z,W)(\cdot))\big]\big|_q \nonumber \\
=&A\big[R(X_A\wedge Y_A),R(Z|_x\wedge W|_x)\big]_{\so}-\big[\hat{R}(AX_A\wedge AY_A),\hat{R}(AZ|_x\wedge AW|_x)\big]_{\so}A \nonumber \\
&+\hat{R}(AX_A\wedge A\widetilde{\Rol}_q(Z|_x\wedge W|_x)Y_A)A+\hat{R}(A\widetilde{\Rol}_q(Z|_x\wedge W|_x)X_A,AY_A)A \nonumber.
\]

Since $q'=(x',\hat{x}';A')\mapsto \nu(\Rol(\wedge^2 T|_{x'}M)(A'))|_{q'}=\spn\{\nu(A'\star Y_{A'})\}$
is a smooth rank 1 distribution on $O_1$,
it follows that it is involutive
and hence for all $X,Y,Z,W\in\VF(M)$,
\[
\big[\nu(\Rol(X\wedge Y)(\cdot)),\nu(\Rol(Z\wedge W)(\cdot))\big]\big|_q
\in \spn\{\nu(A\star Y_A)|_q\},
\]
where we used that $\Rol(\wedge^2 TM)(A)=\spn\{A\star Y_A\}$ as observed above.

We compute the right hand side of this formula in different cases. 
We begin by taking any smooth vector fields $X,Y,Z,W$
with $X|_x=X_A$, $Y|_x=Y_A$, $Z|_x=Z_A$, $W|_x=X_A$.
One gets
\[
& A^{\ol{T}}\nu|_q^{-1}\big[\nu(\Rol(X,Y)(\cdot)),\nu(\Rol(Z,W)(\cdot))\big]\big|_q \nonumber \\
=&\big[R(X_A\wedge Y_A),R(Z_A\wedge X_A)\big]_{\so}-\big[A^{\ol{T}}\hat{R}(AX_A\wedge AY_A)A,A^{\ol{T}}\hat{R}(AZ_A\wedge AX_A)A\big]_{\so} \nonumber \\
&+A^{\ol{T}}\hat{R}(AX_A\wedge \Rol(Z_A\wedge X_A)(A)Y_A)A+A^{\ol{T}}\hat{R}(\Rol(Z_A\wedge X_A)(A)X_A\wedge AY_A)A \\
=&\qmatrix{\beta_A \cr \alpha_A \cr -K_A}\wedge \qmatrix{0 \cr -K_2^\Rol \cr 0}+A^{\ol{T}}\hat{R}(AX_A\wedge 0)A+A^{\ol{T}}\hat{R}(K_2^\Rol AZ_A\wedge AY_A)A \\
=&\qmatrix{-K_AK_2^\Rol \cr 0 \cr -\beta_A K_2^\Rol}-K_2^\Rol\qmatrix{-K_A^1 \cr \xi_A \cr \beta_A}
=\qmatrix{K_2^\Rol(-K_A+K_A^1) \cr K_2^\Rol \xi_A \cr -2\beta_A K_2^\Rol}\in \spn\{\nu(A\star Y_A)|_q\}.
\]
Because $K_2^\Rol(q)\neq 0$, this immediately implies that
\[
-K_A+K_A^1=0,\quad \beta_A=0.
\]
This completes the proof.
\end{proof}

We will now rotate $X_A,Y_A,Z_A$ in such a way that we may also set $\alpha_A$ equal to zero.

\begin{lemma}\label{le2:alpha_zero}
For every $q=(x,\hat{x};A)\in O_1$
there are orthonormal $X_A,Y_A\in T|_x M$
such that $X_A,Y_A,Z_A=\star (X_A\wedge Y_A)$ is an oriented orthonormal basis of $T|_x M$
with respect to which in Lemma \ref{le:Rol1:good_basis}
one has $\alpha_A=0$.
\end{lemma}

\begin{proof}
Fix $q=(x,\hat{x};A)\in O_1$, choose any $X_A,Y_A,Z_A=\star (X_A\wedge Y_A)$ as in Lemma \ref{le:Rol1:good_basis}
and suppose $\alpha_A\neq 0$ (otherwise we are done).
We let $t\in\R$,
\[
\qmatrix{X_A(t)\cr Z_A(t)}=\qmatrix{\cos(t) & \sin(t) \cr -\sin(t) & \cos(t)}\qmatrix{X_A \cr Z_A}.
\]
Then clearly $Y_A(t):=\star (X_A(t)\wedge Z_A(t))=\star (X_A\wedge Z_A)=Y_A$
and $X_A(t),Y_A(t),Z_A(t)$ is an orthonormal positively oriented basis of $T|_x M$.
Since $\widetilde{\Rol}_q$ is a symmetric map $\wedge^2 T|_x M\to \wedge^2 T|_x M$
and since $\star X_A,\star Z_A$ are its eigenvectors corresponding to the eigenvalue $0$,
it follows that $\star X_A(t),\star Z_A(t)$,
which are just rotated $\star X_A,\star Z_A$ in the plane that the form,
are eigenvectors of $\Rol_q$ corresponding to eigenvalue $0$,
i.e. $\Rol_q(X_A(t)\wedge Y_A)=0$, $\Rol_q(Y_A\wedge Z_A(t))=0$ for all $t\in\R$.

Hence the conclusion of Lemma \ref{le:Rol1:good_basis}
holds for basis $X_A(t),Y_A,Z_A(t)$
and we write $\xi_A(t),\alpha_A(t),\beta_A(t),K_A(t),K_A^1(t),K_A^2(t)$
for the coefficients of the matrices of $R$ given there w.r.t. $X_A(t),Y_A,Z_A(t)$.
Then Lemma \ref{le2:beta_zero} implies that $\beta_A(t)=0$, $K_A(t)=K_A^1(t)$ for all $t\in\R$.
Computing now
\[
\alpha_A(t)=&g(R(X_A(t)\wedge Y_A)Z_A(t),X_A(t))
=g(R(Z_A(t)\wedge X_A(t))X_A(t),Y_A(t)) \\
=&-g(R(Z_A\wedge X_A)Y_A,X_A(t)) \\
=&-g(-\alpha_A X_A+\xi_A Z_A,\cos(t)X_A+\sin(t)Z_A) \\
=&-\alpha_A\cos(t)+\xi_A\sin(t).
\]
Thus choosing $t_0\in\R$ such that
\[
\cot(t_0)=\frac{\xi_A}{\alpha_A},
\]
we get that $\alpha_A(t_0)=0$.
As already observed, we also have $\beta_A(t_0)=0$, $K_A^1(t_0)=K_A(t_0)$
and $\Rol_q(X_A(t_0)\wedge Y_A)=0$,
$\Rol_q(Y_A\wedge Z_A(t_0))=0$.
\end{proof}

Since $\alpha_A$ and $\beta_A$ vanish w.r.t $X_A,Y_A,Z_A$,
as chosen by the previous lemma, 
we have that $-K_A$ is an eigenvalue of $R|_x$
with eigenvector $X_A\wedge Y_A$, where $q=(x,\hat{x};A)\in O_1$.
Knowing this, we may prove that even $\xi_A$ is zero as well and that (automatically) $-K_A$
is a at least a double eigenvalue of $R|_x$.
This is given in the lemma that follows.

\begin{lemma}\label{le2:xi_zero}
If $q=(x,\hat{x};A)\in O_1$ and
$X_A,Y_A,Z_A$ as in Lemma \ref{le2:alpha_zero},
then $\xi_A=0$.
\end{lemma}

\begin{proof}
Since for any $q=(x,\hat{x};A)\in O_1$, $-K_A$ is an eigenvalue of $R|_x$,
we know that its value only depends on the point $x$ of $M$
and hence we consider it as a smooth function $-K(x)$ on $M$.

We claim that that $-K(x)$ is at least a double eigenvalue of $R|_x$.
Suppose it is not.
Then in a neighbourhood $U$ of $x$
we have that $-K(y)$ is a simple eigenvalue of $R|_y$ for all $y\in U$.
In that case, we may choose smooth vector fields $X,Y$ on $U$,
taking $U$ smaller if necessary,
such that $X|_y\wedge Y|_y$ is a (non-zero) eigenvector
of $R|_y$ corresponding to $-K(y)$
and $X|_x=X_A$, $Y|_x=Y_A$.
Write $O:=\pi_{Q,M}^{-1}(U)\cap O_1$.

But for any $(y,\hat{y};B)\in O$,
we know that $X_B\wedge Y_B$ is a unit eigenvector of $R|_y$
corresponding to $-K(y)$
and hence, modulo replacing $X$ by $-X$,
we have $X_B\wedge Y_B=X|_y\wedge Y|_y$.

Then for all $(y,\hat{y};B)\in O$ with $y\in U$, one has
\[
\nu(\Rol(X|_y\wedge Y|_y)(B))|_{(y,\hat{y};B)}
=\nu(\Rol(X_B\wedge Y_B)(B))|_{(y,\hat{y};B)}=0
\]
i.e. $\nu(\Rol(X\wedge Y)(\cdot))$ is a zero vector field on the open subset $O$
of the orbit.

If we also take some smooth vector fields $Z,W$
such that $Z|_x=Z_A$, $W|_x=X_A$,
we get by the fact that $\nu(\Rol(X\wedge Y)(\cdot))=0$
and from the computations in the proof of Lemma \ref{le2:beta_zero}
that
\[
0=\nu|_q^{-1}\big[\nu(\Rol(X,Y)(\cdot)),\nu(\Rol(Z,W)(\cdot))\big]\big|_q
=\qmatrix{K_2^\Rol(-K_A+K_A^1) \cr K_2^\Rol \xi_A \cr -2\beta_A K_2^\Rol}
=\qmatrix{0 \cr K_2^\Rol \xi_A \cr 0}.
\]
Since $K_2^\Rol(q)\neq 0$ we get $\xi_A=0$.
But this implies, along with the results obtained in the previous lemma
(i.e. $K=K_A^1$, $\beta_A=\alpha_A=0$)
that w.r.t. the basis $X_A,Y_A,Z_A$,
\[
\star R(X_A\wedge Y_A)=\qmatrix{0 \cr 0 \cr -K_A},\quad 
\star R(Y_A\wedge Z_A)=\qmatrix{-K_A \cr 0 \cr 0},
\]
which means that $X_A\wedge Y_A$ and $Y_A\wedge Z_A$
are linearly independent eigenvectors of $R|_x$
corresponding to the eigenvalue $-K_A=-K(x)$.
This in contradiction to what we assumed in the beginning of the proof.

Hence we have that $-K_A$ is, for every $q=(x,\hat{x};A)\in O_1$,
and eigenvalue of $R|_x$ of multiplicity at least 2.

Finally, since we know that w.r.t. $X_A,Y_A,Z_A$,
\[
\star R(X_A\wedge Y_A)=\qmatrix{0 \cr 0 \cr -K_A},\quad
\star R(Y_A\wedge Z_A)=\qmatrix{-K_A \cr \xi_A \cr 0},\quad
\star R(Z_A\wedge X_A)=\qmatrix{\xi_A \cr -K_A^2 \cr 0},
\]
and since $R|_x$ is a symmetric linear map having double eigenvalue $-K_A$,
we know that there is a unit eigenvector $\omega$ of $R|_x$
corresponding to $-K_A$
which lies in the plane orthogonal to $X_A\wedge Y_A$ (in $\wedge^2 T|_x M$).
Hence, $\omega=\cos(t)Y_A\wedge Z_A+\sin(t)Z_A\wedge X_A$ for some $t\in\R$ and
\[
-K_A\qmatrix{\cos(t)\cr \sin(t)\cr 0}
=&-K_A\star \omega=\star R(\omega)=\cos(t)\star R(Y_A\wedge Z_A)+\sin(t)\star R(Z_A\wedge X_A) \\
=&\cos(t)\qmatrix{-K_A \cr \xi_A \cr 0}+\sin(t)\qmatrix{\xi_A \cr -K_A^2 \cr 0}
=\qmatrix{-K_A\cos(t)+\xi_A\sin(t) \cr \xi_A\cos(t)-K_A^2\sin(t) \cr 0},
\]
where the matrices are again formed w.r.t. $X_A,Y_A,Z_A$.
From the first row we get $\xi_A\sin(t)=0$.
So either $\xi_A=0$ and we are done
or $\sin(t)=0$ which implies that $\omega=1_{\pm} Y_A\wedge Z_A$ with $1_{\pm}\in \{-1,+1\}$ and
hence
\[
\qmatrix{-K_A \cr \xi_A \cr 0}=&\star R(Y_A\wedge Z_A)=1_{\pm} \star R(\omega)=-K_A(1_{\pm} \star \omega)\\
=&-K_A \star (Y_A\wedge Z_A)=\qmatrix{-K_A\cr 0\cr 0},
\]
which gives $\xi_A=0$ anyway.

\end{proof}

The previous lemma implies Proposition \ref{pr:Rol1:good_basis},
since now $-K_A=-K_A^1$, $-K_A^2$ are eigenvalues of $R|_x$
for every $(x,\hat{x};A)\in O_1$
and hence, defining $K(x):=K_A$, $K_2(x):=K_A^2$,
we obtain well defined functions $K,K_2:M\to\R$.

The following Proposition is the last result of this subsection.
Notice that it does need the assumption that $\mc{O}_{\RDist}(q_0)$ 
is not open while the previous results do not need this assumption.

\begin{proposition}\label{pr:Rol1:main}
Suppose $\mc{O}_{\RDist}(q_0)$ is not open in $Q$.
Then there is an open dense subset $O_1^\circ$ of $O_1$
such that for every $q_1=(x_1,\hat{x}_1;A_1)\in O_1^\circ$
there are neighbourhoods $U$ and $\hat{U}$ of $x_1$ and $\hat{x}_1$, respectively,
such that either
\begin{itemize}
\item[(i)] both $(U,g|_U)$, $(\hat{U},\hat{g}|_{\hat{U}})$ are of class $\mc{M}_{\beta}$ or

\item[(ii)] both $(U,g|_U)$, $(\hat{U},\hat{g}|_{\hat{U}})$ are isometric
to warped products $(I\times N,h_f)$, $(I\times\hat{N},\hat{h}_{\hat{f}})$
and $\frac{f'(r)}{f(r)}=\frac{\hat{f}'(r)}{\hat{f}(r)}$, for all $r\in I$.
\end{itemize}

Moreover, there is an oriented orthonormal frame $E_1,E_2,E_3$
(resp. $\hat{E}_1,\hat{E}_2,\hat{E}_3$) defined on $U$ (resp. on $\hat{U}$)
respectively, such that $\star E_1,\star E_3$ (resp. $\star \hat{E}_1,\star \hat{E}_3$)
are eigenvectors of $\hat{R}$ with common eigenvalue $-K(\cdot)$
(resp. $-\hat{K}(\cdot)$)
and one has
\[
A_1E_2|_{x_1}=\hat{E}_2|_{\hat{x}_1}.
\]

\end{proposition}

\begin{proof}
Let $q_1=(x_1,\hat{x}_1;A_1)\in O_1$.
Notice that, as observed in Remark \ref{re:Rol1:good_basis},
either $R|_{x_1}$ or $\hat{R}|_{\hat{x}_1}$
has $-K_2(x_1)$ or $-\hat{K}_2(\hat{x}_1)$, respectively,
as a single eigenvalue.
By symmetry of the problem in $(M,g)$, $(\hat{M},\hat{g})$,
we may assume that this is the case for $R|_{x_1}$.
Hence there is a neighbourhood $U$ of $x_1$
such that $K_2(x)\neq K(x)$ for all $x\in U$.

It is easy to see that there is an open dense subset $O_1'$ of $O_1\cap \pi_{Q,M}^{-1}(U)$
such that, for every $q=(x,\hat{x};A)\in O_1'$, there exists an open neighbourhood $\hat{V}$ of $\hat{x}$
where either $\hat{K}_2=\hat{K}$ on $\hat{V}$ or $\hat{K}_2(\hat{y})\neq \hat{K}(\hat{y})$ for $\hat{y}\in\hat{V}$.
For the rest of the argument, we assume that $q_1$ belongs to $O_1'$.

By shrinking $U$ around $x_1$
and taking a small enough neighbourhood $\hat{U}$ of $\hat{x}_1$,
we may assume there are oriented orthonormal 
frames $E_1,E_2,E_3$ on $U$ (resp. $\hat{E}_1,\hat{E}_2,\hat{E}_3$ on $\hat{U}$)
such that $\star E_1,\star E_2,\star E_3$
(resp. $\hat{E}_1,\hat{E}_2,\hat{E}_3$)
are eigenvectors of $R$ (resp. $\hat{R})$
with eigenvalues $-K(\cdot),-K_2(\cdot),-K(\cdot)$
(resp. $-\hat{K}(\cdot),-\hat{K}_2(\cdot),-\hat{K}(\cdot)$),
where these eigenvalues correspond to those in Proposition \ref{pr:Rol1:good_basis}.

Taking $U$, $\hat{U}$ smaller if necessary, 
we may take $X_A,Y_A,Z_A$ as given by Proposition \ref{pr:Rol1:good_basis}
for $M$ and $\hat{X}_A,\hat{Y}_A,\hat{Z}_A$ for $\hat{M}$
on $\pi_Q^{-1}(U\times\hat{U})\cap O_1'$, which we still denote by $O_1'$.

Since $\star Y_A$ and $\star E_2|_x$ are both eigenvalues of $R|_x$, for $q=(x,\hat{x};A)\in O_1'$,
corresponding to single eigenvalue $-K_2(x)$,
we may moreover assume that $Y_A=E_2|_x$, $\forall q=(x,\hat{x};A)\in O_1'$.

Then because $\nu(\Rol_q(Z_A\wedge X_A))|_q=-K_2^\Rol(q)\nu(A\star E_2)|_q$
is tangent to the orbit $\mc{O}_{\RDist}(q_0)$ at the points $q=(x,\hat{x};A)\in O_1'$,
we may conclude from Proposition \ref{pr:3D-1} 
that
\[
\Gamma=\qmatrix{
\Gamma^1_{(2,3)} &                                0                                & -\Gamma^1_{(1,2)} \cr
\Gamma^1_{(3,1)} & \Gamma^2_{(3,1)} 				& \Gamma^3_{(3,1)} \cr
\Gamma^1_{(1,2)} &                                0                               & \Gamma^1_{(2,3)} \cr
},
\]
where $\Gamma$ and $\Gamma^i_{(j,k)}$ are as defined there.

We will now divide the proof in two parts (cases I and II below), depending whether
$(\hat{M},\hat{g})$ has, in certain areas, constant curvature or not.

{\bf Case I:} Suppose, after shrinking $\hat{U}$ around $x_1$, that
$\hat{K}_2(\hat{x})=\hat{K}(\hat{x})$ for all $\hat{x}\in \hat{U}$.
We also assume that $\hat{U}$ is connected.
This implies by Schur Lemma (see \cite{sakai91}, Proposition II.3.6)
that $\hat{K}_2=\hat{K}$ is constant on $\hat{U}$
and we write simply $\hat{K}$ for this constant.
Again shrinking $\hat{U}$, we may assume that $(\hat{U},\hat{g}|_{\hat{U}})$
is isometric to an open subset of a 3-sphere of curvature $\hat{K}$.

Assume first that $\Gamma^1_{(2,3)}\neq 0$ on $U$.
Then Proposition \ref{pr:special_3D}, case (ii), implies that $\Gamma^1_{(1,2)}=0$ on $U$
and $(\Gamma^1_{(2,3)})^2=K(x)$ is constant on $U$, which must be $\hat{K}$. 
Hence if $\beta:=\Gamma^1_{(2,3)}$, which is constant on $U$,
then $(U,g|_U)$ is of class $\mc{M}_{\beta}$
as is $(\hat{U},\hat{g}|_{\hat{U}})$ and we are done (recall that $\mc{M}_{-\beta}=\mc{M}_{\beta}$)
i.e. this is case (i).

On the other hand, if $\Gamma^1_{(2,3)}=0$ on $U$,
then we have that $(U,g|_U)$, after possibly shrinking $U$,
is isometric, by some $F$, to a warped product $(I\times N,h_f)$ by Proposition \ref{pr:special_3D} case (iii).
At the same time, the space of constant curvature $(\hat{U},\hat{g}|_{\hat{U}})$,
again after shrinking $\hat{U}$ if necessary, can
be presented, isometrically by certain $\hat{F}$, as a warped product $(\hat{I}\times \hat{N},\hat{h}_{\hat{f}})$
as shown in Example \ref{ex:warped_const},
where $\hat{N}$ is a 2-dimensional space of constant curvature.

Because for all $x\in U$ we have $K(x)=\hat{K}$, we get that for all $(r,y)\in I\times N$,
$\hat{r}\in \hat{I}$,
\[
-\frac{f''(r)}{f(r)}=K(F(r,y))=\hat{K}=-\frac{\hat{f}''(\hat{r})}{\hat{f}(\hat{r})}.
\]
Example \ref{ex:warped_const} shows that we may choose $\hat{f}$ such that
$\hat{f}(0)=f(0)$ and $\hat{f}'(0)=f'(0)$, which then implies that $\hat{f}(r)=f(r)$,
for all $r\in I$. This leads us to case (ii)

{\bf Case II:}
We assume here that $\hat{K}_2(\hat{x})\neq \hat{K}(\hat{x})$ for all $\hat{x}\in \hat{U}$.
The same way as for $(M,g)$ above,
this implies that $\hat{Y}_A=\hat{E}_2|_{\hat{x}}$
and that w.r.t. the frame $\hat{E}_1,\hat{E}_2,\hat{E}_3$, Proposition \ref{pr:3D-1} yields
\[
\hat{\Gamma}=\qmatrix{
\hat{\Gamma}^1_{(2,3)} &                                0                                & -\hat{\Gamma}^1_{(1,2)} \cr
\hat{\Gamma}^1_{(3,1)} & \hat{\Gamma}^2_{(3,1)} 				& \hat{\Gamma}^3_{(3,1)} \cr
\hat{\Gamma}^1_{(1,2)} &                                0                               & \hat{\Gamma}^1_{(2,3)} \cr
},
\]
where $\hat{\Gamma}^i_{(j,k)}=\hat{g}(\hat{\nabla}_{\hat{E}_i}\hat{E}_j,\hat{E}_k)$ etc.

We will now claim that for all $(x,\hat{x})\in \pi_Q(O_1')$,
we have
\[
\Gamma^1_{(2,3)}(x)&=\hat{\Gamma}^1_{(2,3)}(\hat{x}) \\
\Gamma^1_{(1,2)}(x)&=\hat{\Gamma}^1_{(1,2)}(\hat{x}).
\]

By Remark \ref{re:Rol1:good_basis}, we have $AY_A=\pm \hat{Y}_A$ for $q=(x,\hat{x};A)\in O_1'$,
and so we get $AE_2|_x=\pm \hat{E}_2|_{\hat{x}}$.
Without loss of generality, we assume that the '$+$' -case holds here.
In particular, if $X\in\VF(M)$, one may differentiate the identity $AE_2=\hat{E}_2$
w.r.t. $\LRD(X)|_q$ to obtain
\[
A\nabla_X E_2=\hat{\nabla}_{AX} \hat{E}_2,\quad \forall q=(x,\hat{x};A)\in O_1'.
\]

Since $AE_1,AE_2,\hat{E}_1,\hat{E}_2\in (AE_2)^\perp=\hat{E}_2^\perp$, there is for every $q\in O_1'$,
a $\varphi=\varphi(q)\in\R$ such that
\[
AE_1|_x&=\cos(\varphi(q)) \hat{E}_1|_{\hat{x}}+\sin(\varphi(q))\hat{E}_3|_{\hat{x}} \\
AE_3|_x&=-\sin(\varphi(q)) \hat{E}_1|_{\hat{x}}+\cos(\varphi(q))\hat{E}_3|_{\hat{x}}.
\]
As usual, we write below $\cos(\varphi(q))=c_{\varphi}$, $\sin(\varphi(q))=s_{\varphi}$
Having these, we compute
\[
A\nabla_{E_1} E_2=&A(-\Gamma^1_{(1,2)}E_1+\Gamma^1_{(2,3)}E_3) \\
=&(-c_{\varphi}\Gamma^1_{(1,2)}-s_{\varphi}\Gamma^1_{(2,3)})\hat{E}_1
+(-s_{\varphi}\Gamma^1_{(1,2)}+c_{\varphi}\Gamma^1_{(2,3)})\hat{E}_3
\]
and, on the other hand,
\[
\hat{\nabla}_{AE_1}\hat{E}_2
=&c_{\varphi}(-\hat{\Gamma}^1_{(1,2)}\hat{E}_1+\hat{\Gamma}^1_{(2,3)}\hat{E}_3)
+s_{\varphi}(-\hat{\Gamma}^1_{(2,3)}\hat{E}_1-\hat{\Gamma}^1_{(1,2)}\hat{E}_3) \\
=&(-c_{\varphi}\hat{\Gamma}^1_{(1,2)}-s_{\varphi}\hat{\Gamma}^1_{(2,3)})\hat{E}_1
+(c_{\varphi}\hat{\Gamma}^1_{(2,3)}-s_{\varphi}\hat{\Gamma}^1_{(1,2)})\hat{E}_3.
\]

Taking $X=E_1$ above and using the last two formulas, we get
\[
&(-c_{\varphi}\Gamma^1_{(1,2)}-s_{\varphi}\Gamma^1_{(2,3)})\hat{E}_1
+(-s_{\varphi}\Gamma^1_{(1,2)}+c_{\varphi}\Gamma^1_{(2,3)})\hat{E}_3
=A\nabla_{E_1} E_2 \\
=&\hat{\nabla}_{AE_1}\hat{E}_2
=(-c_{\varphi}\hat{\Gamma}^1_{(1,2)}-s_{\varphi}\hat{\Gamma}^1_{(2,3)})\hat{E}_1
+(c_{\varphi}\hat{\Gamma}^1_{(2,3)}-s_{\varphi}\hat{\Gamma}^1_{(1,2)})\hat{E}_3
\]
from which
\[
c_{\varphi}(-\Gamma^1_{(1,2)}+\hat{\Gamma}^1_{(1,2)})
+s_{\varphi}(\Gamma^1_{(2,3)}-\hat{\Gamma}^1_{(2,3)})=0.
\]

Next we notice that differentiating the identity $AE_1=c_\varphi \hat{E}_1+s_\varphi\hat{E}_3$ with respect to $\nu(A\star E_2)|_q$ gives
\[
A(\star E_2)E_1=(\nu(A\star E_2)|_q\varphi)(-s_{\varphi}\hat{E}_1+c_{\varphi}\hat{E}_3)
\]
which simplifies to
\[
-AE_3=(\nu(A\star E_2)|_q\varphi) AE_3
\]
and hence yields
\[
\nu(A\star E_2)|_q\varphi=-1,\quad \forall q=(x,\hat{x};A)\in O_1'.
\]
Thus, if $(t,q)\mapsto \Phi(t,q)$ is the flow of $\nu((\cdot)\star E_2)$ in $O_2'$
with initial position at $t=0$ at $q\in O_1'$,
the above implies that $\varphi(\Phi(t,q))=\varphi(q)+t$ for all $t$ such that $|t|$ is small enough.
Since $\sin$ and $\cos$ are linearly independent functions on any non-empty open real interval,
the above relation implies that
\[
-\Gamma^1_{(1,2)}(x)+\hat{\Gamma}^1_{(1,2)}(\hat{x})=0 \\
\Gamma^1_{(2,3)}(x)-\hat{\Gamma}^1_{(2,3)}(\hat{x})=0,
\]
which establishes the claim.

We may now finish the proof of the proposition.
Indeed, if $\Gamma^1_{(2,3)}\neq 0$ on $U$,
Proposition \ref{pr:special_3D} implies that $\Gamma^1_{(2,3)}=:\beta$ is constant and $\Gamma^1_{(1,2)}=0$ on $U$.
If $\hat{x}$ belongs to the set $\pi_{Q,\hat{M}}(O_1')$, which is open in $\hat{M}$,
there is a $q=(x,\hat{x};A)\in O_1'$ where $(x,\hat{x})\in U\times\hat{U}$, by the definition of $O_1'$.
Then what was shown above implies
\[
\hat{\Gamma}^1_{(1,2)}(\hat{x})=\Gamma^1_{(1,2)}(x)=0,
\quad
\hat{\Gamma}^1_{(2,3)}(\hat{x})=\Gamma^1_{(2,3)}(x)=\beta.
\]
Thus shrinking $\hat{U}$ if necessary, this shows that
$\hat{\Gamma}^1_{(1,2)}$ vanishes on $\hat{U}$
and $\hat{\Gamma}^1_{(2,3)}$ is constant $=\beta$ on $\hat{U}$.
We conclude that $(U,g|_U)$ and $(\hat{U},\hat{g}|_{\hat{U}})$
both belong to class $\mc{M}_{\beta}$
and we are in case (i).

Similarly, if  $\Gamma^1_{(2,3)}=0$ on $U$,
the above argument implies that, after taking smaller $\hat{U}$,
$\hat{\Gamma}^1_{(2,3)}=0$ on $\hat{U}$.
Proposition \ref{pr:special_3D} implies that there is, taking smaller $U,\hat{U}$ if needed,
open interval $I=\hat{I}\subset\R$,
smooth functions $f,\hat{f}:I=\hat{I}\to\R$, 2-dimensional Riemannian manifolds
$(N,h)$, $(\hat{N},\hat{h})$ and isometries $F:(I\times N,h_f)\to (U,g|_U)$,
$\hat{F}:(\hat{I}\times\hat{N},\hat{h}_{\hat{f}})\to \hat{U}$ 
such that
\[
\frac{f'(r)}{f(r)}&=\Gamma^1_{(1,2)}(F(r,y)),\quad \forall (r,y)\in I\times N \\
\frac{\hat{f}'(\hat{r})}{\hat{f}(\hat{r})}&=\hat{\Gamma}^1_{(1,2)}(\hat{F}(\hat{r},\hat{y})),
\quad \forall (\hat{r},\hat{y})\in \hat{I}\times \hat{N}.
\]
Clearly we may assume that $0\in I=\hat{I}$
and $F(0,y_1)=x_1$, $\hat{F}(0,\hat{y}_1)=\hat{x}_1$
for some $y_1\in N$, $\hat{y}_1\in\hat{N}$.

Since $t\mapsto (t,y_1)$ and $t\mapsto (t,\hat{y}_1)$ are geodesics in $(I\times N,h_f)$, $(\hat{I}\times \hat{N},\hat{h}_{\hat{f}})$,
respectively,
$\gamma(t):=F(t,y_1)$ and $\hat{\gamma}(t)=\hat{F}(t,\hat{y}_1)$
are geodesics on $M$ and $\hat{M}$.
In addition,
\[
\hat{\gamma}'(0)=\hat{E}_2|_{\hat{x}_1}=A_1E_2|_{x_1}=A_1\gamma'(0),
\]
so $\hat{\gamma}(t)=\hat{\gamma}_{\RDist}(\gamma,q_1)(t)$ for all $t$.
This means that
\[
(F(t,y_1),\hat{F}(t,\hat{y}_1))=(\gamma(t),\hat{\gamma}(t))\in \pi_Q(O_1')
\]
and therefore
\[
\frac{f'(t)}{f(t)}=\Gamma^1_{(1,2)}(F(t,y_1))
=\hat{\Gamma}^1_{(1,2)}(\hat{F}(t,\hat{y}_1))
=\frac{\hat{f}'(t)}{\hat{f}(t)},
\]
for all $t\in I=\hat{I}$.
This shows that we belong to case (ii) 
and allows us to conclude the proof of the proposition.
\end{proof}

We have studied the case where $q$ belongs to $O_1\cup O_2$.
As for the points of $O_0$, one uses Corollary \ref{cor:weak_ambrose} and Remark \ref{re:weak_ambrose}
to conclude that for every $q_0=(x_0,\hat{x}_0;A_0)\in O_0$,
there are open neighbourhoods $U\ni x_0$ and $\hat{U}\ni \hat{x}_0$
such that $(U,g|_U)$ and $(\hat{U},\hat{g}|_{\hat{U}})$ are locally isometric.
With the choice of the set $O$ as the union of $O_0\cup O_1^\circ\cup O_2$, (where $O_1^\circ$ was introduced in Proposition \ref{pr:Rol1:main}), one concludes the proof of Theorem \ref{th:3D-1}.

%%%%%%%%%%%%%%%%%%%%%%%%%%%%%%
\subsection{Proof of Theorem \ref{th:3D-2}}
%%%%%%%%%%%%%%%%%%%%%%%%%%%%%%
The proof of the theorem only concerns Items $(b)$ and $(c)$, which are treated separately in two subsubsections.

\subsubsection{Case where both Manifolds are of Class $\mc{M}_{\beta}$}
Consider two manifolds $(M,g)$ and $(\hat{M},\hat{g})$ of class $\mc{M}_\beta$, 
$\beta\geq 0$ and oriented orthonormal frames $E_1,E_2,E_3$ and $\hat{E}_1,\hat{E}_2,\hat{E}_3$ which are adapted frames for  of $(M,g)$ and $(\hat{M},\hat{g})$ respectively. We will prove that in this situation, the rolling problem is not completely controllable.

We define on $Q$ two subsets
\[
Q_0:=&\{q=(x,\hat{x};A)\in Q\ |\ AE_2\neq \pm \hat{E}_2\} \\
Q_1:=&\{q=(x,\hat{x};A)\in Q\ |\ AE_2=\pm \hat{E}_2\}.
\]

\begin{proposition}\label{pr:mbeta-1}
Let $(M,g)$, $(\hat{M},\hat{g})$ be of class $\mc{M}_{\beta}$ for $\beta\in\R$.
Then for any $q_0=(x_0,\hat{x}_0;A_0)\in Q_1$
one has $\mc{O}_{\RDist}(q_0)\subset Q_1$.
Moreover, $Q_1$ is a closed 7-dimensional submanifold of $Q$
and hence in particular $\dim \mc{O}_{\RDist}(q_0)\leq 7$.
\end{proposition}

\begin{proof}
Define $h_1,h_2:Q\to\R$ by
\[
h_1(q)=\hat{g}(AE_1,\hat{E}_2),\quad h_2(q)=\hat{g}(AE_3,\hat{E}_2),
\]
when $q=(x,\hat{x};A)\in Q$.
It is clear that if $h=(h_1,h_2):Q\to\R^2$, then $Q_1=h^{-1}(0)$.
We will first show that $h$ is regular at the points of $Q_1$, which then implies that $Q_1$
is a closed submanifold of $Q$ of codimension 2 i.e. $\dim Q_1=7$ as claimed.

Before proceeding, we divide $Q_1$ into two disjoint subsets
\[
Q_1^+=&\{q=(x,\hat{x};A)\in Q\ |\ AE_2=+\hat{E}_2\} \\
Q_1^-=&\{q=(x,\hat{x};A)\in Q\ |\ AE_2=-\hat{E}_2\},
\]
whence $Q=Q_1^+\cup Q_1^-$. These are the components of $Q$
and we prove the claims only for $Q_1^+$, the considerations for $Q_1^-$ being completely similar.

First, since for every $q=(x,\hat{x};A)\in Q_1^+$ one has $AE_2=\hat{E}_2$,
it follows that $AE_1,AE_3\in \hat{E}_2^\perp$ and hence there is a smooth $\phi:Q_1^+\to\R$ such that
\[
AE_1&=\cos(\phi)\hat{E}_1+\sin(\phi)\hat{E}_3=:\hat{X}_A \\
AE_3&=-\sin(\phi)\hat{E}_1+\cos(\phi)\hat{E}_3=:\hat{Z}_A.
\]
In the subsequent computations we shorten the notation as $c_{\phi}=\cos(\phi(q))$,
$s_{\phi}=\sin(\phi(q))$.

We have for $q=(x,\hat{x};A)\in Q_1^+$,
\[
\nu(A\star E_3)|_q h_1&=\hat{g}(A(\star E_3)E_1,\hat{E}_2)=\hat{g}(AE_2,\hat{E}_2)=1 \\
\nu(A\star E_1)|_q h_1&=\hat{g}(A(\star E_1)E_1,\hat{E}_2)=0 \\
\nu(A\star E_3)|_q h_2&=\hat{g}(A(\star E_3)E_3,\hat{E}_2)=0 \\
\nu(A\star E_1)|_q h_2&=\hat{g}(A(\star E_1)E_3,\hat{E}_2)=-\hat{g}(AE_2,\hat{E}_2)=-1,
\]
which shows that indeed $h$ is regular on $Q_1^+$.

Next we show that the vectors $\LRD(E_1)|_q,\LRD(E_2)|_q,\LRD(E_3)|_q$
are all tangent to $Q_1^+$ and hence to $Q_1$. This is equivalent to the fact that $\LRD(E_i)|_qh=0$ for $i=1,2,3$.

We compute for $q=(x,\hat{x};A)\in Q_1^+$, recalling that $AE_1=\hat{X}_A$, $AE_2=\pm \hat{E}_2$, $AE_3=\hat{Z}_A$,
\[
\LRD(E_1)|_q h_1=&\hat{g}(A\nabla_{E_1} E_1,\hat{E}_2)+\hat{g}(AE_1,\hat{\nabla}_{\hat{X}_A} \hat{E}_2) \\
=&-\Gamma^1_{(3,1)}\hat{g}(AE_3,\hat{E}_2)+\hat{g}(\hat{X}_A,\beta c_{\phi} \hat{E}_3-\beta s_{\phi} \hat{E}_1) \\
=&-\Gamma^1_{(3,1)}\hat{g}(\hat{Z}_A,\hat{E}_2)+\hat{g}(\hat{X}_A,\beta \hat{Z}_A)=0 \\
\LRD(E_1)|_q h_2=&\hat{g}(A\nabla_{E_1} E_3,\hat{E}_2)+\hat{g}(AE_3,\hat{\nabla}_{\hat{X}_A} \hat{E}_2) \\
=&\hat{g}(A(\Gamma^1_{(3,1)}E_1-\beta E_2),\hat{E}_2)+\hat{g}(\hat{Z}_A,\beta \hat{Z}_A) \\
=&\hat{g}(\Gamma^1_{(3,1)}\hat{X}_A-\beta \hat{E}_2,\hat{E}_2)+\beta=0 \\
\LRD(E_2)|_q h_1=&\hat{g}(A\nabla_{E_2} E_1,\hat{E}_2)+\hat{g}(AE_1,\hat{\nabla}_{\hat{E}_2} \hat{E}_2)
=-\Gamma^2_{(3,1)}\hat{g}(\hat{Z}_A,\hat{E}_2)+0=0 \\
\LRD(E_2)|_q h_2=&\hat{g}(A\nabla_{E_2} E_3,\hat{E}_2)+\hat{g}(AE_3,\hat{\nabla}_{\hat{E}_2} \hat{E}_2)
=\Gamma^2_{(3,1)}\hat{g}(\hat{X}_A,\hat{E}_2)+0=0 \\
\LRD(E_3)|_q h_1=&\hat{g}(A\nabla_{E_3} E_1,\hat{E}_2)+\hat{g}(AE_1,\hat{\nabla}_{\hat{Z}_A} \hat{E}_2) \\
=&\hat{g}(A(\beta E_2-\Gamma^3_{(3,1)}E_3),\hat{E}_2)+\hat{g}(\hat{X}_A,-\beta s_{\phi} \hat{E}_3-\beta c_{\phi} \hat{E}_1) \\
=&\hat{g}(\beta \hat{E}_2-\Gamma^3_{(3,1)}\hat{Z}_A),\hat{E}_2)-\beta \hat{g}(\hat{X}_A,\hat{X}_A)
=\beta-\beta=0 \\
\LRD(E_3)|_q h_2=&\hat{g}(A\nabla_{E_3} E_3,\hat{E}_2)+\hat{g}(AE_3,\hat{\nabla}_{\hat{Z}_A} \hat{E}_2) \\
=&\Gamma^3_{(3,1)}\hat{g}(AE_1,\hat{E}_2)+\hat{g}(\hat{Z}_A,-\beta \hat{X}_A)
=\Gamma^3_{(3,1)}\hat{g}(\hat{X}_A,\hat{E}_2)+0=0.
\]

Thus $\LRD(E_1)|_q,\LRD(E_2)|_q,\LRD(E_3)|_q$ and hence $\RDist$ is tangent
to $Q_1^+$, which implies that any orbit $\mc{O}_{\RDist}(q)$
through a point $q\in Q_1^+$ is also a subset of $Q_1^+$.
The same observation obviously holds for $Q_1^-$ and therefore
the proof is complete.
\end{proof}

Next we will show that if $(M,g)$ and $(\hat{M},\hat{g})$ are of class $\mc{M}_\beta$
with the same $\beta\in\R$,
then the rolling problem of $M$ against $\hat{M}$ is not controllable.

We begin by completing the proposition in the sense that
we show that the orbit can be of dimension exactly 7,
if $(M,g)$, $(\hat{M},\hat{g})$ are not locally isometric.

\begin{proposition}\label{pr:mbeta-2}
Let $(M,g)$, $(\hat{M},\hat{g})$
be Riemannian manifolds of class $\mc{M}_\beta$, $\beta\neq 0$,
and let $q_0=(x_0,\hat{x}_0;A_0)\in Q_1$.
Then if $\mc{O}_{\RDist}(q_0)$ is not an integral manifold of $\RDist$,
one has $\dim\mc{O}_{\RDist}(q_0)=7$.
\end{proposition}

\begin{proof}
Without loss of generality, we may assume that $A_0E_2|_{x_0}=\hat{E}_2|_{\hat{x}_0}$.
Then Proposition \ref{pr:mbeta-1} and continuity imply that $AE_2|_{x}=\hat{E}_2|_{\hat{x}}$ for all $q=(x,\hat{x};A)\in \mc{O}_{\RDist}(q_0)$
and hence that $AE_1|_x,AE_3|_x\in \spn \{\hat{E}_1|_{\hat{x}},\hat{E}_3|_{\hat{x}}\}$.
This combined with Lemma \ref{le:mbeta-2} implies
\[
\widetilde{\Rol}_q(\star E_1)=0,\quad
\widetilde{\Rol}_q(\star E_2)=(-K_2(x)+\hat{K}_2(\hat{x}))(\star E_2),\quad
\widetilde{\Rol}_q(\star E_3)=0,
\]
for $q=(x,\hat{x};A)\in \mc{O}_{\RDist}(q_0)$,
where $-K_2(x),-\hat{K}_2(\hat{x})$ are eigenvalues of $R|_x,\hat{R}|_{\hat{x}}$
corresponding to eigenvectors $\star E_2|_x,\hat{\star}\hat{E}_2|_{\hat{x}}$, respectively.

Since $\mc{O}_{\RDist}(q_0)$ is not an integral manifold of $\RDist$
there is a point $q_1=(x_1,\hat{x}_1;A_1)\in  \mc{O}_{\RDist}(q_0)$
such that $-K_2(x_1)+\hat{K}_2(\hat{x}_1)\neq 0$
(see Corollary \ref{cor:weak_ambrose} and Remark \ref{re:weak_ambrose}).
Then there are open neighbourhoods $U$ and $\hat{U}$ of $x_1$ and $\hat{x}_1$
in $M$ and $\hat{M}$, respectively,
such that $-K_2(x)+\hat{K}_2(\hat{x})\neq 0$ for all $x\in U,\hat{x}\in\hat{U}$.

Define $O:=\pi_{Q}^{-1}(U\times\hat{U})\cap \mc{O}_{\RDist}(q_0)$,
which is an open subset of $\mc{O}_{\RDist}(q_0)$ containing $q_0$.
Because for all $q=(x,\hat{x};A)\in O$ one has
$\nu(\Rol_q(\star E_2))|_q\in T|_q \mc{O}_{\RDist}(q_0)$
and $-K_2(x)+\hat{K}_2(\hat{x})\neq 0$, it follows that
\[
\nu(A\star E_2)|_q\in T|_q \mc{O}_{\RDist}(q_0),\quad
\forall q=(x,\hat{x};A)\in O.
\]
Moreover, $\Gamma^1_{(1,2)}=0$ and $\Gamma^1_{(2,3)}=\beta$ is constant
and hence one may use Proposition \ref{pr:3D-2}, case (i), to conclude
that the vector fields defined by
\[
L_1|_q=&\LNSD(E_1)|_q-\beta \nu(A\star E_1)|_q \\
L_2|_q=&\beta \LNSD(E_2)|_q \\
L_3|_q=&\LNSD(E_3)|_q-\beta \nu(A\star E_2)|_q
\]
are tangent to the orbit $\mc{O}_{\RDist}(q_0)$.
Therefore the linearly independent vectors
\[
\LRD(E_1)|_q,\LRD(E_2)|_q,\LRD(E_3)|_q,\nu(A\star E_2)|_q,L_1|_q,L_2|_q,L_3|_q
\]
are tangent to $\mc{O}_{\RDist}(q_0)$ for all $q\in O$,
which implies that $\dim\mc{O}_{\RDist}(q_0)\geq 7$.
By Proposition \ref{pr:mbeta-1} we conclude that $\dim\mc{O}_{\RDist}(q_0)=7$
\end{proof}

By the previous proposition,
we are left to study the case of an $\RDist$-orbit which passes
through a point $q_0\in Q_0$.

\begin{proposition}\label{pr:mbeta-3}
Let $(M,g)$ and $(\hat{M},\hat{g})$ be two Riemannian manifolds of class $\mc{M}_\beta$, $\beta\neq 0$,
and let $q_0=(x_0,\hat{x}_0;A_0)\in Q_0$.
Write $M^\circ:=\pi_{Q,M}(\mc{O}_{\RDist}(q_0))$, $\hat{M}^\circ:=\pi_{Q,\hat{M}}(\mc{O}_{\RDist}(q_0))$,
which are open  connected subsets of $M$, $\hat{M}$.
Then we have:
\begin{itemize}
\item[(i)] If only one of $(M^\circ,g)$ or $(\hat{M}^\circ,\hat{g})$
has constant curvature, then $\dim\mc{O}_{\RDist}(q_0)=7$.
\item[(ii)] Otherwise $\dim\mc{O}_{\RDist}(q_0)=8$.
\end{itemize}
\end{proposition}

\begin{proof}
As before, we let $E_1,E_2,E_3$ and $\hat{E}_1,\hat{E}_2,\hat{E}_3$ to be some adapted frames of $(M,g)$
and $(\hat{M},\hat{g})$ respectively.
We will not fix the choice of $q_0$ in $Q_0$ (and hence do not define $M^\circ$, $\hat{M}^\circ$)
until the last half of the proof (where we introduce the sets $M_0,M_1,\hat{M}_0,\hat{M}_1$ below).
Notice that Proposition \ref{pr:mbeta-1} implies that $\mc{O}_{\RDist}(q_0)\subset Q_0$,
for every $q_0\in Q_0$.

The fact that $AE_2|_{x}\neq \pm \hat{E}_2|_{\hat{x}}$ for $q=(x,\hat{x};A)\in Q_0$
is equivalent
to the fact that the intersection $(AE_2^\perp|_{x})\cap \hat{E}_2^\perp|_{\hat{x}}$
is non-trivial for all $q=(x,\hat{x};A)\in Q_0$.
Therefore, for a small enough open neighbourhood $\tilde{O}$ of $q_0$ inside $Q_0$,
we may find a smooth functions $\theta,\hat{\theta}:\tilde{O}\to\R$
such that this intersection is spanned by
$AZ_A=\hat{Z}_A$, where
\[
Z_A:=&-\sin(\theta(q))E_1|_x+\cos(\theta(q))E_3|_x \\
\hat{Z}_A:=&-\sin(\hat{\theta}(q))\hat{E}_1|_{\hat{x}}+\cos(\hat{\theta}(q))\hat{E}_3|_{\hat{x}}.
\]
We also define
\[
X_A:=&\cos(\theta(q))E_1|_{x}+\sin(\theta(q))E_3|_{x} \\
\hat{X}_A:=&\cos(\hat{\theta}(q))\hat{E}_1|_{\hat{x}}+\sin(\hat{\theta}(q))\hat{E}_3|_{\hat{x}}.
\]

To unburden the formulas, we write from now on usually $s_\tau:=\sin(\tau(q))$, $c_\tau:=\cos(\tau(q))$
if $\tau:\tilde{O}\to\R$ is some function and the point $q\in \tilde{O}$ is clear from the context.

Since $X_A,E_2|_x,Z_A$ (resp. $\hat{X}_A,\hat{E}_2|_{\hat{x}},\hat{Z}_A$)
form an orthonormal frame for every $q=(x,\hat{x};A)\in \tilde{O}$
and because $A(Z_A^\perp)=\hat{Z}_A^\perp$, 
it follows that there is a smooth $\phi:O'\to\R$ such that
\[
AX_A=&c_{\hat{\phi}} \hat{X}_A+s_{\hat{\phi}} \hat{E}_2=c_{\hat{\phi}}(c_{\hat{\theta}}\hat{E}_1+s_{\hat{\theta}}\hat{E}_3)+s_{\hat{\phi}}\hat{E}_2 \\
AE_2=&-s_{\hat{\phi}} \hat{X}_A+c_{\hat{\phi}} \hat{E}_2=-s_{\hat{\phi}}(c_{\hat{\theta}}\hat{E}_1+s_{\hat{\theta}}\hat{E}_3)+c_{\hat{\phi}}\hat{E}_2 \\
AZ_A=&\hat{Z}_A.
\]
In particular,
\[
\hat{g}(AZ_A,\hat{E}_2)=0,
\]
for all $q=(x,\hat{x};A)\in \tilde{O}$.

Notice that for all $q=(x,\hat{x};A)\in\tilde{O}$, since $A\star Z_A=\hat{\star} \hat{Z}_A A$,
\[
\widetilde{\Rol}_q(\star Z_A)=R(\star Z_A)-A^{\ol{T}}\hat{R}(\hat{\star} \hat{Z}_A)A
=-K\star Z_A+KA^{\ol{T}}\hat{\star} \hat{Z}_AA=0
\]
and hence, since $\widetilde{\Rol}_q:\wedge^2 T|_x M\to \wedge^2 T|_x M$ is a symmetric map,
\[
\widetilde{\Rol}_q(\star X_A)=&-K_1^\Rol(q)\star X_A-\alpha \star E_2 \\
\widetilde{\Rol}_q(\star E_2)=&-\alpha\star X_A-K_2^\Rol(q) \star E_2,
\]
for some smooth functions $K_1^\Rol,K_2^\Rol,\alpha:\tilde{O}\to\R$.

We begin by considering the smooth 5-dimensional distribution $\Delta$ on the open subset $\tilde{O}$
of $Q_0$ spanned by
\[
\LRD(E_1)|_q,\LRD(E_2)|_q,\LRD(E_3)|_q,\nu(A\star E_2)|_q,\nu(A\star X_{A})|_q.
\]
What will be shown is that $\Lie(\Delta)$ spans at every point $q\in \mc{O}$
a smooth distribution $\Lie(\Delta)|_q$ of dimension 8
which, by construction, is then involutive.
We consider $\VF_{\RDist}^k, \VF_{\Delta}^k, \Lie(\Delta)$ as $\Cinf(\tilde{O})$-modules.

Since $X_A=c_\theta E_1+s_\theta E_3$,
in order to compute brackets of the first 4 vector fields above against $\nu(A\star X_{A})|_q$,
we need to know some derivatives of $\theta$.
This will be done next.
We begin by computing
\[
\LRD(X_A)|_qZ_{(\cdot)}=&(-\LRD(X_A)|_q\theta+c_\theta\Gamma^1_{(3,1)}+s_\theta\Gamma^3_{(3,1)})X_A-\beta E_2 \\
\LRD(E_2)|_q Z_{(\cdot)}=&(-\LRD(E_2)|_q\theta+\Gamma^2_{(3,1)})X_A \\
\LRD(Z_A)|_qZ_{(\cdot)}=&(-\LRD(Z_A)|_q\theta-s_\theta\Gamma^1_{(3,1)}+c_\theta\Gamma^3_{(3,1)})X_A.
\]
Differentiating $\hat{g}(AZ_A,\hat{E}_2)=0$ with respect to $\LRD(X_A)|_q$ gives,
\[
0=&\hat{g}(A\LRD(X_A)|_q Z_{(\cdot)},\hat{Y})+\hat{g}(AZ_A,\hat{\nabla}_{AX_A} \hat{E}_2) \\
=&\hat{g}(A(-\LRD(X_A)|_q\theta+c_\theta\Gamma^1_{(3,1)}+s_\theta\Gamma^3_{(3,1)})X_A-\beta E_2),\hat{E}_2) \\
&+\hat{g}(AZ_A,c_{\hat{\phi}}c_{\hat{\theta}}\beta\hat{E}_3-c_{\hat{\phi}}s_{\hat{\theta}}\beta \hat{E}_1) \\
=&s_{\hat{\phi}}(-\LRD(X_A)|_q\theta+c_\theta\Gamma^1_{(3,1)}+s_\theta\Gamma^3_{(3,1)})-\beta c_{\hat{\phi}}
+c_{\hat{\phi}}s_{\hat{\theta}}^2\beta+c_{\hat{\phi}}c_{\hat{\theta}}^2\beta \\
=&s_{\hat{\phi}}(-\LRD(X_A)|_q\theta+c_\theta\Gamma^1_{(3,1)}+s_\theta\Gamma^3_{(3,1)}).
\]
Since $s_{\hat{\phi}}\neq 0$ (because otherwise $AE_2=\pm \hat{E}_2$),
we get
\[
\LRD(X_A)|_q\theta=c_\theta\Gamma^1_{(3,1)}+s_\theta\Gamma^3_{(3,1)}.
\]
In a similar way, differentiating $\hat{g}(AZ_A,\hat{E}_2)=0$ with respect to $\LRD(Z_A)|_q$,
$\LRD(E_2)|_q$,
one finds
\[
\LRD(Z_A)|_q\theta&=-s_\theta\Gamma^1_{(3,1)}+c_\theta\Gamma^3_{(3,1)} \\
\LRD(E_2)|_q\theta&=-\beta+\Gamma^2_{(3,1)}.
\]
Finally, applying $\nu(A\star E_2)|_q$ on the equation $\hat{g}(AZ_A,\hat{E}_2)=0$
gives,
\[
0=&\hat{g}(\nu(A\star E_2)|_q((\cdot)Z_{(\cdot)},\hat{E}_2)
=\hat{g}(A(\star E_2)Z_A-(\nu(A\star E_2)|_q\theta)AX_A,\hat{E}_2) \\
=&(1-\nu(A\star E_2)|_q\theta)\hat{g}(AX_A,\hat{E}_2)
\]
and since $\hat{g}(AX_A,\hat{E}_2)=s_{\hat{\phi}}\neq 0$,
\[
\nu(A\star E_2)|_q\theta=1.
\]

Using the definition of $X_A$ and $Z_A$, we
may now summarize
\[
\LRD(E_1)|_q\theta=\Gamma^1_{(3,1)},\quad &
\LRD(E_2)|_q\theta=-\beta+\Gamma^2_{(3,1)} \\
\LRD(E_3)|_q\theta=\Gamma^3_{(3,1)},\quad &
\nu(A\star E_2)|_q\theta=1.
\]

By Proposition \ref{pr:3D-2} and the fact that $\beta\neq 0$, we see that $\VF_{\Delta}^2$ contains
the vector fields given by
\[
L_1|_q=&\LNSD(E_1)|_q-\beta \nu(A\star E_1)|_q \\
\tilde{L}_2|_q=&\LNSD(E_2)|_q \\
L_3|_q=&\LNSD(E_3)|_q-\beta \nu(A\star E_3)|_q \\
\]
i.e. $\tilde{L}_2=\frac{1}{\beta}L_2$.
Computing
\[
[\LRD(E_1),\nu((\cdot)\star X_{(\cdot)})]|_q
=&-s_\theta\LRD(E_2)|_q+s_\theta \tilde{L}_2|_q-s_\theta \beta\nu(A\star E_2)|_q \\
[\LRD(E_2),\nu((\cdot)\star X_{(\cdot)})]|_q
=&-\LRD(Z_A)|_q-s_\theta L_1|_q+c_\theta L_3|_q \\
[\LRD(E_3),\nu((\cdot)\star X_{(\cdot)})]|_q
=&c_\theta\LRD(E_2)|_q-c_{\theta}\tilde{L}_2|_q-c_\theta \beta\nu(A\star E_2)|_q \\
[\nu((\cdot)\star E_2),\nu((\cdot)\star X_{(\cdot)})]|_q
=&0
\]
and since one also has
\[
[\LRD(E_1),\LRD(E_2)]|_q=&\LRD([E_1,E_2])|_q-s_{\theta} K_1^\Rol \nu(A\star X_A)|_q-s_{\theta}\alpha \nu(A\star E_2)|_q \\
[\LRD(E_2),\LRD(E_3)]|_q=&\LRD([E_2,E_3])|_q-c_{\theta} K_1^\Rol \nu(A\star X_A)|_q-c_{\theta}\alpha \nu(A\star E_2)|_q  \\
[\LRD(E_3),\LRD(E_1)]|_q=&\LRD([E_3,E_1])|_q-\alpha \nu(A\star X_A)|_q-K_2^\Rol\nu(A\star E_2)|_q,
\]
we see using in addition Proposition \ref{pr:3D-2}, case (ii) (the first three Lie brackets there),
that $\VF_{\Delta}^2$ is generated by the following 8 linearly independent vector fields
defined on $\tilde{O}$ by
\[
\LRD(E_1)|_q,\LRD(E_2)|_q,\LRD(E_3)|_q,\nu(A\star E_2)|_q,\nu(A\star X_{A})|_q,
L_1|_q,\tilde{L}_2|_q,L_3|_q.
\]

We now proceed to show that $\Lie(\Delta)=\VF_{\Delta}^2$.
According to Proposition \ref{pr:3D-2} case (ii)
and previous computations,
we know that all the brackets between $\LRD(E_1)$, $\LRD(E_2)$, $\LRD(E_3)$, $\nu((\cdot)\star E_2)$
and $L_1,L_3$ and also $[L_1,L_3]$ belong to $\VF_{\Delta}^2$,
so we are left to compute the bracket of $\nu((\cdot)\star X_{(\cdot)}),\tilde{L}_2$
against $L_1,L_3$
and also $\tilde{L}_2$ against $\LRD(E_1)|_q,\LRD(E_2)|_q,\LRD(E_3)|_q,\nu(A\star E_2)|_q,\nu((\cdot)\star X_{(\cdot)})|_q$.

To do that, we need to know more derivatives of $\theta$.
Since $[\LRD(E_1),\nu((\cdot)\star E_2)]=\LRD(E_3)|_q-L_3|_q$, we get
\[
L_3|_q\theta=&\LRD(E_3)|_q\theta-\LRD(E_1)|_q\big(\underbrace{\nu((\cdot)\star E_2)\theta}_{=1}\big)
+\nu(A\star E_2)|_q\big(\underbrace{\LRD(E_1)\theta}_{=\Gamma^1_{(3,1)}}\big)
=\Gamma^3_{(3,1)}
\]
and similarly, by using $[\LRD(E_3),\nu((\cdot)\star E_2)]=-\LRD(E_1)|_q+L_1|_q$,
\[
L_1|_q\theta=\Gamma^1_{(3,1)}.
\]
On the other hand
\[
\LNSD(E_2)|_q Z_{(\cdot)}=&(-\LNSD(E_2)|_q\theta+\Gamma^2_{(3,1)})X_A,
\]
and to compute $\tilde{L}_2|_q\theta=\LNSD(E_2)|_q\theta$, operate by $\LNSD(E_2)|_q$ onto equation $\hat{g}(AZ_A,\hat{E}_2)=0$
to get
\[
\tilde{L}_2|_q\theta=\Gamma^2_{(3,1)}.
\]
With these derivatives of $\theta$ being available, we easily see that
\[
[L_1,\nu((\cdot)\star X_{(\cdot)})]|_q=&0 \\
[L_1,\tilde{L}_2]|_q=&(\Gamma^2_{(3,1)}+\beta)L_3|_q \\
[L_3,\nu((\cdot)\star X_{(\cdot)})]|_q=&0 \\
[L_3,\tilde{L}_2]|_q=&-(\Gamma^2_{(3,1)}+\beta)L_1|_q \\
[\LRD(E_1),\tilde{L}_2]|_q=&\beta L_3|_q-\LRD(\nabla_{E_2} E_1)|_q \\
[\LRD(E_2),\tilde{L}_2]|_q=&0 \\
[\LRD(E_3),\tilde{L}_2]|_q=&-\beta L_1|_q-\LRD(\nabla_{E_2} E_3)|_q \\
[\nu((\cdot)\star E_2),\tilde{L}_2]|_q=&0 \\
[\nu((\cdot)\star X_{(\cdot)}),\tilde{L}_2]|_q=&0.
\]

Hence we have proved that $\VF_{\Delta}^2$ is involutive and hence
\[
\Lie(\Delta)=\VF_{\Delta}^2.
\]
There being 8 linearly independent generators for $\Lie(\Delta)=\VF_{\Delta}^2$,
we conclude that the distribution $\mc{D}$ spanned pointwise on $\tilde{O}$
by $\Lie(\Delta)$ is integrable by the theorem of Frobenius.

The choice of $q_0\in Q_0$ was arbitrary and
we see that we may build an 8-dimensional smooth involutive distribution $\mc{D}$
by the above construction on the whole $Q_0$.
Since $\RDist\subset \Delta\subset\mc{D}$, we have $\mc{O}_{\RDist}(q_0)\subset \mc{O}_{\mc{D}}(q_0)$
for all $q_0\in Q_0$ and thus $\dim \mc{O}_{\RDist}(q_0)\leq 8$.
We will show when the equality holds here
and show when actually $\dim \mc{O}_{\RDist}(q_0)=7$.

Define
\[
M_0=&\{x\in M\ |\ \beta^2\neq K_2(x)\} \\
M_1=&\{x\in M\ |\ \exists\ \textrm{open}\ V\ni x\ \textrm{s.t.}\ \forall x'\in V,\ \beta^2=K_2(x')\} \\
\hat{M}_0=&\{\hat{x}\in \hat{M}\ |\ \beta^2\neq \hat{K}_2(\hat{x})\} \\
\hat{M}_1=&\{\hat{x}\in \hat{M}\ |\ \exists\ \textrm{open}\ \hat{V}\ni \hat{x}\ \textrm{s.t.}\ \forall \hat{x}'\in \hat{V},\ \beta^2=\hat{K}_2(\hat{x}')\},
\]
and notice that $M_0\cup M_1$ (resp. $\hat{M}_0\cup \hat{M}_1$)
is an open dense subset of $M$ (resp. $\hat{M}$).
At this point we also fix $q_0\in Q_0$ and write $M^\circ=\pi_{Q,M}(\mc{O}_{\RDist}(q_0))$,
$\hat{M}^\circ=\pi_{Q,\hat{M}}(\mc{O}_{\RDist}(q_0))$ as in the statement
of this proposition.

Let $q_1=(x_1,\hat{x}_1;A_1)\in \pi_{Q}^{-1}(M_0\times \hat{M}_0)\cap Q_0$.
Take an open neighbourhood $\tilde{O}$ of $q_1$ in $Q_0$ as above (now for $q_1$ instead of $q_0$
which we fixed)
such that $\pi_Q(\tilde{O})\subset M_0\times \hat{M}_0$,
and introduce on $\tilde{O}$ the vectors $X_A,Z_A,\hat{X}_A,\hat{Z}_A$
along with the angles $\theta,\hat{\theta},\hat{\phi}$, again as above.
Then one computes for $q\in \tilde{O}$,
\[
\qmatrix{\widetilde{\Rol}_q(\star X_A) \cr
\widetilde{\Rol}_q(\star E_2)}
=&\qmatrix{s_{\hat{\phi}}^2(-\beta^2+\hat{K}_2) & c_{\hat{\phi}}s_{\hat{\phi}}(-\beta^2+\hat{K}_2) \cr
(-\beta^2+\hat{K}_2)s_{\hat{\phi}}c_{\hat{\phi}} & -K_2+s_{\hat{\phi}}^2\beta^2+c_{\hat{\phi}}^2\hat{K}_2}
\qmatrix{\star X_A \cr \star E_2} \\
\widetilde{\Rol}_q(\star Z_A)=&0.
\]
The determinant $d(q)$ of the above matrix equals
\[
d(q)=-s_{\hat{\phi}}^2(-K_2+\beta^2)(-\hat{K}_2+\beta^2),
\]
so $d(q)\neq 0$ since $q\in \tilde{O}\subset \pi_{Q}^{-1}(M_0\times \hat{M}_0)\cap Q_0$.
Since $\nu(\Rol(\star E_2)(A))|_{q_1}\in T|_{q_1}\mc{O}_{\RDist}(q_1)$,
we obtain that $\nu(A_1\star E_2)|_{q_1}\in T|_{q_1}\mc{O}_{\RDist}(q_1)$.

If $q_1=(x_1,\hat{x}_1;A_1)\in \pi_{Q}^{-1}(M_0\times \ol{\hat{M}_0})\cap Q_0$,
then one can take a sequence $q_n'=(x_n',\hat{x}_n';A_n')\in \mc{O}_{\RDist}(q_1)$
such that $q_n'\to q_1$ while $\hat{x}_n'\in \hat{M}_0$.
Since $M_0$ and $Q_0$ are open, we have for large enough $n$ that
$q_n'\in \pi_{Q}^{-1}(M_0\times \hat{M}_0)\cap Q_0$,
hence $\nu(A_n'\star E_2)|_{q_n'}\in T|_{q_n'} \mc{O}_{\RDist}(q_1)$
and by taking the limit as $n\to\infty$, we have $\nu(A_1\star E_2)|_{q_1}\in T|_{q_1} \mc{O}_{\RDist}(q_1)$.

Next suppose $q_1=(x_1,\hat{x}_1;A_1)\in \pi_{Q}^{-1}(M_0\times \hat{M}_1)\cap Q_0$.
Then $\widetilde{\Rol}_{q_1}(\star E_1)=\widetilde{\Rol}_{q_1}(\star E_3)=0$,
$\widetilde{\Rol}_{q_1}(\star E_2)=(-K_2(x_1)+\beta^2)\star E_2$ with $K_2(x_1)\neq \beta^2$
and hence $\nu(A\star E_2)|_{q_1}\in T|_{q_1}\mc{O}_{\RDist}(q_1)$.
Thus we have proven that
\[
\nu(A\star E_2)|_q\in T|_q\mc{O}_{\RDist}(q),\quad \forall q\in Q_0\cap \pi_Q^{-1}(M_0\times \hat{M}).
\]
Changing the roles of $M$ and $\hat{M}$
we also have
\[
\nu((\hat{\star} \hat{E}_2)A)|_q\in T|_q\mc{O}_{\RDist}(q),\quad \forall q\in Q_0\cap \pi_Q^{-1}(M\times \hat{M}_0).
\]

We define on $Q$ two 3-dimensional distributions $D,\hat{D}$:
for $q\in Q$ one defines $\hat{D}|_q$ to be the span of
\[
& \hat{K}_1|_q=\LNSD(AE_1)|_q+\beta \nu(A\star E_1)|_q \\
& \hat{K}_2|_q=\LNSD(AE_2)|_q \\
& \hat{K}_3|_q=\LNSD(AE_3)|_q+\beta \nu(A\star E_3)|_q
\]
and $D|_q$ to be the span of
\[
& K_1|_q=\LNSD(A^{\ol{T}}\hat{E}_1)|_q-\beta \nu((\hat{\star} \hat{E}_1)A)|_q \\
& K_2|_q=\LNSD(A^{\ol{T}}\hat{E}_2)|_q \\
& K_3|_q=\LNSD(A^{\ol{T}}\hat{E}_3)|_q-\beta \nu((\hat{\star} \hat{E}_3)A)|_q
\]
We claim that for any $q_1=(x_1,\hat{x}_1;A_1)\in Q$ and any smooth paths $\gamma:[0,1]\to M$,  $\hat{\gamma}:[0,1]\to \hat{M}$  with 
$\gamma(0)=x_1$, $\hat{\gamma}(0)=\hat{x}_1$
there are unique curves $\Gamma,\hat{\Gamma}:[0,1]\to Q$ of the same regularity as $\gamma,\hat{\gamma}$
such that $\Gamma$ is tangent to $D$, $\Gamma(0)=q_1$ and $\pi_{Q,M}(\Gamma(t))=\gamma$
and similarly $\hat{\Gamma}$ is tangent to $\hat{D}$, $\hat{\Gamma}(0)=q_1$ and $\pi_{Q,\hat{M}}(\hat{\Gamma}(t))=\hat{\gamma}$.
The key point here is that $\Gamma,\hat{\Gamma}$ are defined on $[0,1]$ and
not only on a smaller interval $[0,T]$ with $T\leq 1$.
We write these curves as $\Gamma=\Gamma(\gamma,q_1)$ and $\hat{\Gamma}=\hat{\Gamma}(\hat{\gamma},q_1)$,
respectively. Notice that since $(\pi_{Q,\hat{M}})_*D=0$ and $(\pi_{Q,M})_*\hat{D}=0$,
one has
\[
\pi_{Q,\hat{M}}(\Gamma(\gamma,q_1)(t))=\hat{x}_1,
\quad
\pi_{Q,M}(\hat{\Gamma}(\hat{\gamma},q_1)(t))=x_1,
\quad \forall t\in [0,1].
\]

We prove the above claim for $D$ only since the proof for $\hat{D}$ is similar.
Uniqueness and local existence are clear.
Take some extension of $\gamma$ to an interval $]-\epsilon,1+\epsilon[=:I$
and write $\Gamma_1:=\Gamma(\gamma,q_1)$.
Consider a trivialization (which is global since we assumed the frames $E_i$, $\hat{E}_i$, $i=1,2,3$
to be global) of $\pi_{Q}$ given by
\[
\Phi:Q\to M\times\hat{M}\times \SO(n);
\quad (x,\hat{x};A)\mapsto (x,\hat{x},\mc{M}_{F,\hat{F}}(A)),
\]
where $F=(E_1,E_2,E_3)$, $\hat{F}=(\hat{E}_1,\hat{E}_2,\hat{E}_3)$.

Clearly for every $(s,C)\in I\times \SO(n)$ one has
\[
\Phi(\Gamma(\gamma(s+\cdot),\Phi^{-1}(\gamma(s),\hat{x}_1;C))(t))=(\gamma(s+t),\hat{x}_1,B_{(s,C)}(t)),
\]
where $B_{(s,C)}(t)\in \SO(n)$ and $t$ in some small open interval containing $0$.
On $I\times\SO(n)$ we define a vector field
\[
\mc{X}|_{(s,C)}=(\pa{t},\dot{B}_{(s,C)}(0)).
\]
If $\Phi(\Gamma(\gamma,q_1)(t))=(\gamma(t),\hat{x}_1;C_1(t))$, then since
\[
\dif{s}\Phi(\Gamma_1(s))=&\dif{t}\big|_0\Phi(\Gamma(\gamma,q_1)(t+s))
=\dif{t}\big|_0 \Phi\big(\Gamma(\gamma(s+\cdot),\Gamma(\gamma,q_1)(s))(t)\big) \\
=&\dif{t}\big|_0 (\gamma(t+s),\hat{x}_1,B_{(s,C_1(s))}(t))
=(\dot{\gamma}(s),0,(\pr_2)_*\mc{X}|_{(s,C_1(s))}),
\]
we see that $s\mapsto (s,(\pr_3\circ \Phi\circ \Gamma_1)(s))=(s,C_1(s))$
is the integral curve of $\mc{X}$ starting from $(0,C_1(0))$.
Conversely, if $\Lambda_1(t)=(t,C(t))$ is the integral curve of $\mc{X}$ starting from $(0,C_1(0))$,
then $\tilde{\Gamma}_1(t):=\Phi^{-1}(\gamma(t),\hat{x}_1,C(t))$ gives an integral curve of $D$
starting from $q_1$ and $\pi_{Q,M}(\tilde{\Gamma}_1(t))=\gamma(t)$.

Hence the maximal positive interval of definition of $\Gamma_1$
is the same as that of the integral curve $\Lambda_1$ of $\mc{X}$ starting from $(0,C_1)$.
If it is of the form $[0,t_0[$ for some $t_0<1+\epsilon$,
then, because $[0,1]\times\SO(n)$ is a compact subset of $I\times \SO(n)$,
there is a $t_1 \in [0,t_0[$ with $\Lambda_1(t_1)\notin [0,1]\times\SO(n)$
i.e. $t_1\notin [0,1]$ which is only possible if $t_1>1$,
and thus $t_0>1$.
We have shown that the existence of
$\Gamma_1(t)=\Gamma(\gamma,q_1)(t)$ is guaranteed on the whole interval $[0,1]$.

Since for all  $q\in Q_0\cap \pi_Q^{-1}(M_0\times \hat{M})$,
which is an open subset of $Q$, one has $\nu(A\star E_2)|_q\in T|_q\mc{O}_{\RDist}(q)$,
it follows from Proposition \ref{pr:3D-2}
that 
\[
L_1|_q&=\LNSD(E_1)|_q-\beta\nu(A\star E_1)|_q \\
\tilde{L}_2|_q&=\LNSD(E_2)|_q \\
L_3|_q&=\LNSD(E_3)|_q-\beta\nu(A\star E_3)|_q
\]
are tangent to the orbit $\mc{O}_{\RDist}(q)$
and hence so are $\LRD(E_1)|_q-L_1|_q=\hat{K}_1|_q$, $\LRD(E_2)|_q-\tilde{L}_2=\hat{K}_2|_q$
and $\LRD(E_3)|_q-L_3|_q=\hat{K}_3|_q$
i.e.
\[
\hat{D}|_q\subset T|_q \mc{O}_{\RDist}(q),\quad \forall q\in Q_0\cap \pi_Q^{-1}(M_0\times \hat{M}).
\]
Similar argument shows that
\[
D|_q\subset T|_q \mc{O}_{\RDist}(q),\quad \forall q\in Q_0\cap \pi_Q^{-1}(M\times \hat{M}_0).
\]

Assume now that $(M_1\times \hat{M}_0)\cap \pi_Q(\mc{O}_{\RDist}(q_0))\neq\emptyset$
and that $M_0\neq\emptyset$.
Choose any $q_1=(x_1,\hat{x}_1;A_1)\in \mc{O}_{\RDist}(q_0)$
with $(x_1,\hat{x}_1)\in M_1\times \hat{M}_0$ and take any
curve $\gamma:[0,1]\to M$ with $\gamma(0)=x_1$, $\gamma(1)\in M_0$.
Then since $\pi_{Q,\hat{M}}(\Gamma(\gamma,q_1)(t))=\hat{x}_1$,
we have $\pi_Q(\Gamma(\gamma,q_1)(t))\in M\times \hat{M}_0$ for all $t\in [0,1]$
and since also $D|_{q}\subset T|_q \mc{O}_{\RDist}(q_0)$ for all $q\in\mc{O}_{\RDist}(q_0)\cap \pi_Q^{-1}( M\times \hat{M}_0)$,
we have that $\Gamma(\gamma,q_1)(t)\in \mc{O}_{\RDist}(q_0)$ for all $t\in [0,1]$.

Indeed, suppose there is a $0\leq t<1$ with $\Gamma(\gamma,q_1)(t)\notin \mc{O}_{\RDist}(q_0)$
and define $t_1=\inf\{t\in [0,1]\ |\ \Gamma(\gamma,q_1)(t)\notin \mc{O}_{\RDist}(q_0)\}$.
Clearly $t_1>0$.
Because $q_2:=\Gamma(\gamma,q_1)(t_1)\in \pi_Q^{-1}(M\times \hat{M}_0)$,
it follows that for $|t|$ small one has $\Gamma(\gamma,q_1)(t_1+t)\in \mc{O}_{\RDist}(q_2)$,
whence if $t<0$ small, $\Gamma(\gamma,q_1)(t_1+t)\in \mc{O}_{\RDist}(q_2)\cap \mc{O}_{\RDist}(q_0)$,
which means that $q_2\in \mc{O}_{\RDist}(q_0)$
and thus for $t\geq 0$ small $\Gamma(\gamma,q_1)(t_1+t)\in \mc{O}_{\RDist}(q_0)$, a contradiction.

Hence one has $\pi_Q(\Gamma(\gamma,q_1)(1))\in (M_0\times \hat{M}_0)\cap \pi_Q(\mc{O}_{\RDist}(q_0))$.
In other words we have the implication:
\[
(M_1\times \hat{M}_0)\cap \pi_Q(\mc{O}_{\RDist}(q_0))\neq\emptyset,\quad
M_0\neq\emptyset
\quad\Longrightarrow\quad (M_0\times \hat{M}_0)\cap \pi_Q(\mc{O}_{\RDist}(q_0))\neq\emptyset.
\]
By a similar argument, using $\hat{D}$ instead of $D$,
one has that 
\[
(M_0\times \hat{M}_1)\cap \pi_Q(\mc{O}_{\RDist}(q_0))\neq\emptyset,
\quad \hat{M}_0\neq\emptyset
\quad\Longrightarrow\quad
(M_0\times \hat{M}_0)\cap \pi_Q(\mc{O}_{\RDist}(q_0))\neq\emptyset.
\]

Suppose now that there exists $q_1=(x_1,\hat{x}_1;A_1)\in \pi_Q^{-1}(M_0\times \hat{M}_0)\cap \mc{O}_{\RDist}(q_0)$.
We already know that $T|_{q_1}\mc{O}_{\RDist}(q_0)$ contains vectors
\[
& \LRD(E_1)|_{q_1}, \LRD(E_2)|_{q_1}, \LRD(E_3)|_{q_1}, \\
& \nu(A\star E_2)|_{q_1}, \nu((\hat{\star}\hat{E}_2)A)|_{q_1} \\
& L_1|_{q_1},\tilde{L}_2|_{q_1}, L_3|_{q_1}
\]
which are linearly independent since $q_1\in  (M_0\times \hat{M}_0)\cap \pi_Q(\mc{O}_{\RDist}(q_0))$.
Indeed, if one introduces $X_A,Z_A$ and an angle $\phi$ as before, we have
$\sin(\phi(q_1))\neq 0$ as $q_1\in Q_0$ and
\[
\nu((\hat{\star}\hat{E}_2)A_1)|_{q_1}
=\nu(A_1\star (A_1^{\ol{T}}\hat{E}_2))|_{q_1}
=\sin(\phi(q_1))\nu(A_1\star X_{A_1})|_{q_1}+\cos(\phi(q_1))\nu(A_1\star E_2)|_{q_1}.
\]
Therefore $\dim\mc{O}_{\RDist}(q_0)\geq 8$
and since we have also shown that $\dim\mc{O}_{\RDist}(q_0)\leq 8$,
we have that
\[
(M_0\times \hat{M}_0)\cap \pi_Q(\mc{O}_{\RDist}(q_0))\neq\emptyset
\quad\Longrightarrow\quad \dim\mc{O}_{\RDist}(q_0)=8
\]

Write $Q^\circ:=\pi_Q^{-1}(M^\circ\times \hat{M}^\circ)$, which is an open subset of $Q$
and clearly $\mc{O}_{\RDist}(q_0)\subset Q^\circ$.
To finish the proof, we proceed case by case.

\begin{itemize}
\item[a)] Suppose $(\hat{M}^\circ,\hat{g})$ has constant curvature i.e. $\hat{M}_0\cap \hat{M}^\circ=\emptyset$.
By assumption then, $(M^\circ, g)$ does not have constant curvature, which means that $M_0\cap M^\circ\neq\emptyset$.

At every $q=(x,\hat{x};A)\in Q^\circ$, one has 
$\widetilde{\Rol}_q(\star E_1)=\widetilde{\Rol}_q(\star E_3)=0$
and $\widetilde{\Rol}_q(\star E_2)=(-K_2(x)+\beta^2)\star E_2$
and therefore
\[
[\LRD(E_1),\LRD(E_2)]|_q&=\LRD([E_1,E_2])|_q,\quad
[\LRD(E_2),\LRD(E_3)]|_q=\LRD([E_2,E_3])|_q \\
[\LRD(E_3),\LRD(E_1)]|_q&=\LRD([E_3,E_1])|_q+(-K_2(x)+\beta^2)\nu(A\star E_2)|_q.
\]
From these, Proposition \ref{pr:3D-2} case (ii) and from the brackets (as above)
\[
[\LRD(E_1),\tilde{L}_2]|_q=&\beta L_3|_q-\LRD(\nabla_{E_2} E_1)|_q \\
[\LRD(E_3),\tilde{L}_2]|_q=&-\beta L_1|_q-\LRD(\nabla_{E_2} E_3)|_q \\
[\LRD(E_2),\tilde{L}_2]|_q=&0 \\
[\nu((\cdot)\star E_2),\tilde{L}_2]|_q=&0 \\
[L_1,\tilde{L}_2]|_q=&(\Gamma^2_{(3,1)}+\beta)L_3|_q \\
[L_3,\tilde{L}_2]|_q=&-(\Gamma^2_{(3,1)}+\beta)L_1|_q,
\]
we see that the distribution $\tilde{\mc{D}}$on $Q^\circ$ spanned by the 7  linearly independent vector fields
\[
\LRD(E_1),\LRD(E_2),\LRD(E_3),\nu((\cdot)\star E_2),L_1,\tilde{L}_2,L_3
\] 
with $L_1,\tilde{L}_2,L_3$ as above, is involutive. Moreover $\tilde{\mc{D}}$ contains $\RDist|_{Q^\circ}$,
which implies $\mc{O}_{\RDist}(q_0)=\mc{O}_{\RDist|_{Q^\circ}}(q_0)\subset \mc{O}_{\tilde{\mc{D}}}(q_0)$
and hence $\dim \mc{O}_{\RDist}(q_0)\leq 7$.

To show the equality here, notice that since $M_0\cap M^\circ\neq\emptyset$,
one has that $O:=\pi_{Q,M}^{-1}(M_0)\cap \mc{O}_{\RDist}(q_0)$
is an open non-empty subset of $ \mc{O}_{\RDist}(q_0)$.
Moreover, because $K_2(x)\neq \beta^2$ on $M_0\cap M^\circ$,
we get that $\nu(A\star E_2)|_q\in T|_{q}\mc{O}_{\RDist}(q_0)$ for all $q\in O$,
from which one deduces by Proposition \ref{pr:3D-2}, case (i)
that $\tilde{\mc{D}}|_q\subset T|_q \mc{O}_{\RDist}(q_0)$,
which then implies $\dim \mc{O}_{\RDist}(q_0)\geq 7$.
This proves one half of case (i) in the statement of this proposition.

\item[b)] If $(M^\circ,g)$ has constant curvature,
one proves as in case a), by simply changing the roles of $M$ and $\hat{M}$,
that $\dim \mc{O}_{\RDist}(q_0)=7$.
This finishes the proof of case (i) of this proposition.
\end{itemize}

For the last case, we assume that neither $(M^\circ,g)$ nor $(\hat{M}^\circ,\hat{g})$
have constant curvature i.e. we have $M^\circ\cap M_0\neq\emptyset$
and $\hat{M}^\circ\cap \hat{M}_0\neq\emptyset$.

\begin{itemize}

\item[c)] Since $M^\circ\cap M_0\neq\emptyset$,
there is a $q_1=(x_1,\hat{x}_1;A_1)\in \mc{O}_{\RDist}(q_0)$
such that $x_1\in M_0$.
If $\hat{x}_1\in \hat{M}_0$,
we have $(M_0\times \hat{M}_0)\cap \pi_Q(\mc{O}_{\RDist}(q_0))\neq\emptyset$
and which implies, as we have shown, that $\dim\mc{O}_{\RDist}(q_0)=8$.

Suppose then that $\hat{x}_1\in \ol{\hat{M}_1}$.
Then one may choose a sequence $q_n'=(x_n',\hat{x}_n';A_n')\in \mc{O}_{\RDist}(q_0)$
such that $q_n'\to q_1$ and $\hat{x}_n'\in \hat{M}_1$.
Because $M_0$ is open, for $n$ large enough one has $(x_n',\hat{x}_n')\in (M_0\times \hat{M}_1)\cap \pi_Q(\mc{O}_{\RDist}(q_0))$.
Hence $(M_0\times \hat{M}_1)\cap \pi_Q(\mc{O}_{\RDist}(q_0))\neq\emptyset$
and $\emptyset\neq \hat{M}^\circ\cap \hat{M}_0\subset \hat{M}_0$,
which has been shown to imply that  $(M_0\times \hat{M}_0)\cap \pi_Q(\mc{O}_{\RDist}(q_0))\neq\emptyset$
and again $\dim\mc{O}_{\RDist}(q_0)=8$.
\end{itemize}

The proof is complete.
\end{proof}

\begin{remark}
One could adapt the proofs of Propositions \ref{pr:mbeta-1}, \ref{pr:mbeta-2} and \ref{pr:mbeta-3}
to deal also with the case $\beta=0$.
For example, Proposition \ref{pr:mbeta-1} as formulated already is valid in this case,
but the conclusion when $\beta=0$
could be strengthened to $\dim\mc{O}_{\RDist}(q_0)\leq 6$.
However, since a Riemannian manifold of class $\mc{M}_{0}$
is also locally a Riemannian product, and hence locally a warped product,
we prefer to view this special case $\beta=0$ as part
of the subject of subsection \ref{subsec:warped}.
\end{remark}

%%%%%%%%%%%%%%%%%%%%%%%%%%%%%%%%%%%%%%%%%%%%%%%
\subsubsection{Case where both manifolds are Warped Products}\label{subsec:warped}
%%%%%%%%%%%%%%%%%%%%%%%%%%%%%%%%%%%%%%%%%%%%%%%

Suppose $(M,g)=(I\times N,h_f)$ and $(\hat{M},\hat{g})=(\hat{I}\times \hat{N},\hat{h}_{\hat{f}})$,
where $I,\hat{I}\subset\R$ are open intervals, $(N,h)$ and $(\hat{N},\hat{h})$ are connected, oriented $2$-dimensional Riemannian manifolds and the warping functions $f,\hat{f}$ are smooth and positive everywhere.
We write $\pa{r}$ for the canonical, positively directed unit vector field on $(\R,s_1)$
and consider it as a vector field on $(M,g)$ and $(\hat{M},\hat{g})$
as is usual in direct products.
Notice that then $\pa{r}$ is a $g$-unit (resp. $\hat{g}$-unit) vector field on $M$ (resp. $\hat{M}$)
which is orthogonal to $T|_{y} N$ (resp. $T|_{\hat{y}}\hat{N}$)
for every $(r,y)\in M$ (resp. $(\hat{r},\hat{y})\in \hat{N}$).

We will prove that starting from any point point $q_0\in Q=Q(M,\hat{M})$
and if the warping functions $f,\hat{f}$ satisfy extra conditions relative to each other,
then the orbit $\mc{O}_{\RDist}(q_0)$ is either 6- or 8-dimensional.
The first case is formulated in the following proposition.

\begin{proposition}\label{pr:wp-1}
Let $(M,g)=(I\times N,h_f)$, $(\hat{M},\hat{g})=(\hat{I}\times \hat{N},\hat{h}_{\hat{f}})$ be warped products of dimension $3$,
with $I,\hat{I}\subset\R$ open intervals.
Also, let $q_0=(x_0,\hat{x}_0;A_0)\in Q$
be such that if one writes $x_0=(r_0,y_0)$, $\hat{x}_0=(\hat{r}_0,\hat{y}_0)$, then
\begin{align}\label{eq:A_0_S_to_hatS}
A_0\pa{r}\big|_{(r_0,y_0)}=\pa{r}\big|_{(\hat{r}_0,y_0)}.
\end{align}
holds and
\begin{align}\label{eq:Df_over_f}
\frac{f'(t+r_0)}{f(t+r_0)}=\frac{\hat{f}'(t+\hat{r}_0)}{\hat{f}(t+\hat{r}_0)},\quad \forall t\in (I-r_0)\cap (\hat{I}-\hat{r}_0).
\end{align}
Then if $\mc{O}_{\RDist}(q_0)$ is not an integral manifold of $\RDist$, one has $\dim \mc{O}_{\RDist}(q_0)=6$.
\end{proposition}

\begin{proof}
For convenience we write $\kappa(r):=\frac{f'(r+r_0)}{f(r+r_0)}=\frac{\hat{f}'(r+\hat{r}_0)}{\hat{f}(r+\hat{r}_0)}$,
$r\in (I-r_0)\cap (\hat{I}-\hat{r}_0)=:J$.
Let $\gamma$ be a smooth curve in $M$ defined on some interval containing $0$ and such that $\gamma(0)=x_0$
and let $(\gamma(t),\hat{\gamma}(t);A(t))=q_{\RDist}(\gamma,q_0)(t)$ be the rolling curve
generated by $\gamma$ starting at $q_0$ and defined on some (possible smaller) maximal
interval containing $0$.
Write $\gamma(t)=(r(t),\gamma_1(t))$ and $\hat{\gamma}(t)=(\hat{r}(t),\hat{\gamma}_1(t))$
corresponding to the direct products $M=I\times N$ and $\hat{M}=\hat{I}\times \hat{N}$.
Define also,
\[
\zeta(t):=r(t)-r_0,\quad S(t):=&\pa{r}\big|_{\gamma(t)} \\
\hat{\zeta}(t):=\hat{r}(t)-\hat{r}_0,\quad \hat{S}(t):=&A(t)^{-1}\pa{r}\big|_{\hat{\gamma}(t)}
\]
which are vector fields on $M$ along $\gamma$.

Notice that
\[
\dot{\zeta}(t)=\dot{r}(t)=&g\big(\dot{\gamma}(t),\pa{r}\big|_{\gamma(t)}\big)=g(\dot{\gamma}(t),S(t)), \\
\dot{\hat{\zeta}}(t)=\dot{\hat{r}}(t)=&\hat{g}\big(\dot{\hat{\gamma}}(t),\pa{r}\big|_{\hat{\gamma}(t)}\big)
=\hat{g}\big(A(t)\dot{\gamma}(t),\pa{r}\big|_{\hat{\gamma}(t)}\big)=g(\dot{\gamma}(t),\hat{S}(t)).
\]
By Proposition 35, Chapter 7, p. 206 in \cite{oneill83}, we have
\[
\nabla_{\dot{\gamma}(t)}\pa{r}=&\frac{f'(r(t))}{f(r(t))}\big(\dot{\gamma}(t)-\dot{r}(t)\pa{r}\big|_{\gamma(t)}\big) \\
=&\kappa(\zeta(t))\big(\dot{\gamma}(t)-\dot{\zeta}(t)\pa{r}\big|_{\gamma(t)}\big) \\
\hat{\nabla}_{\dot{\hat{\gamma}}(t)}\pa{r}=&\frac{\hat{f}'\big(\hat{r}(t))}{\hat{f}(\hat{r}(t))}\big(\dot{\hat{\gamma}}(t)-\dot{\hat{r}}(t)\pa{r}\big|_{\hat{\gamma}(t)}\big) \\
=&\kappa(\hat{\zeta}(t))\big(\dot{\hat{\gamma}}(t)-\dot{\hat{\zeta}}(t)\pa{r}\big|_{\hat{\gamma}(t)}\big),
\]
i.e.
\[
\nabla_{\dot{\gamma}(t)}S(t)=&\kappa(\zeta(t))(\dot{\gamma}(t)-\dot{\zeta}(t)S(t)) \\
\nabla_{\dot{\gamma}(t)}\hat{S}(t)=&A(t)^{-1}\hat{\nabla}_{\dot{\hat{\gamma}}(t)}\pa{r}
=\kappa(\hat{\zeta}(t))\big(A(t)^{-1}\dot{\hat{\gamma}}(t)-\dot{\hat{\zeta}}(t)A(t)^{-1}\pa{r}\big|_{\hat{\gamma}(t)}\big) \\
=&\kappa(\hat{\zeta}(t))(\dot{\gamma}(t)-\dot{\hat{\zeta}}(t)\hat{S}(t))
\]

Let $\rho\in\Cinf(\R)$ and $t\mapsto X(t)$ be a vector field along $\gamma$
and consider a first order ODE
\[
\begin{cases}
\dot{\rho}(t)=g(\dot{\gamma}(t),X(t)) \\
\nabla_{\dot{\gamma}(t)} X=\kappa(\rho(t))(\dot{\gamma}(t)-\dot{\rho}(t)X(t)).
\end{cases}
\]
By the above we see that the pairs $(\rho,X)=(\zeta,S)$ and $(\rho,X)=(\hat{\zeta},\hat{S})$
both solve this ODE.
Moreover, by assumption $\zeta(0)=0=\hat{\zeta}(0)$
and $\hat{S}(0)=A(0)^{-1}\pa{r}\big|_{\hat{x}_0}=\pa{r}\big|_{x_0}=S(0)$
so these pairs have the same initial conditions and hence
$(\zeta,S)=(\hat{\zeta},\hat{S})$ on the interval where they are both defined.
In other words,
\[
r(t)-r_0=&\hat{r}(t)-\hat{r}_0 \\
A(t)\pa{r}\big|_{\gamma(t)}=&\pa{r}\big|_{\hat{\gamma}(t)}
\]
for all $t$ in the interval where the rolling curve $q_{\RDist}(\gamma,q_0)$ is defined.

Define
\[
Q^*_+=\big\{q=(x,\hat{x};A)=((r,y),(\hat{r},\hat{y});A)\in Q\ |\ r-r_0=\hat{r}-\hat{r}_0,\ A\pa{r}\big|_{x}=\pa{r}\big|_{\hat{x}}\big\}.
\]
By the above considerations,
\[
q_{\RDist}(\gamma,q_0)(t)\in Q^*_+,\quad \forall t
\]
which implies that $\mc{O}_{\RDist}(q_0)\subset Q^*_+$.

We show that $Q^*_+$ is a $6$-dimensional submanifold of $Q$.
Let $q=(x,\hat{x};A)=((r,y),(\hat{r},\hat{y});A)\in Q$
such that $A\pa{r}\big|_{x}=\pa{r}\big|_{\hat{x}}$.
Then for all $\alpha\in\R,X'\in T|_{y} N$ one has
\[
\n{X'}^2_{g}+\alpha^2=\n{X'+\alpha\pa{r}\big|_x}^2_g=\n{A(X'+\alpha\pa{r}\big|_x)}^2_{\hat{g}}
=\n{AX'}^2_{\hat{g}}+2\hat{g}\big(AX',\alpha\pa{r}\big|_{\hat{x}}\big)+\alpha^2.
\]
This implies that
\[
& \n{X'}^2_{g}=\n{AX'}^2_{\hat{g}} \\
& \hat{g}\big(AX',\pa{r}\big|_{\hat{x}}\big)=0
\]
for all $X'\in T|_{y} N$.
Thus $AT|_{y} N\perp \pa{r}\big|_{\hat{x}}$ and also $A\pa{r}\big|_{x}\perp T|_{\hat{y}}\hat{N}$
by assumption.
Define
\[
Q_1^+=\big\{q=(x,\hat{x};A)=((r,y),(\hat{r},\hat{y});A)\in Q\ |\ A\pa{r}\big|_{x}=\pa{r}\big|_{\hat{x}}\big\}
\]
and let $q_1=(x_1,\hat{x}_1;A_1)=((r_1,y_1),(\hat{r}_1,\hat{y}_1);A_1)\in Q_1^+$.
Choose a local oriented $h$- and $\hat{h}$-orthonormal frames $X_1',X_2'$ in $N$ around $y_1$
and $\hat{X}_1',\hat{X}_2'$ in $\hat{N}$ around $\hat{y}_1$.
Let the corresponding domains be $U'$ and $\hat{U}'$.
Writing $E_1=\pa{r}$, $E_2=\frac{1}{f}X_1'$, $E_3=\frac{1}{f}X_2'$ on $M$
and $\hat{E}_1=\pa{r}$, $\hat{E}_2=\frac{1}{\hat{f}}\hat{X}_1'$, $\hat{E}_3=\frac{1}{\hat{f}}\hat{X}_2'$ on $\hat{M}$,
we see that $E_1,E_2,E_3$ and $\hat{E}_1,\hat{E}_2,\hat{E}_3$
are $g$- and $\hat{g}$-orthonormal oriented frames
and we define
\[
& \Psi:V:=\pi_Q^{-1}((\R\times U')\times(\R\times \hat{U}'))\to \SO(3); \\
& \Psi(x,\hat{x};A)=[(\hat{g}(AE_i,\hat{E}_j))_i^j].
\]
This is a chart of $Q$ and clearly
\[
\Psi(V\cap Q_1^+)=(\R\times U')\times(\R\times \hat{U}')\times\Big\{\qmatrix{1& 0 \cr 0 & A' }\ |\ A'\in\SO(2) \Big\}.
\]
This shows that $Q_1^+\cap V$ is a 7-dimensional submanifold of $Q$
and hence $Q_1^+$ is a closed 7-dimensional submanifold of $Q$.

Defining $F:Q_1^+\to\R$ by $F((r,y),(\hat{r},\hat{y});A)=(r-r_0)-(\hat{r}-\hat{r}_0)$,
we see that $Q^*_+=F^{-1}(0)$.
Once we show that $F$ is a submersion, it follows that $Q^*_+$
is a closed codimension 1 submanifold of $Q_1^+$ (i.e. $\dim Q^*_+=7-1=6$)
and thus it is a 6-dimensional submanifold of $Q$.

Indeed, let $q=(x,\hat{x};A)\in Q_1^+$ and let $\gamma(t)$ be an integral curve
of $\pa{r}$ starting from $x$
and $\hat{\gamma}(t)=\hat{x}$ a constant path.
Let $q(t)=(\gamma(t),\hat{\gamma}(t);A(t))$ be the $\NSDist$-lift of $(\gamma,\hat{\gamma})$ starting from $q$.
Then $\dot{\gamma}(t)=\pa{r}\big|_{\gamma(t)}$, $\dot{\hat{\gamma}}(t)=0$
and since $\pa{r}$ is a unit geodesic field on $M$, one has
\[
\dif{t}\hat{g}\big(A(t)\pa{r}\big|_{\gamma(t)},\pa{r}\big|_{\hat{\gamma}(t)}\big)
=\hat{g}\big(A(t)\nabla_{\dot{\gamma}(t)}\pa{r},\pa{r}\big|_{\hat{\gamma}(t)}\big)
+\hat{g}\big(A(t)\pa{r},\hat{\nabla}_{0}\pa{r}\big|_{\hat{\gamma}(t)}\big)
=0.
\]
This shows that $q(t)\in Q_1^+$ for all $t$ and in particular,
$\LNSD(\pa{r}\big|_x)|_q=\dot{q}(0)\in T|_q Q_1^+$.
Then if one writes $\gamma(t)=(r(t),\gamma_1(t))$,
$\hat{\gamma}(t)=\hat{x}=(\hat{r},\hat{y})$=constant,
one has $\dot{r}(t)=1$ and therefore
\[
\dif{t}\big|_0 F(q(t))=\dif{t}\big|_0 \big((r(t)-r_0)-(\hat{r}-\hat{r}_0)\big)=1,
\]
i.e. $F_*\LNSD(\pa{r}\big|_{x})|_q=1$, which shows that $F$ is submersive.
(Alternatively, one could have uses the charts $\Psi$ as above to prove this fact.)

Since we have shown that $\dim Q^*_+=6$ and $\mc{O}_{\RDist}(q_0)\subset Q^*_+$,
it follows that $\mc{O}_{\RDist}(q_0)\leq 6$.
To prove the equality here, we will use the assumption that $\mc{O}_{\RDist}(q_0)$ is not an integral manifold of $\RDist$.

Take local frames $E_i,\hat{E}_i$ as above near $x_1$ and $\hat{x}_1$,
where $q_1=(x_1,\hat{x}_1;A_1)=((r_1,y_1),(\hat{r}_1,\hat{y}_1);A_1)\in\mc{O}_{\RDist}(q_0)$.
The assumption that $\frac{f'(t+r_0)}{f(t+r_0)}=\frac{\hat{f}'(t+\hat{r}_0)}{\hat{f}(t+\hat{r}_0)}$ for all $t\in J$
easily imply that $\frac{f''(t+r_0)}{f(t+r_0)}=\frac{\hat{f}''(t+\hat{r}_0)}{\hat{f}(t+\hat{r}_0)}=:\kappa_2(t)$ for all $t\in J$ as well.
Respect to the frames $\star E_1,\star E_2,\star E_3$
and $\hat{\star} \hat{E}_1,\hat{\star} \hat{E}_2,\hat{\star} E_3$
one has (see Proposition 42, Chapter 7, p. 210 of \cite{oneill83})
\[
R|_{(r,y)}=&\qmatrix{-\frac{\sigma(y)}{f(r)^2}+\kappa(r-r_0)^2 & 0 & 0 \\
0 & \kappa_2(r-r_0) & 0 \\
0 & 0 & \kappa_2(r-r_0)
}, \\
\quad
\hat{R}|_{(\hat{r},\hat{y})}=&\qmatrix{-\frac{\hat{\sigma}(\hat{y})}{\hat{f}(\hat{r})^2}+\kappa(\hat{r}-\hat{r}_0)^2 & 0 & 0 \\
0  & \kappa_2(\hat{r}-\hat{r}_0) & 0 \\
0 & 0 & \kappa_2(\hat{r}-\hat{r}_0)
},
\]
where $\sigma(y)$ and $\hat{\sigma}(\hat{y})$
are the unique sectional (or Gaussian) curvatures of $(N,h)$ and $(\hat{N},\hat{h})$
at points $y,\hat{y}$.
Write
\[
-K_2(r,y)=-\frac{\sigma(y)}{f(r)^2}+\kappa(r-r_0),\quad
-\hat{K}_2(\hat{r},\hat{y})=-\frac{\hat{\sigma}(\hat{y})}{\hat{f}(\hat{r})^2}+\kappa(\hat{r}-\hat{r}_0).
\]

Since $A_1\pa{r}\big|_{x_1}=\pa{r}\big|_{\hat{x}_1}$,
we already know that $A_1 E_2|_{x_1}$ and $A_1 E_3|_{x_1}$
are in the plane $\spn\{\hat{E}_2|_{\hat{x}_1},\hat{E}_3|_{\hat{x}_1}\}$.
This and the fact that $r_1-r_0=\hat{r}_1-\hat{r}_0$ 
imply
\[
\widetilde{\Rol}_{q_1}=\qmatrix{
-K_2(x_1)+\hat{K}_2(\hat{x}_1) & 0 & 0 \cr 
0 & 0 & 0 \cr 
0 & 0 & 0
}
\]
w.r.t. $\star E_1|_{x_1}, \star E_2|_{x_1}, \star E_3|_{x_1}$.

Since $\mc{O}_{\RDist}(q_0)$ is not an integral manifold of $\RDist$,
it follows from Corollary \ref{cor:weak_ambrose} and Remark \ref{re:weak_ambrose}
that there is a $q_1\in\mc{O}_{\RDist}(q_0)$, where $\widetilde{\Rol}_{q_1}\neq 0$.
Hence there is a neighbourhood $O$ of $q_1$ in $\mc{O}_{\RDist}(q_0)$
such that $\widetilde{\Rol}_{q}\neq 0$.
With respect to local frames $E_i,\hat{E}_i$ as above (taking $O$ smaller if necessary),
this means that $K_2(x)\neq \hat{K}_2(\hat{x})$ for all $q=(x,\hat{x};A)\in O$
and since $\nu(\Rol_{q}(\star E_1))|_{q}=(-K_2(x)+\hat{K}_2(\hat{x}))\nu(A\star E_1)|_q$,
we have
\[
\nu(A\star E_1)|_q\in T|_q\mc{O}_{\RDist}(q_0),\quad \forall q\in O.
\]

Hence applying Proposition \ref{pr:3D-2} case (i)
to the frame $F_1:=E_2$, $F_2:=E_1$, $F_3:=E_3$ implies that
the 6 linearly independent vectors (notice that we have $\Gamma^1_{(2,3)}=0$ in that proposition)
\[
\LRD(F_1)|_q,\LRD(F_2)|_q,\LRD(F_3)|_q,\nu(A\star F_2)|_q,L_1|_q,L_3|_q
\]
are tangent to $\mc{O}_{\RDist}(q_0)$ at $q\in O$,
where
\[
L_1&=\LNSD(F_1)|_q-\Gamma^1_{(1,2)}(x)\nu(A\star F_3)|_q \\
L_3&=\LNSD(F_3)|_q+\Gamma^1_{(1,2)}(x)\nu(A\star F_1)|_q,
\]
where $\Gamma^1_{(1,2)}(x)=g(\nabla_{F_1} F_1 F_2)=g(\nabla_{E_2} E_2,E_1)=-\frac{f'(r)}{f(r)}$
if $x=(r,y)$.
This proves that $\dim\mc{O}_{\RDist}(q_0)\geq 6$.
End of the proof.
\end{proof}

\begin{remark}\label{re:condition_for_6-dim_orbit}
The condition $\Rol_{q_1}\neq 0$ in the proof of the previous proposition
was equivalent to the condition $K_2(x_1)\neq \hat{K}_2(\hat{x}_1)$
which again means that if $x_1=(r_1,y_1)$, $\hat{x}_1=(\hat{x}_1,\hat{y}_1)$,
\[
\frac{\sigma(y_1)}{f(r_1)^2}\neq \frac{\hat{\sigma}(\hat{y}_1)}{\hat{f}(r_1)^2}.
\]
where $\sigma(y)$ (resp. $\hat{\sigma}(\hat{y})$) is the sectional curvature
of $(N,h)$ at $y\in N$ (resp. of $(\hat{N},\hat{h})$ at $\hat{y}\in \hat{N}$).
\end{remark}

\begin{remark}\label{re:wp-1}
To show that $\dim \mc{O}_{\RDist}(q_0)\leq 6$ under the assumptions of the
proposition, we showed that if $q=(x,\hat{x};A)\in Q^*_+$, then $q_{\RDist}(\gamma,q)(t)\in Q^*_+$
for any path $\gamma$ starting from $x$. For this we basically used the uniqueness of the
solutions of an ODE.

Alternatively, one could have proceeded exactly in the same way as in the proof
of Proposition \ref{pr:mbeta-1}.
To this end, one defines as there $h_1,h_1:Q\to\R$ and also $F:Q\to\R$ as above as
\[
h_1(q)=\hat{g}(AE_1,\hat{E}_2),
\quad h_2(q)=\hat{g}(AE_3,\hat{E}_2),
\quad F(q)=(r-r_0)-(\hat{r}-\hat{r}_0).
\]
Now write $H=(h_1,h_2,F):Q\to\R^3$, $Q^*:=H^{-1}(0)$ and $Q=Q^*_+\cup Q^*_-$ where $Q^*_{+}$ (resp. $Q^*_-$)
consists of all $q=(x,\hat{x};A)\in Q^*$ where $A\pa{r}=+\pa{r}$
(resp. $A\pa{r}=-\pa{r}$).

Now for all $q\in Q^*_+$,
\[
H_*\nu(A\star E_1)|_q=(0,-1,0),
\quad
H_*\nu(A\star E_3)|_q=(1,0,0),
\quad
H_*\LNSD(\pa{r},0)|_q=(0,0,1),
\]
which shows (again) that $Q^*_+$ is a $6$-dimensional closed submanifold of $Q$
(and so is $Q^*$)
while w.r.t. orthonormal bases $E_1,E_2,E_3$,  $\hat{E}_1,\hat{E}_2,\hat{E}_3$,
where $E_2=\pa{r}$, $\hat{E}_2=\pa{r}$, one has for $q=(x,\hat{x};A)\in Q^*_+$,
since $x=(r,y)$, $\hat{x}=(\hat{r},\hat{y})$ with $r-r_0=\hat{r}-\hat{r}_0=:t$
\[
\LRD(E_1)|_q h_1&=\hat{g}(A(\Gamma^1_{(1,2)}E_2-\Gamma^1_{(3,1)}E_3),\hat{E}_2)
+\hat{g}(AE_1,-\hat{\Gamma}^1_{(1,2)}AE_1) \\
&=-\frac{f'(r)}{f(r)}+\frac{\hat{f}'(\hat{r})}{\hat{f}(\hat{r})}
=-\frac{f'(t+r_0)}{f(t+r_0)}+\frac{\hat{f}'(t+\hat{r}_0)}{\hat{f}(t+\hat{r}_0)}=0 \\
\LRD(E_1)|_q h_2&=\Gamma^1_{(3,1)}\hat{g}(AE_1,\hat{E}_2)
+\hat{g}(AE_3,-\hat{\Gamma}^1_{(1,2)}AE_1)=0 \\
\LRD(E_2)|_q h_1&=-\Gamma^2_{(3,1)}\hat{g}(AE_3,\hat{E}_2)=0 \\
\LRD(E_2)|_q h_2&=\Gamma^2_{(3,1)}\hat{g}(AE_1,\hat{E}_2)=0 \\
\LRD(E_3)|_q h_1&=\LRD(E_3)|_q h_2=0 \\
\LRD(E_1)|_q F&=\LRD(E_2)|_q F=\LRD(E_3)|_q F=0,
\]
hence $\RDist|_q\subset T|_qQ^*_+$ for all $q\in Q^*_+$.
This obviously implies that $\mc{O}_{\RDist}(q)\subset Q^*_+$ for all $q\in Q^*_+$
and thus $\dim \mc{O}_{\RDist}(q)\leq \dim Q^*_+=6$.
\end{remark}

For the following proposition we introduce some notation,
\[
Q_0:=&\{q=((r,y),(\hat{r},\hat{y});A)\in Q\ |\ A\pa{r}\big|_{(r,y)}\neq\pm \pa{r}\big|_{(\hat{r},\hat{y})}\} \\
Q_1^+:=&Q\backslash Q_0=\{q=(x,\hat{x};A)\in Q\ |\ A\pa{r}\big|_{(r,y)}=+\pa{r}\big|_{(\hat{r},\hat{y})}\} \\
Q_1^-:=&Q\backslash Q_0=\{q=(x,\hat{x};A)\in Q\ |\ A\pa{r}\big|_{(r,y)}=-\pa{r}\big|_{(\hat{r},\hat{y})}\} \\
Q_1:=&Q_1^+\cup Q_1^- \\
S_1^+:=&\big\{q=((r,y),(\hat{r},\hat{y});A)\in Q_1^+\ |\ \frac{f'(r)}{f(r)}=+\frac{\hat{f}'(\hat{r})}{\hat{f}(\hat{r})}\big\} \\
S_1^-:=&\big\{q=((r,y),(\hat{r},\hat{y});A)\in Q_1^-\ |\ \frac{f'(r)}{f(r)}=-\frac{\hat{f}'(\hat{r})}{\hat{f}(\hat{r})}\big\} \\
S_1:=&S_1^+\cup S_1^-.
\]
We have that $Q$ decomposes into a disjoint unions
\[
Q=S_1\cup (Q\backslash S_1)=S_1\cup (Q_1\backslash S_1)\cup Q_0.
\]

\begin{proposition}\label{pr:wp-2}
Let $(M,g)=(I\times N,h_f)$ and $(\hat{M},\hat{g})=(\hat{I}\times \hat{N},\hat{h}_{\hat{f}})$,
be warped products with $I,\hat{I}\subset\R$ open intervals and
suppose that there is a constant $K\in\R$ such that
\[
\frac{f''(r)}{f(r)}=-K=\frac{\hat{f}''(\hat{r})}{\hat{f}(\hat{r})},\quad \forall (r,\hat{r})\in I\times \hat{I}.
\]
Let $q_0=(x_0,\hat{x}_0;A_0)\in Q$ and write $M^\circ:=\pi_{Q,M}(\mc{O}_{\RDist}(q_0))$,
$\hat{M}^\circ:=\pi_{Q,\hat{M}}(\mc{O}_{\RDist}(q_0))$.
Assuming that $\mc{O}_{\RDist}(q_0)$ is not an integral manifold of $\RDist$,
we have the following cases:
\begin{itemize}
\item[(i)]  If $q_0\in S_1$, then $\dim\mc{O}_{\RDist}(q_0)=6$;

\item[(ii)] If $q_0\in Q\backslash S_1$ and if only one of $(M^\circ,g)$ or $(\hat{M}^\circ,\hat{g})$
has constant curvature, then $\dim\mc{O}_{\RDist}(q_0)=6$;
\item[(iii)] Otherwise $\dim\mc{O}_{\RDist}(q_0)=8$.
\end{itemize}
\end{proposition}

\begin{proof}
As in the proof of Proposition \ref{pr:wp-1} (see also Remark \ref{re:wp-1}) it is clear that $Q_1$ is a closed 7-dimensional closed
submanifolds of $Q$ and $Q_1^-,Q_1^+$ are disjoint open and closed submanifolds of $Q_1$.
Also, $S_1,S_1^+,S_1^-$ are closed subsets of $Q_1$.

Let us begin with the case where $q_0\in S_1^+$.
Writing $x_0=(r_0,y_0)$, $\hat{x}_0=(\hat{r}_0,\hat{y}_0)$
and defining $w(t):=\frac{f'(t+r_0)}{f(t+r_0)}-\frac{\hat{f}'(t+\hat{r}_0)}{\hat{f}(t+\hat{r}_0)}$, we see that
for all $t\in (I-r_0)\cap (\hat{I}-\hat{r}_0)$,
\[
w'(t)=\underbrace{\frac{f''(t+r_0)}{f(t+r_0)}}_{=-K}-\Big(\frac{f'(t+r_0)}{f(t+r_0)}\Big)^2
-\underbrace{\frac{\hat{f}''(t+\hat{r}_0)}{\hat{f}(t+\hat{r}_0)}}_{=-K}+\Big(\frac{\hat{f}'(t+\hat{r}_0)}{\hat{f}(t+\hat{r}_0)}\Big)^2
\]
i.e.
\[
w'(t)=-w(t)\Big(\frac{f'(t+r_0)}{f(t+r_0)}+\frac{\hat{f}'(t+\hat{r}_0)}{\hat{f}(t+\hat{r}_0)}\Big),\quad w(0)=0.
\]
This shows that $w(t)=0$ for all $t\in (I-r_0)\cap (\hat{I}-\hat{r}_0)$ and hence the assumptions of 
Proposition \ref{pr:wp-1} have been met. Thus $\dim \mc{O}_{\RDist}(q_0)=6$.

On the other hand, if $q_0=(x_0,\hat{x}_0;A_0)\in S_1^-$
and $x_0=(r_0,y_0)$, $\hat{x}_0=(\hat{r}_0,\hat{y}_0)$,
define $\hat{f}^\vee(t):=\hat{f}(-t)$, $\hat{I}^\vee:=-\hat{I}$ and notice that
$\varphi:(\hat{I}\times \hat{N},\hat{h}_{\hat{f}})\to (\hat{I}^\vee\times \hat{N},\hat{h}_{\hat{f}^\vee})=:(\hat{M}^\vee,\hat{g}^\vee)$
given by $(\hat{y},\hat{r})\mapsto (\hat{y},-\hat{r})$ is an isometry,
which induces a diffeomorphism $\Phi:Q\to Q(M,\hat{M}^\vee)$
by $(x,\hat{x};A)\mapsto (x,\varphi(\hat{x});\varphi_*|_{\hat{x}}\circ A)$
which preserves the respective rolling distributions and orbits:
$\Phi_*(\RDist|_q)=\RDist^\vee|_{\Phi(q)}$,
$\Phi(\mc{O}_{\RDist}(q))=\mc{O}_{\RDist^\vee}(\Phi(q))$, the notation being clear here.
But now $\Phi(A_0)=\varphi_*(A_0\pa{r})=-\varphi_*\pa{r}=\pa{r}$
and since $q_0^\vee:=\Phi(q_0)=((r_0,y_0),(-\hat{r}_0,\hat{y}_0);\varphi_*\circ A_0)$,
\[
\frac{(f^\vee)'(-\hat{r}_0)}{f^\vee(-\hat{r}_0)}
=\frac{\dif{t}\big|_0 f^{\vee}(t-\hat{r}_0)}{\hat{f}(\hat{r}_0)}
=\frac{\dif{t}\big|_0 f(\hat{r}_0-t)}{\hat{f}(\hat{r}_0)}
=-\frac{f'(\hat{r}_0)}{\hat{f}(\hat{r}_0)}=\frac{f'(r_0)}{f(r_0)}.
\]
Thus $\Phi(q_0)$ belongs to the set $S_1^+$ of $Q(M,\hat{M}^\vee)$
(which corresponds by $\Phi$ to $S_1^-$ of $Q$) and thus the above argument
implies that $\dim \mc{O}_{\RDist^\vee}(\Phi(q_0))=6$ and therefore $\mc{O}_{\RDist}(q_0)=6$.
Hence we have proven (i).

Now we deal with the case where $q_0\in Q\backslash S_1$.
Up until the second half of the proof, where we introduce the sets $M_0,M_1,\hat{M}_0,\hat{M}_1$,
we assume that the choice of $q_0\in Q\backslash S_1$ is not fixed
(and hence $M^\circ$, $\hat{M}^\circ$ are not defined yet).

So let $q_0=(x_0,\hat{x}_0;A_0)=((r_0,y_0),(\hat{r}_0,\hat{y}_0);A_0)\in Q\backslash S_1$
and choose some orthonormal frame $X_1,X_3$ (resp. $\hat{X}_1,\hat{X}_3$)
on $N$ (resp. $\hat{N})$
defined on an open neighbourhood $U'$ of $y_0$ (resp. $\hat{U}'$ of $\hat{y}_0$)
and consider them, in the natural way, as vector fields on $M$ (resp. $\hat{M})$.
Moreover, assume that $X_1,\pa{r},X_3$ (resp. $\hat{X}_1,\pa{r},\hat{X}_3$)
is oriented.
Writing $E_1=\frac{1}{f}X_1$, $E_2=\pa{r}$, $E_3=\frac{1}{f}X_3$,
and $\hat{E}_1=\frac{1}{\hat{f}}\hat{X}_1$, $\hat{E}_2=\pa{r}$, $\hat{E}_3=\frac{1}{\hat{f}}\hat{X}_3$,
we get positively oriented orthonormal frames of $M$ and $\hat{M}$,
defined on $U:=I\times U'$, $\hat{U}:=\hat{I}\times\hat{U}'$,
respectively.

Then we have, by \cite{oneill83}, Chapter 7, Proposition 42
(one should pay attention that there the definition of the curvature tensor
differs by sign to the definition used here)
that with respect to the frames $\star E_1,\star E_2,\star E_3$
and $\hat{\star} \hat{E}_1,\hat{\star} \hat{E}_2,\hat{\star} E_3$,
\[
R=\qmatrix{-K & 0 & 0 \\
0 & \frac{-\sigma+(f')^2}{f^2} & 0 \\
0 & 0 & -K
},
\quad
\hat{R}=\qmatrix{-K & 0 & 0 \\
0  & \frac{-\hat{\sigma}+(\hat{f}')^2}{\hat{f}^2} & 0 \\
0 & 0 & -K
},
\]
where $\sigma(y)$ and $\hat{\sigma}(\hat{y})$
are the unique sectional (or Gaussian) curvatures of $(N,h)$ and $(\hat{N},\hat{h})$
at points $y,\hat{y}$.
Write $-K_2:=\frac{-\sigma+(f')^2}{f^2}$ and $-\hat{K}_2:=\frac{-\hat{\sigma}+(\hat{f}')^2}{\hat{f}^2}$.

We now take an open neighbourhood $\tilde{O}$ of $q_0$ in $Q$ according to the following cases:
\begin{itemize}
\item[(a)] If $q_0\in Q_0$, we assume that $\tilde{O}\subset Q_0\cap \pi_Q^{-1}(U\times\hat{U})$.
\item[(b)] If $q_0\in Q_1^+\backslash S_1$
(resp. $q_0\in Q_1^-\backslash S_1$)
we assume that $\tilde{O}\subset \pi_Q^{-1}(U\times\hat{U})\backslash (S_1\cup Q_1^-)$
(resp. $\tilde{O}\subset \pi_Q^{-1}(U\times\hat{U})\backslash (S_1\cup Q_1^+)$).
\end{itemize}
Write $\tilde{O}_0:=\tilde{O}\cap Q_0$.
Thus in case (a) one has $\tilde{O}=\tilde{O}_0\ni q_0$
while in case (b) one has $\tilde{O}=\tilde{O}_0\cup (\tilde{O}\cap (Q_1^{\pm}\backslash S_1))$,
as a disjoint union, and $q_0\notin \tilde{O}_0$,
the "$\pm$" depending on the respective situation.
Moreover, if the case (b) occurs, we assume that $q_0\in Q_1^+\backslash S_1$
since the case where $q_0\in Q_1^-\backslash S_1$ is handled in a similar way.

We will still shrink $\tilde{O}$ around $q_0$ whenever convenient and always keep
in mind that $\tilde{O}_0=\tilde{O}\cap Q_0$ even after the shrinking.
Notice that this shrinking does not change the properties in (a) and (b) above.

Moreover, \cite{oneill83}, Chapter 7, Proposition 35
implies that
if $\Gamma$, $\hat{\Gamma}$ are connection
tables w.r.t. $E_1,E_2,E_3$ and $\hat{E}_1,\hat{E}_2,\hat{E}_3$,
respectively,
\[
\Gamma=\qmatrix{
0 & 0 & -\Gamma^1_{(1,2)}\\
\Gamma^1_{(3,1)} & \Gamma^2_{(3,1)} & \Gamma^3_{(3,1)} \\
\Gamma^1_{(1,2)} & 0 & 0
},
\quad \hat{\Gamma}=\qmatrix{
0 & 0 & -\hat{\Gamma}^1_{(1,2)}\\
\hat{\Gamma}^1_{(3,1)} & \hat{\Gamma}^2_{(3,1)} & \hat{\Gamma}^3_{(3,1)} \\
\hat{\Gamma}^1_{(1,2)} & 0 & 0
}
\]
and
\[
W(\Gamma^1_{(1,2)})=&0,\quad \forall W\in E_2^\perp, \\
\hat{W}(\hat{\Gamma}^1_{(1,2)})=&0,\quad \forall \hat{W}\in \hat{E}_2^\perp,
\]
since $\Gamma^1_{(1,2)}(r,y)=-\frac{f'(r)}{f(r)}$ and $\hat{\Gamma}^1_{(1,2)}(\hat{r},\hat{y})=-\frac{\hat{f}'(\hat{r})}{\hat{f}(\hat{r})}$.
Actually one even has $\Gamma^2_{(3,1)}=0$ and $\hat{\Gamma}^2_{(3,1)}=0$,
but we don't use this fact; one could for example
rotate $E_1,E_3$ (resp. $\hat{E}_1,\hat{E}_3$) between them, in a non-constant way,
to destroy this property.

The fact that $AE_2|_{x}\neq \pm \hat{E}_2|_{\hat{x}}$ for $q=(x,\hat{x};A)\in Q_0$
is equivalent
to the fact that the intersection $(AE_2^\perp|_{x})\cap \hat{E}_2^\perp|_{\hat{x}}$
is non-trivial for all $q=(x,\hat{x};A)\in Q_0$.
Therefore, by shrinking $\tilde{O}$ around $q_0$
if necessary, we may find a smooth functions $\theta,\hat{\theta}:\tilde{O}_0\to\R$
such that this intersection is spanned by
$AZ_A=\hat{Z}_A$, where
\[
Z_A:=&-\sin(\theta(q))E_1|_x+\cos(\theta(q))E_3|_x \\
\hat{Z}_A:=&-\sin(\hat{\theta}(q))\hat{E}_1|_{\hat{x}}+\cos(\hat{\theta}(q))\hat{E}_3|_{\hat{x}}.
\]
We also define
\[
X_A:=&\cos(\theta(q))E_1|_{x}+\sin(\theta(q))E_3|_{x} \\
\hat{X}_A:=&\cos(\hat{\theta}(q))\hat{E}_1|_{\hat{x}}+\sin(\hat{\theta}(q))\hat{E}_3|_{\hat{x}}.
\]

To unburden the formulas, we write from now on usually $s_\tau:=\sin(\tau(q))$, $c_\tau:=\cos(\tau(q))$
if $\tau:\tilde{V}\to\R$ is some function, $\tilde{V}\subset Q$, and the point $q\in \tilde{V}$ is clear from the context.

Since $X_A,E_2|_x,Z_A$ (resp. $\hat{X}_A,\hat{E}_2|_{\hat{x}},\hat{Z}_A$)
form an orthonormal frame for every $q=(x,\hat{x};A)\in \tilde{O}_0$
and because $A(Z_A^\perp)=\hat{Z}_A^\perp$, 
it follows that there is a smooth $\phi:\tilde{O}_0\to\R$ such that
\[
AX_A=&c_{\phi} \hat{X}_A+s_{\phi} \hat{E}_2=c_{\phi}(c_{\hat{\theta}}\hat{E}_1+s_{\hat{\theta}}\hat{E}_3)+s_{\phi}\hat{E}_2 \\
AE_2=&-s_{\phi} \hat{X}_A+c_{\phi} \hat{E}_2=-s_{\phi}(c_{\hat{\theta}}\hat{E}_1+s_{\hat{\theta}}\hat{E}_3)+c_{\phi}\hat{E}_2 \\
AZ_A=&\hat{Z}_A.
\]
In particular,
\[
\hat{g}(AZ_A,\hat{E}_2)=0,
\]
for all $q=(x,\hat{x};A)\in \tilde{O}_0$.

It is clear that formulas in Eq. (\ref{eq:Rol2:LR_X_A}) on page \pageref{eq:Rol2:LR_X_A} hold  with $\Gamma^1_{(2,3)}=0$
and $Y=\hat{E}_2$.
Since they are very useful in computations,
we will now derive three relations, two of which simplify Eq. (\ref{eq:Rol2:LR_X_A}),
and all of which play an important role later on in the proof.

Differentiating the identity $\hat{g}(AZ_A,\hat{E}_2)=0$
with respect to $\LRD(X_A)|_q$, $\LRD(E_2)|_q$ and $\LRD(Z_A)|_q$, one at a time, yields on $\tilde{O}_0$,
\[
0=&\hat{g}(A\LRD(X_A)Z_{(\cdot)},\hat{E}_2)+\hat{g}(AZ_A,\hat{\nabla}_{AX_A}\hat{E}_2) \\
=&(-\LRD(X_A)|_q\theta+c_\theta\Gamma^1_{(3,1)}+s_\theta\Gamma^3_{(3,1)})\hat{g}(AX_A,\hat{E}_2)
+\hat{g}(\hat{Z}_A,-c_{\phi}\hat{\Gamma}^1_{(1,2)}\hat{X}_A) \\
=&s_{\phi}(-\LRD(X_A)|_q\theta+c_\theta\Gamma^1_{(3,1)}+s_\theta\Gamma^3_{(3,1)}) \\
0=&\hat{g}(A\LRD(E_2)Z_{(\cdot)},\hat{E}_2)+\hat{g}(AZ_A,\hat{\nabla}_{AE_2}\hat{E}_2) \\
=&(-\LRD(Y)|_q\theta+\Gamma^2_{(3,1)})\hat{g}(AX_A,\hat{E}_2)
+\hat{g}(\hat{Z}_A,s_{\phi}\hat{\Gamma}^1_{(1,2)}\hat{X}_A) \\
=&s_{\theta}(-\LRD(Y)|_q\theta+\Gamma^2_{(3,1)}) \\
0=&\hat{g}(A\LRD(Z_A)Z_{(\cdot)},\hat{E}_2)+\hat{g}(AZ_A,\hat{\nabla}_{AZ_A}\hat{E}_2) \\
=&(-\LRD(Z_A)|_q\theta-s_\theta\Gamma^1_{(3,1)}+c_\theta\Gamma^3_{(3,1)})\hat{g}(AX_A,\hat{E}_2) \\
&+\Gamma^1_{(1,2)}\hat{g}(AE_2,\hat{E}_2)
+\hat{g}(\hat{Z}_A, -\hat{\Gamma}^1_{(1,2)} \hat{Z}_A) \\
=&s_{\phi}(-\LRD(Z_A)|_q\theta-s_\theta\Gamma^1_{(3,1)}+c_\theta\Gamma^3_{(3,1)})
+c_{\phi}\Gamma^1_{(1,2)}-\hat{\Gamma}^1_{(1,2)}.
\]
Define
\[
\lambda(q):=\LRD(Z_A)|_q\theta+s_{\theta}\Gamma^1_{(3,1)}-c_{\theta}\Gamma^3_{(3,1)},\quad q\in \tilde{O}_0,
\]
which is a smooth function on $\tilde{O}_0$.
Since $\sin(\phi(q))=0$ would imply that $AE_2=\pm \hat{E}_2$,
we have  $\sin(\phi(q))\neq 0$ on $\tilde{O}_0\subset Q_0$ and hence
we get
\[
\LRD(X_A)|_q\theta&=c_\theta\Gamma^1_{(3,1)}+s_\theta\Gamma^3_{(3,1)} \\
\LRD(E_2)|_q\theta&=\Gamma^2_{(3,1)} \\
s_{\phi}\lambda&=c_{\phi}\Gamma^1_{(1,2)}-\hat{\Gamma}^1_{(1,2)}.
\]

These formulas, along with $\Gamma^1_{(2,3)}=0$,
simplify Eq. (\ref{eq:Rol2:LR_X_A}) to
\begin{align}\label{eq:warp:LR_X_A}
\LRD(X_A)|_q X_{(\cdot)}&=\Gamma^1_{(1,2)}E_2,\quad
\LRD(E_2)|_q X_{(\cdot)}=0,\quad 
\LRD(Z_A)|_q X_{(\cdot)}=\lambda Z_A
\nonumber \\
\LRD(X_A)|_qE_2&=-\Gamma^1_{(1,2)}X_A,\quad 
\LRD(E_2)|_qE_2=0,\quad
\LRD(Z_A)|_qE_2=-\Gamma^1_{(1,2)}Z_A \nonumber \\
\LRD(X_A)|_q Z_{(\cdot)}&=0,\quad
\LRD(E_2)|_q Z_{(\cdot)}=0,\quad
\LRD(Z_A)|_q Z_{(\cdot)}=-\lambda X_A+\Gamma^1_{(1,2)}E_2,
\end{align}
at $q\in \tilde{O}_0$.
We use these in the rest of the proof without further mention.

We make an interesting remark on the behaviour of $\lambda$
in the case where $q_0\in Q_1^+\backslash S_1$.
For any $q=(x,\hat{x};A)\in (Q_1^+\backslash S_1)\cap \tilde{O}$,
and any sequence (which exist as $Q_1\cap \tilde{O}$ is a nowhere dense subset of $\tilde{O}$)
$q_n\in \tilde{O}_0$, $q_n\to q$, we have $\cos(\phi(q_n))\to \cos(\phi(q))=1$,
hence $0\neq \sin(\phi(q_n))\to 0$.
Because
\[\lim_{n\to\infty} (c_{\phi}\Gamma^1_{(1,2)}-\hat{\Gamma}^1_{(1,2)})(q_n)=(c_{\phi}\Gamma^1_{(1,2)}-\hat{\Gamma}^1_{(1,2)})(q)=\Gamma^1_{(1,2)}(x)-\hat{\Gamma}^1_{(1,2)}(\hat{x})\neq 0
\]
as $q\in Q_1^+\backslash S_1$, we get
\[
\lim_{n\to\infty} \big(\sin(\phi(q_n))\lambda(q_n)\big)\neq 0,\quad \lim_{n\to\infty} \sin(\phi(q_n))=0,
\]
which implies that the sequence $\lambda(q_n)$ is unbounded,
\[
\lim_{n\to\infty} \lambda(q_n)=\pm\infty.
\]
In particular, we see that, even after shrinking $\tilde{O}$,
one cannot extend the definition of $\theta$ in a smooth, or even $C^1$, way onto $\tilde{O}$,
since if this were possible, the definition of $\lambda$ above would imply that $\lambda$ is continuous on $\tilde{O}$
and hence the above sequences $\lambda(q_n)$ would be bounded.
This fact about the unboudedness of $\lambda(q)$ as $q$ approaches $(Q_1^+\backslash S_1)\cap \tilde{O}$
will be used later.
To get around this problem, we will be working for a while uniquely on $\tilde{O}_0$.

Define on $\tilde{O}_0$ a 5-dimensional smooth distribution $\Delta$
spanned by
\[
\LRD(E_1)|_q, \LRD(E_2)|_q, \LRD(E_3)|_q, \nu(A\star E_2)|_q, \nu(A\star X_{A})|_q,\quad q\in \tilde{O}_0.
\]
We will proceed to show that the Lie algebra $\Lie(\Delta)$
spans at every point of $q\in \tilde{O}_0$ a 8-dimensional distribution $\Lie(\Delta)|_q$
which is then necessarily involutive.
Notice that we consider $\VF^k_{\Delta}$, $k=1,2,\dots$ and $\Lie(\Delta)$ as $\Cinf(\tilde{O}_0)$-modules.

Since $\LRD(X_{(\cdot)}),\LRD(E_2),\LRD(Z_{(\cdot)})$ span $\RDist$ on $\tilde{O}_0$,
they generate the module $\VF_{\RDist|_{\tilde{O}_0}}$
and hence $\Lie(\RDist|_{\tilde{O}_0})$.
Moreover, the brackets
\[
[\LRD(X_{(\cdot)}),\LRD(E_2)]|_q=&-\Gamma^1_{(1,2)}\LRD(X_A)|_q \\
[\LRD(E_2),\LRD(Z_{(\cdot)})]|_q=&\Gamma^1_{(1,2)}\LRD(Z_A)|_q-K_1^\Rol \nu(A\star X_A)|_q-\alpha \nu(A\star E_2)|_q \\
[\LRD(Z_{(\cdot)}),\LRD(X_{(\cdot)})]|_q=&\lambda\LRD(Z_A)|_q-\alpha \nu(A\star X_A)|_q-K_2^\Rol \nu(A\star E_2)|_q,
\]
along with the definition of $X_A,Z_A$, show that $\VF_{\RDist|_{\tilde{O}_0}}^2\subset \VF_{\Delta}$.

The first three Lie brackets in Proposition \ref{pr:3D-2} case (ii) show that $\VF_{\Delta}^2$ contains
vector fields $L_1,L_3$ given by
$L_1|_q=\LNSD(E_1)|_q-\Gamma^1_{(1,2)}\nu(A\star E_3)|_q$,
$L_3|_q=\LNSD(E_3)|_q+\Gamma^1_{(1,2)}\nu(A\star E_1)|_q$,
and also $L_2|_q$, which in this setting is just the zero-vector field
on $\tilde{O}_0$.

We define $F_X|_q:=c_{\theta}L_1|_q+s_{\theta}L_3|_q$
and $F_Z|_q:=-s_{\theta}L_1|_q+c_{\theta}L_3|_q-\Gamma^1_{(1,2)}\nu(A\star X_A)|_q$,
hence $F_X,F_X\in \VF_{\Delta}^2$
and one easily sees that they simplify to
\[
F_X|_q=&\LNSD(X_A)|_q-\Gamma^1_{(1,2)}\nu(A\star Z_A)|_q \\
F_Z|_q=&\LNSD(Z_A)|_q.
\]
It is clear that the vector fields
\[
\LRD(X_{(\cdot)}),\LRD(E_2),\LRD(Z_{(\cdot)}),\nu((\cdot)\star E_2),\nu((\cdot)\star X_{(\cdot)}),F_X,F_Z
\]
span the same $\Cinf(\tilde{O}_0)$-submodule of $\VF_{\Delta}^2$ as do
\[
\LRD(E_1),\LRD(E_2),\LRD(E_3),\nu((\cdot)\star E_2),\nu((\cdot)\star X_{(\cdot)}),L_1,L_3.
\]

We now want to find generators of $\VF_{\Delta}^2$.
By what we have already done and said,
it remains us to compute
need to prove that the Lie-brackets between the 4 vector fields
\[
\LRD(X_A)|_q,\LRD(E_2)|_q,\LRD(Z_A)|_q,\nu(A\star E_2)|_q
\]
and $\nu((\cdot)\star X_{(\cdot)})|_q$.

Since we will have to derivate $X_A$,
it follows that the derivatives of $\theta$ will also appear.
That is why we first compute with respect to all the (pointwise linearly independent) vectors that appear above.
As a first step, compute
\[
F_X|_q Z_{(\cdot)}=&(-F_X|_q\theta+c_{\theta}\Gamma^1_{(3,1)}+s_{\theta}\Gamma^3_{(3,1)})X_A \\
F_Z|_q Z_{(\cdot)}=&\LNSD(Z_A)|_q Z_{(\cdot)}=(-F_Z|_q\theta - s_{\theta}\Gamma^1_{(3,1)} + c_{\theta}\Gamma^3_{(3,1)} )X_A+\Gamma^1_{(1,2)}E_2.
\]

Knowing already $\LRD(X_A)|_q\theta,\LRD(Y)|_q\theta,\LRD(Z_A)|_q\theta$,
we derivate the identity $$\hat{g}(AZ_A,\hat{E}_2)=0$$
with respect to $\nu(A\star E_2)|_q,\nu(A\star X_A)|_q,F_X|_q,F_Z|_q$
which gives (notice that the derivative of $\hat{E}_2$ with respect to these vanishes)
\[
0=&\hat{g}(A(\star E_2)Z_A-\nu(A\star E_2)|_q\theta AX_A,\hat{E}_2) \\
=&(1-\nu(A\star E_2))\hat{g}(AX_A,\hat{E}_2)=s_{\phi}(1-\nu(A\star E_2)) \\
0=&\hat{g}(A(\star X_A)Z_A-\nu(A\star X_A)|_q\theta AX_A,\hat{E}_2) \\
=&-\hat{g}(AE_2,\hat{E}_2)-\nu(A\star X_A)|_q\theta \hat{g}(AX_A,\hat{E}_2) \\
=&-c_{\phi}-s_{\phi}\nu(A\star X_A)|_q\theta \\
0=&\hat{g}(-\Gamma^1_{(1,2)}A(\star Z_A)Z_A,\hat{E}_2)
+(-F_X|_q\theta+c_{\theta}\Gamma^1_{(3,1)}+s_{\theta}\Gamma^3_{(3,1)})\hat{g}(AX_A,\hat{E}_2) \\
=&s_{\phi}(-F_X|_q\theta+c_{\theta}\Gamma^1_{(3,1)}+s_{\theta}\Gamma^3_{(3,1)}) \\
0=&(-F_Z|_q\theta - s_{\theta}\Gamma^1_{(3,1)} + c_{\theta}\Gamma^3_{(3,1)} )\hat{g}(AX_A,E_2)+\Gamma^1_{(1,2)}\hat{g}(AE_2,\hat{E}_2) \\
=&s_{\phi}(-F_Z|_q\theta - s_{\theta}\Gamma^1_{(3,1)} + c_{\theta}\Gamma^3_{(3,1)} )
+c_{\phi}\Gamma^1_{(1,2)}
\]
and since $s_{\phi}\neq 0$ on $\tilde{O}_0$,
\[
\nu(A\star E_2)|_q\theta=&1 \\
\nu(A\star X_A)|_q\theta=&-\cot(\hat{\phi}) \\
F_X|_q\theta=&c_{\theta}\Gamma^1_{(3,1)}+s_{\theta}\Gamma^3_{(3,1)} \\
F_Z|_q\theta=&- s_{\theta}\Gamma^1_{(3,1)} + c_{\theta}\Gamma^3_{(3,1)}+\cot(\phi)\Gamma^1_{(1,2)}.
\]
These simplify the above formulas to
\[
F_X|_q Z_{(\cdot)}=&0 \\
F_Z|_q Z_{(\cdot)}=&\LNSD(Z_A)|_q Z_{(\cdot)}=-\cot(\phi)X_A+\Gamma^1_{(1,2)}E_2
\]
and moreover it is now easy to see that for $q\in \tilde{O}_0$,
\[
F_X|_q X_{(\cdot)}=&\Gamma^1_{(1,2)}E_2,\quad
F_X|_q E_2=-\Gamma^1_{(1,2)}X_A \\
F_Z|_q X_{(\cdot)}=&\cot(\phi)\Gamma^1_{(1,2)}Z_A,\quad
F_Z|_q E_2=-\Gamma^1_{(1,2)}Z_A.
\]

The brackets
\[
[\LRD(X_{(\cdot)}),\nu((\cdot)\star X_{(\cdot)})]|_q
=&\cot(\phi)\LRD(Z_A)|_q-\LNSD(A\star (\star X_A)X_A)|_q+\Gamma^1_{(1,2)}\nu(A\star E_2)|_q \\
=&\cos(\phi)\LRD(Z_A)|_q+\Gamma^1_{(1,2)}\nu(A\star E_2)|_q \\
[\LRD(E_2),\nu((\cdot)\star X_{(\cdot)})]|_q=&-\LNSD(A(\star X_A)E_2)|_q+\nu(A\star 0)|_q=F_Z|_q-\LRD(Z_A)|_q \\
[\LRD(Z_{(\cdot)}),\nu((\cdot)\star X_{(\cdot)})]|_q=&-\cot(\phi)\LRD(X_A)|_q-\LNSD(A\star (\star X_A)Z_A)|_q+\nu(A\star (\lambda Z_A)) \\
=& -\cot(\phi)\LRD(X_A)|_q+\LRD(E_2)|_q-(\LNSD(E_2)|_q-\lambda\nu(A\star Z_A)|_q) \\
[\nu(A\star E_2),\nu((\cdot)\star X_{(\cdot)})]|_q=&\nu(A[\star E_2,\star X_A]_{\so})|_q+\nu(A\star Z_A)|_q=0,
\]
show that if on defines
\[
F_Y|_q:=\LNSD(E_2)|_q-\lambda\nu(A\star Z_A)|_q,
\] 
then one may write
\[
[\LRD(Z_{(\cdot)}),\nu((\cdot)\star X_{(\cdot)})]|_q=-\cot(\phi)\LRD(X_A)|_q+\LRD(E_2)|_q-F_Y|_q
\]
and hence we have shown that $\VF_{\Delta}^2$ is generated by vector fields
\[
\LRD(X_{(\cdot)}),\LRD(E_2),\LRD(Z_{(\cdot)}),\nu((\cdot)\star E_2),\nu((\cdot)\star X_{(\cdot)}),F_X,F_Y,F_Z
\]
which are all pointwise linearly independent on $\tilde{O}_0$.

Next we will proceed to show that the $\VF_{\Delta}^2$ generated by the above
8 vector fields is in fact involutive, which then establishes that $\Lie(\Delta)=\VF_{\Delta}^2$.

At first, the last 9 brackets in Proposition \ref{pr:3D-2} (recall that we have $\Gamma^1_{(2,3)}=0$)
show that $[F_Z,F_X]$ and the brackets of $\LRD(X_{(\cdot)})$, $\LRD(E_2)$, $\LRD(Z_{(\cdot)})$, $\nu((\cdot)\star E_2)$,
with $F_X$ and $F_Z$ all belong to $\VF_{\Delta}^2$
as well as do
\[
[F_X,\nu((\cdot)\star X_{(\cdot)})]|_q=&
-\LNSD(-\cot(\phi)Z_A)|_q+\nu(A\star (\LNSD(X_A)|_q X_{(\cdot)}))|_q \\
&-\Gamma^1_{(1,2)}\nu(A[\star Z_A,\star X_A]_{\so}+\nu(A\star Z_A)|_q X_{(\cdot)}-\cot(\phi)A\star X_A)|_q \\
=&\cot(\phi)\LNSD(Z_A)|_q+\nu(A\star F_X|_q X_{(\cdot)})|_q \\
&-\Gamma^1_{(1,2)}\nu(A\star E_2)|_q+\Gamma^1_{(1,2)}\cot(\phi)\nu(A\star X_A)|_q \\
=&\cot(\phi)F_Z|_q+\Gamma^1_{(1,2)}\cot(\phi)\nu(A\star X_A)|_q \\
[F_Z,\nu((\cdot)\star X_{(\cdot)})]|_q=&-\LNSD(\cot(\phi)X_A)|_q+\cot(\phi)\Gamma^1_{(1,2)}\nu(A\star Z_A)|_q \\
=&-\cot(\phi)F_X|_q.
\]
Therefore, it remains to us to prove that the brackets of $F_Y$
with all the other 7 generators of $\VF_{\Delta}^2$, as listed above, also belong to $\VF_{\Delta}^2$.

Now it is clear that since the expression of $F_Y$
involves $\lambda$, which was defined earlier,
we need to know its derivatives in all the possible directions (except in $F_Y$-direction)
as well as the expression for $F_Y|_q\theta$.
We begin by computing this latter derivative.

As usual, the way to proceed is to derivate $0=\hat{g}(AZ_A,\hat{E}_2)$
w.r.t. $F_Y|_q$,
for which, we first compute
\[
F_Y|_q Z_{(\cdot)}=(-F_Y|_q\theta +\Gamma^2_{(3,1)})X_A
\]
and hence (notice that $F_Y|_q \hat{E}_2=0$)
\[
0=\hat{g}(-\lambda A(\star Z_A)Z_A,\hat{E}_2)+(-F_Y|_q\theta +\Gamma^2_{(3,1)})\hat{g}(AX_A,\hat{E}_2)
=s_{\phi}(-F_Y|_q\theta +\Gamma^2_{(3,1)}),
\]
from where
\[
F_Y|_q\theta=\Gamma^2_{(3,1)}.
\]
One then easily computes that on $\tilde{O}_0$,
\[
F_Y|_q X_{(\cdot)}=0,\quad F_Y|_q E_2=0,\quad F_Y|_q Z_{(\cdot)}=0.
\]

To compute the derivatives of $\lambda$, we differentiate the
identity
$s_{\phi}\lambda=c_{\phi}\Gamma^1_{(1,2)}-\hat{\Gamma}^1_{(1,2)}$
proved above.
Obviously, this will require the knowledge of derivatives of $\phi$,
so we begin there.

To do that, one will differentiate the identity $c_{\phi}=\hat{g}(AE_2,\hat{E}_2)$
in different directions. First it is clear that
\[
\hat{\nabla}_{AX_A} \hat{E}_2&=-c_{\phi}\hat{\Gamma}^1_{(1,2)}\hat{X}_A \\
\hat{\nabla}_{AE_2} \hat{E}_2&=s_{\phi}\hat{\Gamma}^1_{(1,2)}\hat{X}_A \\
\hat{\nabla}_{AZ_A} \hat{E}_2&=-\hat{\Gamma}^1_{(1,2)}\hat{Z}_A,
\]
and hence
\[
-s_{\phi}\LRD(X_A)|_q\phi=&\hat{g}(-\Gamma^1_{(1,2)}AX_A,\hat{E}_2)+\hat{g}(AE_2,\hat{\nabla}_{AX_A} \hat{E}_2) \\
=&-s_\phi \Gamma^1_{(1,2)}+\hat{g}(AE_2,-c_{\phi}\hat{\Gamma}^1_{(1,2)}\hat{X}_A) \\
=&-s_\phi \Gamma^1_{(1,2)}+s_{\phi}c_{\phi}\hat{\Gamma}^1_{(1,2)} \\
-s_{\phi}\LRD(E_2)|_q\phi=&\hat{g}(A\LRD(E_2)|_qE_2,\hat{E}_2)
+\hat{g}(AE_2,\hat{\nabla}_{AE_2}\hat{E}_2) \\
=&0+\hat{g}(AE_2,s_{\phi}\hat{\Gamma}^1_{(1,2)}\hat{X}_A)
=-s_{\phi}^2\hat{\Gamma}^1_{(1,2)} \\
-s_{\phi}\LRD(Z_A)|_q\phi=&\hat{g}(-\Gamma^1_{(1,2)}AZ_A,\hat{E}_2)+\hat{g}(AE_2,-\hat{\Gamma}^1_{(1,2)}\hat{Z}_A)=0 \\
-s_{\phi}\nu(A\star E_2)|_q\phi=&\hat{g}(A(\star E_2)E_2,\hat{E}_2)=0 \\
-s_{\phi}\nu(A\star X_A)|_q\phi=&\hat{g}(A(\star X_A)E_2,\hat{E}_2)=\hat{g}(AZ_A,\hat{E}_2)=0 \\
-s_{\phi}F_X|_q\phi=&\hat{g}(-\Gamma^1_{(1,2)}A(\star Z_A)E_2-\Gamma^1_{(1,2)}AX_A,\hat{E}_2)=0 \\
-s_{\phi} F_Z|_q\phi=&\hat{g}(-\Gamma^1_{(1,2)}AZ_A,\hat{E}_2)=0 \\
-s_{\phi} F_Y|_q\phi=&\hat{g}(-\lambda A(\star Z_A)E_2+0,\hat{E}_2)=\lambda\hat{g}(AX_A,\hat{E}_2)=s_{\phi}\lambda.
\]
Because $s_{\phi}\neq 0$ on $\tilde{O}_0$, these yield
\[
\LRD(X_A)|_q\phi=&\Gamma^1_{(1,2)}-c_\phi\hat{\Gamma}^1_{(1,2)} \\
\LRD(E_2)|_q\phi=&s_{\phi}\hat{\Gamma}^1_{(1,2)} \\
F_Y|_q\phi=&-\lambda \\
\LRD(Z_A)|_q\phi=&\nu(A\star E_2)|_q\phi=\nu(A\star X_A)|_q\phi=F_X|_q\phi=F_Z|_q\phi=0.
\]

Next notice that
\[
& \LRD(X_A)|_q\Gamma^1_{(1,2)}=F_X|_q\Gamma^1_{(1,2)}=X_A(\Gamma^1_{(1,2)})=0 \\
& \LRD(E_2)|_q\Gamma^1_{(1,2)}=F_Y|_q\Gamma^1_{(1,2)}=E_2(\Gamma^1_{(1,2)}) \\
& \LRD(Z_A)|_q\Gamma^1_{(1,2)}=F_Z|_q\Gamma^1_{(1,2)}=Z_A(\Gamma^1_{(1,2)})=0,
\]
because $X_A,Z_A\in E_2^\perp$
and similarly, since $\hat{X}_A,\hat{Z}_A\in \hat{E}_2^\perp$,
\[
& \LRD(X_A)|_q\hat{\Gamma}^1_{(1,2)}=AX_A(\hat{\Gamma}^1_{(1,2)})=s_{\phi}\hat{E}_2(\hat{\Gamma}^1_{(1,2)}) \\
& \LRD(E_2)|_q\hat{\Gamma}^1_{(1,2)}=AE_2(\hat{\Gamma}^1_{(1,2)})=c_{\phi}\hat{E}_2(\hat{\Gamma}^1_{(1,2)}) \\
& \LRD(Z_A)|_q\hat{\Gamma}^1_{(1,2)}=AZ_A(\hat{\Gamma}^1_{(1,2)})=0 \\
& F_X|_q\hat{\Gamma}^1_{(1,2)}=F_Y|_q\hat{\Gamma}^1_{(1,2)}=F_Z|_q\hat{\Gamma}^1_{(1,2)}=0.
\]

Finally, derivating the identity $s_\phi \lambda=c_\phi \Gamma^1_{(1,2)}-\hat{\Gamma}^1_{(1,2)}$
and using the previously derived rules,
\[
& c_{\phi}(\Gamma^1_{(1,2)}-c_{\phi}\hat{\Gamma}^1_{(1,2)})\lambda
+s_{\phi}\LRD(X_A)|_q\lambda=-s_{\phi}\Gamma^1_{(1,2)}(\Gamma^1_{(1,2)}-c_{\phi}\hat{\Gamma}^1_{(1,2)})
-s_{\phi}\hat{E}_2(\hat{\Gamma}^1_{(1,2)}) \\
& s_{\phi}c_{\phi}\hat{\Gamma}^1_{(1,2)}\lambda+s_{\phi}\LRD(E_2)|_q\lambda
=-s_{\phi}^2\hat{\Gamma}^1_{(1,2)}\Gamma^1_{(1,2)}+c_{\phi}E_2(\Gamma^1_{(1,2)})-c_{\phi}\hat{E}_2(\hat{\Gamma}^1_{(1,2)}) \\
& s_{\phi}\LRD(Z_A)|_q\lambda=0 \\
& s_{\phi}\nu(A\star E_2)|_q\lambda=0 \\
& s_{\phi}\nu(A\star X_A)|_q\lambda=0 \\
& s_{\phi}F_X|_q\lambda=0 \\
& -c_{\phi}\lambda^2+s_{\phi}F_Y|_q\lambda=s_{\phi}\Gamma^1_{(1,2)}\lambda+c_{\phi}E_2(\Gamma^1_{(1,2)}) \\
& s_{\phi}F_Z|_q\lambda=0
\]
from which the last 6 simplify immediately to
\[
\LRD(Z_A)|_q\lambda=&\nu(A\star E_2)|_q\lambda=\nu(A\star X_A)|_q\lambda=F_X|_q\lambda=F_Z|_q\lambda=0 \\
F_Y|_q\lambda=&\cot(\phi)(E_2(\Gamma^1_{(1,2)})+\lambda^2)+\Gamma^1_{(1,2)}\lambda.
\]
Next simplify $\LRD(E_2)|_q\lambda$ by using first $s_{\phi}\lambda=c_{\phi}\Gamma^1_{(1,2)}-\hat{\Gamma}^1_{(1,2)}$,
\[
s_{\phi}\LRD(E_2)|_q\lambda
=&-s_{\phi}c_{\phi}\hat{\Gamma}^1_{(1,2)}\lambda-s_{\phi}^2\hat{\Gamma}^1_{(1,2)}\Gamma^1_{(1,2)}+c_{\phi}E_2(\Gamma^1_{(1,2)})-c_{\phi}\hat{E}_2(\hat{\Gamma}^1_{(1,2)}) \\
=&-c_{\phi}\hat{\Gamma}^1_{(1,2)}(c_{\phi}\Gamma^1_{(1,2)}-\hat{\Gamma}^1_{(1,2)})-s_{\phi}^2\hat{\Gamma}^1_{(1,2)}\Gamma^1_{(1,2)}+c_{\phi}E_2(\Gamma^1_{(1,2)})-c_{\phi}\hat{E}_2(\hat{\Gamma}^1_{(1,2)}) \\
=&-\Gamma^1_{(1,2)}\hat{\Gamma}^1_{(1,2)}+c_{\phi}E_2(\Gamma^1_{(1,2)})
+c_{\phi}(-\hat{E}_2(\hat{\Gamma}^1_{(1,2)})+(\hat{\Gamma}^1_{(1,2)})^2)
\]
and then using $-K=-\hat{E}_2(\hat{\Gamma}^1_{(1,2)})+(\hat{\Gamma}^1_{(1,2)})^2$,
to obtain
\[
s_{\phi}\LRD(E_2)|_q\lambda=&-\Gamma^1_{(1,2)}\hat{\Gamma}^1_{(1,2)}+c_{\phi}E_2(\Gamma^1_{(1,2)})-c_{\phi}K,
\]
once more $\hat{\Gamma}^1_{(1,2)}=c_{\phi}\Gamma^1_{(1,2)}-s_{\phi}\lambda$,
\[
s_{\phi}\LRD(E_2)|_q\lambda=&-\Gamma^1_{(1,2)}(c_{\phi}\Gamma^1_{(1,2)}-s_{\phi}\lambda)+c_{\phi}E_2(\Gamma^1_{(1,2)})-c_{\phi}K \\
=&c_{\phi}(-K-(\Gamma^1_{(1,2)})^2+E_2(\Gamma^1_{(1,2)}))+s_{\phi}\Gamma^1_{(1,2)}\lambda,
\]
which finally simplifies, thanks to $-K=-E_2(\Gamma^1_{(1,2)})+(\Gamma^1_{(1,2)})^2$ and $s_{\phi}\neq 0$,
\[
\LRD(E_2)|_q\lambda=\lambda\Gamma^1_{(1,2)}.
\]

Next we simplify $\LRD(X_A)|_q\lambda$ by using the same identities as above
when simplifying $\LRD(E_2)|_q\lambda$:
\[
s_{\phi}\LRD(X_A)|_q\lambda
=&-c_{\phi}(\Gamma^1_{(1,2)}-c_{\phi}\hat{\Gamma}^1_{(1,2)})\lambda
-s_{\phi}\Gamma^1_{(1,2)}(\Gamma^1_{(1,2)}-c_{\phi}\hat{\Gamma}^1_{(1,2)})
-s_{\phi}\hat{E}_2(\hat{\Gamma}^1_{(1,2)}) \\
=&-\lambda(s_{\phi}\lambda+\hat{\Gamma}^1_{(1,2)})+c_{\phi}^2\hat{\Gamma}^1_{(1,2)}\lambda  \\
&-s_{\phi}(\Gamma^1_{(1,2)})^2
+s_{\phi}\hat{\Gamma}^1_{(1,2)}(s_\phi\lambda+\hat{\Gamma}^1_{(1,2)})
 -s_{\phi}(K+(\hat{\Gamma}^1_{(1,2)})^2) \\
 =&-s_{\phi}(\lambda^2+(\Gamma^1_{(1,2)})^2+K)
 -\lambda \hat{\Gamma}^1_{(1,2)}+c_{\phi}^2\lambda\hat{\Gamma}^1_{(1,2)}+s_{\phi}^2\hat{\Gamma}^1_{(1,2)}\lambda \\
 =&-s_{\phi}(\lambda^2+(\Gamma^1_{(1,2)})^2+K),
\]
which implies, at last,
\[
\LRD(X_A)|_q\lambda=-(\lambda^2+(\Gamma^1_{(1,2)})^2+K).
\]

Finally, on $\tilde{O}_0$, we compute the brackets
\[
[\LRD(X_A),F_Y]|_q
=&\LNSD(-\Gamma^1_{(1,2)}X_A)|_q-\LRD(\LNSD(E_2)|_q X_{(\cdot)})|_q \\
&+\nu(AR(X_A\wedge E_2)-\hat{R}(AX_A\wedge 0)A)|_q
-\LRD(X_A)|_q\lambda \nu(A\star Z_A)|_q \\
&-\lambda\big(-\LNSD(A(\star Z_A)X_A)-\LRD(\nu(A\star Z_A)|_q X_{(\cdot)})
+\nu(A\star 0)|_q\big) \\
=&-\Gamma^1_{(1,2)}F_X|_q-\LRD(F_Y|_q X_{(\cdot)})|_q
+\lambda \LRD(E_2)|_q-\lambda F_Y|_q \\
&+\underbrace{(-(\Gamma^1_{(1,2)})^2-K-\LRD(X_A)|_q\lambda-\lambda^2)}_{=0}\nu(A\star Z_A)|_q \\
[\LRD(E_2),F_Y]|_q
=&-\LRD(E_2)|_q\lambda \nu(A\star Z_A)|_q-\lambda (-\LNSD(A(\star Z_A)E_2)|_q+\nu(A\star 0)|_q) \\
=&-\lambda \LRD(X_A)|_q+\lambda F_X|_q
+\underbrace{(\lambda\Gamma^1_{(1,2)}-\LRD(E_2)|_q\lambda )}_{=0}\nu(A\star Z_A)|_q \\
[\LRD(Z_A),F_Y]|_q
=&\LNSD(-\Gamma^1_{(1,2)}Z_A)|_q+\LRD(\LNSD(E_2)|_q Z_{(\cdot)}) \\
&+\nu(AR(Z_A\wedge E_2)-\hat{R}(AZ_A\wedge 0)A)|_q
-\LRD(Z_A)|_q\lambda \nu(A\star Z_A)|_q \\
&-\lambda \big(-\LNSD(A(\star Z_A)Z_A)|_q+\LRD(\nu(A\star Z_A)|_q Z_{(\cdot)})|_q) \\
&-\lambda\nu(A\star (-\lambda X_A+\Gamma^1_{(1,2)}E_2)|_q \\
=&-\Gamma^1_{(1,2)}F_Z|_q+\LRD(F_Y|_q Z_{(\cdot)})|_q+K\nu(A\star X_A)|_q \\
&-\underbrace{\LRD(Z_A)|_q\lambda}_{=0}\nu(A\star Z_A)|_q-\lambda\nu(A\star (-\lambda X_A+\Gamma^1_{(1,2)}E_2)|_q
\]
\[
[\nu((\cdot)\star E_2),F_Y]|_q
=&-\nu(A\star E_2)|_q\lambda \nu(A\star Z_A)|_q \\
&-\lambda\nu(A[\star E_2,\star Z_A]_{\so}-\nu(A\star E_2)|_q\theta A\star X_A)|_q \\
=&-\nu(A\star E_2)|_q\lambda \nu(A\star Z_A)|_q=0 \\
[\nu((\cdot)\star X_{(\cdot)},F_Y]|_q
=&-\nu(A\star \LNSD(E_2)|_q X_{(\cdot)})|_q
-\nu(A\star X_A)|_q\lambda \nu(A\star Z_A)|_q \\
&-\lambda\nu(A[\star X_A,\star Z_A]_{\so} -\nu(A\star X_A)|_q\theta A\star X_A)|_q \\
&-\lambda\nu(-A\star \nu(A\star Z_A)|_qX_{(\cdot)})      )|_q \\
=&-\nu(A\star \underbrace{F_Y|_q X_{(\cdot)}}_{=0})|_q-\underbrace{\nu(A\star X_A)|_q\lambda}_{=0} \nu(A\star Z_A)|_q \\
&-\lambda(A\star (-E_2+\cot(\phi)X_A))|_q \\
[F_Z,F_Y]|_q=&\LNSD(-\Gamma^1_{(1,2)}Z_A-\LNSD(E_2)|_qZ_{(\cdot)})|_q+\nu(AR(Z_A\wedge E_2))|_q \\
&-F_Z|_q\lambda\nu(A\star Z_A)|_q
-\lambda(-\LNSD(\nu(A\star Z_A)|_q Z_{(\cdot)})+\nu(A\star F_Z|_q Z_{(\cdot)})|_q) \\
=&-\Gamma^1_{(1,2)}F_Z|_q-\LNSD(\underbrace{F_Y|_q Z_{(\cdot)}}_{=0})|_q+K\nu(A\star X_A)|_q \\
&-\underbrace{F_Z|_q\lambda}_{=0}\nu(A\star Z_A)|_q
-\lambda\nu(A\star (-\cot(\phi)X_A+\Gamma^1_{(1,2)}E_2))|_q
\]
and finally, noticing that $-\lambda F_X|_q+\Gamma^1_{(1,2)}F_Y|_q=-\lambda\LNSD(X_A)|_q+\Gamma^1_{(1,2)}\LNSD(E_2)|_q$,
\[
[F_X,F_Y]|_q
=&\LNSD(\LNSD(X_A)|_q E_2-\LNSD(E_2)|_q X_{(\cdot)})|_q+\nu(AR(X_A\wedge E_2))|_q \\
&-\LRD(X_A)|_q\lambda\nu(A\star Z_A)|_q+E_2(\Gamma^1_{(1,2)})\nu(A\star Z_A)|_q \\
&-\lambda(-\LNSD(\nu(A\star Z_A)|_q X_{(\cdot)})+\nu(A\star \LNSD(X_A)|_q Z_{(\cdot)})|_q) \\
&+\Gamma^1_{(1,2)}\nu(A\star \LNSD(E_2)|_q Z_{(\cdot)})|_q
+\Gamma^1_{(1,2)}\nu(A\star Z_A)|_q\lambda \nu(A\star Z_A)|_q \\
=&-\Gamma^1_{(1,2)}\LNSD(X_A)|_q-\LNSD(\underbrace{F_Y|_q X_{(\cdot)}}_{=0})|_q \\
&+\nu(A\star \underbrace{(-\lambda F_X|_q+\Gamma^1_{(1,2)}F_Y|_q)Z_{(\cdot)}}_{=0})|_q \\
&+(-K-F_X|_q\lambda+E_2(\Gamma^1_{(1,2)}))\nu(A\star Z_A)|_q \\
=&-\Gamma^1_{(1,2)}F_X|_q+(-K-F_X|_q\lambda+E_2(\Gamma^1_{(1,2)})-(\Gamma^1_{(1,2)})^2)\nu(A\star Z_A)|_q,
\]
which, after using $F_X|_q\lambda=0$ and Eq. \eqref{eq:RsE3sE3},
simplifies to $[F_X,F_Y]|_q=-\Gamma^1_{(1,2)}F_X|_q$.

Since all these Lie brackets also belong to $\VF_{\Delta}^2$,
we may finally conclude that $\VF_{\Delta}^2$ is involutive and therefore
\[
\Lie(\Delta)=\VF_{\Delta}^2.
\]
Therefore the span of $\Lie(\Delta)$ at each point $\tilde{O}_0$
is 8-dimensional subspace of $T|_q Q$, since $\VF_{\Delta}^2$
is generated by 8 pointwise linearly independent vector fields.

Since $q_0\in Q\backslash S_1$ was arbitrary and since the choice
of $X_A,E_2,Z_A$ in $\tilde{O}_0$ are unique up to multiplication by $-1$,
we have shown that on $Q_0$ there is a smooth
5-dimensional distribution $\Delta$ containing $\RDist|_{Q_0}$ such that $\Lie(\Delta)=\VF_{\Delta}^2$
spans an 8-dimensional distribution $\mc{D}$ and which is then, by construction, involutive.

We already know from the beginning of the proof that $q\in S_1$ implies that
$\mc{O}_{\RDist}(q)\subset S_1$ so, equivalently,
$q\in Q\backslash S_1$ implies that $\mc{O}_{\RDist}(q)\subset Q\backslash S_1$.
Hence we are interested to see how $\mc{D}$
can be extended on all over $Q\backslash S_1$
i.e. we have to see how to define it on $Q_1\backslash S_1$.

For this purpose, we define the Sasaki metric $G$ on $Q$
by
\[
& \mc{X}=\LNSD(X,\hat{X})|_q+\nu(A\star Z)|_q,\quad \mc{Y}=\LNSD(Y,\hat{Y})|_q+\nu(A\star W)|_q \\
& G(\mc{X},\mc{Y})=g(X,Y)+\hat{g}(\hat{X},\hat{Y})+g(Z,W),
\]
for $q=(x,\hat{x};A)\in Q$, $X,Y,Z,W\in T|_x M$, $\hat{X},\hat{Y}\in T|_{\hat{x}} \hat{M}$.
Notice that any vector $\mc{X}\in T|_q Q$ can be written in the form $\LNSD(X,\hat{X})|_q+\nu(A\star Z)|_q$
for some $X,\hat{X},Z$ as above.

Since $\mc{D}$ is a smooth codimension 1 distribution on $Q_0$, it has
a smooth normal line bundle $\mc{D}^\perp$ w.r.t. $G$ defined on $Q_0$ which uniquely determines $\mc{D}$.
We will use the Sasaki metric $G$ to determine a smooth vector field $\mc{N}$ near a point $q_0\in Q_1\backslash S_1$
spanning $\mc{D}^\perp$.

So let $q_0\in Q_1\backslash S_1$ and assume, as before, that $q_0\in Q_1^+\backslash S_1$
the case of $Q_1^-\backslash S_1$ being handled similarly.
Take the frames $E_1,E_2,E_3$, $\hat{E}_1,\hat{E}_2,\hat{E}_3$ 
and $\tilde{O}$, $\tilde{O}_0$, $X_A,Z_A$ as done above (the case (b)).
Because $\cos(\phi(q_0))\Gamma^1_{(1,2)}(x_0)-\hat{\Gamma}^1_{(1,2)}(\hat{x}_0)\neq 0$,
one may assume after shrinking $\tilde{O}$ around $q_0$
that we have $\cos(\phi(q))\Gamma^1_{(1,2)}(x)-\hat{\Gamma}^1_{(1,2)}(\hat{x})\neq 0$
for all $q=(x,\hat{x};A)\in \tilde{O}$,
which then implies that $\lambda(q)\neq 0$ on $\tilde{O}_0$.
Here to say what is the value of $\cos(\phi(q))$ even at $q\in Q_1\backslash S_1$,
we use the fact that $\cos(\phi(q))=g(AE_2,\hat{E}_2)$ for all $q\in \tilde{O}$
(though $\phi(q)$ is not \emph{a priori} defined).

To determine a smooth vector field $\mc{N}\in \mc{D}^\perp$ on $\tilde{O}_0$,
we write
\[
\mc{N}|_q=&a_1\LNSD(X_A)|_q+a_2+\LNSD(E_2)|_q+a_3\LNSD(Z_A)|_q \\
&+b_1\LNSD(AX_A)|_q+b_2+\LNSD(AE_2)|_q+b_3\LNSD(AZ_A)|_q \\
&+v_1\nu(A\star X_A)|_q+v_2\nu(A\star E_2)|_q+v_3\nu(A\star Z_A)|_q
\]
and since this must be $G$-orthogonal to $\mc{D}$, we get
\[
0=&G(\mc{N},\LRD(X_A))=a_1+b_1,\quad 0=G(\mc{N},\LRD(E_2))=a_2+b_2,\quad 0=G(\mc{N},\LRD(Z_A))=a_3+b_3 \\
0=&G(\mc{N},\nu(A\star X_A))=v_1,\quad 0=G(\mc{N},\nu(A\star E_2))=v_2 \\
0=&G(\mc{N},F_X)=a_1-\Gamma^1_{(1,2)}v_3,\quad 0=G(\mc{N},F_Y)=a_2-\lambda v_3,
\quad 0=G(\mc{N},F_Z)=a_3.
\]
So if we set $v_3=\frac{1}{\lambda}$
and introduce the notation
\[
\LRD^\perp(X)|_q:=\LNSD(X,-AX)\in \NSDist|_q,\quad q=(x,\hat{x};A)\in Q,\ X\in T|_x M
\]
we get a smooth vector field $\mc{N}$ on $\tilde{O}_0$
which is $G$-perpendicular to $\mc{D}$ and is given by
\[
\mc{N}|_q=&\frac{1}{\lambda(q)}\Gamma^1_{(1,2)}(x)\LRD^\perp(X_A)|_q+\LRD^\perp(E_2)+\frac{1}{\lambda(q)}\nu(A\star Z_A)|_q,\quad q\in \tilde{O}_0 \\
=&\frac{c_{\theta}}{\lambda(q)}(\Gamma^1_{(1,2)}\LRD^\perp(E_1)|_q+\nu(A\star E_3)|_q)
+\LRD^\perp(E_2)|_q \\
&+\frac{s_{\theta}}{\lambda(q)}(\Gamma^1_{(1,2)}\LRD^\perp(E_3)|_q-\nu(A\star E_1)|_q).
\]
i.e.
\[
\mc{N}|_q=H_1(q)\mc{X}_1|_q+\mc{X}_2|_q+H_3(q)\mc{X}_3|_q,
\]
where $\mc{X}_1,\mc{X}_2,\mc{X}_3$ are pointwise
linearly independent smooth vector fields on $\tilde{O}$ (and not only $\tilde{O}_0$)
\[
\mc{X}_1|_q=&\Gamma^1_{(1,2)}\LRD^\perp(E_1)|_q+\nu(A\star E_3)|_q, \\
\mc{X}_2|_q=&\LRD^\perp(E_2)|_q, \\
\mc{X}_3|_q=&\Gamma^1_{(1,2)}\LRD^\perp(E_3)|_q-\nu(A\star E_1)|_q,
\]
while $H_1,H_3$ are smooth functions on $\tilde{O}_0$
defined by
\[
H_1&=\frac{\cos(\theta)}{\lambda},\quad H_3=\frac{\sin(\theta)}{\lambda}.
\]
Notice that $\theta$ and $\lambda$ cannot be extended in a smooth or even $C^1$-way from $\tilde{O}_0$
to $\tilde{O}$, but as we will show, one can extend $H_1,H_3$
in at least $C^1$-way onto $\tilde{O}$.

First, since $\lambda(q)\to \pm\infty$ while $\cos(\theta(q))$, $\sin(\theta(q))$ stay bounded,
it follows that $H_1,H_3$ extend uniquely to $\tilde{O}\cap Q_1$
by declaring $H_1(q)=H_3(q)=0$ for all $q\in\tilde{O}\cap Q_1$.
Of course, these extensions, which we still denote by $H_1,H_3$, are continuous functions on $\tilde{O}$.

Next objective is to show that $H_1,H_3$ are at least $C^1$ on $\tilde{O}$.
For this, let $\mc{X}\in\VF(\tilde{O})$ and decompose it uniquely as
\[
\mc{X}=\sum_{i=1}^3 a_i\LRD(E_i)+\sum_{i=1}^3 b_i\LNSD(E_i)+\sum_{i=1}^3 v_i\nu((\cdot)\star E_i),
\]
with $a_i,b_i,v_i\in\Cinf(\tilde{O})$.

We will need to know the derivatives of $\theta$ and $\lambda$
in all the directions on $\tilde{O}_0$. These have been computed above by using the frame $X_A,E_2,Z_A$
instead of $E_1,E_2,E_3$ except in the direction of $\nu(A\star Z_A)|_q$.
As before, one computes (using that $s_{\phi}\neq 0$ on $\tilde{O}_0$ as usual),
\[
& \nu(A\star Z_A)|_q \theta=0 ,\quad
\nu(A\star Z_A)|_q \phi=1 ,\\
& \nu(A\star Z_A)|_q \lambda=-\Gamma^1_{(1,2)}(x)-\lambda(q)\cot(\phi(q)).
\]
One now easily computes that on $\tilde{O}_0$,
\[
\mc{X}(\theta)=&(-a_1s_{\theta}+a_3c_{\theta})\lambda+(-b_1s_{\theta}\Gamma^1_{(1,2)} +b_3c_{\theta}\Gamma^1_{(1,2)}-v_1c_{\theta}-v_3s_{\theta})\cot(\phi)
+B_1(q) \\
\mc{X}(\lambda)=&(-a_1c_{\theta}-a_3s_{\theta})\lambda^2
+(-b_1c_{\theta}\Gamma^1_{(1,2)}-b_3s_{\theta}\Gamma^1_{(1,2)}+v_1s_{\theta}-v_3c_{\theta})\lambda \cot(\phi) \\
&+a_2\Gamma^1_{(1,2)}\lambda+b_2\cot(\phi)E_2(\Gamma^1_{(1,2)})+B_2(q),
\]
where
\[
B_1(q)&=(a_1+b_1)\Gamma^1_{(3,1)}+(a_2+b_2)\Gamma^2_{(3,1)}+(a_3+b_3)\Gamma^3_{(3,1)}+v_2 \\
B_2(q)&=(-a_1c_{\theta}-a_3s_{\theta})((\Gamma^1_{(1,2)})^2+K)+(-b_1c_{\theta}-b_3s_{\theta})(\Gamma^1_{(1,2)})^2
+(v_1s_{\theta}-v_3c_{\theta})\Gamma^1_{(1,2)}.
\]
Then
\[
\mc{X}(H_1)=&-s_{\theta}\frac{\mc{X}(\theta)}{\lambda}-c_{\theta}\frac{\mc{X}(\lambda)}{\lambda^2} \\
=&a_1+(b_1\Gamma^1_{(1,2)}+v_3)\frac{\cot(\phi)}{\lambda}
-\frac{a_2c_{\theta}\Gamma^1_{(1,2)}}{\lambda}-\frac{b_2c_{\theta}E_2(\Gamma^1_{(1,2)})}{\lambda}\frac{\cot(\phi)}{\lambda}
-\frac{s_{\theta}B_1}{\lambda}-\frac{c_{\theta}B_2}{\lambda^2} \\
\mc{X}(H_3)=&c_{\theta}\frac{\mc{X}(\theta)}{\lambda}-s_{\theta}\frac{\mc{X}(\lambda)}{\lambda^2} \\
=&a_3+(b_3\Gamma^1_{(1,2)}-v_1)\frac{\cot(\phi)}{\lambda}
-\frac{a_2s_{\theta}\Gamma^1_{(1,2)}}{\lambda}-\frac{b_3s_{\theta}E_2(\Gamma^1_{(1,2)})}{\lambda}\frac{\cot(\phi)}{\lambda}
+\frac{c_{\theta}B_1}{\lambda}-\frac{s_{\theta}B_2}{\lambda^2}.
\]
Since $s_{\theta}\lambda=c_{\theta}\Gamma^1_{(1,2)}-\hat{\Gamma}^1_{(1,2)}$, one has
\[
\frac{\cot(\phi)}{\lambda}=\frac{c_{\phi}}{c_{\phi}\Gamma^1_{(1,2)}-\hat{\Gamma}^1_{(1,2)}}
\]
and therefore as $q$ tends to a point $q_1$ of $Q_1^+\cap \tilde{O}$, we have 
$$\lim_{q\to q_1}\frac{\cot(\phi)}{\lambda}=\frac{1}{\Gamma^1_{(1,2)}-\hat{\Gamma}^1_{(1,2)}}.$$
Since $B_1,B_2$ stay bounded as $q$ approaches a point of $Q_1^+\cap \tilde{O}$,
we get for every $q_1=(x_1,\hat{x}_1;A_1)\in Q_1^+\cap \tilde{O}$ that
\[
\lim_{q\to q_1} \mc{X}(H_1)&=a_1(q_1)+\frac{b_1(q_1)\Gamma^1_{(1,2)}(x_1)+v_3(q_1)}{\Gamma^1_{(1,2)}(x_1)-\hat{\Gamma}^1_{(1,2)}(\hat{x}_1)}=:D_{\mc{X}}H_1(q_1), \\
\lim_{q\to q_1} \mc{X}(H_3)&=a_3(q_1)+\frac{b_3(q_1)\Gamma^1_{(1,2)}(x_1)-v_1(q_1)}{\Gamma^1_{(1,2)}(x_1)-\hat{\Gamma}^1_{(1,2)}(\hat{x}_1)}=:D_{\mc{X}}H_3(q_1).
\]
From these,
it is now readily seen that $H_1,H_3$ are differentiable on $\tilde{O}\cap Q_1^+$
with $\mc{X}|_{q_1}(H_1)=D_{\mc{X}}H_1(q_1)$,
$\mc{X}|_{q_1}(H_3)=D_{\mc{X}}H_3(q_1)$
and that $H_1,H_3$ are $C^1$-functions on $\tilde{O}$.

Now that we have extended $H_1,H_3$,
we have that $\mc{N}$ is a well defined $C^1$-vector field on $\tilde{O}$
and since $\mc{D}=\mc{N}^\perp$ w.r.t. $G$ on $\tilde{O}_0$,
it follows that $\mc{D}$ extends in $C^1$-sense on $\tilde{O}$.
Since $q_0\in Q_1^+\backslash S_1$ was arbitrary
and because the case $q_0\in Q_1^-\backslash S_1$ is handled similarly,
we see that $\mc{D}$ can be extended onto the open subset $Q\backslash S_1$ of $Q$
as a (at least) $C^1$-distribution, which is $C^\infty$ on $Q_0$.
Since $\RDist|_{Q\backslash S_1}\subset \mc{D}$
and because $q\in Q\backslash S_1$ implies that $\mc{O}_{\RDist}(q)\subset Q\backslash S_1$
as we have seen,
it follows that for every $q_0\in Q\backslash S_1$
we have $\mc{O}_{\RDist}(q_0)\subset \mc{O}_{\mc{D}}(q_0)$
where the orbit on the right is \emph{a priori} an immersed $C^1$-submanifold of $Q\backslash S_1$.
However, since $\mc{D}$ is involutive and $\dim\mc{D}=8$ on $Q\backslash S_1$,
we get by the $C^1$-version of the Frobenius theorem that
$\dim\mc{O}_{\mc{D}}(q_0)=8$ and hence
\[
\dim\mc{O}_{\RDist}(q_0)\leq\dim \mc{O}_{\mc{D}}(q_0)=8,
\]
for every $q_0\in Q\backslash S_1$.

We will now investigate when the equality holds here.
Define
\[
M_0&=\{x\in M\ |\ K_2(x)\neq K\} \\
M_1&=\{x\in M\ |\ \exists\ \textrm{open}\ V\ni x\ \textrm{s.t.}\ K_2(x')=K\ \forall x'\in V\} \\
\hat{M}_0&=\{\hat{x}\in \hat{M}\ |\ \hat{K}_2(\hat{x})\neq K\} \\
\hat{M}_1&=\{\hat{x}\in \hat{M}\ |\ \exists\ \textrm{open}\ \hat{V}\ni \hat{x}\ \textrm{s.t.}\ K_2(\hat{x}')=K\ \forall \hat{x}'\in \hat{V}\}
\]
and notice that $M_0\cup M_1$
(resp. $\hat{M}_0\cup \hat{M}_1$) is a dense subset of $M$
(resp. $\hat{M}$).
Here we also fix the choice of $q_0=(x_0,\hat{x}_0;A_0)\in Q\backslash S_1$
and define $M^\circ=\pi_{Q,M}(\mc{O}_{\RDist}(q_0))$, $\hat{M}^\circ=\pi_{Q,\hat{M}}(\mc{O}_{\RDist}(q_0))$
as in the statement.
Write also $Q^\circ:=\pi_Q^{-1}(M^\circ\times \hat{M}^\circ)$ and notice that $\mc{O}_{\RDist}(q_0)\subset Q^\circ$.

We define on $Q$ two 2-dimensional distributions $D$ and $\hat{D}$.
For every $q_1=(x_1,\hat{x}_1;A_1)\in Q$, take orthonormal frames $E_1,E_2,E_3$,
$\hat{E}_1,\hat{E}_2,\hat{E}_3$ of $M,\hat{M}$ defined
on open neighbourhoods $U,\hat{U}$ of $x_1,\hat{x}_1$ with $E_2=\pa{r}$, $\hat{E}_2=\pa{r}$.
Then for $q\in \pi_Q^{-1}(U\times \hat{U})\cap Q$,
the 2-dimensional plane $D|_q$ is spanned by
\[
K_1|_q&=\LNSD(A^{\ol{T}}\hat{E}_1)|_q-\hat{\Gamma}^1_{(1,2)}(x)\nu((\hat{\star} \hat{E}_3)A)|_q \\
K_3|_q&=\LNSD(A^{\ol{T}}\hat{E}_3)|_q+\hat{\Gamma}^1_{(1,2)}(x)\nu((\hat{\star} \hat{E}_1)A)|_q,
\]
and $\hat{D}|_q$ is spanned by
\[
\hat{K}_1|_q&=\LNSD(AE_1)|_q+\Gamma^1_{(1,2)}(x)\nu(A\star E_3)|_q \\
\hat{K}_3|_q&=\LNSD(AE_3)|_q-\Gamma^1_{(1,2)}(x)\nu(A\star E_1)|_q.
\]
Obviously different choices of frames $E_i$, $\hat{E}_i$, $i=1,2,3$,
give $K_1,K_3,\hat{K}_1,\hat{K}_3$ that span the same planes $D,\hat{D}$,
since we have fixed the choice of $E_2=\pa{r}$, $\hat{E}_2=\pa{r}$.

Exactly as in proof of Proposition \ref{pr:mbeta-3},
one can show that for every $q_1=((r_1,y_1),(\hat{r}_1,\hat{y}_1);A_1)\in Q$
and smooth paths $\gamma:[0,1]\to N$, $\hat{\gamma}:[0,1]\to\hat{N}$
with $\gamma(0)=y_1$, $\hat{\gamma}(0)=\hat{y}_1$
there are unique smooth paths $\Gamma,\hat{\Gamma}:[0,1]\to Q$
such that for all $t\in [0,1]$,
\[
& \dot{\Gamma}(t)\in D|_{\Gamma(t)},\quad \Gamma(0)=q_1,\quad (\pi_{Q,M}\circ\Gamma)(t)=(r_1,\gamma(t)) \\
& \dot{\hat{\Gamma}}(t)\in \hat{D}|_{\hat{\Gamma}(t)},\quad \hat{\Gamma}(0)=q_1,\quad
(\pi_{Q,\hat{M}}\circ\hat{\Gamma})(t)=(\hat{r}_1,\hat{\gamma}(t)).
\]
Notice that since $(\pi_{Q,\hat{M}})_*D=0$ (resp. $(\pi_{Q,M})_*\hat{D}=0$),
one has $\pi_{Q,\hat{M}}(\Gamma(t))=\hat{x}_1$ (resp. $\pi_{Q,M}(\hat{\Gamma}(t))=x_1$) for all $t\in [0,1]$.
We write these as $\Gamma=\Gamma(\gamma,q_1)$, $\hat{\Gamma}=\hat{\Gamma}(\hat{\gamma},q_1)$.

If $E_2=\pa{r}$, $\hat{E}_2=\pa{r}$, then
by exactly the same arguments as in the proof of Proposition \ref{pr:mbeta-3} we have
\[
& \nu(A\star E_2)|_q\in T|_q\mc{O}_{\RDist}(q),\quad \forall q\in Q_0\cap \pi_Q^{-1}(M_0\times \hat{M}) \\
& \nu((\hat{\star} \hat{E}_2)A)|_q\in T|_q\mc{O}_{\RDist}(q),\quad \forall q\in Q_0\cap \pi_Q^{-1}(M\times \hat{M}_0).
\]
We will see that here one may replace $Q_0$ by $Q\backslash S_1$.

Take frames $E_i$, $\hat{E}_i$, $i=1,2,3$, as above when defining $D,\hat{D}$
for some $q_1\in Q_1\backslash S_1$. We may assume here without
loss of generality that $q_1\in Q_1^+\backslash S_1$
since the case $q_1\in Q_1^-\backslash S_1$ can be dealt with in a similar way.
 
If $h_1,h_2:\pi_Q^{-1}(U\times\hat{U})\to\R$
are defined as $h_1(q)=\hat{g}(AE_1,\hat{E}_2)$, $h_2(q)=\hat{g}(AE_3,\hat{E}_2)$,
we have $Q_1\cap \pi_Q^{-1}(U\times\hat{U})=(h_1,h_2)^{-1}(0)$
and $(h_1,h_2):\pi_Q^{-1}(U\times\hat{U})\to\R^2$ is a regular map
at the points of $Q_1$ (see e.g. Remark \ref{re:wp-1} or the proof of Proposition \ref{pr:mbeta-1}).

Since $q_1\in Q_1^+\backslash S_1$, then $\LRD(E_1)|_{q_1} h_1=\Gamma^1_{(1,2)}(x_1)-\hat{\Gamma}^1_{(1,2)}(\hat{x}_1)\neq 0$
and $\LRD(E_3)|_{q_1} h_2=\Gamma^1_{(1,2)}(x_1)-\hat{\Gamma}^1_{(1,2)}(\hat{x}_1)\neq 0$,
which shows that $\mc{O}_{\RDist}(q_1)$ intersects $Q_1^+$
transversally at $q_1$
(hence at every point $q\in \mc{O}_{\RDist}(q_1)$), by dimensional reasons (because $\dim Q_1=7$, $\dim Q=9$).
From this we may conclude that $\mc{O}_{\RDist}(q_1)\cap Q_1$ is a smooth closed submanifold
of $\mc{O}_{\RDist}(q_1)$
and that there is a sequence $q_n'=(x_n',\hat{x}_n';A_n')\in \mc{O}_{\RDist}(q_1)\cap Q_0$
such that $q_n'\to q_1$.

Now if $q_1\in \pi_Q^{-1}(M_0\times \hat{M})\cap Q_1\backslash S_1$,
then we know that for $n$ large enough, $q_n'\in  \pi_Q^{-1}(M_0\times \hat{M})\cap Q_0$
and hence $\nu(A\star E_2)|_{q_n'}\in T|_{q_n'}\mc{O}_{\RDist}(q_n')=T|_{q_n'}\mc{O}_{\RDist}(q_1)$.
Taking the limit implies that $\nu(A\star E_2)|_{q_1}\in T|_{q_1}\mc{O}_{\RDist}(q_1)$.
Similarly, if $q_1\in \pi_Q^{-1}(M\times \hat{M}_0)\cap Q_1\backslash S_1$,
one has $\nu((\hat{\star} \hat{E}_2)A)|_{q_1}\in T|_{q_1}\mc{O}_{\RDist}(q_1)$.

Hence we have that if $E_2=\pa{r}$, $\hat{E}_2=\pa{r}$, then
\[
& \nu(A\star E_2)|_q\in T|_q\mc{O}_{\RDist}(q),\quad \forall q\in (Q\backslash S_1)\cap \pi_Q^{-1}(M_0\times \hat{M}) \\
& \nu((\hat{\star} \hat{E}_2)A)|_q\in T|_q\mc{O}_{\RDist}(q),\quad \forall q\in (Q\backslash S_1)\cap \pi_Q^{-1}(M\times \hat{M}_0).
\]

For every $q\in (Q\backslash S_1)\cap \pi_Q^{-1}(M_0\times \hat{M})$, which is
an open subset of $Q$, one has $\nu(A\star E_2)|_q\in T|_q\mc{O}_{\RDist}(q)$
with $E_2=\pa{r}$ and hence by Proposition \ref{pr:3D-2}, case (i), it follows that
\[
& L_1|_q=\LNSD(E_1)|_q-\Gamma^1_{(1,2)}(x)\nu(A\star E_3)|_q \\
& L_3|_q=\LNSD(E_3)|_q+\Gamma^1_{(1,2)}(x)\nu(A\star E_1)|_q,
\]
are tangent to $\mc{O}_{\RDist}(q)$,
where $E_1,E_2=\pa{r},E_3$ is an orthonormal frame in an open neighbourhood of $x_1$.
But because $\hat{K}_1|_q=\LRD(E_1)|_q-L_1|_q$,
$\hat{K}_3|_q=\LRD(E_3)|_q-L_3|_q$,
we get that
\[
\hat{D}|_q\subset T|_{q}\mc{O}_{\RDist}(q),\quad \forall q\in (Q\backslash S_1)\cap \pi_Q^{-1}(M_0\times \hat{M}).
\]
Moreover, if $q=(x,(\hat{r},\hat{y});A)\in (Q\backslash S_1)\cap \pi_Q^{-1}(M_0\times \hat{M})$
and if $\hat{\gamma}:[0,1]\to \hat{N}$ is any curve with $\hat{\gamma}(0)=\hat{y}$,
then one shows with exactly the same argument as in the proof of Proposition \ref{pr:mbeta-3} that
\[
\hat{\Gamma}(\hat{\gamma},q)(t)\in \mc{O}_{\RDist}(q)\cap \pi_Q^{-1}(M_0\times \hat{M}),\quad \forall t\in [0,1].
\]
In particular,
\[
\exists q=(x,(\hat{r},\hat{y});A)\in (Q\backslash S_1)\cap \pi_Q^{-1}(M_0\times \hat{M})
\quad\Longrightarrow\quad
\{x\}\times (\{\hat{r}\}\times \hat{N})\subset \pi_Q(\mc{O}_{\RDist}(q)).
\]

A similar argument shows that
\[
D|_q\subset T|_{q}\mc{O}_{\RDist}(q),\quad \forall q\in (Q\backslash S_1)\cap \pi_Q^{-1}(M\times \hat{M}_0)
\]
and that for all $q=((r,y),\hat{x};A)\in (Q\backslash S_1)\cap \pi_Q^{-1}(M\times \hat{M}_0)$
and $\gamma:[0,1]\to N$ with $\gamma(0)=y$,
\[
\Gamma(\gamma,q)(t)\in \mc{O}_{\RDist}(q)\cap \pi_Q^{-1}(M\times \hat{M}_0),\quad \forall t\in [0,1].
\]
In particular,
\[
\exists q=((r,y),\hat{x};A)\in (Q\backslash S_1)\cap \pi_Q^{-1}(M\times \hat{M}_0)
\quad\Longrightarrow\quad
(\{r\}\times N)\times \{\hat{x}\}\subset \pi_Q(\mc{O}_{\RDist}(q)).
\]

Everything so far has been very much the same as in the proof of Proposition \ref{pr:mbeta-3}
and continues to be so, with few minor changes (notably, here $\dim D=\dim\hat{D}=2$ instead of $3$).

Suppose that $(M_1\times \hat{M}_0)\cap \pi_{Q}(\mc{O}_{\RDist}(q_0))\neq\emptyset$.
Take $q_1=(x_1,\hat{x}_1;A_1)\in \pi_Q^{-1}(M_1\times \hat{M}_0)\cap \mc{O}_{\RDist}(q_0)$,
with $x_1=(r_1,y_1)$.
If $\sigma(y)$ is the unique sectional curvature of $N$ at $y$, we have
\[
K_2(r_1,y_1)=\frac{\sigma(y_1)-(f'(r_1))^2}{f(r_1)^2}=K.
\]
We go from here case by case.

\begin{itemize}
\item[(I)] 
Suppose $N$ does not have a constant curvature.
Then there is a $y_2\in N$ with $\sigma(y_2)\neq \sigma(y_1)$
and hence
\[
K_2(r_1,y_2)=\frac{\sigma(y_2)-(f'(r_1))^2}{f(r_1)^2}\neq K,
\]
i.e. $(r_1,y_2)\in M_0$.

Since $q_1\in \mc{O}_{\RDist}(q_0)\subset Q\backslash S_1$, 
we have by the above that
\[
((r_1,y_2),\hat{x}_1)\in (\{r_1\}\times N)\times \{\hat{x}_1\}\subset \pi_{Q}(\mc{O}_{\RDist}(q_1))=\pi_{Q}(\mc{O}_{\RDist}(q_0))
\]
and since $((r_1,y_2),\hat{x}_1)\in M_0\times\hat{M}_0$, we get that
which implies that $(M_0\times \hat{M}_0)\cap \pi_{Q}(\mc{O}_{\RDist}(q_0))\neq\emptyset$.

\item[(II)] Suppose that $(N,h)$ has constant curvature $C$ i.e. $\sigma(y)=C$ for all $y\in N$.
We write $K_2(r,y)=K_2(r)$ on $M$ since its value only depends on $r\in I$
and notice that for all $r\in I$,
\[
\di{K_2}{r}=-2\frac{f'(r)}{f(r)}(K_2(r)-K).
\]
But $K_2(r_1)=K$, so by the uniqueness of solutions of ODEs, we get $K_2(r)=K$ for all $r\in I$
and hence $(M,g)$ has constant curvature $K$.
\end{itemize}

Of course, regarding case (II), it is clear that
if $(M,g)$ has constant curvature $K$, then $(N,h)$ has a constant curvature.

Hence we have proved that
if $(M,g)$ does not have a constant curvature and if
$(M_1\times \hat{M}_0)\cap \pi_{Q}(\mc{O}_{\RDist}(q_0))\neq\emptyset$,
then also $(M_0\times \hat{M}_0)\cap \pi_Q(\mc{O}_{\RDist}(q_0))\neq\emptyset$.

The argument being symmetric in $(M,g)$, $(\hat{M},\hat{g})$,
we also have that
if $(\hat{M},\hat{g})$ does not have a constant curvature and if
$(M_0\times \hat{M}_1)\cap \pi_{Q}(\mc{O}_{\RDist}(q_0))\neq\emptyset$,
then also $(M_0\times \hat{M}_0)\cap \pi_Q(\mc{O}_{\RDist}(q_0))\neq\emptyset$.

Notice that $(M^\circ,g)$ and $(\hat{M}^\circ,\hat{g})$ cannot both have constant curvature,
since this violates the assumption that $\mc{O}_{\RDist}(q_0)$ is not an integral manifold of $\RDist$
(see Corollary \ref{cor:weak_ambrose} and Remark \ref{re:weak_ambrose}).
We can now finish the proof by considering, again,
different cases.

\begin{itemize}
\item[a)] Assume that $(\hat{M}^\circ,\hat{g})$ has a constant curvature, which must then be $K$.
We have $\hat{M}_0\cap \hat{M}^\circ=\emptyset$.
If $E_2=\pa{r}$, then
Hence, $\widetilde{\Rol}_q(\star X)=0$ for all $q\in Q^\circ=\pi_Q^{-1}(M^\circ\times\hat{M}^\circ)$, $X\in E_2^\perp$
while $\widetilde{\Rol}_q(\star E_2)=(-K_2(x)+K)\star E_2$.

At $q_1=(x_1,\hat{x}_1;A_1)\in Q^\circ$,
take an open neighbourhood $U$ of $x_1$ and an ortonormal basis $E_1,E_2,E_3$ with $E_2=\pa{r}$
and let $D_1$ be a distribution on $\pi_{Q,M}^{-1}(U)$ spanned by
\[
\LRD(E_1), \LRD(E_2), \LRD(E_3), \nu((\cdot)\star E_2), L_1, L_3
\]
where $L_1,L_3$ are as in Proposition \ref{pr:3D-2}.
Obviously, one defines in this way a 6-dimensional smooth distribution $D_1$ on the whole $Q^\circ$
and the above from of $\widetilde{\Rol}_q$, $q\in Q^\circ$, along with Proposition \ref{pr:3D-2}, case (ii), reveal that it is involutive (recall that $\Gamma^1_{(2,3)}=0$ there).
Clearly, $\RDist\subset D_1$ on $Q^\circ$ and since $\mc{O}_{\RDist}(q_0)\subset Q^\circ$,
we have $\mc{O}_{\RDist}(q_0)\subset \mc{O}_{D_1}(q_0)$
and hence $\dim \mc{O}_{\RDist}(q_0)\leq 6$.

Because $(M^\circ,g)$ does not have constant curvature (as noticed previously),
we have $M_0\cap M^\circ\neq\emptyset$
and thus $O:=\mc{O}_{\RDist}(q_0)\cap \pi_{Q,M}^{-1}(M_0)$ is a non-empty
open subset of $\mc{O}_{\RDist}(q_0)$.
For every $q=(x,\hat{x};A)\in O$, one has $\widetilde{\Rol}_q(\star E_2)=(-K_2(x)+K)\star E_2\neq 0$
and hence that $\nu(A\star E_2)|_q\in T|_q\mc{O}_{\RDist}(q_0)$.
Therefore, Proposition \ref{pr:3D-2}, case (i), implies that $D_1|_{O}$ is
tangent to $\mc{O}_{\RDist}(q_0)$.
This gives $\dim \mc{O}_{\RDist}(q_0)\geq 6$
and hence $\dim \mc{O}_{\RDist}(q_0)=6$.

\item[b)] If $(M^\circ,g)$ has constant curvature,
then the argument of case a) with the roles of $(M,g)$, $(\hat{M},\hat{g})$ interchanged,
shows that $\dim \mc{O}_{\RDist}(q_0)=6$.
\end{itemize}

Hence we have proven (ii). For the rest of the cases,
we may assume that neither $(M^\circ,g)$ nor $(\hat{M}^\circ,\hat{g})$
has constant curvature i.e.
$M^\circ\cap M_0\neq\emptyset$, $\hat{M}^\circ\cap \hat{M}_0\neq\emptyset$.

\begin{itemize}
\item[c)] Suppose $(M_0\times \hat{M}_0)\cap \pi_{Q}(\mc{O}_{\RDist}(q_0))\neq\emptyset$
and let $q_1=(x_1,\hat{x}_1;A_1)\in \pi_Q^{-1}(M_0\times \hat{M}_0)\cap \mc{O}_{\RDist}(q_0)$.
We already know that $T|_{q_1}\mc{O}_{\RDist}(q_0)$ contains vectors
\[
& \LRD(E_1)|_{q_1}, \LRD(E_2)|_{q_1}, \LRD(E_3)|_{q_1}, \\
& \nu(A\star E_2)|_{q_1}, \nu((\hat{\star}\hat{E}_2)A)|_{q_1} \\
& L_1|_{q_1}, L_3|_{q_1},\hat{L}_1|_{q_1},\hat{L}_3|_{q_1},
\]
where
\[
& \hat{L}_1|_{q_1}=\LNSD(\hat{E}_1)|_{q_1}+\hat{\Gamma}^1_{(1,2)}(\hat{x}_1)\nu((\hat{\star} \hat{E}_3)A_1)|_{q_1} \\
& \hat{L}_3|_{q_1}=\LNSD(\hat{E}_3)|_{q_1}-\hat{\Gamma}^1_{(1,2)}(\hat{x}_1)\nu((\hat{\star}\hat{E}_1)A_1)|_{q_1}.
\]
Moreover, these span an 8-dimensional subspace of $T|_{q_1}\mc{O}_{\RDist}(q_0)$,
at least if $q_1\in Q_0$.

Indeed, if $q_1\in Q_0$, one introduces $X_{A_1},Z_{A_1},\hat{X}_{A_1},\hat{Z}_{A_1}$ and an angles $\phi,\theta,\hat{\theta}$ as before, we have
$\sin(\phi(q_1))\neq 0$ and
\[
\nu((\hat{\star}\hat{E}_2)A_1)|_{q_1}
=&\nu(A_1\star (A_1^{\ol{T}}\hat{E}_2))|_{q_1} \\
=&\sin(\phi(q_1))\nu(A_1\star X_{A_1})|_{q_1}+\cos(\phi(q_1))\nu(A_1\star E_2)|_{q_1}, \\
c_{\theta} L_1|_{q_1}+s_{\theta} L_3|_{q_1}=&\LNSD(X_{A_1})|_{q_1}-\Gamma^1_{(1,2)}(x_1)\nu(A_1\star Z_{A_1})|_{q_1} \\
-s_{\theta} L_1|_{q_1}+c_{\theta} L_3|_{q_1}=&\LNSD(Z_{A_1})|_{q_1}+\Gamma^1_{(1,2)}(x_1)\nu(A_1\star X_{A_1})|_{q_1} \\
c_{\hat{\theta}} \hat{L}_1|_{q_1}+s_{\hat{\theta}} \hat{L}_3|_{q_1}=&\LNSD(\hat{X}_{A_1})|_{q_1}+\hat{\Gamma}^1_{(1,2)}(x_1)\nu(A_1\star Z_{A_1})|_{q_1} \\
=&c_{\phi}\LNSD(A_1X_{A_1})|_{q_1}-s_{\phi}\LNSD(A_1E_2)|_{q_1}+\hat{\Gamma}^1_{(1,2)}(x_1)\nu(A_1\star Z_{A_1})|_{q_1} \\
-s_{\hat{\theta}} \hat{L}_1|_{q_1}+c_{\hat{\theta}} \hat{L}_3|_{q_1}=&\LNSD(A_1Z_{A_1})|_{q_1}-\hat{\Gamma}^1_{(1,2)}(x_1)\nu(A_1\star (A^{\ol{T}}\hat{X}_{A_1}))|_{q_1} \\
=&\LNSD(A_1Z_{A_1})|_{q_1}-\hat{\Gamma}^1_{(1,2)}(x_1)\big(c_{\phi}\nu(A_1\star X_{A_1})|_{q_1}-s_{\phi}\nu(A_1\star E_2)|_{q_1}\big).
\]
On the other hand, if $q_1\in Q_1$, then since $Q_1$ is transversal to $\mc{O}_{\RDist}(q_0)$
at $q_1$, we can replace $q_1$ by a nearby $q_1'\in \pi_Q^{-1}(M_0\times \hat{M}_0)\cap \mc{O}_{\RDist}(q_0)\cap Q_0$
and the above holds at $q_1'$.

Therefore $\dim\mc{O}_{\RDist}(q_0)\geq 8$
and since we have also shown that $\dim\mc{O}_{\RDist}(q_0)\leq 8$,
we have the equality. 

\item[d)] Since $M^\circ\cap M_0\neq\emptyset$,
there is a $q_1=(x_1,\hat{x}_1;A_1)\in \mc{O}_{\RDist}(q_0)$ such that $x_1\in M_0$.
If $\hat{x}_1\in \hat{M}_0$, one has that $(M_0\times \hat{M}_0)\cap \pi_{Q}(\mc{O}_{\RDist}(q_0))\neq\emptyset$
and hence case c) implies that $\dim\mc{O}_{\RDist}(q_0)\leq 8$.

But if $\hat{x}_1\notin \hat{M}_0$, then $\hat{x}_1\in \ol{\hat{M}_1}$.
Therefore, we may find a sequence $q_n'=(x_n',\hat{x}_n';A_n')\in \mc{O}_{\RDist}(q_0)$
such that $q_n'\to q_1$ and $\hat{x}_n'\in \hat{M}_1$.
So for $n$ large enough, we have $(x_n',\hat{x}_n')\in (M_0\times \hat{M}_1)\cap \pi_Q(\mc{O}_{\RDist}(q_0))$.

Thus $(\hat{M},\hat{g})$ does not have constant curvature
and $(M_0\times \hat{M}_1)\cap \pi_Q(\mc{O}_{\RDist}(q_0))\neq\emptyset$
which we have shown to imply that $(M_0\times\hat{M}_0)\cap \pi_{Q}(\mc{O}_{\RDist}(q_0))\neq\emptyset$
from which the above case c) implies that $\dim\mc{O}_{\RDist}(q_0)\leq 8$.
\end{itemize}

The cases c) and d) above give (iii) and therefore the proof is complete.
\end{proof}

\begin{remark} It is not difficult to see that Proposition \ref{pr:wp-1} generalizes to higher dimension as follows. Keeping the same notations as before, let $(M,g)=(I,s_1)\times_{f}  (N,h)$ and $(\hat{M},\hat{g})=(\hat{I},s_1)\times_{\hat{f}}(\hat{N},\hat{h})$, $I,\hat{I}\subset\R$,
be warped products where $(N,h)$ and $(\hat{N},\hat{h})$ are now connected, oriented $(n-1)$-dimensional Riemannian manifolds. As before, let $q_0=(x_0,\hat{x}_0;A_0)\in Q$
be such that if we write $x_0=(r_0,y_0)$, $\hat{x}_0=(\hat{r}_0,\hat{y}_0)$,
then \eqref{eq:A_0_S_to_hatS} and \eqref{eq:Df_over_f} hold true.
Then, the exact argument of Proposition \ref{pr:wp-1} yields that the orbit $\mc{O}_{\RDist}(q_0)$ has dimension at most equal to $n(n+1)/2$.

Note that one can have equality, if the $(n-1)$-dimensional manifolds $(N,h)$
and $(\hat{N},\hat{h})$ are such that that the corresponding $\widetilde{\Rol}_{q_0'}$ operator (in $(n-1)$-dimensional setting)
is invertible at $q_0'=(y_0,\hat{y}_0;A_0')\in Q(N,\hat{N})$, where
$A_0':\pa{r}\big|_{x_0}^\perp\to \pa{r}\big|_{\hat{x}_0}^\perp$ is the restriction of $A_0$ 
and if we also assume that $f(r_0)=1$, $\hat{f}(r_0)=1$, an assumption that
can always be satisfied after rescaling the metrics of $(N,h)$ and $(\hat{N},\hat{h})$.
\end{remark}

\begin{remark}
Notice that in the particular situation where the warped products are in fact Riemannian products $(M,g)=(I\times N,s_1\oplus h)$,
$(\hat{M},\hat{g})=(\hat{I}\times \hat{N},s_1\oplus \hat{h})$, i.e. where $f=\hat{f}=1$,
then the fact that for every $q_0=(x_0,\hat{x}_0;A_0)\in Q$
the orbit $\mc{O}_{\RDist}(q_0)$ is at most of dimension 8,
can be deduced more easily by using Theorem \ref{th:NS_orbit_V}.

Indeed, Theorem \ref{th:NS_orbit_V}
tells us that $\pi_{\mc{O}_{\NSDist}(q_0)}^{-1}(x_0,\hat{x}_0)=\hat{H}|_{\hat{x}_0}A_0H|_{x_0}$
Since $M$ and $\hat{M}$ are Riemannian
products and because the holonomy group of $(\R,s_1)$ is trivial,
one has $H|_{x_0}=H^{h}|_{y_0}$
and $\hat{H}|_{\hat{x}_0}=H^{\hat{h}}|_{\hat{y}_0}$,
isomorphically, where $x_0=(r_0,y_0)$, $\hat{x}_0=(\hat{r}_0,\hat{y}_0)$.
But $\dim H^{h}|_{y_0}\leq \dim \SO(2)=1$
and $\dim H^{\hat{h}}|_{\hat{y}_0}\leq 1$,
so
\[
\dim\pi_{\mc{O}_{\NSDist}(q_0)}^{-1}(x_0,\hat{x}_0)=\dim(\hat{H}|_{\hat{x}_0}A_0H|_{x_0})\leq
\dim \hat{H}|_{\hat{x}_0}+\dim H|_{x_0}\leq 2
\]
which implies that $\dim \mc{O}_{\NSDist}(q_0)\leq 3+3+2=8$.
Because $\dim\mc{O}_{\RDist}(q_0)\subset \dim\mc{O}_{\NSDist}(q_0)$,
we also have $\dim\mc{O}_{\RDist}(q_0)\leq 8$.
\end{remark}

%%%%%%%%%%%%%%%%%%%%%%%%%%
\section{Rolling of Spaces of Different Dimensions}\label{diff-dim}
%%%%%%%%%%%%%%%%%%%%%%%%%%

\subsection{Definitions of the State Space and the Rolling Distributions}
\begin{definition}
Let $(M,g)$, $(\hat{M},\hat{g})$ be Riemannian manifolds
of dimensions $n=\dim(M)\geq 2$ and $\hat{n}=\dim(\hat{M})\geq 2$,
not necessarily equal. Then one defines:
\begin{itemize}
\item[(i)] if $n\leq \hat{n}$,
\[
Q(M,\hat{M}):=\{A\in T^*M\otimes T\hat{M}\ |\ \hat{g}(AX,AY)=g(X,Y),\ X,Y\in T|_x M,\ x\in M\},
\]
the set of isometric infinitesimal immersions.
This defines a smooth manifold of $T^*M\otimes T\hat{M}$ of dimension
\[
\dim(Q):=n+\hat{n}+n(\hat{n}-n)+\frac{n(n-1)}{2}
=n+\hat{n}+n\hat{n}-\frac{n(n+1)}{2}.
\]

\item[(ii)] If $n\geq \hat{n}$,
\[
Q(M,\hat{M})=\{A\in T^*M\otimes T\hat{M}\ |\ & \hat{g}(AX,AY)=g(X,Y),\ X,Y\in 
(\ker A)^\perp,\\ & x\in M, \ 
 A\ \textrm{is onto a tangent space of $\hat{M}$}\},
\]
where $L^\perp$ is the orthogonal subspace of $L\subset T|_x M$ w.r.t. $g$.
This defines a smooth submanifold of $T^*M\otimes T\hat{M}$ of dimension
\[
\dim(Q)=n+\hat{n}+\hat{n}(n-\hat{n})+\frac{\hat{n}(\hat{n}-1)}{2}
=       n+\hat{n}+n\hat{n}-\frac{\hat{n}(\hat{n}+1)}{2}.
\]
\end{itemize}

If $n=\hat{n}$ and $M,\hat{M}$ are oriented we also demand in (i) and (ii) that the elements of 
$Q$ to preserve the orientations. Hence we recover the definition used before.
\end{definition}

One defines distributions $\NSDist$, $\RDist$ on $Q$ and the lifts $\LNSD$, $\LRD$
as before.
In both cases the dimension of $\NSDist$ is $n+\hat{n}$
and that of $\RDist$ is $n$.
Notice that by the above definition the dimension of $Q(M,\hat{M})$ is the same
as that of $Q(\hat{M},M)$. These manifolds are actually diffeomorphic
as the next proposition shows.

Before proceeding, we introduce some notations.
Given $(M,g)$ and $(\hat{M},\hat{g})$ as before, we write $Q=Q(M,\hat{M})$
and $\hat{Q}=Q(\hat{M},M)$.
We write $\NSDist$, $\RDist$, $\LNSD$ and $\LRD$
on $Q$ as before but on $\hat{Q}$ we write the corresponding objects
as $\widehat{\NSDist}$, $\widehat{\RDist}$, $\widehat{\LNSD}$ and $\widehat{\LRD}$.
Thus $\dim \RDist=n$ but $\dim\widehat{\RDist}=\hat{n}$.
As before, for $q=(x,\hat{x};A)\in Q$ we write $A^{\ol{T}}:T|_{\hat{x}} \hat{M}\to T|_{x} M$
the $(g,\hat{g})$-transpose of $A$ i.e., $g(X,A^{\ol{T}}\hat{Y})=\hat{g}(AX,\hat{Y})$
for all $X\in T|_x M$, $\hat{Y}\in T|_{\hat{x}}\hat{M}$.

\begin{proposition}\label{pr:diffeo_Q_hatQ}
For every $(x,\hat{x};A)\in Q$, one has $(\hat{x},x;A^{\ol{T}})\in \hat{Q}$
and the application
\[
\ol{T}:Q\to\hat{Q};\quad \ol{T}(x,\hat{x};A)=(\hat{x},x;A^{\ol{T}}),
\]
is a diffeomorphism. Moreover, this diffeomorphism $\ol{T}$ is an isometry of fiber bundles
$\pi_{Q}\to \pi_{\widehat{Q}}$
that preserves the no-spinning distributions on these manifolds i.e.,
\[
\ol{T}_*\NSDist=\widehat{\NSDist}.
\]
\end{proposition}

\begin{proof}
Suppose w.l.o.g. that $n\leq \hat{n}$.
It is clear that $A^{\ol{T}}A=\id_{T|_x M}$ for every $(x,\hat{x};A)\in Q$
and $\ker (A^{\ol{T}})=\IM(A)^{\perp}$
and thus if $\hat{X},\hat{Y}\in \ker (A^{\ol{T}})^{\perp}=\IM(A)$, one gets
\[
g(A^{\ol{T}}\hat{X},A^{\ol{T}}\hat{Y})
=g(A^{\ol{T}}AX,A^{\ol{T}}AY)=g(X,Y)=\hat{g}(AX,AY)=\hat{g}(\hat{X},\hat{Y}),
\]
where $X,Y\in T|_x M$ were such that $AX=\hat{X}$, $AY=\hat{Y}$.
This proves that $\ol{T}(x,\hat{x};A)$ is actually an element of $\hat{Q}$.

Let then $\hat{q}=(\hat{x},x;B)\in \hat{Q}$
and define
\[
\ol{S}(\hat{q})=(x,\hat{x};B^{\ol{T}})\in T^*M\otimes T\hat{M}.
\]
Since $\IM(B^{\ol{T}})=\ker(B)^\perp$, we have for $X,Y\in T|_x M$,
\[
g(B^{\ol{T}}X,B^{\ol{T}}Y)=\hat{g}(BB^{\ol{T}}X,BB^{\ol{T}}Y)=\hat{g}(X,Y),
\]
directly from the definition of $\hat{Q}$ and
since $BB^{\ol{T}}=\id_{T|_{x} M}$ (since $n\leq \hat{n}$).
This shows that $\ol{S}:\hat{Q}\to Q$.

Moreover, one clearly has that 
$\ol{T}$ and $\ol{S}$ are maps inverse to each other.
They are obviously smooth, hence $Q$ and $\hat{Q}$ are diffeomorphic.
Also, $\ol{T}$ is actually a bundle isomorphism $\pi_{Q}\to \pi_{\hat{Q}}$
whose inverse as a bundle isomorphism is $\ol{S}$.

Finally, observe that if $\gamma$, $\hat{\gamma}$ are smooth paths in $M$, $\hat{M}$
starting at $x_0,\hat{x}_0$, respectively, at $t=0$,
and if $q_0=(x_0,\hat{x}_0;A_0)\in Q(M,\hat{M})$
then
\[
& (P_0^t(\hat{\gamma})\circ A_0\circ P_t^0(\gamma))^{\ol{T}}
=P_0^t(\gamma)\circ A_0^{\ol{T}}\circ P_t^0(\hat{\gamma}),
\]
so
\[
\ol{T}\big(\gamma(t),\hat{\gamma}(t);P_0^t(\hat{\gamma})\circ A_0\circ P^t_0(\gamma)\big)
=\big(\hat{\gamma}(t),\gamma(t);P_0^t(\gamma)\circ \ol{T}(x_0,\hat{x}_0;A_0)\circ P_t^0(\hat{\gamma})\big),
\]
which immediately shows, by differentiating $\dif{t}\big|_0$ and using the definition of $\LNSD$, that
\[
\ol{T}_*|_{q_0}\LNSD(X, \hat{X})\Big|_{q_0}
=\widehat{\LNSD}(\hat{X}, X)\Big|_{\ol{T}(q_0)}
\]
where $X=\dot{\gamma}(0)$, $\hat{X}=\dot{\hat{\gamma}}(0)$.
This proves in particular that $\ol{T}_*$ maps $\NSDist$
isomorphically onto $\widehat{\NSDist}$. This completes the proof.

\end{proof}

\begin{corollary}
In the case $n\leq \hat{n}$, one has for $q_0=(x_0,\hat{x}_0;A_0)\in Q(M,\hat{M})$ and $X\in T|_x M$,
\[
\ol{T}_*|_{q_0} \LRD(X)|_{q_0}=\widehat{\LRD}(A_0 X)|_{\Phi(q_0)}.
\] 
In particular, $\ol{T}_*\RDist\subset \widehat{\RDist}$.
\end{corollary}

\begin{proof}
Indeed, for $X\in T|_{x_0} M$ and $q_0=(x_0,\hat{x}_0;A_0)$ one has
\[
\ol{T}_*|_{q_0} \LRD(X)|_{q_0}
=\ol{T}_*|_{q_0} \LNSD(X, A_0 X)\Big|_{q_0}
=\widehat{\LNSD}(A_0 X, X)\Big|_{\ol{T}(q_0)}
=\widehat{\LRD}(A_0 X)|_{\ol{T}(q_0)},
\]
since $X=(A_0^{\ol{T}})(A_0 X)=\ol{T}(q_0)(A_0X)$.
Hence $\ol{T}$ maps $\RDist$ of $Q(M,\hat{M})$ into $\widehat{\RDist}$ of $Q(\hat{M},M)$.

\end{proof}

\begin{remark}
Recall that the distribution $\RDist$ on $Q(M,\hat{M})$ has dimension $n$
and $\widehat{\RDist}$ on $Q(\hat{M},M)$ has dimension $\hat{n}$.
Hence the inclusion $\ol{T}_*\RDist\subset \widehat{\RDist}$ is strict whenever $n<\hat{n}$.
This shows that the model of rolling of manifolds of different dimensions against each other
is not symmetric with respect to the order of the manifolds $M$ and $\hat{M}$.         
\end{remark} 

We can now provide a description of the vertical fiber $V|_q(\pi_Q)$ for a point 
$q=(x,\hat{x};A)\in Q$.

\begin{proposition}\label{prop-vert-fiber}
If $q=(x,\hat{x};A)\in Q$, then the vertical fiber $V|_q(\pi_Q)$ is given by 
\[
V|_q(\pi_Q)=\begin{cases}
\nu\big(\big\{B\in T^*|_x M\otimes T|_{\hat{x}}\hat{M}\ |\ A^{\ol{T}}B\in \so(T|_x M)\big\}\big)\big|_q, & \textrm{if}\ n\leq\hat{n},\cr
\nu\big(\big\{B\in T^*|_x M\otimes T|_{\hat{x}}\hat{M}\ |\ BA^{\ol{T}}\in \so(T|_{\hat{x}} \hat{M})\big\}\big)\big|_q, & \textrm{if}\ n\geq\hat{n}.
\end{cases}.
\]
\end{proposition}
\begin{proof}  Let $q=(x,\hat{x};A)\in Q$ and $B\in T^*|_xM\otimes T|_{\hat{x}}\hat{M}$. Proving the proposition amounts to show that $\nu(B)|_q$ (which is \emph{a priori} only an element of $V|_q(\pi_{T^*M\otimes T\hat{M}})$) belongs to $V|_q(\pi_Q)$ if and only if
\begin{itemize}
\item[(i)] $A^{\ol{T}}B\in\so(n)$, if $n\leq \hat{n}$,
\item[(ii)] $BA^{\ol{T}}\in\so(\hat{n})$, if $n\geq \hat{n}$.
\end{itemize}

Choose first a $\pi_Q$-vertical curve $q(t)=(x,\hat{x};A(t))$ inside $Q$
such that $A(0)=A$ i.e., $q(t)\in \pi_Q^{-1}(x,\hat{x})\subset T^*|_xM\otimes T|_{\hat{x}}\hat{M}$.

In the case (i), we have $\hat{g}(A(t)X,A(t)Y)=g(X,Y)$ for all $X,Y\in T|_x M$
so by differentiating, at $t=0$,
$g(AX,BY)+g(BX,AY)=0$ for all $X,Y\in T|_x M$,
where $B\in T^*|_x M\otimes T|_{\hat{x}}\hat{M}$
is such that $A'(0)=\nu(B)|_q$.
This condition can be written as $g(B^{\ol{T}}AX,Y)+g(A^{\ol{T}}BX,Y)=0$ for all $X,Y$
and hence $B^{\ol{T}}A+A^{\ol{T}}B=0$.
The result follows, since for a given $(x,\hat{x};A)\in Q$,
the set of $B\in T^*|_xM\otimes T|_{\hat{x}}\hat{M}$ s.t. $A^{\ol{T}}B\in\so(T|_x M)$
has dimension equal to $\dim \pi_{Q}^{-1}(x,\hat{x})$.

In the case (ii), we have $g(A(t)X,A(t)Y)=g(X,Y)$ for all $X,Y\in (\ker A(t))^\perp=\IM (A(t)^{\ol{T}})$.
Choose $\hat{X},\hat{Y}\in T|_{\hat{x}} M$.
Then $g(A(t)^{\ol{T}}\hat{X},A(t)^{\ol{T}}\hat{Y})=\hat{g}(\hat{X},\hat{Y})$,
since $A(t)A(t)^{\ol{T}}=\id_{T|_{\hat{x}}\hat{M}}$,
and so by differentiating at $t=0$, we get $g(A^{\ol{T}}\hat{X},B^{\ol{T}}\hat{Y})+g(B^{\ol{T}}\hat{X},A^{\ol{T}}\hat{Y})=0$,
where $B\in T^*|_x M\otimes T|_{\hat{x}}\hat{M}$
is such that $A'(0)=\nu(B)|_q$.
This clearly means that $BA^{\ol{T}}+AB^{\ol{T}}=\id_{T|_{\hat{x}}\hat{M}}$
and the result follows.

\end{proof}

\begin{remark}
The case (ii) considered above could be handled by using the diffeomorphism $\ol{T}:Q\to\hat{Q}$ introduced in Proposition \ref{pr:diffeo_Q_hatQ}.
Indeed, if $n\geq\hat{n}$, we may apply (i) on $\hat{Q}$
to obtain that for $q'=(\hat{x},x;A')\in \hat{Q}$, 
we have that $V|_{q'}(\pi_{\hat{Q}})$ consists of $B'\in T^*|_{\hat{x}}\hat{M}\otimes T|_x M$
such that ${A'}^{\ol{T}}B\in\so(T|_{\hat{x}}\hat{M})$.
But taking $q=(x,\hat{x};A)$, $q'=\ol{T}(q)$, $B\in T^*|_x M\otimes T|_{\hat{x}}\hat{M}$
and $B'=B^{\ol{T}}$,
this means $AB^{\ol{T}}\in \so(T|_{\hat{x}}\hat{M})$ i.e., $BA^{\ol{T}}\in \so(T|_{\hat{x}}\hat{M})$.
\end{remark}

As an interesting special cases of rolling with different dimension,
we next consider the cases where $n=\hat{n}+1$ or $\hat{n}=n+1$.

\begin{proposition}\label{pr:hatM_codim1}
Let $(M,g)$, $(\hat{M},\hat{g})$ be oriented Riemannian manifolds
of dimensions $n$ and $\hat{n}=n-1$, with $n\geq 2$.
Define $(\hat{M}^{(1)},\hat{g}^{(1)})$ to be the Riemannian product $(\R\times\hat{M},s_1\oplus\hat{g})$,
with the obvious orientation,
and write $Q^{(1)}=Q(M,\hat{M}^{(1)})$
and let $\LRD^{(1)}$, $\RDist^{(1)}$ to be the rolling lift and the rolling distribution on $Q^{(1)}$.
We define for every $a\in\R$,
\[
\iota_a:Q\to Q^{(1)};
\quad \iota_a(x,\hat{x};A)=(x,(a,\hat{x});A^{(1)}),
\]
where $A^{(1)}:T|_x M\to T|_{(a,\hat{x})} (\R\times \hat{M})$ is defined as follows:
$A^{(1)}\in Q^{(1)}$ and
\[
A^{(1)}|_{(\ker A)^\perp}=&(0,A|_{(\ker A)^\perp}) \\
A^{(1)}(\ker A)=&\R\pa{r}\big|_{(a,\hat{x})}\times \{0\},
\]
where $\pa{r}$ is the canonical vector field on $\R$ in the positive direction,
which we consider to be a vector field on $\hat{M}^{(1)}$ in the usual way.

Then for every $a\in\R$ the map $\iota_a$ is an embedding and
for every $q_0=(x_0,\hat{x}_0;A_0)\in Q$, $a_0\in\R$ and $X\in T|_x M$ one has
\[
\LRD(X)|_{q_0}=&\Pi_*\LRD^{(1)}(X)|_{\iota_{a_0}(q_0)} \\
\mc{O}_{\RDist}(q_0)=&\Pi(\mc{O}_{\RDist}^{(1)}(\iota_{a_0}(q_0))).
\]
where
\[
\Pi:Q^{(1)}\to Q;\quad \Pi(x,(a,\hat{x}),A^{(1)})=(x,\hat{x};(\pr_2)_*\circ A^{(1)}),
\]
is a surjective submersion and $\pr_2:\R\times M\to M$ is the projection onto the second factor.
\end{proposition}

\begin{proof}
Let $\gamma$ be a path in $M$ starting at $x_0$
and $q(t)=(\gamma(t),\hat{\gamma}(t);A(t)):=q_{\RDist}(\gamma,q_0)(t)$.
We define
a path on $q^{(1)}(t)=(\gamma(t),\hat{\gamma}^{(1)}(t);A^{(1)}(t))$
on $Q^{(1)}$ as follows:
\[
\hat{\gamma}^{(1)}(t):=&\Big(a_0+\int_0^t \iota_{a_0}(A_0)p^{T}(A_0)P_s^0(\gamma)\dot{\gamma}(s)\diff s,\hat{\gamma}(t)\Big) \\
A^{(1)}(t):=&P_0^t(\hat{\gamma}^{(1)})\circ \iota_{a_0}(A_0)\circ P_t^0(\gamma),
\]
where for every $q=(x,\hat{x};A)\in Q$ one defines
\[
& p^\perp(A):T|_x M\to (\ker A)^\perp \\
& p^T(A):T|_x M\to \ker A,
\]
as the $g$-orthogonal projections.

We will show that $q^{(1)}$ is the rolling curve on $Q^{(1)}$ starting from $\iota_{a_0}(q_0)$.
Indeed, clearly $q^{(1)}(0)=(\gamma(0),(a_0,\hat{\gamma}(0));\iota_{a_0}(A_0))=\iota_{a_0}(q_0)$
and $A^{(1)}(t)\in Q^{(1)}$ for all $t$.
Also, on one hand 
\[
\dot{\hat{\gamma}}^{(1)}(t)
=\big(b(t)\pa{r}\big|_{\hat{\gamma}^{(1)}(t)},\dot{\hat{\gamma}}(t)\big),
\]
where $b(t)$ is defined by $\iota_{a_0}(A_0)p^{T}(A_0)P_t^0(\gamma)\dot{\gamma}(t)=:(b(t)\pa{r}|_{(\hat{x}_0,a_0)},0)$,
while on the other hand
\[
A^{(1)}(t)\dot{\gamma}(t)
=&P_0^t(\hat{\gamma}^{(1)})\iota_{a_0}(A_0)P_t^0(\gamma)\dot{\gamma}(t) \\
=&P_0^t(\hat{\gamma}^{(1)})\iota_{a_0}(A_0)(p^T(A_0)+p^\perp(A_0))P_t^0(\gamma)\dot{\gamma}(t).
\]
Now since $A_0X=A_0p^{\perp}(A_0)X$ for every $X\in T|_{x_0} M$, we have
\[
A(t)\dot{\gamma}(t)=&P_0^t(\hat{\gamma})A_0P_t^0(\gamma)\dot{\gamma}(t) \\
=&P_0^t(\hat{\gamma})A_0p^{\perp}(A_0)P_t^0(\gamma)\dot{\gamma}(t)
\]
and also
\[
\iota_{a_0}(A_0)p^T(A_0)X=(\iota_{a_0}(A_0)p^T(A_0)X,0).
\]
Since $M^{(1)}$ is a Riemannian product, we have for every $\hat{X}\in T|_{\hat{x}_0} \hat{M}\subset T|_{(a_0,\hat{x}_0)} (\R\times\hat{M})$
\[
P_0^t(\hat{\gamma}^{(1)})\hat{X}=&(0,P_0^t(\hat{\gamma})\hat{X}) \\
P_0^t(\hat{\gamma}^{(1)})\pa{r}|_{\iota_{a_0}(q_0)}=&\pa{r}\big|_{\gamma^{(1)}(t)}
\]
where the latter implies that since $\iota_{a_0}(A_0)p^T(A_0)P_t^0(\gamma)\dot{\gamma}(t)=(b(t)\pa{r}\big|_{(\hat{x}_0,a_0)},0)$,
then
\[
P_0^t(\hat{\gamma}^{(1)})\iota_{a_0}(A_0)p^T(A_0)P_t^0(\gamma)\dot{\gamma}(t)
=(b(t)P_0^t(\gamma^{(1)})\pa{r}\big|_{(\hat{x}_0,a_0)},0)=(b(t)\pa{r}\big|_{\gamma^{(1)}(t)},0).
\]
Therefore,
\[
A^{(1)}(t)\dot{\gamma}(t)
=&\big(b(t)\pa{r}\big|_{\gamma^{(1)}(t)},P_0^t(\hat{\gamma})\iota_{a_0}(A_0)p^\perp(A_0)P_t^0(\gamma)\dot{\gamma}(t)\big) \\
=&\big(b(t)\pa{r}\big|_{\gamma^{(1)}(t)},P_0^t(\hat{\gamma})A_0P_t^0(\gamma)\dot{\gamma}(t)\big)
=(b(t)\pa{r}\big|_{\gamma^{(1)}(t)},A(t)\dot{\gamma}(t)) \\
=&(b(t)\pa{r}\big|_{\gamma^{(1)}(t)},\dot{\hat{\gamma}}(t))=\dot{\hat{\gamma}}^{(1)}(t).
\]
This and the definition of $A^{(1)}(t)$ show
that
$q^{(1)}(t)=q_{\RDist^{(1)}}(\gamma,\iota_{a_0}(q_0))(t)$ for all $t$.

Finally, since $A^{(1)}(t)=P_0^t(\hat{\gamma}^{(1)})\iota_{a_0}(A_0)(p^T(A_0)+p^\perp(A_0))P_t^0(\gamma)$
and by what was said above about parallel transport in Riemannian product,
it follows that
\[
(\pr_1)_*A^{(1)}(t)=P_0^t(\hat{\gamma}^{(1)})\iota_{a_0}(A_0)p^\perp(A_0)P_t^0(\gamma)
=P_0^t(\gamma)\iota_{a_0}(A_0)P_t^0(\gamma),
\]
which proves that
\[
\Pi\big(q_{\RDist^{(1)}}(\gamma,\iota_{a_0}(q_0))(t)\big)=q_{\RDist}(\gamma,q_0)(t)
\]
and hence $\mc{O}_{\RDist}(q_0)\subset \Pi\big(\mc{O}_{\RDist^{(1)}}(\iota_{a_0}(q_0))\big)$
as well as
\[
\Pi_*\LRD^{(1)}(\dot{\gamma}(0))|_{\iota_{a_0}(q_0)})
=\Pi_*\dot{q}_{\RDist^{(1)}}(\gamma,\iota_{a_0}(q_0))(0)
=\dot{q}_{\RDist}(\gamma,q_0)(0)=\LRD(\dot{\gamma}(0))|_{q_0}.
\]

Finally, if $q^{(1)}=(x,(a,\hat{x});A^{(1)})\in \mc{O}_{\RDist^{(1)}}(\iota_{a_0}(q_0))$,
take a path $\gamma$ in $M$ starting from $x_0$
such that $q^{(1)}=q_{\RDist^{(1)}}(\gamma,q_0)(1)$.
By what was done above, it follows that
$\Pi\big(q_{\RDist^{(1)}}(\gamma,\iota_{a_0}(q_0))(t)\big)=q_{\RDist}(\gamma,q_0)(t)$
and thus, evaluating this at $t=1$
gives $\Pi(q^{(1)})\in \mc{O}_{\RDist}(q_0)$,
whence $\Pi\big(\mc{O}_{\RDist^{(1)}}(\iota_{a_0}(q_0))\big)\subset \mc{O}_{\RDist}(q_0)$.

The claim that $\iota_a$ is an embedding for every $a\in\R$ is obvious.
We still want to prove that $\Pi$ is a surjective submersion.
This follows trivially from $\Pi\circ\iota_a=\id_Q$.
End of the proof.
\end{proof}

\begin{corollary}\label{cor:hatM_codim1}
With the assumptions and notations of Proposition \ref{pr:hatM_codim1},
if the orbit $\mc{O}_{\RDist}(q_0)$ is not open in $Q$ for some $q_0\in Q$,
then $\mc{O}_{\RDist^{(1)}}(\iota_{a_0}(q_0))$
is not open in $Q^{(1)}$.
\end{corollary}

\begin{proof}
Suppose $\mc{O}_{\RDist^{(1)}}(\iota_{a_0}(q_0))$ were open in $Q^{(1)}$,
then since $\Pi:Q^{(1)}\to Q$ is a smooth submersion,
it is an open map and hence its image
$\Pi(\mc{O}_{\RDist^{(1)}}(\iota_{a_0}(q_0)))=\mc{O}_{\RDist}(q_0)$
is open. End of the proof.

\end{proof}

As a consequence of this corollary and Theorem \ref{th:3D-1}
we get the following theorem concerning 
non-controllability in the case $n=3$, $\hat{n}=2$,
whence $\dim Q=8$.
Recall that $Q=Q(M,\hat{M})$ is connected
and thus the rolling problem $(R)$
is non-controllable if and only if there exists a $q_0\in Q$
such that $\mc{O}_{\RDist}(q_0)$ is not open in $Q$.

\begin{theorem}
Let $n=\dim M=3$, $\hat{n}=\dim \hat{M}=2$
and $q_0=(x_0,\hat{x}_0;A_0)\in Q$.

If the orbit $\mc{O}_{\RDist}(q_0)$
is not open in $Q$, then there exists an open dense subset $O$ of $\mc{O}_{\RDist}(q_0)$
such that for every $q_1=(x_1,\hat{x}_1;A_1)\in O$
there is an open neighbourhood $U$ of $x_1$
for which it holds that
$(U,g|_U)$ is isometric to some warped product $(I\times N,h_f)$,
where $I\subset\R$ is an open interval and the warping function $f$ satisfying $f''=0$.
\end{theorem}

\begin{proof}
Let $(\hat{M}^{(1)},\hat{g}^{(1)})$ be the Riemannian product $(\R\times \hat{M},s_1\oplus \hat{g})$
and let $a_0\in\R$.

Since the orbit $\mc{O}_{\RDist}(q_0)$ is not open in $Q$,
it follows from Corollary \ref{cor:hatM_codim1} that the orbit $\mc{O}_{\RDist^{(1)}}(\iota_{a_0}(q_0))$ is not open in $Q^{(1)}$.
But then Theorem \ref{th:3D-1}
provides an open dense subset $O^{(1)}$ of $\mc{O}_{\RDist^{(1)}}(\iota_{a_0}(q_0))$
such that one of (a)-(c) there holds.
Then $O:=\Pi(O^{(1)})$ is open an dense in $\mc{O}_{\RDist}(q_0)$.
Let $q_1=(x_1,\hat{x}_1;A_1)\in O$, choose $q_1^{(1)}\in O^{(1)}$ such that $\Pi(q^{(1)})=q_1$,
whence $q^{(1)}=\iota_{a_1}(q_1)$ for some $a_1\in\R$.

Let $U,\hat{U}^{(1)}$ be the neighbourhoods of $x_1,(a_1,\hat{x}_1)$
as in Theorem \ref{th:3D-1} corresponding to $q_1^{(1)}$. We choose $\hat{U}^{(1)}$ to be of the form $\hat{U}\times I$ for $I\subset\R$
an open interval and $\hat{U}\subset \hat{M}$ open.

If (a) there holds, it means that $(U,g|_U)$ is (locally) isometric to the Riemannian product $I\times \hat{U}$
(hence we have $f=1$).

The case (b) cannot occur, since $(\hat{U}^{(1)},\hat{g}^{(1)}|_{\hat{U}^{(1)}})$,
as a Riemannian product, cannot be of class $\mc{M}_{\beta}$ for $\beta>0$.

Suppose therefore that (c) holds.
Let $F:(I\times N,h_f)\to U$ and $\hat{F}:(\hat{I}\times \hat{N},\hat{h}_{\hat{f}})\to \hat{U}$
be the isomorphisms.
The eigenvalues of the curvature tensor $\hat{R}^{(1)}|_{(a_1,\hat{x})}$ being $0,-\hat{\sigma}(\hat{x}),0$,
with $\hat{\sigma}(\hat{x})$ the sectional curvature of $(\hat{M},\hat{g})$ at $\hat{x}$,
we see that the warping function $f$ must satisfy $f''=0$.
\end{proof}

\begin{proposition}\label{pr:M_codim1}
Let $(M,g)$, $(\hat{M},\hat{g})$ be oriented Riemannian manifolds
of dimensions $n=\hat{n}-1$ and $\hat{n}$, with $\hat{n}\geq 2$.
Define $(M^{(1)},g^{(1)})$ to be the Riemannian product $(\R\times M,s_1\oplus g)$,
with the obvious orientation,
and write $Q^{(1)}=Q(M^{(1)},\hat{M})$
and let $\LRD^{(1)}$, $\RDist^{(1)}$ to be the rolling lift and the rolling distribution on $Q^{(1)}$.
We define for every $a\in\R$,
\[
\iota_a:Q\to Q^{(1)};
\quad \iota_a(x,\hat{x};A)=((a,x),\hat{x};A^{(1)}),
\]
where $A^{(1)}:T|_{(a,x)} (M\times\R)\to T|_{\hat{x}} \hat{M}$ is defined as follows:
$A^{(1)}\in Q^{(1)}$ and
\[
A^{(1)}|_{T|_x M}=&A \\
A^{(1)}\pa{r}\big|_{(a,x)}\in &(\IM A)^\perp,
\]
where $\pa{r}$ is the canonical vector field on $\R$ in the positive direction,
which we consider to be a vector field on $M^{(1)}$ in the usual way.

Then for every $a\in\R$ the map $\iota_a$ is an embedding and
for every $q_0=(x_0,\hat{x}_0;A_0)\in Q$, $a_0\in\R$ and $X\in T|_x M\subset T|_{(a,x)} (\R\times M)$ one has
\[
(\iota_{a_0})_*\LRD(X)|_{q_0}=&\LRD^{(1)}(X)|_{\iota_{a_0}(q_0)}.
\]
Moreover, if one defines
\[
\Pi:Q^{(1)}\to Q;\quad \Pi((x,a),\hat{x};A^{(1)})=(x,\hat{x};A\circ  (i_a)_*),
\]
where $i_a:M\to \R\times M$; $x\mapsto (a,x)$
and if $\Delta_{\mathrm{R}}$
is the subdistribution of $\RDist^{(1)}$ defined by
\[
\Delta_{\mathrm{R}}|_{q^{(1)}}=(\iota_a)_*\RDist|_{\Pi(q^{(1)})},\quad \forall q^{(1)}=((a,x),\hat{x};A^{(1)})\in Q^{(1)},
\]
then
\[
\iota_{a_0}(\mc{O}_{\RDist}(q_0))=&\mc{O}_{\Delta_{\mathrm{R}}}(\iota_{a_0}(q_0))\subset \mc{O}_{\RDist^{(1)}}(\iota_{a_0}(q_0)).
\]
\end{proposition}

\begin{proof}
The claim that $\iota_a$ is an embedding is obvious
and since $\Pi\circ\iota_a=\id_Q$, it follows that $\Pi$ is a submersion.

Let $\gamma$ be a path in $M$ starting from $x_0$
and $q(t)=(\gamma(t),\hat{\gamma}(t);A(t))=q_{\RDist}(\gamma,q_0)(t)$.
We define a path $q^{(1)}(t)=(\gamma^{(1)}(t),\hat{\gamma}(t);A^{(1)}(t))$
on $Q^{(1)}$ by
\[
\gamma^{(1)}(t)=&(a_0,\gamma(t)) \\
A^{(1)}(t)=&P_0^t(\hat{\gamma})\circ \iota_{a_0}(A_0)\circ P_0^t(\gamma^{(1)}).
\]
Since $M^{(1)}$ is a Riemannian product,
\[
P_0^t(\gamma^{(1)})\dot{\gamma}^{(1)}(t)
=P_0^t(\gamma^{(1)})(0,\dot{\gamma}(t))=(0,P_0^t(\gamma)\dot{\gamma}(t))
\]
and so
\[
A^{(1)}(t)\dot{\gamma}^{(1)}(t)=P_0^t(\hat{\gamma})\iota_{a_0}(A_0)(P_0^t(\gamma)\dot{\gamma}(t),0)
=P_0^t(\hat{\gamma})A_0P_0^t(\gamma)\dot{\gamma}(t)=A(t)\dot{\gamma}(t)=\dot{\hat{\gamma}}(t).
\]
Since also $q^{(1)}(0)=\iota_{a_0}(q_0)$, this proves that $q^{(1)}(t)=q_{\RDist^{(1)}}(\gamma^{(1)},\iota_{a_0}(q_0))(t)$ for all $t$.

Next notice that
\[
\pi_{Q^{(1)}}(\iota_{a_0}(q(t)))=((a_0,\gamma(t)),\hat{\gamma}(t))=(\gamma^{(1)}(t),\hat{\gamma}(t))=\pi_{Q^{(1)}}(q^{(1)}(t))
\]
and if $X\in T|_x M\subset T|_{(a_0,x)}(\R\times M)$,
\[
A^{(1)}(t)(0,X)=A(t)X=\iota_{a_0}(A(t))X
\]
and since $A^{(1)}(t)T|_{\gamma(t)} M\perp A^{(1)}(t)\pa{r}\big|_{(\gamma^{(1)}(t)}$,
and $(\iota_{a_0}\circ A(t))T|_x M\perp (\iota_{a_0}\circ A(t))\pa{r}\big|_{(\gamma^{(1)}(t)}$,
we must have, by orientation,
\[
A^{(1)}(t)\pa{r}\big|_{(\gamma^{(1)}(t)}= (\iota_{a_0}\circ A(t))\pa{r}\big|_{(\gamma^{(1)}(t)}.
\]
This proves that $\iota_{a_0}(q(t))=q^{(1)}(t)$
and therefore
\[
\LRD^{(1)}((\dot{\gamma}(0),0))|_{\iota_{a_0}(q_0)}
=\LRD^{(1)}(\dot{\gamma}^{(1)}(0))|_{\iota_{a_0}(q_0)}
=\dot{q}^{(1)}(0)=(\iota_{a_0})_*\dot{q}(t)=(\iota_{a_0})_*\LRD(\dot{\gamma}(0))|_{q_0},
\]
which shows that for every $X\in T|_{x_0} M\subset T|_{(a_0,x_0)} (\R\times M)$ one has
\[
\LRD^{(1)}(X)|_{\iota_{a_0}(q_0)}=(\iota_{a_0})_*\LRD(X)|_{q_0}.
\]

Notice also that since we proved that $\iota_{a_0}(q_{\RDist}(\gamma,q_0)(t))=q_{\RDist^{(1)}}(\gamma,q_0)(t)$,
we have shown that
\[
\iota_{a_0}(\mc{O}_{\RDist}(q_0))\subset \mc{O}_{\RDist^{(1)}}(\iota_{a_0}(q_0)).
\]
Also, for every $q\in \mc{O}_{\RDist}(q_0)$ we have (recall that $\Pi\circ\iota_{a_0}=\id_Q$)
\[
\Delta_{\mathrm{R}}|_{\iota_{a_0}(q)}=(\iota_{a_0})_*\RDist|_{q}\subset T|_{\iota_{a_0}(q_0)} (\iota_{a_0}(\mc{O}_{\RDist}(q_0))).
\]
Finally, because $\iota_{a_0}|_{\mc{O}_{\RDist}(q_0)}$ is an immersion into $Q^{(1)}$ and
\[
\Delta_{\mathrm{R}}|_{\iota_{a_0}(\mc{O}_{\RDist}(q_0))}=(\iota_{a_0})_*\RDist|_{\mc{O}_{\RDist}(q_0)},
\]
we have
\[
\mc{O}_{\Delta_{\textrm{R}}}(\iota_{a_0}(q_0))=\iota_{a_0}(\mc{O}_{\RDist}(q_0)).
\]
The proof is finished.
\end{proof}

\begin{corollary}
With the assumptions and notations of Proposition \ref{pr:M_codim1},
if the orbit $\mc{O}_{\RDist}(q_0)$ is open in $Q$ for some $q_0\in Q$,
then $\mc{O}_{\RDist^{(1)}}(\iota_{a_0}(q_0))$
has codimension at most $1$ in $Q^{(1)}$.
\end{corollary}

\begin{proof}
The dimensions of $Q$ and $Q^{(1)}$ are $\dim Q=2\hat{n}-1+\frac{\hat{n}(\hat{n}-1)}{2}=\dim Q^{(1)}-1$,
so if $\mc{O}_{\RDist}(q_0)$ is open in $Q$, one has $\dim\mc{O}_{\RDist}(q_0)=\dim Q$ and thus,
\[
\dim  \mc{O}_{\RDist^{(1)}}(\iota_{a_0}(q_0))
\geq \dim \mc{O}_{\Delta_{\mathrm{R}}}(\iota_{a_0}(q_0))
=\dim \iota_{a_0}(\mc{O}_{\RDist}(q_0))
=\dim Q=\dim Q^{(1)}-1.
\]
\end{proof}

%%%%%%%%%%%%%%%%%%%%%%%
\subsection{Controllability Results}
%%%%%%%%%%%%%%%%%%%%%%%

\subsubsection{Rolling Problem $(NS)$}
Since Theorem \ref{th:NS_orbit_V} and Corollary \ref{cor:NS_orbit_V_inf} 
evidently hold as such in the case of non-equal dimensions (i.e., $n\neq\hat{n}$),
we will be more interested to see how Theorem \ref{theo-00}
could be formulated.
We first need a definition.

\begin{definition}
For $n,\hat{n}\geq 2$, we define
\[
\SO(n;\hat{n}):=\begin{cases}
\{A\in (\R^n)^*\otimes \R^{\hat{n}}\ |\ A^TA=\id_{\R^n}\}, & \mathrm{if}\ n<\hat{n}, \cr
\{A\in (\R^n)^*\otimes \R^{\hat{n}}\ |\ AA^T=\id_{\R^{\hat{n}}}\}, & \mathrm{if}\ n>\hat{n}, \cr
\SO(n), & \mathrm{if}\ n=\hat{n},
\end{cases}
\]
where $(\R^n)^*\otimes \R^{\hat{n}}$ is the set of $n\times\hat{n}$ real matrices
and $A^T$ denotes the usual transpose of matrices. 
\end{definition}

\begin{theorem}
Fix some orthonormal frames $F$, $\hat{F}$ of $M$, $\hat{M}$
at $x\in M$, $\hat{x}\in\hat{M}$ and
let $\mathfrak{h}=\mathfrak{h}|_F\subset\so(n)$, $\hat{\mathfrak{h}}=\hat{\mathfrak{h}}|_{\hat{F}}\subset\so(\hat{n})$
be the holonomy Lie-algebras of $M$, $\hat{M}$ w.r.t to these frames.
Then the control system $\sns$ is completely controllable if and only if
for every $A\in \SO(n;\hat{n})$,
\[
A\mathfrak{h}-\hat{\mathfrak{h}}A=
\begin{cases}
\big\{B\in(\R^n)^*\otimes \R^n\ |\ A^{T}B\in \so(n)\big\}, & \textrm{if}\ n\leq \hat{n}, \cr
\big\{B\in(\R^n)^*\otimes \R^n\ |\ BA^{T}\in \so(\hat{n})\big\}, & \textrm{if}\ n\geq \hat{n}.
\end{cases}
\]

\end{theorem}

\subsubsection{Rolling Problem $(R)$}
Notice that Proposition \ref{pr:R_comm_L2}
still holds when $n=\dim M$ is not equal to $\hat{n}=\dim \hat{M}$.
The rolling curvature of $\RDist$ on $Q$ is denoted as before
but that of $\widehat{\RDist}$ on $\hat{Q}$ is written as $\widehat{\Rol}$ i.e.,
\[
\widehat{\Rol}(\hat{X},\hat{Y})(B)
=B\hat{R}(\hat{X},\hat{Y})-R(B\hat{X},B\hat{Y})B,
\]
for $(\hat{x},x;B)\in \hat{Q}$ and $\hat{X},\hat{Y}\in T|_{\hat{x}}\hat{M}$.

As a consequence, we have a generalization of Corollary \ref{cor:2.5:1}.

\begin{corollary}
Use the notations introduced previously
and assume that $n\leq\hat{n}$.
Then the following two cases are equivalent:
\begin{itemize}
\item[(i)] $\RDist$ is involutive,
\item[(ii)] $(M,g)$ and $(\hat{M},\hat{g})$ have constant and equal curvature.
\end{itemize}
Also, if $n<\hat{n}$, then there is an equivalence between the two cases below:
\begin{itemize}
\item[(1)] $\widehat{\RDist}$ is involutive,
\item[(2)] $(M,g)$ and $(\hat{M},\hat{g})$ are both flat.
\end{itemize}
\end{corollary}

\begin{proof}
For some of the notations, see the proof of Corollary \ref{cor:2.5:1}.

(i) $\Longrightarrow$ (ii):
Assume that $\RDist$ is involutive.
This is equivalent to the vanishing of $\Rol$
i.e.,
\[
A(R(X,Y)Z)=\hat{R}(AX,AY)(AZ),\quad \forall (x,\hat{x};A)\in Q,\ X,Y,Z\in T|_x M,
\]
which implies
\[
\sigma_{(X,Y)}=&g(R(X,Y)Y,X)
=\hat{g}(A(R(X,Y)Y),AX) \\
=&\hat{g}(\hat{R}(AX,AY)(AY),AX)=\hat{\sigma}_{(AX,AY)},
\]
for every $X,Y$ orthonormal in $T|_x M$ and $(x,\hat{x};A)\in Q$.

Let $x\in M$, $\hat{x}\in\hat{M}$ be arbitrary points
and $X,\hat{X}\in T|_x M$ and $\hat{X},\hat{Y}\in T|_{\hat{x}}\hat{M}$
be arbitrary pairs of orthonormal vectors.

Choose any vectors $X_{3},\dots,X_n\in T|_x M$ and $\hat{X}_3,\dots,\hat{X}_{\hat{n}}\in T|_{\hat{x}}\hat{M}$
such that $X,Y,X_3,\dots,X_n$
and $\hat{X},\hat{Y},\hat{X}_3,\dots,\hat{X}_{\hat{n}}$ are positively oriented orthonormal frames.

Since $n\leq\hat{n}$, we may define $q=(x,\hat{x};A)\in Q$
by
\[
AX=\hat{X},\quad AY=\hat{Y},\quad AX_i=\hat{X}_i,\quad i=3,\dots,n,
\]
to obtain that $\sigma_{(X,Y)}=\hat{\sigma}_{(\hat{X},\hat{Y})}$.
Thus $(M,g)$ and $(\hat{M},\hat{g})$ have equal and constant curvature,
since the orthonormal pairs $X,Y$ and $\hat{X},\hat{Y}$ were arbitrary
and chosen independently from one another.

(ii) $\Longrightarrow$ (i):
Since $(M,g)$, $(\hat{M},\hat{g})$ both have equal constant curvature, say $k\in\R$,
we have
\[
& R(X,Y)Z=k\big(g(Y,Z)X-g(X,Z)Y\big),\quad X,Y,Z\in T|_x M,\ x\in M, \\
& \hat{R}(\hat{X},\hat{Y})\hat{Z}=k\big(\hat{g}(\hat{Y},\hat{Z})\hat{X}-\hat{g}(\hat{X},\hat{Z})\hat{Y}\big),\quad
\hat{X},\hat{Y},\hat{Z}\in T|_{\hat{x}}\hat{M},\ \hat{x}\in \hat{M}.
\]
On the other hand, if $(x,\hat{x};A)\in Q$, $X,Y,Z\in T|_x M$ we get
\[
& \hat{R}(AX,AY)(AZ)=k(\hat{g}(AY,AZ)(AX)-\hat{g}(AX,AZ)(AY)) \\
=& A(k(g(Y,Z)X-g(X,Z)Y)=A(R(X,Y)Z).
\]
This implies that $\Rol(X,Y)(A)=0$ since $Z$ was arbitrary. 
Hence $\RDist$ is involutive.

(2) $\Longrightarrow$ (1):
In this case $R=0$ and $\hat{R}=0$ so that
clearly $\widehat{\Rol}(\hat{X},\hat{Y})(B)\hat{Z}=B(\hat{R}(\hat{X},\hat{Y})\hat{Z})-R(B\hat{X},B\hat{Y})(B\hat{Z})=0$
for all $\hat{X},\hat{Y},\hat{Z}\in T|_{\hat{x}}\hat{M}$ and $(\hat{x},x;B)\in \hat{Q}$.
This proves that $\widehat{\RDist}$ is involutive.

(1) $\Longrightarrow$ (2):
Assume that $\widehat{\RDist}$ is involutive
i.e.,
\[
B(\hat{R}(\hat{X},\hat{Y})\hat{Z})=R(B\hat{X},B\hat{Y})(B\hat{Z}),
\quad \forall (\hat{x},x;B)\in\hat{Q},\ \hat{X},\hat{Y},\hat{Z}\in T|_{\hat{x}}\hat{M}.
\]
Then
\[
\sigma_{(B\hat{X},B\hat{Y})}=g(B(\hat{R}(\hat{X},\hat{Y})\hat{Y}),B\hat{X}),
\]
or
\[
\sigma_{(X,Y)}=g(B(\hat{R}(B^{\ol{T}}X,B^{\ol{T}}Y)(B^{\ol{T}}Y),X)
=\hat{\sigma}_{(B^{\ol{T}}X,B^{\ol{T}}Y)}.
\]

Given any $x\in M$, $\hat{x}\in\hat{M}$, $X,Y\in T|_x M$ and $\hat{X},\hat{Y}\in T|_{\hat{x}} M$,
choose some $X_3,\dots,X_n\in T|_x M$, $\hat{X}_3,\dots,\hat{X}_{\hat{n}}\in T|_{\hat{x}}\hat{M}$
such that 
\[
X,Y,X_3,\dots,X_n\hbox{
and }\hat{X},\hat{Y},\hat{X}_3,\dots,\hat{X}_{\hat{n}},
\]
are positively oriented orthonormal frames.
We define
\[
& B\hat{X}=X,\quad B\hat{Y}=Y,\quad B\hat{X}_i=X_i,\quad i=3,\dots,n, \\
& B\hat{X}_i=0,\quad i=n+1,\dots,\hat{n}
\]
so that $q=(\hat{x},x;B)\in \hat{Q}$,
$B^{\ol{T}}X=\hat{X}$, $B^{\ol{T}}Y=\hat{Y}$
and hence $\sigma_{(X,Y)}=\hat{\sigma}_{(\hat{X},\hat{Y})}$.
Thus $(M,g)$, $(\hat{M},\hat{g})$ have constant and equal curvature.

Suppose that the common constant curvature of $(M,g)$, $(\hat{M},\hat{g})$ is $k\in\R$.
We need to show that $k=0$.
Choose any $(\hat{x},x;B)\in\hat{Q}$.
Since $n<\hat{n}$, we may choose non-zero vectors $\hat{X}\in \ker B$ and $\hat{Y}\in(\ker B)^\perp$.
Then
\[
0=&\widehat{\Rol}(\hat{X},\hat{Y})(B)\hat{X}
=k\big(\hat{g}(\hat{Y},\hat{X})B\hat{X}-\hat{g}(\hat{X},\hat{X})B\hat{Y}\big)
-R(B\hat{X},B\hat{Y})(B\hat{X}) \\
=&k(0-\n{\hat{X}}_{\hat{g}}^2B\hat{Y})-0=-k\n{\hat{X}}_{\hat{g}}^2B\hat{Y}
\]
and since $\n{\hat{X}}_{\hat{g}}\neq 0$ and $B\hat{Y}\neq 0$ (since $0\neq \hat{Y}\in(\ker B)^\perp$),
it follows that $k=0$.
This completes the proof.
\end{proof}

We may also easily generalize Corollary \ref{cor:weak_ambrose}. The use will be made of Gauss-formula, which relates the curvature of a submanifold to that of the ambient Riemannian manifold and O'Neill-formulas, which relate the various curvatures related to Riemannian submersions
(see \cite{sakai91}, Propositions 3.8, 6.1, 6.2 and Corollary 6.3, Chapter II).
Since the proof is slightly less trivial, we state this as a proposition.

\begin{proposition}\label{pr:weak-C-A-H}
Suppose that $(M,g)$ and $(\hat{M},\hat{g})$ are complete with $\dim M=n$, 
$\dim \hat{M}=\hat{n}$. The following cases are equivalent:
\begin{itemize}
\item[(i)] There exists a $q_0=(x_0,\hat{x}_0;A_0)\in Q$
such that $\mc{O}_{\RDist}(q_0)$ is an integral manifold of $\RDist$.

\item[(ii)] There exists a $q_0=(x_0,\hat{x}_0;A_0)\in Q$
such that
\[
\Rol(X,Y)(A)=0,\quad \forall (x,\hat{x};A)\in \mc{O}_{\RDist}(q_0),\ X,Y\in T|_x M.
\]

\item[(iii)] There is a complete Riemannian manifold $(N,h)$,
a Riemannian covering map $F:N\to M$
and a smooth map $G:N\to\hat{M}$ such that:
\begin{itemize}
\item[(1)] If $n\leq \hat{n}$, $G$ is a Riemannian immersion that maps $h$-geodesics to $\hat{g}$-geodesics.

\item[(2)] If $n\geq \hat{n}$, $G$ is a Riemannian submersion
such that the co-kernel distribution $(\ker G_*)^{\perp}\subset TN$ is involutive
and the fibers $G^{-1}(\hat{x})$, $\hat{x}\in \hat{M}$, are totally geodesic submanifolds of $(N,h)$.
\end{itemize}
\end{itemize}

Moreover, in the case (iii)-(2), we may choose $N$ to be simply connected
and then $(N,h)$ is a Riemannian product of $(N_1,h_1)$, $(N_2,h_2)$,
where $\dim N_1=\hat{n}$, $\dim N_2=n-\hat{n}$,
the space $(N_1,h_1)$ is the universal Riemannian covering of $(\hat{M},\hat{g})$
and $G$ is given by
\[ 
G:N=N_1\times N_2\to \hat{M};\quad G(y_1,y_2)=\hat{\pi}(y_1)
\]
where $\hat{\pi}:N_1\to \hat{M}$ is a Riemannian covering map.
\end{proposition}

\begin{proof}
(i) $\iff$ (ii): This is proved with exactly the same argument that was used in the proof of Corollary \ref{cor:weak_ambrose}.

(i) $\Rightarrow$ (iii):
Let $N:=\mc{O}_{\RDist}(q_0)$ and $h:=(\pi_{Q,M}|_N)^*(g)$ i.e.,
for $q=(x,\hat{x};A)\in N$ and $X,Y\in T|_x M$,
define 
\[
h(\LRD(X)|_q,\LRD(Y)|_q)=g(X,Y).
\]
If $F:=\pi_{Q,M}|_N$ and $G:=\pi_{Q,\hat{M}}|_N$,
we  immediately see that $F$ is a local isometry (note that $\dim(N)=n$).
The completeness of $(N,h)$  follows from the completeness of $M$ and $\hat{M}$
using Proposition \ref{pr:rol_geodesic} which holds in verbatim also in the case where $n\neq\hat{n}$.
Hence by Proposition II.1.1 in \cite{sakai91},
$F$ is a (surjective) Riemannian covering.

Suppose then that $n\leq \hat{n}$.
Then for $q=(x,\hat{x};A)\in N$, $X,Y\in T|_x M$, one has
\[
\hat{g}(G_*(\LRD(X)|_q),G_*(\LRD(Y)|_q))
=\hat{g}(AX,AY)=g(X,Y)=h(\LRD(X)|_q,\LRD(Y)|_q),
\]
i.e., $G$ is a Riemannian immersion.
Moreover, if $\ol{\Gamma}:[0,1]\to N$ is an $h$-geodesic,
it is tangent to $\RDist$ and since it projects by $F$ to a $g$-geodesic $\gamma$,
it follows that $\ol{\Gamma}=q_{\RDist}(\gamma,\ol{\Gamma}(0))$
and Proposition \ref{pr:rol_geodesic} shows that
$G\circ \ol{\Gamma}=\hat{\gamma}_{\RDist}(\gamma,\ol{\Gamma}(0))$
is a $\hat{g}$-geodesic. This proves (iii)-(1).

On the other hand, if $n\geq\hat{n}$,
then for $q=(x,\hat{x};A)\in N$, any $X\in T|_x M$
s.t. $\LRD(X)|_q\in (\ker G_*|_q)^{\perp}$
and $Z\in \ker A$,
we have $G_*(\LRD(Z)|_q)=AZ=0$ i.e., $\LRD(Z)|_q\in\ker (G_*|_q)$
from which $g(X,Z)=h(\LRD(X)|_q,\LRD(Z)|_q)=0$.
This shows that $X\in(\ker A)^\perp$
and therefore, for all $X,Y\in T|_x M$
such that $\LRD(X)|_q,\LRD(Y)|_q\in (\ker G_*|_q)^{\perp}$,
we get $\hat{g}(G_*(\LRD(X)|_q),G_*(\LRD(Y)|_q))=h(\LRD(X)|_q,\LRD(Y)|_q)$ as above.
This proves that $G:N\to\hat{M}$ is a Riemannian submersion.

For any $\ol{X},\ol{Y}\in\VF(N)$ orthonormal and tangent to $(\ker G_*)^\perp$
around a point $q\in N$, we have 
$\sigma^h_{(\ol{X},\ol{Y})}=\hat{\sigma}_{(G_*\ol{X},G_*\ol{Y})}$ ($\sigma^h$ is the sectional curvature on $N$)
in that neighbourhood because $F$ is a Riemannian covering map
and because
\[
& \hat{\sigma}_{(G_*\ol{X},G_*\ol{Y})}
=\hat{g}(\hat{R}(G_*\ol{X},G_*\ol{Y})(G_*\ol{Y}),G_*\ol{X}))
=\hat{g}(\hat{R}(AF_*\ol{X},AF_*\ol{Y})(AF_*\ol{Y}),AF_*\ol{X})) \\
=&\hat{g}(AR(F_*\ol{X},F_*\ol{Y})F_*\ol{Y},AF_*\ol{X}))
=g(R(F_*\ol{X},F_*\ol{Y})F_*\ol{Y},A^{\ol{T}}AF_*\ol{X})) \\
=&g(R(F_*\ol{X},F_*\ol{Y})F_*\ol{Y},F_*\ol{X})=\sigma_{(F_*\ol{X},F_*\ol{Y})},
\]
since $\Rol=0$, $F_*\ol{X}\in(\ker A)^\perp$ on $N$ and where we wrote $A=G_*\circ (F_*|_q)^{-1}$ for $q=(x,\hat{x};A)$
in the chosen neighbourhood.
By Corollary 6.3, Chapter II in \cite{sakai91}, it follows that
for any $\ol{X},\ol{Y}\in\VF(N)$ tangent to $(\ker G_*)^\perp$ in an open set,
$[\ol{X},\ol{Y}]$ is tangent to $(\ker G_*)^\perp$ in that open set.
Thus $(\ker G_*)^\perp$ is involutive.

We still need to prove that the $G$-fibers are totally geodesic.
Let $q=(x,\hat{x};A)\in Q$, $V\in \ker G_*|_q=T|_q(G^{-1}(\hat{x}))$.
Then $V=\LRD(u)|_q$ for some $u\in T|_x M$
and if $\gamma$ is the $g$-geodesic starting from $x$ with the initial velocity $u$,
then $\Gamma:=q_{\RDist}(\gamma,q)$ is the $h$-geodesic with initial velocity $V$
(since $F$ is a Riemannian covering)
and also $\hat{\gamma}:=\hat{\gamma}_{\RDist}(\gamma,q)$ is a $\hat{g}$-geodesic by Proposition \ref{pr:rol_geodesic}
with initial velocity $\dot{\hat{\gamma}}(0)=A\dot{\gamma}(0)=Au=G_* V=0$,
by the choice of $V$.
But this means that $\hat{\gamma}$ is a constant curve, $\hat{\gamma}(\cdot)\equiv \hat{x}$ for all $t$,
which implies that $G(\Gamma(t))=\hat{\gamma}(t)=\hat{x}$ for all $t$
i.e., $\Gamma(t)\in G^{-1}(\hat{x})$.
This proves that every fiber $G^{-1}(\hat{x})$, $\hat{x}\in \hat{M}$,
is a totally geodesic submanifold of $(N,h)$
and so we have finally proved (iii)-(2).

(iii) $\Rightarrow$ (ii):
Let $x_0\in M$ and choose $z_0\in N$ such that $F(z_0)=x_0$.
Define $\hat{x}_0=G(z_0)\in\hat{M}$
and $A_0:=G_*|_{z_0}\circ (F_*|_{z_0})^{-1}:T|_{x_0} M\to T|_{\hat{x}_0}\hat{M}$.

The fact that $q_0=(x_0,\hat{x}_0;A_0)$ belongs to $Q$ can be seen as follows:
if (iii)-(1) holds, we have
\[
& \hat{g}(A_0X,A_0Y)=\hat{g}\big(G_*|_{z_0}((F_*|_{z_0})^{-1}X),G_*|_{z_0}((F_*|_{z_0})^{-1}Y)\big) \\
=& h((F_*|_{z_0})^{-1}X),(F_*|_{z_0})^{-1}Y)=g(X,Y),
\]
where we used that $G$ is a Riemannian immersion and that $F$ is a Riemannian covering map.
On the other hand, if (iii)-(2) holds
and if $X,Y\in(\ker A_0)^{\perp}$, clearly $(F_*|_{z_0})^{-1}X,(F_*|_{z_0})^{-1}Y\in(\ker G_*|_{z_0})^{\perp}$
and hence also $\hat{g}(A_0X,A_0Y)=g(X,Y)$ since $G$ is a Riemannian submersion.

Let $\gamma:[0,1]\to M$ be a smooth curve with $\gamma(0)=x_0$.
Since $F$ is a smooth covering map, there is a unique smooth curve $\Gamma:[0,1]\to N$
with $\gamma=F\circ\Gamma$.
Define $\hat{\gamma}=G\circ\Gamma$
and $A(t)=G_*|_{\Gamma(t)}\circ (F_*|_{\Gamma(t)})^{-1}$, $t\in [0,1]$.
As before, it follows that $A(t)\in Q$ for all $t\in [0,1]$ and
\begin{align}\label{eq:p_wa:0}
\dot{\hat{\gamma}}(t)=G_*|_{\Gamma(t)}\dot{\Gamma}(t)=A(t)\dot{\gamma}(t).
\end{align}

We claim that, for all $t\in [0,1]$,
\begin{align}\label{eq:p_wa:1}
\ol{\nabla}_{(\dot{\gamma}(t),\dot{\hat{\gamma}}(t))} A(\cdot)=0,
\end{align}
and
\begin{align}\label{eq:p_wa:2}
\Rol(\cdot,\cdot)(A(t))=0.
\end{align}

Indeed, suppose now that (iii)-(1) holds.
This means that, for every $z\in N$, there is a neighbourhood $U$ of $z$ in $N$
such that $G(U)$ is a totally geodesic submanifold of $(\hat{M},\hat{g})$
and $G:U\to \hat{M}$ is an isometric embedding.
Now if $X$ is a vector field parallel along $\gamma$ in $M$,
then since $F$ is a Riemannian covering,
there is a unique vector field $\ol{X}$ parallel along $\Gamma$ in $(N,h)$
such that $F_*\ol{X}=X$. For any $t_0\in [0,1]$, choose $U$ as above
around $\Gamma(t_0)$. Then near $t_0$
we have that $G_*\ol{X}$ is parallel to $\hat{\gamma}$ in $(\hat{M},\hat{g})$.
This proves that
\[
0=\hat{\nabla}_{\dot{\hat{\gamma}}(t)}(G_*\ol{X}(t))
=\hat{\nabla}_{\dot{\hat{\gamma}}(t)} (A(\cdot)X(\cdot))
=\big(\ol{\nabla}_{(\dot{\gamma}(t),\dot{\hat{\gamma}}(t))} A(\cdot)\big)X(t),
\]
and since $X(t)$ was an arbitrary field parallel along $\gamma$,
we have $\ol{\nabla}_{(\dot{\gamma}(t),\dot{\hat{\gamma}}(t))} A(\cdot)=0$
i.e., (\ref{eq:p_wa:1}).

Since, locally, the shape operator of $G(N)$ w.r.t $(\hat{M},\hat{g})$ vanishes
and $G(N)$ is locally Riemannian embedded submanifold of $(\hat{M},\hat{g})$,
we also have $G_*(R^h(\ol{X},\ol{Y})\ol{Z})=\hat{R}(G_*\ol{X},G_*\ol{Y})(G_*\ol{Z})$ for all $\ol{X},\ol{Y},\ol{Z}\in T|_z N$, $z\in N$
(see \cite{sakai91}, Proposition 3.8, Chapter II)
and hence for all $X,Y,Z\in T|_{\gamma(t)} M$,
\[
& A(t)(R(X,Y)Z)=G_*(F_*|_{\Gamma(t)})^{-1}(R(X,Y)Z)
=G_*(R^h(\ol{X},\ol{Y})\ol{Z}) \\
=&\hat{R}(G_*\ol{X},G_*\ol{Y})(G_*\ol{Z})=\hat{R}(A(t)X,A(t)Y)(A(t)Z),
\]
where $\ol{X}=(F_*|_{\Gamma(t)})^{-1} X$ etc.
This proves (\ref{eq:p_wa:2}).

On the other hand, suppose (iii)-(2) holds.
First we see that Eq. (\ref{eq:p_wa:2}) follows from \cite{sakai91}, Proposition 6.2, Chapter II (the operators $A$ and $T$ there vanish,
by assumptions on $N$ and $G$) and the fact that $F$ is a Riemannian covering.

To prove (\ref{eq:p_wa:1}) we proceed as follows.
Taking the simply connected covering of $N$, lifting the metric $h$ and composing
$G$ and $F$ with the projection from this covering to $N$,
we see that the conditions (iii)-(2) still hold
and thus we may assume that $N$ was simply connected in the first place.
Take any piecewise $C^1$ curve $\omega$ on $N$
and let $V_0\in\ker G_*|_{\omega(0)}$, $X_0\in (\ker G_*)^\perp|_{\omega(0)}$.
If $Z(t)$ is the parallel translate of $X_0+V_0$ along $\omega$,
we get from \cite{sakai91}, Proposition 6.1, Chapter II
(again, the operators $A$ and $T$ there vanish by assumptions)
that 
\[
0=&\nabla^h_{\dot{\omega}(t)} Z(t)
=(\nabla^h_{\dot{\omega}(t)} Z(t)^\perp)^\perp
+(\nabla^h_{\dot{\omega}(t)} Z(t)^T)^T,
\]
where for $Y\in TN$ we wrote $Y^T$ and $Y^\perp$ for the
components of $Y$ in the distributions $(\ker G_*)^\perp$ and $\ker G_*$, respectively
(this notation is in accordance with the notation in the referred result of \cite{sakai91}
and is not completely compatible with ours).
This proves that $Z(t)^T$ and $Z(t)^\perp$ are fields parallel to $\omega$
and since $Z(0)^T=X_0$, $Z(0)^\perp=V_0$,
we have that $Z(t)^T$ and $Z(t)^\perp$ are the parallel translates of $X_0$ and $V_0$,
respectively.
But this implies that
\[
(P^{\nabla^h})_0^t(\omega)(\ker G_*|_{\omega(0)})=\ker G_*|_{\omega(t)},
\quad
(P^{\nabla^h})_0^t(\omega)\big((\ker G_*|_{\omega(0)})^\perp\big)=(\ker G_*|_{\omega(t)})^\perp,
\]
i.e., $TN=\ker G_*\oplus (\ker G_*)^\perp$
is a splitting to $TN$ into two subbundles that are invariant under $\nabla^h$-parallel transport.

Since $N$ is simply connected and complete,
it follows from de Rham's Theorem (see \cite{sakai91}, Theorem 6.11, Chapter II)
that $(N,h)=(N_1,h_1)\times (N_2,h_2)$, a Riemannian product,
where $(N_1,h_1)$ and $(N_2,h_2)$ are both complete and simply connected
and $TN_1=(\ker G_*)^\perp$, $TN_2=\ker G_*$.

To see now that Eq. (\ref{eq:p_wa:1}) holds,
let $X$ be a vector field parallel along $\gamma$ in $M$,
write $\Gamma=(\Gamma_1,\Gamma_2)$,
take $\ol{X}=(\ol{X}_1,\ol{X}_2)$
(w.r.t $TN=TN_1\oplus TN_2$)
to be the unique lift of $X$ onto a vector field along $\Gamma$ in $N$
and compute
\[
0=&A(t)\nabla_{\dot{\gamma}(t)} X(\cdot)
=G_*\nabla^h_{(\dot{\Gamma}_1(t),\dot{\Gamma}_2(t))} \ol{X}(\cdot)
=G_*\nabla^h_{\dot{\Gamma}_1(t)}\ol{X}_1+G_*\nabla^h_{\dot{\Gamma}_2(t)}\ol{X}_2
=G_*\nabla^h_{\dot{\Gamma}_1(t)}\ol{X}_1,
\]
since $\nabla^h_{\dot{\Gamma}_2(t)}\ol{X}_2\in TN_2=\ker G_*$.
On the other hand, $G^{y_2}:N_1\to \hat{M}$; $y_1\mapsto G(y_1,y_2)$
is a local isometry for any $y_2\in N_2$ and hence
\[
0=&G_*\nabla^h_{\dot{\Gamma}_1(t)}\ol{X}_1
=(G^{\Gamma_2(t)})_*\nabla^h_{\dot{\Gamma}_1(t)}\ol{X}_1
=\hat{\nabla}_{(G^{\Gamma_2(t)})_*\dot{\Gamma}_1(t)} ((G^{\Gamma_2(t)})_*\ol{X}_1) \\
=&\hat{\nabla}_{\dot{\hat{\gamma}}(t)} (G_*\ol{X})
=\hat{\nabla}_{\dot{\hat{\gamma}}(t)} (A(\cdot)X(\cdot)),
\]
since $\dot{\hat{\gamma}}(t)=G_*\dot{\Gamma}(t)=G_*\dot{\Gamma}_1(t)=(G^{\Gamma_2(t)})_*\dot{\Gamma}_1(t)$
and $G_*\ol{X}=G_*\ol{X}_1=(G^{\Gamma_2(t)})_*\ol{X}_1$.
Thus (\ref{eq:p_wa:1}) holds and
this finishes the proof of (\ref{eq:p_wa:1})-(\ref{eq:p_wa:2}) in the case that (iii)-(2) holds.

Thus we have shown, because of (\ref{eq:p_wa:0}) and (\ref{eq:p_wa:1}), that $t\mapsto (\gamma(t),\hat{\gamma}(t);A(t))$ is the unique rolling
curve along $\gamma$ starting at $q_0=(x_0,\hat{x}_0;A_0)$ and defined on $[0,1]$
and therefore curves of $Q$ formed in this fashion fill
up the orbit $\mc{O}_{\RDist}(q_0)$.
Therefore, Eq. (\ref{eq:p_wa:2}) implies that $\Rol$ vanishes on $\mc{O}_{\RDist}(q_0)$
which means that we are in case (ii).

To prove the last claim in the statement of the proposition,
we continue the deduction done above in the case that condition (iii)-(2) holds.
Since for any $y_2\in N_2$, $G^{y_2}:N_1\to\hat{M}$
is a local Riemannian isometry and $N_1$ is simply connected and complete, 
it follows that $G^{y_2}$ is a universal Riemannian covering of $\hat{M}$.
We show that the map $G^{y_2}$ is independent
of the choice of $y_2\in N_2$ i.e., that $G(y_1,y_2)=G(y_1,y_2')$ for all $y_1\in N_1$ and $y_2,y_2'\in N_2$.
Indeed, take any smooth path $\Gamma_2$ in $N_2$ from $\Gamma_2(0)=y_2$ to $\Gamma_2(1)=y_2'$.
Then,
$\dif{t} (G_{y_1}\circ\Gamma_2(t))=G_*\dot{\Gamma}_2(t)=0$ for all $t$
since $\dot{\Gamma}_2(t)\in TN_2=\ker G_*$.
This shows that $G_{y_1}\circ\Gamma_2$ is a constant curve in $\hat{M}$
and thus $G(y_1,y_2)=G_{y_1}(\Gamma_2(0))=G_{y_1}(\Gamma_2(1))=G(y_1,y_2')$.

We fix $y'_2\in N_2$ and define $\hat{\pi}:=G^{y'_2}:N_1\to\hat{M}$
which is a universal Riemannian covering.
By what we just proved, it holds that $G(y_1,y_2)=G(y_1,y_2')=\hat{\pi}(y_1)$
which establishes the claim.
\end{proof}

\appendix

%%%%%%%%%%%%%%%%%%%%%%%%%%%%%%%%%%%%%%%%%%%%%%%%%%%%
\section{Fiber and Local Coordinates Point of View}\label{sec:2.2}\label{app:local}
%%%%%%%%%%%%%%%%%%%%%%%%%%%%%%%%%%%%%%%%%%%%%%%%%%%%

Let $F=(X_i)$, $\hat{F}=(\hat{X}_i)$
be (not necessarily orthonormal) local frames of $M$ and $\hat{M}$
defined on the open subsets $U,\hat{U}$, respectively.
We have local frames of 1-forms $F^*=((\theta^i),U)$, $\hat{F}^*=((\hat{\theta}^j),\hat{U})$
dual to these frames i.e., defined by $\theta^j(X_i)=\delta^j_i$,
$\hat{\theta}^j(\hat{X}_i)=\delta^j_i$.
The Christoffel symbols $\Gamma^k_{ij}=(\Gamma_F)^k_{ij}$,
$\hat{\Gamma}^k_{ij}=(\hat{\Gamma}_{\hat{F}})^k_{ij}$
 of $\nabla$, $\hat{\nabla}$
w.r.t the frames $F$, $\hat{F}$ are defined by (see \cite{spivakII99}, p. 266)
$\nabla_{X_i} X_j=\sum_k \Gamma^k_{ij} X_k$,
$\hat{\nabla}_{\hat{X}_i} \hat{X}_j=\sum_k \hat{\Gamma}^k_{ij} \hat{X}_k$.
Any $(x,\hat{x};A)\in T^*|_x(M)\otimes T|_{\hat{x}}(\hat{M})$
with $(x,\hat{x})\in U\times\hat{U}$ can be written in the form
\[
A=\sum_{i,j} A^i_j \theta^j|_x\otimes \hat{X}_i|_{\hat{x}},
\]
i.e., $(\pr_2\circ \tau_{F,\hat{F}})(A)=[A^i_j]$ (see section \ref{sec:1.1}).
Moreover, if $F,\hat{F}$ are orthonormal frames, then $A\in Q$ if and only if $[A^i_j]\in \SO(n)$.

Let $t\mapsto \gamma(t)$, $t\mapsto \hat{\gamma}(t)$, $t\in I$, be smooth curves
in $U$, $\hat{U}$, respectively, such that $\gamma(0)=x_0$, $\hat{\gamma}(0)=\hat{x}_0$,
where $I$ is a compact interval containing $0$.
Moreover, take $q_0=(x_0,\hat{x}_0;A_0)\in T^* M\otimes T\hat{M}$.
The no-spinning condition (\ref{eq:nospin2}) (i.e., the parallel translation equation)
for the  curve $t\mapsto q(t)=(\gamma(t),\hat{\gamma}(t);A(t))$
(i.e., $A(t)=P_0^t(\gamma,\hat{\gamma}) A$)
starting at $q$ can be written as 
\begin{align}\label{eq:2.2:0a}
\di{A^i_j}{t}(t)
-\sum_{k,m} \Gamma^k_{m,j}(\gamma(t)) A^i_k(t) v^m(t)
+\sum_{k,m} \hat{\Gamma}^i_{m,k}(\hat{\gamma}(t)) A^k_j(t) \hat{v}^m(t)=0,
\end{align}
where $t\in I$ and
\[
\dot{\gamma}(t)=\sum_i v^i(t) X_i|_{\gamma(t)},\quad
\dot{\hat{\gamma}}(t)=\sum_i \hat{v}^i(t) \hat{X}_i|_{\hat{\gamma}(t)}.
\]
This shows immediately that the equation for no-spinning
is a linear ODE in $A^i_j$ and thus the solution with the initial condition $A(0)=A_0$ exists for the whole interval $I$ where $\gamma,\hat{\gamma}$ are defined.
The control system $\sns$ can now be written locally in the form
\[
\begin{cases}
\displaystyle \dot{\gamma}(t)=\sum_i v^i(t) X_i|_{\gamma(t)}, \\
\displaystyle \dot{\hat{\gamma}}(t)=\sum_i \hat{v}^i(t) \hat{X}_i|_{\hat{\gamma}(t)}, \\
\displaystyle \di{A^i_j}{t}(t)
=\sum_{k,m} \Gamma^k_{m,j}(\gamma(t)) A^i_k(t) v^m(t)
-\sum_{k,m} \hat{\Gamma}^i_{m,k}(\hat{\gamma}(t)) A^k_j(t) \hat{v}^m(t)
\end{cases},
\quad t\in I,
\]
or
\[
\begin{cases}
\displaystyle \dot{\gamma}(t)=\sum_i v^i(t) X_i|_{\gamma(t)}, \\
\displaystyle \dot{\hat{\gamma}}(t)=\sum_i \hat{v}^i(t) \hat{X}_i|_{\hat{\gamma}(t)}, \\
\displaystyle \dot{A}(t)
=\sum_i v^i\LNSD(X_i, 0)|_{A(t)}+\sum_i \hat{v}^i\LNSD(0, \hat{X}_i)|_{A(t)},
\end{cases}
\quad t\in I,
\]
where the controls $v=(v^1,\dots,v^n)$, $\hat{v}=(\hat{v}^1,\dots,\hat{v}^n)$
are elements of $L^1_{\mathrm{loc}}(I,\R^n)$
(actually an open subset of it since the images of $\gamma,\hat{\gamma}$
should belong to $U,\hat{U}$, respectively).
From this local form, we see that $\sns$ is a driftless control affine system.

The curve $t\mapsto q(t)=(\gamma(t),\hat{\gamma}(t);A(t))$
is a rolling curve i.e., satisfies conditions (\ref{eq:nospin}) and (\ref{eq:noslip})
if and only if
\begin{align}\label{eq:RD_equation_local} 
& \hat{v}^i(t)=\sum_k A^i_k(t) v^k(t), \\
& \di{A^i_j}{t}(t)-\sum_m \Big(\sum_{k}\Gamma^k_{m,j}(\gamma(t)) A^i_k(t)-\sum_{k,l}  \hat{\Gamma}^i_{l,k}(\hat{\gamma}(t)) A^k_j(t) A^l_m(t)\Big)v^m(t)=0,
\end{align}
where
$\dot{\gamma}(t)=\sum_i v^i(t) X_i|_{\gamma(t)}$, $\dot{\hat{\gamma}}(t)=\sum_i \hat{v}^i(t) \hat{X}_i|_{\hat{\gamma}(t)}$.
Supposing that $\hat{U}$ is a domain of a coordinate chart $\hat{\phi}=(\hat{x}^1,\dots,\hat{x}^n)$ of $\hat{M}$
and taking as the frame $\hat{F}$ the coordinate fields $\hat{X}_i=\pa{\hat{x}^i}$,
the previous equation can be written as
\begin{align}\label{simon} 
& \di{\hat{\gamma}^i}{t}(t)=\sum_k A^i_k(t) v^k(t), \\
& \di{A^i_j}{t}(t)=\sum_m \Big(\sum_{k}\Gamma^k_{m,j}(\gamma(t)) A^i_k(t)-\sum_{k,l} {\hat{\Gamma}}^i_{l,k}(\hat{\gamma}(t)) A^k_j(t) A^l_m(t)\Big)v^m(t),
\end{align}
where $\hat{\gamma}^i=\hat{x}^i\circ\hat{\gamma}$ and $t\in I$.
This system is nonlinear in $\hat{\gamma}^i, A^i_j$
and thus the existence of solutions, for a given initial condition $\hat{\gamma}(0)=\hat{x}_0$, $A(0)=A_0$ cannot be guaranteed
on a given compact interval $I\ni 0$ where $\gamma$ is defined
(even in a case where one is able to get $\hat{U}=\hat{M}$).

Moving back to a general frame $\hat{F}$ (i.e., we are not assuming that it consists of coordinate vector fields),
the local form of the control system $\srol$
can be written as
\[
\begin{cases}
\displaystyle \dot{\gamma}(t)=\sum_i v^i(t)X_i|_{\gamma(t)}, \\
\displaystyle \dot{\hat{\gamma}}(t)=\sum_i v^i(t)A(t)X_i|_{\gamma(t)}, \\
\displaystyle \di{A^i_j}{t}(t)=\sum_m \Big(\sum_{k}\Gamma^k_{m,j}(\gamma(t)) A^i_k(t)-\sum_{k,l} \hat{\Gamma}^i_{l,k}(\hat{\gamma}(t)) A^k_j(t) A^l_m(t)\Big)v^m(t),
\end{cases}
\quad t\in I,
\]
or
\[
\begin{cases}
\displaystyle \dot{\gamma}(t)=\sum_i v^i(t)X_i|_{\gamma(t)}, \\
\displaystyle \dot{\hat{\gamma}}(t)=\sum_i v^i(t)A(t)X_i|_{\gamma(t)}, \\
\displaystyle \dot{A}(t)=\sum_{i} v^i(t)\LRD(X_i)|_{A(t)},
\end{cases}
\quad t\in I,
\]
where the controls $v=(v^1,\dots,v^n)$ are elements of $L^1_{\mathrm{loc}}(I,\R^n)$
as above.
From this local form, we see that $\srol$ is a driftless control affine system.

%%%%%%%%%%%%%%%%%%%%%%%%%%%%%%%%%%%%%%%%%%%%%%%%%%%%
\section{Sasaki-metric on $T^*M\otimes T\hat{M}$ and $Q$}\label{sec:2.4}
%%%%%%%%%%%%%%%%%%%%%%%%%%%%%%%%%%%%%%%%%%%%%%%%%%%%

The no-spinning distribution $\NSDist$ can be given a natural sub-Riemannian structure
(see e.g. \cite{montgomery06})
with the sub-Riemannian metric $h_{\mathrm{NSD}}:=(\pi_Q^*(\ol{g}))|_{\NSDist}$
since $(\pi_Q)_*|_{\NSDist|_{(x,\hat{x};A)}}$ is a linear isomorphism $T|_{(x,\hat{x};A)} Q\to T|_{(x,\hat{x})}(M\times\hat{M})$
at each point $q=(x,\hat{x};A)\in Q$ and 
$(\pi_Q)_*(\LNSD(X, \hat{X})|_q)=(X, \hat{X})$,
for every $q=(x,\hat{x};A)\in Q$ and $X\in T|_x M$, $\hat{X}\in T|_{\hat{x}}\hat{M}$.

Actually, we have more since there is a Sasaki-metric $\ol{g}^1_1$ on the whole tensor space $T^*(M)\otimes T(\hat{M})$ given
in the following.

\begin{definition}\label{def:2.4:1}
The \emph{Sasaki-metric} $\ol{g}^1_1$ on $T^*(M)\otimes T(\hat{M})$ is defined by
\begin{align}\label{eq:2.4:1}
\ol{g}^1_1(\xi,\eta)=&\ol{g}\big(\LNSD|_q^{-1}(\pr_1(\xi)),\LNSD|_q^{-1}(\pr_1(\eta))\big)\nonumber\\
&+((g^*\otimes \hat{g})\circ\tau)\big(\nu|_q^{-1}(\pr_2(\xi)),\nu|_q^{-1}(\pr_2(\eta))\big), 
\end{align}
where $$q=(x,\hat{x};A)\in T^*(M)\otimes T(\hat{M}),\  \xi,\eta\in T|_{q} (T^*(M)\otimes T(\hat{M})),$$
$\pr_1$, $\pr_2$ are projections of the decomposition 
$$T(T^*(M)\otimes T(\hat{M}))=\NSDist\oplus V(\pi_{T^*(M)\otimes T(\hat{M})}),$$
(see (\ref{eq:2.1:4}))
onto the first and  second factors, $\LNSD|_q^{-1}$, $\nu|_q^{-1}$
are the inverse maps of 
\begin{align}
T|_{(x,\hat{x})} (M\times\hat{M})&\to T|_q(T^*(M)\otimes T(\hat{M}))\nonumber\\
\ol{X}&\mapsto \LNSD(\ol{X})|_q,
\end{align}
and 
\begin{align}
(T^*(M)\otimes T(\hat{M}))|_{(x,\hat{x})}&\to V|_q(\pi_{T^*(M)\otimes T(\hat{M})})\nonumber\\
B&\mapsto \nu(B)|_q.
\end{align}
Note that $g^*:T^*(M)\otimes T^*(M)\to\R$ is the dual metric induced by $g$
and finally $\tau$ is the $\R$-linear isomorphism
\[
(T^*M\otimes T\hat{M})\otimes (T^* M\otimes T\hat{M})\to (T^*M\otimes T^*M)\otimes (T\hat{M}\otimes T\hat{M})
\]
uniquely determined by
\[
\tau\big((\omega\otimes \hat{X})\otimes (\theta\otimes\hat{Y})\big)=(\omega\otimes \theta)\otimes (\hat{X}\otimes\hat{Y}).
\]
\end{definition}

Denote by $\ol{g}_Q$ the restriction (i.e., the pull-back) of the metric $\ol{g}^1_1$ onto $Q$.

Let us now use the local frames and notation as in Appendix \ref{sec:2.2}.
Writing $\xi,\eta\in T|_{q} (T^*M\otimes T\hat{M})$, $q=(x,\hat{x};A)$, as
\[
\xi&=\sum_i \big(\xi^i X_i|_x+\hat{\xi}^i \hat{X}_i|_{\hat{x}}\big)+\sum_{i,j} {\xi_{\nu}}^i_j \theta^j|_x\otimes \hat{X}_i|_{\hat{x}}, \\ 
\eta&=\sum_i \big(\eta^i X_i|_x+\eta{\xi}^i \hat{X}_i|_{\hat{x}}\big)+\sum_{i,j} {\eta_{\nu}}^i_j\theta^j|_x\otimes \hat{X}_i|_{\hat{x}}, 
\]
one gets
\[
\pr_2(\xi)&=\xi-\LNSD\sum_i \big(\xi^i X_i|_x+\hat{\xi}^i \hat{X}_i|_{\hat{x}}\big)\big|_q\\
&=\sum_{i,j} \Big({\xi_{\nu}}^i_j-\sum_{k,m}\Gamma^k_{m,j} A^i_k \xi^m+{\bf \hat{\Gamma}}^i_{m,k} A^k_j\hat{\xi}^m\Big)\theta^i|_{x}\otimes \hat{X}_j|_{\hat{x}},
\]
the similar formula holding for $\pr_2(\eta)$ and hence
\[
\ol{g}^1_1=&\sum_{i,j} \xi^i \eta^j g(X_i|_x,X_j|_x)+\sum_{i,j} \hat{\xi}^i \hat{\eta}^j g(\hat{X}_i|_{\hat{x}},\hat{X}_j|_{\hat{x}}) \\
&+\sum_{i,j,\alpha,\beta} \Big({\xi_{\nu}}^i_j-\sum_{k,m}\Gamma^k_{m,j} A^i_k \xi^m+{\bf \hat{\Gamma}}^i_{m,k} A^k_j\hat{\xi}^m\Big)\\
&\cdot\Big({\eta_{\nu}}^\alpha_\beta-\sum_{k,m}\Gamma^k_{m,\beta} A^\alpha_k \xi^m+{\bf \hat{\Gamma}}^\alpha_{m,k} A^k_\beta\hat{\xi}^m\Big)g^*(\theta^i,\theta^\alpha)\hat{g}(\hat{X}_j,\hat{X}_\beta).
\]

Moreover, with this choice of the Riemannian metric on $T^*M\otimes T\hat{M}$ and $Q$ we
have the following result.

\begin{proposition}
Let $U,V\in C^\infty(\pi_{T^*M\otimes T\hat{M}},\pi_{T^*M\otimes T\hat{M}})$,
$\ol{X}\in T|_{(x_0,\hat{x}_0)} (M\times \hat{M})$ and $q_0=(x_0,\hat{x}_0;A_0)$.
Then the Sasaki-metric $\ol{g}^1_1$ has the following properties:
\begin{itemize}
\item[(i)] Letting $\tr=\tr_{T|_{x_0}^*M\otimes T|_{x_0}M}$ denote the trace of linear maps $T|_{x_0} M\to T|_{x_0} M$
and $\ol{T}$ the $(g,\hat{g})$-transpose of the linear maps $T|_{x_0} M\to T|_{\hat{x}_0} \hat{M}$, one has
\begin{align}
\ol{g}^1_1|_{q_0}\big(\nu(U(A_0))|_{q_0},\nu(V(A_0))|_{q_0}\big)=\tr\big(U(A_0)^{\ol{T}}V(A_0)\big)
\end{align}

\item[(ii)] Choosing a smooth local $\pi_{T^*M\otimes T\hat{M}}$-section $\tilde{A}$ s.t. $\tilde{A}|_{(x_0,\hat{x}_0)}=A_0$ and $\ol{\nabla}_{\ol{Y}}\tilde{A}=0$
for all $\ol{Y}\in T|_{(x_0,\hat{x}_0)} (M\times\hat{M})$,
\begin{align}
&\LNSD(\ol{X})|_{q_0} \big(\ol{g}^1_1\big(\nu(U(\cdot)),\nu(V(\cdot))\big)\big) \\
=&\ol{g}^1_1(\nu(\ol{\nabla}_{\ol{X}} U(\tilde{A}))|_{q_0},\nu(V(A_0))|_{q_0})+\ol{g}^1_1(\nu(U(A_0))|_{q_0},\nu(\ol{\nabla}_{\ol{X}} V(\tilde{A}))|_{q_0}). \nonumber
\end{align}

\end{itemize}
The same result holds if we throughout replace $T^*M\otimes T\hat{M}$ by $Q$
and $\ol{g}^1_1$ by $\ol{g}_Q$
with the exception that $U,V\in C^\infty(\pi_{Q},\pi_{T^*M\otimes T\hat{M}})$
s.t. $U(A),V(A)\in A\so(T|_x M)$ for all $(x,\hat{x};A)\in Q$.
\end{proposition}

\begin{proof}
Let $(X_i)$, $(\hat{X}_i)$ be smooth $g$,$\hat{g}$-orthonormal frames of vector fields $M$, $\hat{M}$ defined on the neighborhoods $U$, $\hat{U}$ of $x_0$, $\hat{x}_0$. Denote by $(\theta^i)$, $(\hat{\theta}^i)$ the corresponding dual frames.
Then there are unique functions $a^i_j,b^i_j\in\Cinf(\pi_{T^*M\otimes T\hat{M}}^{-1}(U\times \hat{U}))$ s.t.
\[
U(A)=\sum_{i,j} a^i_j(x,\hat{x};A) \theta^j|_x\otimes \hat{X}_i|_{\hat{x}},
\quad
V(A)=\sum_{i,j} b^i_j(x,\hat{x};A) \theta^j|_x\otimes \hat{X}_i|_{\hat{x}},
\]
and thus (below we will denote $a^i_j(x_0,\hat{x}_0;A_0),b^i_j(x_0,\hat{x}_0;A_0)$ simply by $a^i_j$, $b^i_j$)
\[
& \ol{g}^1_1|_{q_0}\big(\nu(U(A_0))|_{q_0},\nu(V(A_0))|_{q_0}\big)\\
=&((g^*\otimes \hat{g})\circ\tau)\big(\sum_{i,j,k,l} a^i_j b^k_l (\theta^j|_{x_0}\otimes \hat{X}_i|_{\hat{x}_0})\otimes (\theta^l|_{x_0}\otimes \hat{X}_k|_{\hat{x}_0})\big) \\
=&\sum_{i,j,k,l} a^i_j b^k_l g^*(\theta^j|_{x_0}\otimes \theta^l|_{x_0})\otimes \hat{g}(\hat{X}_j|_{\hat{x}_0}\otimes \hat{X}_k|_{\hat{x}_0})
=\sum_{i,j,k,l} a^i_j b^k_l \delta^{jl}\delta_{jk} \\
=&\sum_{i,j} a^i_j b^i_j=\sum_{i} \hat{g}(U(A_0)X_i,V(A_0)X_i)=\tr(U(A_0)^{\ol{T}} V(A_0)).
\]
This proves (i).

Next, by the definition of $\LNSD$ and the choice of $\tilde{A}$,
we have
\[
& \LNSD(\ol{X})|_{q_0} \big(\ol{g}^1_1\big(\nu(U(\cdot)),\nu(V(\cdot))\big)\big)
=\ol{X}\ol{g}^1_1\big(\nu(U(\tilde{A}))|_{\tilde{A}},\nu(V(\tilde{A})|_{\tilde{A}})\big) \\
=&\ol{X}\big(\sum_{i,j} a^i_j(\tilde{A}) b^i_j(\tilde{A})\big)
=\sum_{i,j} \ol{X}(a^i_j(\tilde{A})) b^i_j(q_0)+\sum_{i,j} a^i_j(q_0) \ol{X}(b^i_j(\tilde{A})).
\]
Assuming for simplicity that $(X_i)$, $(\hat{X}_i)$ were chosen so that
$\nabla_Y X_i=0$, $\hat{\nabla}_{\hat{Y}} \hat{X}_i=0$ for all $i$ and $Y\in T|_{x_0} M$, $\hat{Y}\in T|_{\hat{x}_0}\hat{M}$
(and hence $\nabla_Y \theta^i=0$ for all $i$ and $Y\in T|_{x_0} M$),
we get
\[
\ol{\nabla}_{\ol{X}}(U(\tilde{A}))=\LNSD(\ol{X})|_{q_0}(a^i_j)\theta^j|_{x_0}\otimes \hat{X}_i|_{\hat{x}_0},\
\ol{\nabla}_{\ol{X}}(V(\tilde{A}))=\LNSD(\ol{X})|_{q_0}(b^i_j)\theta^j|_{x_0}\otimes \hat{X}_i|_{\hat{x}_0}.
\]
This shows that
\[
& \sum_{i,j} \LNSD(\ol{X})|_{q_0}(a^i_j) b^i_j(q_0)+\sum_{i,j} a^i_j(q_0) \LNSD(\ol{X})|_{q_0}(b^i_j) \\
=&\tr\big((\ol{\nabla}_{\ol{X}}(U(\tilde{A})))^{\ol{T}}V(A_0)\big)
+\tr\big(U(A_0)^{\ol{T}}(\ol{\nabla}_{\ol{X}}(V(\tilde{A}))\big) \\
=&\ol{g}^1_1(\nu(\ol{\nabla}_{\ol{X}} U(\tilde{A}))|_{q_0},\nu(V(A_0))|_{q_0})+\ol{g}^1_1(\nu(U(A_0))|_{q_0},\nu(\ol{\nabla}_{\ol{X}} V(\tilde{A}))|_{q_0}).
\]

\end{proof}

\begin{proposition}\label{sas1}
The maps $\pi_{T^*(M)\otimes T(\hat{M})}$ and $\pi_Q$ are surjective Riemannian submersions onto $M\times\hat{M}$.
Hence the restrictions of the Levi-Civita connection $\nabla^{\ol{g}^1_1}$ and the Riemannian curvature $R^{\ol{g}^1_1}$ on the $\NSDist$-horizontal fields are respectively given by
\[
& \nabla^{\ol{g}^1_1}_{\LNSD(\ol{X})} \LNSD(\ol{Y})|_q=\LNSD(\ol{\nabla}_{\ol{X}} \ol{Y})|_q+\frac{1}{2}\nu(AR(X,Y)-\hat{R}(\hat{X},\hat{Y})A)|_q,\]
and
\[
& \ol{g}^1_1\big(R^{\ol{g}^1_1}(\LNSD(\ol{X}),\LNSD(\ol{Y}))\LNSD(\ol{X}),\LNSD(\ol{Y})\big) \\
=&\ol{g}(\ol{R}(\ol{X},\ol{Y})\ol{X},\ol{Y})
+\frac{3}{4}\n{\nu(AR(X,Y)-\hat{R}(\hat{X},\hat{Y})A)|_q}_{\ol{g}^1_1},
\]
with $q=(x,\hat{x};A)\in T^*M\otimes T\hat{M}$, $\ol{X}=(X,\hat{X})$ where $X\in\VF(M),\hat{X}\in\VF(\hat{M})$ and similarly for $\ol{Y}$.
The same formulas hold if one replaces $\ol{g}^1_1$ with $\ol{g}_Q$.
\end{proposition}

\begin{proof}
The first statement is obvious by construction. For the statement about the connection, we use Koszul's formula (cf. \cite{lee02}),
to notice that
\[
\nabla^{\ol{g}^1_1}_{\LNSD(\ol{X})} \LNSD(\ol{Y})=\LNSD(\ol{\nabla}_{\ol{X}} \ol{Y})+\frac{1}{2}[\LNSD(\ol{X}),\LNSD(\ol{Y})]^v,
\]
where for $Z\in T (T^*M\otimes T\hat{M})$ we denote $Z=Z^h+Z^v$ with $Z^h\in\NSDist$ and $Z^v\in V(\pi_{T^*M\otimes T\hat{M}})$.
The fact about the Riemannian curvature is deduced similarly (see \cite{lee02}).
\end{proof}

\begin{theorem}\label{th:2.4:1}
Suppose $t\mapsto q(t)=(x(t),\hat{x}(t);A(t))$ is a smooth curve on $T^*(M)\otimes T(\hat{M})$
that is \emph{$\NSDist$-horizontal} i.e., $\dot{q}(t)\in \NSDist$ for all $t$.
Then the following are equivalent:
\begin{itemize}
\item[(i)] $t\mapsto q(t)=(x(t),\hat{x}(t);A(t))$ is a geodesic of $(T^*(M)\otimes T(\hat{M}),\ol{g}^1_1)$
\item[(ii)] $t\mapsto x(t)$ and $t\mapsto \hat{x}(t)$ are geodesics of $(M,g)$ and $(\hat{M},\hat{g})$ respectively.
\end{itemize}
Moreover, in this case $t\mapsto A(t)$ is given by parallel transport as follows:
\begin{align}\label{eq:parallel_trans_A:II}
A(t)=P^t_0(\hat{x})\circ A(0)\circ P^0_t(x)=P^t_0(x,\hat{x})A(0).
\end{align}
The same facts hold if $(T^*M\otimes T\hat{M},\ol{g}^1_1)$ is replaced by $(Q,\ol{g}_Q)$.
\end{theorem}

Notice that the claim of the theorem can also be written more compactly as follows:
For any $q=(x,\hat{x};A)\in T^*M\otimes T\hat{M}$,
\begin{align}\label{eq:proj_of_exp}
\pi_{T^*M\otimes T\hat{M}}\circ \exp_q^{\ol{g}_1^1}\circ \LNSD|_q=\ol{\exp}_{(x,\hat{x})},
\end{align} 
with a similar formula holding when $T^*M\otimes T\hat{M}$ is replaced by $Q$.
\begin{proof}
This follows from the fact that $\pi_{T^*M\otimes T\hat{M}}$ (resp. $\pi_Q$) is a Riemannian submersion.
Nevertheless, for the sake of convenience we outline the easy proof here.

The fact that $t\mapsto q(t)=(x(t),\hat{x}(t);A(t))$ is $\NSDist$-horizontal implies that
$\dot{q}(t)=\LNSD(\dot{x}(t), \dot{\hat{x}}(t))|_{q(t)}$ for all $t$.
Thus, by Proposition \ref{sas1}, we get
\[
\nabla^{\ol{g}^1_1}_{\dot{q}(t)} \dot{q}=\LNSD\big(\nabla_{\dot{x}(t)}\dot{x}, \hat{\nabla}_{\dot{\hat{x}}(t)}\dot{\hat{x}}\big)\big|_{q(t)},
\]
since $R(\dot{x}(t),\dot{x}(t))=0$, $\hat{R}(\dot{\hat{x}}(t),\dot{\hat{x}}(t))=0$.
The claim follows from this since $\LNSD(\cdot)|_{q(t)}$ is a linear isomorphism for each $t$.
Also, Eq. (\ref{eq:parallel_trans_A:II}) follows easily from the definition of $\NSDist$ and Eq. (\ref{eq:parallel_trans_A}).

\end{proof}

\begin{corollary}\label{cor:2.4:1}
The $\NSDist$-horizontal curve on $Q$ is a geodesic of $(Q,\ol{g}_Q)$ if
and only if it is a geodesic of $(T^*(M)\otimes T(\hat{M}),\ol{g}^1_1)$.
\end{corollary}

\begin{theorem}\label{th:2.4:2}
The Riemannian manifolds $(M,g)$, $(\hat{M},\hat{g})$ are complete Riemannian manifolds if and only if $(T^*(M)\otimes T(\hat{M}),\ol{g}^1_1)$ or $(Q,\ol{g}_Q)$ is complete.
\end{theorem}

\begin{proof}
The completeness of $(T^*(M)\otimes T(\hat{M}),\ol{g}^1_1)$ (resp. $(Q,\ol{g}_Q)$)
implies the completeness of $(M,g)$, $(\hat{M},\hat{g})$
since $\pi_{T^*(M)\otimes T(\hat{M})}$ (resp. $\pi_Q$) is a Riemannian submersion onto $(M\times\hat{M},\ol{g})$
and $\pr_1$, $\pr_2$ are Riemannian submersions from $M\times\hat{M}$ onto $M$ and $\hat{M}$, respectively
(recall that Riemannian submersions map geodesics to geodesics).
This proves the direction ``$\Leftarrow$''.

Thus assume that $(M,g)$, $(\hat{M},\hat{g})$ are complete, which is equivalent to the completeness of $(M\times\hat{M},\ol{g})$.
Let $(x_n,\hat{x}_n;A_n)\in T^*(M)\otimes T(\hat{M})$ be a Cauchy-sequence.
Then $(x_n,\hat{x}_n)$ is a Cauchy sequence in $M\times\hat{M}$ and hence converges to a point $(y,\hat{y})\in M\times\hat{M}$ since $M\times\hat{M}$ is a complete (metric) space.
Choose a local trivialization $\tau:\pi_{T^*(M)\times T(\hat{M})}^{-1}(U\times\hat{U})\to (U\times\hat{U})\times \gl(n)$ of $T^*(M)\times T(\hat{M})$
induced by some coordinate charts $((x^i),U)$, $((\hat{x}^i),\hat{U})$ (see Appendix
 \ref{sec:2.2})
of $M$, $\hat{M}$ around $y$, $\hat{y}$ respectively.
By Proposition II.1.1 in \cite{sakai91},
the metric $d_{\ol{g}^1_1}$ induced by $\ol{g}^1_1$ on $T^*(M)\otimes T(\hat{M})$, and hence on $(U\times\hat{U})\times \gl(n)$ through $\tau$,
gives the original manifold topology. Choose an open neighbourhood $V\times\hat{V}$ of $(y,\hat{y})$
such that $\ol{V}\times \ol{\hat{V}}$ is a compact subset of $U\times\hat{U}$.
Then $(\ol{V}\times\ol{\hat{V}})\times \gl(n)$ is a complete space that contains a Cauchy-sequence $((x_n,\hat{x}_n),a_n)=\tau(x_n,\hat{x}_n;A_n)$
for $n$ large enough. Hence it converges to $(y,\hat{y},a)\in (\ol{V}\times\ol{\hat{V}})\times \gl(n)$
and thus $(x_n,\hat{x}_n;A_n)$ converges to $\tau^{-1}(y,\hat{y},a)=(y,\hat{y};A)$. This proves the completeness of $T^*(M)\otimes T(M)$.
We thus get the completeness of $Q$ since $Q$ is a closed submanifold of $T^*(M)\otimes T(M)$.

\end{proof}

\begin{remark}\label{re:completeness_of_vector_bundle}
More generally, suppose $\pi:E\to N$ is a smooth bundle,
$(E,G)$, $(N,g)$ are Riemannian manifolds, $\pi$ is a Riemannian submersion
and the typical fiber of $\pi$ is complete (i.e., all the fibers $\pi^{-1}(x)$ are complete subsets of $E$).
Then the argument of the previous proof applies
and shows that $E$ is a complete Riemannian manifold if and only if $M$ is a complete Riemannian manifold.
\end{remark}

We record the following result.

\begin{proposition}\label{pr:2.5:2}
Let $N$ be an integral manifold of $\RDist$
and equip it with the Riemannian metric $\ol{g}_N:=\ol{g}^1_1|_N$.
Then $(N,\ol{g}_N)$ is a totally geodesic submanifold of $(T^*(M)\otimes T(\hat{M}),\ol{g}^1_1)$.

The same claim holds if one replaces $(T^*M\otimes T\hat{M},\ol{g}^1_1)$ by $(Q,\ol{g}_Q)$ and assumes that $N\subset Q$.
\end{proposition}

\begin{proof}
The assumptions immediately imply that the projection 
$$\pi_N:=\pr_1\circ \pi_{T^*(M)\otimes T(\hat{M})}|_N,$$
is a local isometric diffeomorphism
from $(N,\ol{g}_N)$ into $(M,g)$
since $\pr_1\circ \pi_{T^*(M)\otimes T(\hat{M})}$ maps $\RDist$ isometrically onto $TM$ by the definition of $\ol{g}^1_1$ and $\RDist$.

Now if $t\mapsto (x(t),\hat{x}(t);A(t))$, $t\in ]a,b[$,
is a geodesic of $N$ then (since it is tangent to $\RDist$)
we have $\dot{\hat{x}}(t)=A(t)\dot{x}(t)$
and $t\mapsto x(t)=\pi_N(x(t),\hat{x}(t);A(t))$, $t\in ]a,b[$, is a geodesic of $(M,g)$.
We have
\[
\hat{\nabla}_{\dot{\hat{x}}(t)}\dot{\hat{x}}(t)
=\hat{\nabla}_{\dot{\hat{x}}(t)} (A(\cdot)\dot{x})
=(\ol{\nabla}_{(\dot{x}(t),\dot{\hat{x}}(t))} A)\dot{x}(t)+A(t)\nabla_{\dot{x}(t)}\dot{x},
\]
and once we use the facts that $\nabla_{\dot{x}(t)} \dot{x}=0$ (since $x$ is a geodesic on $M$) and
$\ol{\nabla}_{(\dot{x}(t),\dot{\hat{x}}(t))} A=0$ (by the definition of $\NSDist$)
to conclude that $\hat{\nabla}_{\dot{\hat{x}}(t)}\dot{\hat{x}}(t)=0$
i.e., $t\mapsto \hat{x}(t)$, $t\in ]a,b[$, is 
a geodesic of $\hat{M}$. Thus Theorem \ref{th:2.4:1}
implies that $t\mapsto (x(t),\hat{x}(t);A(t))$ is a ($\RDist$-horizontal)
geodesic of $(T^*(M)\otimes T(\hat{M}),\ol{g}^1_1)$.
The proof is complete.
\end{proof}

%%%%%%%%%%%%%%%%%%%%%%%%%%%%%%%%%%%%%%%%%%%%
\section{The Rolling Problem Embedded in $\R^N$}
%%%%%%%%%%%%%%%%%%%%%%%%%%%%%%%%%%%%%%%%%%%%

In this section, we compare the rolling model
defined by the state space $Q=Q(M,\hat{M})$,
whose dynamics is governed by the conditions (\ref{eq:nospin2})-(\ref{eq:noslip}) (or, equivalently, by $\RDist$),
with the rolling model of two $n$-dimensional manifolds embedded in $\R^N$ as given in \cite{sharpe97} (Appendix B). See also \cite {norway}, \cite{huper07}. 

Let us first fix $N\in\N$ and introduce some notations. The special Euclidean group of $\R^N$ is the set $\Euc(N):=\mathrm{SO}(N)\times \R^N$
equipped with the group operation $\star$ 
given by 
\[
(p,A)\star (q,B)=(Aq+p,AB),\quad (p,A),(q,B)\in \Euc(N).
\]
We identify $\SO(N)$ with the subgroup $\{0\}\times \SO(N)$
of $\Euc(N)$, while $\R^N$ is identified with the normal subgroup $\R^N\times \{\id_{\R^N}\}$ of $\Euc(N)$.
With these identifications, the action $\star$ of the subgroup $\SO(N)$ on the normal subgroup $\R^N$
is given by
\[
(p,A)\star q=Aq+p,\quad (p,A)\in\Euc(N),\ p\in\R^N.
\]

Let $\mc{M}$ and $\hat{\mc{M}}\subset \R^N$ be two (embedded) submanifolds of dimension $n$.
For every $z\in \mc{M}$, we identify $T|_z \mc{M}$ with a subspace of $\R^N$ (the same holding in the case of $\hat{\mc{M}}$) i.e., elements of $T|_z \mc{M}$
are derivatives $\dot{\sigma}(0)$ of curves $\sigma:I\to M$ with $\sigma(0)=z$ ($I\ni 0$ a nontrivial real interval).

The \emph{rolling of $\mc{M}$ against $\hat{\mc{M}}$ without slipping or twisting} in the sense of \cite{sharpe97} 
is realized by a smooth curves $G:I\to \Euc(N)$; $G(t)=(p(t),U(t))$ ($I$ a nontrivial real interval) called the \emph{rolling map}
and $\sigma:I\to \mc{M}$ called the \emph{development} curve
such that the following conditions (1)-(3) hold for every $t\in I$: 
\begin{itemize}
\item[(1)] (a) $\hat{\sigma}(t):=G(t)\star\sigma(t)\in \hat{\mc{M}}$ and \\
(b) $T|_{\hat{\sigma}(t)} (G(t)\star\mc{M})=T|_{\hat{\sigma}(t)} \hat{\mc{M}}$.

\item[(2)] No-slip: $\dot{G}(t)\star\sigma(t)=0$.

\item[(3)] No-twist: (a) $\dot{U}(t)U(t)^{-1} T|_{\hat{\sigma}(t)}\hat{\mc{M}}\subset (T|_{\hat{\sigma}(t)}\hat{\mc{M}})^{\perp}$ (tangential no-twist), \\
(b) $\dot{U}(t)U(t)^{-1} (T|_{\hat{\sigma}(t)}\hat{\mc{M}})^\perp \subset T|_{\hat{\sigma}(t)}\hat{\mc{M}}$ (normal no-twist).
\end{itemize} 
The orthogonal complements are taken w.r.t. the Euclidean inner product of $\R^N$.
In condition (2) we define the action '$\star$' of $\dot{G}(t)=(\dot{U}(t),\dot{p}(t))$
on $\R^N$ by the same formula as for the action '$\star$' of $\Euc(N)$ on $\R^N$.

The two manifolds $M$ and $\hat{M}$ are embedded inside $\R^N$ by embeddings $\iota:M\to\R^N$ and $\hat{\iota}:\hat{M}\to\R^N$
and their metrics $g$ and $\hat{g}$ are induced from the Euclidean metric $s_N$ of $\R^N$ i.e., $g=\iota^*s_N$ and $\hat{g}=\hat{\iota}^* s_N$.
In the above setting, we take now $\mc{M}=\iota(M)$, $\hat{\mc{M}}=\hat{\iota}(\hat{M})$.

For $z\in\mc{M}$ and $\hat{z}\in\hat{\mc{M}}$, consider the linear orthogonal projections
$$
P^T:T|_z\R^N\to T|_z\mc{M}\hbox{ and }P^\perp:T|_z\R^N\to T|_z\mc{M}^\perp,$$
and 
$$
\hat{P}^T:T|_{\hat{z}}\R^N\to T|_{\hat{z}}\hat{\mc{M}}\hbox{ and }\hat{P}^\perp:T|_{\hat{z}}\R^N\to T|_{\hat{z}}\hat{\mc{M}}^\perp,$$ 
respectively.

For $X\in T|_z \R^N$ and $Y\in \Gamma(\pi_{T\R^N}|_{\mc{M}})$ (here $\pi_{T\R^N}|_{\mc{M}}$ is the pull-back bundle of $T\R^N$ over $\mc{M}$),
we use  $\nabla^\perp_{X} Y$ to denote $P^\perp(\nabla^{s_N}_X Y)$
and one proceeds similarly $\hat{\nabla}^\perp_{\hat{X}} \hat{Y}=\hat{P}^\perp(\nabla^{s_N}_{\hat{X}} \hat{Y})$ 
for $\hat{X}\in T|_{\hat{z}} \R^N$ and $Y\in \Gamma(\pi_{T\R^N}|_{\hat{\mc{M}}})$.
We notice that, for any $z\in\mc{M}$, $X\in T|_z \mc{M}$ and $Y\in\VF(\mc{M})$, we
have 
$$\nabla^{s_N}_X Y=\iota_*(\nabla_{\iota^{-1}_* (X)} \iota_*^{-1}(Y))+\nabla^\perp_{X} Y,$$
and similarly on $\hat{\mc{M}}$.

Notice that $\nabla^\perp$ and $\hat{\nabla}^\perp$ determine (by restriction) connections
of vector bundles $\pi_{T\mc{M}^\perp}:T\mc{M}^\perp\to\mc{M}$ and $\pi_{T\hat{\mc{M}}^\perp}:T\hat{\mc{M}}^\perp\to\hat{\mc{M}}$.
These connections can then be used in an obvious way to determine a connection
$\ol{\nabla}^\perp$ on the vector bundle 
$$
\pi_{(T\mc{M}^\perp)^*\otimes T\mc{M}^\perp}:(T\mc{M}^\perp)^*\otimes T\mc{M}^\perp\to \mc{M}\times\hat{\mc{M}}.$$

Let us take any rolling map $G:I\to\Euc(N)$, $G(t)=(p(t),U(t))$ and development curve $\sigma:I\to\mc{M}$
and define $x=\iota^{-1}\circ \sigma$. We will go throught the meaning of each of the above conditions (1)-(3).

\begin{itemize}
\item[(1)] (a) Since $\hat{\sigma}(t)\in\hat{\mc{M}}$, we may define a smooth curve $\hat{x}:=\hat{\iota}^{-1}\circ\hat{\sigma}$. \\
(b) One easily sees that 
$$
U(t)T|_{\hat{\sigma}(t)} \mc{M}=T|_{\hat{\sigma}(t)} (G(t)\star\mc{M})=T|_{\hat{\sigma}(t)} \hat{\mc{M}}.
$$
Thus $A(t):=\hat{\iota}^{-1}_*\circ U(t)\circ \iota_*|_{T|_{x(t)} M}$ defines a map $T|_{x(t)} M\to T|_{\hat{x}(t)}\hat{M}$,
which is also orthogonal i.e., $A(t)\in Q|_{(x(t),\hat{x}(t))}$ for all $t$.
Moreover, if $B(t):=U(t)|_{T|_{\sigma(t)} \mc{M}^\perp}$, then $B(t)$ is a map $T|_{\sigma(t)} \mc{M}^\perp\to T|_{\hat{\sigma}(t)} \mc{\hat{M}}^\perp$
and, by a slight abuse of notation, we can write $U(t)=A(t)\oplus B(t)$.

Thus Condition (1) just determines a smooth curve $t\mapsto (x(t),\hat{x}(t);A(t))$ inside the state space $Q=Q(M,\hat{M})$.

\item[(2)] We compute 
\[
0=&\dot{G}(t)\star\sigma(t)
=\dot{U}(t)\sigma(t)+\dot{p}(t) \\
=&\dif{t}(G(t)\star\sigma(t))-U(t)\dot{\sigma}(t)=\dot{\hat{\sigma}}(t)-U(t)\circ\iota_*\circ\iota^{-1}_*\circ\dot{\sigma}(t),
\]
which, once composed with $\hat{\iota}^{-1}_*$ from the left, gives $0=\dot{\hat{x}}(t)-A(t)\dot{x}(t)$.
This is exactly the no-slip condition, Eq. (\ref{eq:noslip}).

\item[(3)] Notice that, on $\R^N\times\R^N=\R^{2N}$, the sum metric $s_N\oplus s_N$ is just $s_{2N}$.
Moreover, if $\gamma:I\to\R^N$ is a smooth curve, then smooth vector fields $X:I\to T(\R^N)$ along $\gamma$
can be identified with smooth maps $X:I\to\R^N$
and with this observation one has: $\dot{X}(t)=\nabla^{s_N}_{\dot{\gamma}(t)} X$.

(a) Since $U(t)=A(t)\oplus B(t)$, we get, for $t\mapsto \hat{X}(t)\in T|_{\hat{\sigma}(t)} \hat{\mc{M}}$, that
\[
& \dot{U}(t)U(t)^{-1}\hat{X}(t)=\nabla^{s_{2N}}_{(\dot{\sigma},\dot{\hat{\sigma}})(t)} \hat{X}(\cdot)-U(t)\nabla^{s_{2N}}_{(\dot{\sigma},\dot{\hat{\sigma}})(t)} (U(\cdot)^{-1}\hat{X}(\cdot)) \\
=&P^T\big(\hat{\nabla}^{s_N}_{\dot{\hat{\sigma}}(t)} \hat{X}(\cdot)\big)+\hat{\nabla}^\perp_{\dot{\hat{\sigma}}(t)} \hat{X}(\cdot)\\
&-U(t)\big(P^T\big(\nabla^{s_N}_{\dot{\sigma}(t)} (A(\cdot)^{-1}\hat{X}(\cdot))\big)+\nabla^{\perp}_{\dot{\sigma}(t)} (A(\cdot)^{-1}\hat{X}(\cdot))\big) \\
=&\big(\ol{\nabla}_{(\dot{x},\dot{\hat{x}})(t)} A(\cdot)\big)A(t)^{-1}(\hat{\iota}_*^{-1}\hat{X}(t))
+\big(\hat{\nabla}^\perp_{\dot{\hat{\sigma}}(t)} \hat{X}(\cdot)-B(t)\nabla^{\perp}_{\dot{\sigma}(t)} (A(\cdot)^{-1}\hat{X}(\cdot))\big),
\]
from which it is clear that the tangential no-twist condition corresponds to the condition that $\ol{\nabla}_{(\dot{x}(t),\dot{\hat{x}}(t))} A(\cdot)=0$.
This means exactly that $t\mapsto (x(t),\hat{x}(t);A(t))$ is tangent to $\NSDist$ for all $t\in I$.
Thus, the tangential no-twist condition (3)-(a) is equivalent to the no-spinning condition, Eq. (\ref{eq:nospin}). 

(b) Choose $t\mapsto \hat{X}^\perp(t)\in T|_{\hat{\sigma}(t)} \hat{\mc{M}}^\perp$ and calculate as above 
\[
\dot{U}(t)U(t)^{-1}\hat{X}^\perp(t) 
&=P^T(\nabla^{s_N}_{\dot{\hat{\sigma}}(t)} \hat{X}^\perp(\cdot))+\hat{\nabla}^\perp{\dot{\hat{\sigma}}(t)}\\
&-U(t)\big(P^T\big(\nabla^{s_N}_{\dot{\sigma}(t)} (B(\cdot)^{-1}\hat{X}(\cdot))\big)+\nabla^{\perp}_{\dot{\sigma}(t)} (B(\cdot)^{-1}\hat{X}(\cdot))\big) \\
=&\Big(P^T(\nabla^{s_N}_{\dot{\hat{\sigma}}(t)} \hat{X}^\perp(\cdot)-A(t) P^T\big(\nabla^{s_N}_{\dot{\sigma}(t)} (B(\cdot)^{-1}\hat{X}(\cdot))\big)\Big)\\
&+\big(\ol{\nabla}^{\perp}_{(\dot{\sigma}(t),\dot{\hat{\sigma}}(t))} B(\cdot) \big) B(t)^{-1}\hat{X}(t),
\]
and hence we see that the normal no-twist condition (3)-(b) corresponds to the condition that
\[
\ol{\nabla}^\perp_{(\sigma(t),\hat{\sigma}(t))} B(\cdot)=0,\quad \forall t.
\]
In a similar spirit to how Definition \ref{def:2:2} was given, one easily sees that this condition
just amounts to say that $B$ maps parallel translated normal vectors to $\mc{M}$ to parallel translated normal vectors to $\hat{\mc{M}}$. More precisely, if $X_0\in T\mc{M}^\perp$ and $X(t)=(P^{\nabla^\perp})^t_0(\sigma) X_0$ is a parallel translate of $X_0$ along $\sigma$ w.r.t. to the connection $\nabla^\perp$
(notice that $X(t)\in T|_{\sigma(t)} \mc{M}^\perp$ for all $t$),
then the normal no-twist condition (3)-(b) requires that $t\mapsto B(t)X(t)$ (which is the same as $U(t)X(t)$)
is parallel to $t\mapsto \hat{\sigma}(t)$ w.r.t the connection $\hat{\nabla}^\perp$
i.e., for all $t$,
$$
B(t)((P^{\nabla^\perp})^t_0(\sigma)X_0)=(P^{\hat{\nabla}^\perp})^t_0(\hat{\sigma})(B(0)X_0).
$$

\end{itemize}

We formulate the preceding remarks to a proposition.

\begin{proposition}\label{pr:equivalent_models}
Let $\iota:M\to\R^N$ and $\hat{\iota}:\hat{M}\to\R^N$ be smooth embeddings and let $g=\iota^*(s_N)$ and $\hat{g}=\hat{\iota}^*(s_N)$.
Fix points $x_0\in M$, $\hat{x}_0\in\hat{M}$ and an element $B_0\in \SO(T|_{\iota(x_0)} \mc{M}^\perp,T|_{\hat{\iota}(\hat{x}_0)} \hat{\mc{M}}^\perp)$.
Then, there is a bijective correspondence between the smooth curves $t\mapsto (x(t),\hat{x}(t);A(t))$ of $Q$
tangent to $\NSDist$ (resp. $\RDist$), satisfying $(x(0),\hat{x}(0))=(x_0,\hat{x}_0)$
and the pairs of smooth curves $t\mapsto G(t)=(p(t),U(t))$ of $\Euc(N)$ and $t\mapsto \sigma(t)$ of $\mc{M}$ which satisfy the conditions (1), (3) (resp. (1),(2),(3) i.e., rolling maps)
and $U(0)|_{T|_{\sigma(0)} \mc{M}^\perp}=B_0$.
\end{proposition}

\begin{proof}
Let $t\mapsto q(t)=(x(t),\hat{x}(t);A(t))$ to be a smooth curve in $Q$ such that $(x(0),\hat{x}(0))=(x_0,\hat{x}_0)$.
Denote $\sigma=\iota\circ x$, $\hat{\sigma}=\hat{\iota}\circ\hat{x}$ and let $B(t)=(P^{\ol{\nabla}^\perp})_0^t((\sigma,\hat{\sigma}))B_0$
be the parallel translate of $B_0$ along $t\mapsto (\sigma(t),\hat{\sigma}(t))$ w.r.t the connection $\ol{\nabla}^\perp$.
We define 
$$
U(t):=(\hat{\iota}_*\circ A(t)\circ \iota^{-1}_*)\oplus B(t):T|_{\sigma(t)}\mc{M}\to T|_{\hat{\sigma}(t)}\hat{\mc{M}},
$$
and $p(t)=\hat{\sigma}(t)-U(t)\sigma(t)$. Then, by the above remarks, the smooth curve $t\mapsto G(t)=(p(t),U(t))$ 
satisfies Conditions (1),(3) (resp. (1),(2),(3)) if $t\mapsto q(t)$ is tangent to 
$\NSDist$ (resp. $\RDist$).
This clearly gives the claimed bijective correspondence.
\end{proof}

%%%%%%%%%%%%%%%%%%%%%%%%%%%%%%%%%%%%%%%%%%%%%%%%%%%
\section{Special Manifolds in 3D Riemanniann Geometry}
\subsection{Preliminaries}\label{app:3D}
%%%%%%%%%%%%%%%%%%%%%%%%%%%%%%%%%%%%%%%%%%%%%%%%%%%

On an oriented Riemannian manifold $(M,g)$ one defines the Hodge-dual $\star_M$ as
the linear map uniquely defined by
\[
\star_M:\wedge^k T|_x M\to \wedge^{n-k} T|_x M;
\quad \star_M (X_1\wedge\dots\wedge X_k)=X_{k+1}\wedge\dots\wedge X_n,
\]
with $x\in M$, $k=0,\dots,n=\dim M$
and $X_1,\dots,X_n\in T|_x M$ any oriented basis.

For an oriented Riemannian manifold $(M,g)$ and $x\in M$,
one defines $\so(T|_x M)$ as the set of $g$-antisymmetric linear maps $T|_x M\to T|_x M$.
Moreover, we write $\so(M)$ as disjoint union of $\so(T|_x M)$, $x\in M$.
If $A,B\in \so(T|_x M)$, we define
\[
[A,B]_{\so}:=A\circ B-B\circ A\in \so(T|_x M).
\]
Also, we define the following natural isomorphism $\phi$ by
\[
\phi:{\wedge}^2 TM\to \so(M);\quad \phi(X\wedge Y):=g(\cdot,X)Y-g(\cdot,Y)X.
\]
Using this isomorphism, we may consider, for each $x\in M$, the curvature tensor $R$ of $(M,g)$ 
as a linear map,
\[
\mc{R}:\wedge^2 T|_x M \to \wedge^2 T|_x M;
\quad \mc{R}(X\wedge Y):=\phi^{-1}(R(X,Y)),
\]
where $X,Y\in T|_x M$. Here of course $R(X,Y)$, as an element of $T^*|_x M\otimes T|_x M$,
belongs to $\so(T|_x M)$.
It is a standard fact that $\mc{R}$ is a symmetric map
when $\wedge^2 T|_x M$ is endowed with the inner product, also written as $g$,
\[
g(X\wedge Y,Z\wedge W):=g(X,Z)g(Y,W)-g(X,W)g(Y,Z).
\]
Notice also that for $A,B\in \so(T|_x M)$,
\[
\tr(AB)=g(\phi^{-1}(A),\phi^{-1}(B)).
\]
The map $\mc{R}$ is usually called the curvature operator
and we will, with a slight abuse of notation,
write it simply as $R$.

In dimension $\dim M=3$
one has $\star_M^2=\id$ when $\star_M$ is the map
$\wedge^2 TM\to TM$ and $T M\to \wedge^2 TM$.

Let $X,Y,Z\in T|_x M$ be an orthonormal positively oriented basis.
Then
\[
\star_M (X\wedge Y)=Z,\quad
\star_M (Y\wedge Z)=X,\quad
\star_M (Z\wedge X)=Y.
\]
In terms of this basis $X,Y,Z$ one has
\[
\star_M \phi^{-1}\qmatrix{
0 & -\alpha & \beta \cr
\alpha & 0 & -\gamma \cr
-\beta & \gamma & 0
}=\qmatrix{\gamma \cr \beta \cr \alpha}.
\]
Indeed, since $X,Y,Z\in T|_x M$ form an orthonormal positively oriented basis, then
\[
\phi\Big(\star_M\qmatrix{\gamma \cr \beta \cr \alpha}\Big)
=&\phi\big(\gamma (\star_M X)+\beta (\star_M Y)+\gamma (\star_M Z)\big) 
=\phi\Big(
\gamma Y\wedge Z + \beta Z\wedge X + \alpha X\wedge Y
\Big) \\
=& \qmatrix{
0 & -\alpha & \beta \cr
\alpha & 0 & -\gamma \cr
-\beta & \gamma & 0
}.
\]

\begin{lemma}\label{le:speciality_of_3D}
If $(M,g)$ is a 3-dimensional oriented Riemannian manifold and $x\in M$.
\begin{itemize}
\item[(i)] Then each 2-vector $\xi\in \wedge^2 T|_x M$ is pure i.e.
there exist $X,Y\in T|_x M$ such that $\xi=X\wedge Y$.
\item[(ii)] For every $X,Y\in T|_x M$ one has
\[
[\phi(\star_M X),\phi(\star_M Y)]_{\so}=\phi(X\wedge Y).
\]
\end{itemize}
\end{lemma}

\begin{proof}
(i) To see this, it is enough to take $X,Y$ such that $X,Y,\star_M \xi$ are orthonormal in $T|_x M$
and form a positively oriented basis
and that $\n{X}_g=\n{\star_M\xi}_g$, $\n{Y}_g=1$.
Then
\[
g(\star_M (X\wedge Y),X)=\star_M (X\wedge Y\wedge X)=0,
\quad g(\star_M (X\wedge Y),Y)=\star_M (X\wedge Y\wedge Y)=0,
\]
so $X,Y$ are orthogonal to $\star_M (X\wedge Y)$ and to $\star_M \xi$,
hence $\star_M (X\wedge Y)$ is parallel to $\star_M \xi$
i.e., $X\wedge Y$ is parallel to $\xi$. Since
$\n{X\wedge Y}_g=\n{X}_g\n{Y}_g=\n{\xi}_g$
and taking into account the assumption that  $X,Y,\star_M \xi$
and $X,Y,\star_M (X\wedge Y)$ are oriented basis
(in that order), it follows that $X\wedge Y=\xi$.

(ii) If $X,Y,Z\in T|_x M$ form an orthonormal basis and
if $U,V\in T|_x M$, then with respect to the basis $X,Y,Z$,
\[
U=&\alpha X+\beta Y+\gamma Z=\qmatrix{\gamma\cr \beta\cr \alpha},\quad
V=a X+b Y+cZ=\qmatrix{c\cr b\cr a} \\
U\wedge V=&\qmatrix{\gamma\cr \beta\cr \alpha}\wedge \qmatrix{c\cr b\cr a}
=\star_M\qmatrix{\beta a-\alpha b \cr -\gamma a+\alpha c \cr \gamma b-\beta c} \\
=&\phi^{-1}\qmatrix{
0 & -\gamma b+\beta c & -\gamma a+\alpha c \cr
\gamma b-\beta c & 0 & -\beta a+\alpha b \cr
\gamma a-\alpha c & \beta a-\alpha b & 0
} \\
[\phi(\star_M U),\phi(\star_M V)]_{\so}
=&\Big[
\qmatrix{
0 & -\alpha & \beta \cr
\alpha & 0 & -\gamma \cr
-\beta & \gamma & 0
},
\qmatrix{
0 & -a & b \cr
a & 0 & -c \cr
-b & c & 0
}
\Big]_{\so} \\
=&\qmatrix{
0 & \beta c-\gamma b & \alpha c-\gamma a \cr
-(\beta c-\gamma b) & 0 & \alpha b-\beta a\cr
-(\alpha c-\gamma a) & -(\alpha b-\beta a) & 0
}.
\]
\end{proof}

%%%%%%%%%%%%%%%%%%%%%%%%%%%%%%%%%%%%%%%
\subsection{Manifolds of class $M_{\beta}$}\label{app:m_beta}
%%%%%%%%%%%%%%%%%%%%%%%%%%%%%%%%%%%%%%%

In this subsection, we define and investigate some properties
of special type of 3-dimensional manifolds.

Following the paper \cite{agrachev10} we make the following definition.

\begin{definition}
A 3-dimensional manifold $M$ is called a \emph{contact manifold of type $(\kappa,0)$}
where $\kappa\in\Cinf(M)$ if
there are everywhere linearly independent vector fields $F_1,F_2,F_3\in\VF(M)$
and smooth functions $c,\gamma_1,\gamma_3\in\Cinf(M)$
such that
\[
[F_1,F_2]=&c F_3 \\
[F_2,F_3]=&c F_1 \\
[F_3,F_1]=&-\gamma_1F_1+F_2-\gamma_3 F_3
\]
and
\[
-\kappa=F_3(\gamma_1)-F_1(\gamma_3)+(\gamma_1)^2+(\gamma_3)^2-c.
\]
We call the frame $F_1,F_2,F_3$ an \emph{(normalized) adapted frame of $M$}
and $c,\gamma_1,\gamma_2$ the corresponding structure functions.
\end{definition}

\begin{remark}
\begin{itemize}
\item[(i)]
If $(M,g)$ is as above,
then defining $\lambda\in\Gamma(\pi_{T^*M})$ by
\[
\lambda(F_1)=\lambda(F_3)=0,\quad
\lambda(F_2)=1
\]
one sees that $\lambda$ is a contact form on $(M,g)$
and $F_2$ is its Reeb vector field. Indeed, if $X\in\VF(M)$,
write $X=a_1F_1+a_2F_2+a_3F_3$ and compute
\[
\diff\lambda(F_2,X)=&X(\lambda(F_2))-F_2(\lambda(X))-\lambda([F_2,X]) \\
=&0-F_2(a_2)-\big(F_2(a_2)+\lambda(a_1[F_2,F_1]+a_3[F_2,F_3])\big)=0.
\]

\item[(ii)]
Contact manifolds in 3D are essentially classified by two functions $\kappa,\chi$
defined on these manifolds. Thus one could say in general that
a contact manifold is of class $(\kappa,\chi)$.
We are interested here only in the case where $\chi=0$.
For information on the classification
of contact manifolds, definition of $\chi$ and references, see \cite{agrachev10}.
\end{itemize}
\end{remark}

One may define on such a manifold a Riemannian metric
in a natural way by declaring $F_1,F_2,F_3$ orthogonal.
The structure of connections coefficients and the eigenvalues of the corresponding
curvature tensor are given in the following lemma.

\begin{lemma}\label{le:mbeta-1}
Let $M$ be a contact manifold of type $(\kappa,0)$
with adapted frame $F_1,F_2,F_3$ and structure functions $c,\gamma_1,\gamma_2$.
If $g$ is the unique Riemannian metric which makes $F_1,F_2,F_3$
orthonormal,
then the connection table w.r.t. $F_1,F_2,F_3$ is
\[
\Gamma=\qmatrix{
\frac{1}{2} & 0 & 0 \\
\gamma_1 & c-\frac{1}{2} & \gamma_3 \\
0 & 0 & \frac{1}{2}
},
\]

Moreover, at each point, $\star F_1,\star F_2,\star F_3$ (with $\star$ the Hodge dual)
are eigenvectors of the curvature tensor $R$
with eigenvalues $-K,-K_2(\cdot),-K$, respectively, where
\[
K=&\frac{1}{4},\quad (\textrm{constant}) \\
K_2(x)=&\kappa(x)-\frac{3}{4},\quad x\in M.
\]
\end{lemma}

\begin{proof}
Since $F_1,F_2,F_3$ are $g$-orthonormal, the Koszul formula simplifies to
\[
2\Gamma^i_{(j,k)}=2g(\nabla_{F_i} F_j,F_k)
=g([F_i,F_j],F_k)-g([F_i,F_k],F_j)-g([F_j,F_k],F_i).
\]
With this we have
\[
2\Gamma^1_{(2,3)}&=c+1-c=-1,\quad
2\Gamma^1_{(3,1)}=\gamma_1-0+\gamma_1=2\gamma_1,\quad
2\Gamma^1_{(1,2)}=0 \\
2\Gamma^2_{(2,3)}&=0,\quad
2\Gamma^2_{(3,1)}=c+c-1=2c-1,\quad
2\Gamma^2_{(1,2)}=0 \\
2\Gamma^3_{(2,3)}=&0,\quad
2\Gamma^3_{(3,1)}=0-2(-\gamma_3)=2\gamma_3,\quad
2\Gamma^3_{(1,2)}=1+c-c=1.
\]

Since $\Gamma^1_{(2,3)}=\Gamma^3_{(1,2)}=\frac{1}{2}$
and $\Gamma^1_{(1,2)}=0=-\Gamma^3_{(2,3)}$
and they are constants,
the conditions of Proposition \ref{pr:special_3D} are fulfilled
and hence $\star F_1,\star F_2,\star F_3$ are eigenvectors
of $R$ with eigenvalues $-K,-K_2,-K$ where
\[
-K=&F_2(0)+0-\big(\frac{1}{2}\big)^2=-\frac{1}{4} \\
-K_2=&\big(\frac{1}{2}\big)^2+F_3(\gamma_1)-F_1(\gamma_3)+(\gamma_1)^2-2\cdot\frac{1}{2}\cdot (c-\frac{1}{2})+(\gamma_3)^2
=-\kappa+\frac{3}{4}
\]

\end{proof}

To justify somewhat our next definition, we make the following remark.

\begin{remark}\label{re:m_beta:scaling}
Notice that if $\beta\in\R$, $\beta\neq 0$ and $g_{\beta}:=\beta^{-2}g$
then the Koszul-formula gives,
\[
2g_{\beta}(\nabla^{g_{\beta}}_{F_i} F_j,F_k)
=&\beta^{-2} g([F_i,F_j],F_k)-\beta^{-2} g([F_i,F_k],F_j)-\beta^{-2} g([F_j,F_k],F_i) \\
=&2\beta^{-2}\Gamma^i_{(j,k)},
\]
because $g_{\beta}(F_i,F_j)=\beta^{-2}\delta_{ij}$.
Then, $E_i:=\beta F_i$, $i=1,2,3$, is a $g_{\beta}$-orthonormal basis 
and if $(\Gamma_\beta)^i_{(j,k)}=g_{\beta}(\nabla_{E_i} E_j,E_k)$,
then for every $i,j,k$.
\[
\beta^{-3}(\Gamma_\beta)^i_{(j,k)}=\beta^{-3}g_{\beta}(\nabla^{g_{\beta}}_{E_i} E_j,E_k)
=g_{\beta}(\nabla^{g_{\beta}}_{F_i} F_j,F_k)=\beta^{-2}\Gamma^i_{(j,k)}
\]
i.e. $(\Gamma_\beta)^i_{(j,k)}=\beta\Gamma^i_{(j,k)}$.
\end{remark}

\begin{definition}
A 3-dimensional Riemannian manifold $(M,g)$
is said to belong to \emph{class $\mc{M}_{\beta}$},
for $\beta\in\R$,
if there exists an orthonormal frame $E_1,E_2,E_3\in\VF(M)$
with respect to which the connection table is of the form
\[
\Gamma=\qmatrix{
\beta & 0 & 0 \\
\Gamma^1_{(3,1)} & \Gamma^2_{(3,1)} & \Gamma^3_{(3,1)} \\
0 & 0 & \beta
}.
\]
In this case the frame $E_1,E_2,E_3$ is called an \emph{adapted frame of $(M,g)$}.
\end{definition}

\begin{remark}
For a given $\beta\in\R$, one can say that a Riemannian space $(M,g)$
is \emph{locally of class $\mc{M}_{\beta}$},
if every $x\in M$ has an open neighbourhood $U$ such that
$(U,g|_U)$ is of class $\mc{M}_{\beta}$.
Since we are interested in local results, we usually speak of manifolds of (globally) class $\mc{M}_{\beta}$.
\end{remark}

\begin{lemma}\label{le:mbeta-2}
If $\beta\neq 0$ and $(M,g)$ is of class $\mc{M}_{\beta}$ with an adapted frame,
then $\star E_1,\star E_2,\star E_3$
are eigenvectors of $R$ with eigenvalues $-\beta^2,-K_2(\cdot),-\beta^2$,
where
\[
-K_2(x)=\beta^2+E_3(\Gamma^1_{(3,1)})-E_1(\Gamma^3_{(3,1)})+(\Gamma^1_{(3,1)})^2+(\Gamma^3_{(3,1)})^2
-2\beta\Gamma^2_{(3,1)},\quad x\in M.
\]
\end{lemma}

\begin{proof}
Immediate from Proposition \ref{pr:special_3D}, Eq. \eqref{eq:RsE2sE2}.
\end{proof}

Next lemma is the converse of what has been done before the above definition.

\begin{lemma}\label{le:m_beta:relations}
Let $(M,g)$ be of class $\mc{M}_{\beta}$, $\beta\neq 0$, with an adapted frame $E_1,E_2,E_3$.
Then $M$ is a contact manifold of type $(\kappa,0)$
with (normalized) adapted frame $F_i:=\frac{1}{2\beta}E_i$, $i=1,2,3$.
Moreover, $\kappa$ and the structure functions $c,\gamma_1,\gamma_3$ are given by
\[
c=&\frac{\beta+\Gamma^2_{(3,1)}}{2\beta},\quad \gamma_1=\frac{\Gamma^1_{(3,1)}}{2\beta},\quad \gamma_3=-\frac{\Gamma^3_{(3,1)}}{2\beta} \\
\kappa(x)=&\frac{K_2(x)}{4\beta^2}+\frac{3}{4},\quad x\in M.
\]
\end{lemma}

\begin{proof}
From the torsion freeness of the Levi-Civita connection on $(M,g)$
and from the connection table w.r.t. $E_1,E_2,E_3$, we get
\[
[E_1,E_2]=&(\beta+\Gamma^2_{(3,1)})E_3 \\
[E_2,E_3]=&(\beta+\Gamma^2_{(3,1)})E_1 \\
[E_3,E_1]=&-\Gamma^1_{(3,1)}E_1+2\beta E_2-\Gamma^3_{(3,1)}E_3.
\]
From this and the fact that $\beta\neq 0$, the claims are immediate.
\end{proof}

\begin{remark}\label{re:m_beta:scaling2}
\begin{itemize}
\item[(i)] We notice that the classes $\mc{M}_{\beta}$ and $\mc{M}_{-\beta}$
are the same. Indeed, if $(M,g)$ is of class $\mc{M}_{\beta}$
and $E_1,E_2,E_3$ is an adapted orthonormal frame,
then $(M,g)$ is  of class $\mc{M}_{-\beta}$
with a adapted frame $F_1,F_2,F_3$ where $F_1=E_3$, $F_3=E_1$
(i.e. the change of orientation of $E_1,E_3$ plane moves from $\mc{M}_{\beta}$
to $\mc{M}_{-\beta}$).
So in this sense it would be better to speak of Riemannian manifolds of class $\mc{M}_{\beta}$
with $\beta\geq 0$ or of class $\mc{M}_{|\beta|}$.

\item[(ii)] If one has a Riemannian manifold $(M,g)$ of class $\mc{M}_{\beta}$,
then scaling the metric by $\lambda\neq 0$ one gets a Riemannian manifold $(M,\lambda^2 g)$
of class $\mc{M}_{\beta/\lambda}$.
This follows from Remark \ref{re:m_beta:scaling} above.
\end{itemize}
\end{remark}

\begin{remark}\label{re:m_beta:0}
If $(M,g)$ is of class $\mc{M}_0$,
then since $\beta=0$ and $\Gamma^1_{(1,2)}=0$, one deduces e.g. from Theorem \ref{th:warped_product}
that $(M,g)$ is locally a warped product.
Converse is easily seen to be true i.e. that a Riemannian product manifold
is locally of class $\mc{M}_0$. Hence there are many non-isometric
spaces of class $\mc{M}_0$.
\end{remark}

To conclude this subsection,
we will show that that for every $\beta\in\R$ there exist
3-dimensional Riemannian manifolds
of class $\mc{M}_\beta$ which are not all isometric.
See also \cite{agrachev10}.

\begin{example}
\begin{itemize}
\item[(i)] Let $M$ be $\SO(3)$. There one has left-invariant vector fields $E_1,E_2,E_3$
such that
\[
[E_1,E_2]=&E_3 \\
[E_2,E_3]=&E_1 \\
[E_3,E_1]=&E_2
\]
Hence with the metric $g$ rendering $E_1,E_2,E_3$ orthonormal,
we get a space $(M,g)$ of class $\mc{M}_{1/2}$.
By the definition of $\kappa$ and Lemma \ref{le:m_beta:relations} we have $\kappa=1$ and $K_2=\frac{1}{4}$.

\item[(ii)] Let $M$ be the Heisenberg group $H_3$.
Here one has left-invariant vector fields $E_1,E_2,E_3$
which satisfy
\[
[E_1,E_2]=&0 \\
[E_2,E_3]=&0 \\
[E_3,E_1]=&E_2.
\]
Hence $M$ with the metric for which $E_1,E_2,E_3$ are orthonormal,
is of class $\mc{M}_{1/2}$ and $\kappa=0$, $K_2=-\frac{3}{4}$.

\item[(iii)]
Take $M$ to be $\mathrm{SL}(2)$.
Then on $M$ there are left-invariant vector fields with commutators
\[
[E_1,E_2]=&-E_3 \\
[E_2,E_3]=&-E_1 \\
[E_3,E_1]=&E_2.
\]
If $g$ is a metric with respect to which $E_1,E_2,E_3$ are orthonormal,
it follows that $M$ is of class $\mc{M}_{1/2}$,
Here $\kappa=-1$, so $K_2=-\frac{7}{4}$.

Notice that if one takes the "usual" basis of $\mathfrak{sl}(2)$
as $a,b,c$ satisfying,
\[
[c,a]=2a,\quad [c,b]=-2b,\quad [a,b]=c,
\]
then one may define $e_1=\frac{a+b}{2}$, $e_2=\frac{a-b}{2}$, $e_3=\frac{c}{2}$
to obtain
\[
[e_1,e_2]=-e_3,\quad
[e_2,e_3]=-e_1,\quad
[e_3,e_1]=e_2.
\]
\end{itemize}
None of the examples in (i)-(iii) of Riemannian manifolds of class $\mc{M}_{\beta}$ with $\beta=\frac{1}{2}$
are (locally) isometric one to the other.
This fact is immediately read from the different values of $K_2$ (constant).

Hence by Remarks \ref{re:m_beta:scaling2} and \label{re:m_beta:0},
we see that for every $\beta\in\R$ there are non-isometric Riemannian
manifolds of the same class $\mc{M}_{\beta}$.
\end{example}

%%%%%%%%%%%%%%%%%%%%%%%%%%%%%%%%%%%%%%%%%%%%%%%%%%%%
\subsection{Warped Products}\label{app:warped}
%%%%%%%%%%%%%%%%%%%%%%%%%%%%%%%%%%%%%%%%%%%%%%%%%%%%

\begin{definition}
Let $(M,g)$, $(N,h)$ be Riemannian manifolds and $f\in\Cinf(M)$.
Define a metric $h_f$ on $M\times N$
\[
h_f=\pr_1^*(g)+(f\circ \pr_1)^2\pr_2^*(h),
\]
where $\pr_1,\pr_2$ are projections onto the first and second factor of $M\times N$,
respectively.
Then the Riemannian manifold $(M\times N,h_f)$
is called a \emph{warped product of $(M,g)$ and $(N,h)$} with the \emph{warping function $f$}.

One may write $(M\times N,h_f)$ as $(M,g)\times_f (N,h)$ for short
and $h_f$ as $g\oplus_f h$ if there is a risk of ambiguity.
\end{definition}

We are mainly interested in the case where $(M,g)=(I,s_1)$,
where $I\subset\R$ is an open non-empty interval and $s_1$ is the standard Euclidean metric on $\R$.
By convention, we write $\pa{r}$ for the natural positively directed unit (w.r.t. $s_1$) vector field on $\R$
and identify it in the canonical way as a vector field on
the product $I\times N$ and notice that it is also a unit vector field w.r.t. $h_f$.

Since needed in section \ref{se:3D},
we formulate, and proof, (a local version of) the main result of \cite{hiepko79}
in 3-dimensional case.
The general result allows one to detect Riemannian spaces which
are locally warped products.
In our setting we use it (in the below form)
to detect when a 3-dimensional Riemannian manifold $(M,g)$
is, around a given point, a warped product of the form $(I\times N,h_f)$, with $I\subset\R$, $f\in \Cinf(I)$,
and $(N,h)$ a 2-dimensional Riemannian manifold.

\begin{theorem}(\cite{hiepko79})\label{th:warped_product}
Let $(M,g)$ be a Riemannian manifold of dimension 3.
Suppose that at every point $x_0\in M$
there is an orthonormal frame $E_1,E_2,E_3$ defined in a
neighbourhood of $x_0$ such that
the connection table w.r.t. $E_1,E_2,E_3$ on this neighbourhood
is of the form
\[
\Gamma=\qmatrix{
0 &                                0                                & -\Gamma^1_{(1,2)} \cr
\Gamma^1_{(3,1)} & \Gamma^2_{(3,1)} & \Gamma^3_{(3,1)} \cr
\Gamma^1_{(1,2)} &                                0                               & 0 \cr
}
\]
and moreover
\[
X(\Gamma^1_{(1,2)})=0,\quad \forall X\in E_2^\perp.
\]

Then there is a neighbourhood $U$ of $x$,
an interval $I\subset\R$, $f\in\Cinf(I)$ and a 2-dimensional Riemannian manifold $(N,h)$
such that $(U,g|_U)$ is isometric to the warped product $(I\times N,h_f)$.
If $F:(I\times N,h_f)\to (U,g|_U)$ is the isometry in question, then
for all $(r,y)\in I\times N$,
\[
\frac{f'(r)}{f(r)}=&-\Gamma^1_{(1,2)}(F(r,y)) \\
F_*\pa{r}\big|_{(r,y)}=&E_2|_{\phi(r,y)}.
\]
\end{theorem}

\begin{proof}
Write $O$ for the domain of $E_1,E_2,E_3$.
According to the main result in \cite{hiepko79},
whose proof in our special case will be outlined below,
it is enough to prove the following:
\begin{itemize}
\item[(i)] $E_2$ is a geodesic vector field and the 2-dimensional distribution $E_2^\perp$
is integrable.
\item[(ii)] There exists a function $\eta\in\Cinf(O)$
such that for all $U,W\in E_2^\perp$,
\[
g(\nabla_U W,E_2)=\eta g(U,W)
\]
and
\[
g(\nabla_{U} (\eta E_2),E_2)=0.
\]
\end{itemize}

(i) Since $0=\Gamma^2_{(1,2)}=-g(\nabla_{E_2} E_2,E_1)$
and $0=\Gamma^2_{(3,1)}=-g(\nabla_{E_2} E_2,E_3)$
by assumption and since by normality $0=\Gamma^2_{(2,2)}=g(\nabla_{E_2} E_2,E_2)$,
we see that $E_2$ is a geodesic vector field.

As for the integrability of $E_2$, observe that
since $E_2^\perp$ is spanned by $E_1,E_3$ and
by assumption and torsion freeness of the Levi-Civita connection,
\[
g([E_1,E_3],E_2)=g(\nabla_{E_1} E_3-\nabla_{E_3} E_1,E_2)=\Gamma^1_{(3,2)}-\Gamma^2_{(1,2)}=0-0=0
=0.
\]
which shows that $[E_1,E_3]$ is tangent to $E_2^\perp$.
This shows that condition (i) holds.

As for (ii), we see that
\[
& g(\nabla_{E_1} E_3,E_2)=\Gamma^1_{(3,2)}=0,\quad g(\nabla_{E_3} E_1,E_2)=\Gamma^3_{(1,2)}=0 \\
& g(\nabla_{E_1} E_1,E_2)=\Gamma^1_{(1,2)}=-\Gamma^3_{(2,3)}=g(\nabla_{E_3} E_3,E_2).
\]
Thus defining $\eta:=\Gamma^1_{(1,2)}$,
we see that if $U,W\in\VF(M)$ are tangent to $E_2^\perp$
and hence we may write $U=aE_1+bE_3$, $W=cE_1+dE_3$ for some $a,b,c,d\in\VF(M)$,
\[
g(\nabla_{U} W,E_2)
=&ac \underbrace{\Gamma^1_{(1,2)}}_{=\eta}+ad\underbrace{\Gamma^1_{(3,2)}}_{=0}+bc\underbrace{\Gamma^3_{(1,2)}}_{=0}+bd\underbrace{\Gamma^3_{(3,2)}}_{=-\Gamma^3_{(2,3)}=\eta} \\
=&\eta (ac+bd)=\eta g(U,W).
\]

Finally, if $U=aE_1+bE_3$,
\[
g(\nabla_U(\eta E_2),E_2)
=U(\eta)+\eta \underbrace{g(\nabla_U E_2,E_2)}_{=\sfrac{1}{2}Ug(E_2,E_2)=0}=U(\eta)=0.
\]
This proves that (ii) holds.

We will now show how to construct $f$, $(I\times N,h_f)$ and $F$.
First notice that a simple computation of e.g. $R(E_1\wedge E_2)$
with respect to the frame $E_1,E_2,E_3$ yields
\[
-K=-E_2(\Gamma^1_{(1,2)})+(\Gamma^1_{(1,2)})^2.
\]

Since $E_2^\perp$ is integrable,
there is an integral manifold $N$ of $E_2^\perp$ through $x_0$.
Let $r(x)=d_g(N,x)$ be the geodesic distance from $x$ to $N$.
By shrinking $N$ (such that it still contains $x_0$) if necessary,
we may find real numbers $a<0<b$ such that on $V:=r^{-1}(]a,b[)$
the function $r$ is smooth. Write $I=]a,b[$
and define $F:I\times N\to M$ by
\[
F(t,y)=\exp_y(tE_2|_y),
\]
which is smooth, has image actually in $V$
 and has a smooth inverse $H:V\to I\times N$ defined by
 \[
H(x)=(r(x),\exp_x(-r(x)E_2|_x)).
 \]
Thus $F$ is a diffeomorphism $I\times N\to V$.

Define $f$ by an ODE
\[
f'(t)=-\eta(F(t,x_0))f(t),\quad f(0)=1,
\]
where $\eta=\Gamma^1_{(1,2)}$ as above.
We take $I=]a,b[$ a smaller interval containing $0$ if necessary,
so that $f(t)>0$ for all $t\in I$.
Also, define the metric $h$ on $N$ by
\[
h(X,Z)=\frac{1}{f(H(x))^2}g(X,Z),\quad X,Z\in E_2^\perp|_x=T|_x N,\quad x\in N.
\]

Finally, we will prove that $F$ is an isometry between $(I\times N,h_f)$ and $(V,g|_V)$.
At first we notice that
\[
\n{F_*\pa{t}\big|_{(t,y)}}_g
=\n{E_2|_{F(t,y)}}_g=1=\n{\pa{t}\big|_{(t,y)}}_{h_f}.
\]

Let $X\in T|_y N\subset T|_{(t,y)} (I\times N)$.
Since
\[
F(t,y)=(\pi_{TM}\circ \Phi_G)(t,E_2|_y),
\]
where $G\in\VF(M)$ is the geodesic vector field and $\Phi_G$ its flow,
we get that
\[
F_*X=J_{(E_2|_{y},X,\nabla_X E_2)}(t),
\]
where for $u,A,B\in T|_x M$, $J_{(u,A,B)}(t)$ is the Jacobi field of $(M,g)$
along $t\mapsto \exp_x(tu)$
with initial conditions $J_{(u,A,B)}(0)=A$, $\nabla_{\pa{t}}\big|_0 J_{(u,A,B)}(t)=B$.
One easily computes that
\[
\nabla_X E_2=-\eta(y) X.
\]

Because $\nabla_{E_2} E_1=-\Gamma^2_{(3,1)} E_3$,
$\nabla_{E_2} E_3=\Gamma^2_{(3,1)} E_1$,
we have that
for any vector field $U$ perpendicular to $E_2$, the covariant derivative $\nabla_{E_2} U$
is also perpendicular to $E_2$.
This implies that if we set $\gamma(t):=F(t,y)=\exp_y(tE_2|_y)$,
then $P_0^t(\gamma)X$ is perpendicular to $E_2$, whose integral
curve $\gamma$ is (since $E_2$ is a geodesic vector field).

With this noticed, we take as an \emph{ansatz} for the Jacobi field $J_{(E_2|_{y},X,\nabla_X E_2)}(t)$
a vector field $J(t)$ along $\gamma$ of the form
\[
J(t)=\lambda(t)P_0^t(\gamma)X.
\]
This satisfies
\[
R(\dot{\gamma}(t),J(t))\dot{\gamma}(t)=-K(\gamma(t))(\dot{\gamma}(t)\wedge J(t))\dot{\gamma}(t)
=-K(\gamma(t))J(t).
\]

Moreover, $J(0)=\lambda(0)X$,
\[
\nabla_{\dot{\gamma}(t)} J=\lambda'(t)P_0^t(\gamma)X,
\]
so $\nabla_{\dot{\gamma}(t)} J(0)=\lambda'(0) X$ and
\[
\nabla_{\dot{\gamma}(t)} \nabla_{\dot{\gamma}} J=\lambda''(t)P_0^t(\gamma)X
=\frac{\lambda''(t)}{\lambda(t)}J(t)
\]
and thus we see that, in order to make $J(t)$ a Jacobi field with the same initial conditions
as for $J_{(u,A,B)}(t)$, we should choose $\lambda$ as a solution to
\[
\lambda''(t)=-K(\gamma(t))\lambda(t),\quad \lambda(0)=1,\quad\lambda'(0)=-\eta(y).
\]

By uniqueness of solutions to 2nd order ODE,
we hence have $J_{(E_2|_{y},X,\nabla_X E_2)}(t)=J(t)$ i.e.
\[
F_*X=\lambda(t)P_0^t(\gamma)X.
\]
Now one also has
\[
-K(\gamma(t))=-E_2|_{\gamma(t)}(\eta)+\eta(\gamma(t))^2
=\dif{t}\big(\frac{f'(t)}{f(t)})+\big(\frac{f'(t)}{f(t)}\big)^2=\frac{f''(t)}{f(t)},
\]
Since $f(0)=1=\lambda(0)$, $f'(0)=-\eta(y)f(0)=-\eta(y)=\lambda'(0)$, we see that $\lambda(t)=f(t)$ for all $t\in I$
and hence
\[
F_*X=f(t)P_0^t(\gamma)X.
\]

Thus we finally have
\[
\n{F_*X}^2_g=f(t)^2\n{X}^2_g=\n{X}^2_{h_f}
\]
which establishes the fact that $F$ is an isometry $(I\times N,h_f)\to (V,g|_V)$.
Notice that by definition of $F$,
\[
F_*\pa{r}\big|_{(r,y)}=\dif{t}\big|_r F(t,y)=\dif{t}\big|_r \exp_{y} (tE_2|_y)=E_2|_{F(t,y)},
\]
since $E_2$ is a geodesic vector field.

We still need to prove first of the formulas at the end of the statement of the theorem.
Let $(r,y)\in I\times N$.
Choose a path $\gamma$ in $N$ from $x_0\in N$ to $y\in N$.
Since $t\mapsto (r,\gamma(t))$ is perpendicular to $\pa{r}|_{(r,\gamma(t)})$ in $I\times N$,
it follows that, since $F$ is an isometry, $t\mapsto F(r,\gamma(t))$
is perpendicular to $E_2^\perp$.
Hence $\dif{t} \Gamma^1_{(1,2)}(F(r,\gamma(t)))=0$,
which implies that $\Gamma^1_{(1,2)}(F(r,x_0))=\Gamma^1_{(1,2)}(F(r,y))$.
On the other hand, by the definition of $f$ and $\eta$,
one has $\frac{f'(r)}{f(r)}=-\Gamma^1_{(1,2)}(F(r,x_0))$.
This completes the proof.
\end{proof}

Following example shows how to build constant curvature
spaces from other constant curvature spaces
as a warped product illustrates the concept.
The example will also be used in the proof of Proposition \ref{pr:Rol1:main}.

\begin{example}\label{ex:warped_const}
In this example we show how one can locally write
a 3-dimensional Riemannian manifold $(M,g)$ of constant curvature $K$
as a warped product $(I\times N,h_f)$,
where $I$ is an open real interval containing $0$
and $(N,h)$ is a 2-dimensional space of constant curvature.

Let $\sigma,K\in\R$
and take any 2-dimensional Riemannian space $(N,h)$ of constant curvature $\sigma$,
let $a,b\in\R,a>0$, and take $f$ to be the solution to
\[
f''(t)=-Kf(t),\quad f(0)=a,\quad f'(0)=b.
\]
This solution is positive for at least some open interval $I\subset\R$ containing $0$.
Since
\[
\dif{t}(Kf(t)^2+f'(t)^2)=2Kf(t)f'(t)+2f'(t)f''(t)=2f(t)f'(t)-2Kf'(t)f(t)=0,
\]
we have
\[
Kf(t)^2+f'(t)^2=Kf(0)^2+f'(0)^2=Ka^2+b^2.
\]
Compute also
\[
\dif{t}\big(\frac{-\sigma+f'(t)^2}{f(t)^2}\big)
=&\frac{2f'(t)f''(t)f(t)^2-2f(t)f'(t)(-\sigma+f'(t)^2)}{f(t)^4} \\
=&\frac{-2f'(t)f(t)(Kf(t)^2+f'(t)^2)+2\sigma f(t)f'(t)}{f(t)^4} \\
=&\frac{-2f'(t)f(t)(Ka^2+b^2)+2\sigma f(t)f'(t)}{f(t)^4} \\
=&-2(Ka^2+b^2-\sigma)\frac{f'(t)}{f(t)^4}
\]
and define
\[
S_K:=\{(a,b,\sigma)\in\R^3\ |\ a>0,\ Ka^2+b^2=\sigma\}.
\]
If $(a,b,\sigma)\in S_K$, it follows that $\dif{t}\big(\frac{-\sigma+f'(t)^2}{f(t)^2}\big)=0$
and hence
\[
\frac{-\sigma+f'(t)^2}{f(t)^2}=\frac{-\sigma+b^2}{a^2}=-K.
\]

Suppose $y\in N$, $t\in\R$ and that $X,Z,\in T|_y N$ are $h$-orthonormal.
We consider $X,Z$ and $T:=\pa{t}|_t$, where $\pa{t}$ is the canonical positively directed
vector field on $I\subset\R$,
as vectors in $T|_{(t,y)} (I\times N)$ in the usual way.
Then by \cite{oneill83}, Chapter 7, Proposition 42,
we get (notice that there the definition of the curvature tensor
differs by sign from ours),
\[
R^{h_f}(X,T)T=&KX \\
R^{h_f}(X,Z)T=&0 \\
R^{h_f}(T,X)Z=&0 \\
R^{h_f}(T,X)X=&Kf(t)^2T \\
R^{h_f}(X,Z)X=&(-\sigma+f'(t)^2)Z.
\]
Thus if $X,Z$ is an $h$-orthonormal frame on $N$, it follows that $E_1:=\frac{1}{f}X$, $E_2:=T$, $E_3:=\frac{1}{f}Z$
is an $g$-orthonormal frame on $I\times N$ and
\[
R^{h_f}(E_1\wedge E_2)=-KE_1\wedge E_2 \\
R^{h_f}(E_2\wedge E_3)=-KE_2\wedge E_3 \\
R^{h_f}(E_3\wedge E_1)=\frac{-\sigma+f'(t)^2}{f(t)^2}E_3\wedge E_1.
\]

Hence for every choice $(a,b,\sigma)\in S_K$ we
get also that $R^{h_f}(E_3\wedge E_1)=-KE_3\wedge E_1$
which then allows us to conclude that $(I\times N,h_f)$
is a space of constant curvature $K$.

For any $(a,b)\in\R^2$ with $a>0$, one may define 
as $\sigma(a,b)=Ka^2+b^2$, which then implies that $(a,b,\sigma(a,b))\in S_K$
and hence the set $\{(a,b)\in\R^2\ |\ \exists \sigma\in\R\ \textrm{s.t.}\ (a,b,\sigma)\in S_K\}$
is the open right half-plane of $\R^2$.

The conclusion here is that for every $K\in\R$ and $a,b\in\R$ with $a>0$,
one can construct a warped product $(I\times N,h_f)$
which has constant curvature $K$ and
where $N$ is a space of constant curvature (with curvature $Ka^2+b^2$)
and that the warping function satisfies: $f(0)=a$, $f'(0)=b$.
This will be used in the proof of Proposition \ref{re:m_beta:0}.
\end{example}

Here we will make a remark about the intersection of the classes $\mc{M}_{\beta}$
and the class of warped products of the form $(I,s_1)\times_f (N,h)$
with $(N,h)$ two-dimensional.

\begin{remark}
Suppose that a Riemannian 3-manifold $(M,g)$
is at the same time a warped product $(I\times N,h_f)$
and belongs to class $\mc{M}_{\beta}$,
with $E_1,E_2,E_3$ an adapted frame.
As a warped product, the curvature tensor $R$
has eigenvalues (functions) $-K(\cdot),-K_2(\cdot),-K(\cdot)$ (with some eigenvector fields)
where $-K(r,y)=\frac{f''(r)}{f(r)}$.
Since as a $\mc{M}_{\beta}$, the operator $R$ has eigenvalues $-\beta^2,-K'_2(\cdot),-\beta^2$,
we must have (taking any combination) that $K=\beta^2$ is a constant and $K_2'=K_2$ everywhere on $M$.
Let us now consider three different cases:
\begin{itemize}
\item[(i)] If $\beta=0$, then it immediately follows that $(M,g)$ is a Riemannian product,
since $f$ is constant.

\item[(ii)] Suppose that $\beta\neq 0$ and $K_2(r_0,y_0)\neq \beta^2$
at some point $(r_0,y_0)\in M$.
Then there is a neighbourhood $U$ of $(r_0,y_0)$ where $K_2\neq \beta^2$.
It follows that $E_2|_{(r,y)}=\pm \pa{r}\big|_{(r,y)}$ for $(r,y)\in U$,
from which it follows that in the connection table w.r.t. $E_1,E_2,E_3$
one must have also $\Gamma^1_{(1,2)}(r,y)=0$ and $\beta=\Gamma^1_{(2,3)}(r,y)=0$ for all $(r,y)\in U$.
Therefore $-\frac{f''(r)}{f(r)}=K=\beta^2=0$ for all $r\in I$,
which implies that $f'(r)$ is a constant function.
But $\Gamma^1_{(1,2)}(r,y)=-\frac{f'(r)}{f(r)}$ vanishes on $U$,
hence $f'(r)$ vanishes for some $r$ and hence $f'(r)=0$ for all $r\in I$.
This implies that $(M,g)$ is a Riemannian product.

\item[(iii)] If $\beta\neq 0$ and $K_2(r,y)=\beta^2$,
then $(M,g)$ has a constant curvature $\beta^2$
and hence is locally isometric to a sphere of curvature $\beta^2$.
\end{itemize}

As a conclusion, if a warped product $(M,g)=(I\times N,h_f)$
belongs to class $\mc{M}_{\beta}$,
then either it is (a) a Riemannian product ($\beta=0$, $f$ constant)
or (b) a space of constant curvature $\beta^2$.
Both (a) and (b) occur if and only if $(M,g)$ is flat.
\end{remark}

%%%%%%%%%%%%%%%%%%%%%%%%
\subsection{Technical propositions}
%%%%%%%%%%%%%%%%%%%%%%%%

Since we will be dealing frequently with orthonormal frames
and connection coefficients,
it is convenient to define the following concept.

\begin{definition}
Let $(M,g)$ be a 3-dimensional Riemannian manifold.
If $E_1,E_2,E_3$ is an orthonormal frame of $M$ defined on an open set $U$,
then $\Gamma^j_{(i,k)}=g(\nabla_{E_j} E_i,E_k)$,
we call the matrix
\[
\Gamma=\qmatrix{
\Gamma^1_{(2,3)} & \Gamma^2_{(2,3)} & \Gamma^3_{(2,3)} \cr
\Gamma^1_{(3,1)} & \Gamma^2_{(3,1)} & \Gamma^3_{(3,1)} \cr
\Gamma^1_{(1,2)} & \Gamma^2_{(1,2)} & \Gamma^3_{(1,2)}
},
\]
the \emph{connection table w.r.t. $E_1,E_2,E_3$}.
To emphasize the frame, we may write $\Gamma=\Gamma_{(E_1,E_2,E_3)}$.
\end{definition}

\begin{remark}
\begin{itemize}
\item[(i)] Since $E_1,E_2,E_3$ is orhonormal,
one has $\Gamma^i_{(j,k)}=-\Gamma^i_{(k,j)}$ for all $i,j,k$.
These relations mean that to know all the connection
coefficients (of an orthonormal frame),
it is enough to know exactly 9 of them.
It is these 9 coefficients, that appear in the connection table.

\item[(ii)] Here it is important that the frame $E_1,E_2,E_3$ is ordered
and hence one should speak of the connection table w.r.t. $(E_1,E_2,E_3)$
(as in the notation $\Gamma=\Gamma_{(E_1,E_2,E_3)}$),
but since we always list the frame in the correct order, there
will be no room for confusion.

\item[(iii)] Notice that the above connection table could be written as $\Gamma=[(\Gamma_{\star i}^j)_j^i]$,
if one writes $\star 1=(2,3)$, $\star 2=(3,1)$ and $\star 3=(1,2)$ i.e.
\[
\Gamma=\qmatrix{
\Gamma^1_{\star 1} & \Gamma^2_{\star 1} & \Gamma^3_{\star 1} \\
\Gamma^1_{\star 2} & \Gamma^2_{\star 2} & \Gamma^3_{\star 2} \\
\Gamma^1_{\star 3} & \Gamma^2_{\star 3} & \Gamma^3_{\star 3}
}.
\]
\end{itemize}
\end{remark}

\begin{proposition}\label{pr:special_3D}
Suppose $(M,g)$ is a 3-dimensional Riemannian manifold
and in some neighbourhood of $x\in M$
there is an orthonormal frame
$E_1,E_2,E_3$ defined on an open set $U$
with respect to which the connection table is of the form
\[
\Gamma=\qmatrix{
\Gamma^1_{(2,3)} &                                0                                & -\Gamma^1_{(1,2)} \cr
\Gamma^1_{(3,1)} & \Gamma^2_{(3,1)} 				& \Gamma^3_{(3,1)} \cr
\Gamma^1_{(1,2)} &                                0                               & \Gamma^1_{(2,3)} \cr
},
\]
on $U$, and moreover it holds that
\[
& V(\Gamma^1_{(2,3)})=0,\quad V(\Gamma^1_{(1,2)})=0,\quad \forall V\in E_2|_y^\perp,\quad y\in U, \\
\]

Then the following are true:
\begin{itemize}
\item[(i)]
For every $y\in U$, $\star E_1|_y,\star E_2|_y,\star E_3|_y$ are
eigenvectors of $R$ with
eigenvalues $-K(y),-K_2(y),-K(y)$, respectively
(i.e. the eigenvalues of $\star E_1|_y$ and $\star E_3|_y$ coincide).

\item[(ii)] If $\Gamma^1_{(2,3)}\neq 0$ on $U$
and if $U$ is connected,
it follows that on $U$ the coefficient $\Gamma^1_{(2,3)}$ is constant,
$\Gamma^1_{(1,2)}=0$ and $K(y)=(\Gamma^1_{(2,3)})^2$ (constant).
Hence $(U,g|_U)$ is of class $\mc{M}_{\beta}$, for $\beta=\Gamma^1_{(2,3)}$.

\item[(iii)] If $\Gamma^1_{(2,3)}=0$ in the open set $U$,
then every $y\in U$ has a neighbourhood $U'\subset U$
such that $(U',g|_{U'})$ is isometric to a warped product $(I\times N,h_f)$
where $I\subset\R$ is an open interval.
Moreover, if $F:I\times_f N\to (U',g|_{U'})$ is the isometry in question, then
\[
\frac{f'(r)}{f(r)}=&-\Gamma^1_{(1,2)}(F(r,y)),\quad \forall (r,y)\in I\times N \\
F_*\pa{r}\big|_{(r,y)}=&E_2|_{F(r,y)}.
\]
\end{itemize}
Moreover, one has
\begin{align}
0=&-E_2(\Gamma^1_{(2,3)})+2\Gamma^1_{(1,2)}\Gamma^1_{(2,3)}  \label{eq:RsE3sE1} \\
-K=&-E_2(\Gamma^1_{(1,2)})+(\Gamma^1_{(1,2)})^2-(\Gamma^1_{(2,3)})^2  \label{eq:RsE3sE3} \\
-K_2=&E_3(\Gamma^1_{(3,1)})-E_1(\Gamma^3_{(3,1)})+(\Gamma^1_{(3,1)})^2+(\Gamma^3_{(3,1)})^2 \label{eq:RsE2sE2} \\
&-2\Gamma^1_{(2,3)}\Gamma^2_{(3,1)}+(\Gamma^1_{(1,2)})^2+(\Gamma^1_{(2,3)})^2 \nonumber
\end{align}
\end{proposition}

\begin{proof}
(i) We begin by computing in the basis $\star E_1,\star E_2,\star E_3$ that
\[
R(E_3\wedge E_1)
=&\qmatrix{-\Gamma^1_{(1,2)} \cr \Gamma^3_{(3,1)} \cr \Gamma^1_{(2,3)}}\wedge\qmatrix{\Gamma^1_{(2,3)} \cr \Gamma^1_{(3,1)} \cr \Gamma^1_{(1,2)}}
+\qmatrix{E_3(\Gamma^1_{(2,3)}) \cr E_3(\Gamma^1_{(3,1)}) \cr E_3(\Gamma^1_{(1,2)})}
-\qmatrix{-E_1(\Gamma^1_{(1,2)}) \cr E_1(\Gamma^3_{(3,1)}) \cr E_1(\Gamma^1_{(2,3)})} \\
&+\Gamma^1_{(3,1)}\qmatrix{\Gamma^1_{(2,3)} \cr \Gamma^1_{(3,1)} \cr \Gamma^1_{(1,2)}}
-2\Gamma^1_{(2,3)}\qmatrix{0 \cr \Gamma^2_{(3,1)}\cr 0}
+\Gamma^3_{(3,1)}\qmatrix{-\Gamma^1_{(1,2)} \cr \Gamma^3_{(3,1)} \cr \Gamma^1_{(2,3)}}
=\qmatrix{0 \cr -K_2 \cr 0},
\]
where we omitted the further computation of row 2 and wrote it simply as $-K_2$
and use the fact that $E_i(\Gamma^1_{(1,2)})=0$, $E_i(\Gamma^1_{(2,3)})=0$ for $i\in\{1,3\}$.
Thus $\star E_2|_y$ is an eigenvector of $R|_y$ for all $y$.

Now fix $y\in U$.
Since $R|_y$ is a symmetric linear map $\wedge^2T|_y M$ to itself
and since $\star E_2|_y$ is an eigenvector for $R|_y$,
we know that the other eigenvectors lie in $\star E_2|_y$,
which is spanned by $\star E_1|_y,\star E_3|_y$.
By rotating $E_1,E_3$ among themselves
by a constant matrix, we may well assume that $\star E_1|_y,\star E_3|_y$
are eigenvectors of $R|_y$
corresponding to eigenvalues, say, $-K_1(y),-K_3(y)$.
We want to show that $K_1(y)=K_3(y)$.

Computing $R|_y(E_1\wedge E_2)$ in the basis $\star E_1|_y,\star E_2|_y,\star E_3|_y$ gives
(we write simply $\Gamma^i_{(j,k)}$ for $\Gamma^i_{(j,k)}(y)$ etc.)
\[
\qmatrix{0 \cr 0\cr -K_3(y)}=&R|_y(\star E_3)=R|_y(E_1\wedge E_2) \\
=&\qmatrix{\Gamma^1_{(2,3)} \cr \Gamma^1_{(3,1)} \cr \Gamma^1_{(1,2)}}\wedge\qmatrix{0 \cr \Gamma^2_{(3,1)} \cr 0}
+\qmatrix{0 \cr E_1(\Gamma^2_{(3,1)}) \cr 0}
-\qmatrix{E_2(\Gamma^1_{(2,3)}) \cr E_2(\Gamma^1_{(3,1)}) \cr E_2(\Gamma^1_{(1,2)})} \\
&+\Gamma^1_{(1,2)}\qmatrix{\Gamma^1_{(2,3)} \cr \Gamma^1_{(3,1)} \cr \Gamma^1_{(1,2)}}-(\Gamma^1_{(2,3)}+\Gamma^2_{(3,1)})\qmatrix{-\Gamma^1_{(1,2)} \cr \Gamma^3_{(3,1)} \cr \Gamma^1_{(2,3)}} \\
=&\qmatrix{
-E_2(\Gamma^1_{(2,3)})+2\Gamma^1_{(1,2)}\Gamma^1_{(2,3)} \cr
E_1(\Gamma^2_{(3,1)})-E_2(\Gamma^1_{(3,1)})+\Gamma^1_{(1,2)}\Gamma^1_{(3,1)}-(\Gamma^1_{(2,3)}+\Gamma^2_{(3,1)})\Gamma^3_{(3,1)} \cr
-E_2(\Gamma^1_{(1,2)})+(\Gamma^1_{(1,2)})^2-(\Gamma^1_{(2,3)})^2
},
\]
from where
\[
-K_3(y)=&-E_2|_y(\Gamma^1_{(1,2)})+(\Gamma^1_{(1,2)}(y))^2-(\Gamma^1_{(2,3)}(y))^2.
\]
Similarly, computing $R|_y(E_2\wedge E_3)$ in basis $\star E_1|_y,\star E_2|_y,\star E_3|_y$,
\[
\qmatrix{-K_1(y)\cr 0\cr 0}
=&R|_y(\star E_1)=R|_y(E_2\wedge E_3) \\
=&\qmatrix{0 \cr \Gamma^2_{(3,1)} \cr 0}
\wedge \qmatrix{-\Gamma^1_{(1,2)} \cr \Gamma^3_{(3,1)} \cr \Gamma^1_{(2,3)}}
+\qmatrix{-E_2(\Gamma^1_{(1,2)}) \cr E_2(\Gamma^3_{(3,1)}) \cr E_2(\Gamma^1_{(2,3)})}
-\qmatrix{0 \cr E_3(\Gamma^2_{(3,1)}) \cr 0}\\
&-(\Gamma^2_{(3,1)}+\Gamma^1_{(2,3)})\qmatrix{\Gamma^1_{(2,3)} \cr \Gamma^1_{(3,1)} \cr \Gamma^1_{(1,2)}}
-\Gamma^1_{(1,2)}\qmatrix{-\Gamma^1_{(1,2)} \cr \Gamma^3_{(3,1)} \cr \Gamma^1_{(2,3)}} \\
=&\qmatrix{-E_2(\Gamma^1_{(1,2)})-(\Gamma^1_{(2,3)})^2+(\Gamma^1_{(1,2)})^2 \cr
E_2(\Gamma^3_{(3,1)})-E_3(\Gamma^2_{(3,1)})-(\Gamma^2_{(3,1)}+\Gamma^1_{(2,3)})\Gamma^1_{(3,1)}
-\Gamma^1_{(1,2)}\Gamma^3_{(3,1)} \cr
E_2(\Gamma^1_{(2,3)})-2\Gamma^1_{(1,2)}\Gamma^1_{(2,3)}
}
\]
leads us to
\[
-K_1(y)&=-E_2|_y(\Gamma^1_{(1,2)})-(\Gamma^1_{(2,3)}(y))^2+(\Gamma^1_{(1,2)}(y))^2.
\]
By comparing to the result of the computations of $R|_y(E_1\wedge E_2)$
and $R|_y(E_2\wedge E_3)$
implies that $K_1(y)=K_3(y)$.
In other words, if one writes $K(y)$ for this common value $K_1(y)=K_3(y)$,
one sees that $E_2|_y^\perp$ is contained in the eigenspace of $R|_y$
corresponding to the eigenvalue $-K(y)$.
This finishes the proof of (i).

(ii) Suppose now that $\Gamma^1_{(2,3)}\neq 0$ on
an open connected subset $U$ of $\pi_{\mc{O}_{\RDist}(q_0),M}(O)$.
Then since $E_1(\Gamma^1_{(2,3)})=0$, $E_3(\Gamma^1_{(2,3)})=0$ on $U$,
one has, on $U$,
$$
[E_3,E_1](\Gamma^1_{(2,3)})=E_3(E_1(\Gamma^1_{(2,3)}))
E_1(E_3(\Gamma^1_{(2,3)}))=0.
$$
On the other hand,
$$
[E_3,E_1]=-\Gamma^1_{(3,1)}E_1+2\Gamma^1_{(2,3)}E_2-\Gamma^3_{(3,1)}E_3,
$$
so
\[
0=&[E_3,E_1](\Gamma^1_{(2,3)})
=-\Gamma^1_{(3,1)}E_1(\Gamma^1_{(2,3)})+2\Gamma^1_{(2,3)}E_2(\Gamma^1_{(2,3)})-\Gamma^3_{(3,1)}E_3(\Gamma^1_{(2,3)}) \\
=&2\Gamma^1_{(2,3)}E_2(\Gamma^1_{(2,3)}).
\]
Since $\Gamma^1_{(2,3)}\neq 0$ everywhere on $U$,
one has $E_2(\Gamma^1_{(2,3)})=0$ on $U$.
Because $E_1,E_2,E_3$ span $TM$ on $U$,
we have that all the derivatives of $\Gamma^1_{(2,3)}$ vanish on $U$
and thus it is constant.

From first row of the computation of $R(E_1\wedge E_2)$ in the case (ii) above,
one gets
\[
0=-E_2(\Gamma^1_{(2,3)})+2\Gamma^1_{(1,2)}\Gamma^1_{(2,3)}=2\Gamma^1_{(1,2)}\Gamma^1_{(2,3)},
\]
which implies $\Gamma^1_{(1,2)}=0$ on $U$.
Finally from the last row
computation of $R(E_1\wedge E_2)$
(recall that $K_1(y)=K_3(y)=:K(y)$)
\[
-K(y)=-E_2(\Gamma^1_{(1,2)})+(\Gamma^1_{(1,2)})^2-(\Gamma^1_{(2,3)})^2=-(\Gamma^1_{(2,3)})^2.
\]
This concludes the proof of (ii).

(iii) This case follows from Theorem \ref{th:warped_product}.
\end{proof}

\end{document}